\documentclass[10pt,twoside,openright]{report} 
\pdfoutput=1
\usepackage{amsmath,amsfonts,amssymb,amsthm,array}
\usepackage{amsmath}
\usepackage{algorithm}
\usepackage{caption}
\usepackage{subcaption}
\usepackage{algpseudocode,algorithm,algorithmicx}
\usepackage{bm}
\usepackage{color}
\usepackage{graphicx}

\newcommand{\Exp}{\mathbb{E}}

\newcommand{\E}[1]{{\mathbb{E}\left[#1\right] }}    

\newcommand{\Prob}{\mathbb{P}}
\newcommand{\R}{\mathbb{R}}

\algrenewcommand\algorithmicrequire{\textbf{Input:}}
\algrenewcommand\algorithmicensure{\textbf{Initialize:}}

\usepackage{multirow}

\newcommand{\bA}{\mathbf{A}}
\newcommand{\bB}{\mathbf{B}}
\newcommand{\bC}{\mathbf{C}}
\newcommand{\bD}{\mathbf{D}}

\newcommand{\bH}{\mathbf{H}}
\newcommand{\bI}{\mathbf{I}}
\newcommand{\bL}{\mathbf{L}}
\newcommand{\bM}{\mathbf{M}}

\newcommand{\bP}{\mathbf{P}}
\newcommand{\bQ}{\mathbf{Q}}

\newcommand{\bS}{\mathbf{S}}

\newcommand{\bW}{\mathbf{W}}
\newcommand{\bZ}{\mathbf{Z}}
\newcommand{\bSigma}{\mathbf{\Sigma}}
\newcommand{\bU}{\mathbf{U}}

\newcommand{\bV}{\mathbf{V}}
\newcommand{\bLambda}{\mathbf{\Lambda}}
\newcommand{\eqdef}{:=}

\newcommand{\cC}{{\cal C}}
\newcommand{\cD}{{\cal D}}
\newcommand{\cE}{{\cal E}}
\newcommand{\cF}{{\cal F}}
\newcommand{\cG}{{\cal G}}

\newcommand{\cL}{{\cal L}}

\newcommand{\cN}{{\cal N}}
\newcommand{\cO}{{\cal O}}

\newcommand{\cQ}{{\cal Q}}

\newcommand{\cS}{{\cal S}}

\newcommand{\cV}{{\cal V}}
\newcommand{\cX}{{\cal X}}

\newcommand{\mA}{{\bf A}}
\newcommand{\mB}{{\bf B}}

\newcommand{\mH}{{\bf H}}
\newcommand{\mI}{{\bf I}}

\newcommand{\mL}{{\bf L}}
\newcommand{\mM}{{\bf M}}

\newcommand{\mS}{{\bf S}}

\newcommand{\mU}{{\bf U}}

\newcommand{\mW}{{\bf W}}

\newcommand{\mZ}{{\bf Z}}

\theoremstyle{plain}
\newtheorem{thm}{Theorem}[]

\newtheorem{lem}[thm]{Lemma}

\newtheorem{prop}[thm]{Proposition}
\newtheorem{defn}[thm]{Definition}
\newtheorem{rem}{Remark}[]
\newtheorem{cor}[thm]{Corollary}
\theoremstyle{remark}
\newtheorem{exa}{Example}[]

\newcommand{\blue}[1]{ {\color{black}#1}}
\newcommand{\red}[1]{{\color{black}#1}} 

\newcommand\encircle[1]{%
  \tikz[baseline=(X.base)] 
    \node (X) [draw, shape=circle, inner sep=0] {\strut #1};}

\newtheorem*{assumption*}{\assumptionnumber}
\providecommand{\assumptionnumber}{}
\makeatletter
\newenvironment{assumption}[2]
 {%
  \renewcommand{\assumptionnumber}{\textbf{Assumption #1${#2}$}}%
  \begin{assumption*}%
  \protected@edef\@currentlabel{#1${#2}$}%
 }
 {%
  \end{assumption*}
 }
\makeatother

\providecommand{\kernel}[1]{{\rm Null}\left( #1\right)}
\providecommand{\range}[1]{{\rm Range}\left( #1\right)}

\usepackage[colorinlistoftodos,bordercolor=orange,backgroundcolor=orange!20,linecolor=orange,textsize=scriptsize]{todonotes}
\usepackage{hyperref}

\newcommand{\ac}{\alpha}

\DeclareMathOperator{\Var}{\mathbb{V}}
\usepackage{cancel}
\usepackage{array}
\newcolumntype{P}[1]{>{\centering\arraybackslash}p{#1}}
\newcolumntype{M}[1]{>{\centering\arraybackslash}m{#1}}

\newcommand{\algG}{14} 
\newcommand{\algE}{16} 
\newcommand{\algB}{15} 
\newcommand{\algN}{17} 

\newcommand{\ones}{{\bf 1}}

\title{ Randomized Iterative Methods for Linear Systems: Momentum, Inexactness and Gossip}
\author{Nicolas Loizou}
\date{2019}

\usepackage[phd]{edmaths}

\begin{document}
\maketitle

\declaration

\dedication{To my parents, \\ Christaki and Agathi.}

\begin{abstract}
In the era of big data, one of the key challenges is the development of novel optimization algorithms that can accommodate vast amounts of data while at the same time satisfying constraints and limitations of the problem under study. The need to solve optimization problems is ubiquitous in essentially all quantitative areas of human endeavor, including industry and science. In the last decade there has been a surge in the demand from practitioners, in fields such as machine learning, computer vision, artificial intelligence, signal processing and data science, for new methods able to cope with these new large scale problems. 

In this thesis we are focusing on the design, complexity analysis and efficient implementations of such algorithms. In particular, we are interested in the development of randomized first order iterative methods for solving large scale linear systems, stochastic quadratic optimization problems and the distributed average consensus problem. 

In Chapter~\ref{ChapterMomentum}, we study several classes of stochastic optimization algorithms enriched with  {\em heavy ball momentum}. Among the methods studied are: stochastic gradient descent, stochastic Newton, stochastic proximal point  and stochastic dual subspace ascent. This is the first time momentum variants of several of these methods are studied. We choose to perform our analysis in a setting in which all of the above methods are equivalent: \blue{convex quadratic problems.} We prove global non-asymptotic linear convergence rates for all methods and various measures of success, including primal function values, primal iterates, and dual function values. We also show that the primal iterates converge at an accelerated linear rate \red{in a somewhat weaker sense}. This is the first time a linear rate is shown for the stochastic heavy ball method (i.e., stochastic gradient descent method with momentum). Under somewhat weaker conditions,  we establish a sublinear convergence rate for Ces\`{a}ro averages of primal iterates.  Moreover, we propose a novel concept, which we call {\em stochastic momentum}, aimed at decreasing the cost of performing the momentum step. We prove linear convergence of several stochastic methods with stochastic momentum, and show that in some sparse data regimes and for sufficiently small momentum parameters, these methods enjoy better overall complexity than methods with deterministic momentum. Finally, we perform extensive numerical testing on artificial and real datasets. 

In Chapter~\ref{ChapterInexact}, we present a convergence rate analysis of \emph{inexact} variants of stochastic gradient descent, stochastic Newton, stochastic proximal point and stochastic subspace ascent. A common feature of these methods is that in their update rule a certain sub-problem needs to be solved exactly. We relax this requirement by allowing for the sub-problem to be solved inexactly. In particular, we propose and analyze inexact randomized iterative methods for solving three closely related problems: a convex stochastic quadratic optimization problem, a best approximation problem and its dual -- a concave quadratic maximization problem. We provide iteration complexity results under several assumptions on the inexactness error. Inexact variants of many popular and some more exotic methods, including randomized block Kaczmarz, randomized Gaussian Kaczmarz and randomized block coordinate descent, can be cast as special cases. Finally, we present numerical experiments which demonstrate the benefits of allowing inexactness.

When the data describing a given optimization problem is big enough, it becomes impossible to store it on a single machine. In such situations, it is usually preferable to distribute the data among the nodes of a cluster or a supercomputer.  In one such setting the nodes cooperate to minimize  the sum (or average) of private functions (convex or non-convex) stored at the nodes. Among the most popular protocols for solving this problem in a decentralized fashion (communication is allowed only between neighbors) are randomized gossip algorithms.

In Chapter~\ref{ChapterGossip} we propose a new approach for the design and analysis of randomized gossip algorithms which can be used to solve the distributed average consensus problem, a fundamental problem in distributed computing, where each node of a network initially holds a number or vector, and the aim is to calculate the average of these objects by communicating only with its neighbors (connected nodes). The new approach consists in establishing new connections to recent literature on randomized iterative methods for solving large-scale linear systems. Our general framework recovers a comprehensive array of well-known gossip protocols as special cases and allow for the development of block and arbitrary sampling variants of all of these methods. In addition, we present novel and provably accelerated randomized gossip protocols where in each step all nodes of the network update their values using their own information but only a subset of them exchange messages. The accelerated protocols are the first randomized gossip algorithms that converge to consensus with a provably accelerated linear rate. The theoretical results are validated via computational testing on typical wireless sensor network topologies.

Finally, in Chapter~\ref{ChapterPrivacy}, we move towards a different direction and present the first randomized gossip algorithms for solving the average consensus problem while at the same time protecting the private values stored at the nodes as these may be sensitive.
In particular, we develop and analyze three privacy preserving variants of the randomized pairwise gossip algorithm (``randomly pick an edge of the network and then replace the values stored at vertices of this edge by their average'') first proposed by Boyd et al.~\cite{boyd2006randomized} for solving the average consensus problem. The randomized methods we propose are all dual in nature.  That is, they are designed to solve the dual of the best approximation optimization formulation of the average consensus problem.
We call our three privacy preservation techniques ``Binary Oracle'', ``$\epsilon$-Gap Oracle'' and ``Controlled Noise Insertion''. We give iteration complexity bounds for the proposed privacy preserving randomized gossip protocols and perform extensive numerical experiments. 

\end{abstract}


\chapter*{Acknowledgments}
I would like to express my sincere gratitude to my supervisor, Prof. Peter Richt\'{a}rik, for his guidance through every facet of the research world. Thank you for being such a great advisor, mentor and teacher, and for being an inspiring role model that I live up to both in my academic and personal life. Thanks for your continuous feedback, your constant encouragement, your great suggestions for writing and presentation, for offering career advice and for all the nice moments. I couldn't ask for a better advisor, mentor and a friend during my PhD studies.

I would also like to thank my second supervisor and mentor Dr. Lukasz Szpruch for various discussions about research and differences between related fields. I am also grateful to my examination committee, Prof. Coralia Cartis and Prof. Miguel F. Anjos for their suggestions for improvement and for their time.

For stimulating discussions and fun we had together during last four years, I thank other past and current members of our research group: Robert Mansel Gower, Jakub Kone\v{c}n\'{y}, Dominik Csiba, Filip Hanzely, Konstantin Mishchenko, Samuel Horv\'{a}th, Aritra Dutta, El Houcine Bergou, Xun Qian, Adil Salim, Dmitry Kovalev, Elnur Gasanov, Alibek Sailanbayev. I would also like to thank my office mates Matt Booth, Juliet Cooke and Rodrigo Garcia Nava and the rest of my friends from Operational Research and Optimization group, Chrystalla Pavlou, Spyros Pougkakiotis, Minerva Martin del Campo, Xavier Cabezas Garcia, Ivet Galabova, Saranthorn Phusingha, Marion Lemery and Wenyi Qin for all the nice moments.

I am indebted to the University of Edinburgh and to the Principal's Career Development PhD Scholarship  for funding my PhD studies. I would like to thank 
the school of Mathematics of the University of Edinburgh, for providing me with a wonderful work environment and for funding several trips during my PhD studies.
I am also extremely grateful to Gill Law, Iain Dornan and Tatiana Chepelina who were always very helpful in overcoming every formal or administrative problems I encountered.
I am most grateful to the Dr. Laura Wisewell Travel scholarship for funding my conference travels expenses in 2016 and 2017 and to Prof. Peter Richt\'{a}rik for funding my conference travels expenses in 2018 and 2019. The chance to participate in international conferences, with all the experts in my area in one place, was invaluable.
I would also like to thank the A.G. Leventis Foundation for the financial support during both my MSc and PhD studies.

I appreciate the advice, discussions and fun I had with many amazing people during my internship at Facebook AI Research in Montreal. More specifically, I would like to thank Dr. Mike Rabbat, Dr. Nicolas Ballas and Mahnoud Assran for the excellent collaboration and Prof. Joelle Pineau, Prof. Pascal Vincent, Dr. Adriana Romero, Dr. Michal Drozdzal for the great research environment. A special thanks should go to Dr. Mike Rabbat for the immerse trust and support I received during my internship. Mike has been an inspiration to me, a great enthusiastic collaborator, and warm person in general. 

For willingness to help and provide recommendation, reference, or connection at various stages of my study, I would like to thank Dr. Martin Tak\'{a}\v{c}, Dr. Martin Jaggi, Dr. Panos Parpas, Dr. Wolfram Wiesemann, Dr. Lin Xiao, Dr. Kimon Fountoulakis and Dr. Anastasios Kyrillidis. In addition, I am extremely thankful to Prof. Apostolos Burnetas and Prof. Apostolos Giannopoulos for introducing me to the areas of operational research and convex analysis, respectively.
 
This journey would not have been possible without the support of my parents, Christaki and Agathi. Thank you for encouraging me in all of my pursuits and inspiring me to follow my dreams. I dedicate this thesis to you.

Finally, I would like to thank the two people that were the most important to me during my time in Edinburgh:

My brother George, thank you for all the funny moments and the joyful experiences. You have been a constant support throughout my PhD studies and throughout my life.  You have made my time in Edinburgh one of a kind. :) 

My girlfriend, Katerina, thank you for your endless patience and for your continuous encouragement over the last four years. Thank you for reminding me what is important in life, for traveling around the world with me, for supporting all of my decisions. Thank you for being so amazing!!! 
\tableofcontents

 

\chapter{Introduction}
\label{ChapterIntroduction}
In this thesis we study the design and analysis of new efficient randomized iterative methods for solving large scale linear systems, stochastic quadratic optimization problems, the best approximation problem and quadratic optimization problems. A large part of the thesis is also devoted to the development of efficient methods for obtaining average consensus on large scale networks. As we will explain later in more detail, some of our proposed algorithms for solving the average consensus problem are carefully constructed special cases of methods for solving linear systems. All methods presented in the thesis (except two algorithms in the last chapter that converge with a sublinear rate) converge with global linear convergence rates, which means that they achieve an approximate solution of the problem fast.

In this introductory chapter we present the setting shared throughout the thesis and explain the relationships between the four problems mentioned above.~We describe some baseline methods for solving these problems and present their convergence rates. Finally, we give a summary of the main contributions of each chapter. 

\paragraph{Organization of thesis.} The thesis is divided into two main parts. 
In the first part (Chapters~\ref{ChapterMomentum} and~\ref{ChapterInexact}) we present and analyze novel \textit{momentum} (Chapter~\ref{ChapterMomentum}) and \textit{inexact} (Chapter~\ref{ChapterInexact}) variants of several randomized iterative methods for solving three closely related problems: 
\begin{enumerate}
\item[(i)] stochastic convex quadratic minimization, 
\item[(ii)] best approximation, and 
\item[(iii)] (bounded) concave quadratic maximization.
\end{enumerate} 

In the second part (Chapters~\ref{ChapterGossip} and~\ref{ChapterPrivacy}), we focus on the design and analysis of novel randomized gossip algorithms for solving the average consensus problem. This is a fundamental problem in distributed computing with the following goal: each node of a network initially holds a number or a vector, and the aim is for every node to calculate the average of these objects in a decentralized fashion (communicating with neighbors only). The proposed decentralized algorithms are inspired by recent advances in the area of randomized numerical linear algebra and optimization.  In particular, in Chapter \ref{ChapterGossip} we propose a new framework for the design and analysis of efficient randomized gossip protocols. We show how randomized iterative methods for solving linear systems can be interpreted as gossip algorithms when applied to special systems encoding the underlying network. Using the already developed framework of Chapter \ref{ChapterGossip},  we move towards a different direction and in Chapter~\ref{ChapterPrivacy} we present the first randomized gossip algorithms for solving the average consensus problem while at the same time protecting the information about the initial private values stored at the nodes. 

Excluding some introductory results presented in this section, each chapter of the thesis is self-contained, including the objective, contributions, definitions and notation. However, to the extend that this was possible and meaningful, a unified notation has been adopted throughout the thesis.

\section{Thesis' Philosophy: Place in the Literature}
Here we start by providing a bird's-eye view of the main concepts and a simple first explanation of the context and of the problems under study.  We present how this thesis is related to different areas of research and provide important connections that allow readers with different backgrounds to easily navigate through the main contributions of this work.

\subsection{Bridge across several communities}

This thesis is related to three different areas of research: linear algebra, stochastic optimization and machine learning. 

\begin{figure}[H]
  \centering
\includegraphics[width=14cm]{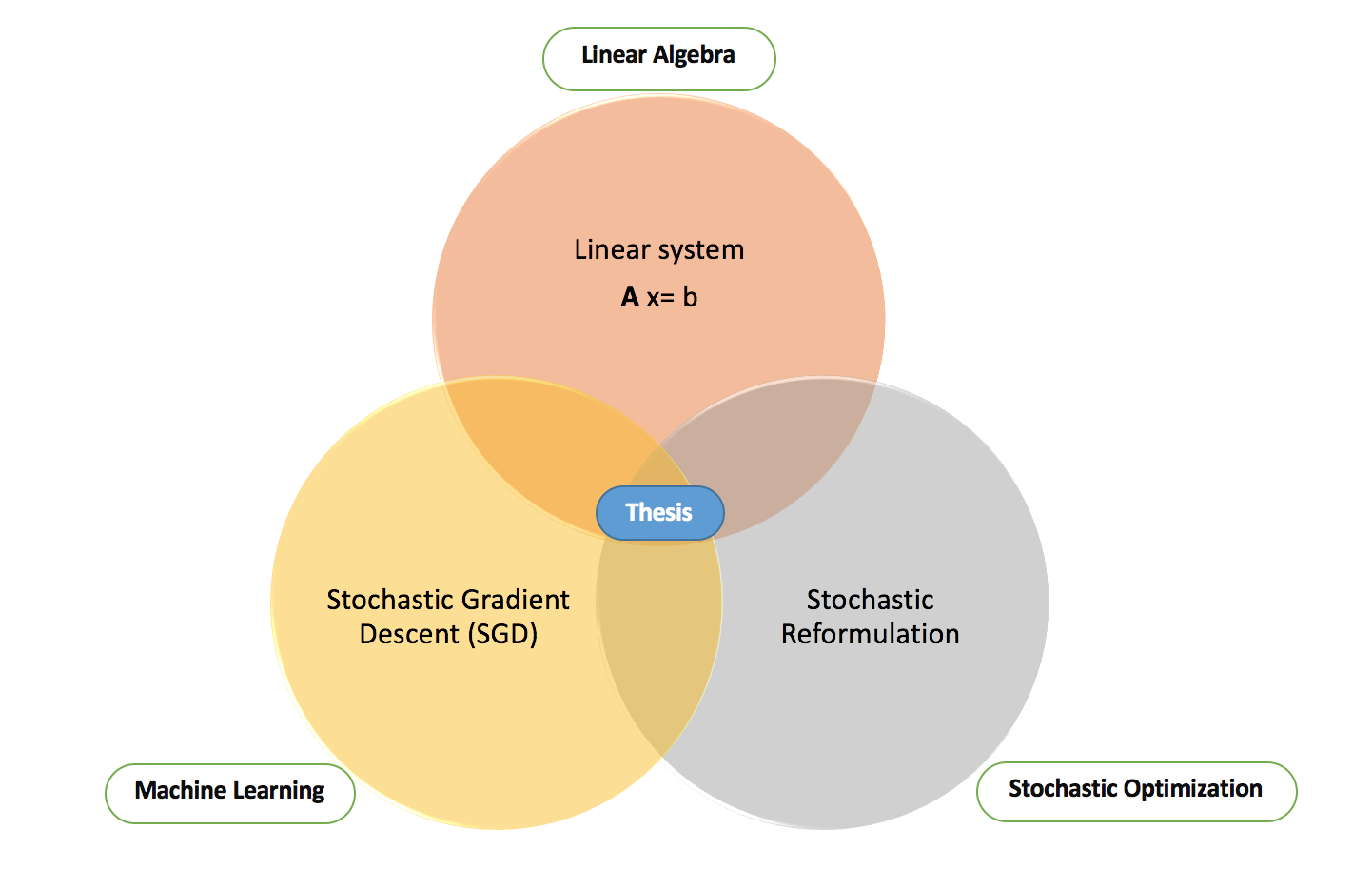}
\caption{\small Thesis' place in the literature. Linear Algebra (through methods for solving linear systems), Stochastic Optimization (through stochastic reformulations) and Machine Learning (through stochastic gradient descent) are the three main areas of research related to this thesis.}
\label{ConnectionsCyclesIntro}
\end{figure}

\noindent \textbf{Linear Algebra}

Linear systems form the backbone of most numerical codes used in industry and academia.
Solving large linear systems is a central problem in numerical linear algebra and plays an important role in computer science, mathematical computing, optimization, signal processing, engineering and many other fields.

In this thesis we are concerned with the problem of solving a \emph{consistent} linear system.
In particular,  given a matrix $\mA\in \R^{m\times n}$ and a vector $b \in \R^m$, we are interested to solve the problem:
\begin{equation} 
\label{LinearSystem_IntroThesis}
\mA x = b.
\end{equation}

\textbf{Main Assumption: Consistency.} Throughout the thesis we assume that the linear system \eqref{LinearSystem_IntroThesis} has a solution $x^* \in \R^n$ (not necessarily unique) that satisfies  $\mA x^* =b$. That is, the linear system is consistent, i.e., $\cL \eqdef \{x : \bA x=b \} \neq \emptyset$. We make no extra assumption on the form, positive definiteness, rank or any other property of matrix $\bA$. Thus, all methods proposed in this thesis converge under virtually no additional assumptions on the system beyond consistency. However, our methods are particularly well suited for the case of large over-determined linear systems. That is, to the case when the number of linear equations (rows) of the matrix is much larger than number of columns (variables) ($m \gg n$).\\

\noindent \textbf{Stochastic Optimization}

This thesis is also related to the stochastic optimization literature, through a recently proposed  stochastic optimization reformulation of linear systems \cite{ASDA}. 

It is well known that the linear system \eqref{LinearSystem_IntroThesis} can be expressed as an optimization problem as follows \cite{needell2014stochastic}:
\begin{equation} 
\label{StochProblem_Explanation}
\min_{x\in \R^n} \frac{1}{2}\| \bA x-b\|^2 = \frac{1}{2} \sum_{i=1}^m \left(\bA_{i:}^\top x- b_i \right)^2,
\end{equation}
where $\bA_{i:}$ denotes the $i^{th}$ row of matrix $\bA$. 
Note that if we denote with $\cF$ the solution set of problem \eqref{StochProblem_Explanation}, then $\cF=\cL$, where $\cL$ is the solution set of the consistent linear system \eqref{LinearSystem_IntroThesis}.

The above approach of reformulating a linear system to an optimization problem is without doubt one of the most popular. However as we will later explain in more detail, it is not the only one. For example, one may instead consider the more general formulation
\begin{equation} 
\label{cnaidalks}
\min_{x\in \R^n} \frac{1}{2}\| \bA x-b\|_{\bV}^2=\frac{1}{2}\left( \bA x-b\right)^\top \bV \left( \bA x-b \right) , 
\end{equation}
where $\bV \in \R^{m \times m}$ is a symmetric and positive definite matrix.
In \cite{ASDA}, Richt\'{a}rik and Tak\'{a}\v{c}  proposed a stochastic optimization reformulation of linear systems similar to \eqref{cnaidalks}. In particular, they consider \eqref{cnaidalks} with $\bV=\Exp_{\bS \sim \cD}[\bH]$ and allow  $\bV$ to be positive semi-definite. The expectation is over random matrices $\mS$ ($\bH$ is matrix expression involving random matrix $\bS$) that depend on an arbitrary user-defined distribution $\cD$ and the matrix $\bA$ of the linear system \eqref{LinearSystem_IntroThesis}.  
Under a certain assumption on $\cD$, for which the term {\em exactness} was coined in \cite{ASDA}, the solution set of the stochastic optimization reformulation is identical to the solution set of the linear system.  In \cite{ASDA}, the authors provide necessary and sufficient conditions for exactness. Later in Sections~\ref{naksla} and \ref{SGDSection_Intro} we describe exactness and comment on its importance in more detail.

In this thesis we design and analyze randomized iterative methods (stochastic optimization algorithms) for solving the stochastic convex quadratic minimization reformulation proposed in \cite{ASDA}.\\

\noindent \textbf{Machine Learning}

Stochastic optimization problems are at the heart of many machine learning and statistical techniques used in data science. Machine learning practitioners, as part of their data analysis,  are often interested in minimizing function $f$ which in full generality takes the form:
\begin{equation}
\label{ncalkxao}
\min_{x\in \R^n}  \big[f(x) = \Exp_{i \sim \cD} {f_{i}(x)} \big],
\end{equation}
where $\Exp_{i \sim \cD}$ denotes the expectation over an arbitrary distribution $\cD$. 

If we further assume that the distribution $\cD$ is uniform over $m$ functions $f_{i}$, the stochastic optimization problem is simplified to the finite-sum structure problem:
\begin{equation}
\label{ncasklda}
\min_{x\in \R^n} f(x)= \frac{1}{m}\sum_{i=1}^m {f_{i}(x)}.
\end{equation}
Problem \eqref{ncasklda} is referred to as Empirical Risk Minimization (ERM), and is one of the key optimization problems arising in large variety of models, ranging from simple linear regressions to deep learning.  For example, note that problem \eqref{StochProblem_Explanation} is also a special case of ERM \eqref{ncasklda} when functions $f_i: \R^n \rightarrow \R$ are chosen to be $f_i(x) \eqdef  \frac{m}{2}\left(\bA_{i:}^\top x- b_i \right)^2$. 

A trivial benchmark for solving problem \eqref{ncasklda} in the case of differentiable function is Gradient Descent (GD). That is,
$$x^{k+1} = x^k - \omega^k \nabla f(x^k)=x^k - \omega^k\frac{1}{m}\sum_{i=1}^m \nabla {f_{i}(x)},$$
where $\omega^k$ is the stepsize parameter (learning rate). 
However in modern machine learning applications the number $m$ of component functions $f_i$ can be very large ($m \gg n$). As a result, computing the full gradient in each iteration is prohibitively expensive and GD becomes impractical for most state-of-the-art applications.

To avoid such issues, machine learning practitioners use Stochastic Gradient Descent (SGD), a randomized variant of GD, first proposed by Robbins and Monro \cite{robbins1951stochastic} in 1951 and which has enjoyed a lot of success ever since. For solving \eqref{ncasklda},  SGD first uniformly at random samples function $f_i$ (where $i \in \{1,\dots,m\}$) and then performs the iteration:
$$x^{k+1} = x^k - \omega^k \nabla f_i(x^k),$$
where $\omega^k$ is the stepsize parameter (learning rate).  
SGD has become the workhorse for training supervised machine learning problems which have the generic form  \eqref{ncasklda} and many papers devoted to the understanding of its convergence behavior in different applications and under different assumptions on the functions $f_i$ \cite{nemirovskii1983problem, nemirovski2009robust, HardtRechtSinger-stability_of_SGD, Pegasos, Shamir013,pmlr-v80-nguyen18c, gower2019sgd, vaswani2018fast}.

This thesis is closely related to machine learning literature and the papers devoted to the analysis of SGD and its variants. In particular, besides other methods, we focus on analyzing SGD and two of its most popular variants: SGD with momentum (Chapter~\ref{ChapterMomentum}) and Inexact SGD (Chapter~\ref{ChapterInexact}) for solving the stochastic optimization reformulation of linear systems proposed in \cite{ASDA}.

\subsection{Roadmap}
\begin{figure}[t!]
  \centering
\includegraphics[width=14cm]{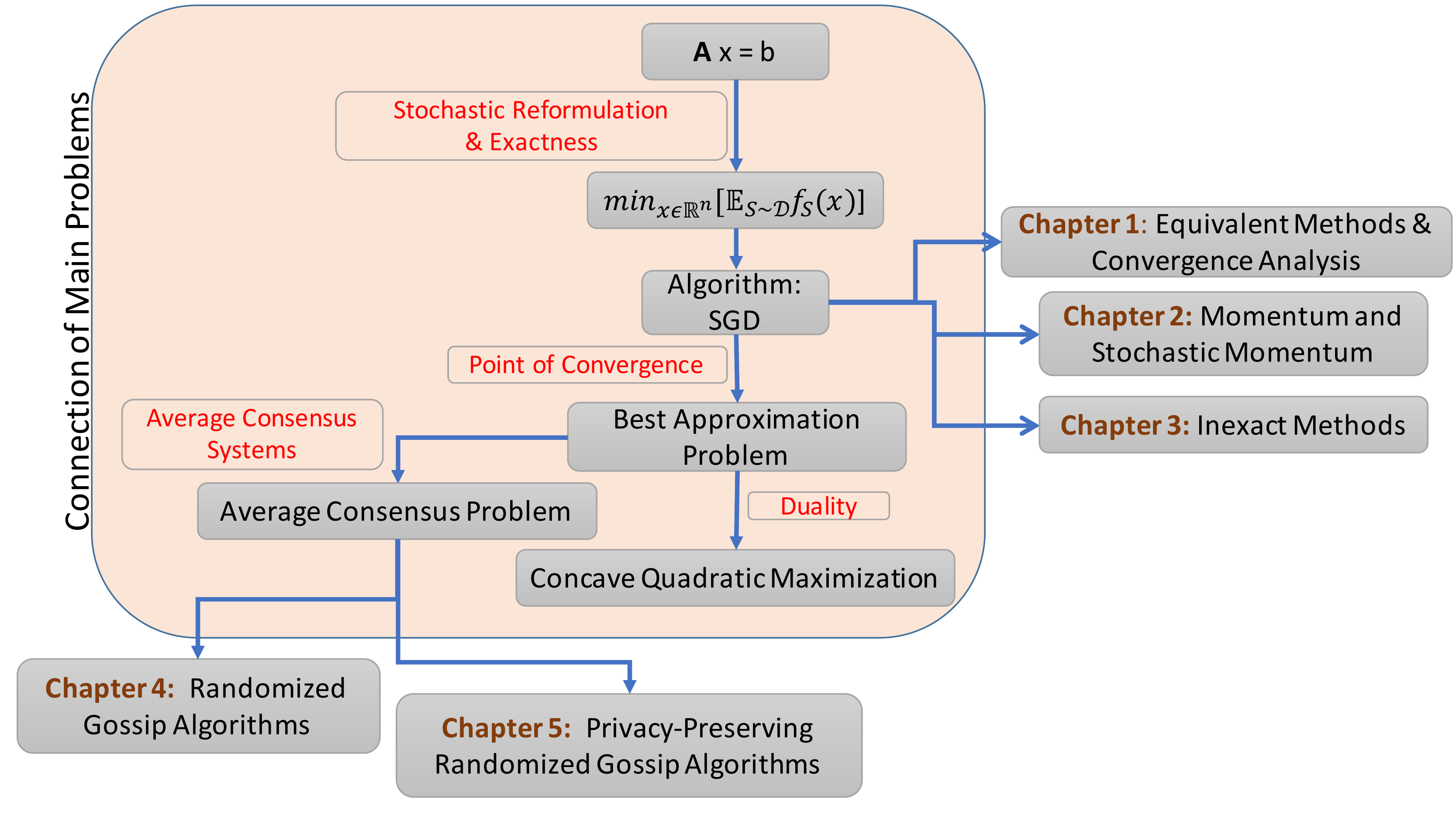}
\caption{\small Roadmap of the Thesis. Relationships between the four main problems under study: (i) linear system, (ii) stochastic convex quadratic minimization, (iii) best approximation, (iv) (bounded) concave quadratic maximization. Connection to the average consensus problem and the main chapters of the thesis.}
\label{HierarchyThesis}
\end{figure}

In this subsection,  by following the flowchart of Figure~\ref{HierarchyThesis} we present the hierarchy of the main problems under study, explain the relationships between them and provide a brief summary of the chapters of the thesis. More details will be provided in the remaining sections of the Introduction.

In this thesis, we are studying the problem of solving large-dimensional consistent \textbf{linear systems} of the form $\bA x=b$. In particular, we are adopting the stochastic optimization reformulation of linear systems first proposed in \cite{ASDA}. As we have already briefly mentioned, under a certain assumption (\emph{exactness}) on the randomness of the stochastic reformulation, the solution set of the \textbf{stochastic convex quadratic minimization problem} is equal to the solution set of the original linear system. Hence, solving the stochastic optimization problem $\Exp_{\bS \sim \cD} {f_{\mS}(x)}$ is equivalent to solving the original linear system.

For solving the stochastic convex quadratic minimization problem one can use \emph{stochastic gradient descent} (SGD), a popular stochastic optimization algorithm, particularly useful in machine learning applications (large scale setting). In Section~\ref{SGDSection_Intro} we explain how other stochastic optimization methods, like \emph{stochastic Newton }(SN) method and \emph{stochastic proximal point} (SPP) method, are identical to SGD for solving this particular problem and provide a simple analysis for their linear (exponential) convergence. 

As it turns out, SGD and its equivalent methods converge to one particular minimizer of the stochastic optimization problem: the projection of their starting point $x^0$ onto the solution set of the linear system \eqref{LinearSystem_IntroThesis}. This leads to the \textbf{best approximation problem}, which is the problem of projecting a given vector onto the solution space of the linear system. The best approximation problem is popular in numerical linear algebra and is normally solved using sketching techniques. We show how the sketch-and-project method proposed in \cite{gower2015randomized} for solving the best approximation problem has also identical updates to SGD. 

The dual of the best approximation problem is a \textbf{bounded unconstrained concave quadratic maximization problem}. In this thesis, we are also interested in the development and convergence analysis of efficient, dual in nature, algorithms for directly solving the dual of the best approximation problem. The baseline method for solving the dual of the best approximation problem is \emph{stochastic dual subspace accent} (SDSA) \cite{gower2015stochastic}. As we will explain later, the random iterates of SGD, SN and SPP arise as affine images of the random iterates produced by SDSA.

In Chapters~\ref{ChapterMomentum} and~\ref{ChapterInexact} we study novel \textit{momentum} and \textit{inexact} variants of several randomized iterative methods for solving the above problems.  Among the methods studied are: stochastic gradient descent, stochastic Newton, stochastic proximal point  and stochastic dual subspace ascent. 

As we can also see in Figure~\ref{HierarchyThesis}, a large part of the thesis will be devoted to the development of efficient methods for solving the \textbf{average consensus (AC) problem}.  In particular, we will explain how the AC problem can be expressed as a best approximation problem once we choose special linear systems encoding the underlying network (average consensus systems). In Chapter~\ref{ChapterGossip} we show how classical randomized iterative methods for solving the best approximation problem can be interpreted as gossip algorithms and explain in detail their decentralized nature. In Chapter~\ref{ChapterPrivacy}, we present the first privacy-preserving randomized gossip algorithms that solve the AC problem while at the same time protect the private values stored at the nodes as these may be sensitive.

\section{Stochastic Optimization Reformulation of Linear Systems} 
\label{naksla}
The starting point of this thesis is a general framework for studying consistent linear systems via carefully designed {\em stochastic reformulations} recently proposed by Richt\'{a}rik and Tak\'{a}\v{c} \cite{ASDA}. In particular, given the consistent linear system \eqref{LinearSystem_IntroThesis},
the authors provide four reformulations in the form of a stochastic optimization problem,  stochastic linear system,  stochastic fixed point problem and a stochastic feasibility problem.  These reformulations are equivalent in the sense that their solution sets are identical. That is, the set of minimizers of  the stochastic optimization problem is equal to the set of solutions of the stochastic linear system and so on. Under a certain assumption on the randomness defining these reformulations, for which the term {\em exactness} was coined in \cite{ASDA}, the solution sets  of these reformulations are equal to the solution set of the linear system. 

For the sake of a simplified narrative, in this thesis we choose to focus mostly on one of the above reformulations: the stochastic convex quadratic optimization problem, which can be expressed as follows:
\begin{equation}
\label{StochReform_IntroThesis} 
\min_{x\in \R^n} f(x) \eqdef \Exp_{\bS \sim \cD} {f_{\mS}(x)}.
\end{equation}
Here the expectation is over random matrices $\mS$ drawn from an arbitrary, user defined, distribution $\cD$ and $f_{\mS}$ is a stochastic convex quadratic function of a  least-squares type, defined as 
 \begin{equation}
 \label{f_s_IntroThesis}
 f_{\mS}(x) \eqdef \frac{1}{2}\|\mA x - b\|_{\mH}^2 = \frac{1}{2}(\mA x - b)^\top \mH (\mA x - b).
 \end{equation} 
Function  $f_{\mS}$ depends on the matrix $\mA\in \R^{m\times n}$ and vector $b\in \R^m$ of the linear system \eqref{LinearSystem_IntroThesis} and on a random symmetric positive semidefinite matrix  
\begin{equation}
\label{MatrixH_IntroThesis}
\mH \eqdef  \mS (\mS^\top \mA \mB^{-1} \mA^\top \mS)^\dagger \mS^\top.
\end{equation}
The $n\times n$ positive definite matrix $\mB$, in the expression of matrix $\bH$, defines the geometry of the space and throughout the thesis, gives rise to an inner product 
\begin{equation}
\label{InnerProduct_IntroThesis}
\langle x,z \rangle_\mB \eqdef \langle \mB x, z\rangle
\end{equation}
and the induced norm $\|x\|_\mB\eqdef (x^\top \mB x)^{1/2}$ on $\R^n$. By $\dagger$ we denote the Moore-Penrose pseudoinverse.  

\paragraph{On Moore-Pernose Pseudoinverse:} The Moore-Pernose pseudoinverse matrix (or simply pseudoinverse) $\bM^\dagger$ of a matrix $\bM$ was first intoduced by Moore \cite{moore1920reciprocal} and Penrose \cite{penrose1956best,penrose1955generalized} in their pioneering work. 

A computationally simple and accurate way to compute the matrix $\bM^\dagger$ is by using the singular value decomposition \cite{golub2012matrix,desoer1963note}. That is, if $\bM= \bU \bSigma \bV^\top$  is the singular value decomposition of matrix $\bM$, then $\bM^\dagger= \bU \bSigma^\dagger \bV^\top$, where the diagonal matrix $\bSigma^\dagger$  is computed by taking the reciprocal of each non-zero element on the diagonal of matrix $\bSigma$, leaving the zeros in place. That is $\bSigma^\dagger_{ii}=\frac{1}{\bSigma_{ii}}$ for all $\bSigma_{ii}>0$.   Note that by its definition Moore-Penrose pseudoinverse is uniquely defined for all matrices (not necessarily square) whose entries are real or complex numbers.

It is worth to highlight that in this thesis, we applied the Moore-Pernose pseudoinverse only on symmetric positive semidefinite matrices. In particular, if $\bM \succeq 0$ is a symmetric $m \times m$ matrix then its pseudoinverse will appear as a part of the expression $\bM^\dagger b$, where $b \in \R^m$.  Using properties of pseudoinverse this is equivalent to the least-norm solution of the least-squares problem $\min_x \|\bM x -b\|^2$ \cite{golub2012matrix, gower2015stochastic}. Hence, if the system $\bM x = b$ has a solution the following holds:
\begin{equation}
\label{moore}
\bM^\dagger b= \text{argmin} \|x\|^2 
\quad \text{subject to}  \quad \bM x = b.
\end{equation}
Let us know present some basic properties of the pseudoinverse:
\begin{itemize}
\item If matrix $\bM$ is invertible, its pseudoinverse is its inverse. That is, $\bM^\dagger=\bM^{-1}$.
\item The pseudoinverse of the pseudoinverse is the original matrix. That is , $(\bM^\dagger)^\dagger=\bM$
\item If $\bM$ is symmetric positive semidefinite matrix , then $\bM^\dagger$ is also symmetric positive semidefinite matrix.
\item There are several identities that can be used to cancel certain subexpressions or expand expressions involving pseudoinverses: $\bM \bM^\dagger\bM =\bM$, $\bM^\dagger \bM \bM^\dagger=\bM^\dagger$, $(\bM^\top \bM)^\dagger \bM^\top=\bM^\dagger$, $(\bM^\top)^\dagger = (\bM^\dagger)^\top$  and $(\bM \bM^\top)^\dagger =(\bM^\dagger)^\top \bM^\dagger $. For more useful identities see \cite{golub2012matrix}.
\end{itemize}

As we have already mentioned, problem \eqref{StochReform_IntroThesis} is constructed in such a way that the set of minimizers of $f$ is identical to the set of solutions of the given (consistent) linear system \eqref{LinearSystem_IntroThesis}.  In this sense, \eqref{StochReform_IntroThesis} can  be seen as the reformulation of the linear system \eqref{LinearSystem_IntroThesis} into a stochastic optimization problem. As argued in \cite{ASDA}, such reformulations  provide an explicit connection between the fields of linear algebra and stochastic optimization, and allow the transfer of knowledge, techniques, and algorithms from one field to another. For instance, the randomized Kaczmarz method of Strohmer and Vershynin \cite{RK} for solving \eqref{LinearSystem_IntroThesis} is equivalent to the stochastic gradient descent method applied to \eqref{StochReform_IntroThesis}, with $\cD$ corresponding to a discrete distribution over unit coordinate vectors in $\R^m$ \cite{needell2014stochastic}. However, the flexibility of being able to choose $\cD$ arbitrarily allows for numerous generalizations of the randomized Kaczmarz method \cite{ASDA}. Likewise, provably faster variants of the randomized Kaczmarz method (for instance, by utilizing importance sampling) can be designed using the connection.

Since their introduction in \cite{ASDA}, stochastic reformulations of otherwise deterministic problems have found surprising applications in various areas, and are hence an
important object of study in its own right.  For instance, using a different stochastic reformulation Gower et al. \cite{gower2019sgd} performed a tight convergence analysis of stochastic gradient descent in a more general convex setting, while \cite{gower2018stochastic} utilized ``controlled" stochastic reformulations to develop a new approach to variance reduction for finite-sum problems appearing in machine learning. Further, this approach led to the development of the first accelerated quasi-Newton matrix update rules in the literature \cite{gower2018accelerated} and to the design of efficient randomized projection methods for convex feasibility problems \cite{necoara2018randomized}; all solving open problems in the literature.

\paragraph{Closed form expressions.}

We shall often refer to matrix expressions involving the random matrix $\bS$ and the matrices $\bB$ and $\bA$. In order to keep these expressions brief throughout the thesis, it will be useful to define the $n \times n$ matrix:
\begin{equation}
\label{ZETA_Intro}
\mZ \eqdef \mA^\top \mH \mA \overset{\eqref{MatrixH_IntroThesis}}{=} {\bA}^\top \mS(\mS^\top\mA \mB^{-1} \mA^\top \mS)^{\dagger} \mS^{\top} \bA.
\end{equation}

Using matrix $\bZ$ we can easily express important quantities related to the problems under study. For example, the stochastic functions $f_\mS$ defined in \eqref{f_s_IntroThesis} can be also expressed as 
\begin{equation}
\label{f_sZeta_IntroThesis}
f_{\mS}(x) \overset{\eqref{f_s_IntroThesis}}{=} \frac{1}{2}(\mA x - b)^\top \mH (\mA x - b)= \frac{1}{2} (x-x^*)^\top \bZ (x-x^*),
\end{equation}
where $x^* \in \cL$. In addition, the gradient and the Hessian of $f_\mS$ with respect to the $\mB$ inner product \eqref{InnerProduct_IntroThesis} are equal to
\begin{equation}
\label{Gradf_S_IntroThesis}
\nabla f_\mS(x) \overset{\eqref{f_s_IntroThesis}}{=} \mB^{-1} \mA^\top \mH (\mA x - b)  = \mB^{-1} \mA^\top \mH \mA (x -x^*)  \overset{\eqref{ZETA_Intro}}{=} \mB^{-1} \mZ (x -x^*),
\end{equation}
where $x^* \in \cL$ and $\nabla^2 f_{\mS}(x)=\bB^{-1}\bZ$ \cite{ASDA}. 

Using the above expressions, the gradient and the Hessian of the objective function $f$ of problem \eqref{StochReform_IntroThesis} are given by 
$$\nabla f(x)= \Exp_{\bS \sim \cD} {[\nabla f_{\mS}(x)]} =\mB^{-1}\E{\mZ}(x-x^*),$$ and
$$\nabla^2 f = \Exp_{\bS \sim \cD}[\nabla^2 f_{\mS}(x)]= \mB^{-1}\E{\mZ},$$ respectively.

\paragraph{Projections.} Let $\Upsilon \subseteq \R^n$ be a closed convex set. 
Throughout the thesis, with $\Pi_{\Upsilon, \bB}$ we denote the projection operator onto $\Upsilon$, in the $\bB$-norm. That is, $\Pi_{\Upsilon, \bB}(x) \eqdef \arg\min_{x' \in \Upsilon} \|x'-x\|_{\bB}$.
In particular, we are interested in the projection onto $\cL$. An explicit formula for the projection onto $\cL$ is given by
\begin{equation}
\label{Projection_IntroThesis}
\Pi_{\cL,\bB}(x)\eqdef \arg\min_{x' \in \cL} \|x'-x\|_{\bB} =x-\bB^{-1}\bA^\top (\bA\bB^{-1}\bA ^ \top )^\dagger (\bA x-b).
\end{equation}
A formula for the projection onto the \textit{sketched system }$\cL_\mS \eqdef \{x : \bS^\top \bA x=\bS^\top b \}$ is obtained by simply replacing matrix $\mA$ and vector $b$ in \eqref{Projection_IntroThesis} with the matrix $\mS^\top \mA$ and vector $\mS^\top b$, respectively. In this case we write $\Pi_{\cL_{\bS},\bB}(x)$. 

\paragraph{On complexity results.}
The complexity of the linearly convergent methods presented in this thesis is described by the spectrum of the following key matrix: 
\begin{equation}
\label{MatrixW_IntroThesis}
\mW \eqdef \mB^{-1/2} \E{\mZ}\mB^{-1/2}.
\end{equation} 
Matrix $\bW$ has the same spectrum as the Hessian matrix 
$\nabla^2 f = \bB^{-1}\E{\mZ}$ and at the same time is symmetric and positive semidefinite (with respect to the standard inner product). Note that $\nabla^2 f $ is a not symmetric matrix (although it is self-adjoint with respect to the $\bB$-inner product).

Let $\mW = \mU \bLambda \mU^\top = \sum_{i=1}^n \lambda_i u_i u_i^\top $ be the eigenvalue decomposition of $\mW$, where  $\mU = [u_1,\dots,u_n]$ is an orthonormal matrix composed of eigenvectors, and $\bLambda=\textbf{Diag}(\lambda_1, \lambda_2, \dots ,\lambda_n)$ is the diagonal matrix of eigenvalues with $\lambda_1 \geq \lambda_2 \geq \cdots \geq \lambda_{n}$. In this thesis, by $\lambda_{\min}^+$ we will denote the smallest nonzero eigenvalue, and by $\lambda_{\max}  = \lambda_1$ the largest eigenvalue of matrix $\bW$. It was shown in \cite{ASDA} that $0\leq \lambda_i \leq 1$ for all $i\in [n]$.

\paragraph{Main Assumption: Exactness.} Note that in view of \eqref{f_sZeta_IntroThesis}, $f_\mS(x) = 0$ whenever $x\in \cL$. However, $f_\mS$ can be zero also for points $x$ outside of $\cL$. Clearly, $f$ is nonnegative, and $f(x)=0$ for $x\in \cL$. However, without further assumptions, the set of minimizers of $f$ can be larger than $\cL$. The exactness assumption mentioned above ensures that this does not happen. For necessary and sufficient conditions for exactness, we refer the reader to \cite{ASDA}. One of a number of equivalent characterizations of exactness is the condition: 
\begin{equation}
\label{exactnessCondition}
{\rm Null}(\Exp_{\bS \sim \cD} [\bZ])={\rm Null}(\bA).
\end{equation}
For this thesis it suffices to remark that a sufficient condition for exactness is to require $\E{\mH}$ to be positive definite. This is easy to see by observing that $f(x) = \E{f_\mS(x)} =\frac{1}{2}\|\mA x - b\|^2_{\E{\mH}}. $ In other words,  if $\cX=\text{argmin} f(x)$ is the solution set of the stochastic optimization problem \eqref{StochReform_IntroThesis} and $\cL$ the solution set of the linear system \eqref{LinearSystem_IntroThesis}, then the notion of exactness is captured by: $$\cX=\cL$$

\section{Stochastic Gradient Descent (SGD) and Equivalent Iterative Methods}
\label{SGDSection_Intro}
Problem  \eqref{StochReform_IntroThesis}  has several peculiar characteristics which are of key importance to this thesis. For instance, the Hessian of $f_{\mS}$ is a (random) projection matrix, which can be used to show that $f_{\mS}(x) = \frac{1}{2}\|\nabla f_{\mS}(x)\|_\mB^2=\frac{1}{2} \nabla f_{\mS}(x)^\top \mB \nabla f_{\mS}(x)$ (see equation \eqref{normbound} in Lemma~\ref{cnaoisna}). Moreover, as we have already mentioned the Hessian of $f$ has all eigenvalues bounded by 1, and so on. These characteristics can be used to show that several otherwise {\em distinct}  stochastic algorithms for solving the stochastic optimization problem \eqref{StochReform_IntroThesis} are {\em identical}.

In particular, the following optimization methods for solving \eqref{StochReform_IntroThesis} are identical 
\begin{itemize}
\item Stochastic Gradient Descent (SGD):
\begin{equation}
\label{SGD_IntroThesis}
x^{k+1} = x^k - \omega \nabla f_{\mS_k}(x^k),
\end{equation}
\item Stochastic Newton Method (SN)\footnote{In this method we take the $\mB$-pseudoinverse of the Hessian of $f_{\mS_k}$ instead of the classical inverse, as the inverse does not exist. When $\mB=\mI$, the $\mB$ pseudoinverse specializes to the standard Moore-Penrose pseudoinverse.}:
\begin{equation}
\label{SNM_IntroThesis}
x^{k+1} = x^k - \omega (\nabla^2 f_{\mS_k}(x^k))^{\dagger_\mB} \nabla f_{\mS_k}(x^k),
\end{equation}
\item Stochastic Proximal Point Method (SPP)\footnote{In this case, the equivalence only works for $0<\omega\leq 1$.}:
\begin{equation}
\label{SPPM_IntroThesis}
x^{k+1} = \arg\min_{x\in \R^n} \left\{ f_{\mS_k}(x) + \frac{1-\omega}{2\omega}\|x-x^k\|_{\mB}^2\right\}.
\end{equation}
\end{itemize}
In all methods $\omega>0$, is a fixed stepsize and  $\mS_k$ is sampled afresh in each iteration from distribution $\cD$. 

Note that the equivalence of these methods for solving problem \eqref{StochReform_IntroThesis} is useful for the purposes of the thesis as it allows us to study their variants  {\em with momentum} (Chapter \ref{ChapterMomentum}) and their \emph{inexact} variants (Chapter \ref{ChapterInexact}) by studying a single algorithm only. 

Using the closed form expression \eqref{Gradf_S_IntroThesis} of the gradient of functions $f_{\bS}$, the update rules of the equivalent algorithms \eqref{SGD_IntroThesis},\eqref{SNM_IntroThesis} and \eqref{SPPM_IntroThesis} can be also written as:
\begin{equation}
\label{Basic_IntroThesis}
 x^{k+1}=x^k-\omega \mB^{-1}\mA^\top \mS_{k} (\mS_{k}^\top \mA \mB^{-1} \mA^\top \mS_{k})^{\dagger} \mS_{k}^\top(\mA x^k-b)
\end{equation}

Following \cite{ASDA}, we name the algorithmic update of equation~\eqref{Basic_IntroThesis}, \textit{basic method} and we use this in several parts of this thesis to simultaneously refer to the above equivalent update rules.

By choosing appropriately the two main parameters of the basic method,  the matrix $\bB$ and distribution $\cD$ of the random matrices $\bS$, we can recover a comprehensive array of well known algorithms for solving linear systems as special cases, such as the randomized Kaczmarz method, randomized Gauss Seidel (randomized coordinate descent) and their block variants. In addition, it is worth to notice that the basic method allows for a much wider selection of these two parameters, which means that it is possible to obtain a number of new specific and possibly more exotic algorithms as special cases. Hence, by having a convergence analysis for the general method \eqref{Basic_IntroThesis} we can easily obtain the convergence rates of all these special cases by choosing carefully the combinations of the two main parameters.

\begin{exa}
As a special case of the general framework,  let us choose $\bB=\bI$ and $\bS_k=e_i$, where $i \in [m]$ is chosen in each iteration independently, with probability $p_i>0$. Here with $e_i \in \R^m$ we denote the $i^{\text{th}}$ unit coordinate vector in $ \R^m$.
 In this setup the update rule \eqref{Basic_IntroThesis} simplifies to:
\begin{equation}
\label{RK_IntroThesis}
x^{k+1}=x^k - \omega \frac{\bA_{i :} x^k -b_{i}}{\|\bA_{i :}\|^2} \bA_{i :}^ \top  .
\end{equation}
where $\bA_{i:}$ indicates the $i^{th}$ row of matrix $\bA$. This is a relaxed variant (stepsize not necessarily $\omega=1$) of the  randomized Kaczmarz method \cite{RK}.
\end{exa} 

\paragraph{On Exactness.}
An important assumption that is required for the convergence analysis of the randomized iterative methods presented in this thesis is \textit{exactness}. The exactness property is of key importance for the setting under study, and should be seen as an assumption on the distribution $\cD$ and not on matrix $\bA$. 

Clearly, an assumption on the distribution $\cD$ of the random matrices $\bS$ should be required for the convergence of \eqref{Basic_IntroThesis}. For an instance, if $\cD$ in the randomized Kaczmarz method \eqref{RK_IntroThesis},  is such that, $\bS=e_1$ with probability 1, where $e_1 \in \R^m$  be the $1^{\text{st}}$ unit coordinate vector in $ \R^m$, then the algorithm will select the same row of matrix $\bA$ in each step. For this choice of distribution it is clear that the algorithm will not converge to a solution of the linear system. The exactness assumption guarantees that this will not happen.

For necessary and sufficient conditions for exactness, we refer the reader to \cite{ASDA}. For this thesis, it suffices to remark that the exactness condition is very weak, allowing $\cD$ to be virtually any reasonable distribution of random matrices. For instance, as we have already mentioned, a sufficient condition for exactness is for the matrix $\E{\mH}$ to be positive definite \cite{gower2015stochastic}. 

A much stronger condition than exactness is $\E{\mZ}\succ 0$ which has been used for the convergence analysis of \eqref{Basic_IntroThesis} in \cite{gower2015randomized}. In this case, the matrix $\bA \in \R^{m \times n}$ of the linear system requires to have full column rank and as a result the consistent linear system has a unique solution. 

\section{Best Approximation and its Dual Problem} 
\label{BestaprooximationSection_INtro}

\paragraph{Best approximation problem.}
It was shown in \cite{ASDA} that SGD, SN and SPP converge to a very particular minimizer of $f$: the projection in the $\bB$-norm, of the starting point $x^0$ onto the solution set of the linear system \eqref{LinearSystem_IntroThesis}. That is, $x^*=\Pi_{\cL,\bB}(x^0)$. This naturally leads  to the \textit{best approximation problem}, which is the problem of projecting a given vector onto the solution space of the linear system \eqref{LinearSystem_IntroThesis}:
\begin{equation}
\label{BestApproximation_IntroThesis}
\min_{x\in \R^n} P(x) \eqdef \frac{1}{2}\|x-x^0\|_\mB^2 \quad \text{subject to} \quad \mA x = b.
\end{equation}
Note that, unlike the linear system \eqref{LinearSystem_IntroThesis}, which is allowed to have multiple solutions, the best approximation problem has always (from its construction) a unique solution. 

For solving problem \eqref{BestApproximation_IntroThesis}, the \emph{Sketch and Project Method (SPM)}: 

\begin{equation}
\label{Sket}
\begin{aligned}
& \quad \quad \quad x^{k+1}
& & \eqdef \text{argmin}_{x \in \R^n} \|x-x^k\|_{\bB}^2\\
& \text{subject to}
& & \bS_k^\top \bA x=\bS_k^\top b \;.
\end{aligned}
\end{equation}
was analyzed in \cite{gower2015randomized, gower2015stochastic}. The name ``sketch-and-project" method is justified by the iteration structure which consists of two steps: (i) draw a random matrix $\bS_k$ from distribution $\cD$  and formulate the {\em sketched} system $\cL_{\mS_k}$, (ii) {\em project} the last iterate $x^k$ onto $\cL_{\mS_k}$.  Analysis in \cite{gower2015randomized} was done under the assumption that $\mA$ has full column rank. This assumption was lifted in \cite{gower2015stochastic}, and a {\em duality} theory for the method developed. 

Using the closed form expression of projection \eqref{Projection_IntroThesis}, the iterative process of \eqref{Sket} can be equivalently written as \cite{gower2015randomized}:
\begin{equation}
\label{spstepsize1}
x^{k+1} =  \Pi_{\cL_{\mS_k}, \bB}(x^k) \overset{\eqref{Projection_IntroThesis}}{=} x^k-\bB^{-1}\bA^\top \bS_k (\bS_k^\top \bA\bB^{-1}\bA ^ \top \bS_k)^\dagger (\bS_k^\top \bA x^k- \bS_k^\top b),
\end{equation}
and for the more general case of $\omega\neq1$ the update takes the form:
\begin{equation}
\label{SPM_IntroThesis}
x^{k+1} =  \omega \Pi_{\cL_{\mS_k}, \bB}(x^k) + (1-\omega) x^k.
\end{equation}

By combining the definition of projection \eqref{SPM_IntroThesis} and the update rule of equation \eqref{Projection_IntroThesis} it can be easily observed that the sketch and project method is identical to the basic method \eqref{Basic_IntroThesis}. As a result, it is also identical to the previously mentioned algorithms, SGD \eqref{SGD_IntroThesis}, SN \eqref{SNM_IntroThesis} and SPP \eqref{SPPM_IntroThesis}.  Thus, these methods can be also interpreted as randomized projection algorithms.

\paragraph{On Sketching.} In numerical linear algebra, sketching is one of the most popular techniques used for the evaluation of an approximate solution of large dimensional linear systems $\bA x= b$ where $\bA \in \R^{m \times n}$ and $b \in \R^m$ \cite{woodruff2014sketching}. 

Let $\bS  \in \R^{m \times q}$ be a random matrix with the same number of rows as $\bA$ but far fewer columns ($n \gg q$). The goal of sketching is to design the distribution of random matrix $\bS$ such that the solutions set of the much smaller (and potential much easier to solve) sketched system $\bS^\top \bA x=\bS^\top b$ to be close to the solution set of the original large dimensional system, with high probability. That is, $\|\bS^\top \bA x-\bS^\top b\|^2 \leq (1+\epsilon)\|\bA x-b\|^2$ with high probability. Determining the random matrix $\bS$ that satisfy the above constraint can be challenging and often depends on the properties and form of matrix $\bA$. For recent advances in the area of sketching we suggest \cite{woodruff2014sketching, ghashami2016frequent, desai2016improved, mahoney2011randomized, drineas2011faster, pilanci2016iterative}.

In our setting, as we have already described above, sketching is part of our iterative process. In each iteration sketching is used in combination with a projection step in order to evaluate an exact solution of the sketched system. 

\paragraph{On Sketch and Project Methods.} Variants of the sketch-and-project method  have been recently proposed for solving several other problems.
 \cite{gower2016randomized, gower2016linearly} use sketch-and-project ideas for the development of linearly convergent randomized iterative methods for computing/estimating the inverse and pseudoinverse of a large matrix, respectively. A limited memory variant of the stochastic block BFGS method  for solving the empirical risk minimization problem arising in machine learning  was proposed by \ \cite{gower2016stochastic}. Tu et al.\ \cite{tu2017breaking} utilize the sketch-and-project framework to show that breaking block locality can accelerate block Gauss-Seidel methods. In addition, they develop an accelerated variant of the method for a specific distribution $\cD$.  An accelerated (in the sense of Nesterov) variant of the sketch and prokect method proposed in \cite{gower2018accelerated} for the more general Euclidean setting and applied to matrix inversion and quasi-Newton updates. As we have already mentioned, in \cite{ASDA},  through the development of stochastic reformulations, a stochastic gradient descent interpretation of the sketch and project method have been proposed. Similar stochastic reformulations that have the sketch and project method as special case were also proposed in the more general setting of convex optimization \cite{gower2019sgd} and in the context of variance-reduced methods \cite{gower2018stochastic}.

\paragraph{The dual problem.}
Duality is an important tool in optimization literature and plays a major role in the development and understanding of many popular randomized optimization algorithms.
In this thesis, we are also interested in the development of efficient, dual in nature, algorithms for directly solving the dual of the best approximation problem.

In particular, the Lagrangian dual of \eqref{BestApproximation_IntroThesis} is the (bounded) unconstrained concave quadratic maximization problem\footnote{Technically problem~\eqref{DualProblem_IntroThesis} is both the Lagrangian and the Fenchel dual of \eqref{BestApproximation_IntroThesis} \cite{gower2015stochastic}.}
\begin{equation}
\label{DualProblem_IntroThesis}
\max_{y\in \R^m} D(y) \eqdef (b-\bA x^0)^\top y - \frac{1}{2}\|\bA^\top y\|^2_{\bB^{-1}}.
\end{equation}
Boundedness follows from consistency. It turns out that by varying $\mA, \mB$ and $b$ (but keeping consistency of the linear system), the dual problem in fact captures {\em all} bounded unconstrained concave quadratic maximization problems.

Let us define an affine mapping from $\R^m$ to $\R^n$ as follows:
\begin{equation}
\label{Corresp_IntroThesis1}
\phi(y) \eqdef x^0 + \mB^{-1} \mA^\top y.
\end{equation}
It turns out, from Fenchel duality\footnote{For more details on Fenchel duality, see \cite{borwein2010convex}.}, that for any dual optimal $y^*$, the vector $\phi(y^*)$ must be primal optimal \cite{gower2015stochastic}. That is:
\begin{equation}
\label{Corresp_IntroThesisOptimal}
x^* = \phi(y^*) = x^0 + \mB^{-1} \mA^\top y^*.
\end{equation}

A dual variant of the basic method for solving problem \eqref{DualProblem_IntroThesis} was first proposed in \cite{gower2015stochastic}. The dual method---{\em Stochastic Dual Subspace Ascent (SDSA)}---updates the dual vectors $y^k$ as follows:
\begin{equation}
\label{SDSA_IntroThesis}
y^{k+1} = y^k + \omega  \mS_k \lambda^k,
\end{equation}
where the random matrix $\mS_k$ is sampled afresh in each iteration from distribution $\cD$, and $\lambda^k$ is chosen to maximize the dual objective $D$: $\lambda^k \in \arg\max_\lambda D(y^k + \mS_k \lambda)$.  More specifically, SDSA is defined by picking the maximizer with the smallest (standard Euclidean) norm. This leads to the formula:
\begin{equation}
\label{lambdak_IntroThesis}
\lambda^k =  \left(\mS_k^\top \bA \bB^{-1}\bA^\top \mS_k \right)^\dagger\bS_k^\top \left(b-\bA(x^0 + \bB^{-1}\bA^\top y^k) \right).
\end{equation}

Note that, SDSA proceeds by moving in random subspaces spanned by the random columns of $\mS_k$. In the special case when $\omega=1$ and  $y^0=0$,  Gower and Richt\'{a}rik \cite{gower2015stochastic} established the following relationship (affine mapping $\phi:\R^m \mapsto \R^n$) between the iterates $\{x^k\}_{k\geq0}$ produced by the primal methods \eqref{SGD_IntroThesis}, \eqref{SNM_IntroThesis},  \eqref{SPPM_IntroThesis},  \eqref{SPM_IntroThesis} (which are equivalent), and the iterates $\{y^k\}_{k\geq0}$ produced by the dual method \eqref{SDSA_IntroThesis}:
\begin{equation}
\label{Corresp_IntroThesis}
x^{k} = \phi(y^k) \overset{\eqref{Corresp_IntroThesis1}}{=} x^0 + \mB^{-1} \mA^\top y^k.
\end{equation}

In Section~\ref{IntroductoryAnalysis} we show with a simple proof how this equivalence extends beyond the $\omega=1$ case, specifically for $0<\omega<2$ (see Proposition~\ref{bcaoskla}). Later, a similar approach will be used in the derivation of the momentum and inexact variant of SDSA in Chapters~\ref{ChapterMomentum} and \ref{ChapterInexact} respectively.

An interesting property that holds between the suboptimalities of the primal methods and SDSA is that the dual suboptimality of $y$ in terms of the dual function values is equal to the primal suboptimality of $\phi(y)$ in terms of distance \cite{gower2015stochastic}. That is,
\begin{equation}
\label{identity}
D(y^*)-D(y)=\frac{1}{2}\| \phi(y^*)- \phi(y)\|^2_{\bB}.
\end{equation}
This simple-to-derive result (by combining the expression of the dual function $D(y)$ \eqref{DualProblem_IntroThesis} and the equation \eqref{Corresp_IntroThesis1}) gives for free the convergence analysis of SDSA, in terms of dual function suboptimality once the analysis of the primal methods is available (see Proposition~\ref{PropositionDualPrimal} in Section~\ref{IntroductoryAnalysis}). 

Note that SDSA update \eqref{SDSA_IntroThesis}+\eqref{lambdak_IntroThesis} depends on the same two parameters, matrix $\bB$ and distribution $\cD$, of the basic method (SGD, SN and SPP). Therefore, similar to the previous subsection, by choosing appropriately the two parameters we can recover many known algorithms as special cases of SDSA. 

\begin{exa} 
Let $\bB=\bI$ and $\bS_k=e_i$, where $i \in [m]$ is chosen in each iteration independently, with probability $p_i>0$. Here with $e_i \in \R^m$ we denote the $i^{\text{th}}$ unit coordinate vector in $ \R^m$.
 In this setup the update rule \eqref{SDSA_IntroThesis} simplifies to:
\begin{equation}
\label{RCD_DualThesis}
y^{k+1}=y^k + \omega \frac{b_{i}-\bA_{i :} x^0 -\bA_{i :} \bA^\top y^k}{\|\bA_{i :}\|^2} e_{i}.
\end{equation}
where $\bA_{i:}$ indicates the $i^{th}$ row of matrix $\bA$. 
This is the randomized coordinate ascent method \cite{nesterov2012efficiency} applied to the dual problem. Having said that, the analysis provided in \cite{nesterov2012efficiency} does not apply because the objective of the dual problem is not strongly concave function.
\end{exa}

\section{Simple Analysis of Baseline Methods}
\label{IntroductoryAnalysis}
Having presented the problems that we are interested in this thesis and explained the relationships between them, let us know present some interesting properties of our setting and a simple convergence analysis of the baseline methods for solving them.  

In Sections~\ref{SGDSection_Intro} and \ref{BestaprooximationSection_INtro} we have introduced these baseline methods. As a reminder, these are the stochastic gradient descent (SGD) \eqref{SGD_IntroThesis}, stochastic Newton method (SN) \eqref{SNM_IntroThesis}, stochastic proximal point method (SPP) \eqref{SPPM_IntroThesis}, sketch and project method (SPM) \eqref{SPM_IntroThesis}, and  stochastic dual subspace ascent (SDSA)\eqref{SDSA_IntroThesis}.

To simplify the presentation, in the remaining sections of the introduction we focus on two of these algorithms: SGD and SDSA. Recall at this point that SGD, SN, SPP and SPM have identical updates for the problems under study. This means that the analysis presented here for SGD holds for all of these methods.

We start by presenting some interesting properties of the stochastic quadratic optimization problem \eqref{StochReform_IntroThesis} and discussing connections with existing literature. Then we focus on the convergence analysis results.

\subsection{Technical preliminaries}
Recently, linear convergence of opimization methods has been established under several conditions that are satisfied in many realistic scenarios. We refer the interested reader to \cite{karimi2016linear} and \cite{necoara2015linear} for more details on these conditions and how they are related to each other.

In this thesis, we are particularly interested in the Quadratic Growth condition (QG).  We say that a function $f$ satisfies the QG inequality if the following holds for some $\mu>0$:
\begin{equation}
\label{QuadraticGrowth}
\frac{\mu}{2} \|x-x^*\|^2 \leq f(x) -f(x^*). 
\end{equation}
Here $x^*$ is the projection of vector $x$ onto the solution set $\cX$ of the optimization problem $\min_{x \in \R^n} f(x)$ and $f(x^*)$ denotes the optimal function value.

Under this condition it can be shown that SGD with constant stepsize $\omega$ converges with a linear rate up to a neighborhoud around the optimal point that is proportional to the value of $\omega$ \cite{karimi2016linear}.  For general convex function the convergence of SGD is sublinear \cite{robbins1951stochastic,nemirovski2009robust}.

In \cite{ASDA} it was shown the the function of the stochastic quadratic optimization problem \eqref{StochReform_IntroThesis} satisfies the QG condition as well. In particular the following lemma was proved.

\begin{lem}[Quadratic bounds, \cite{ASDA}]
\label{bounds}
For all $x \in \R^n$ and $x^* \in \cL$ the objective function of the stochastic optimization problem \eqref{StochReform_IntroThesis} satisfies:
\begin{equation}
\label{b1}
\lambda_{\min}^+ f(x) \leq \frac{1}{2} \|\nabla f(x) \|^2_{\bB} \leq \lambda_{\max} f(x)
\end{equation}
and 
\begin{equation}
\label{b2}
f(x) \leq  \frac{\lambda_{\max}}{2} \|x-x^*\|^2_{\bB}.
\end{equation}
Moreover, if exactness is satisfied, and we let $x^* =\Pi_{\cL, \bB}(x)$, we have 
\begin{equation}
\label{b3}
\frac{\lambda_{\min}^+}{2} \|x-x^*\|^2_{\bB} \leq  f(x)  .
\end{equation}
\end{lem}

Note that inequality \eqref{b3} is precisely the quadratic growth condition \eqref{QuadraticGrowth} with $\mu=\lambda_{\min}^+$ and $f(x^*)=0$. 

The following identities were also established in \cite{ASDA}. For completeness, we include different (and somewhat simpler) proofs here.
\begin{lem}
\label{cnaoisna}
For all $x \in \R^n$ and any $\bS \sim \cD$ we have
\begin{equation}
\label{normbound}
f_{\bS}(x) = \frac{1}{2}\|\nabla f_{\bS}(x)\|^2_{\bB}.
\end{equation}
Moreover, if $x^*\in \cL $ (i.e., if $x^*$ satisfies $\mA x^* =b$), then for all $x\in \R^n$ we have
\begin{equation}
\label{functionequivalence}
f_{\bS}(x) = \frac{1}{2}\langle \nabla f_{\bS}(x),x-x^* \rangle_{\bB}, 
\end{equation}
and
\begin{equation}
\label{asnda}
f(x) = \frac{1}{2}\langle \nabla f(x),x-x^* \rangle_{\bB}.
\end{equation}
\end{lem}

\begin{proof} In view of \eqref{Gradf_S_IntroThesis}, and  since $\mZ \mB^{-1} \mZ = \mZ$ (see \cite{ASDA}), we have
\begin{eqnarray*}\|\nabla f_{\bS}(x)\|^2_{\bB} & \overset{\eqref{Gradf_S_IntroThesis}}{=}& \|\bB^{-1} \mZ (x-x^*)\|^2_{\bB} = (x-x^*)^\top \mZ \bB^{-1} \bZ (x-x^*) =  (x-x^*)^\top \mZ (x-x^*) \\
&\overset{\eqref{ZETA_Intro}}{=} &(x-x^*)^\top \mA^\top \mH \mA (x-x^*)  = (\mA x-b)^\top \mH (\mA  x- b )  \overset{\eqref{f_s_IntroThesis}}{=} 2f_{\bS}(x).\end{eqnarray*}
Moreover,
\begin{eqnarray*}
\langle \nabla f_{\bS}(x),x-x^* \rangle_{\bB} & \overset{\eqref{Gradf_S_IntroThesis}}{=}& \langle \bB^{-1} \bZ (x-x^*),x-x^* \rangle_{\bB}\\
& =& (x-x^*)^\top \bZ \bB^{-1} \bB (x-x^*) \quad \overset{\eqref{f_sZeta_IntroThesis}}{=} \quad 2f_{\bS}(x).
\end{eqnarray*}
By taking expectations in the last identity with respect to the random matrix $\bS$, we get
$
\langle \nabla f(x),x-x^* \rangle_{\bB}=2f(x).
$
\end{proof}

\paragraph{No need for variance reduction}
SGD  is arguably  one of the most popular algorithms in machine learning. Unfortunately, SGD suffers from slow convergence, which is due to the fact that the variance of the stochastic gradient as an estimator of the gradient does not naturally diminish. For this reason, SGD is typically used with a decreasing stepsize rule, which ensures that the variance converges to zero. However, this has an adverse effect on the convergence rate. For instance, SGD has a sublinear rate even if the function to be minimized is strongly convex (conergence to the optimum point, not to a neighborhoud around it). To overcome this problem, a new class of so-called  {\em variance-reduced} methods was developed over the last 8 years, including SAG \cite{schmidt2017minimizing}, SDCA \cite{SDCA, richtarik2014iteration}, SVRG/S2GD  \cite{johnson2013accelerating, S2GD}, minibatch SVRG/S2GD \cite{mS2GD}, and SAGA \cite{defazio2014saga, defazio2016simple}.
 
In our setting, we assume that the linear system \eqref{LinearSystem_IntroThesis} is feasible. Thus, it follows that the stochastic gradient vanishes at the optimal point (i.e., $\nabla f_{\bS}(x^*)=0$ for any $\mS$). This suggests that additional variance reduction techniques are not necessary since the variance of the stochastic gradient drops to zero as we approach the optimal point $x^*$. In particular, in our context, SGD with fixed stepsize enjoys linear rate without any variance reduction strategy (see Theorem \ref{BasicMethodConvergence}). Hence, in this thesis we can bypass the development of variance reduction techniques, which allows us to focus on the momentum term in Chapter~\ref{ChapterMomentum} on the inexact computations in Chapter~\ref{ChapterInexact} and on gossip protocols that converge to consensus in Chapters~\ref{ChapterGossip} and \ref{ChapterPrivacy}.

\subsection{Theoretical guarantees}
The following convergence rates of SGD and SDSA are easy to establish, having the bounds and identities of Lemmas \ref{bounds} and \ref{cnaoisna}. Nevertheless we present the statements of the theorems and proofs for completeness and because we use similar ideas and approaches in the rest of the thesis. For the benefit of the reader, we also include the derivations of the two equations \eqref{Corresp_IntroThesis} and \eqref{identity} presented in the previous section which connect the primal and the dual iterates.

\begin{thm}[\cite{ASDA}]
\label{BasicMethodConvergence}
Let assume exactness and let $\{x^k\}_{k=0}^\infty$ be the iterates produced by SGD with constant stepsize $\omega\in(0,2)$.  Set $x^*=\Pi_{\cL,\bB}(x^0)$. Then,
\begin{equation}
\label{convergenceSGD}
\Exp[\|x^{k}-x^*\|^2_{\bB}] \leq\left[1-\omega(2-\omega) \lambda_{\min}^+ \right]^k \|x^0-x^*\|^2_{\bB}.
\end{equation}
\end{thm}
\begin{proof}
\begin{eqnarray}
\label{x_k_omega}
\|x^{k+1}-x^*\|^2_{\bB} & \overset{\eqref{SGD_IntroThesis}}{=}& \|x^k-\omega \nabla f_{\bS_k}(x^k)-x^*\|^2_{\bB}\notag\\
&=& \|x^k-x^*\|^2_{\bB}-2 \omega \langle x^k-x^*,\nabla f_{\bS_k}(x^k) \rangle_{\bB} +\omega^2 \|\nabla f_{\bS_k}(x^k)\|^2_{\bB} \notag\\
& \overset{\eqref{normbound},\eqref{functionequivalence}}{=} & \|x^k-x^*\|^2_{\bB}-4\omega f_{\bS_k}(x^k)+2\omega^2 f_{\bS_k}(x^k) \notag\\
&=& \|x^k-x^*\|^2_{\bB}-2\omega(2-\omega)f_{\bS_k}(x^k).
\end{eqnarray}
By taking expectation with respect to $\mS_k$ and using quadratic growth inequality \eqref{b3}:
\begin{eqnarray}
\label{exp_x_k_omega}
\Exp_{\mS_k}[\|x^{k+1}-x^*\|^2_{\bB}] &=& \|x^k-x^*\|^2_{\bB}-2\omega(2-\omega)f(x^k)\notag\\
&\overset{\omega \in (0,2), \, \eqref{b3}}{\leq}& \|x^k-x^*\|^2_{\bB}-\omega(2-\omega) \lambda_{\min}^+  \|x^k-x^*\|^2_{\bB}\notag\\
&=&\left[1-\omega(2-\omega) \lambda_{\min}^+ \right] \|x^k-x^*\|^2_{\bB}.
\end{eqnarray}
Taking expectation again and by unrolling the recurrence we obtain \eqref{convergenceSGD}.
\end{proof}

\begin{prop} 
\label{bcaoskla}
Let $\{x^k\}_{k=0}^\infty$ be the iterates produced by SGD\eqref{SGD_IntroThesis} with $\omega\in(0,2)$. Let $y^0=0$, and let $\{y^k\}_{k=0}^\infty$ be the iterates of SDSA \eqref{SDSA_IntroThesis}.   Assume that the  methods use the same stepsize $\omega \in(0,2)$ and the same sequence of random matrices $\mS_k$.  Then $x^k = \phi(y^k) = x^0 + \mB^{-1} \mA^\top y^k$ holds
for all $k$. That is, the primal iterates arise as affine images of the dual iterates.
\end{prop}
\begin{proof}
First note that
\begin{equation}
\label{acnsjdakdsa}
\nabla f_{\mS_k}(\phi(y^k))\overset{\eqref{Gradf_S_IntroThesis}}{=} \mB^{-1}\mA^\top  \mS_k ( \mS_k^\top \mA \mB^{-1} \mA^\top \mS_k)^\dagger \mS_k^\top (\mA \phi(y^k) - b) \overset{\eqref{lambdak_IntroThesis}}{=}  - \mB^{-1}\mA^\top  \mS_k \lambda^k.
\end{equation}
We now use this to show that
\begin{eqnarray*}
\phi(y^{k+1}) &\overset{\eqref{Corresp_IntroThesis}}{=}& x^0  + \bB^{-1}\bA^\top y^{k+1} \overset{\eqref{SDSA_IntroThesis}}{=} x^0+\bB^{-1}\bA^\top\left[y^k+\omega \bS_k \lambda^k \right]\\
&=& x^0+\bB^{-1}\bA^\top y^k +\omega \bB^{-1}\bA^\top \bS_k \lambda^k \\
&\overset{\eqref{Corresp_IntroThesis},\eqref{acnsjdakdsa}}{=}& \phi(y^k) -\omega \nabla f_{\bS_k}(\phi(y^k)).
\end{eqnarray*}
So, the sequence of vectors $\{\phi(y^k)\}$ satisfies the same recursion to the sequence $\{x^k\}$ defined by SGD.  It remains to check that the starting vectors of both recursions coincide. Indeed, since $y^0=0$, we have $x^0 = \phi(0) =  \phi(y^0)$.
\end{proof}

\begin{prop}[\cite{gower2015stochastic}]
\label{PropositionDualPrimal}
Let $y^*$ be a solution of the dual problem \eqref{DualProblem_IntroThesis}. Then,
$$D(y^*)-D(y^k)=\frac{1}{2}\|x^k-x^*\|^2_{\bB}.$$
\end{prop}
\begin{proof}
\begin{eqnarray}
D(y^*)-D(y^k) & \overset{\eqref{DualProblem_IntroThesis}}{=}& (b-\bA x^0)^\top y^* - \frac{1}{2}\|\bA^\top y^*\|^2_{\bB^{-1}} -(b-\bA x^0)^\top y^k + \frac{1}{2}\|\bA^\top y^k\|^2_{\bB^{-1}}\notag\\
& =& (b-\bA x^0)^\top (y^*-y^k) - \frac{1}{2} (\bA^\top y^*)^\top {\bB^{-1}}\bA^\top y^* \notag\\&& + \,\, \frac{1}{2} (\bA^\top y^k)^\top \bB^{-1} \bA^\top y^k\notag\\
& \overset{(*)}{=}& (\bA \bB^{-1}\bA^\top y^*)^\top (y^*-y^k) - \frac{1}{2} (\bA^\top y^*)^\top {\bB^{-1}}\bA^\top y^* \notag\\&& +\,\, \frac{1}{2} (\bA^\top y^k)^\top \bB^{-1} \bA^\top y^k\notag\\
& =& \frac{1}{2} (y^*-y^k) \bA \bB^{-1} \bA^\top (y^*-y^k)\notag\\
& =& \frac{1}{2}\| \bB^{-1} \bA^\top (y^*-y^k)\|^2_{\bB}\notag\\
& \overset{\eqref{Corresp_IntroThesisOptimal}+\eqref{Corresp_IntroThesis}}{=} & \frac{1}{2}\|x^k-x^*\|^2_{\bB}.
\end{eqnarray}
In the $(*)$ equality above we use \eqref{Corresp_IntroThesisOptimal} for the optimal primal and dual values. In particular, $x^0=x^*-\mB^{-1} \mA^\top y^* \Rightarrow \bA x^0=\bA x^*-\bA\mB^{-1} \mA^\top y^*=b-\bA\mB^{-1} \mA^\top y^*$.
\end{proof}

The next theorem has been proved in \cite{gower2015stochastic} for the case of $\omega=1$.  Here we extend this convergence to the more general case of $0<\omega<2$.
 
\begin{thm}
\label{TheoremSDSA_IntroThesis}
Let us assume exactness. Choose $y^0= 0 \in \R^m$. Let $\{y^k\}_{k=0}^\infty$ be the sequence of random iterates produced by SDSA with stepsize $0\leq \omega \leq 2$. Then,
\begin{equation}
\label{convergenceSDSA}
\Exp[D(y^*)-D(y^k)] \leq \left[1-\omega(2-\omega) \lambda_{\min}^+ \right]^k \left( D(y^*)-D(y^0) \right).
\end{equation}
\end{thm}

\begin{proof} This follows by applying Theorem~\ref{BasicMethodConvergence} together with Proposition~\ref{PropositionDualPrimal}.
\end{proof}

\subsection{Iteration Complexity}
In several parts of this thesis we compare the performance of linearly convergent algorithms using their iteration complexity bounds. That is, we derive a lower bound on the number of iterations that are sufficient to achieve a prescribed accuracy. The following lemma shows the derivation of this bound for the sequence $\{A^k\}_{k=0}^\infty$.

\begin{lem}
\label{laksoakls}
Consider a non-negative sequence $\{A^k\}_{k=0}^\infty$ satisfying 
\begin{equation}
\label{laksmnoia}
A^k \leq \rho^k A^0,
\end{equation}
where $\rho \in (0,1)$. Then, for a given $\epsilon \in (0,1)$ and for:
\begin{equation}
\label{anlksado}
k \geq \frac{1}{1-\rho} \log \left(\frac{1}{\epsilon} \right)
\end{equation}
it holds that:
\begin{equation}
A^k \leq \epsilon A^0
\end{equation}
\end{lem}

\begin{proof}
Note that since $\rho \in (0,1)$, we have:
\begin{equation}
\label{dnqoiakl}
\frac{1}{1-\rho}\log \left(\frac{1}{\rho}\right) >1.
\end{equation}
Therefore,
\begin{equation}
\label{loasloaslkna}
\log\left(\frac{A^0}{A^k}\right)\overset{\eqref{laksmnoia}}{\geq} k \log \left(\frac{1}{\rho}\right)\overset{\eqref{anlksado}}{\geq} \frac{1}{1-\rho} \log \left(\frac{1}{\epsilon} \right) \log \left(\frac{1}{\rho}\right)\overset{\eqref{dnqoiakl}}{\geq}\log \left(\frac{1}{\epsilon} \right).
\end{equation}
Applying exponentials to the above inequality completes the proof.
\end{proof}

As an instance, of how the above lemma can be used, recall the convergence result of Theorem~\ref{BasicMethodConvergence}, where we have proved that SGD with constant stepsize $\omega \in (0,2)$ converges as follows:
$$
\Exp[\|x^{k}-x^*\|^2_{\bB}] \leq\left[1-\omega(2-\omega) \lambda_{\min}^+ \right]^k \|x^0-x^*\|^2_{\bB}.
$$
In this setting, Lemma~\ref{laksoakls} can be utilized with $A^k=\Exp[\|x^{k}-x^*\|^2_{\bB}]$ and $\rho=1-\omega(2-\omega) \lambda_{\min}^+$ to obtain:
$$k \geq \frac{1}{\omega(2-\omega) \lambda_{\min}^+} \log \left(\frac{1}{\epsilon} \right) \quad \Rightarrow \quad \Exp[\|x^k-x^*\|_{\bB}^2]\leq \epsilon \|x^0-x^*\|_{\bB}^2.$$

In this case we say that SGD converges with iteration complexity
$$\cO\left(\frac{1}{\omega(2-\omega) \lambda_{\min}^+} \log \left(\frac{1}{\epsilon} \right)\right).$$

\section{Structure of the Thesis}
\label{StructureIntro}
In the remainder of this section we give a summary of each chapter of this thesis. The detailed proofs and careful deductions of any claims made here are left to the chapters.

\subsection{Chapter 2: Randomized Iterative Methods with Momentum and Stochastic Momentum}
The baseline first-order method for minimizing a differentiable function $f$ is the \textit{gradient descent (GD)} method, $$x^{k+1} = x^k - \omega^k \nabla f(x^k),$$ where $\omega^k>0$ is a stepsize \cite{cauchy1847methode}. For convex functions with $L$-Lipschitz gradient,  GD converges at at the rate of $\cO(L/\epsilon)$. When, in addition, $f$ is $\mu$-strongly convex, the rate is linear: $\cO((L /\mu) \log(1/\epsilon))$ \cite{nesterov2013introductory}. To improve the convergence behavior of the method, Polyak proposed to modify GD by the introduction of a (heavy ball) momentum term\footnote{
\blue{A more popular, and certainly theoretically much better understood alternative to Polyak's momentum is the momentum introduced by Nesterov \cite{nesterov1983method, nesterov2013introductory}, leading to the famous {\em accelerated gradient descent (AGD)} method. This method converges non-asymptotically and globally; with optimal sublinear rate $\cO(\sqrt{L/\epsilon})$  \cite{nemirovskii1983problem} when applied to minimizing a smooth convex objective function (class $\cF^{1,1}_{0,L}$), and with the optimal linear  rate $\cO(\sqrt{L/\mu} \log(1/\epsilon))$ when minimizing smooth strongly convex functions (class $\cF^{1,1}_{\mu,L}$).   Recently, variants of Nesterov's momentum have also been introduced for the acceleration of stochastic gradient descent. We refer the interested reader to \cite{ghadimi2016accelerated, allen2017katyusha, jofre2017variance, jalilzadeh2018optimal, Zhou2018, zhou2018icml, Loopless} and the references therein. Both Nesterov's and Polyak's update rules are known in the literature as ``momentum'' methods. In Chapter~\ref{ChapterMomentum}, however, we focus exclusively on Polyak's heavy ball momentum.}}, $\beta(x^k-x^{k-1})$ \cite{polyak1964some, polyak1987introduction}. This leads to the gradient descent method with momentum (mGD), popularly  known as the {\em heavy ball method:} 
$$x^{k+1} = x^k - \omega^k \nabla f(x^k) + \beta (x^k-x^{k-1}).$$  More specifically, Polyak proved that with the correct choice of the stepsize parameters $\omega^k$ and momentum parameter $\beta$, a \textit{local} accelerated linear convergence rate of  $\cO(\sqrt{L/\mu}\log(1/\epsilon))$  can be achieved in the case of twice continuously differentiable, $\mu$-strongly convex objective functions with $L$-Lipschitz gradient \cite{polyak1964some, polyak1987introduction}.

The theoretical behavior of the above deterministic heavy ball method is now well understood in different settings. In contrast to this, there has been less progress in understanding the convergence behavior of  {\em stochastic} variants of the heavy ball method. The key method in this category is stochastic gradient descent with momentum (mSGD; stochastic heavy ball method): $x^{k+1} = x^k -\omega^k g (x^k)  + \beta (x^k-x^{k-1}),$ where $g$ is an unbiased estimator of the true gradient $\nabla f(x^k)$.
 In our setting, where our goal is to solve the stochastic optimization problem \eqref{StochReform_IntroThesis},  mSGD takes the following form:
$$x^{k+1} = x^k - \omega \nabla f_{\mS_k}(x^k)+ \beta (x^k-x^{k-1}),$$
where $\omega > 0$ denotes a fixed stepsize and matrix $\mS_k$ is sampled afresh in each iteration from distribution $\cD$.

In Chapter \ref{ChapterMomentum}, we study several classes of stochastic optimization algorithms enriched with the {\em heavy ball momentum} for solving the three closely related problems already described in the introduction. These are the stochastic quadratic optimization problem \eqref{StochReform_IntroThesis}, the best approximation problem \eqref{BestApproximation_IntroThesis} and the dual quadratic optimization problem \eqref{DualProblem_IntroThesis}. Among the methods studied are: stochastic gradient descent \eqref{SGD_IntroThesis}, stochastic Newton \eqref{SNM_IntroThesis}, stochastic proximal point \eqref{SPPM_IntroThesis} and stochastic dual subspace ascent \eqref{SDSA_IntroThesis}. This is the first time momentum variants of several of these methods are studied.  
We prove global non-asymptotic linear convergence rates for all methods and various measures of success, including primal function values, primal iterates, and dual function values. We also show that the primal iterates converge at an accelerated linear rate in a somewhat weaker sense. This is the first time a linear rate is shown for the stochastic heavy ball method (i.e., stochastic gradient descent method with momentum). Under somewhat weaker conditions,  we establish a sublinear convergence rate for Ces\`{a}ro averages of primal iterates.  Moreover, we propose a novel concept, which we call {\em stochastic momentum}, aimed at decreasing the cost of performing the momentum step. We prove linear convergence of several stochastic methods with stochastic momentum, and show that in some sparse data regimes and for sufficiently small momentum parameters, these methods enjoy better overall complexity than methods with deterministic momentum. Finally, we perform extensive numerical testing on artificial and real datasets. 

\subsection{Chapter 3: Inexact Randomized Iterative Methods}
A common feature of existing randomized iterative methods is that in their update rule a particular subproblem needs to be solved \emph{exactly}. In the large scale setting, often this step can be computationally very expensive. The purpose of the work in Chapter \ref{ChapterInexact} is to reduce the cost of this step by incorporating \emph{inexact updates} in the stochastic methods under study.

From a stochastic optimization viewpoint, we analyze the performance of inexact SGD (iSGD):
$$x^{k+1}=x^k-\omega \nabla f_{\mS_k}(x^k)  + \epsilon^{k},$$
where $\omega >0$ denotes a fixed stepsize, matrix $\mS_k$ is sampled afresh in each iteration from distribution $\cD$ and $\epsilon^{k}$ represents a (possibly random) error coming from inexact computations.

In Chapter~\ref{ChapterInexact}, we propose and analyze {\em inexact} variants of the exact algorithms presented in previous sections for solving the stochastic optimization problem \eqref{StochReform_IntroThesis}, the best approximation problem \eqref{BestApproximation_IntroThesis} and the dual problem \eqref{DualProblem_IntroThesis}. Among the methods studied are: stochastic gradient descent (SGD), stochastic Newton (SN), stochastic proximal point (SPP), sketch and project method (SPM) and stochastic subspace ascent (SDSA). 
In all of these methods, a certain potentially expensive calculation/operation needs to be performed in each step; it is this operation that we propose to be performed \emph{inexactly}.  For instance, in the case of SGD, it is the computation of the stochastic gradient $\nabla f_{\mS_k}(x^k)$, in the case of  SPM is the computation of the projection $\Pi_{\cL_{\mS}, \bB}(x^k)$, and in the case of SDSA it is the computation of the dual update $\mS_k \lambda^k$.

We perform an iteration complexity analysis under an abstract notion of inexactness and also under a more structured form of inexactness appearing in practical scenarios. Typically, an inexact solution of these subproblems can be obtained much more quickly than the exact solution. Since in practical applications the savings thus obtained are larger than the increase in the number of iterations needed for convergence, our inexact methods can be dramatically faster.

Inexact variants of many popular and some more exotic methods, including  randomized block Kaczmarz, randomized Gaussian Kaczmarz and randomized block coordinate descent, can be cast as special cases of our analysis. Finally, we present numerical experiments which demonstrate the benefits of allowing inexactness.

\subsection{Chapter 4: Revisiting Randomized Gossip Algorithms}
In Chapter~\ref{ChapterGossip} we present a new framework for the analysis and design of randomized gossip algorithms for solving the average consensus (AC) problem,  a fundamental problem in distributed computing and multi-agent systems.

In the AC problem we are given an undirected connected network $\cG=(\cV,\cE)$ with node set $\cV=\{1,2,\dots,n\}$ and edges $\cE$. Each node $i \in \cV$ ``knows'' a private value $c_i \in \R$. The goal of AC is for every node to compute the average of these private values, $\bar{c}\eqdef\frac{1}{n}\sum_i c_i$, in a decentralized fashion. That is, the exchange of information can only occur between connected nodes (neighbors). 

In an attempt to connect the AC problem to optimization, consider the simple optimization problem:
\begin{equation}
\label{canosxaadada}
\min_{x = (x_1,\dots, x_n) \in \R^n} \frac{1}{2} \|x-c\|^2 
\quad \text{subject to}  \quad x_1=x_2=\dots=x_n, 
\end{equation}
where $c=(c_1,\dots,c_n)^\top$ is the vector of the initial private values $c_i$. Observe that its optimal solution $x^*$ must necessarily satisfy $x^*_i=\bar{c}$ for all $i \in [n]$, where $\bar{c}$ is the value that each node needs to compute in the AC problem. Note also that, if we represent the constraint $x_1=x_2=\dots=x_n$ of \eqref{canosxaadada} as a linear system then the optimization problem \eqref{canosxaadada} is an instance of the best approximation problem \eqref{BestApproximation_IntroThesis} with $\bB=\bI$ (identity matrix). Perhaps, there is a deeper link here? Indeed, it turns out that properly chosen randomized algorithms for solving \eqref{canosxaadada} can be interpreted as decentralized protocols for solving the AC problem.

A simple way to express the constraint of problem \eqref{canosxaadada} as linear system $\bA x= b$ is by selecting $\bA$ to be the incidence matrix of the network and the right hand side to be the zero vector ($b=0 \in \R^m$). By using this system the most basic randomized gossip algorithm (``randomly pick an edge $e=(i,j) \in \cE$ and then replace the values stored at vertices $i$ and $ j$ by their average") is an instance of the randomized Kaczmarz (RK) method \eqref{RK_IntroThesis} for solving consistent linear systems, applied to this system.

Using this observation as a starting point, in Chapter~\ref{ChapterGossip} we show how classical randomized iterative methods for solving linear systems can be interpreted as gossip algorithms when applied to special systems encoding the underlying network and explain in detail their decentralized nature. 
Our general framework recovers a comprehensive array of well-known gossip algorithms as special cases, including the pairwise randomized gossip algorithm and path averaging gossip, and allows for the development of provably faster variants. The flexibility of the new approach enables the design of a number of new specific gossip methods. 
For instance, we propose and analyze novel \textit{block} and the first provably  \textit{accelerated} randomized gossip protocols, and \textit{dual} randomized gossip algorithms.

From a numerical analysis viewpoint, our work is the first that explores in depth the decentralized nature of randomized iterative methods for linear systems and proposes them as methods for solving the average consensus problem. 

We evaluate the performance of the proposed gossip protocols by performing extensive experimental testing on typical wireless network topologies.

\subsection{Chapter 5: Privacy Preserving Randomized Gossip Algorithms}
In Chapter~\ref{ChapterPrivacy}, we present three different approaches to solving the Average Consensus problem while at the same time protecting the information about the initial values of the nodes. To the best of our knowledge, this work is the first which combines the \emph{gossip framework} with the privacy concept of protection of the initial values. 

The randomized methods we propose are all dual in nature. That is, they solve directly the dual problem \eqref{DualProblem_IntroThesis}. In particular, the three different techniques that we use for preserving the privacy are ``Binary Oracle'', ``$\epsilon$-Gap Oracle'' and ``Controlled Noise Insertion''. 

{\bf Binary Oracle:}
We propose to reduce the amount of information transmitted in each iteration to a single bit. More precisely, when an edge is selected, each corresponding node will only receive information whether the value on the other node is smaller or larger. Instead of setting the value on the selected nodes to their average, each node increases or decreases its value by a pre-specified amount.

{\bf $\epsilon$-Gap Oracle:}
In this case, we have an oracle that returns one of three options and is parametrized by $\epsilon$. If the difference in values of sampled nodes is larger than $\epsilon$, an update similar to the one in Binary Oracle is taken. Otherwise, the values remain unchanged. An advantage compared to the Binary Oracle is that this approach will converge to a certain accuracy and stop there, determined by $\epsilon$ (Binary Oracle will oscillate around optimum for a fixed stepsize). However, in general, it will disclose more information about the initial values.

{\bf Controlled Noise Insertion:}
This approach protects the initial values by inserting noise in the process. Broadly speaking, in each iteration, each of the sampled nodes first adds a noise to its current value, and an average is computed afterwards. Convergence is guaranteed due to the correlation in the noise across iterations. Each node remembers the noise it added last time it was sampled, and in the following iteration, the previously added noise is first subtracted, and a fresh noise of smaller magnitude is added. Empirically, the protection of initial values is provided by first injecting noise into the system, which propagates across the network, but is gradually withdrawn to ensure convergence to the true average.

We give iteration complexity bounds for all proposed privacy preserving randomized gossip algorithms and perform extensive numerical experiments.

\section{Summary}
The content of this thesis is based on the following publications and preprints:

\paragraph{Chapter 2}
\begin{itemize}
\item Nicolas Loizou and Peter Richt\'{a}rik.``Momentum and Stochastic Momentum for Stochastic Gradient, Newton, Proximal Point and Subspace Descent Methods", arXiv preprint arXiv:1712.09677 (2017). \cite{loizou2017momentum}
\item Nicolas Loizou and Peter Richt\'{a}rik. ``Linearly Convergent Stochastic Heavy Ball Method for Minimizing Generalization Error", Workshop on Optimization for Machine Learning, NIPS 2017. \cite{loizou2017linearly}
\end{itemize}

\paragraph{Chapter 3}
\begin{itemize}
\item Nicolas Loizou and Peter Richt\'{a}rik. ``Convergence Analysis of Inexact Randomized Iterative Methods", arXiv preprint arXiv:1903.07971 (2019). \cite{loizou2019Inexact}
\end{itemize}

\paragraph{Chapter 4}
\begin{itemize}
\item Nicolas Loizou and Peter Richt\'{a}rik. ``A New Perspective on Randomized Gossip Algorithms", IEEE Global Conference on Signal and Information Processing (GlobalSIP), pp.440-444, 2016  \cite{LoizouRichtarik}
\item Nicolas Loizou and Peter Richt\'{a}rik. ``Accelerated Gossip via Stochastic Heavy Ball Method." 56th Annual Allerton Conference on Communication, Control, and Computing (Allerton) (pp. 927-934), 2018. \cite{loizou2018accelerated}
\item Nicolas Loizou, Mike Rabbat and Peter Richt\'{a}rik. ``Provably Accelerated Randomized Gossip Algorithms" 2019 IEEE International Conference on Acoustics, Speech and Signal Processing (ICASSP), pp. 7505-7509 \cite{loizou2018provably}
\item Nicolas Loizou and Peter Richt\'{a}rik. ``Revisiting Randomized Gossip Algorithms: General Framework, Convergence Rates and Novel Block and Accelerated Protocols", arXiv preprint arXiv:1905.08645, \cite{loizou2019revisiting}.
\end{itemize}

\paragraph{Chapter 5}
\begin{itemize}
\item Filip Hanzely, Jakub Kone\v{c}n\'{y}, Nicolas Loizou,  Peter Richt\'{a}rik and  Dmitry Grishchenko. ``A Privacy Preserving Randomized Gossip Algorithm via Controlled Noise Insertion", NeurIPS 2018 - Privacy Preserving Machine Learning Workshop,  \cite{hanzely2019privacy}
\item  Filip Hanzely, Jakub Kone\v{c}n\'{y}, Nicolas Loizou,  Peter Richt\'{a}rik and Dmitry Grishchenko.  ``Privacy Preserving Randomized Gossip Algorithms",  arXiv preprint arXiv:1706.07636, 2017 \cite{hanzely2017privacy}
\end{itemize}

During the course of my study, I also co-authored the following works which were not used in the formation of this thesis:
\begin{itemize}
\item Mahmoud Assran, Nicolas Loizou, Nicolas Ballas and Mike Rabbat. ``Stochastic Gradient Push for Distributed Deep Learning", Proceedings of the 36th International Conference on Machine Learning (ICML), 2019 \cite{assran2018stochastic}
\item Robert Mansel Gower, Nicolas Loizou, Xun Qian, Alibek Sailanbayev, Egor Shulgin and Peter Richt\'{a}rik. ``SGD: General Analysis and Improved Rates" Proceedings of the 36th International Conference on Machine Learning (ICML), 2019 \cite{gower2019sgd}
\item Nicolas Loizou. ``Distributionally Robust Games with Risk-Averse Players", In Proceedings of 5th International Conference on Operations Research and Enterprise Systems (ICORES), 186-196, 2016 \cite{loizou2016distributionally}
\end{itemize}

In \cite{gower2019sgd}, we propose a general theory describing the convergence of Stochastic Gradient Descent (SGD) under the ``arbitrary sampling paradigm". Our theory describes the convergence of an infinite array of variants of SGD, each of which is associated with a specific probability law governing the data selection rule used to form minibatches. This is the first time such an analysis is performed, and most of our variants of SGD were never explicitly considered in the literature before. 

In \cite{assran2018stochastic}, we study Stochastic Gradient Push (SGP), an algorithm which combines PushSum gossip protocol with stochastic gradient updates for distributed deep learning. We prove that SGP converges to a stationary point of smooth, non-convex objectives at the same sub-linear rate as SGD, that all nodes achieve consensus, and that SGP achieves a linear speedup with respect to the number of compute nodes. Furthermore, we empirically validate the performance of SGP on image classification (ResNet-50, ImageNet) and machine translation (Transformer, WMT’16 En- De) workloads.

In \cite{loizou2016distributionally} we present a new model of incomplete information games without private information in which the players use a distributionally robust optimization approach to cope with the payoff uncertainty. 

\chapter{Randomized Iterative Methods with Momentum and Stochastic Momentum}
\label{ChapterMomentum}

\section{Introduction} \label{sec:intro}

Two of the most popular algorithmic ideas  for solving optimization problems involving big volumes of data are {\em stochastic approximation}  and {\em momentum}. By stochastic approximation we refer to the practice pioneered by Robins and Monro \cite{robbins1951stochastic} of replacement of costly-to-compute quantities (e.g.,  gradient of the objective function) by cheaply-to-compute {\em stochastic} approximations thereof (e.g., unbiased estimate of the gradient).  By momentum we refer to the {\em heavy ball} technique originally developed by Polyak \cite{polyak1964some} to accelerate the convergence rate of gradient-type methods.  

While much is known about the effects of {\em stochastic approximation} and {\em momentum} in isolation, surprisingly little is known about the {\em combined effect} of these two popular algorithmic techniques. For instance, to the best of our knowledge, there is no context in which a method combining stochastic approximation with momentum  is known to have a linear convergence rate. One of the  contributions of this work is to show that there are important problem classes for which a linear rate can indeed be established for a range of stepsize and momentum parameters.

\subsection{The setting}
In this chapter we study the three closely related problems described in the introduction of this thesis. These are: 
\begin{enumerate}
\item[(i)] stochastic quadratic optimization \eqref{StochReform_IntroThesis}, 
\item[(ii)] best approximation \eqref{BestApproximation_IntroThesis}, and 
\item[(iii)] (bounded) concave quadratic maximization \eqref{DualProblem_IntroThesis}.
\end{enumerate} 

In particular we are interested in the complexity analysis and efficient implementation of momentum variants of the baseline algorithms presented in the introduction. As a reminder, these methods are the following: stochastic gradient descent (SGD), stochastic Newton method (SN), stochastic proximal point methods (SPP), sketch and project method (SPM) and stochastic dual subspace ascent (SDSA).

We are not aware of any successful attempts to analyze momentum variants of SN and SPP, SPM and SDSA and to the best of our knowledge there are no linearly convergent variants of SGD with momentum in any setting.

In addition, we propose and analyze a novel momentum strategy for SGD, SN, SPP and SPM, which we call \emph{stochastic momentum}. It is a stochastic approximation of the popular deterministic heavy ball momentum which in some situations could be particularly beneficial in terms of overall complexity. Similar to the classical momentum, we prove linear convergence rates for this momentum strategy. 

\subsection{Structure of the chapter}

This chapter is organized as follows. In Section~\ref{sec:contributions} we summarize our contributions in the context of existing literature. In  Section~\ref{sec:primal} we describe and analyze  primal  methods  with momentum (mSGD, mSN and mSPP), and in Section~\ref{sec:dual} we describe and analyze the dual method with momentum (mSDSA). In Section~\ref{sec:sm} we describe and analyze primal methods with stochastic momentum (smSGD, smSN and smSPP). Numerical experiments are presented in Section~\ref{sec:experiments}.  Proofs of all key results can be found in Section~\ref{ProofsMomentum}.

 \subsection{Notation}
\label{notationSection}
The following notational conventions are used in this chapter. Boldface upper-case letters denote matrices; $\bI$ is the identity matrix. By $\cL$ we denote the solution set of the linear system $\bA x=b$. By $\cL_{\bS}$, where $\mS$ is a random matrix, we denote the solution set of the {\em sketched} linear system $\bS^\top \bA x= \bS^\top b$. \blue{By $\bA_{i:}$ and $\bA_{:j}$ we indicate the $i_{th}$ row and the $j_{th}$ column of matrix $\bA$, respectively.} Unless stated otherwise, throughout the chapter, $x^*$ is the projection of $x^0$ onto $\cL$ in the $\mB$-norm: $x^*=\Pi_{\cL, \bB}(x^0)$. We also write $[n]\eqdef \{1,2, \dots ,n\}$. \blue{Finally, we say that a function $f : \R^n\rightarrow \R$ belongs to the class $\cF^{1,1}_{0,L}$ if it is convex, continuously differentiable, and its gradient is Lipschitz continuous with constant $L$. If in addition the function $f$ is $\mu$-strongly convex with strong convexity constant $\mu>0$, then we say that it belongs to the class $\cF^{1,1}_{\mu,L}$. When it is also twice continuously differentiable, it belongs to the function class $\cF^{2,1}_{\mu,L}$.}

\section{Momentum Methods and Main Contributions} \label{sec:contributions}
In this section we give a brief review of the relevant literature, and provide a  summary of our contributions.

\subsection{Heavy ball method}

As we have already mentioned in Section~\ref{StructureIntro}, Polyak's seminal work \cite{polyak1964some, polyak1987introduction} showed that deterministic heavy ball method:
$$x^{k+1} = x^k - \omega^k \nabla f(x^k) + \beta (x^k-x^{k-1}),$$
converges with a \textit{local} accelerated linear convergence rate of  $\cO(\sqrt{L/\mu}\log(1/\epsilon))$  in the case of twice continuously differentiable, $\mu$-strongly convex objective functions with $L$-Lipschitz gradient (function class $\mathcal{F}_{\mu, L}^{2,1}$). 

Recently, Ghadimi et al.~\cite{ghadimi2015global} performed a \textit{global} convergence analysis for the heavy ball method. In particular, the authors showed that for a certain combination of the stepsize and momentum parameter, the method  converges sublinearly to the optimum when the objective function is convex  and has Lipschitz gradient ($f\in \mathcal{F}_{0, L}^{1,1}$), and  linearly when the function is also strongly convex ($f\in \mathcal{F}_{\mu, L}^{1,1}$). A particular selection of the parameters $\omega$ and $\beta$ that gives the desired accelerated linear rate was not provided.
 
To the best of our knowledge, despite considerable amount of work on the heavy ball method, there is still no  global convergence analysis which would guarantee an accelerated linear rate for $f\in  \mathcal{F}_{\mu, L}^{1,1}$. However, in the special case of a strongly convex quadratic, an elegant proof was recently proposed in \cite{lessard2016analysis}. Using the notion of integral quadratic constraints from robust control theory, the authors proved that by choosing $\omega^k = \omega=4 / (\sqrt{L}+\sqrt{\mu})^2$ and $\beta=(\sqrt{L/\mu}-1)^2/(\sqrt{L/\mu}+1)^2$, the heavy ball method enjoys a global \textit{asymptotic} accelerated convergence rate of $\cO(\sqrt{L/\mu} \log(1/\epsilon))$. The aforementioned results are summarized in the first part of Table~\ref{ComparisonWIthHeavy}.

Extensions of the heavy ball method have been recently proposed in the proximal setting \cite{ochs2015ipiasco}, non-convex setting \cite{ochs2014ipiano, zavriev1993heavy} and for distributed optimization \cite{ghadimi2013multi}. For more recent analysis under several combinations of assumptions we suggest \cite{sun2019non,sun2019heavy,kulakova2018non}.

\subsection{Stochastic heavy ball method}

In contrast to the recent advances in our theoretical understanding of the (classical) heavy ball method, there has been less progress in understanding the convergence behavior of  {\em stochastic} variants of the heavy ball method. The key method in this category is stochastic gradient descent with momentum (mSGD; aka: stochastic heavy ball method):
\[x^{k+1} = x^k -\omega^k g (x^k)  + \beta (x^k-x^{k-1}),\]
where $g$ is an unbiased estimator of the true gradient $\nabla f(x^k)$. While mSGD is used extensively in practice, especially in deep learning \cite{sutskever2013importance, szegedy2015going, krizhevsky2012imagenet, wilson2017marginal}, its convergence behavior is not very well understood.

In fact, we are aware of only two  papers, both recent, which set out to study the complexity of mSGD: the work of Yang et al.\ \cite{yang2016unified}, and the work of Gadat et al.\ \cite{gadat2016stochastic}. In the former paper, a unified convergence analysis for stochastic gradient methods with momentum  (heavy ball and Nesterov's momentum) was proposed; and an analysis for  both convex and non convex functions was performed. For a general Lipschitz continuous convex objective function with bounded variance, a rate of $\cO(1/\sqrt{k})$ was proved. For this, the authors employed a decreasing stepsize strategy: $\omega^k=\omega^0/\sqrt{k+1}$, where $\omega^0$ is a positive constant.  In \cite{gadat2016stochastic}, the authors first describe several almost sure convergence results in the case of general non-convex coercive functions, and then provide a complexity analysis for the case of quadratic strongly convex function. However, the established rate is slow. More precisely, for strongly convex quadratic  and coercive functions, mSGD  with diminishing stepsizes $\omega^k=\omega^0 / k^\beta$ was shown to convergence as $\cO(1/k^\beta)$ when the momentum parameter is $\beta <1$, and with the rate $\cO(1/\log k)$ when $\beta=1$. The convergence rates established in  both of these papers  are sublinear. In particular, no insight is provided into whether the inclusion of the momentum term provides what \blue{it} was aimed to provide: acceleration. 

The above results are summarized in the second part of Table~\ref{ComparisonWIthHeavy}. From this perspective, our contribution lies in providing an in-depth analysis of mSGD (and, additionally, of SGD with stochastic momentum).  

Many recent papers have built upon our analysis \cite{loizou2017momentum,loizou2017linearly} and have already extended our results in several settings. For more details see  \cite{can2019accelerated,ma2018quasi,devraj2018zap,arnold2019reducing,devraj2018optimal}. 

\paragraph{On definitions of convergence presented in Table~\ref{ComparisonWIthHeavy}.}
In Table~\ref{ComparisonWIthHeavy} we present two main notions to characterize the convergence guarantees presented in the literature for the analysis of deterministic and stochastic heavy ball methods. These are, (i) Local/ Global convergence and (ii) Asymptotic/ Non-asymptotic convergence. For clarity, in this paragraph, we present these definitions of convergence.

\emph{Local convergence} we have only if the convergence guarantees depend on the starting point of the method. That is, if we can guarantee convergence only if $x^0$ is in a neighborhood of the optimal point $x^*$. If the method convergence for any starting point then we have \emph{global convergence.} 

We have \emph{asymptotic convergence} when the provided rate can be shown to hold only after specific number of iterations.  For example, we say that a deterministic method converges with asymptotic linear rate if there is $K>0$ such that for $k>K$ we have $\|x^k-x^*\|^2 <\rho ^k \|x^0-x^*\|$, where $\rho\in (0,1)$. We have a \emph{non-asymptotic convergence} if the rate satisfy the above definition for $K=0$.

\begin{table}
\begin{center}
\scalebox{0.75}{
\begin{tabular}{ |c | c| c| c | c| }
 \hline
 Method & Paper & Rate & Assumptions on $f$ & Convergence\\
 \hline \hline
\multirow{4}{*}{\begin{tabular}{c}Heavy Ball\\ (mGD) \end{tabular}} & Polyak, 1964 \cite{polyak1964some} & accelerated linear & $\mathcal{F}_{\mu, L}^{2,1}$ & local   \\ 
& Ghadimi et al, 2014 \cite{ghadimi2015global} & sublinear & $ \mathcal{F}_{0, L}^{1,1}$ & global \\
  & Ghadimi et al, 2014 \cite{ghadimi2015global} & linear  & $ \mathcal{F}_{\mu, L}^{1,1}$ & global \\
 & Lessard et al, 2016 \cite{lessard2016analysis} &  accelerated linear & $ \mathcal{F}_{\mu, L}^{1,1}$ + quadratic 
& global, asymptotic \\
 \hline
\multirow{3}{*}{\begin{tabular}{c}Stochastic \\ Heavy Ball  \\ (mSGD)\end{tabular}}  & Yang et al. 2016 \cite{yang2016unified} & sublinear & $ \mathcal{F}_{0, L}^{1,1}$ + bounded variance & global, non-asymptotic \\
& Gadat et al, 2016 \cite{gadat2016stochastic} & sublinear & $ \mathcal{F}_{\mu, L}^{1,1}$ + other assumptions & global, non-asymptotic \\
& \textbf{THIS CHAPTER} & {\bf see Table~\ref{OurResults}}  & $ \mathcal{F}_{0, L}^{1,1}$ + quadratic & global, non-asymptotic \\
 \hline
\end{tabular}}
\end{center}
\caption{Known complexity results for gradient descent with momentum (mGD, aka: heavy ball method),  and stochastic gradient descent with momentum (mSGD, aka: stochastic heavy ball method). We give the first linear and accelerated rates for mSGD. For full details on iteration complexity results we obtain, refer to Table~\ref{OurResults}.}
\label{ComparisonWIthHeavy}
\end{table}

\subsection{Connection to incremental gradient methods}

Assuming $\cD$ is discrete distribution (i.e., we sample from $M$ matrices, $\mS^1,\dots,\mS^{M}$, where $\mS^i$ is chosen with probability $p_i>0$. Here, $0<M \in R$ is fixed.), we can write the stochastic optimization problem \eqref{StochReform_IntroThesis} in the {\em finite-sum} form
\begin{equation}\label{eq:finite-sum}\min_{x\in \R^n} f(x) = \sum_{i=1}^{M} p_i f_{\mS^i}(x).\end{equation}
Choosing $x^0=x^1$, mSGD with fixed stepsize $\omega^k=\omega$ applied to \eqref{eq:finite-sum} can be written in the form
\begin{equation}
x^{k+1} =  x^k-\omega \sum_{t=1}^k\beta^{k-t} \nabla f_{\bS_t} (x^t)+\beta^k(x^1-x^0)
\overset{x^0=x^1}{=} x^k-\omega \sum_{t=1}^k\beta^{k-t} \nabla f_{\bS_t} (x^t),
\label{ExpansionOfUpdate}
\end{equation} 
where $\mS_t = \mS^i$ with probability $p_i$. Problem~\eqref{eq:finite-sum} can be also solved using incremental  average/aggregate gradient methods, such as the IAG method of Blatt et al.\ \cite{blatt2007convergent}. These methods have a similar form to \eqref{ExpansionOfUpdate}; \blue{however the past gradients are aggregated somewhat differently.} While \eqref{ExpansionOfUpdate} uses a geometric weighting of the gradients, the incremental average gradient methods use a uniform/arithmetic weighting. The  stochastic average gradient (SAG) method of Schmidt et al.\ \cite{schmidt2017minimizing} can be also written in a similar form. Note that mSGD  uses a geometric weighting of previous gradients, while the the incremental and stochastic average gradient methods use an arithmetic weighting. Incremental and incremental average gradient methods are widely studied algorithms for minimizing objective functions which can expressed as a sum of finite convex functions. For a review of key works on incremental methods and a detailed presentation of the connections with stochastic gradient descent, we refer the interested reader to the excellent survey of Bertsekas~ \cite{bertsekas2011incremental}; see also the work of Tseng~ \cite{tseng1998incremental}. 

In \cite{gurbuzbalaban2017convergence}, an incremental average gradient method with momentum was proposed for minimizing strongly convex functions. It was proved that the method converges to the optimum with linear rate. The rate is always worse than that of the no-momentum variant. However, it was  shown experimentally that in practice the method is faster, especially in problems with high condition number.  In our setting, the objective function has a very specifc structure \eqref{StochReform_IntroThesis}. It is not a finite sum problem as the distribution $\cD$ could be continous; and we also do not assume strong convexity. Thus, the convergence analysis of \cite{gurbuzbalaban2017convergence} can not be directly applied to our problem.

\subsection{Summary of contributions}

We now summarize the contributions of this chapter.

{\bf New momentum methods. } We study several  classes of stochastic optimization algorithms (SGD, SN, SPP and SDSA) {\em with momentum}, which we call mSGD, mSN, mSPP and mSDSA, respectively (see the first and second columns of Table~\ref{tbl:all_methods}). We do this in a simplified  setting with quadratic objectives where all of these algorithms are  equivalent. These methods can be seen as solving three related optimization problems: the stochastic optimization problem \eqref{StochReform_IntroThesis}, the best approximation problem \eqref{BestApproximation_IntroThesis} and its dual.  To the best of our knowledge, momentum variants of SN, SPP and SDSA were not analyzed before. 

{\footnotesize
\begin{table}[t!]
\begin{center}
\scalebox{0.8}{
\begin{tabular}{|c|c|c|}
\hline
&& \\
 \begin{tabular}{ccc}no momentum \\ ($\beta=0$) \end{tabular} & \begin{tabular}{cc} momentum \\ ($\beta\geq 0$) \end{tabular}  & \begin{tabular}{cc} stochastic momentum \\ ($\beta\geq 0$) \end{tabular} \\
&& \\
\hline
\hline
&& \\
 \begin{tabular}{c} SGD \cite[$\omega=1$]{gower2015randomized},  \cite[$\omega>0$]{ASDA}  \\ \\
  $ x^{k+1}= x^k - \omega \nabla f_{\mS_k}(x^k)$  
   \end{tabular}  & \begin{tabular}{ccc} {\bf mSGD} [Sec~\ref{sec:primal}]\\ \\
$ + \beta (x^k-x^{k-1})$ \end{tabular} & 
 \begin{tabular}{c}
 {\bf smSGD} [Sec~\ref{sec:sm}] \\ \\
 $ + n \beta e_{i_k}^\top (x^k-x^{k-1}) e_{i_k}$
 \end{tabular} 
  \\
&& \\
\hline
&& \\
  \begin{tabular}{c} 
 SN \cite{ASDA} \\
 \\
 $  x^{k+1}= x^k - \omega (\nabla^2 f_{\mS_k}(x^k))^{\dagger_\mB} \nabla f_{\mS_k}(x^k)$
  \end{tabular}
 & \begin{tabular}{c}{\bf mSN} [Sec~\ref{sec:primal}] \\\\
 $ + \beta (x^k-x^{k-1})$
\end{tabular} 
 & 
 \begin{tabular}{c}
 {\bf smSN} [Sec~\ref{sec:sm}]\\ \\
 $ + n \beta e_{i_k}^\top (x^k-x^{k-1}) e_{i_k}$
 \end{tabular} 
 \\
&& \\
\hline
&& \\
  \begin{tabular}{c}  
 SPP \cite{ASDA} \\ \\
$ x^{k+1}=  \arg\min_x \left\{ f_{\mS_k}(x) + \frac{1-\omega}{2\omega}\|x-x^k\|_{\mB}^2\right\}
$
  \end{tabular}
&   \begin{tabular}{c}{\bf mSPP} [Sec~\ref{sec:primal}] \\ \\
 $ + \beta (x^k-x^{k-1})$
\end{tabular} 
&  \begin{tabular}{c}
 {\bf smSPP} [Sec~\ref{sec:sm}]\\ \\
 $ + n \beta e_{i_k}^\top (x^k-x^{k-1}) e_{i_k}$
 \end{tabular} 
\\
&& \\
\hline
&& \\
  \begin{tabular}{ccc}SDSA \cite[$\omega=1$]{gower2015stochastic} \\ \\ 
  $y^{k+1} = y^k + \mS_k \lambda^k$ \end{tabular}  
  &  \begin{tabular}{c}{\bf mSDSA} [Sec~\ref{sec:dual}] \\ \\
 $ + \beta (y^k-y^{k-1})$
\end{tabular} 
  &   \\
&& \\
\hline
\end{tabular}
}
\end{center}

\caption{All methods analyzed in this chapter. The methods highlighted in bold (with momentum and stochastic momentum) are new. SGD = Stochastic Gradient Descent,  SN = Stochastic Newton, SPP = Stochastic Proximal Point, SDSA = Stochastic Dual Subspace Ascent. At iteration $k$, matrix $\mS_k$ is drawn in an i.i.d.\ fashion from distribution $\cD$, and a stochastic step is performed.}
\label{tbl:all_methods}
\end{table}
}

{\bf Linear rate.} We prove several (global and non-asymptotic) linear convergence results for our primal momentum methods mSGD/mSN/mSPP. First, we establish a linear rate for the decay of $\E{\|x^k-x^*\|_\mB^2}$ to zero, for a range of stepsizes $\omega> 0$ and momentum parameters $\beta\geq 0$. We show that the same rate holds for the decay of the expected function values $\E{f(x^k)-f(x^*)}$ of \eqref{StochReform_IntroThesis} to zero. Further, the same rate holds for  mSDSA, in particular, this is for the convergence of the dual objective to the optimum. For a summary of these results, and pointers to the \blue{relevant} theorems, refer to lines 1, 2 and 6 of Table~\ref{OurResults}. Unfortunately, the theoretical rate for all our momentum methods is optimized for $\beta = 0$, and  gets worse as the momentum parameter increases. However, no prior linear rate for any of these methods with momentum are known. We give the first linear convergence rate for SGD with momentum (i.e., for the stochastic heavy ball method). 

\begin{table}[t!]
\begin{center}
\scalebox{0.7}{
\begin{tabular}{ |c|c|c|c|c|c| }
 \hline
 Algorithm & $\omega$& \begin{tabular}{c} momentum \\ $\beta$ \end{tabular} & \begin{tabular}{c} Quantity \\ converging to 0 \end{tabular} & \begin{tabular}{c}Rate\\(all: global, non-asymptotic)\end{tabular} & Theorem \\
 \hline
 \hline
 mSGD/mSN/mSPP & $(0,2)$ &  $ \geq 0$ &   $\Exp\left[\|x^k-x^*\|^2_{\bB}\right]$ & linear &\ref{L2}\\
 \hline
mSGD/mSN/mSPP & $(0,2)$ &  $\geq 0$ &   $\Exp[f(x^k) - f(x^*)]$  & linear &\ref{L2}\\
 \hline
mSGD/mSN/mSPP & $(0,2)$ & $\geq 0$ & $\E{f(\hat{x}^k) - f(x^*)}$  & sublinear: $\cO(1/k)$& \ref{cesaro} \\
 \hline 
mSGD/mSN/mSPP  & 1 & $\left(1-\sqrt{0.99 \lambda_{\min}^+}\right)^2 $ & $\|\Exp[x^k-x^*]\|^2_{\bB}$ &  accelerated linear & \ref{theoremheavyball}\\
 \hline
  mSGD/mSN/mSPP  & $\frac{1}{\lambda_{\max}}$ & $ \displaystyle  \left(1-\sqrt{0.99 \frac{\lambda_{\min}^+}{\lambda_{\max}}}\right)^2 $ & $\|\Exp[x^k-x^*]\|^2_{\bB}$   &  \begin{tabular}{c}accelerated linear\\ (better than for $\omega=1$) \end{tabular} & \ref{theoremheavyball}\\
  \hline
  mSDSA & $(0,2)$ & $\geq 0$ & $\E{D(y^*) - D(y^k)}$ & linear & \ref{thm:dual-conv}  \\
  \hline
 smSGD/smSN/smSPP & $(0,2)$ & $\geq 0$ &   $\Exp[\|x^k-x^*\|^2_{\bB}]$ & linear &  \ref{thm:DSHB-L2}\\
 \hline
  smSGD/smSN/smSPP  & $(0,2)$ & $\geq 0$ &   $\Exp[f(x^k)-f(x^*)]$  & linear & \ref{thm:DSHB-L2}\\
 \hline
\end{tabular}}
\end{center}
\caption{Summary of the iteration complexity results obtained in this chapter. Parameters of the methods: $\omega$ (stepsize) and $\beta$ (momentum term). In all cases, $x^*=\Pi_{\cL,\bB}(x^0)$ is the solution of the best approximation problem.  Theorem~\ref{cesaro} refers to Ces\`{a}ro averages:  $\hat{x}^k = \frac{1}{k}\sum_{t=0}^{k-1}x^t$. Theorem~\ref{thm:dual-conv} refers to suboptimality in dual function values ($D$ is the dual function).}
\label{OurResults}
\end{table}

{\bf Accelerated linear rate.} We then study the decay of the larger quantity $\|\E{x^k-x^*}\|_\mB^2$  to zero.  In this case, we establish  an {\em accelerated} linear rate, which depends on the square root of the condition number (of the Hessian of $f$). This is a quadratic speedup when compared to  the no-momentum methods as these depend on the condition number. See lines 4 and 5 of Table~\ref{OurResults}. To the best of our knowledge, this is the first time an  accelerated rate is obtained for the stochastic heavy ball method (mSGD).  Note that there are no global non-asymptotic accelerated linear rates proved even in the non-stochastic setting (i.e., for the heavy ball method). Moreover, we are not aware of any accelerated linear convergence results for  the stochastic proximal point method.

{\bf Sublinear rate for Ces\`{a}ro averages.} We show that the Ces\`{a}ro averages, $\hat{x}^k = \frac{1}{k}\sum_{t=0}^{k-1}x^t$, of all primal momentum methods enjoy a sublinear $\cO(1/k)$ rate (see line 3 of Table~\ref{OurResults}). This holds under weaker assumptions than those which lead to the linear convergence rate. 

{\bf Primal-dual correspondence.} We show that SGD, SN and SPP with momentum arise as affine images of SDSA with momentum (see Theorem~\ref{thm:dual_corresp}). This extends the result of \cite{gower2015stochastic} where this was shown for the no-momentum methods ($\beta=0$) and in the special case of the unit stepsize ($\omega=1$).

{\bf Stochastic momentum.} We propose a new momentum strategy, which we call {\em stochastic momentum}.   Stochastic momentum is a stochastic (coordinate-wise) approximation of the deterministic momentum, and hence is much less costly, which in some situations leads to computational savings in each iteration. On the other hand, the additional noise introduced this way increases the number of iterations needed for convergence. We analyze the SGD, SN and SPP methods with stochastic momentum, and prove linear convergence rates. We prove that in some settings  the overall complexity of SGD with stochastic momentum is better than the overall complexity of SGD with momentum. For instance, this is the case if we consider the randomized Kaczmarz (RK) method as a special case of SGD, and if $\mA$ is sparse.

{\bf Space for generalizations.} We hope that the present work can serve as a starting point for the development of SN, SPP and SDSA methods with momentum for more general classes (beyond special quadratics) of convex and perhaps also nonconvex optimization problems. In such more general settings, however, the symmetry which implies equivalence of these algorithms will break, and hence a different analysis will be needed for each method.

\section{Primal Methods with Momentum} \label{sec:primal}
 Applied to problem \eqref{StochReform_IntroThesis}, i.e., 
$\min_{x\in \R^n} f(x) = \E{f_\mS(x)},$
  the gradient descent method with momentum (also known as the  heavy ball method) of Polyak \cite{polyak1964some, polyak1987introduction} takes the form
\begin{equation}
\label{HB}
x^{k+1} = x^k - \omega \nabla f(x^k) + \beta(x^k - x^{k-1}),
\end{equation}
where $\omega>0$ is a stepsize and $\beta\geq 0$ is a momentum parameter. Instead of marrying the momentum term with gradient descent, we can marry it with SGD. This leads to SGD with momentum (mSGD), also known as the {\em stochastic heavy ball method}:
\begin{equation}\label{eq:SHB-intro} x^{k+1} = x^k - \omega \nabla f_{\mS_k}(x^k) + \beta(x^k-x^{k-1}).\end{equation}

Since SGD is equivalent to SN and SPP, this way we obtain momentum variants of the stochastic Newton  (mSN) and  stochastic proximal point (mSPP) methods.  The method is formally described below:

\begin{algorithm}[H]
  \caption{mSGD / mSN / mSPP}
  \label{anjsdnaodala}
  \begin{algorithmic}[1]
    \Require{Distribution $\mathcal{D}$ from which method samples matrices; positive definite matrix $\bB \in \R^{n\times n}$; stepsize/relaxation parameter $\omega \in \R$; the heavy ball/momentum parameter $\beta$.}
    \Ensure{Choose initial points $x^0,x^1\in \R^n$}
 \For{$k=1,2,\cdots$}
 \State Generate a fresh sample $\bS_k \sim {\cal D}$
 \State Set $x^{k+1}=x^k -\omega \nabla f_{\bS_k}(x^k) + \beta(x^k - x^{k-1}) $
 \EndFor
 \State {\bf Output:} The last iterate $x^k$
 \end{algorithmic}
\end{algorithm}

To the best of our knowledge, momentum variants of SN and SPP were not considered  in the literature before. Moreover, as far as we know, there are no momentum variants of even deterministic variants of \eqref{SNM_IntroThesis}, \eqref{SPPM_IntroThesis} and \eqref{SPM_IntroThesis}, such as incremental or batch Newton method,  incremental or batch proximal point method and  incremental or batch projection method; not even for a problem formulated differently.

In the rest of this section we state our convergence results for mSGD/mSN/mSPP. 

\subsection{Convergence of iterates and function values: linear rate}
\red{In this section we study the convergence rate of the quantity $\Exp[\|x^k-x^*\|_{\mB}^2]$ to zero for mSGD/mSN/mSPP}. We show that for a range  of stepsize parameters $\omega > 0$ and  momentum terms $\beta \geq 0$, the method enjoys global  linear convergence rate; see \eqref{eq:nfiug582}. To the best of our knowledge, these results are the first of their kind for the stochastic heavy ball method.
\red{As a corollary of this result, we obtain convergence of the expected function values; see \eqref{eq:b78gbf80hf}.}

\begin{thm}
\label{L2}
Choose $x^0= x^1\in \R^n$.  Assume exactness. Let $\{x^k\}_{k=0}^\infty$ be the sequence of random iterates produced by mSGD/mSN/mSPP.  Assume $0< \omega < 2$ and $\beta \geq 0$ and that the expressions
\[a_1 \eqdef 1+3\beta+2\beta^2 - (\omega(2-\omega) +\omega\beta)\lambda_{\min}^+, \qquad \text{and}\qquad
a_2 \eqdef \beta +2\beta^2 + \omega \beta \lambda_{\max}\]
satisfy $a_1+a_2<1$. Let $x^* = \Pi_{\mathcal{L},\bB}(x^0)$. Then 
\begin{equation}\label{eq:nfiug582}\Exp\left[\|x^{k}-x^*\|^2_{\bB}\right] \leq q^k (1+\delta)  \|x^{0}-x^*\|^2_{\bB}\end{equation}
and 
\begin{equation}\label{eq:b78gbf80hf}\Exp\left[f(x^k)\right] \leq q^k  \frac{\lambda_{\max}}{2} (1+\delta) \|x^{0}-x^*\|^2_{\bB},\end{equation}
where  $q=\frac{a_1+\sqrt{a_1^2+4a_2}}{2}$ and $\delta=q-a_1$. Moreover, $a_1+a_2 \leq q <1$.
\end{thm}
\begin{proof} See Section~\ref{app:1}.
\end{proof}

In the above theorem we obtain a global linear rate. To the best of our knowledge, this is the first time that linear rate is established for a stochastic variant of the heavy ball method (mSGD) in any setting. All existing results are sublinear. These seem to be  the first momentum variants of SN and SPP methods.

If we choose $\omega \in (0,2)$, then the condition $a_1+a_2<1$ is satisfied for all   \begin{equation}\label{rangesSHB} 0\leq \beta< \frac{1}{8} \left( -4+\omega \lambda_{\min}^+-\omega \lambda_{\max} +\sqrt{(4-\omega \lambda_{\min}^++\omega \lambda_{\max})^2+16\omega (2-\omega) \lambda_{\min}^+ }\right).\end{equation}

If $\beta=0$, mSGD reduces to  SGD analyzed in \cite{ASDA}. In this special case, $q = 1-\omega(2-\omega)\lambda_{\min}^+$, which is the rate established in \cite{ASDA}. Hence, our result is more general.

Let $q(\beta)$ be the rate as a function of $\beta$. Note that since $\beta\geq 0$, we have
\begin{eqnarray} q(\beta) &\geq & a_1 + a_2 \notag \\
&=& 1 + 4\beta + 4\beta^2 + \omega\beta(\lambda_{\max}-\lambda_{\min}^+) - \omega(2-\omega)\lambda_{\min}^+ \notag \\
&\geq & 1-\omega(2-\omega)\lambda_{\min}^+ = q(0).\label{eq:qbeta}\end{eqnarray}
Clearly, the lower bound on $q$ is an increasing function of $\beta$. 
Also, for any $\beta$ the rate is always inferior to that of SGD ($\beta=0$). It is an open problem whether one can prove a strictly better rate for mSGD than for SGD.

Our next proposition states that $\Pi_{\cL,\mB}(x^k) = x^*$ \blue{(recall that $x^*\eqdef \Pi_{\cL, \mB}(x^0)$)} for all iterations $k$ of mSGD. \blue{This invariance property plays an important role in our convergence analysis, } and ``explains'' why the algorithm to converges to $x^*$. 

\begin{prop}\label{prop:projections}
Let $x^0=x^1\in \R^n$ be the starting points of the mSGD method and let $\{x^k\}$ be the random iterates generated by mSGD. Then $\Pi_{\cL,\bB}(x^k)=\Pi_{\cL,\bB}(x^0)$ for all $k\geq 0$.
\end{prop}
\begin{proof} Note that in view of \eqref{f_s_IntroThesis}, $ \nabla f_{\mS}(x) = \mB^{-1}\mA^\top \mH (\mA x - b) \in {\rm Range}(\mB^{-1}\mA^\top)$. Since 
\[x^{k+1} = x^k -\omega \nabla f_{\bS_k}(x^k) + \beta(x^k - x^{k-1}),\]
and since $x^0=x^1$, it can shown by induction that $x^k \in x^0 + {\rm Range}(\mB^{-1}\mA^\top)$ for all $k$. However, ${\rm Range}(\mB^{-1}\mA^\top)$ is the orthogonal complement to ${\rm Null}(\mA)$ in the $\mB$-inner product. Since $\cL$ is parallel to ${\rm Null}(\mA)$, vectors $x^k$ must have the same $\mB$-projection onto $\cL$ for all $k$: $\Pi_{\cL, \mB}(x^0) = x^*$.
\end{proof}

\blue{This property also intuitively explains  why mSGD converges to the projection of the {\em starting} point onto $\cL$. Indeed, one may ask: why is the starting point special? After all, each iterate depends on the previous two iterates only, and all older iterates, including the starting point $x^0$, seem to be eventually ``forgotten''. Still, the iterative process has the property that all iterates live in the affine space passing through $x^0$ and orthogonal to $\cL$, which means that the projection of {\em all} iterates onto $\cL$ is identical.
}

\subsection{Ces\`{a}ro average: sublinear rate without exactness assumption}

In this section we present the convergence analysis of the function values computed on the Ces\`{a}ro average. Again our results are global in nature. To the best of our knowledge \blue{these} are the first results that show $\cO(1/k)$ convergence of the stochastic heavy ball method. Existing results apply in more general settings at the expense of slower rates. In particular, \cite{yang2016unified} and \cite{gadat2016stochastic} get $\cO(1/\sqrt{k})$ and $\cO(1/k^{\beta})$ convergence \blue{when $\beta \in (0,1)$}, respectively. When $\beta=1$, \cite{gadat2016stochastic} gets $\cO(1/\log(k))$ rate.

\begin{thm}
\label{cesaro}
Choose $x^0=x^1$ and let $\{x^k\}_{k=0}^\infty$ be the random iterates produced by mSGD/mSN/ mSPP, where the momentum parameter $0\leq \beta <1$ and relaxation parameter (stepsize) $\omega > 0$ satisfy $\omega + 2\beta <2$. Let $x^*$ be any vector satisfying $f(x^*)=0$. If we let $\hat{x}^k=\frac{1}{k}\sum_{t=1}^{k}x^t$, then
$$\Exp\left[f(\hat{x}^k)\right] \leq \frac{(1-\beta)^2\|x^0-x^*\|_{\mB}^2 + 2\omega \beta f(x^0)}{2\omega(2-2\beta-\omega) k}.$$
\end{thm}

\begin{proof} See Section~\ref{app:acc212}.
\end{proof}

In the special case of $\beta=0$, the above theorem gives the rate
$$\Exp\left[f(\hat{x}^k)\right] \leq \frac{\|x^0-x^*\|^2_{\bB}}{2\omega (2 -\omega) k} .$$
This is the convergence rate for Ces\`{a}ro averages of the ``basic method'' (i.e., SGD) established in \cite{ASDA}. 

Our proof strategy  is similar to  \cite{ghadimi2015global} in which the first global convergence analysis of the (deterministic) heavy ball method was presented. There it was shown that when the objective function has a Lipschitz continuous gradient, the Ces\`{a}ro averages of the iterates converge to the optimum at a rate of $\cO(1/k)$. To the best of our knowledge, there are no results in the literature that prove the same rate of convergence in the stochastic case for any class of objective functions.  

In \cite{yang2016unified} the authors analyzed  mSGD for general Lipshitz continuous convex objective functions (with bounded variance) and proved the {\em sublinear} rate $\cO(1/\sqrt{k})$. In \cite{gadat2016stochastic}, a complexity analysis is provided for the case of quadratic strongly convex smooth coercive functions. A  sublinear convergence  rate of $\cO(1/k^\beta)$, where $\beta \in (0,1)$, was proved. In contrast to our results, where we assume fixed stepsize $\omega$, both papers analyze mSGD with diminishing stepsizes.

\subsection{Accelerated linear rate for expected iterates}

In this section we show that by a proper combination of the relaxation (stepsize) parameter  $\omega$ and the momentum parameter $\beta$, mSGD/mSN/mSPP enjoy an {\em accelerated} linear convergence rate in mean.  \blue{That is, while SGD  needs $\cO(\theta\log (1/\epsilon))$ iterations to find $x^k$ such that $\|\E{x^k}-x^*\|_\mB^2 \leq \epsilon$ \cite{ASDA},  mSGD only needs $\cO(\sqrt{\theta}\log (1/\epsilon))$ iterations (see Theorem~\ref{theoremheavyball}(ii)), where $\theta = \lambda_{\max}/\lambda_{\min}^+$.  The word {\em acceleration} typically refers to improvement from a leading factor of $\theta$ to $\sqrt{\theta}$, which is significant in the ill-conditioned case, i.e., when $\theta$ is very large. In other words, the linear rate  $(1-\sqrt{0.99/\theta})^k$ of mSGD  is much better than linear rate $(1-1/\theta)^k$ of SGD, which in view of \eqref{eq:qbeta} is better  than the linear rate of mSGD  established in \eqref{eq:nfiug582} (for a different quantity converging to zero) . }

\begin{thm}
\label{theoremheavyball}
Assume exactness. Let $\{x^k\}_{k=0}^{\infty}$ be the sequence of random iterates produced by mSGD / mSN / mSPP, started with $x^0, x^1 \in \R^n$ satisfying the relation $x^0-x^1 \in {\rm Range}(\bB^{-1} \bA^ \top)$, with relaxation parameter (stepsize)  $0<\omega \leq1/\lambda_{\max}$ and momentum parameter  $(1-\sqrt{\omega \lambda_{\min}^+})^2 < \beta <1$. Let $x^* = \Pi_{\cL,\mB}(x^0)$. Then there exists constant $C >0$ such that for all $k\geq0$ we have 
$$\left\|\Exp\left[x^{k} -x^*\right] \right\|_{\bB}^2  \leq \beta^k C.$$

\begin{itemize}
\item[(i)] If we choose $ \omega= 1$ and $\beta= \left(1- \sqrt{0.99 \lambda_{\min}^+}\right) ^2$ then
 $\left\|\Exp\left[x^{k} -x^*\right] \right\|_{\bB}^2  \leq \beta^k C$
and the iteration complexity becomes
$ \cO\left(\sqrt{1/ \lambda_{\min}^+}\log(1/\epsilon)\right)$.
\item[(ii)] If we choose $ \omega= 1/\lambda_{\max}$ and $\beta= \left(1- \sqrt{ \frac{0.99\lambda_{\min}^+}{\lambda_{\max}}}\right)^2$ then
$\left\|\Exp\left[x^{k} -x^*\right] \right\|_{\bB}^2  \leq \beta^k C$
and the iteration complexity becomes
$\cO \left(\sqrt{\lambda_{\max}/ \lambda_{\min}^+}\log(1/\epsilon)\right)$.

\end{itemize}
\end{thm}
\begin{proof} See Section~\ref{app:acc}.
\end{proof}

Note that the convergence factor is precisely equal to the value of the momentum parameter $\beta$. Let $x$ be any random vector in $\R^n$ with finite mean $\Exp[x]$, and $x^*\in \R^n$ is any reference vector (for instance, any solution of $\mA x = b$). Then we have the identity (see, for instance \cite{gower2015randomized})
\begin{equation}
\label{weakstrong}
\E{\|x-x^*\|_{\bB}^2} = \left\|\E{x-x^*} \right\|_{\bB}^2 + \E{\|x-\Exp[x]\|^2_{\bB}}.
\end{equation}
\red{This means that the quantity $\E{\|x-x^*\|_{\bB}^2}$ appearing in the convergence result of Theorem~\ref{L2} is larger than 
$\|\E{x-x^*}\|_{\bB}^2$ appearing in the the convergence result of Theorem~\ref{theoremheavyball}, and hence harder to push to zero. As a corollary, the convergence  rate of $\E{\|x-x^*\|_{\bB}^2}$ to zero established in Theorem~\ref{L2})   implies the same rate for the convergence of  $\|\E{x-x^*}\|_{\bB}^2$ to zero. However, note that in Theorem~\ref{theoremheavyball} we have established an {\em accelerated} rate for $\|\E{x-x^*}\|_{\bB}^2$.  A similar theorem, also obtaining an accelerated rate for $\|\E{x-x^*}\|_{\bB}^2$, was established in \cite{ASDA} for an accelerated variant of SGD in the sense of Nesterov.}

\section{Dual Methods with Momentum} \label{sec:dual}

In the previous sections we focused on methods for solving the stochastic optimization problem \eqref{StochReform_IntroThesis} and the best approximation problem~\eqref{BestApproximation_IntroThesis}. In this section we focus on the dual of the best approximation problem, and propose a momentum variant of SDSA, which we call mSDSA. 

\begin{algorithm}[H]
  \caption{Stochastic Dual Subspace Ascent with Momentum (mSDSA)}
  \begin{algorithmic}[1]
    \Require{Distribution $\mathcal{D}$ from which method samples matrices; positive definite matrix $\bB \in \R^{n\times n}$; stepsize/relaxation parameter $\omega \in \R$ the heavy ball/momentum parameter $\beta$. }
    \Ensure{Choose initial points $y^0 =  y^1 = 0 \in \R^m$}
 \For{$k=1,2,\cdots$}
 \State Draw a fresh $\bS_k \sim \cD$
 \State Set $\lambda^k = \left(\mS_k^\top \bA \bB^{-1}\bA^\top \mS_k \right)^\dagger\bS_k^\top \left(b-\bA(x^0 + \bB^{-1}\bA^\top y^k) \right)$
 \State Set $y^{k+1}=y^k+\omega \mS_k \lambda^k +\beta(y^k-y^{k-1})$
 \EndFor
 \State {\bf Output:} last iterate $y^k$
 \end{algorithmic}
\end{algorithm}

\subsection{Correspondence between primal and dual methods}

In our first result we show that the random iterates of the mSGD/mSN/mSPP methods arise as an affine image of mSDSA under the mapping $\phi$ defined in \eqref{Corresp_IntroThesis1}.

\begin{thm}[Correspondence Between Primal and Dual Methods]  \label{thm:dual_corresp} Let $x^0=x^1$ and let $\{x^k\}$ be the iterates of mSGD/mSN/mSPP. Let $y^0=y^1=0$, and let $\{y^k\}$ be the iterates of mSDSA.   Assume that the  methods use the same stepsize $\omega > 0$, momentum parameter $\beta\geq 0$, and the same sequence of random matrices $\mS_k$.  Then  
\[x^k = \phi(y^k) = x^0 + \mB^{-1} \mA^\top y^k\]
for all $k$. That is, the primal iterates arise as affine images of the dual iterates.
\end{thm}
\begin{proof}
First note that
\begin{eqnarray*}\nabla f_{\mS_k}(\phi(y^k)) &\overset{\eqref{Gradf_S_IntroThesis}}{=}& \mB^{-1}\mA^\top  \mS_k ( \mS_k^\top \mA \mB^{-1} \mA^\top \mS_k)^\dagger \mS_k^\top (\mA \phi(y^k) - b) \; =\;  - \mB^{-1}\mA^\top  \mS_k \lambda^k.
\end{eqnarray*}
We now use this to show that
\begin{eqnarray*}
\phi(y^{k+1}) &\overset{\eqref{Corresp_IntroThesis}}{=}& x^0  + \bB^{-1}\bA^\top y^{k+1}\\
&=& x^0+\bB^{-1}\bA^\top\left[y^k+\omega \bS_k \lambda^k +\beta(y^k-y^{k-1})\right]\\
&=& \underbrace{x^0+\bB^{-1}\bA^\top y^k}_{\phi(y^k)}+\omega \underbrace{\bB^{-1}\bA^\top \bS_k \lambda^k}_{-\nabla f_{\bS_k}(\phi(y^k))} +\beta \bB^{-1}\bA^\top (y^k-y^{k-1})\\
&=& \phi(y^k)- \omega \nabla f_{\bS_k}(\phi(y^k)) +\beta  (\bB^{-1}\bA^\top y^k-\bB^{-1}\bA^\top y^{k-1})\\
&\overset{\eqref{Corresp_IntroThesis}}{=}& \phi(y^k) -\omega \nabla f_{\bS_k}(\phi(y^k)) +\beta  (\phi(y^k) - \phi(y^{k-1})).
\end{eqnarray*}
So, the sequence of vectors $\{\phi(y^k)\}$ mSDSA satisfies the same recursion of degree as the sequence $\{x^k\}$ defined by mSGD.  It remains to check that the first two elements of both recursions coincide. Indeed, since $y^0=y^1=0$ and $x^0=x^1$, we have $x^0 = \phi(0) =  \phi(y^0)$, and $x^1 = x^0= \phi(0) =  \phi(y^1)$.
\end{proof}

\subsection{Convergence}

We are now ready to state a linear convergence convergence result describing the behavior of mSDSA in terms of the dual function values $D(y^k)$.

\begin{thm}[Convergence of dual objective] 
\label{thm:dual-conv}
Choose $y^0= y^1\in \R^n$.  Assume exactness. Let $\{y^k\}_{k=0}^\infty$ be the sequence of random iterates produced by mSDSA.  Assume $0\leq \omega \leq 2$ and $\beta \geq 0$ and that the expressions
\[a_1 \eqdef 1+3\beta+2\beta^2 - (\omega(2-\omega) +\omega\beta)\lambda_{\min}^+, \qquad \text{and}\qquad
a_2 \eqdef \beta +2\beta^2 + \omega \beta \lambda_{\max}\]
satisfy $a_1+a_2<1$. Let $x^* = \Pi_{\mathcal{L},\bB}(x^0)$ and let $y^*$ be any dual optimal solution. Then 
\begin{equation}
\label{eq:nfiug5822}
\Exp[D(y^*)-D(y^k)] \leq q^k (1+\delta) \left[D(y^*)-D(y^0)\right]
\end{equation}
where  $q=\frac{a_1+\sqrt{a_1^2+4a_2}}{2}$ and $\delta=q-a_1$. Moreover, $a_1+a_2 \leq q <1$.
\end{thm}
\begin{proof} This follows by applying Theorem~\ref{L2} together with Theorem \ref{thm:dual_corresp} and the identity $\frac{1}{2}\|x^k-x^0\|^2_\mB = D(y^*) - D(y^k)$.
\end{proof}

Note that for $\beta=0$, mSDSA simplifies to SDSA. Also recall that for unit stepsize ($\omega=1$), SDSA was analyzed in  \cite{gower2015randomized}. In the $\omega=1$ and $\beta=0$ case, our result specializes to that established in  \cite{gower2015randomized}. Following similar arguments to those in  \cite{gower2015randomized},  the same rate of convergence can be proved for the duality gap $\Exp[P(x^k)-D(y^k)]$.

\section{Methods  with Stochastic Momentum} \label{sec:sm}

To motivate {\em stochastic momentum}, for simplicity fix $\mB=\mI$, and assume that $\mS_k$ is chosen as the $j$th random unit coordinate vector  of $\R^m$ with probability $p_j>0$. In this case, SGD \eqref{SGD_IntroThesis} reduces to the randomized Kaczmarz method for solving the linear system $\mA x = b$, first analyzed for $p_j\sim \|\mA_{j:}\|^2$ by Strohmer and Vershynin \cite{RK}. 

In this case, mSGD becomes the {\em randomized Kaczmarz method with momentum} (mRK), and the iteration \eqref{eq:SHB-intro}  takes the explicit form
\[x^{k+1} = x^k - \omega \frac{\bA_{j:} x^k -b_j}{\|\bA_{j:}\|^2} \bA_{j:}^ \top + \beta(x^k-x^{k-1}).\]
Note that the cost of one iteration of this method is $\cO(\|\mA_{j:}\|_0 + n)$, where the cardinality term $\|\mA_{j:}\|_0$ comes from the stochastic gradient part, and $n$ comes from the momentum part.
When $\mA$ is sparse, the second term will dominate. Similar considerations apply for many other (but clearly not all) distributions~$\cD$. 

In such circumstances, we propose to replace the expensive-to-compute momentum term by a cheap-to-compute stochastic approximation \blue{term}. In particular, we let $i_k$ be chosen from $[n]$ uniformly at random,  and replace $x^k-x^{k-1}$ with $v_{i_k} \eqdef e_{i_k}^\top (x^k-x^{k-1})e_{i_k}$,  where $e_{i_k} \in \R^n$ is the $i_k$-th unit basis vector in $\R^n$, and $\beta$ with $\gamma \eqdef n\beta$. Note that $v_{i_k}$ can be computed in $\cO(1)$ time. Moreover, \begin{equation}\label{eq:b877gfff}\Exp_{i_k}[\gamma v_{i_k}] =\beta( x^k-x^{k-1}).\end{equation} Hence,  we replace the momentum term by  an unbiased estimator, which allows us to cut the cost to $\cO(\|\mA_{j:}\|_0)$.

\subsection{Primal methods with stochastic momentum}

We now propose a  variant of the SGD/SN/SPP methods employing stochastic momentum  (smSGD/smSN/smSPP). Since SGD, SN and SPP are equivalent, we will describe the development from the perspective of SGD.
In particular, we propose the following method:
\begin{equation}\label{eq:DSHB-intro} x^{k+1} = x^k - \omega \nabla f_{\mS_k}(x^k) + \gamma e_{i_k}^\top (x^k-x^{k-1})e_{i_k}.\end{equation}

The method is formalized below:

\begin{algorithm}[H]
  \caption{smSGD/smSN/smSPP}
  \begin{algorithmic}[1]
    \Require{Distribution $\mathcal{D}$ from which the method samples matrices;  stepsize/relaxation parameter $\omega \in \R$ the heavy ball/momentum parameter $\beta$.}
    \Ensure{ Choose initial points $x^1=x^0  \in \R^n$; set $\bB  = \mI\in \R^{n\times n}$}
 \For{$k=1,2,\cdots$}
 \State Generate a fresh sample $\bS_k \sim {\cal D}$.
 \State Pick $i_k\in [n]$ uniformly at random 
 \State Set $x^{k+1}=x^k -\omega \nabla f_{\bS_k}(x^k) + \gamma e_{i_k}^\top (x^k - x^{k-1})e_{i_k} $
 \EndFor
 \State {\bf Output:} The last iterate $x^k$
 \end{algorithmic}
\end{algorithm}

\subsection{Convergence}

In the next result we establish linear convergence of smSGD/smSN/smSPP. For this we will require the matrix $\mB$ to be equal to the identity matrix.

\begin{thm}\label{thm:DSHB-L2} Choose $x^0= x^1\in \R^n$.  Assume exactness. Let $\mB=\mI$.  Let $\{x^k\}_{k=0}^\infty$ be the sequence of random iterates produced by smSGD/smSN/smSPP.  Assume $0< \omega < 2$ and $\gamma \geq 0$ and that the expressions
\begin{equation}\label{eq:98ys8h89dh} a_1 \eqdef 1+3\frac{\gamma}{n}+2\frac{\gamma^2}{n} - \left(\omega(2-\omega) +\omega\frac{\gamma}{n}\right)\lambda_{\min}^+, \qquad \text{and}\qquad
a_2 \eqdef \frac{1}{n}(\gamma +2\gamma^2 + \omega \gamma \lambda_{\max}) 
\end{equation}
satisfy $a_1+a_2<1$. Let $x^* = \Pi_{\mathcal{L},\bI}(x^0)$. Then 
\begin{equation}\label{eq:nfiug582X}\Exp\left[\|x^{k}-x^*\|^2\right] \leq q^k (1+\delta)  \|x^{0}-x^*\|^2
\end{equation}
and 
$\Exp\left[f(x^k)\right] \leq q^k  \frac{\lambda_{\max}}{2} (1+\delta) \|x^{0}-x^*\|^2,$
where  $q\eqdef \frac{a_1+\sqrt{a_1^2+4a_2}}{2}$ and $\delta\eqdef q-a_1$. Moreover, $ a_1+a_2 \leq q < 1$.
\end{thm}

\begin{proof} See Section~\ref{app:7}.
\end{proof}

It is straightforward to see that if we choose $\omega \in (0,2)$, then the condition $a_1+a_2<1$ is satisfied for all $\gamma$ belonging to the interval   \[ 0\leq \gamma < \frac{1}{8}\left (-4+\omega \lambda_{\min}^+-\omega \lambda_{\max}+\sqrt{(4-\omega \lambda_{\min}^++\omega \lambda_{\max})^2+16n\omega (2-\omega)  \lambda_{\min}^+}\right).\]
The  upper bound  is similar to that for mSGD/mSN/mSPP; the only difference is an extra factor of $n$ next to the constant~16.

\subsection{Momentum versus stochastic momentum}
\label{comparison}

As indicated above, if we wish to compare mSGD with momentum parameter $\beta$ to smSGD with momentum parameter $\gamma$, it makes sense to set $\gamma = \beta n$. Indeed, this is because in view of \eqref{eq:b877gfff}, the momentum term in smSGD will then be an unbiased estimator of the deterministic momentum term used in mSGD.

Let $q(\beta)$ be the convergence constant for mSGD with stepsize $\omega=1$ and an admissible momentum parameter $\beta\geq 0$. Further, let $\bar{a}_1(\gamma),\bar{a}_2(\gamma),\bar{q}(\gamma)$ be the convergence constants for smSGD with stepsize $\omega=1$ and momentum parameter $\gamma$.  We have
\begin{eqnarray*}\bar{q}(\beta n) \geq  \bar{a}_1(\beta n) + \bar{a}_2(\beta n) &\overset{\eqref{eq:98ys8h89dh}}{=}& 1 + 4\beta + 4\beta^2n + \beta(\lambda_{\max}-\lambda_{\min}^+) - \lambda_{\min}^+\\
&\overset{\eqref{eq:qbeta}}{=}& a_1(\beta) + a_2(\beta) + 4\beta^2 (n-1)\\
&\geq & a_1(\beta) + a_2(\beta).
 \end{eqnarray*}
 Hence, the lower bound on the rate for smSGD is worse than the lower bound for mSGD. 
 
The same conclusion holds for the convergence rates themselves. Indeed, note that since $\bar{a}_1(\beta n) - a_1(\beta) = 2\beta^2(n-1) \geq 0$ and $\bar{a}_2(\beta n) - a_2(\beta) = 2\beta^2 (n-1) \geq 0$, we have
\[\bar{q}(\beta n) = \frac{\bar{a}_1(\beta n) + \sqrt{\bar{a}^2_1(\beta n) +4 \bar{a}_2(\beta n)}}{2} \geq \frac{a_1(\beta ) + \sqrt{a_1^2(\beta ) +4 a_2(\beta )}}{2} = q(\beta),\]
and hence the rate of mSGD is always better than that of smSGD.

However, the expected cost of a single iteration of mSGD may be significantly larger than that of smSGD. Indeed, let $g$ be the expected cost of evaluating a stochastic gradient. Then we need to compare $\cO(g+n)$ (mSGD) against $\cO(g)$ (smSGD). If $g\ll n$, then one iteration of smSGD is significantly cheaper than one iteration of mSGD. Let us now compare the total complexity to investigate the trade-off between the rate and cost of stochastic gradient evaluation. Ignoring constants, the total cost of the two methods (cost of a single iteration multiplied by the number of iterations) is:
\begin{equation} \label{eq:C-SHB}C_{\text{mSGD}}(\beta) \eqdef \frac{g+n}{1-q(\beta)} = \frac{g+n}{1-\frac{a_1(\beta) + \sqrt{a_1^2(\beta) + 4 a_2(\beta)}}{2}}, \end{equation}
and
\begin{equation} \label{eq:C-DSHB}C_{\text{smSGD}}(\beta n) \eqdef \frac{g}{1-\bar{q}(\beta n)} = \frac{g}{1-\frac{\bar{a}_1(\beta n) + \sqrt{\bar{a}_1^2(\beta n) + 4 \bar{a}_2(\beta n)}}{2}}. \end{equation}

Since  \begin{equation}\label{eq:iod886562}q(0) = \bar{q}(0 n),\end{equation} and since $q(\beta)$ and $\bar{q}(\beta n)$ are continuous functions of $\beta$, then because $g+n > g$, for small enough $\beta$ we will have 
$C_{\text{mSGD}}(\beta) > C_{\text{smSGD}}(\beta n).$ In particular, the speedup of smSGD compared to mSGD for $\beta \approx 0$ will be close to
\[\frac{C_{\text{mSGD}}(\beta)}{C_{\text{smSGD}}(\beta n)}  \approx \lim_{\beta' \to_+ 0} \frac{C_{\text{mSGD}}(\beta')}{C_{\text{smSGD}}(\beta' n)} \overset{\eqref{eq:C-SHB}+\eqref{eq:C-DSHB}+\eqref{eq:iod886562}}{=} \frac{g+n}{g} = 1 + \frac{n}{g}.\]

Thus, we have shown the following statement.

\begin{thm} \label{thm:DSHBspeedup} For small momentum parameters satisfying $\gamma=\beta n$, the total complexity of smSGD is approximately $1+ n/g$ times smaller than the total complexity of mSGD, where $n$ is the number of columns of $\mA$, and $g$ is the expected cost of evaluating a stochastic gradient $\nabla f_{\mS}(x)$.  
\end{thm}

\section{Special Cases: Randomized Kaczmarz with Momentum and Randomized Coordinate Descent with Momentum} \label{sec:cases}

In Table~\ref{SpecialCasesAlgorithms} we specify several special instances of mSGD  by choosing distinct combinations of the parameters $\cD$ and $\bB$.  We use $e_i$ to denote the $i$th unit coordinate vector in $\R^m$, and $\bI_{:C}$ for the column submatrix of the $m\times m$ identity matrix indexed by (a random) set $C$.

\begin{table}[t!]
\begin{center}
\scalebox{0.8}{
\begin{tabular}{ | p{4.3cm} | p{1.1 cm} | p{1.2cm} |  p{8.7cm} | }
  \hline
 \multicolumn{4}{| c |}{\textbf{Variants of mSGD}}\\
 \hline
 \begin{center}Variant of mSGD \end{center}& \begin{center}$\mS$\end{center} & \begin{center}$\bB$\end{center}  & \begin{center}$x^{k+1}$\end{center}\\
 \hline
  \hline
\begin{center}{\bf mRK}: randomized Kaczmarz with momentum \end{center} & $$e_{i}$$ & $$\bI$$ & $$x^k -\omega \frac{\bA_{i :} x^k -b_{i}}{\|\bA_{i :}\|^2} \bA_{i :}^ \top + \beta(x^k - x^{k-1}) $$\\
  \hline
 \begin{center} {\bf mRCD = mSDSA}: randomized coordinate desc. with momentum \end{center} & $$e_{i}$$ &  $$\bA\succ 0$$  & $$ x^k -\omega \frac{(\bA_{i:})^\top x^k -b_i}{\bA_{ii}} e_i  + \beta(x^k - x^{k-1}) $$\\
 \hline
 \begin{center} {\bf mRBK}: randomized block Kaczmarz with momentum \end{center} & $$\bI_{:C}$$ & $$\bI$$ &  $$x^k -\omega \bA_{C:}^\top (\bA_{C:}\bA_{C:}^\top)^\dagger (\bA_{C:}x^k-b_C) + \beta(x^k - x^{k-1}) $$\\
  \hline  
\begin{center}{\bf mRCN = mSDSA}:  randomized coordinate Newton descent with momentum  \end{center}& $$\bI_{:C}$$ & $$\bA\succ0$$& $$x^k -\omega \bI_{:C} (\bI_{:C}^\top\bA \bI_{:C})^\dagger \bI_{:C}^\top (\bA x^k-b) + \beta(x^k - x^{k-1}) $$\\
  \hline
\begin{center}{\bf mRGK}: randomized Gaussian Kaczmarz \end{center}   &$$ N(0,\bI)$$ & $$\bI$$ & $$x^k -\omega \frac{\mS^\top (\bA x^k -b)}{\|\bA^\top \mS \|^2} \bA^\top \mS + \beta(x^k - x^{k-1}) $$\\
  \hline 
 \begin{center} {\bf mRCD}: randomized coord. descent (least squares)\end{center} & $$ \bA_{:i}$$ & $$\bA^\top \bA$$ & $$x^k -\omega \frac{(\bA_{:i})^\top (\bA x^k -b)}{\|\bA_{:i}\|^2} e_i + \beta(x^k - x^{k-1}) $$ \\
 \hline
\end{tabular}
}
\end{center}
\caption{Selected special cases of mSGD.  In the special case of $\mB=\mA$, mSDSA is directly equivalent to mSGD (this is due to the primal-dual relationship \eqref{Corresp_IntroThesis}; see also Theorem~\ref{thm:dual_corresp}).  Randomized coordinate Newton  (RCN) method was first proposed in \cite{qu2015sdna}; mRCN is its momentum variant. Randomized Gaussian Kaczmarz (RGK) method was first proposed in \cite{gower2015randomized}; mRGK is its momentum variant.
}
\label{SpecialCasesAlgorithms}
\end{table}

The updates for smSGD can be derived by substituting the momentum term $\beta(x^k-x^{k-1})$ with its stochastic variant $n\beta e_{i_k}^\top (x^k-x^{k-1})e_{i_k}$. We do not aim to be comprehensive. For more details on the possible combinations of the parameters $\bS$ and $\bB$  we refer the interested reader to Section~3 of \cite{gower2015randomized}. 

In the rest of this section we present in detail two special cases: the randomized Kaczmarz method with momentum (mRK) and the randomized coordinate descent method with momentum (mRCD). Further, we compare \red{the convergence rates obtained in Theorem~\ref{theoremheavyball} (i.e., bounds on $\|\E{x^k}-x^*\|_\mB^2$) with rates that can be inferred from known results for their no-momentum variants.}

\subsection{mRK: randomized Kaczmarz with momentum} We now provide a discussion on \blue{mRK} (the method in the first row of Table~\ref{SpecialCasesAlgorithms}).  Let $\bB= \bI$ and let pick in each iteration the random matrix $\bS=e_i$ with probability $p_i=\|\bA_{i:}\|^2 / \|\bA\|_F^2$. In this setup the update rule of the mSGD simplifies to $$x^{k+1}=x^k -\omega \frac{\bA_{i:} x^k -b_i}{\|\bA_{i:}\|^2} \bA_{i:}^ \top + \beta(x^k - x^{k-1}) $$
and 
\begin{eqnarray}
\bW &\overset{\eqref{MatrixW_IntroThesis}}{=}& \bB^{-1/2} \mA^\top \Exp_{\mS\sim \cD}[\mH] \mA \bB^{-1/2} \; = \; \Exp[\mA ^\top \mH \mA] \notag \\
&=& \sum_{i=1}^{m}p_i\frac{\bA_{i:}^\top\bA_{i:}}{\|\bA_{i:}\|^2} \; = \; \frac{1}{\|\bA\|^2_{F}} \sum_{i=1}^{m}\bA_{i:}^\top\bA_{i:} \; =\; \frac{\bA^\top \bA}{\|\bA\|^2_F}.\label{matrixW}
\end{eqnarray}

The objective function takes the following form:
\begin{equation}
\label{functionRK}
f(x)=\Exp_{\bS \sim \mathcal{D}} [f_{\bS}(x)]= \sum_{i=1}^m p_i f_{\bS_i}(x)=\frac{\|\bA x -b\|^2}{2\|\bA\|^2_{F}}.
\end{equation}

For $\beta=0$, this method reduces to the \textit{randomized Kaczmarz method} with relaxation, first analyzed in \cite{ASDA}. If we also have $ \omega=1$, this is equivalent with the \textit{randomized Kaczmarz method} of Strohmer and Vershynin \cite{RK}.  RK without momentum ($\beta=0$)  and without relaxation ($\omega=1$) 
converges  with iteration complexity \cite{RK, gower2015randomized, gower2015stochastic} of \begin{equation}\label{eq:biugs897*9h8}\cO\left(\frac{1}{\lambda_{\min}^+(\bW)}\log(1/\epsilon)\right)=\cO\left(\frac{\|\bA\|^2_F}{\lambda_{\min}^+(\bA^\top \bA)}\log(1/\epsilon)\right).\end{equation}

In contrast, based on Theorem~\ref{theoremheavyball} we have
\begin{itemize}
\item For $ \omega= 1$ and $\beta= \left(1- \sqrt{0.99 \lambda_{\min}^+}\right)^2=\left(1- \sqrt{\frac{0.99}{\|\bA\|^2_F} \lambda_{\min}^+(\bA^\top \bA)}\right)^2$, the iteration complexity of the mRK is: \[ \cO\left(\sqrt{\frac{\|\bA\|^2_F}{\lambda_{\min}^+(\bA^\top \bA)}}\log(1/\epsilon)\right).\]
\item For $ \omega= \|\bA\|^2_F/\lambda_{\max}(\bA^\top \bA)$ and $\beta= \left(1- \sqrt{ \frac{0.99 \lambda_{\min}^+(\bA^\top \bA)}{\lambda_{\max}(\bA^\top \bA)}}\right)^2$
the iteration complexity becomes:
\[\cO\left(\sqrt{\frac{ \lambda_{\max}(\bA^\top \bA)}{\lambda_{\min}^+(\bA^\top \bA)}}\log(1/\epsilon)\right).\]
\end{itemize}

 This is quadratic improvement on the  previous best result \eqref{eq:biugs897*9h8}.

\paragraph{Related Work.} 
The Kaczmarz method for solving consistent linear systems was originally introduced by Kaczmarz in 1937 \cite{kaczmarz1937angenaherte}. This classical method selects the rows to project onto in a cyclic manner.  In practice, many different selection rules can be adopted. For non-random selection rules (cyclic, greedy, etc) we refer the interested reader to \cite{popa1995least,byrne2008applied, nutini2016convergence, popa2017convergence, CsibaPL}. In this work we are interested in {\em randomized} variants of the Kaczmarz method, first analyzed by Strohmer and Vershynin \cite{RK}. In \cite{RK} it was shown that RK converges with a linear convergence rate to the unique solution of a full-rank consistent linear system. This result sparked renewed interest in design of randomized methods for solving linear systems \cite{needell2010randomized, RBK, eldar2011acceleration, MaConvergence15, zouzias2013randomized, l2015randomized, schopfer2016linear, loizou2017linearly}.  All existing results on accelerated variants of RK use the Nesterov's approach of acceleration \cite{lee2013efficient, liu2016accelerated, tu2017breaking, ASDA}. To the best of our knowledge, no convergence analysis of mRK exists in the literature (Polyak's momentum). Our work fills this gap.

\subsection{mRCD: randomized coordinate descent with momentum} 

We now provide a discussion on the mRCD method (the method in the second row of Table~\ref{SpecialCasesAlgorithms}). If the matrix $\bA$ is positive definite, then we can choose $\bB= \bA$  and $\bS=e_i$ with probability $p_i=\frac{\bA_{ii}}{{\rm Trace}(\bA)}$. It is easy to see that $\bW=\frac{\bA}{{\rm Trace}(\bA)}$. In this case,  $\bW$ is positive definite and as a result, $\lambda_{\min}^+(\bW)=\lambda_{\min}(\bW)$. Moreover, we have
\begin{equation}
\label{functionRCD}
f(x)=\Exp_{\bS \sim \mathcal{D}} [f_{\bS}(x)]= \sum_{i=1}^m p_i f_{\bS_i}(x)=\frac{\|\bA x -b\|^2}{2{\rm Trace}(\bA)}.
\end{equation}

For $\beta=0$ and $ \omega=1$ the method is equivalent with  \emph{randomized coordinate descent} of Leventhal and Lewis \cite{leventhal2010randomized}, which was shown to converge with iteration complexity  \begin{equation}\label{eq:hbis8690d9}\text{Previous  best result:} \qquad  \cO\left(\frac{{\rm Trace}(\bA)}{\lambda_{\min}(\bA)}\log(1/\epsilon)\right).\end{equation}

In contrast, following  Theorem~\ref{theoremheavyball}, we can obtain the following iteration complexity results for mRCD:
 \begin{itemize}
  \item For $ \omega= 1$ and $\beta= \left(1- \sqrt{\frac{0.99}{{\rm Trace}(\bA)} \lambda_{\min}(\bA)}\right)^2$, the iteration complexity is \[ \cO\left(\sqrt{\frac{{\rm Trace}(\bA)}{\lambda_{\min}(\bA)}}\log(1/\epsilon)\right).\]
\item For  $ \omega= {\rm Trace}(\bA)/\lambda_{\max}(\bA)$ and $\beta= \left(1- \sqrt{ \frac{0.99 \lambda_{\min}(\bA)}{\lambda_{\max}(\bA)}}\right)^2 $  the iteration complexity becomes 
\[\cO\left(\sqrt{\frac{ \lambda_{\max}(\bA)}{\lambda_{\min}(\bA)}}\log(1/\epsilon)\right).\]
 \end{itemize}
 
 This is quadratic improvement on the previous best result \eqref{eq:hbis8690d9}.

\paragraph{Related Work.}  It is known that if $\mA$ is positive definite, the popular {\em randomized Gauss-Seidel} method can be interpreted as randomized coordinate descent (RCD). RCD methods were first analyzed by Lewis and Leventhal in the context of linear systems and least-squares problems \cite{leventhal2010randomized}, and later extended by several authors to more general settings, including smooth convex optimization \cite{nesterov2012efficiency}, composite convex optimization \cite{richtarik2014iteration}, and parallel/subspace descent variants  \cite{richtarik2016parallel}. These results were later further extended to handle  arbitrary sampling distributions  \cite{qu2016coordinate,qu2016coordinate2, qu2015quartz, SCP} . Accelerated variants of RCD were studied in \cite{lee2013efficient, fercoq2015accelerated, allen2016even}. For other non-randomized coordinate descent variants and their convergence analysis, we refer the reader to \cite{wright2015coordinate,nutini2015coordinate, CsibaPL}. To the best of our knowledge, mRCD and smRCD have never  been analyzed before in any setting.

\subsection{Visualizing the acceleration mechanism}
\label{graphica}

We devote this section to the graphical illustration of the acceleration mechanism behind momentum. Our goal is to shed more light on how the proposed algorithm works in practice. For simplicity, we illustrate this by comparing RK and mRK.


\begin{figure}[h]
\centering
\begin{subfigure}{.5\textwidth}
  \centering
  \includegraphics[width=1\linewidth]{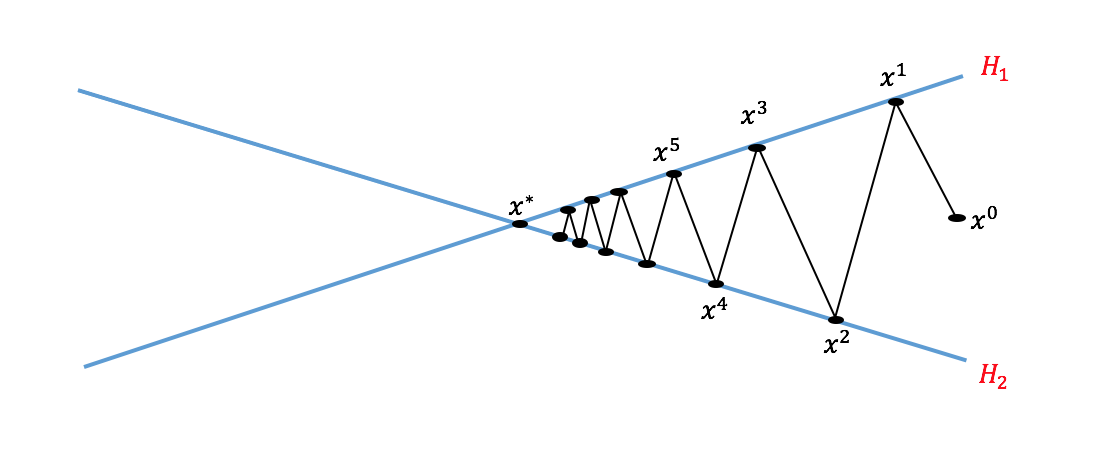}
  \caption{\footnotesize Randomized Kaczmarz Method \cite{RK}}
\end{subfigure}%
\begin{subfigure}{.5\textwidth}
  \centering
  \includegraphics[width=1\linewidth]{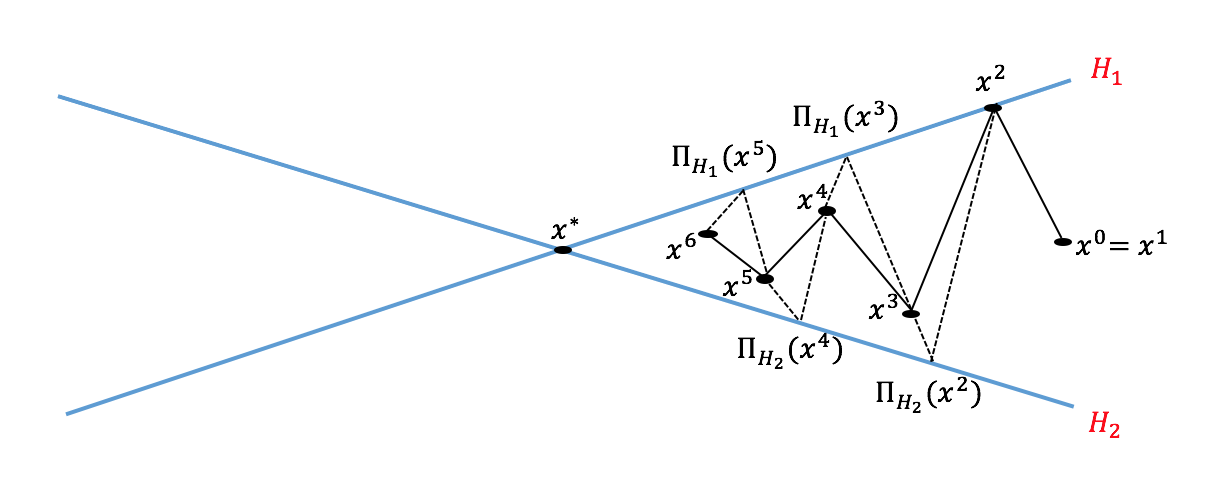}
  \caption{\footnotesize Randomized Kaczmarz Method with Momentum}
\end{subfigure}
\caption{Graphical interpretation of the {\em randomized Kaczmarz method} and the {\em randomized Kaczmarz method with momentum} in a simple example with only two  hyperplanes $H_i=\{x\;:\; \bA_{i:} x=b_i\}$ where $i=1,2$ and a unique solution $x^*$.}
\label{graphicalrepresentation}
\end{figure}

In Figure~\ref{graphicalrepresentation} we present in a simple $\R^2$ illustration of the difference between the workings of RK and mRK. Our goal is to show graphically how the addition of momentum leads to acceleration.  Given iterate $x^k$, one can think of the update rule of the mRK \eqref{eq:SHB-intro} in two steps:
\begin{enumerate}
\item \emph{The Projection:} The projection step corresponds to the first part $x^k - \omega  \nabla f_{\bS_k}(x^k)$ of the  mRK update \eqref{eq:SHB-intro} and it means that the current iterate $x^k$ is projected onto a randomly chosen hyperplane $H_i$\footnote{In the plots of Figure ~\ref{graphicalrepresentation}, the hyperplane of each update is chosen in an alternating fashion for illustration purposes}. The value of the stepsize $\omega \in (0,2)$ defines whether the projection is exact or not. When $\omega=1$ (no relaxation) the projection is exact, that is the point  $\Pi_{H_i}(x^k)$ belongs in the hyperplane $H_i$.  In Figure~\ref{graphicalrepresentation} all projections are exact.
\item \emph{Addition of the momentum term}: The momentum term (right part of the update rule) $\beta (x^k-x^{k-1})$ forces the next iterate $x^{k+1}$ to be closer to the solution $x^*$ than the corresponding point $\Pi_{H_i}(x^k)$. Note also that the vector $x^{k+1}-\Pi_{H_i}(x^k)$ is always parallel to $x^k-x^{k-1}$ for all $k\geq0$.
\end{enumerate}

\begin{rem}
In the example of Figure~\ref{graphicalrepresentation}, the performance of mRK is similar to the performance of RK until iterate $x^3$. After this point, the momentum parameter becomes more effective and the mRK method accelerates. This behavior appears also in our experiments in the next section where we work with  matrices with many rows. There we can notice that the momentum parameter seems to become more effective after the first $m+1$ iterations.
\end{rem}

\section{Numerical Experiments} 
\label{sec:experiments}

In this section we study the computational behavior of the two proposed algorithms, mSGD and smSGD. In particular, we focus mostly on the evaluation of the performance of mSGD. To highlight the usefulness of smSGD, an \blue{empirical} verification of Theorem~\ref{thm:DSHBspeedup} is presented in subsection~\ref{DSHB comp}. As we have already mentioned, both mSGD and smSGD can be interpreted as sketch-and-project methods (with relaxation), and as a result a comprehensive array of well-known algorithms can be recovered as special cases by varying the main parameters of the methods (check Section~\ref{sec:cases}). In our experiments we focus on the popular special cases of randomized Kaczmarz method (RK) and the randomized coordinate descent method (RCD) without relaxation ($\omega=1$), and  show the practical benefits of \blue{adding} the momentum term\footnote{The experiments were repeated with various values of the main parameters and initializations, and similar results were obtained in all cases.}. The choice of the stepsize $\omega=1$ is not arbitrary. Recently, in \cite{ASDA} both relaxed RK and relaxed RCD were analyzed, and it was proved that the quantity $\E{\|x^k-x^*\|^2_{\bB}}$ converges linearly to zero for $\omega \in (0,2)$, and that the best convergence rate is obtained precisely for $\omega=1$. Thus the comparison is with the best-in-theory  no-momentum variants.

Note that, convergence analysis of the error $\E{\|x^k-x^*\|^2_{\bB}}$ and of the expected function values $\E{f(x^k)}$ in Theorem~\ref{L2} shows that mSGD enjoys global non-asymptotic linear convergence rate but not faster than the no-momentum method. 
The accelerated linear convergence rate has been obtained only in the weak sense (Theorem \ref{theoremheavyball}). Nevertheless, in practice as indicated from our experiments, mSGD is faster than its no momentum variant. Note also that in all of the presented experiments the momentum parameters $\beta$ of the methods are chosen to be positive constants that do not depend on parameters that are not known to the users such as $\lambda_{\min}^+$ and $\lambda_{\max}$.

In comparing the methods with their momentum variants we use both the relative error measure $\|x^k-x^*\|^2_\bB / \|x^0-x^*\|^2_\bB $ and the function values $f(x^k)$\footnote{Remember that in our setting we have $f(x^*)=0$ for the optimal solution $x^*$ of the best approximation problem; thus $f(x)-f(x^*)=f(x)$. The function values $f(x^k)$ refer to function~\eqref{functionRK} in the case of RK and to function~\eqref{functionRCD} for the RCD.  For block variants the objective function of problem \eqref{StochReform_IntroThesis} has also closed form expression but it can be very difficult to compute. In these cases one can instead evaluate the quantity $\|\bA x-b\|^2_{\bB}$.}. In all implementations, except for the experiments on average consensus (Section~\ref{consensus}), the starting point is chosen to be $x^0=0$. In the case of average consensus the starting point must be the vector with the initial private values of the nodes of the network. 
All the code for the experiments is written in the Julia programming language.
For the horizontal axis we use either the number of iterations or the wall-clock time measured using the tic-toc Julia function.

This section is divided in three main experiments. In the first one we evaluate the performance of the mSGD method in the special cases of mRK and mRCD for solving both synthetic consistent Gaussian systems and consistent linear systems with real matrices. In the second experiment we computationally verify Theorem \ref{thm:DSHBspeedup} (comparison between the mSGD and smSGD methods). In the last experiment building upon the recent results of \cite{LoizouRichtarik} we show how the addition of the momentum accelerates the pairwise randomized gossip (PRG) algorithm for solving the average consensus problem.

\begin{table}[H]
\begin{center}
{\footnotesize
\begin{tabular}{| p{3cm} | p{3cm} | p{3cm} | p{3cm} |  }
 \hline
Assumptions & No-momentum, & Momentum, &  Stochastic Momentum,\\
& $\beta=0$ & $\beta\geq0$ &$\beta\geq0$\\
   \hline
$\bA$ general, $\bB=\bI$ & RK & mRK & smRK\\
 \hline
$\bA\succ 0$, $\bB=\bA$ & RCD & mRCD & smRCD\\
  \hline
$\bA$ incidence matrix, $\bB=\bI$ & PRG &  mPRG & smPRG \\
  \hline
\end{tabular}
}
\end{center}
\caption{Abbreviations of the algorithms (special cases of general framework) that we use in the numerical evaluation section. In all methods the random matrices are chosen to be unit coordinate vectors in $\R^m$  ($\bS=e_i$). With PRG we denote the Pairwise Randomized Gossip algorithm for solving the average consensus problem first proposed in \cite{boyd2006randomized}. Following similar notation with the rest of the chapter with mPRG and smPRG we indicate its momentum and stochastic momentum variants respectively.}
\end{table}

\subsection{Evaluation of mSGD}
\label{subsectionEvaluation}
In this subsection we study the computational behavior of mRK and mRCD when they compared with their no momentum variants for both synthetic and real data.

\subsubsection{Synthetic Data}
\label{gaussiansyntheric}
The synthetic data for this comparison is generated as follows\footnote{Note that in the first experiment we use Gaussian matrices which by construction are full rank matrices with probability 1 and as a result the consistent linear systems have unique solution. Thus, for any starting point $x^0$, the vector $z$ that \blue{is} used to create the linear system is the solution mSGD converges to.  This is not true for general consistent linear systems, with no full-rank matrix. In this case, the solution   $x^{*}=\Pi_{\cL,\bB}(x^0)$ that mSGD converges to is not necessarily equal to $z$. For this reason, in the evaluation of the relative error measure $\|x^k-x^*\|^2_\bB / \|x^0-x^*\|^2_\bB$, one should be careful and use the value $x^*=x^0+\bA^\dagger (b- \bA x^0)\overset{x^0=0}{=} \bA^\dagger b$.}.

\textbf{For mRK:}
All elements of matrix $\bA \in \R^{m \times n}$ and vector $z \in \R^n$ are chosen to be i.i.d $\mathcal{N}(0,1)$. Then the right hand side of the linear system is set to $b=\bA z$. With this way the consistency of the linear system with matrix $\bA$ and right hand side $b$ is ensured. 

\textbf{For mRCD:}  A Gaussian matrix $\bP \in \R^{m \times n}$ is generated and then matrix $\bA = \bP^\top \bP \in \R^{n \times n}$ is used in the linear system. The vector $z\in \R^n$ is chosen to be i.i.d $\mathcal{N}(0,1)$ and again to ensure consistency of the linear system, the right hand side is set to $b=\bA z$.

In particular for the evaluation of mRK we generate Gaussian matrices with $m=300$ rows and several columns while for the case of mRCD the matrix $\bP$ is chosen to be Gaussian with $m=500$ rows and several columns\footnote{RCD converge to the optimal solution only in the case of positive definite matrices. For this reason $\bA = \bP^\top \bP \in \R^{n \times n}$ is used which with probability $1$ is a full rank matrix}. 
Linear systems of these forms were extensively studied \cite{RK, geman1980limit} and it was shown that the quantity $1/\lambda_{\min}^+$(condition number) can be easily controlled. 

For each linear system we run mRK (Figures~\ref{RKperformace1} and~\ref{RKperformace12})  and mRCD (Figures~\ref{RCDperformance1} and~\ref{RCDperformance12}) for several values of momentum parameters $\beta$ and fixed stepsize $\omega=1$ and we plot the performance of the methods (average after 10 trials) for both the relative error measure and the function values. Note that for $\beta=0$ the methods are equivalent with their no-momentum variants RK and RCD respectively.

From Figures~\ref{RKperformace1}, \ref{RKperformace12}, \ref{RCDperformance1} and~\ref{RCDperformance12} it is clear that the addition of momentum term leads to an improvement in the performance of both, RK and RCD. More specifically, from the four figures we observe the following:

\begin{itemize}
\item For the well conditioned linear systems ($1/\lambda_{\min}^+$ small) it is known that even the no-momentum variant converges rapidly to the optimal solution. In these cases the benefits of the addition of momentum are not obvious. The momentum term is beneficial for the case where the no-momentum variant ($\beta=0$) converges slowly, that is when  $1/\lambda_{\min}^+$ is large (ill-conditioned linear systems). 
\item For the case of fixed stepsize $\omega=1$, the problems with small condition number require smaller momentum parameter $\beta$ to have faster convergence. Note the first two rows of Figures~\ref{RKperformace1} and \ref{RCDperformance1}, where $\beta=0.3$ or $\beta =0.4$, are good options. 
\item For large values of $1/\lambda_{\min}^+$, it seems that the choice of $\beta=0.5$ is the best. As an example for matrix $\bA\in \R^{300 \times 280}$ in Figure~\ref{RKperformace1}, (where $1/\lambda_{\min}^+=208,730$), note that to reach relative error $10^{-10}$, RK needs around 2 million iterations, while mRK with momentum parameter $\beta=0.5$ requires only half that many iterations. The acceleration is obvious also in terms of time where in 12 seconds the mRK with momentum parameter $\beta=0.5$ achieves relative error of the order $10^{-9}$ and RK requires more than 25 seconds to obtain the same accuracy. 
\item We observe that both mRK and mRCD, with appropriately chosen momentum parameters $0< \beta \leq 0.5$, always converge  faster than their no-momentum variants, RK and RCD, respectively. This is a smaller momentum parameter than $\beta \approx 0.9$ which is being used extensively with mSGD for training deep neural networks \cite{zhang2017yellowfin, wilson2017marginal, sutskever2013importance}. 
\item In \cite{de2017accelerated} a stochastic power iteration with momentum is proposed for principal component analysis (PCA). There it was demonstrated empirically that a naive application of momentum to the stochastic power iteration does not result in a faster method. To achieve faster convergence, the authors proposed mini-batch and variance-reduction techniques on top of the addition of momentum. In our setting, mere addition of the momentum term  to SGD (same is true for  special cases such as RK and RCD) leads to empirically faster methods.
\end{itemize}
\begin{figure}[!]
\centering
\begin{subfigure}{.35\textwidth}
  \centering
  \includegraphics[width=1\linewidth]{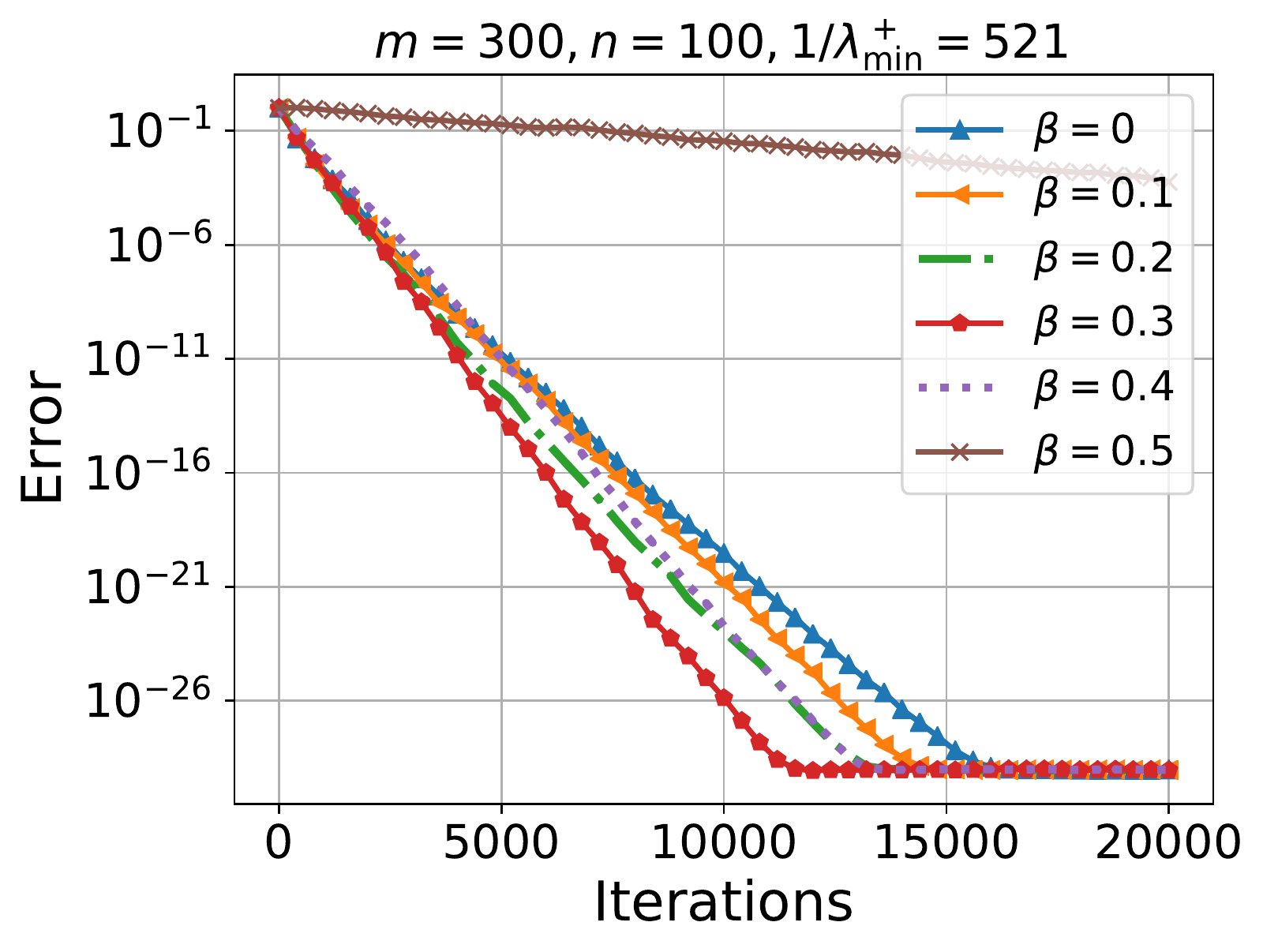}
\end{subfigure}%
\begin{subfigure}{.35\textwidth}
  \centering
  \includegraphics[width=1\linewidth]{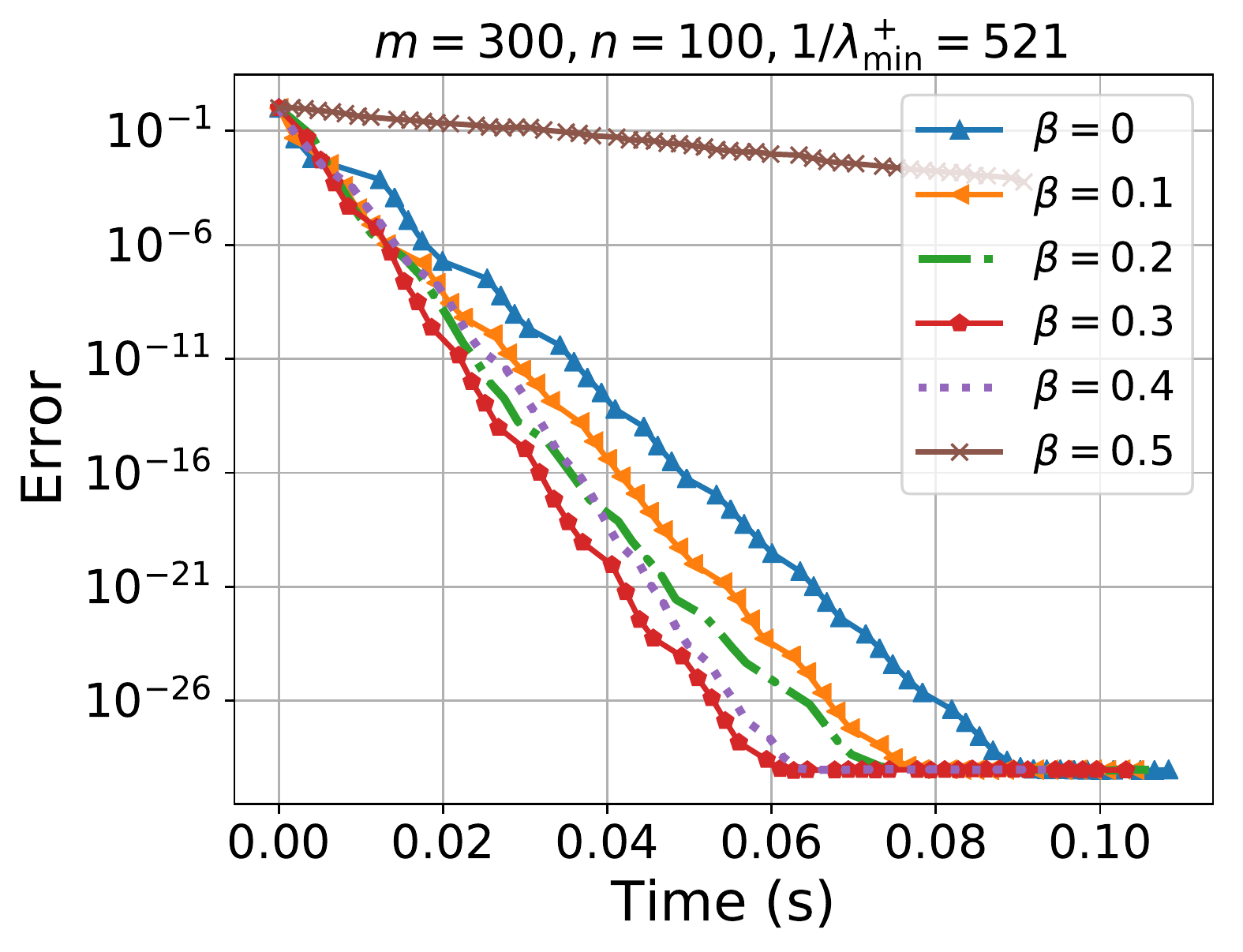}
\end{subfigure}\\
\begin{subfigure}{.35\textwidth}
  \centering
  \includegraphics[width=1\linewidth]{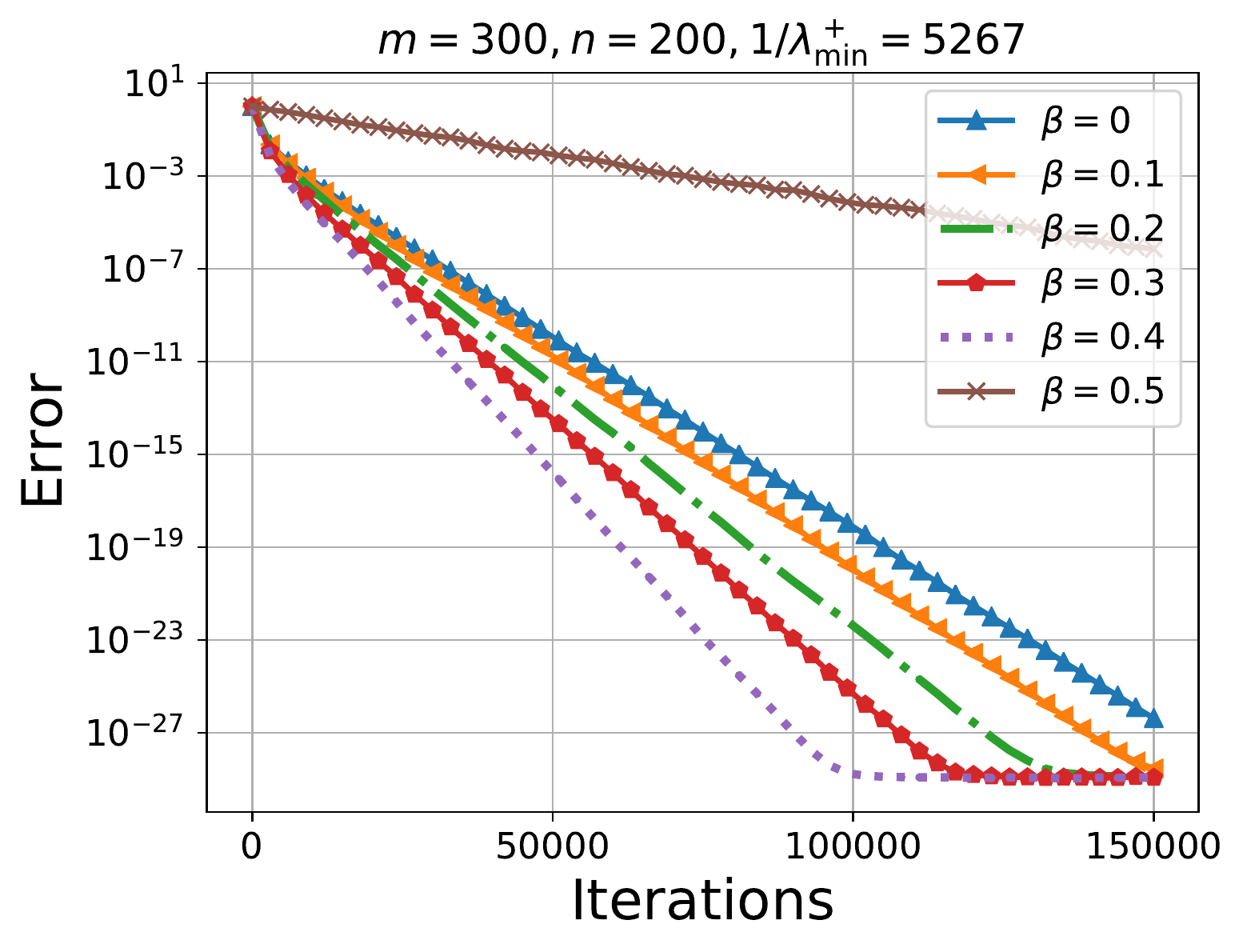}
\end{subfigure}%
\begin{subfigure}{.35\textwidth}
  \centering
  \includegraphics[width=1\linewidth]{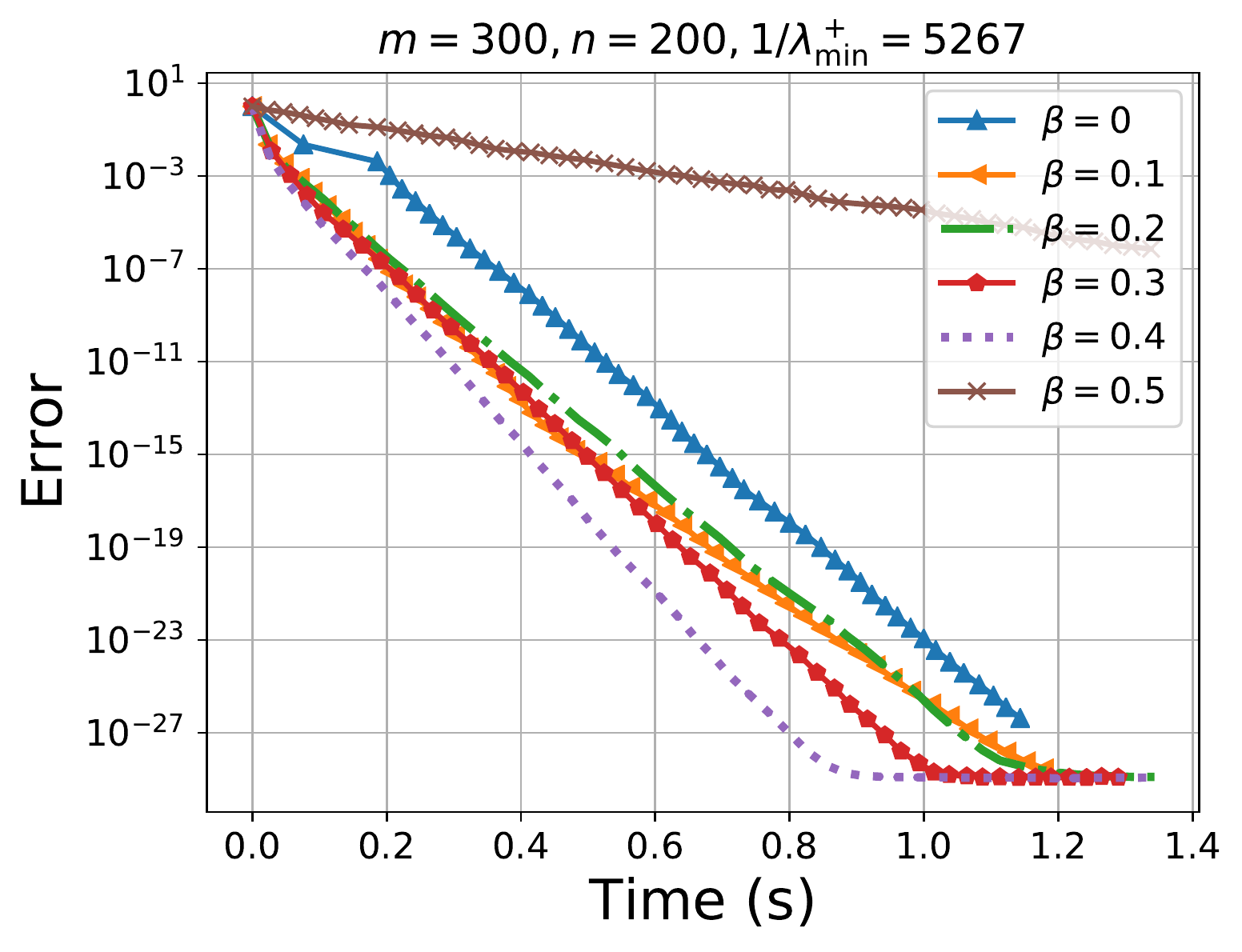}
\end{subfigure}\\
\begin{subfigure}{.35\textwidth}
  \centering
 \includegraphics[width=1\linewidth]{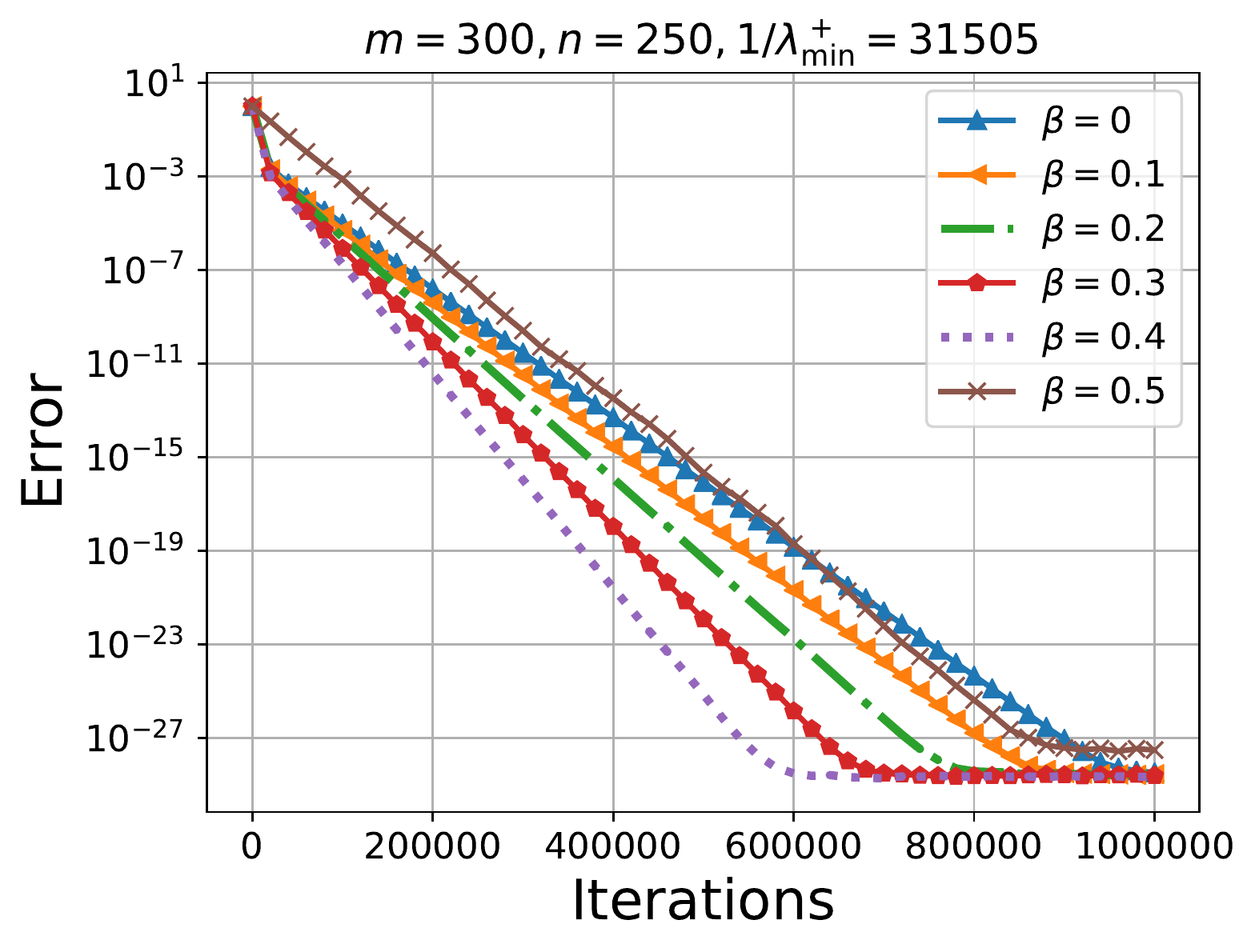}
\end{subfigure}%
\begin{subfigure}{.35\textwidth}
  \centering
\includegraphics[width=1\linewidth]{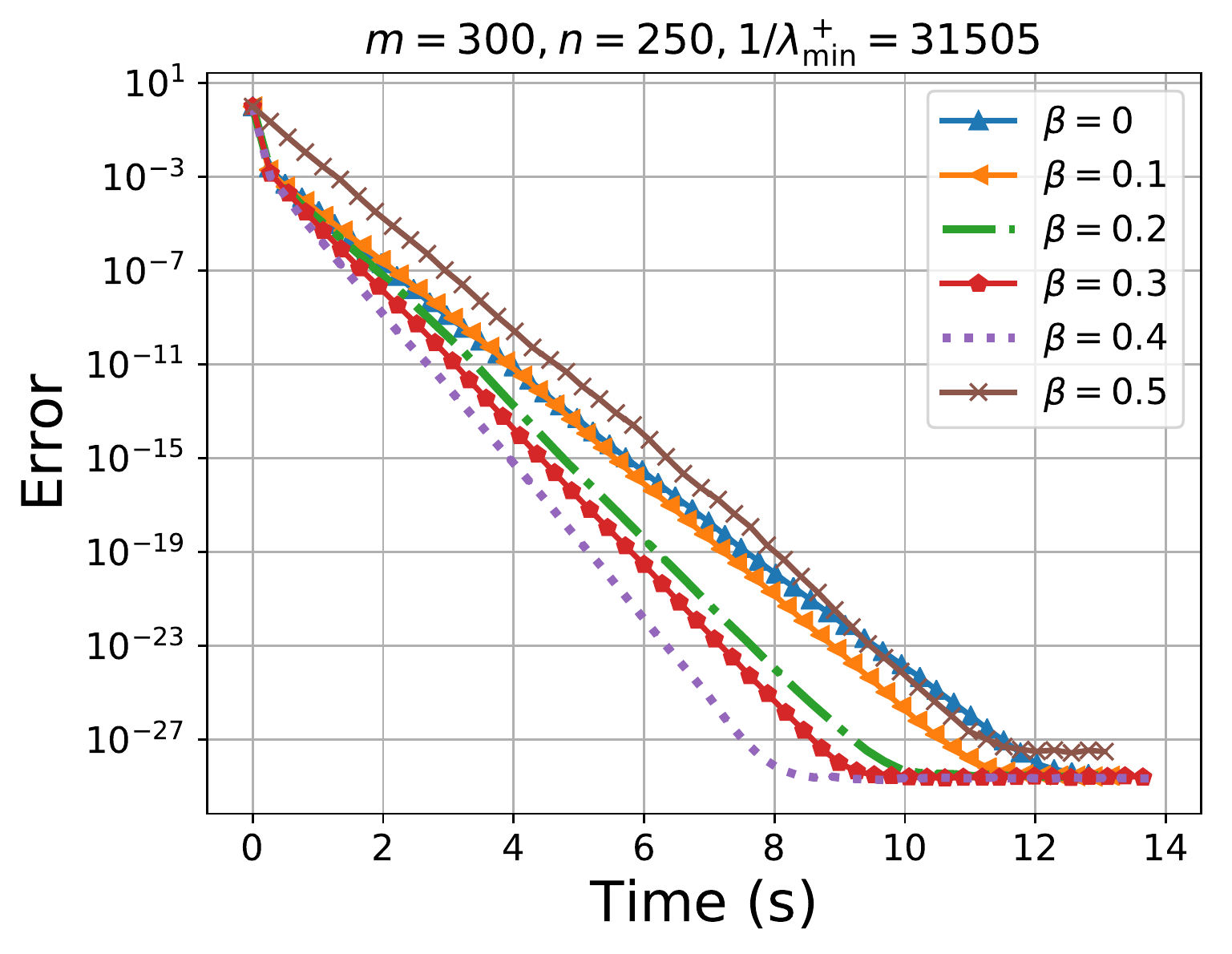}
\end{subfigure}\\
\begin{subfigure}{.35\textwidth}
  \centering
  \includegraphics[width=1\linewidth]{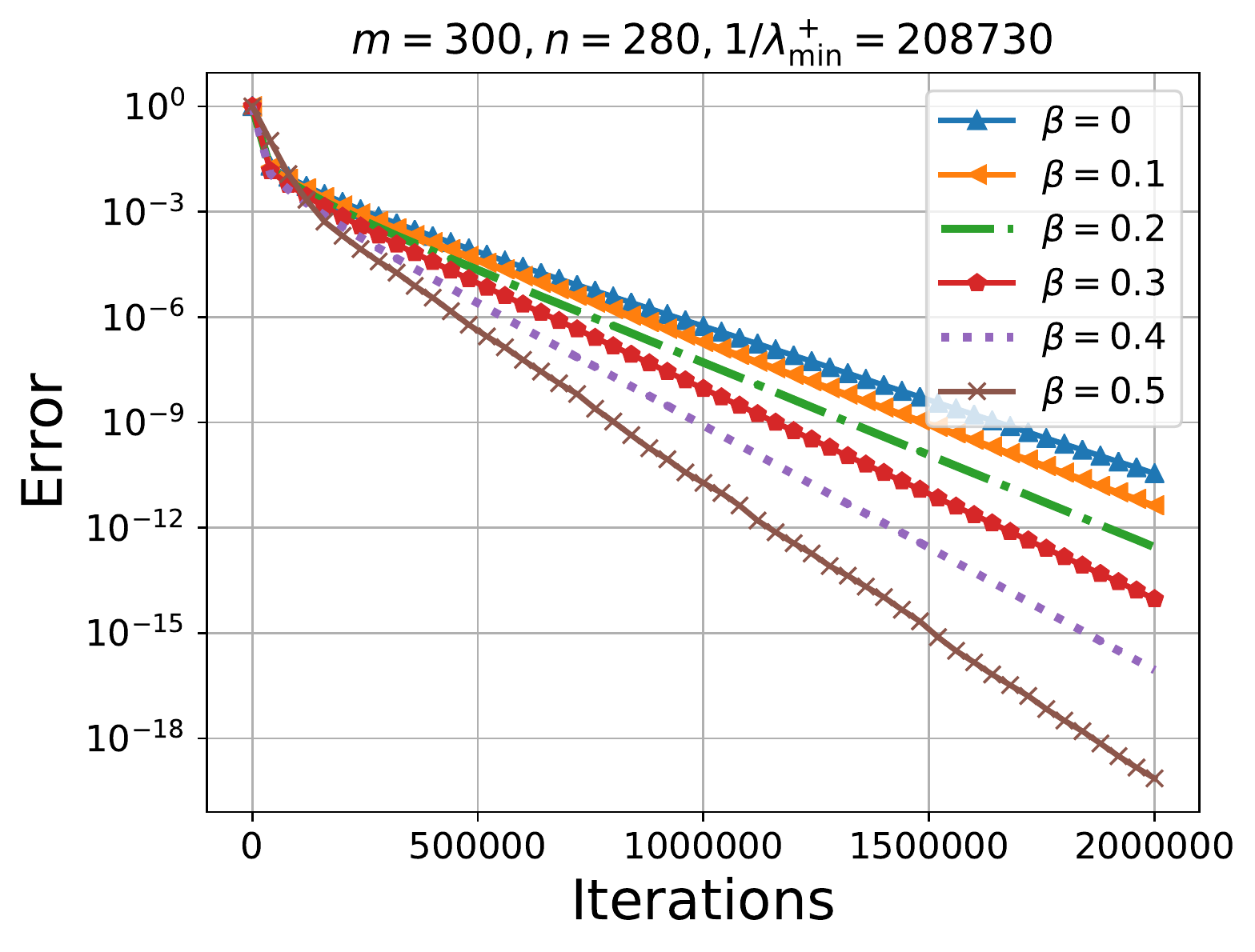}
\end{subfigure}%
\begin{subfigure}{.35\textwidth}
  \centering
  \includegraphics[width=1\linewidth]{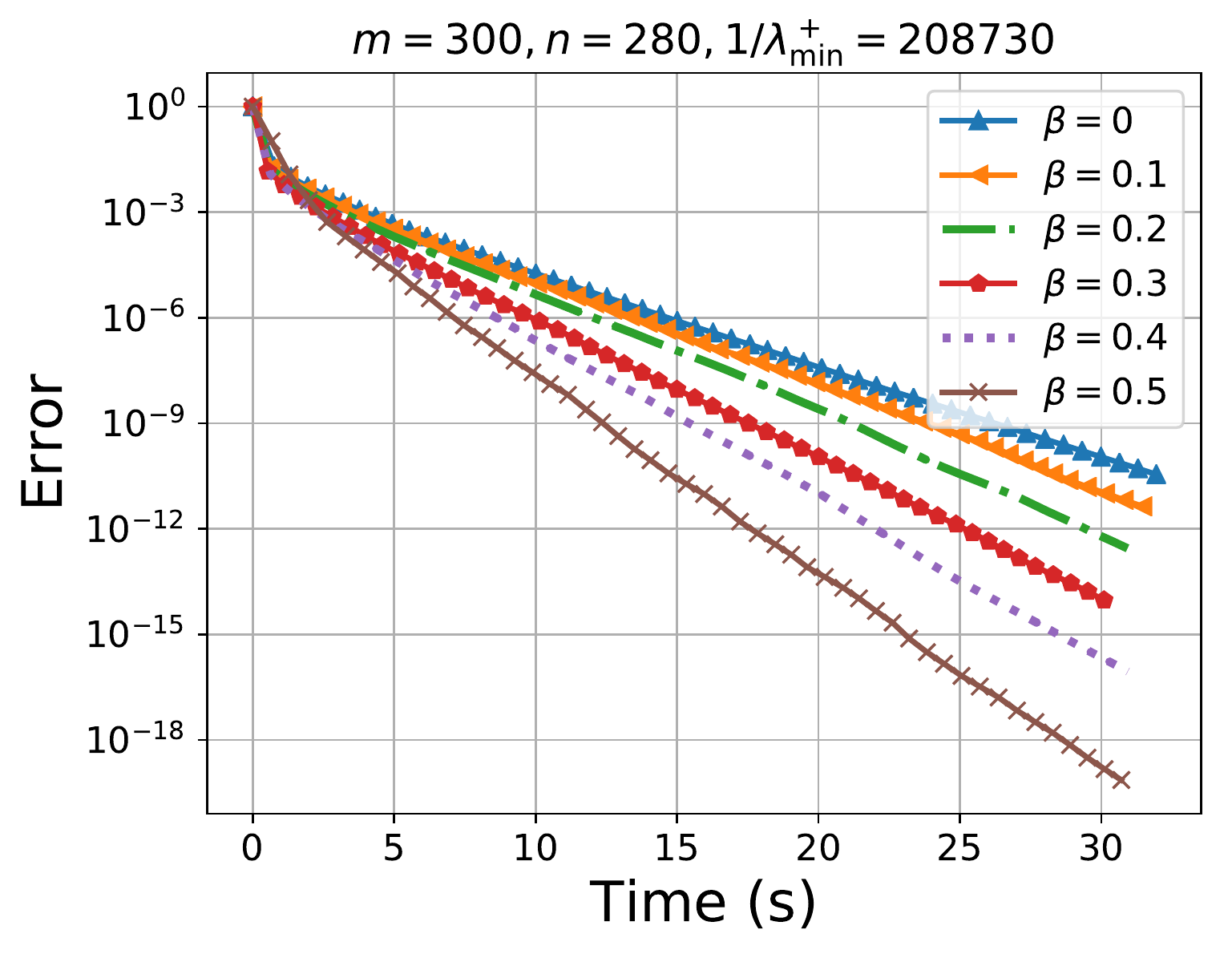}
\end{subfigure}\\
\begin{subfigure}{.35\textwidth}
  \centering
  \includegraphics[width=1\linewidth]{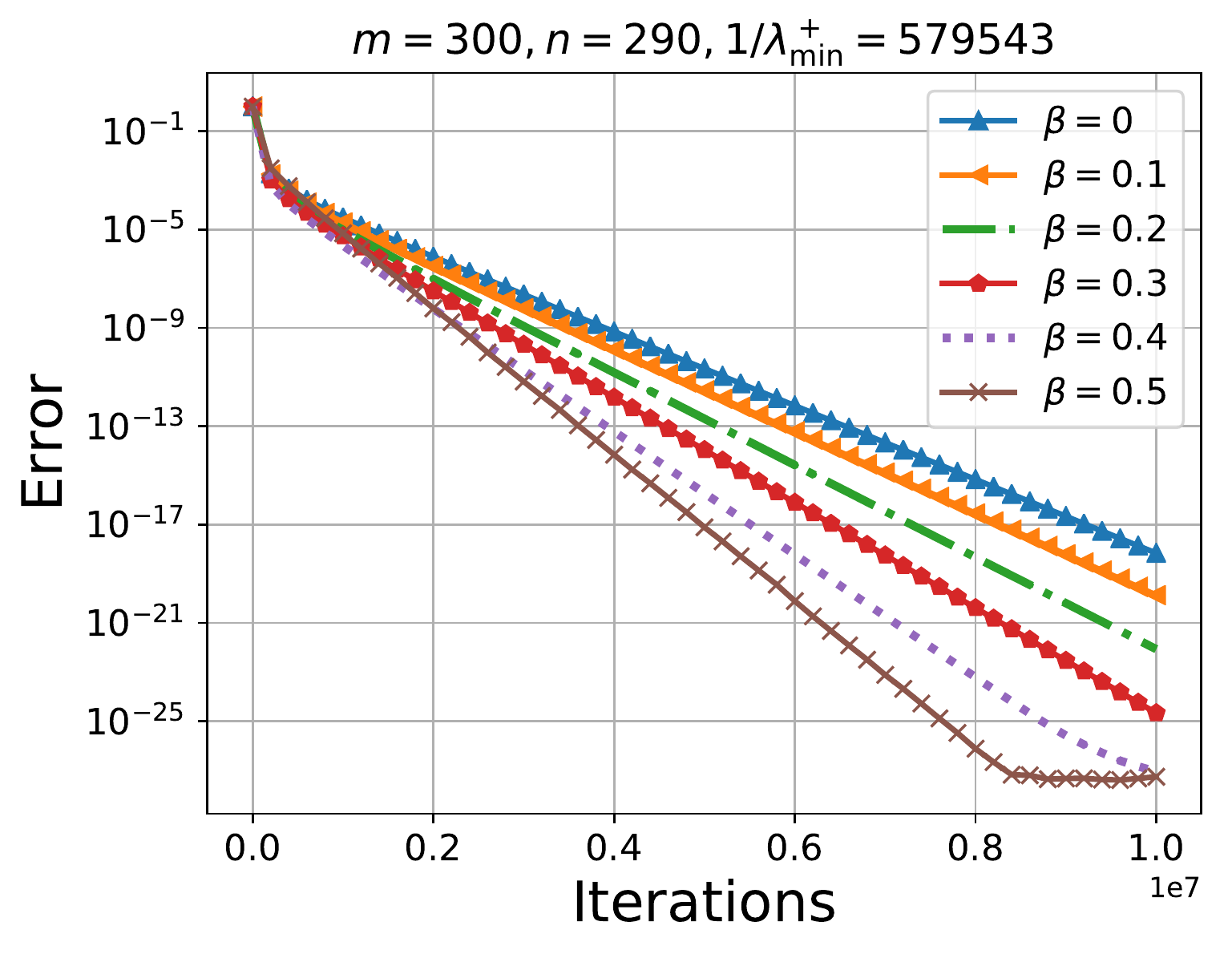}
\end{subfigure}%
\begin{subfigure}{.35\textwidth}
  \centering
  \includegraphics[width=1\linewidth]{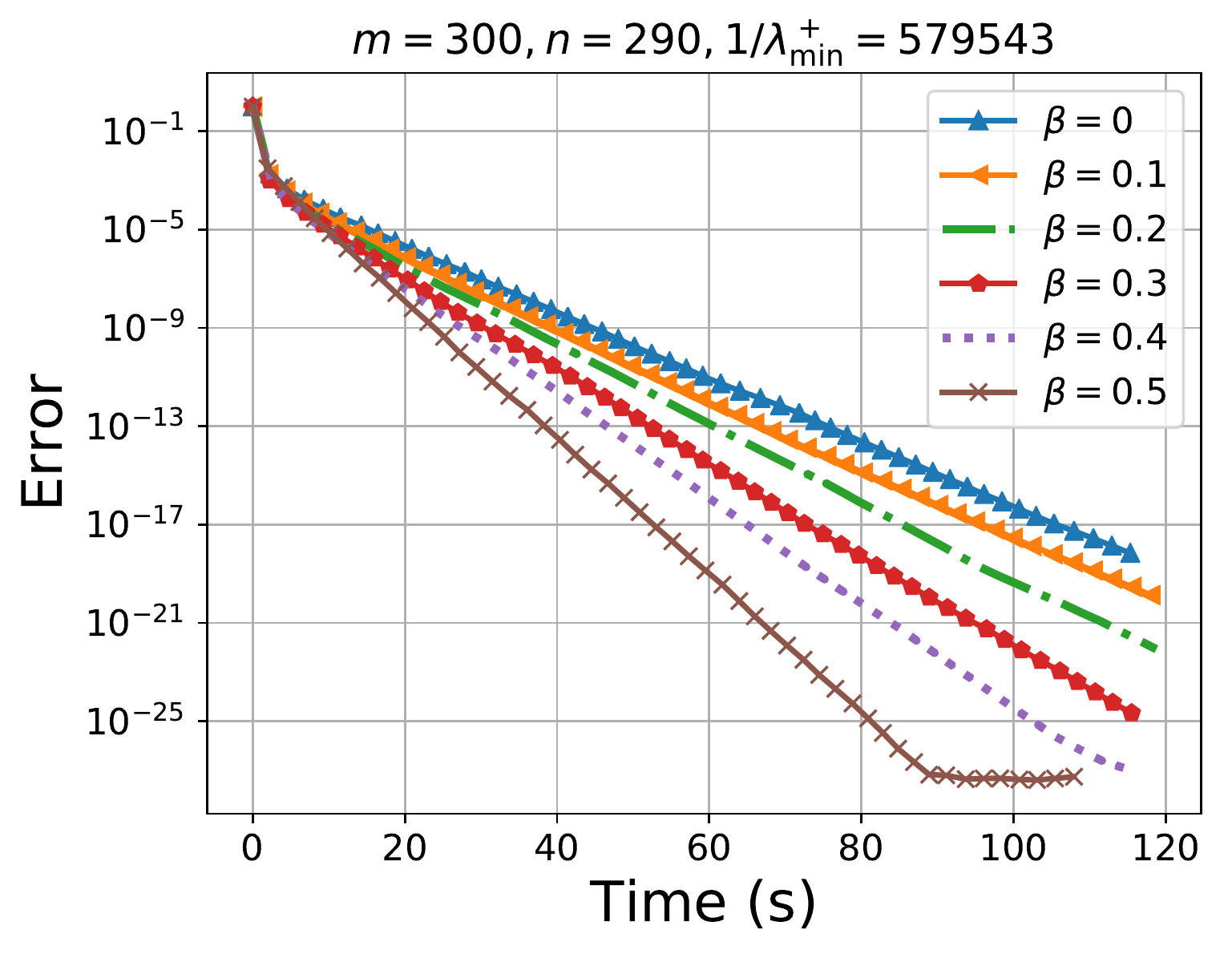}
\end{subfigure}\\
\caption{Performance of mRK \blue{(the method in the first row of Table~\ref{SpecialCasesAlgorithms})} for fixed stepsize $\omega=1$ and several momentum parameters $\beta$ for consistent linear systems with Gaussian matrix $\bA$ with $m=300$ rows and $n=100,200,250,280,290$ columns. The graphs in the first (second) column plot iterations (time) against residual error. All plots are averaged over 10 trials. The title of each plot indicates the dimensions of the matrix $\bA$ and the value of $1/\lambda_{\min}^+$. The ``Error" on the vertical axis represents the relative error $\|x^k-x^*\|^2_\bB / \|x^0-x^*\|^2_\bB \overset{\bB=\bI, x^0=0}{=}\|x^k-x^*\|^2 / \|x^*\|^2_\bB$.}
\label{RKperformace1}
\end{figure}

\begin{figure}[!]
\centering
\begin{subfigure}{.35\textwidth}
  \centering
  \includegraphics[width=\linewidth]{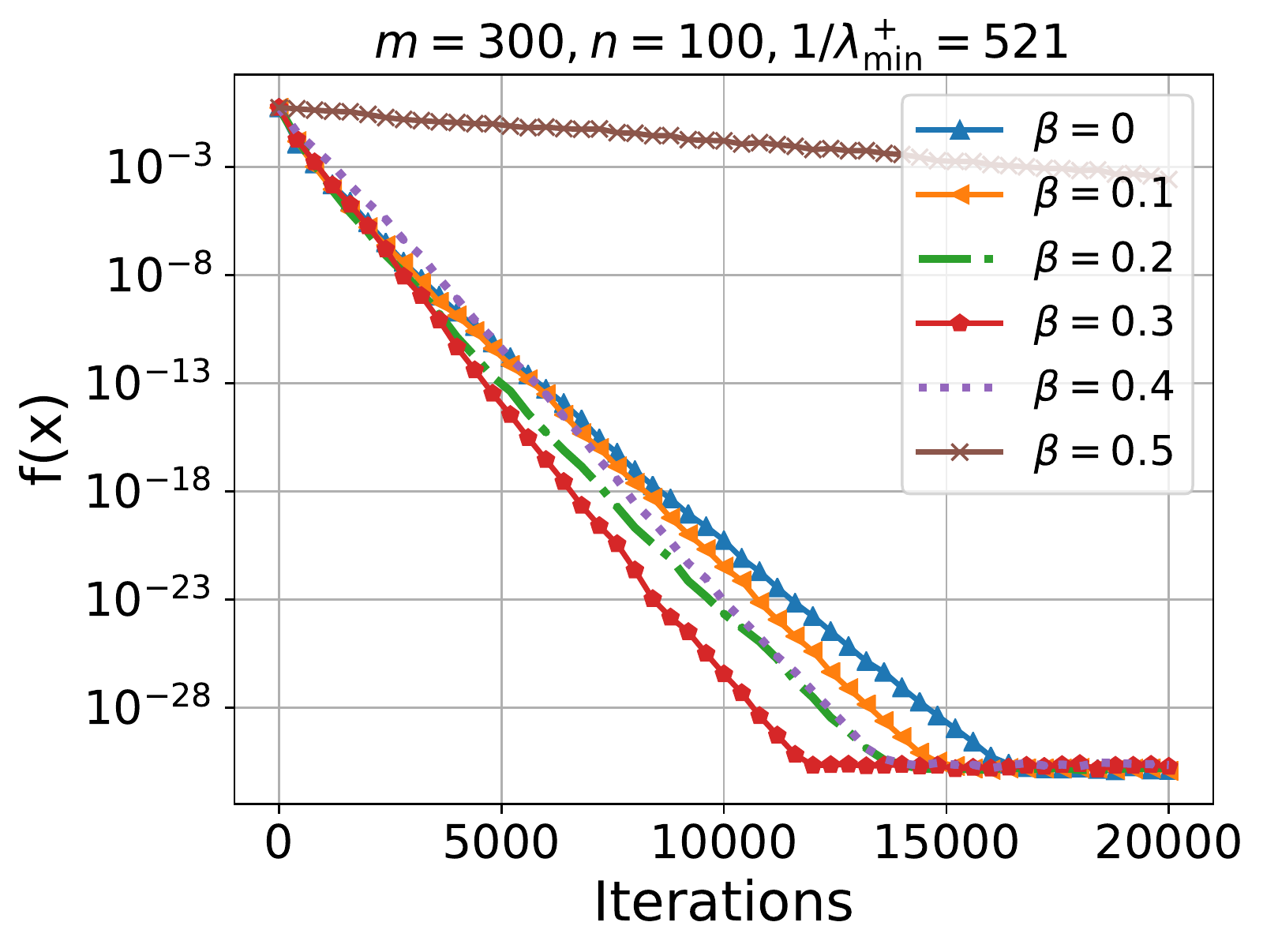}
\end{subfigure}
\begin{subfigure}{.35\textwidth}
  \centering
  \includegraphics[width=1\linewidth]{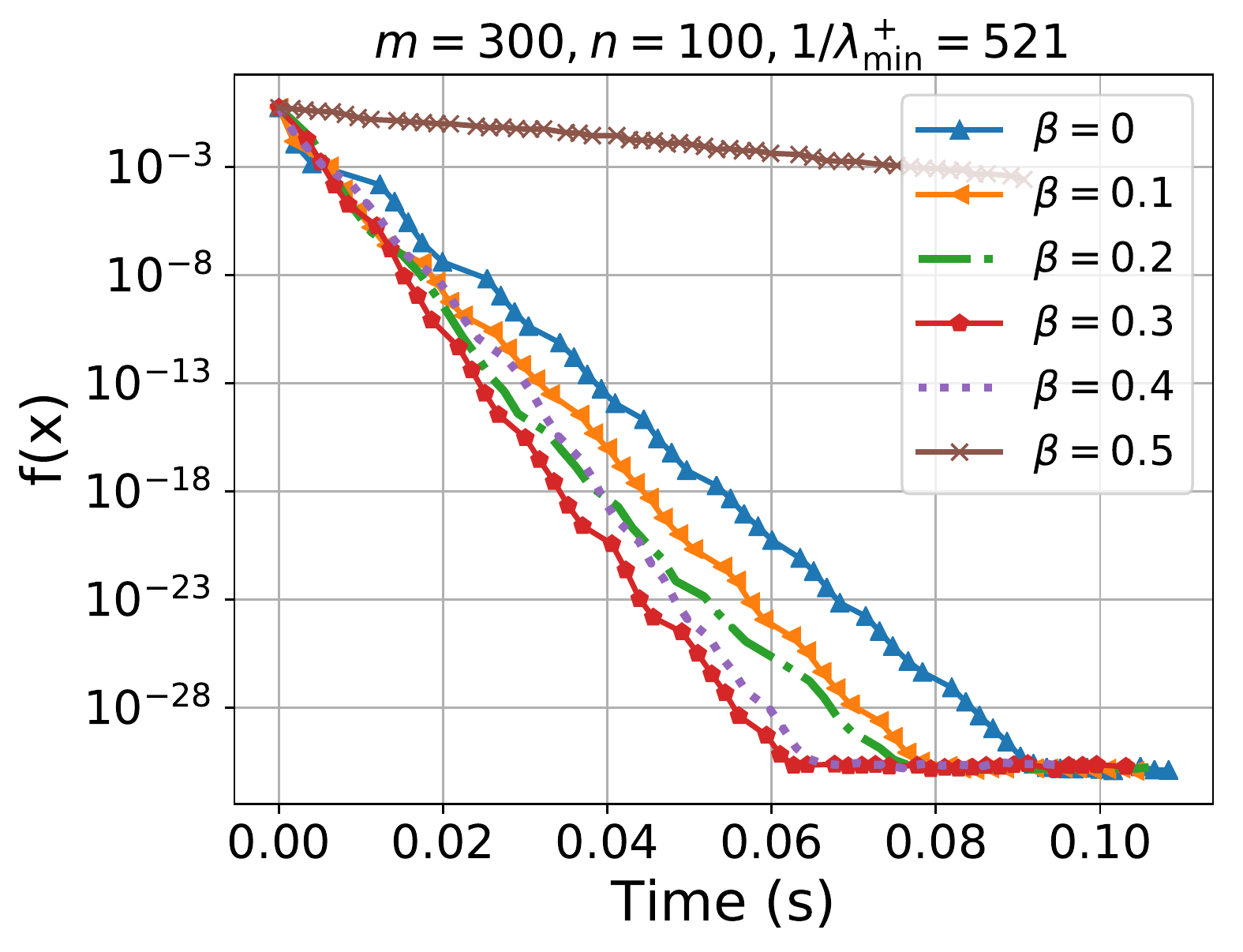}
\end{subfigure}\\
\begin{subfigure}{.35\textwidth}
  \centering
  \includegraphics[width=\linewidth]{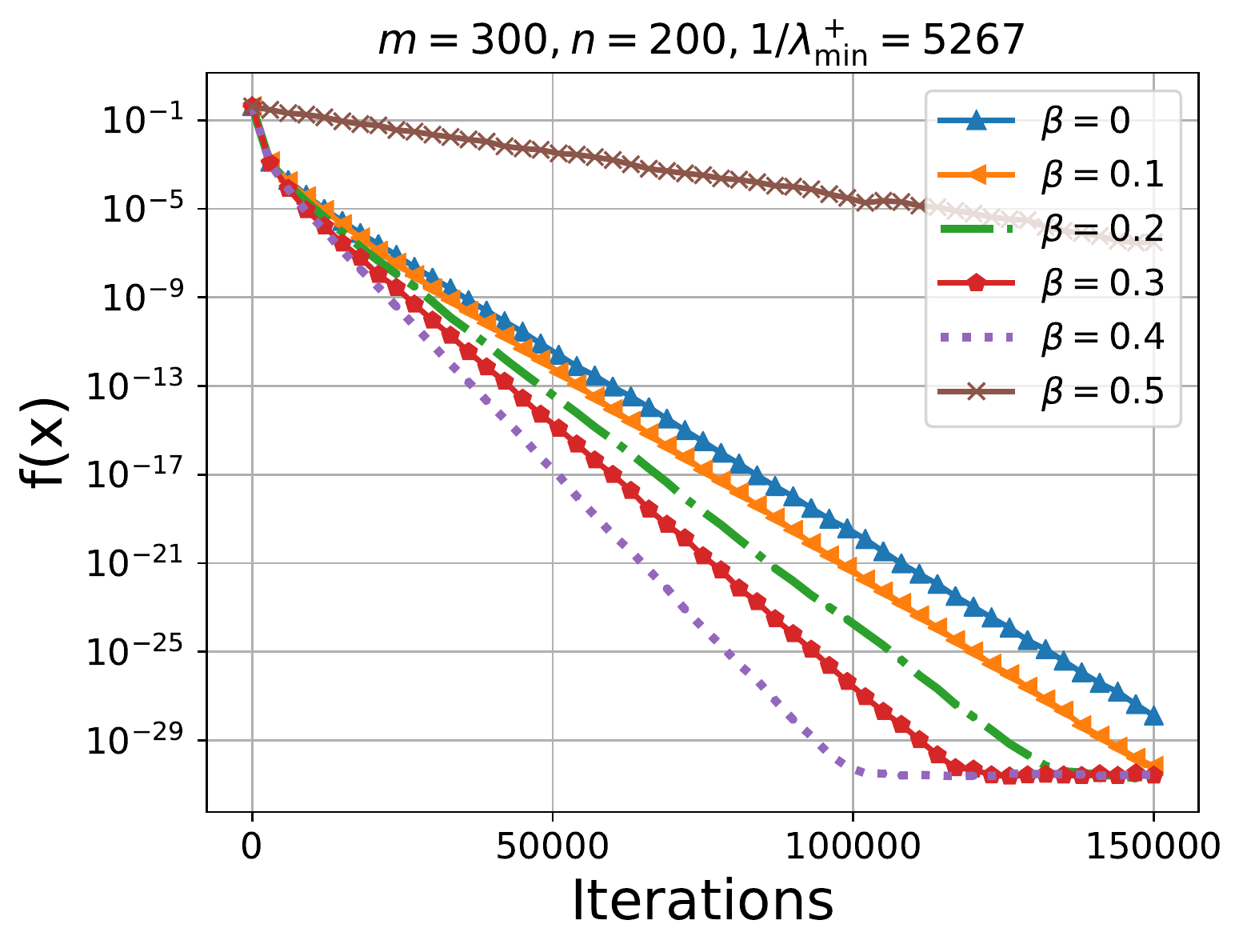}
\end{subfigure}
\begin{subfigure}{.35\textwidth}
  \centering
  \includegraphics[width=1\linewidth]{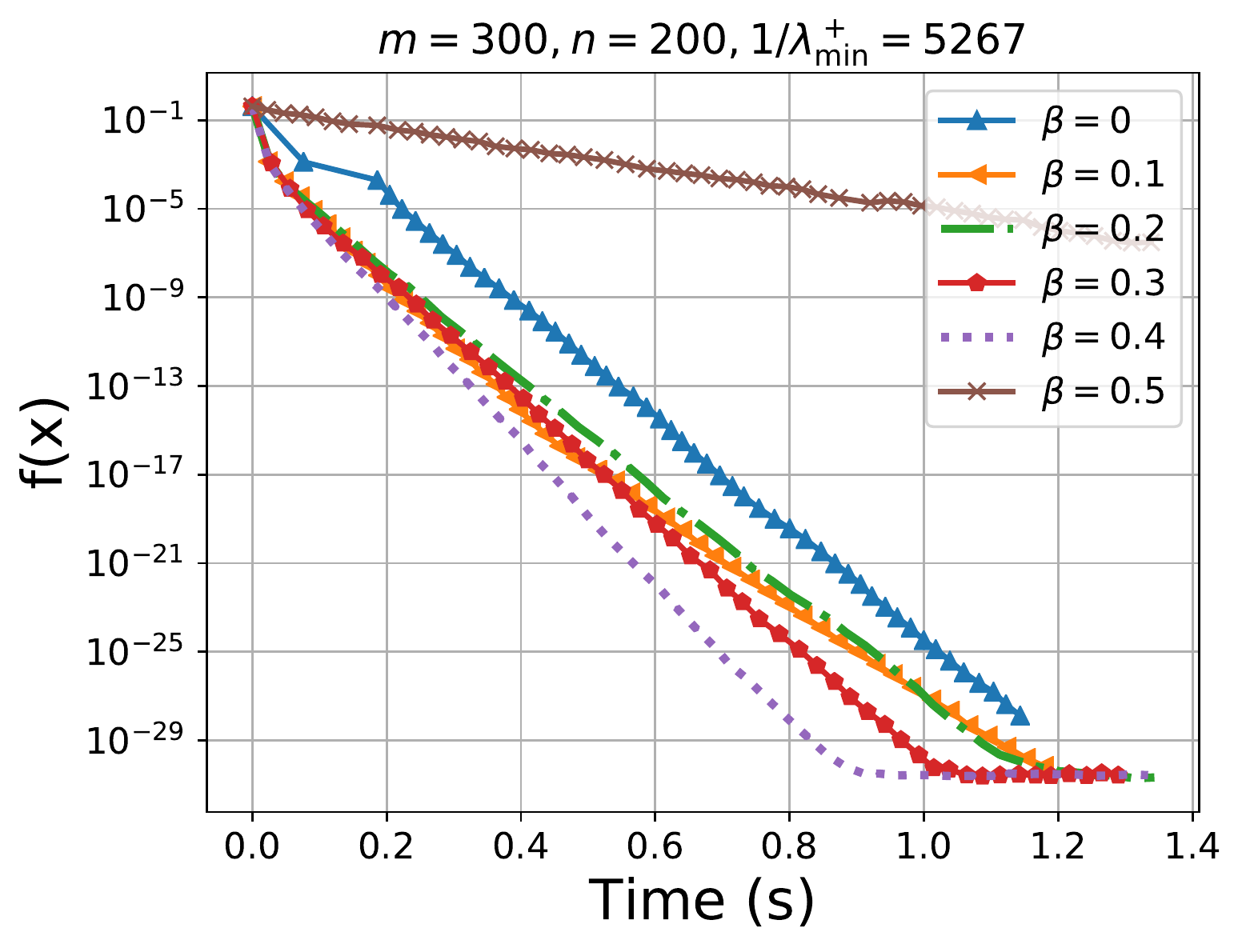}
\end{subfigure}\\
\begin{subfigure}{.35\textwidth}
  \centering
    \includegraphics[width=1\linewidth]{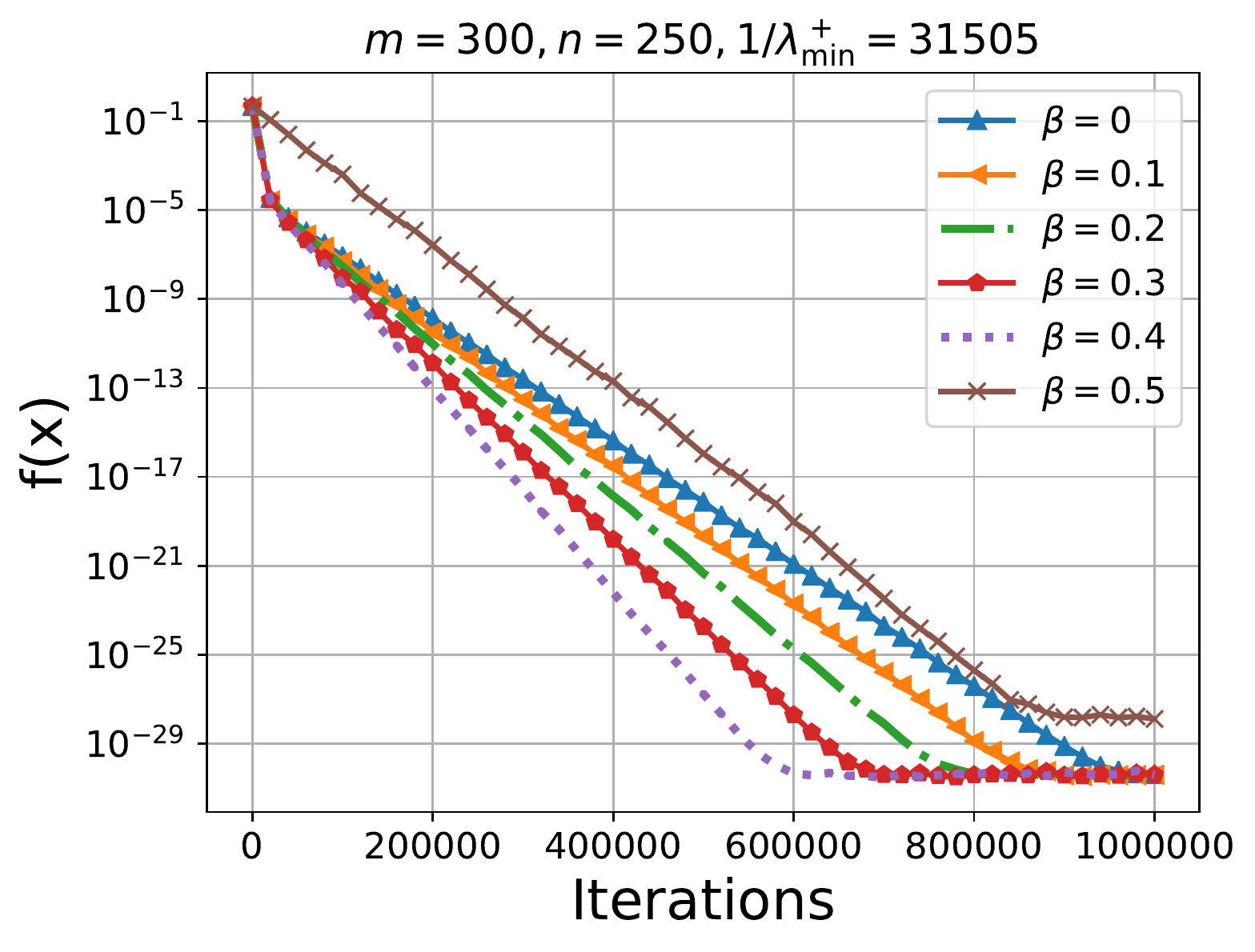}
\end{subfigure}
\begin{subfigure}{.35\textwidth}
  \centering
  \includegraphics [width=1\linewidth] {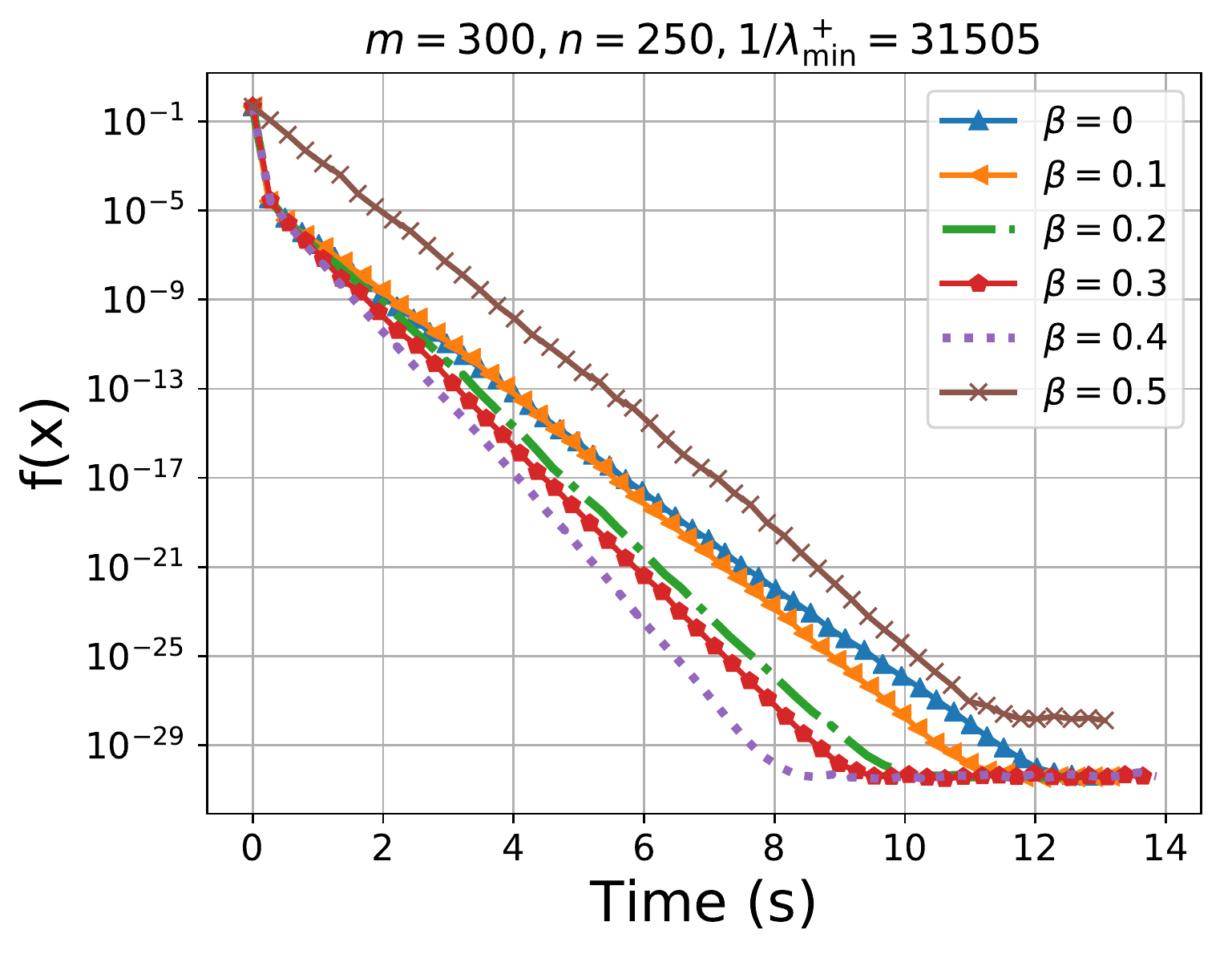} 
\end{subfigure}\\
\begin{subfigure}{.35\textwidth}
  \centering
  \includegraphics[width=1\linewidth]{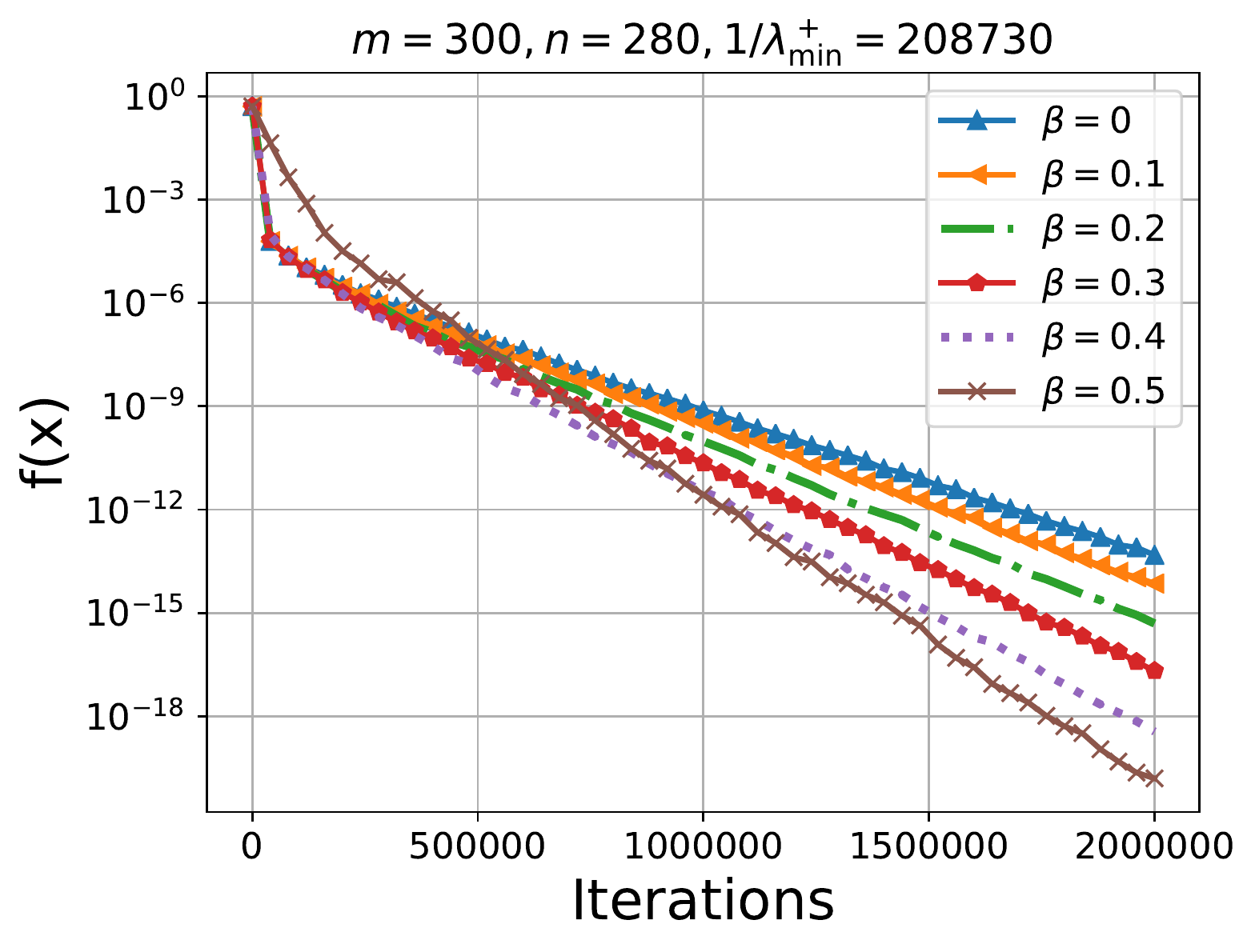}
\end{subfigure}
\begin{subfigure}{.35\textwidth}
  \centering
  \includegraphics[width=1\linewidth]{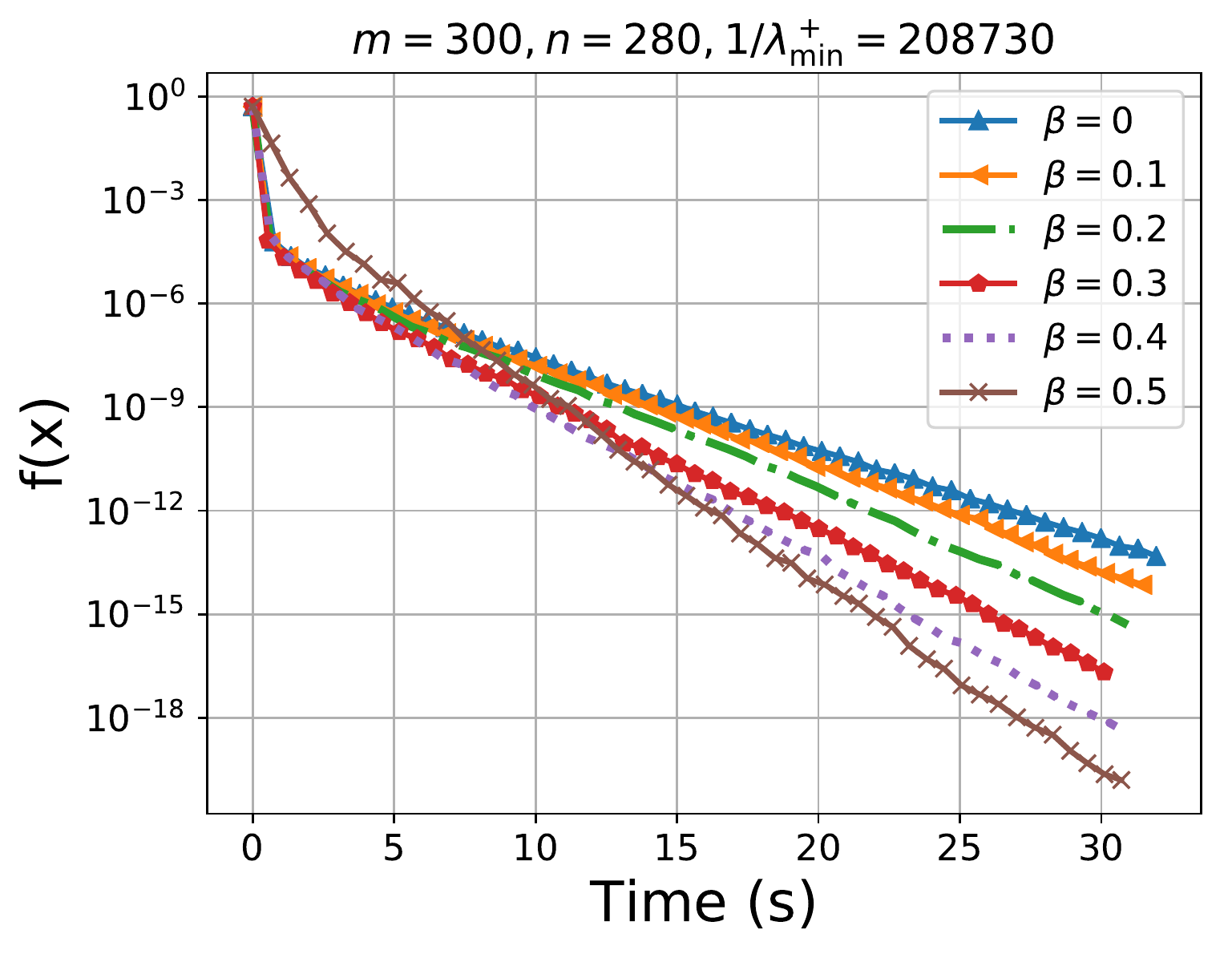}
\end{subfigure}\\
\begin{subfigure}{.35\textwidth}
  \centering
  \includegraphics[width=1\linewidth]{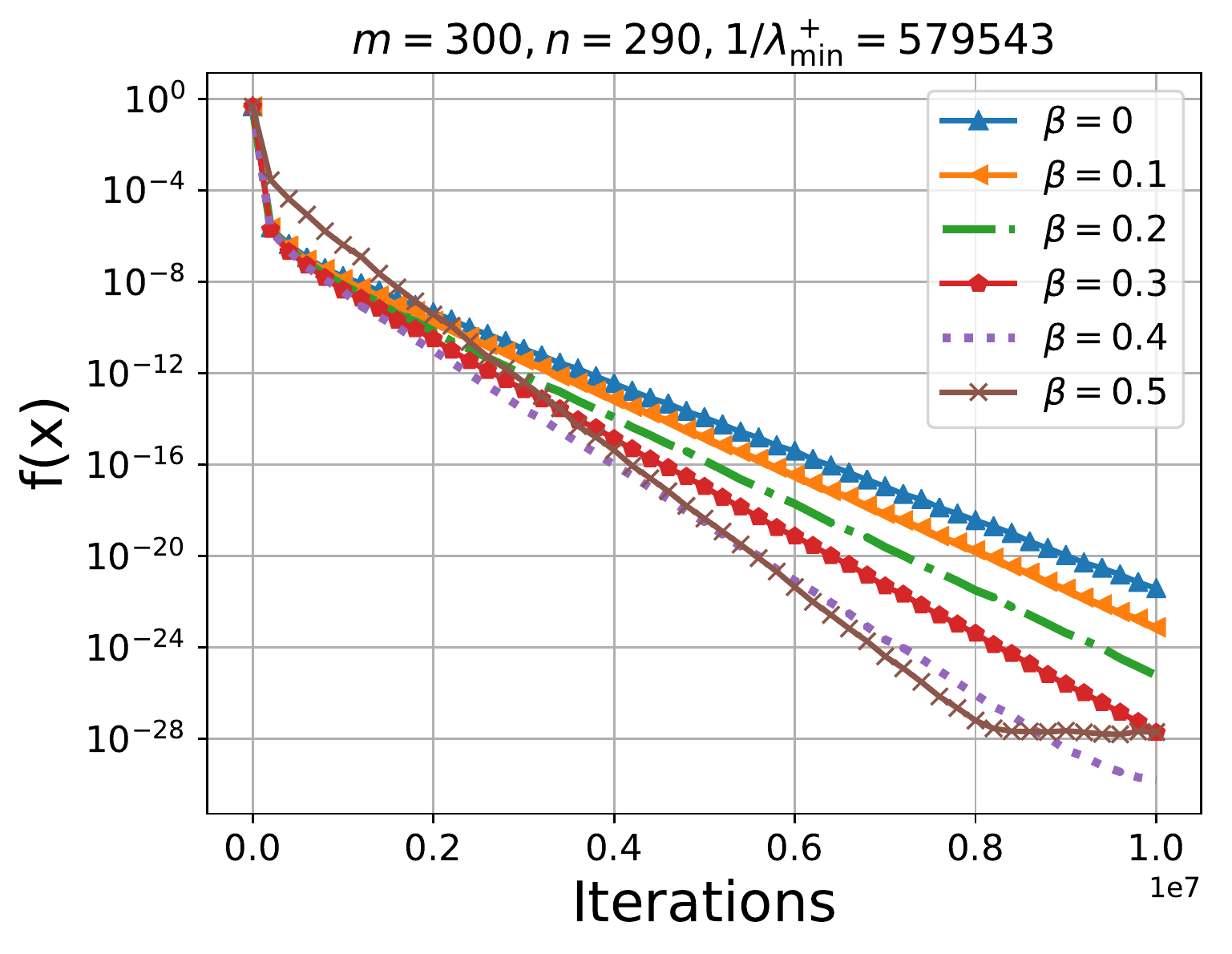}
\end{subfigure}
\begin{subfigure}{.35\textwidth}
  \centering
  \includegraphics[width=1\linewidth]{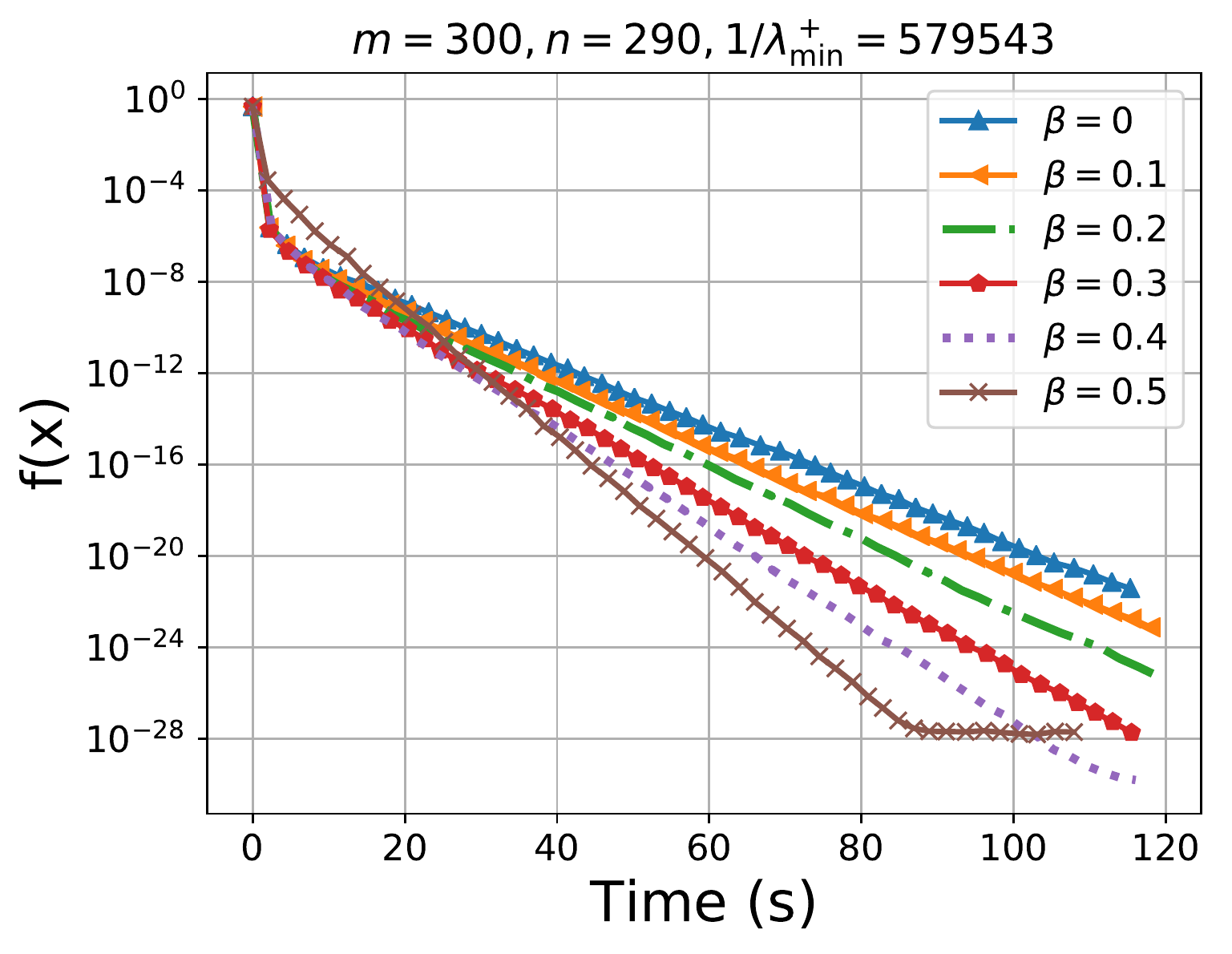}
\end{subfigure}\\
\caption{Performance of mRK \blue{(the method in the first row of Table~\ref{SpecialCasesAlgorithms})} for fixed stepsize $\omega=1$ and several momentum parameters $\beta$ for consistent linear systems with Gaussian matrix $\bA$ with $m=300$ rows and $n=100,200,250,280,290$ columns. The graphs in the first (second) column plot iterations (time) against function values. All plots are averaged over 10 trials. The title of each plot indicates the dimensions of the matrix $\bA$ and the value of $1/\lambda_{\min}^+$. The function values $f(x^k)$ refer to function~\eqref{functionRK}.}
\label{RKperformace12}
\end{figure}

\begin{figure}[!]
\centering
\begin{subfigure}{.35\textwidth}
  \centering
  \includegraphics[width=1\linewidth]{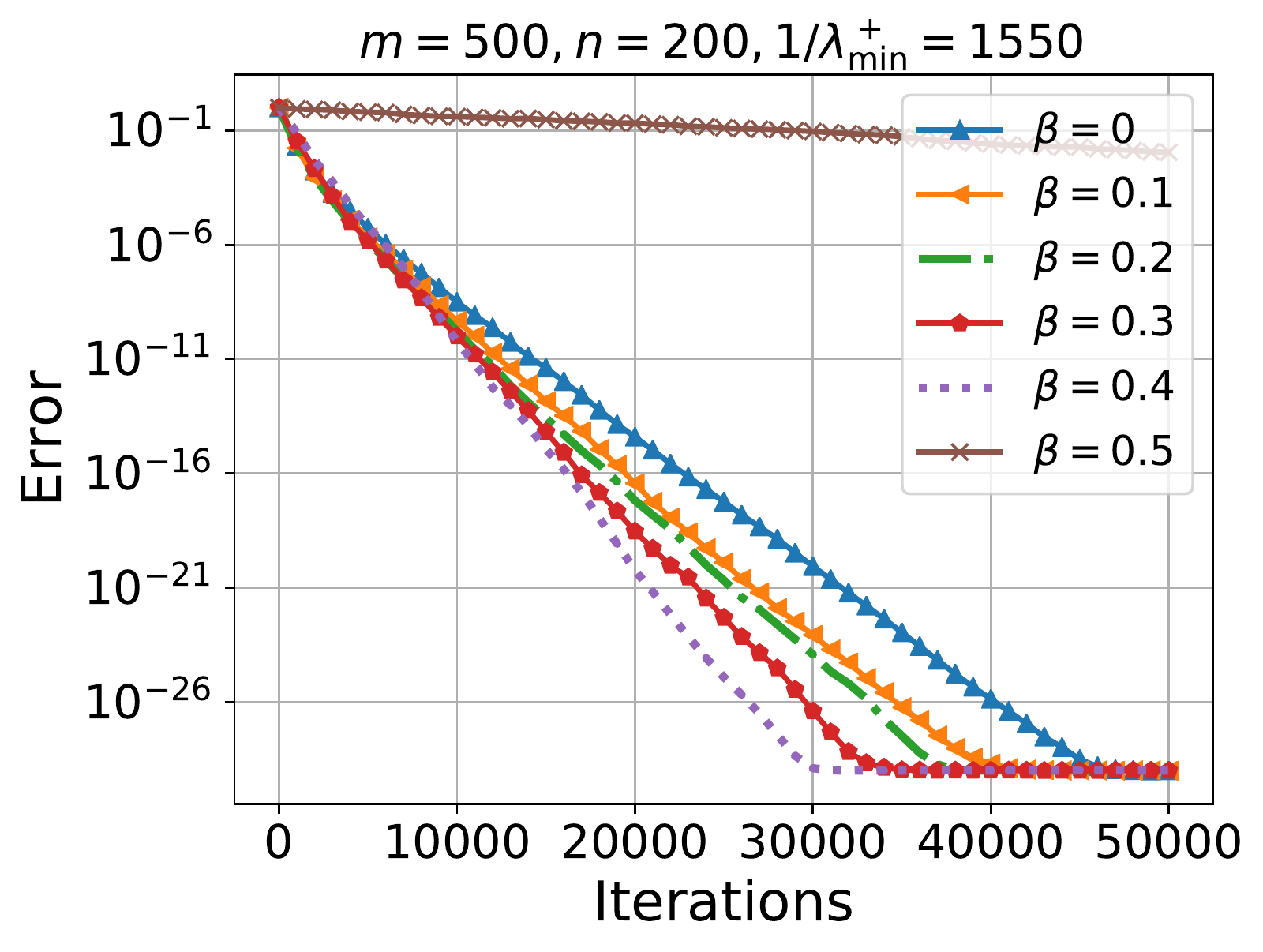}
\end{subfigure}%
\begin{subfigure}{.35\textwidth}
  \centering
  \includegraphics[width=1\linewidth]{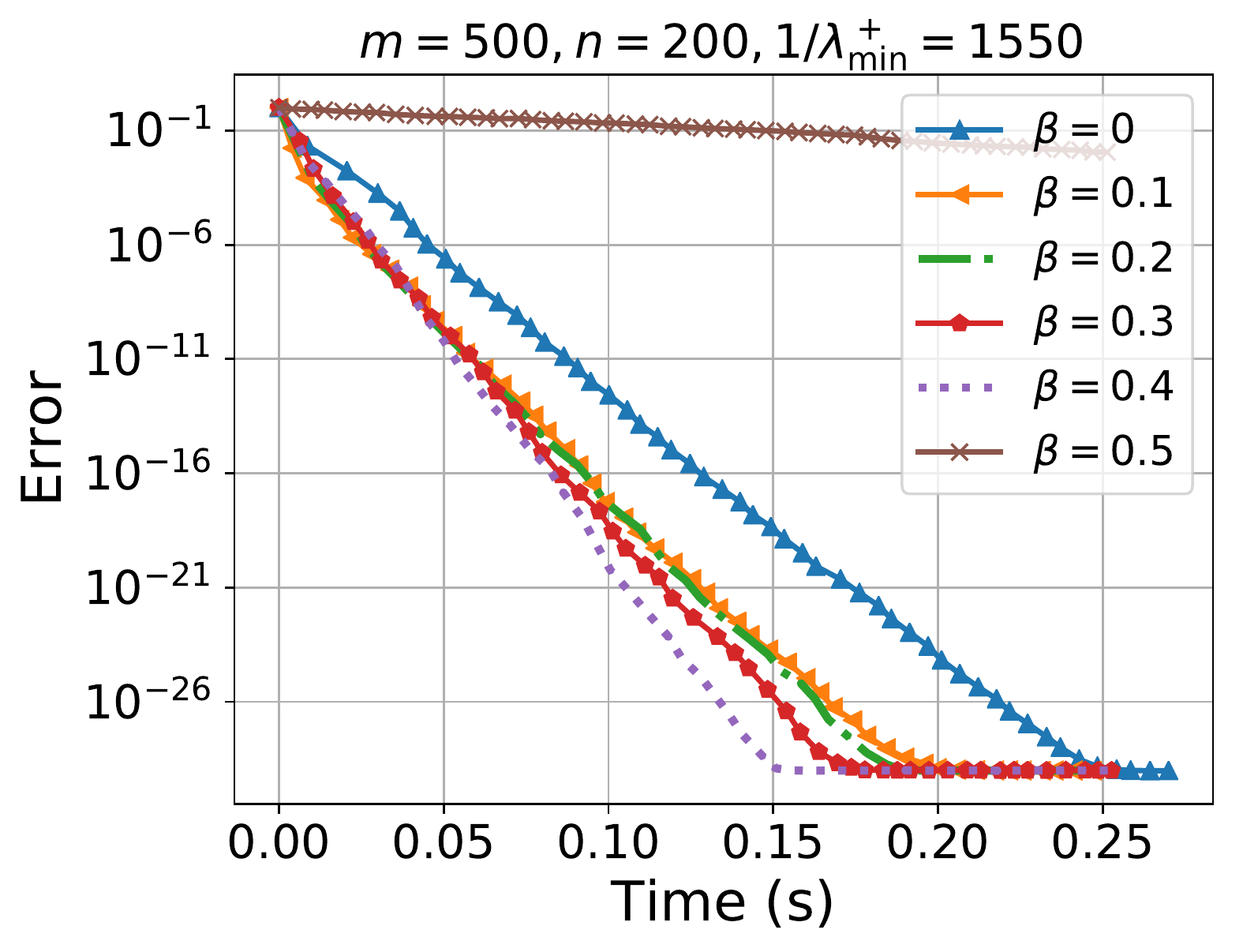}
\end{subfigure}\\
\begin{subfigure}{.35\textwidth}
  \centering
  \includegraphics[width=1\linewidth]{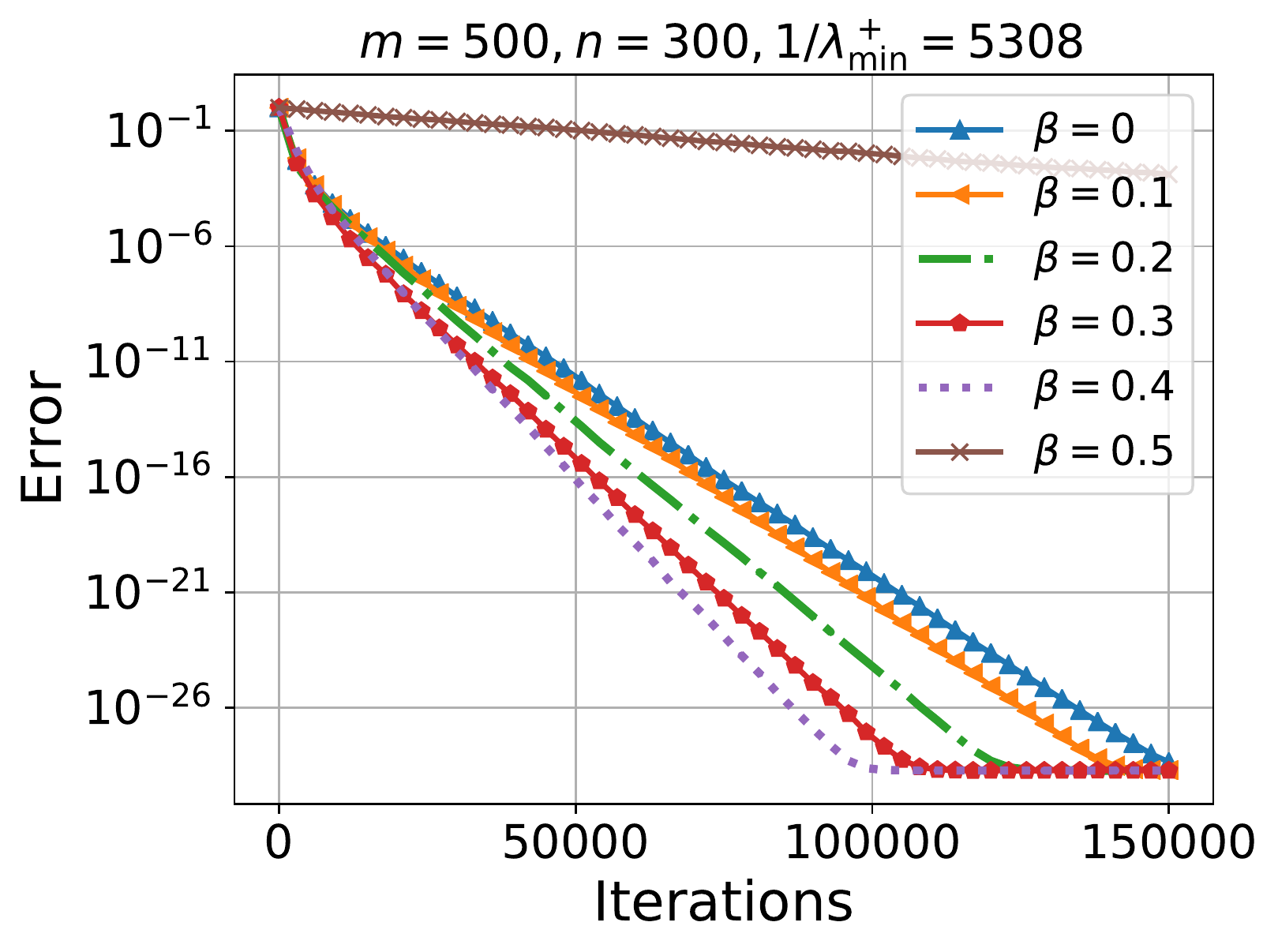}
\end{subfigure}%
\begin{subfigure}{.35\textwidth}
  \centering
  \includegraphics[width=1\linewidth]{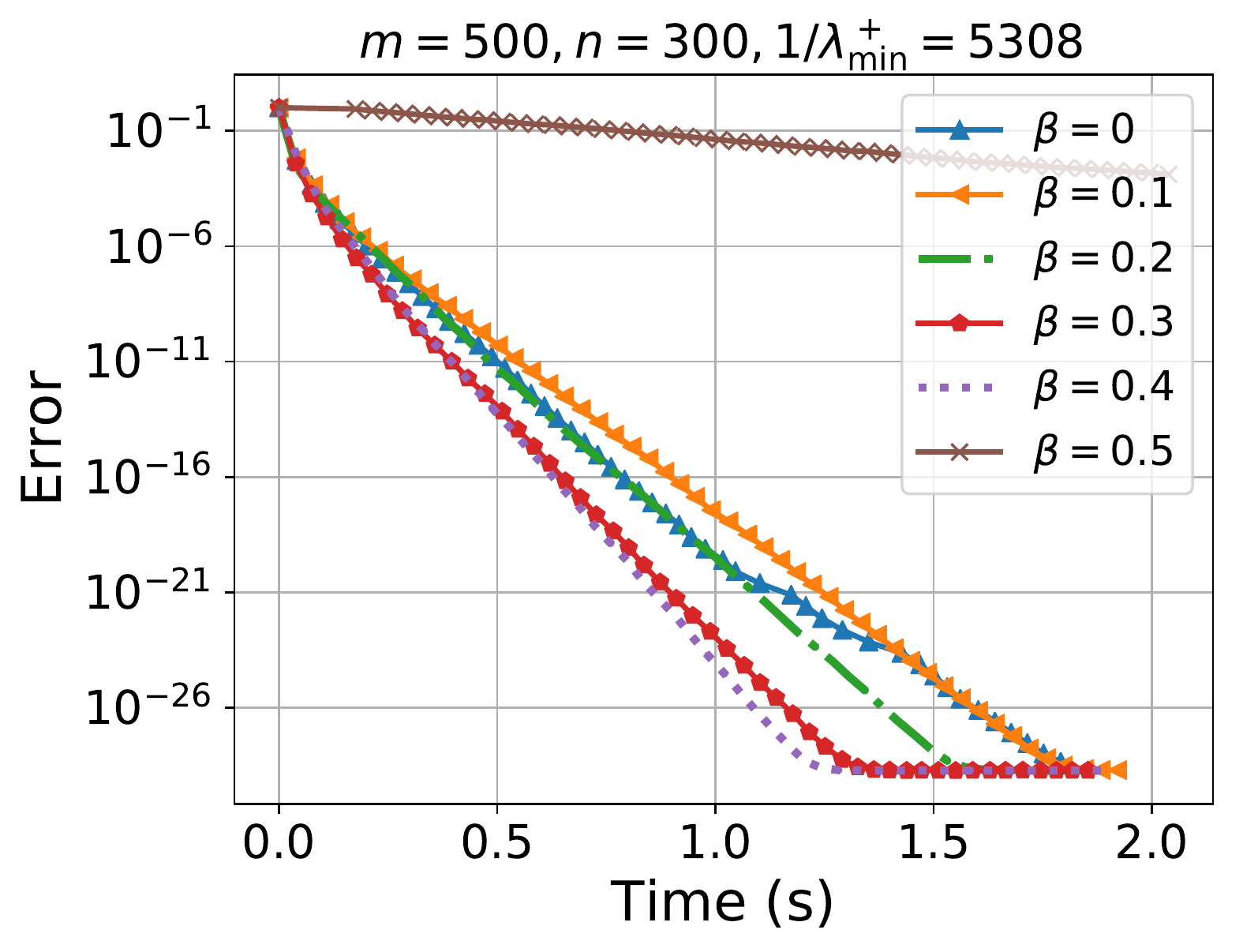}
\end{subfigure}\\
\begin{subfigure}{.35\textwidth}
  \centering
  \includegraphics[width=1\linewidth]{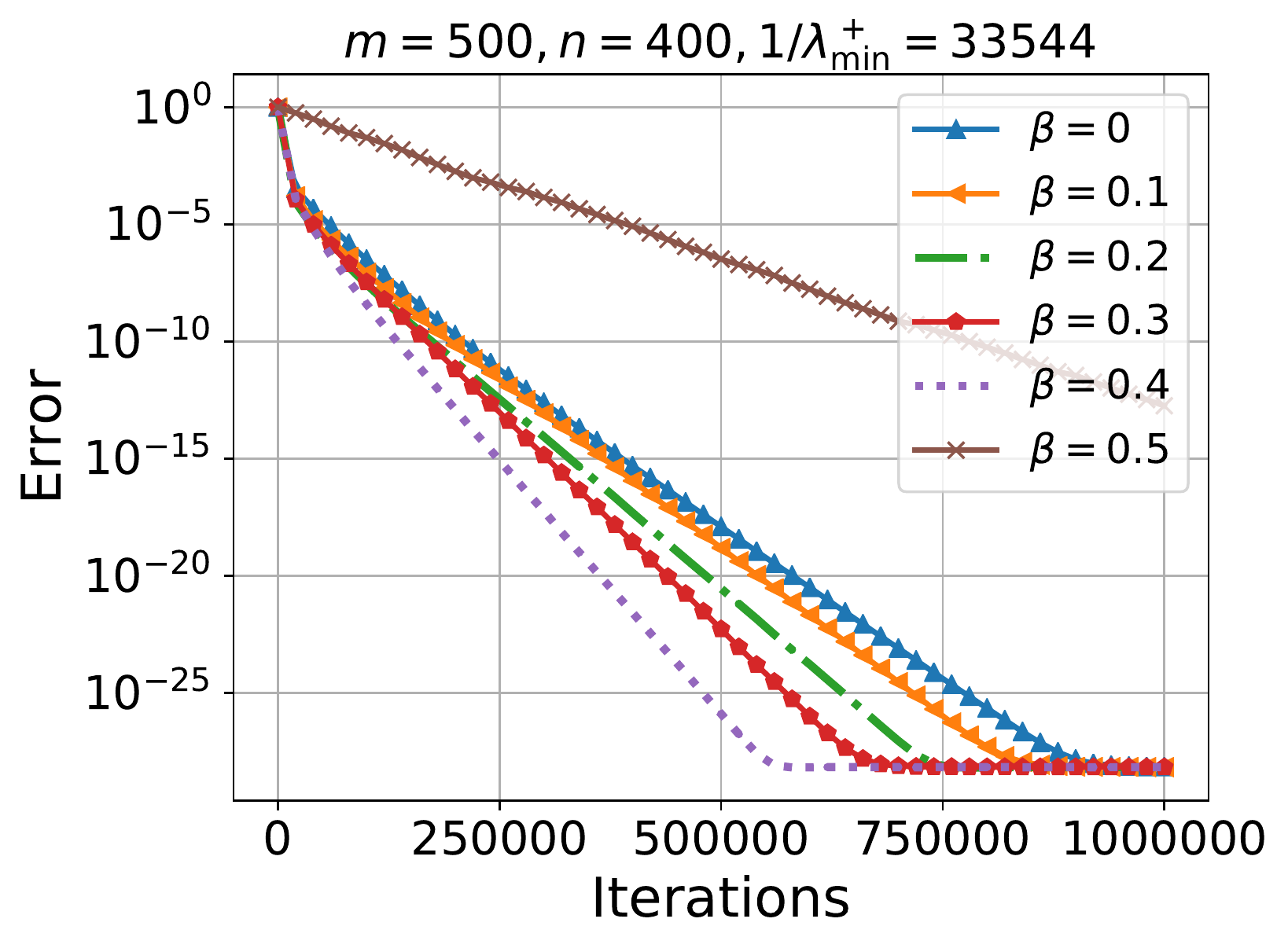}
\end{subfigure}%
\begin{subfigure}{.35\textwidth}
  \centering
  \includegraphics[width=1\linewidth]{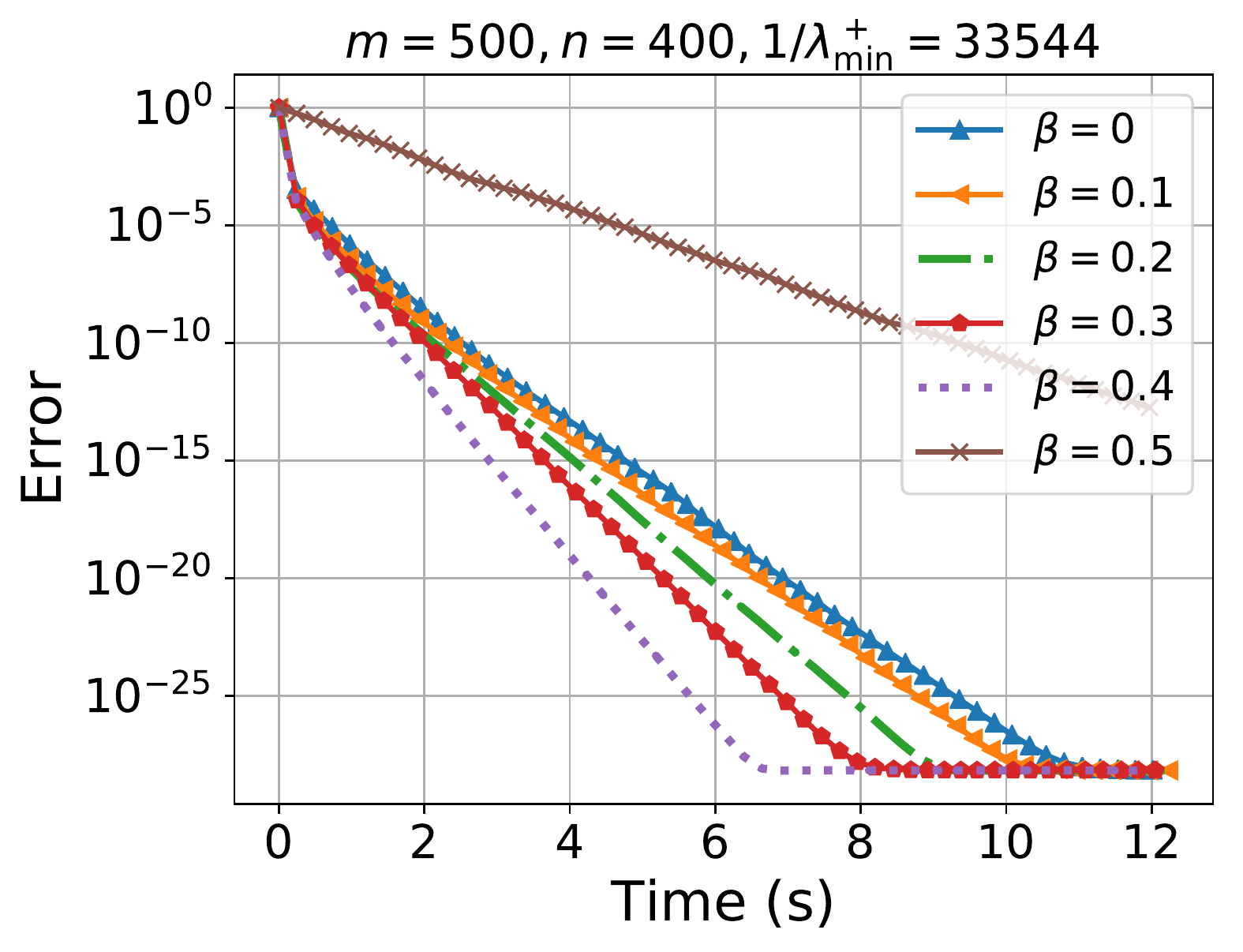}
\end{subfigure}\\
\begin{subfigure}{.35\textwidth}
  \centering
  \includegraphics[width=1\linewidth]{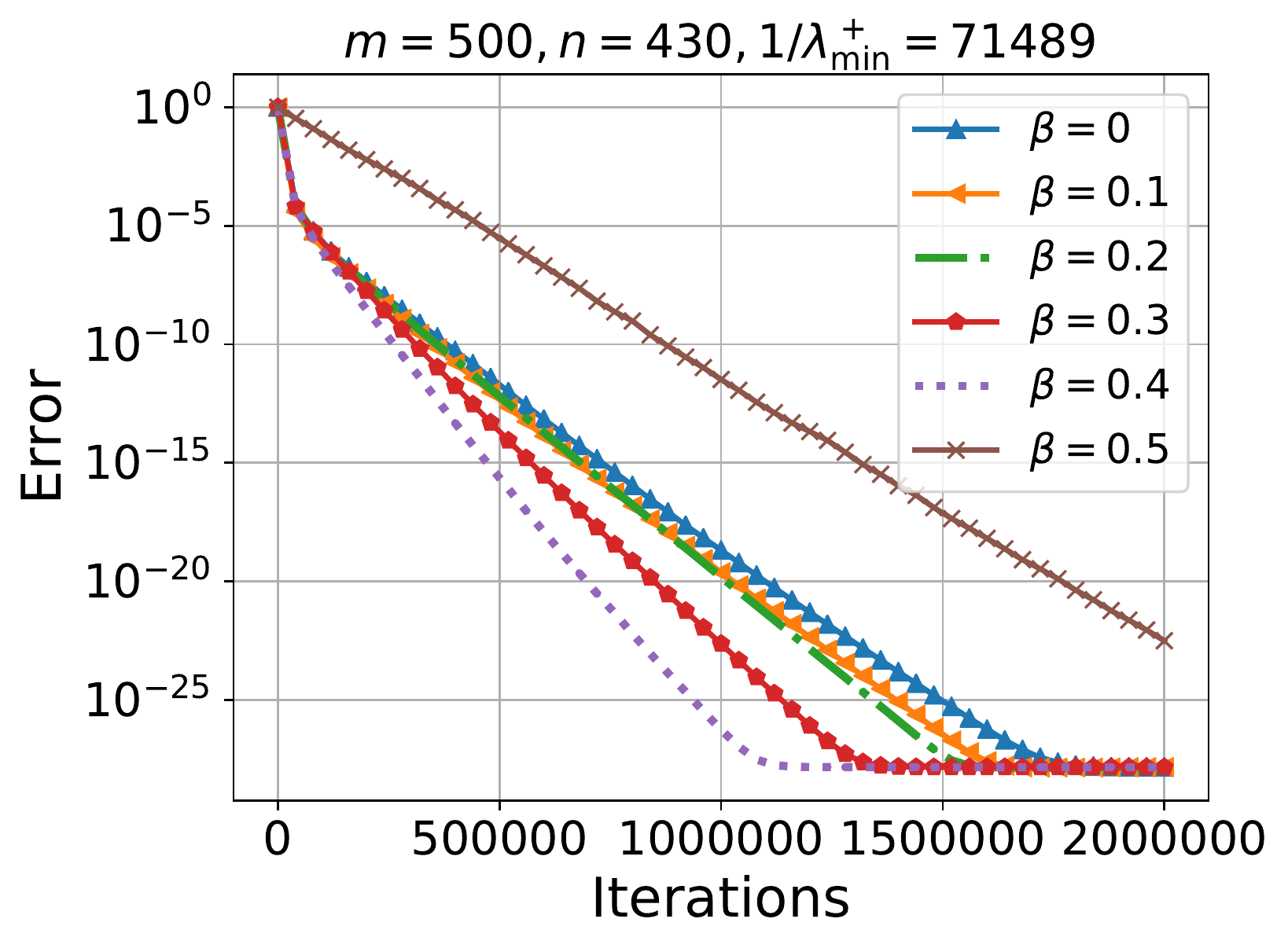}
\end{subfigure}%
\begin{subfigure}{.35\textwidth}
  \centering
  \includegraphics[width=1\linewidth]{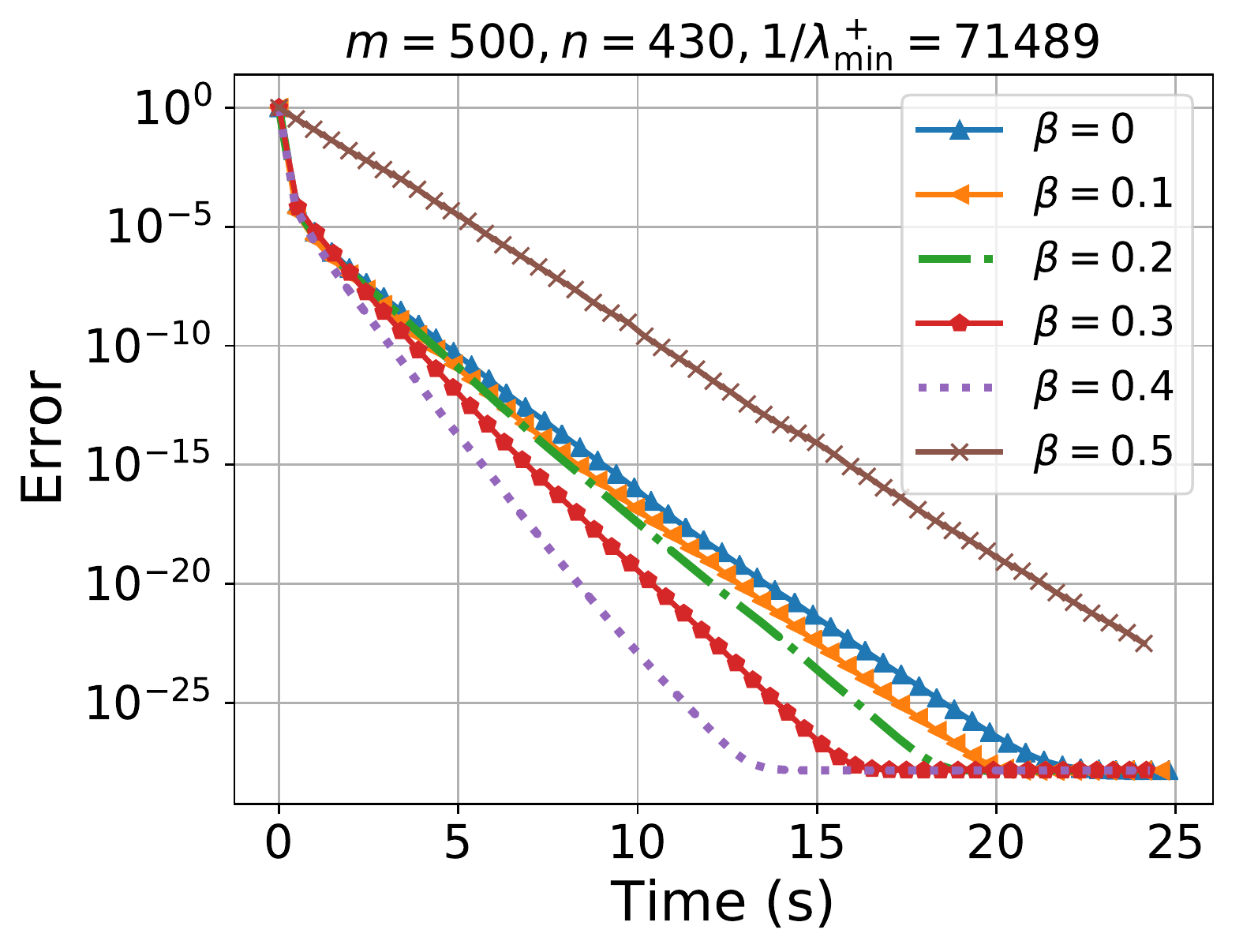}
\end{subfigure}\\
\begin{subfigure}{.35\textwidth}
  \centering
  \includegraphics[width=1\linewidth]{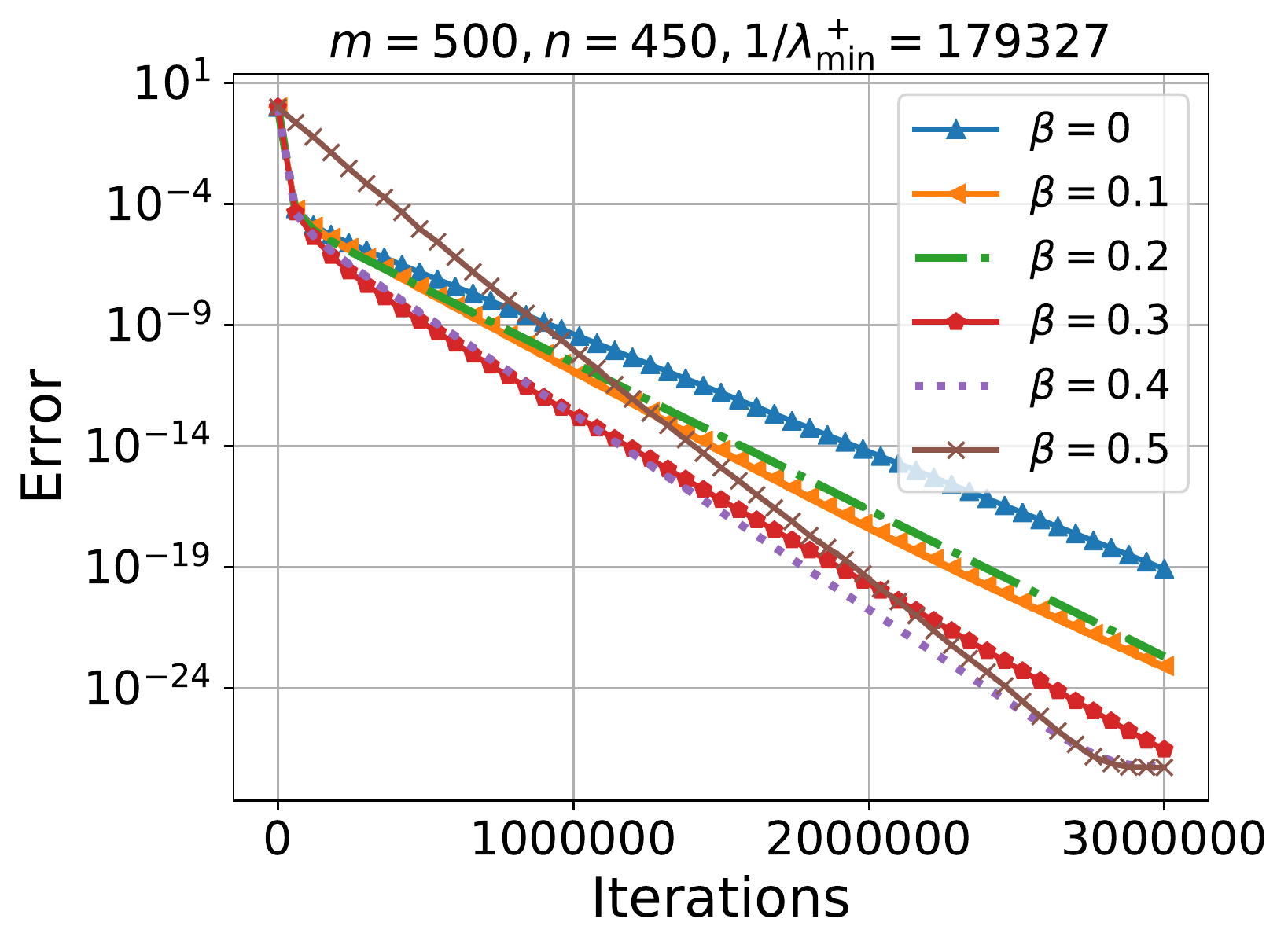}
\end{subfigure}%
\begin{subfigure}{.35\textwidth}
  \centering
  \includegraphics[width=1\linewidth]{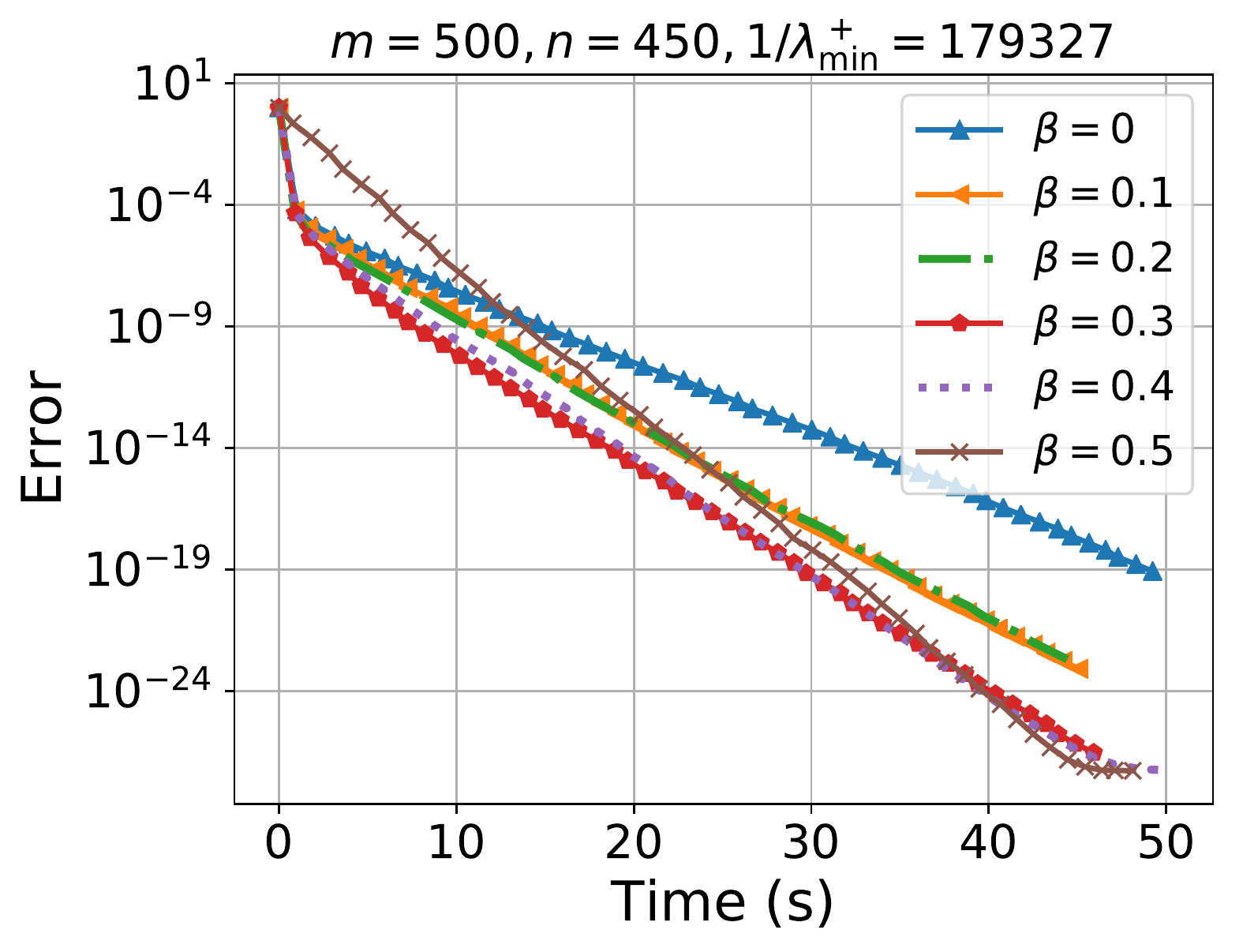}
\end{subfigure}
\caption{Performance of mRCD \blue{(the method in the second row of Table~\ref{SpecialCasesAlgorithms})}. for fixed stepsize $\omega=1$ and several momentum parameters $\beta$ for consistent linear systems with positive definite matrices $\bA=\bP^\top \bP$ where $\bP \in \R^{m \times n}$ is Gaussian matrix with $m=500$ rows and $n=200,300,400,430,450$. The graphs in the first (second) column plot iterations (time) against residual error. All plots are averaged over 10 trials. The title of each plot indicates the dimensions of the matrix $\bP$ and the value of $1/\lambda_{\min}^+$. The ``Error" on the vertical axis represents the relative error $\|x^k-x^*\|^2_\bB / \|x^0-x^*\|^2_\bB \overset{\bB=\bA, x^0=0}{=}\|x^k-x^*\|^2_\bA / \|x^*\|^2_\bA$.}
\label{RCDperformance1}
\end{figure}

\begin{figure}[!]
\centering
\begin{subfigure}{.35\textwidth}
  \centering
  \includegraphics[width=1\linewidth]{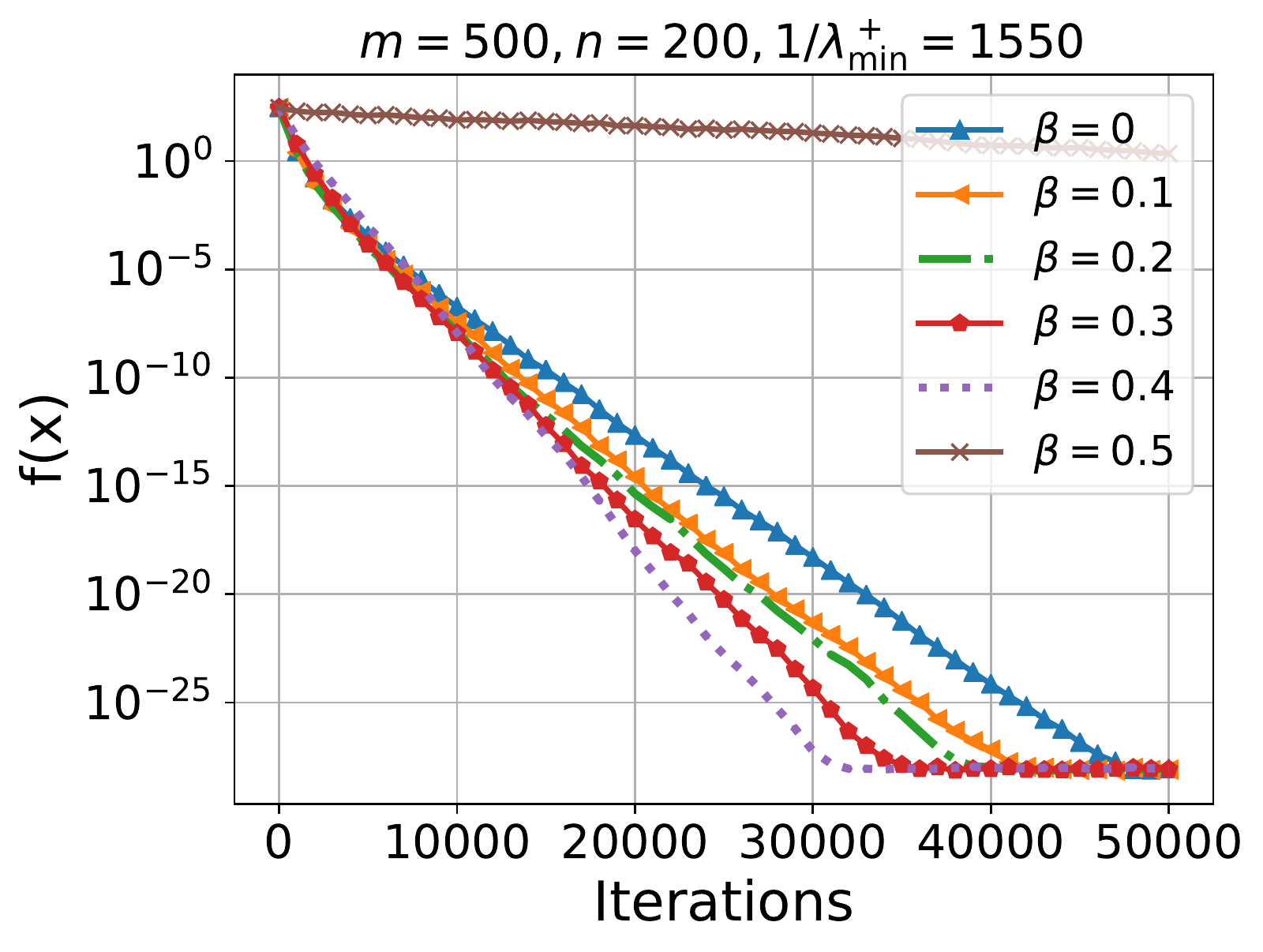}
\end{subfigure}
\begin{subfigure}{.35\textwidth}
  \centering
  \includegraphics[width=1\linewidth]{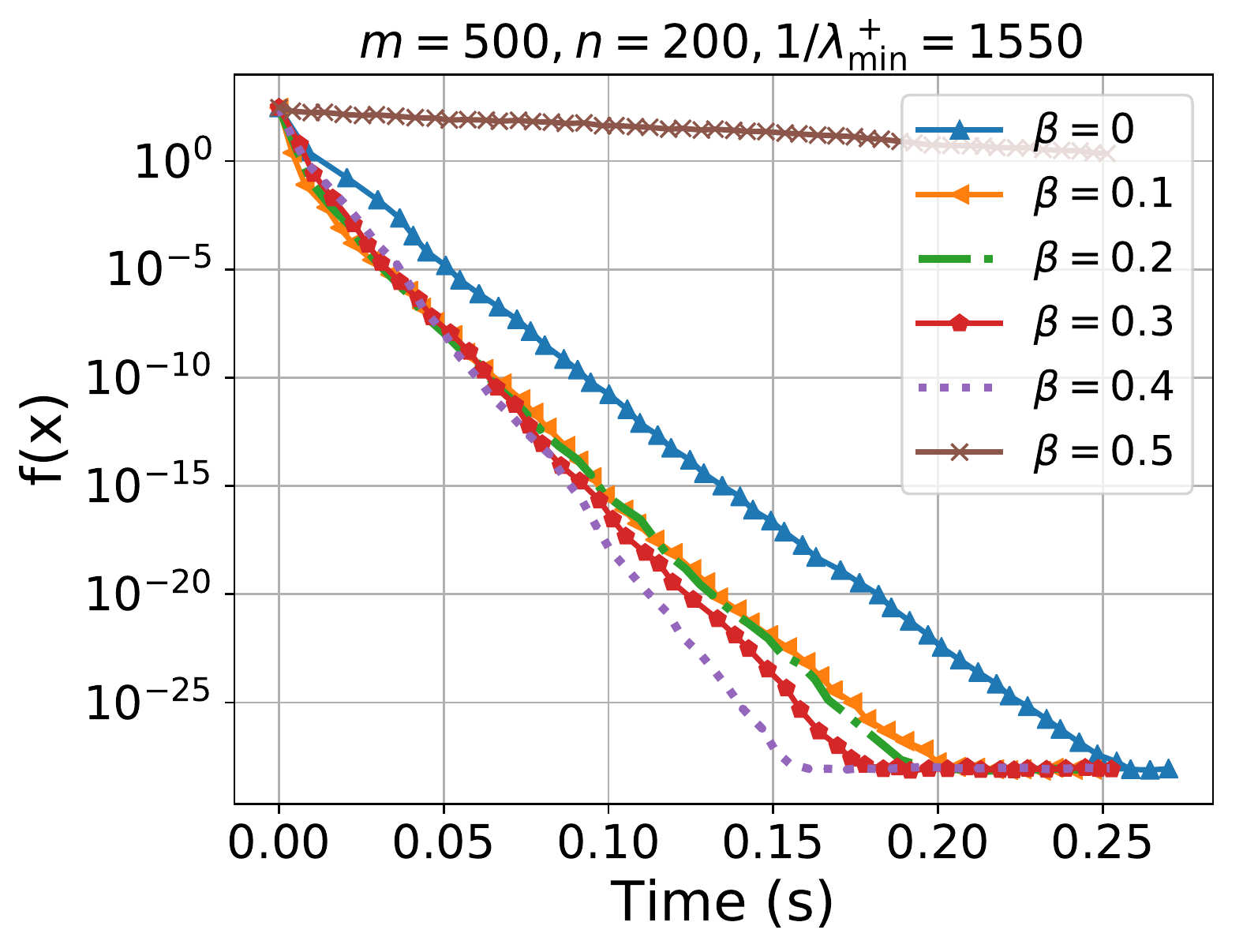}
\end{subfigure}\\
\begin{subfigure}{.35\textwidth}
  \centering
  \includegraphics[width=1\linewidth]{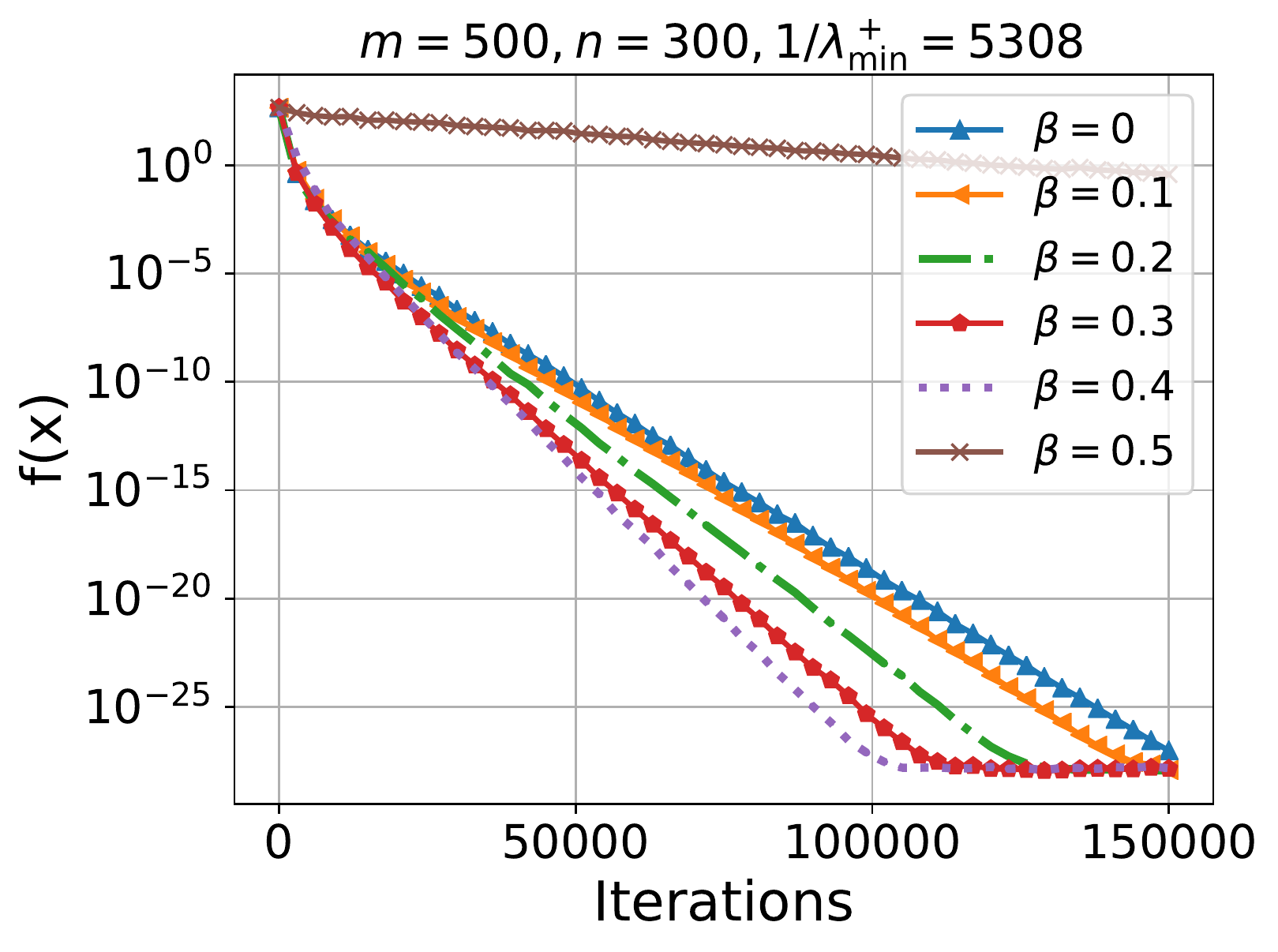}
\end{subfigure}
\begin{subfigure}{.35\textwidth}
  \centering
  \includegraphics[width=1\linewidth]{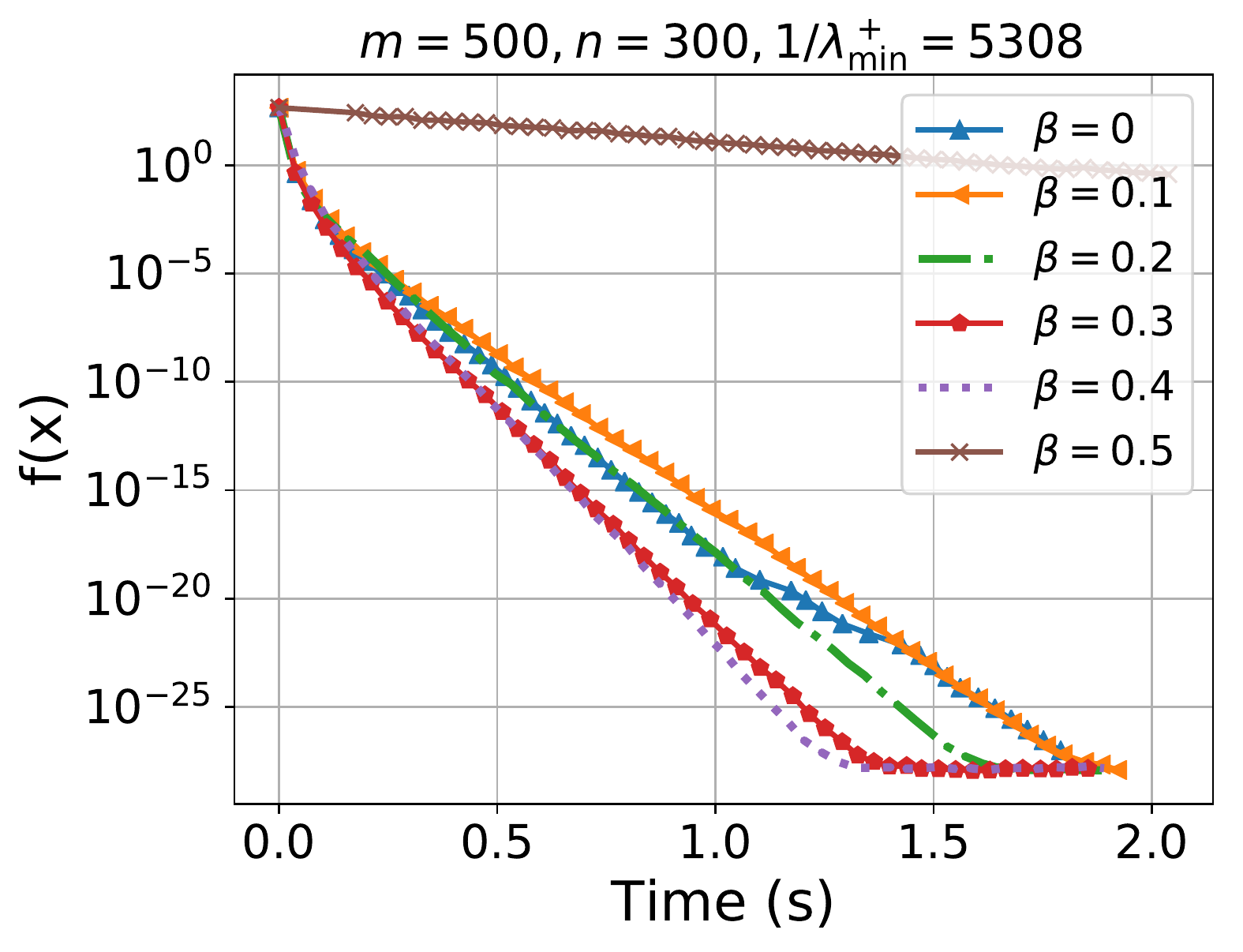}
\end{subfigure}\\
\begin{subfigure}{.35\textwidth}
  \centering
  \includegraphics[width=1\linewidth]{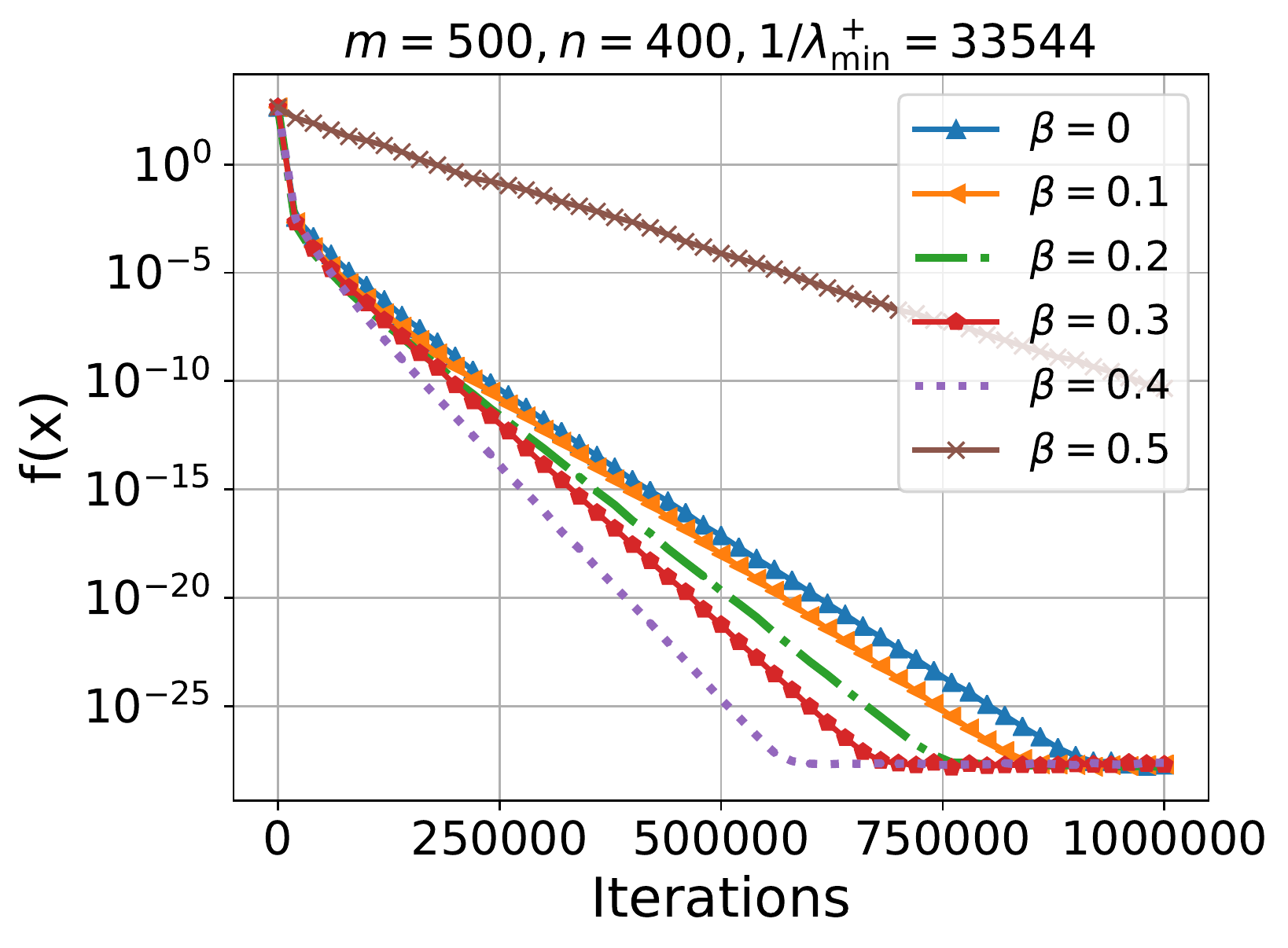}
\end{subfigure}
\begin{subfigure}{.35\textwidth}
  \centering
  \includegraphics[width=1\linewidth]{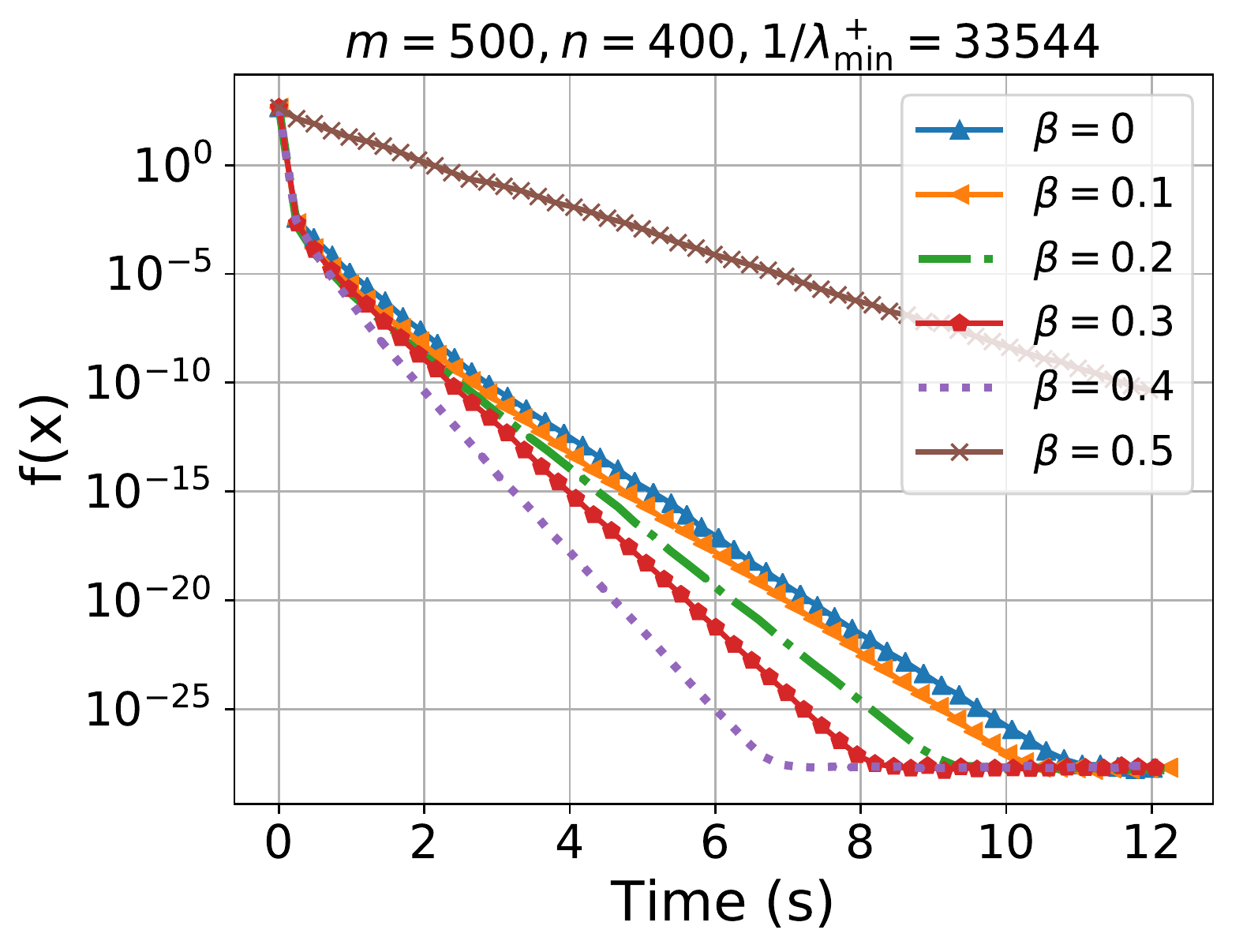}
\end{subfigure}\\
\begin{subfigure}{.35\textwidth}
  \centering
  \includegraphics[width=1\linewidth]{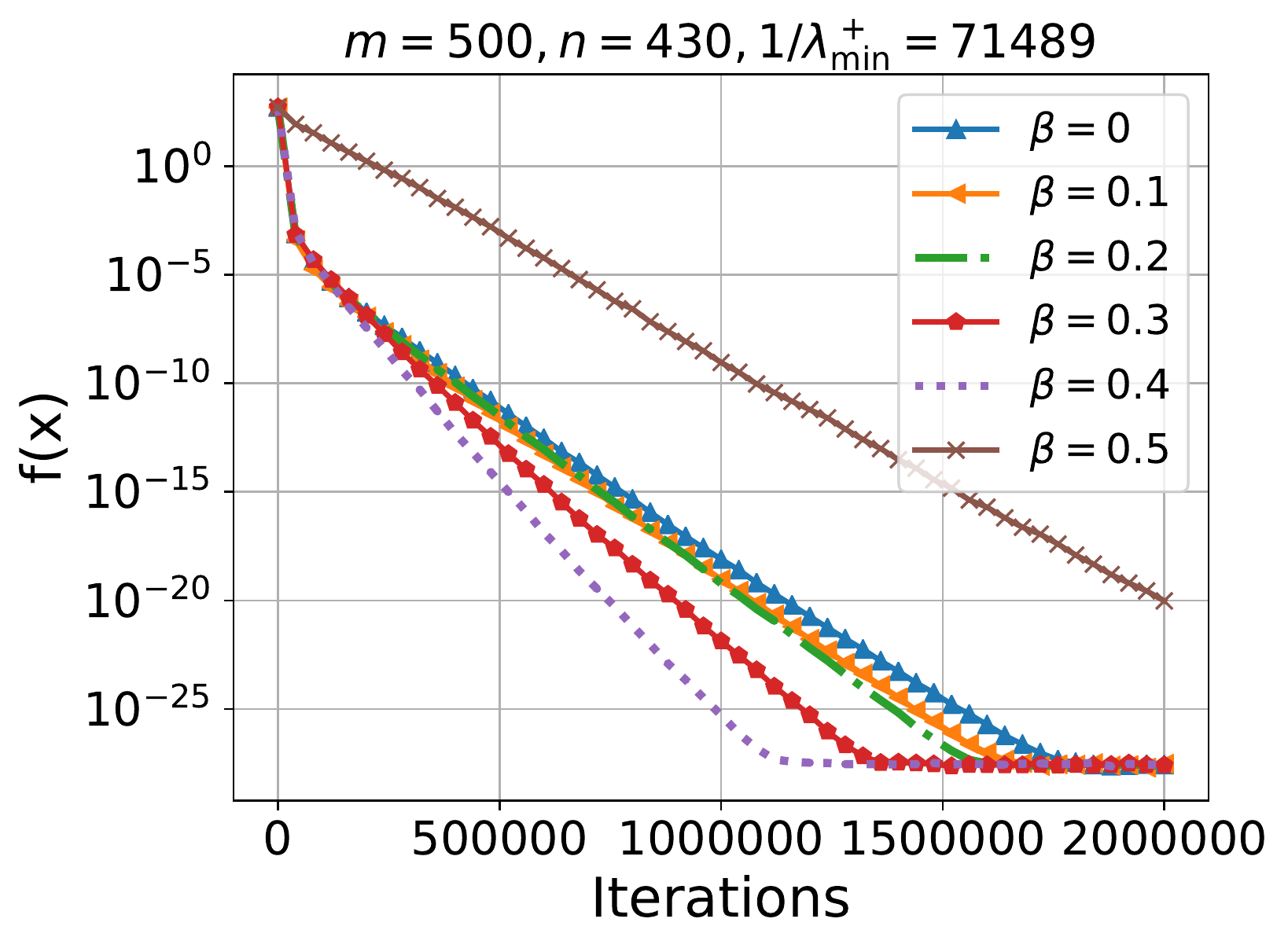}
\end{subfigure}
\begin{subfigure}{.35\textwidth}
  \centering
  \includegraphics[width=1\linewidth]{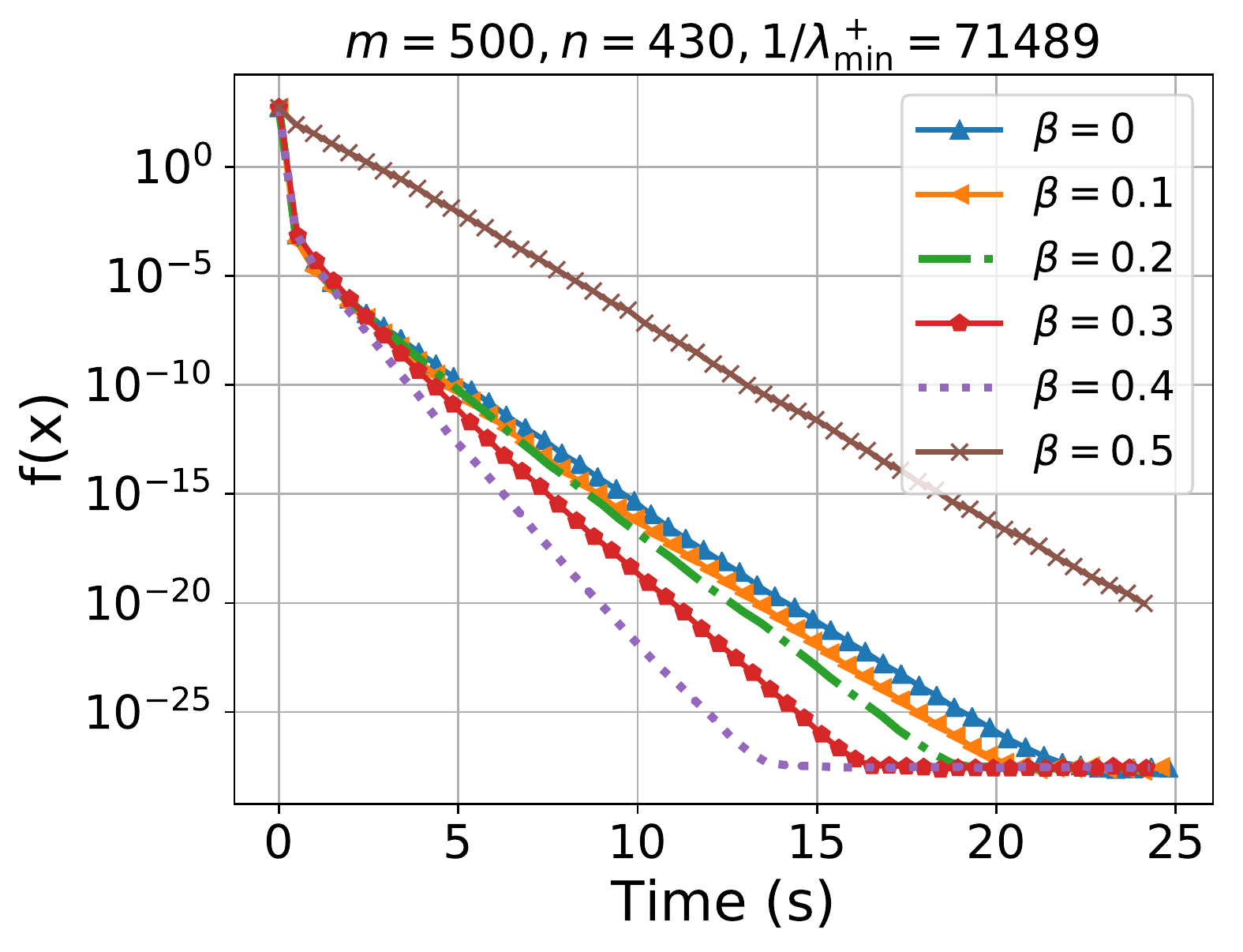}
\end{subfigure}\\
\begin{subfigure}{.35\textwidth}
  \centering
  \includegraphics[width=1\linewidth]{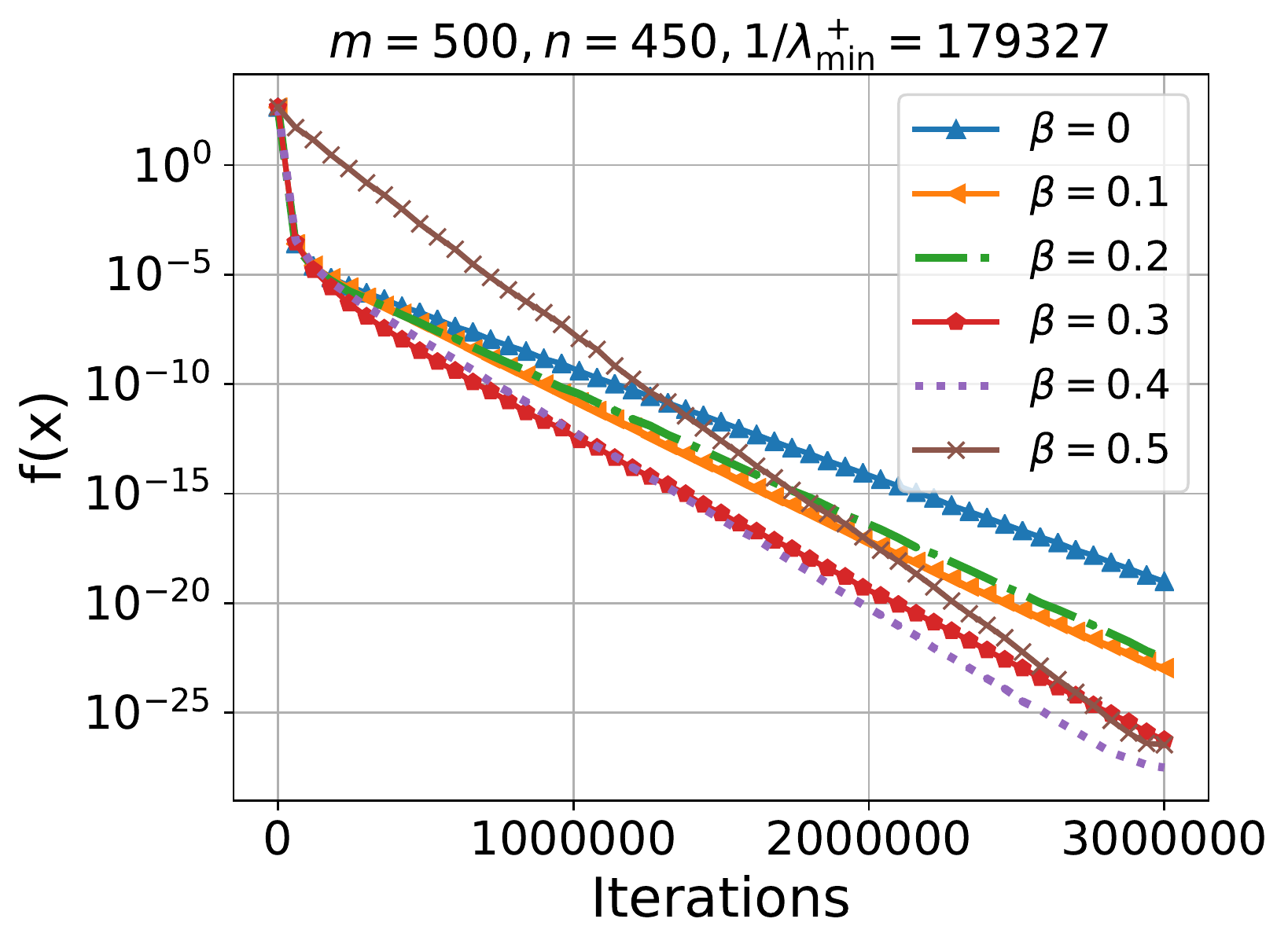}
\end{subfigure}
\begin{subfigure}{.35\textwidth}
  \centering
  \includegraphics[width=1\linewidth]{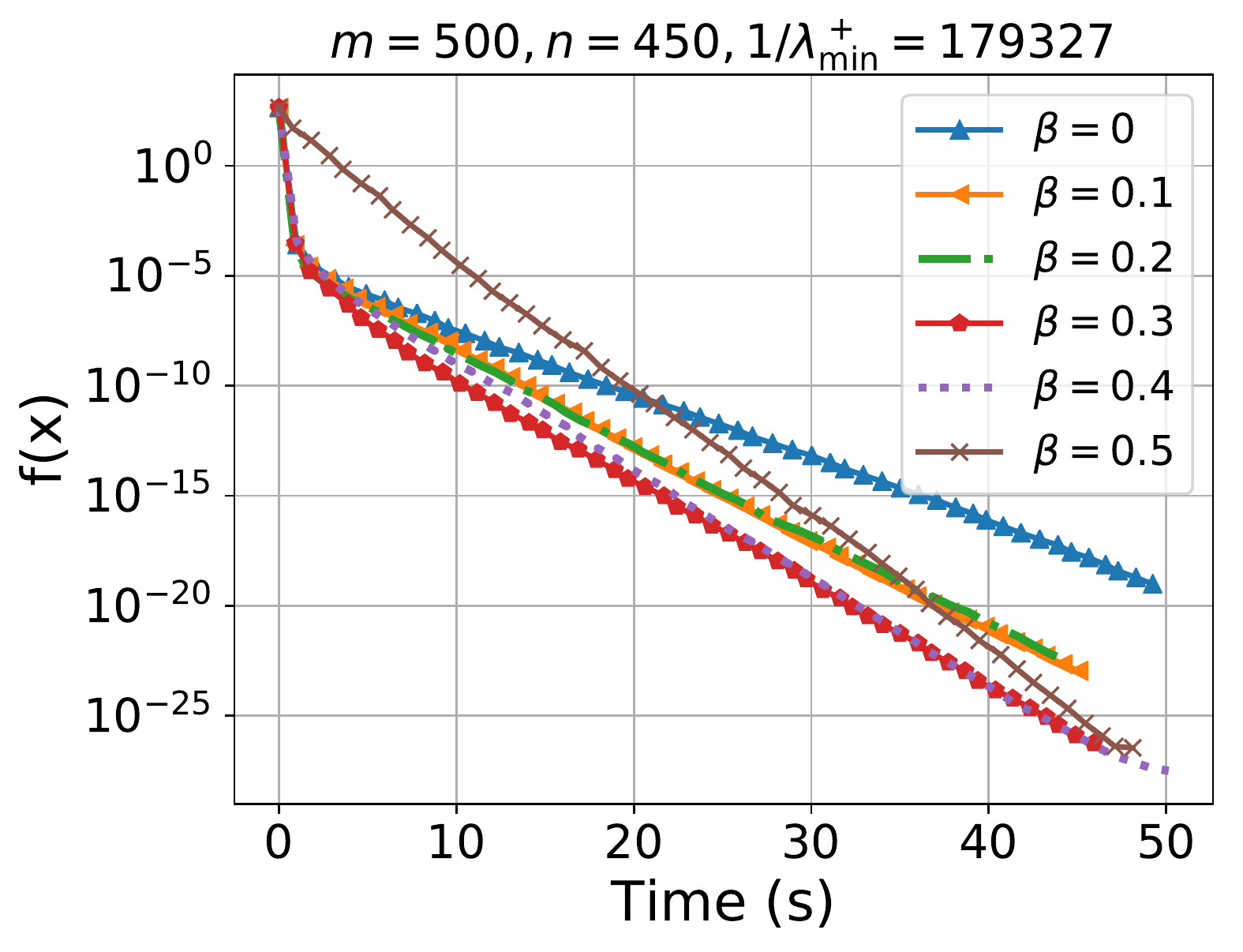}
\end{subfigure}
\caption{Performance of mRCD \blue{(the method in the second row of Table~\ref{SpecialCasesAlgorithms})}. for fixed stepsize $\omega=1$ and several momentum parameters $\beta$ for consistent linear systems with positive definite matrices $\bA=\bP^\top \bP$ where $\bP \in \R^{m \times n}$ is Gaussian matrix with $m=500$ rows and $n=200,300,400,430,450$. The graphs in the first (second) column plot iterations (time) against function values. All plots are averaged over 10 trials. The title of each plot indicates the dimensions of the matrix $\bP$ and the value of $1/\lambda_{\min}^+$. The function values $f(x^k)$ refer to function~\eqref{functionRCD}.}
\label{RCDperformance12}
\end{figure}

\subsubsection{Real Data}
In the following experiments we test the performance of mRK using real matrices (datasets) from the library of support vector machine problems LIBSVM \cite{chang2011libsvm}. Each dataset consists of a matrix $\bA \in \R^{m \times n}$  ($m$ features and $n$ characteristics) and a vector of labels $b \in \R^m$.  In our experiments we choose to use only the matrices of the datasets and ignore the label vector. As before, to ensure consistency of the linear system,  we choose a Gaussian vector $z\in \R^n$ and the right hand side of the linear system is set to $b=\bA z$. 
Similarly as in  the case of synthetic data, mRK is tested for several values of momentum parameters $\beta$ and fixed stepsize $\omega=1$. 

In Figure~\ref{RealDataplots} the performance of all methods for both relative error measure $\|x^k-x^*\|^2 / \|x^*\|^2_\bB$ and function values $f(x^k)$ is presented. Note again that $\beta=0$ represents the baseline RK method. The addition of momentum parameter is again often beneficial and leads to faster convergence. As an example, inspect the plots for the \textit{mushrooms} dataset in Figure~\ref{RealDataplots}, where mRK with $\beta=0.5$ is much faster than the simple RK method in all presented plots, both in terms of iterations and time. In particular, the addition of a momentum parameter leads to visible speedup for the datasets \textit{mushrooms}, \textit{splice}, \textit{a9a} and \textit{ionosphere}. For these datasets the acceleration is obvious in all plots both in terms of relative error and function values. For the datasets  \textit{australian}, \textit{gisette} and \textit{madelon} the speedup is less obvious in the plots of the relative error, while for the plots of function values it is not present at all.

\begin{figure}[!]
\centering
\begin{subfigure}{.23\textwidth}
  \centering
  \includegraphics[width=1\linewidth]{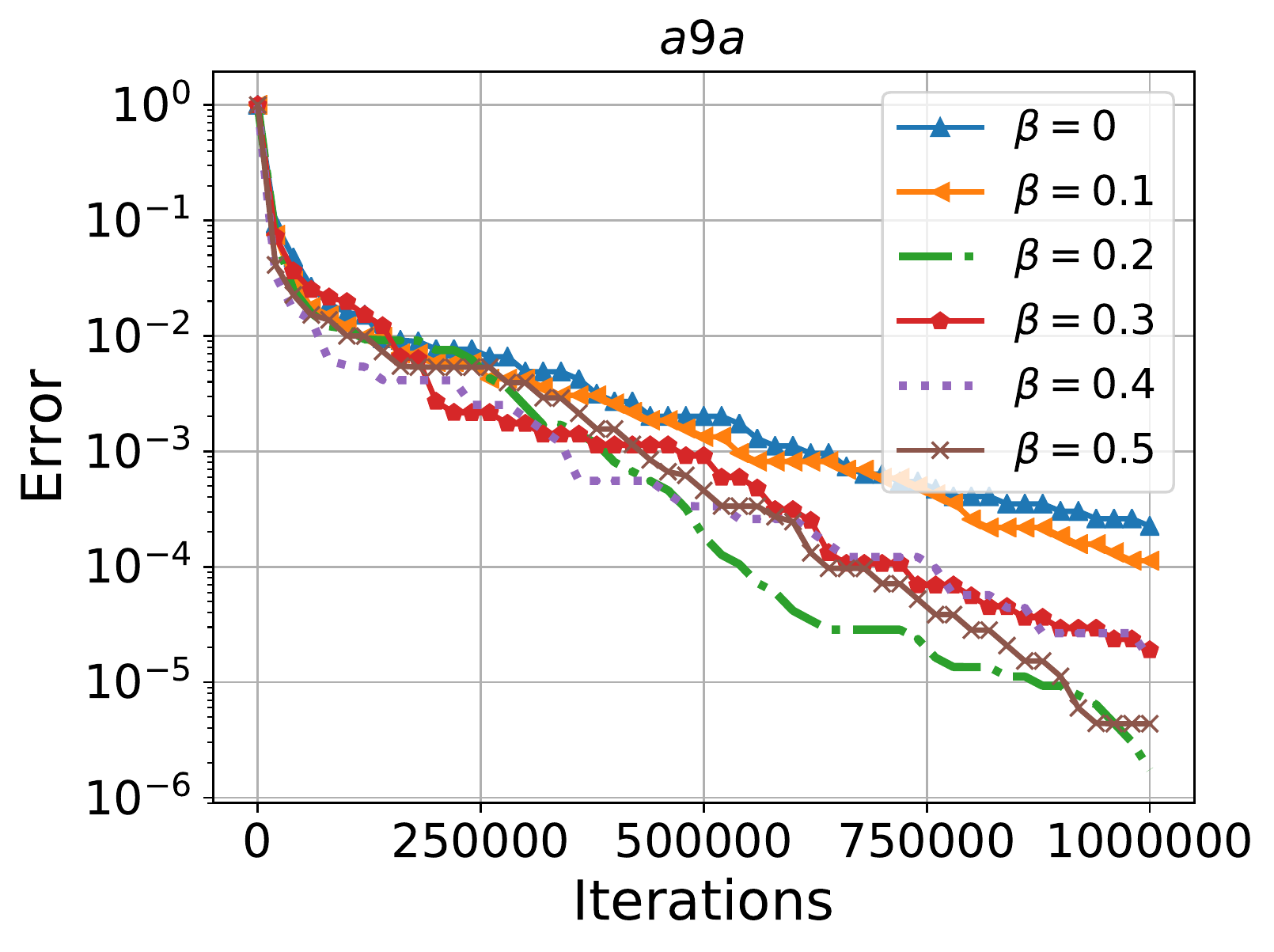}
\end{subfigure}%
\begin{subfigure}{.23\textwidth}
  \centering
  \includegraphics[width=1\linewidth]{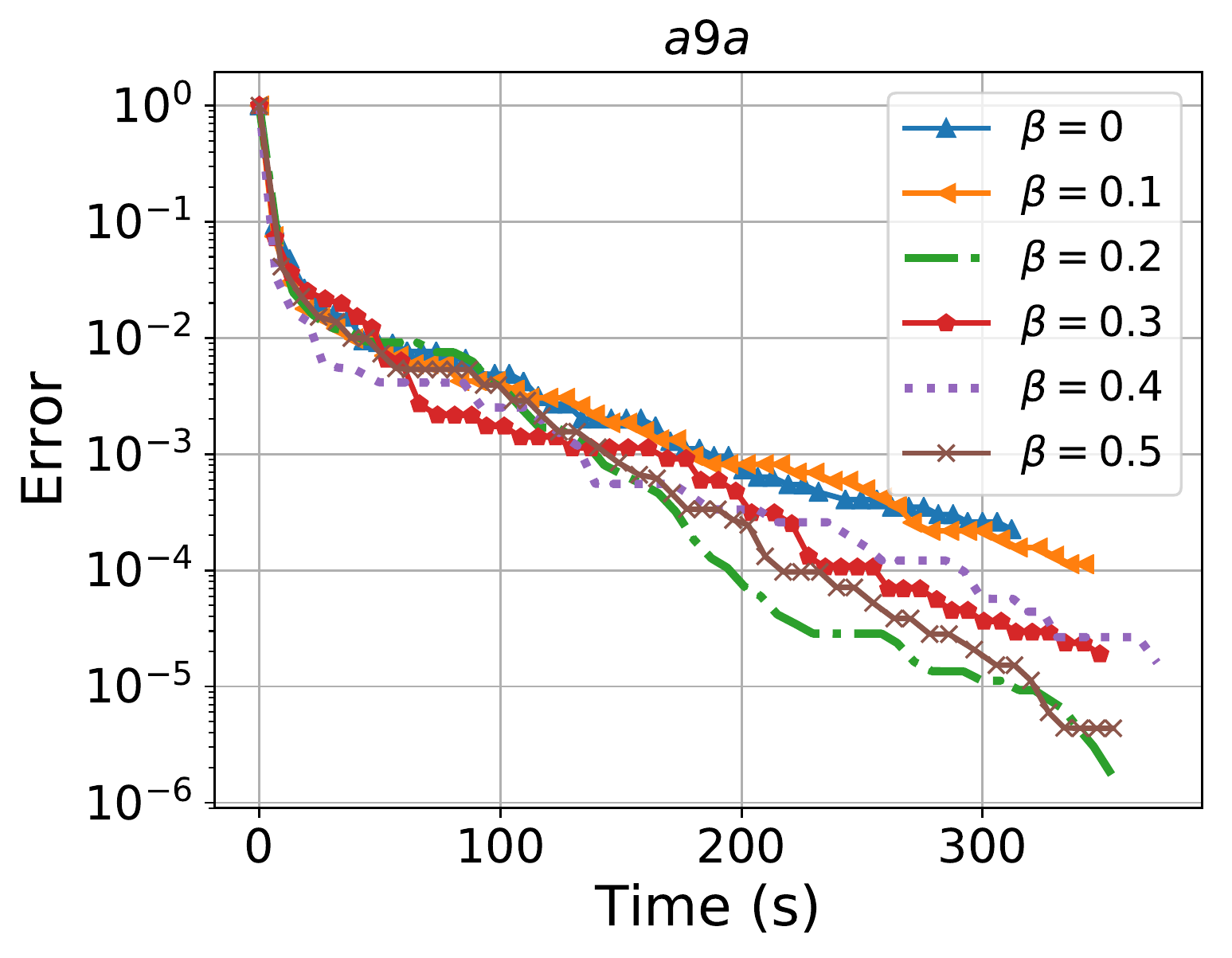}
\end{subfigure}
\begin{subfigure}{.23\textwidth}
  \centering
  \includegraphics[width=1\linewidth]{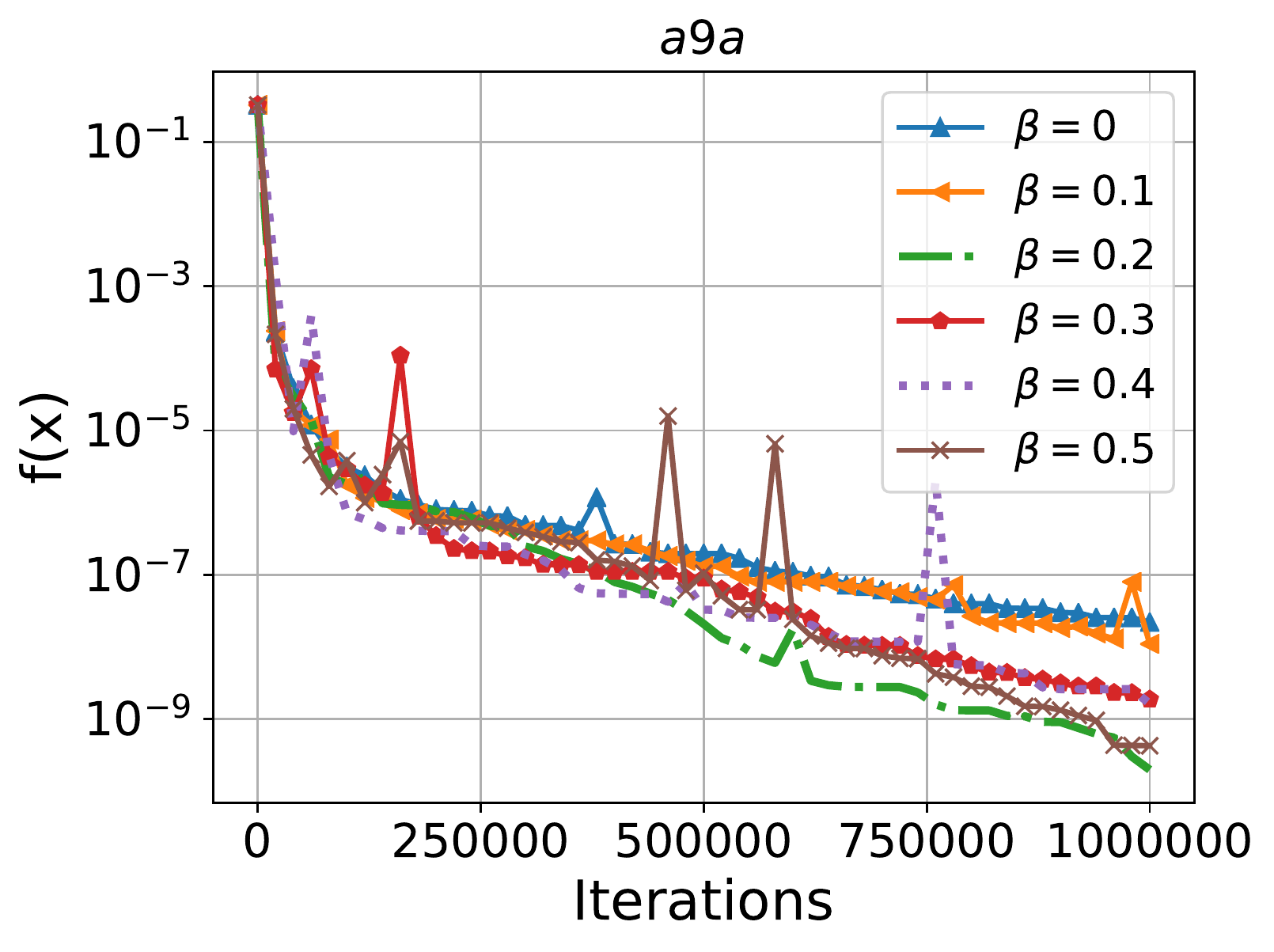}
\end{subfigure}
\begin{subfigure}{.23\textwidth}
  \centering
  \includegraphics[width=1\linewidth]{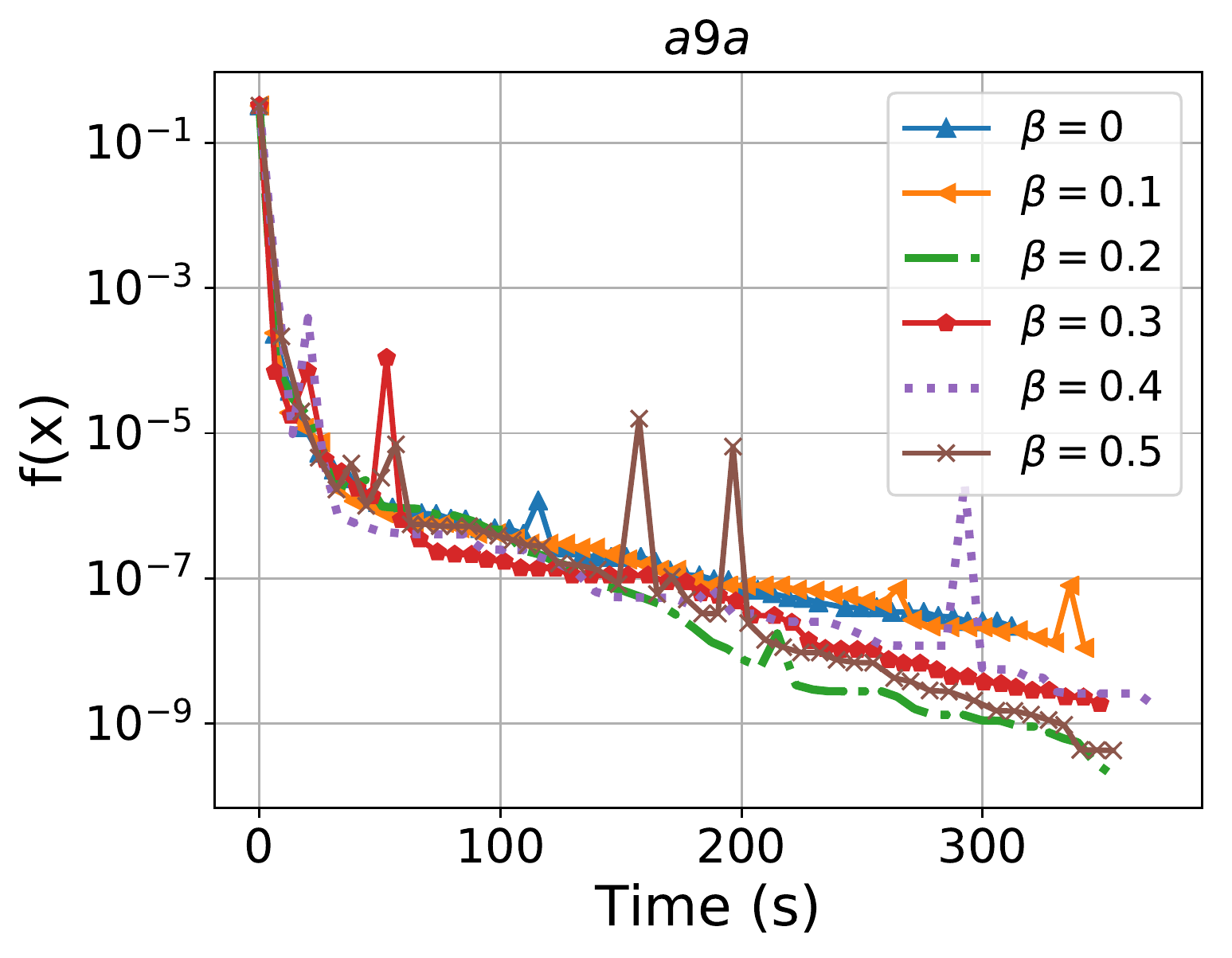}
\end{subfigure}\\
\begin{subfigure}{.23\textwidth}
  \centering
  \includegraphics[width=1\linewidth]{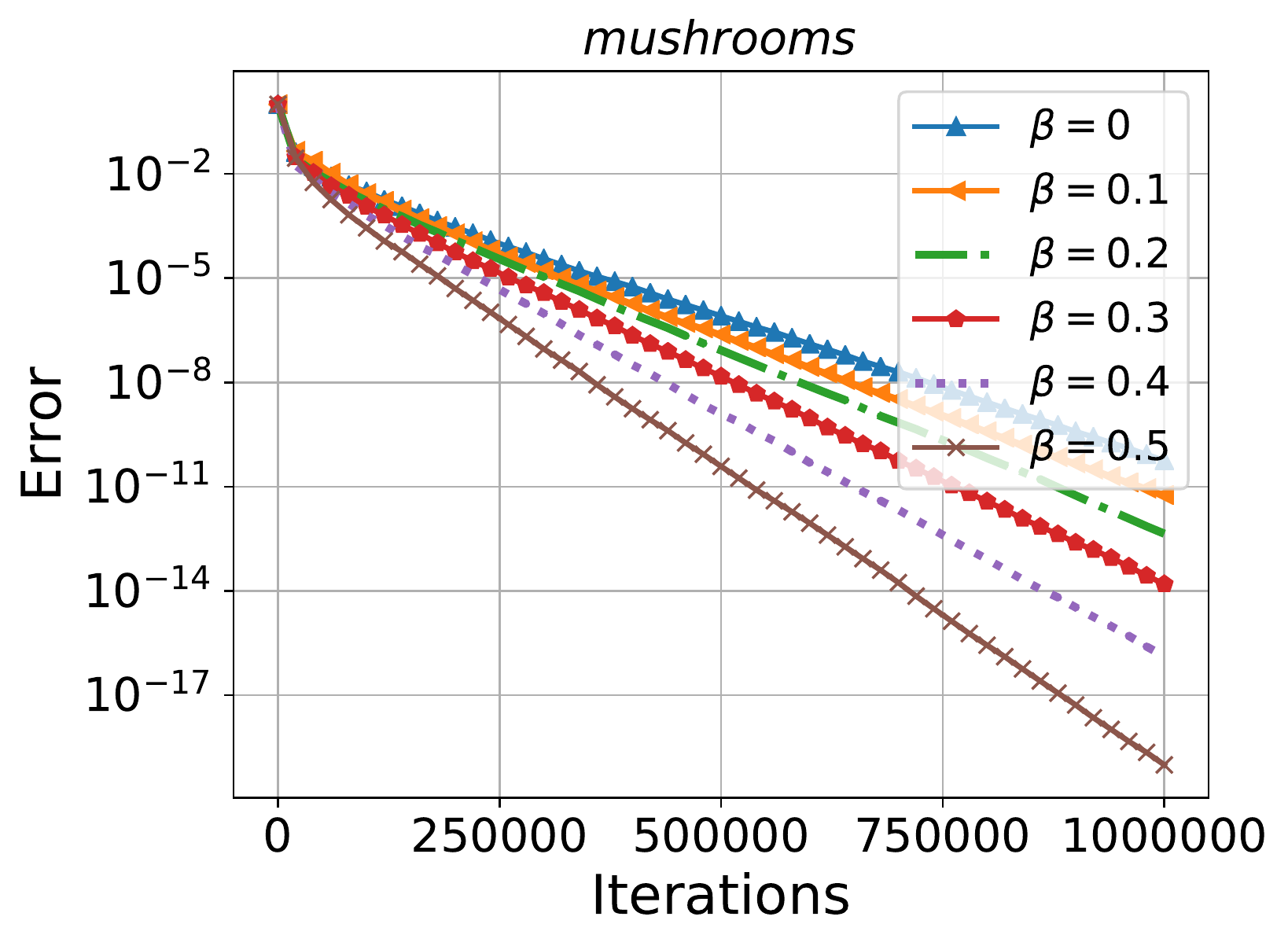}
\end{subfigure}%
\begin{subfigure}{.23\textwidth}
  \centering
  \includegraphics[width=1\linewidth]{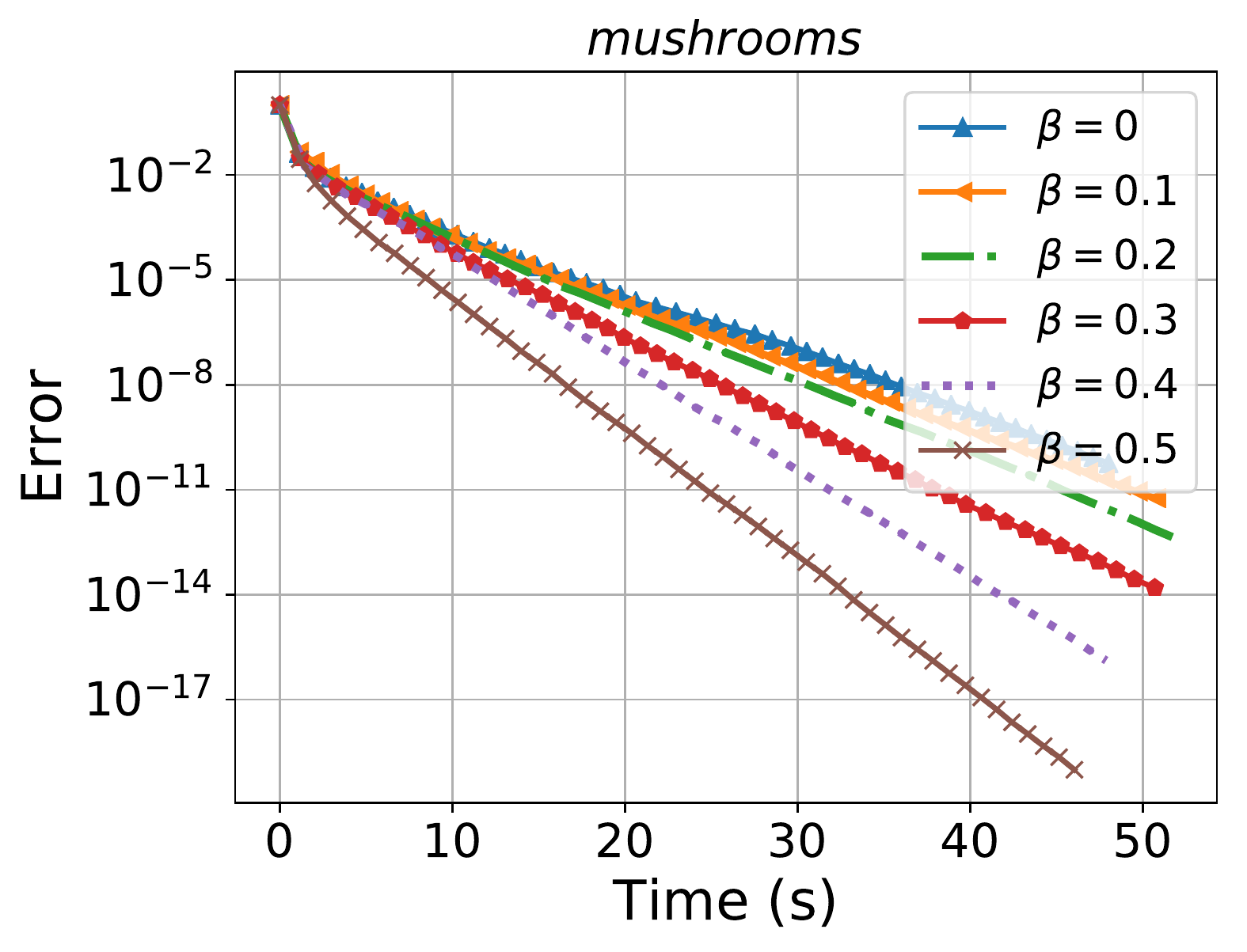}
\end{subfigure}
\begin{subfigure}{.23\textwidth}
  \centering
  \includegraphics[width=1\linewidth]{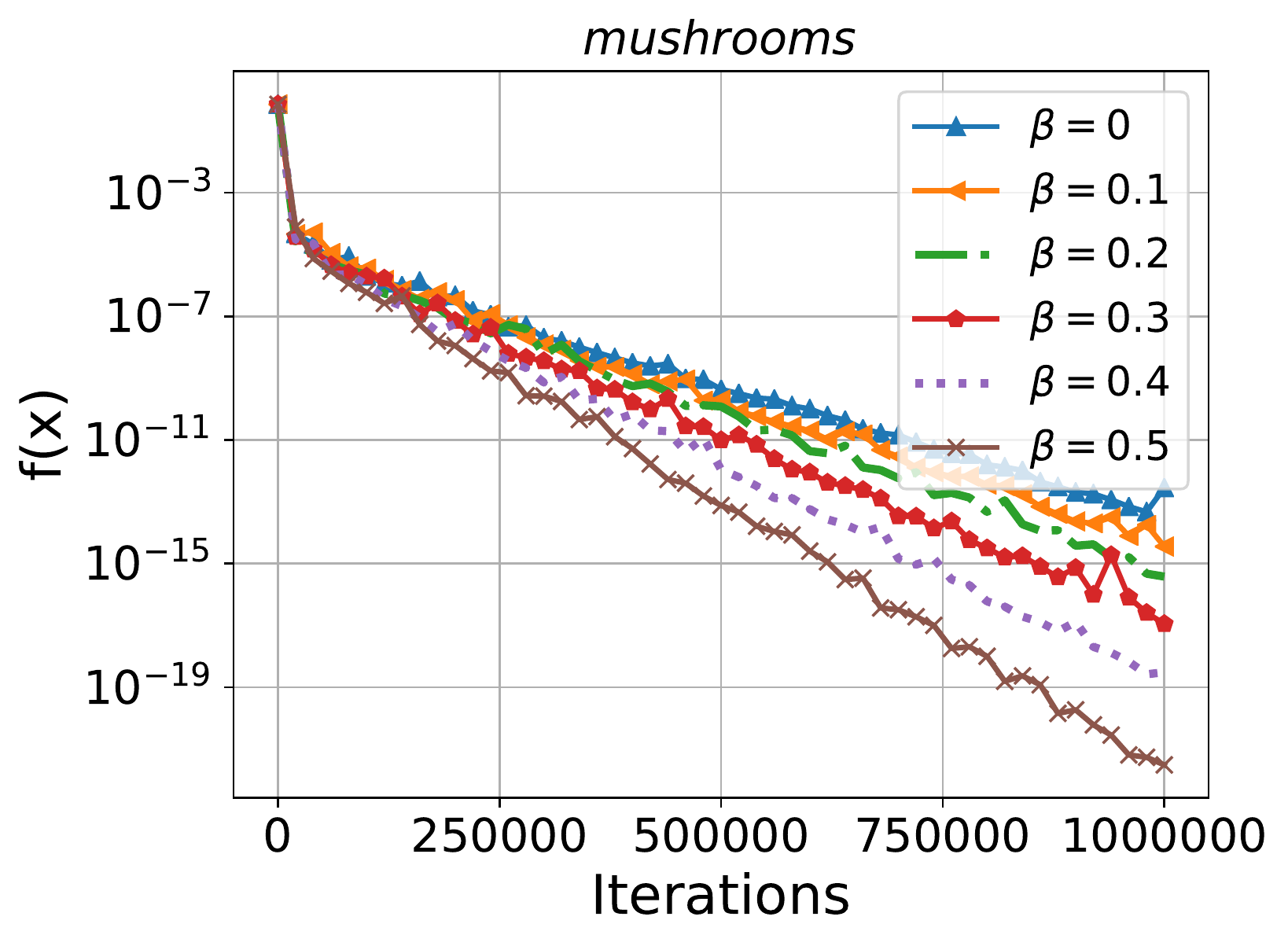}
\end{subfigure}
\begin{subfigure}{.23\textwidth}
  \centering
  \includegraphics[width=1\linewidth]{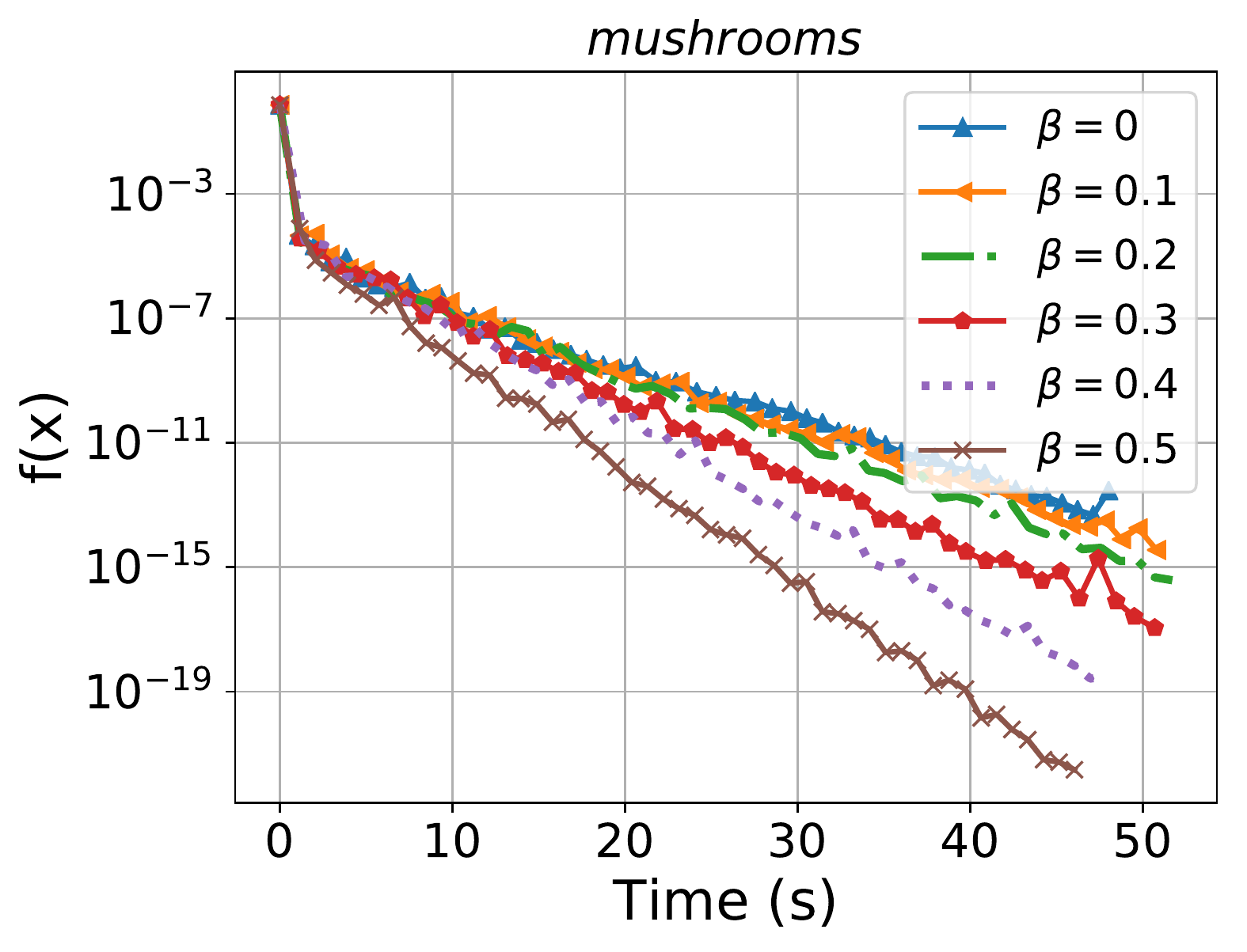}
\end{subfigure}\\
\begin{subfigure}{.23\textwidth}
  \centering
  \includegraphics[width=1\linewidth]{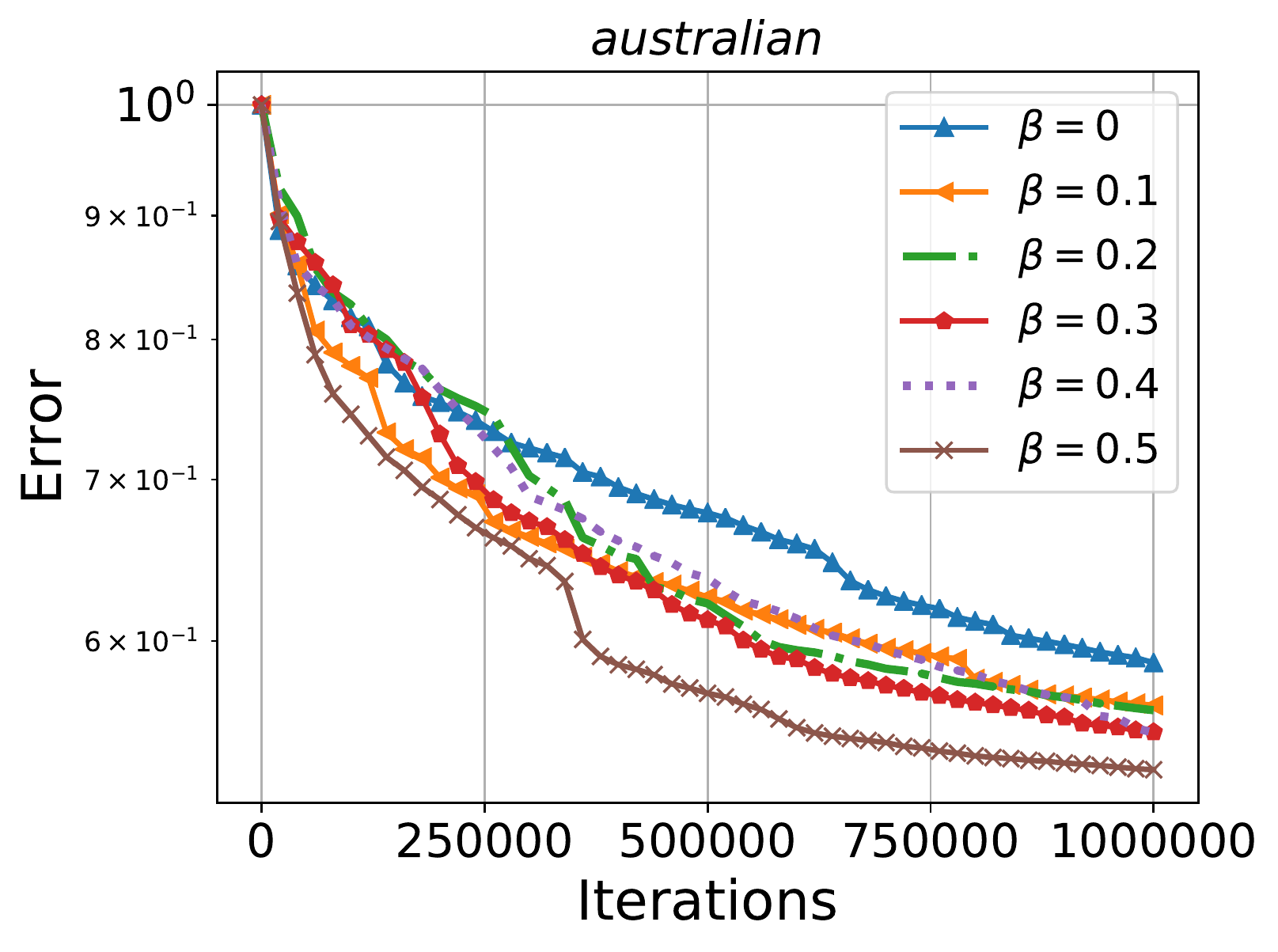}
\end{subfigure}%
\begin{subfigure}{.23\textwidth}
  \centering
  \includegraphics[width=1\linewidth]{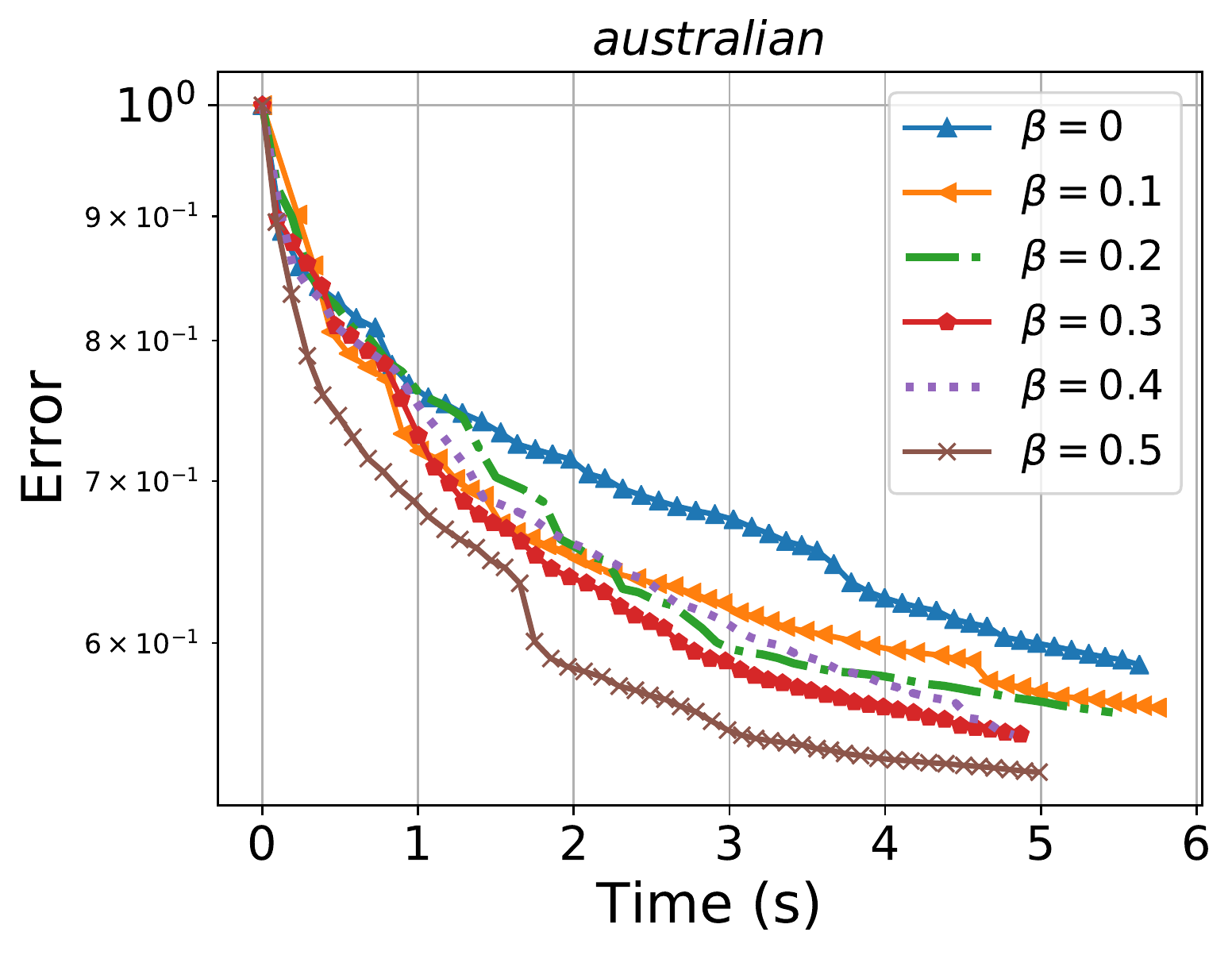}
\end{subfigure}
\begin{subfigure}{.23\textwidth}
  \centering
  \includegraphics[width=1\linewidth]{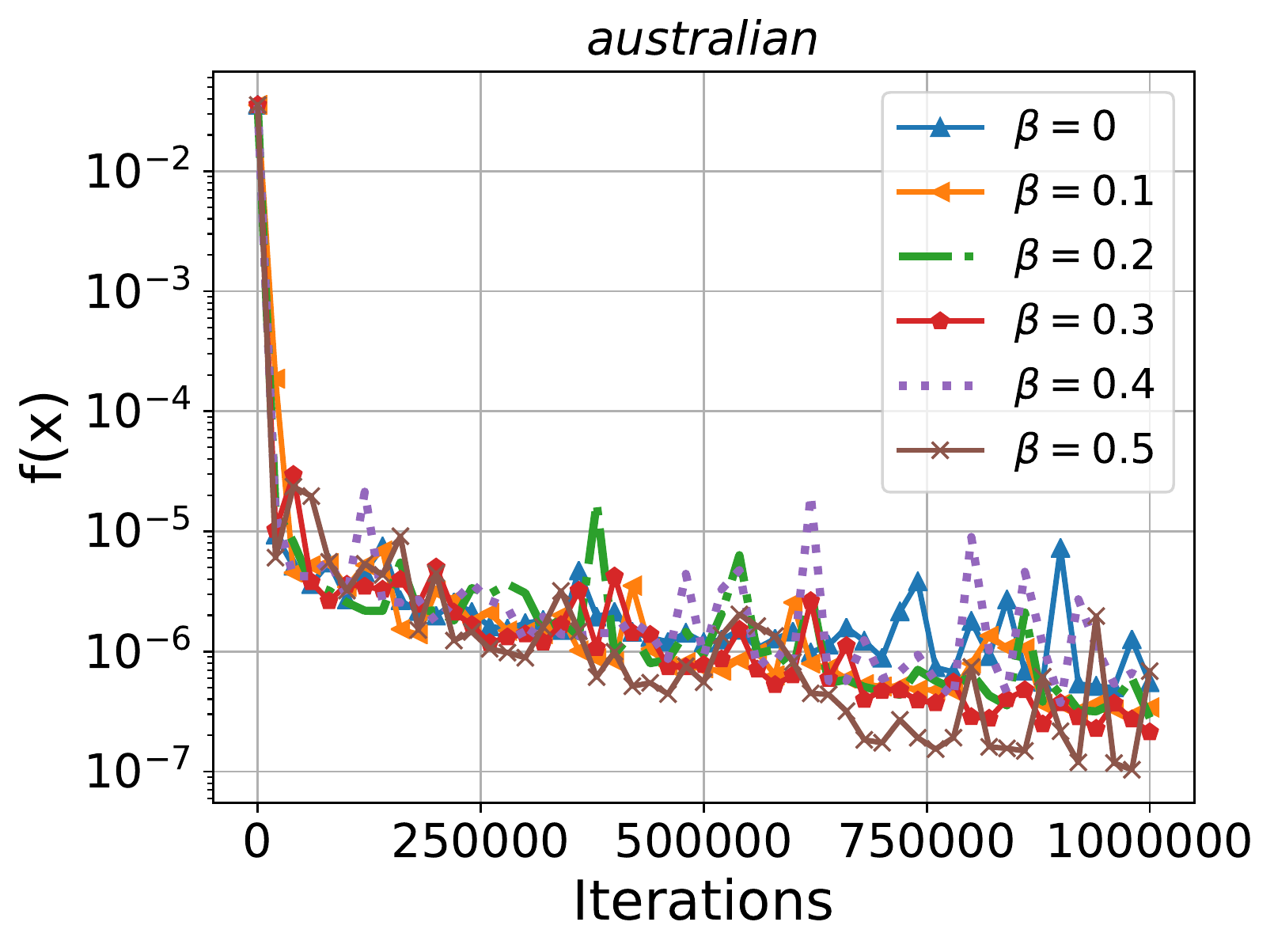}
\end{subfigure}
\begin{subfigure}{.23\textwidth}
  \centering
  \includegraphics[width=1\linewidth]{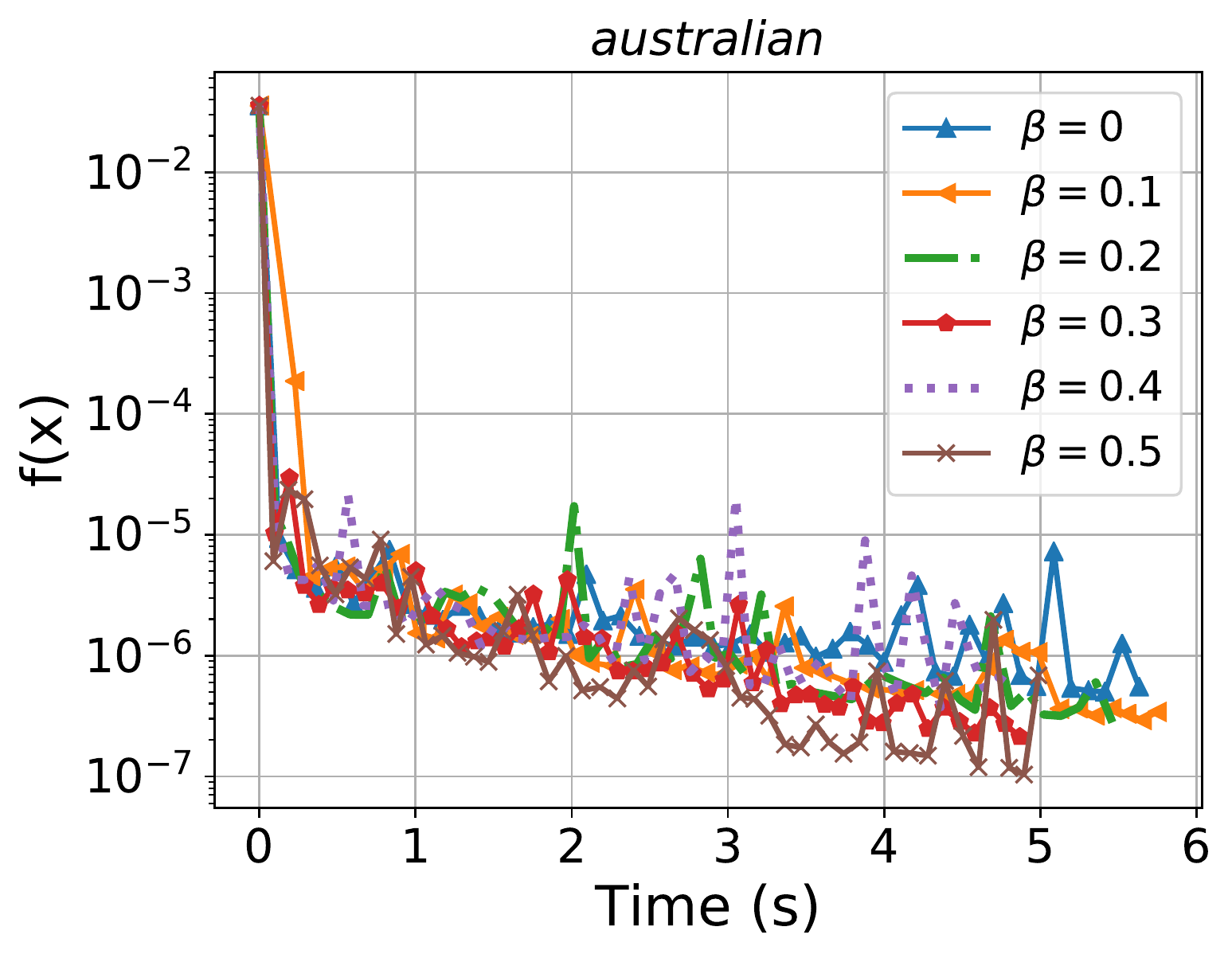}
\end{subfigure}\\
\begin{subfigure}{.23\textwidth}
  \centering
  \includegraphics[width=1\linewidth]{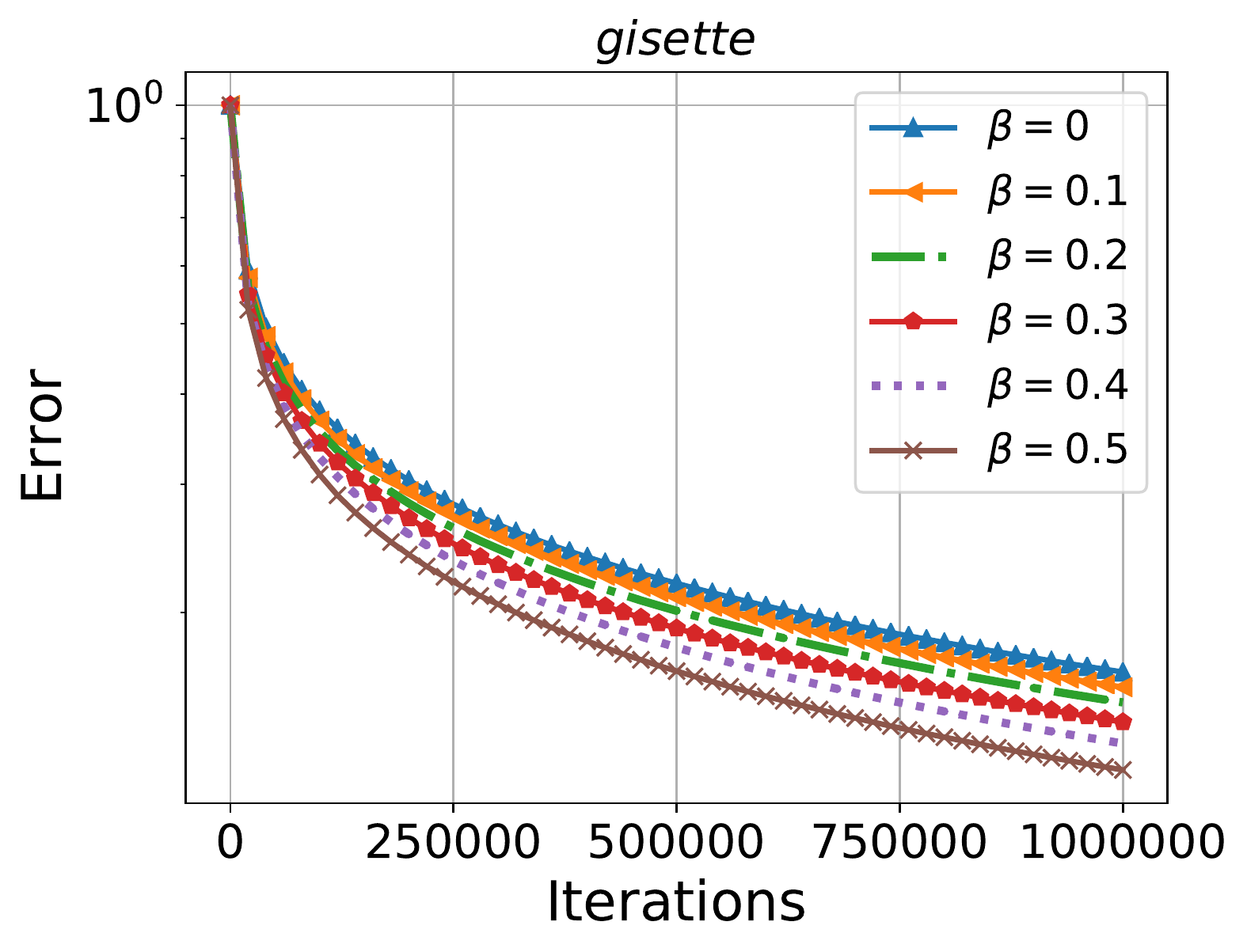}
\end{subfigure}%
\begin{subfigure}{.23\textwidth}
  \centering
  \includegraphics[width=1\linewidth]{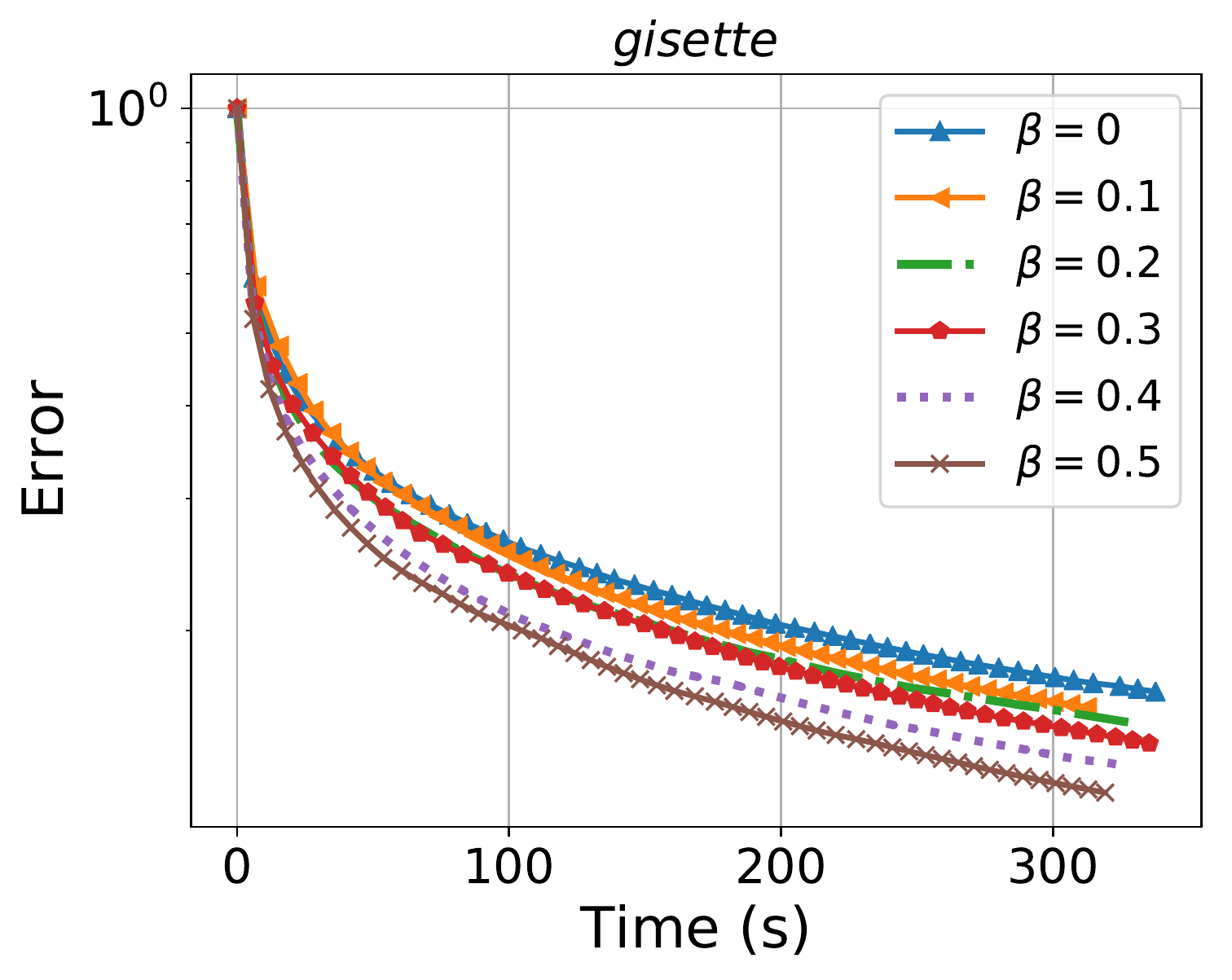}
\end{subfigure}
\begin{subfigure}{.23\textwidth}
  \centering
  \includegraphics[width=1\linewidth]{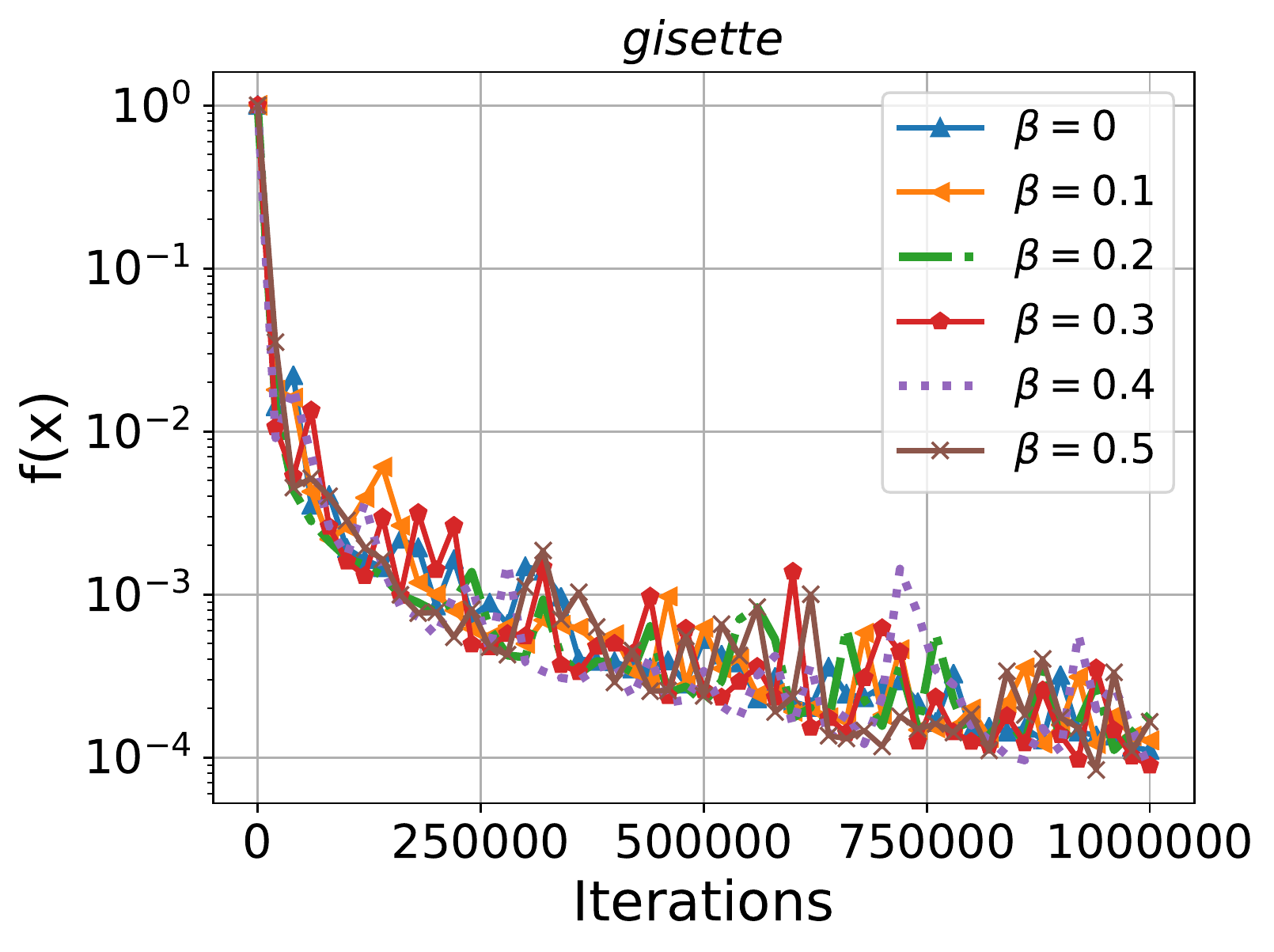}
\end{subfigure}
\begin{subfigure}{.23\textwidth}
  \centering
  \includegraphics[width=1\linewidth]{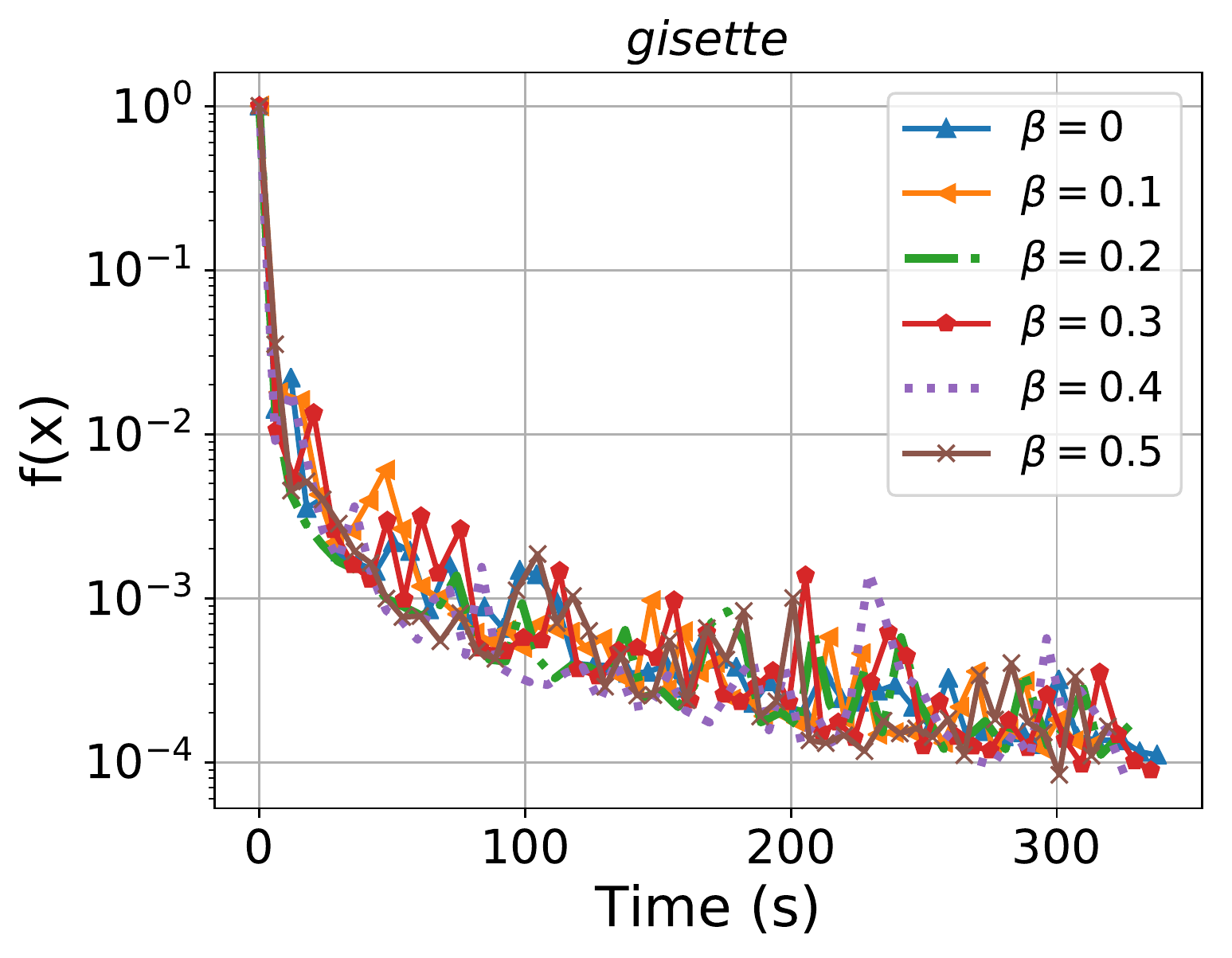}
\end{subfigure}\\
\begin{subfigure}{.23\textwidth}
  \centering
  \includegraphics[width=1\linewidth]{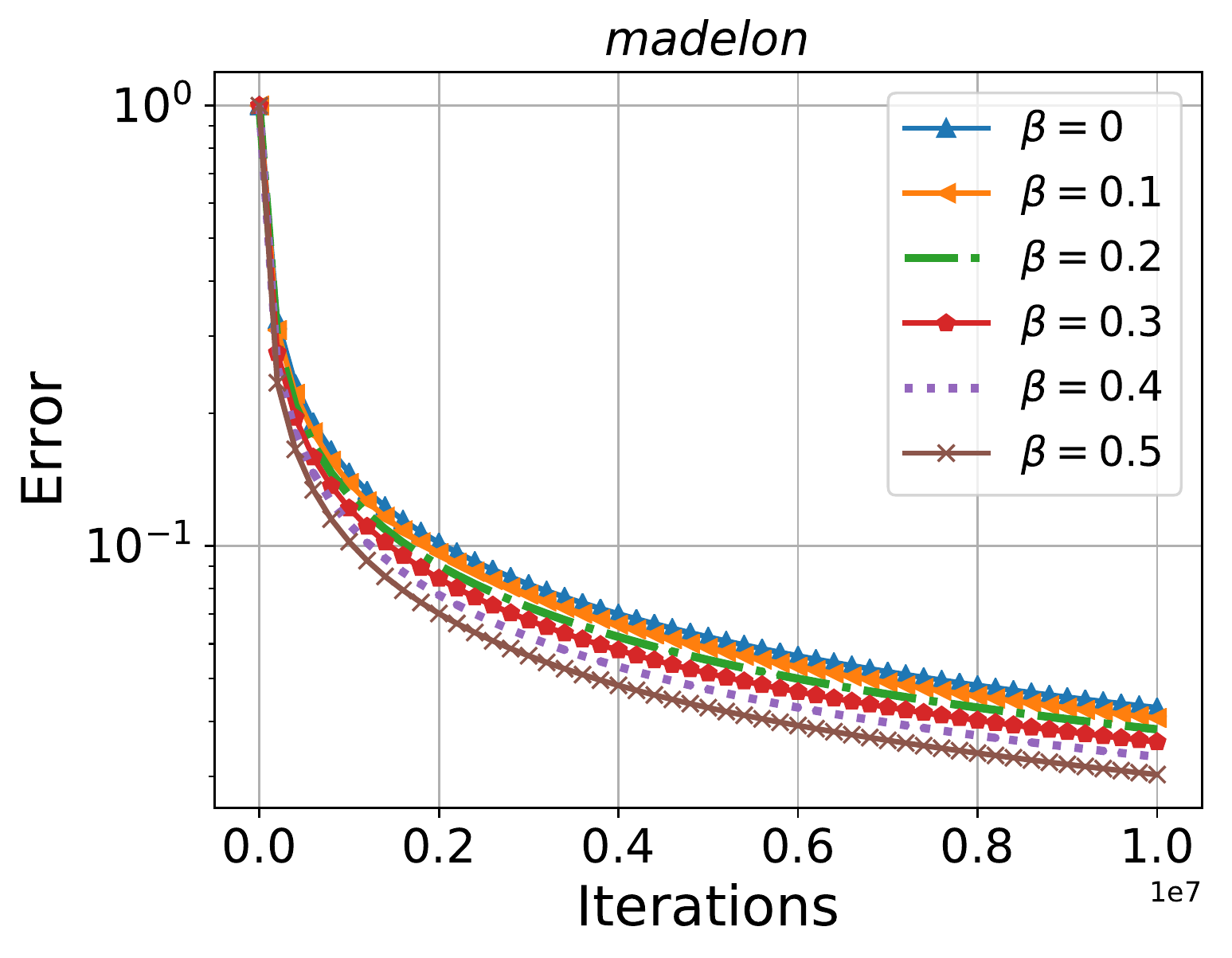}
\end{subfigure}%
\begin{subfigure}{.23\textwidth}
  \centering
  \includegraphics[width=1\linewidth]{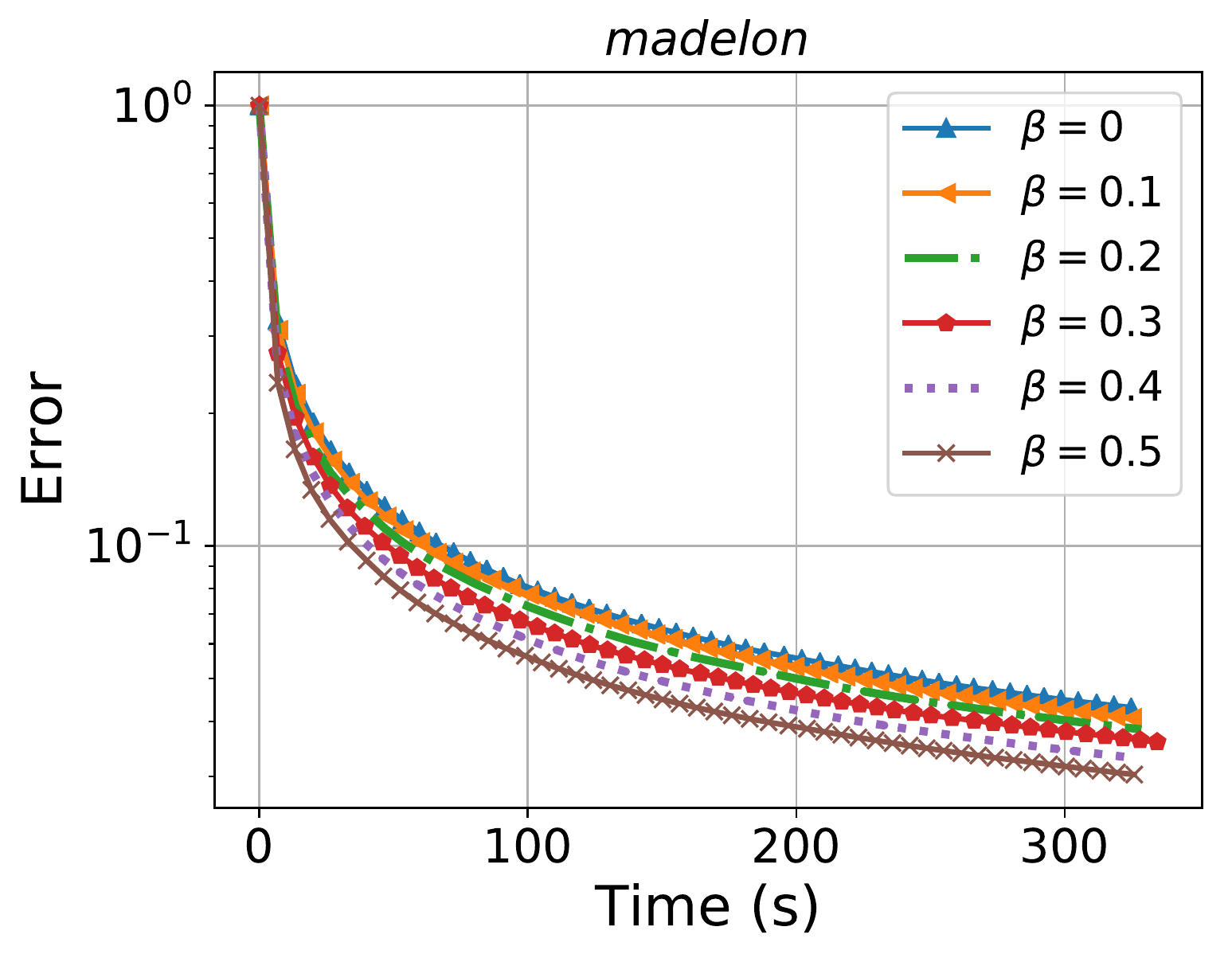}
\end{subfigure}
\begin{subfigure}{.23\textwidth}
  \centering
  \includegraphics[width=1\linewidth]{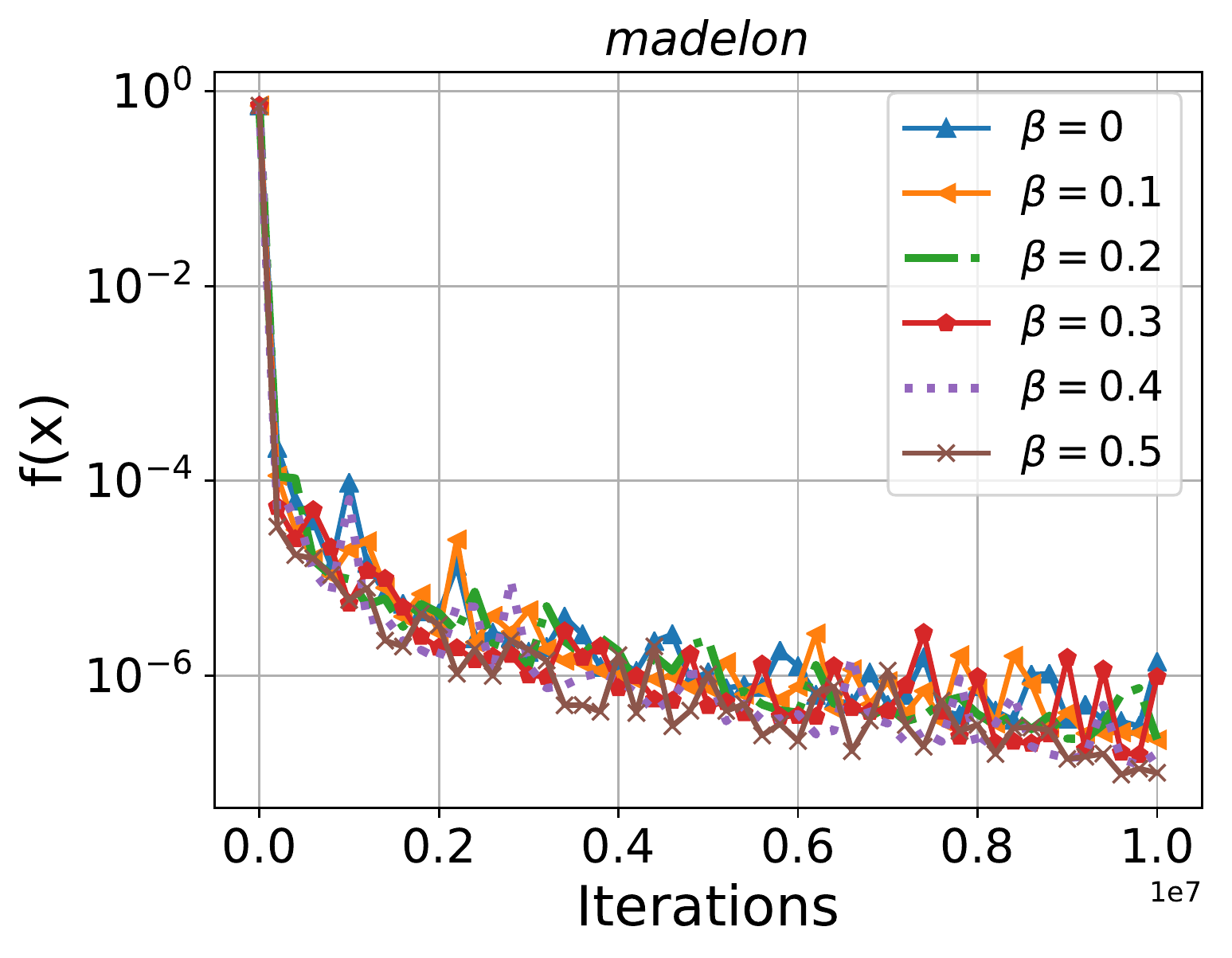}
\end{subfigure}
\begin{subfigure}{.23\textwidth}
  \centering
  \includegraphics[width=1\linewidth]{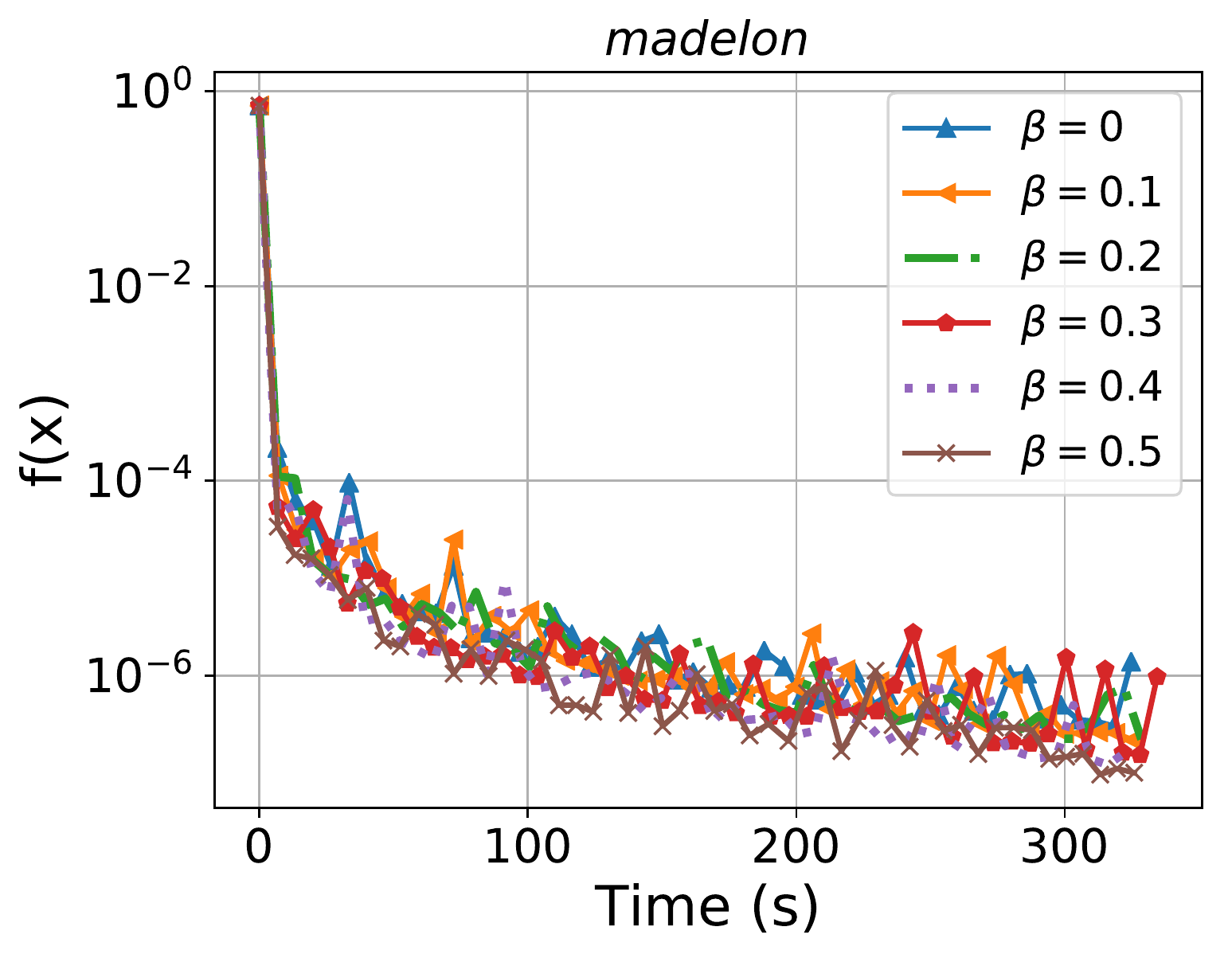}
\end{subfigure}\\
\begin{subfigure}{.23\textwidth}
  \centering
  \includegraphics[width=1\linewidth]{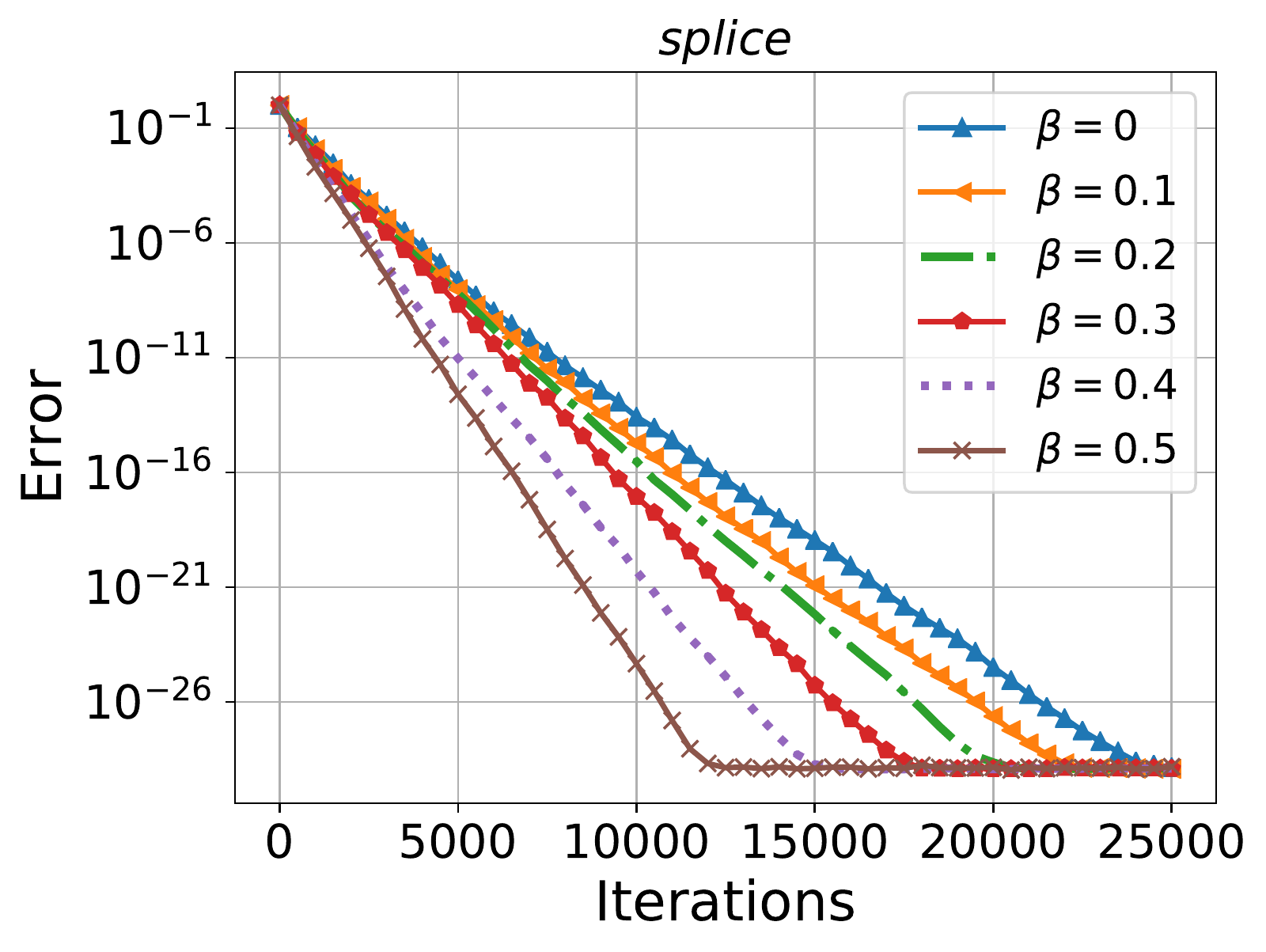}
\end{subfigure}%
\begin{subfigure}{.23\textwidth}
  \centering
  \includegraphics[width=1\linewidth]{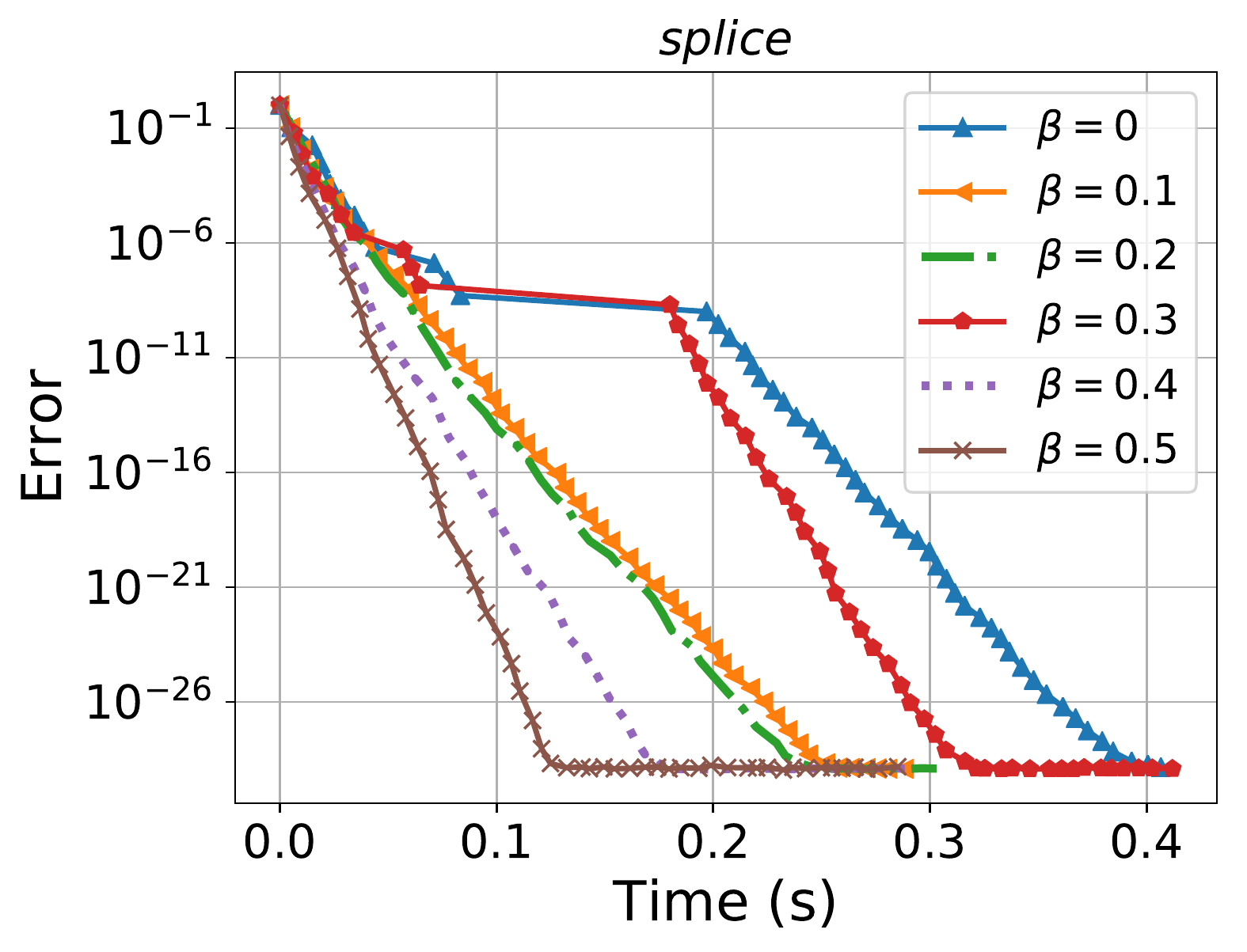}
\end{subfigure}
\begin{subfigure}{.23\textwidth}
  \centering
  \includegraphics[width=1\linewidth]{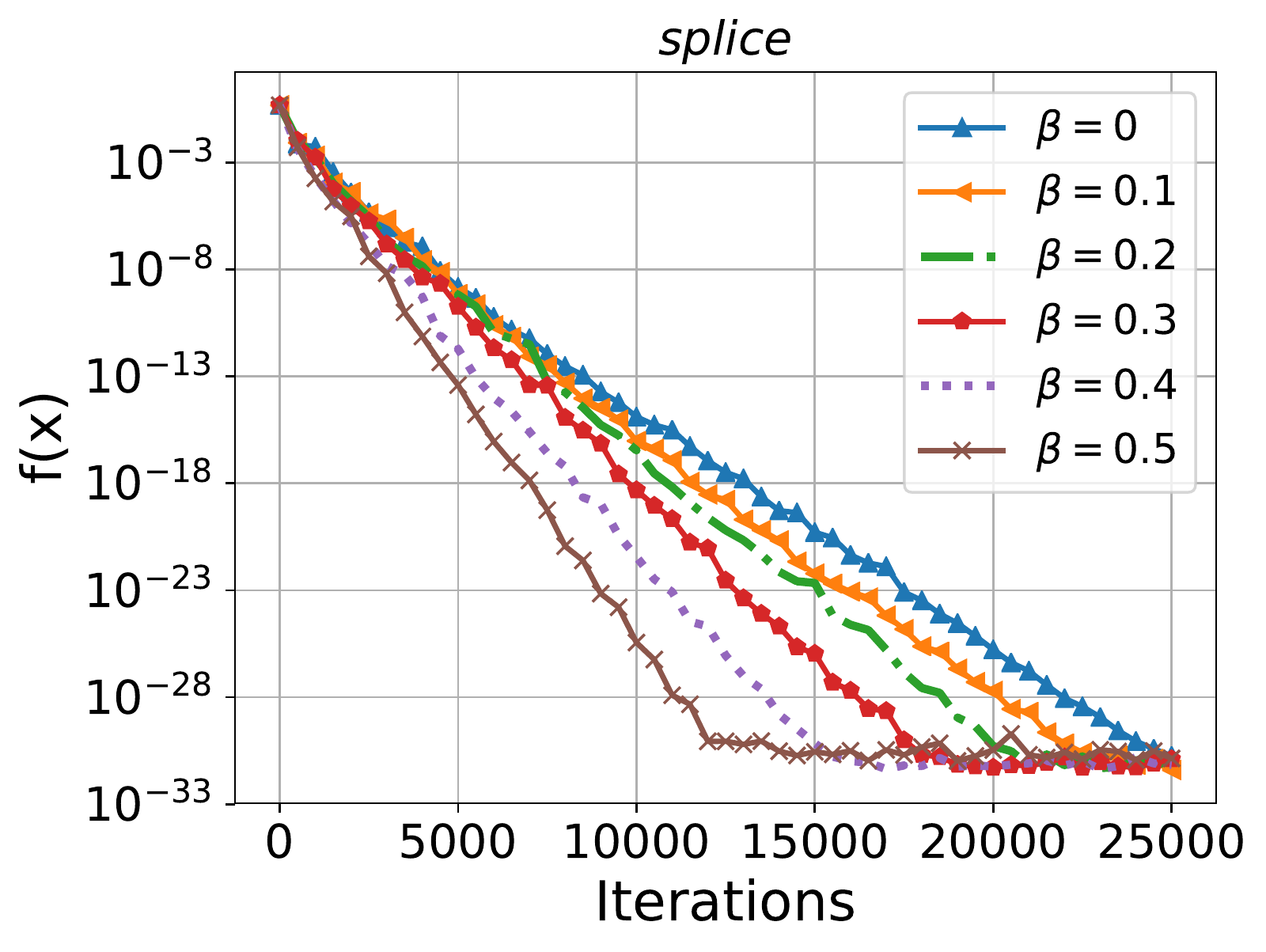}
\end{subfigure}
\begin{subfigure}{.23\textwidth}
  \centering
  \includegraphics[width=1\linewidth]{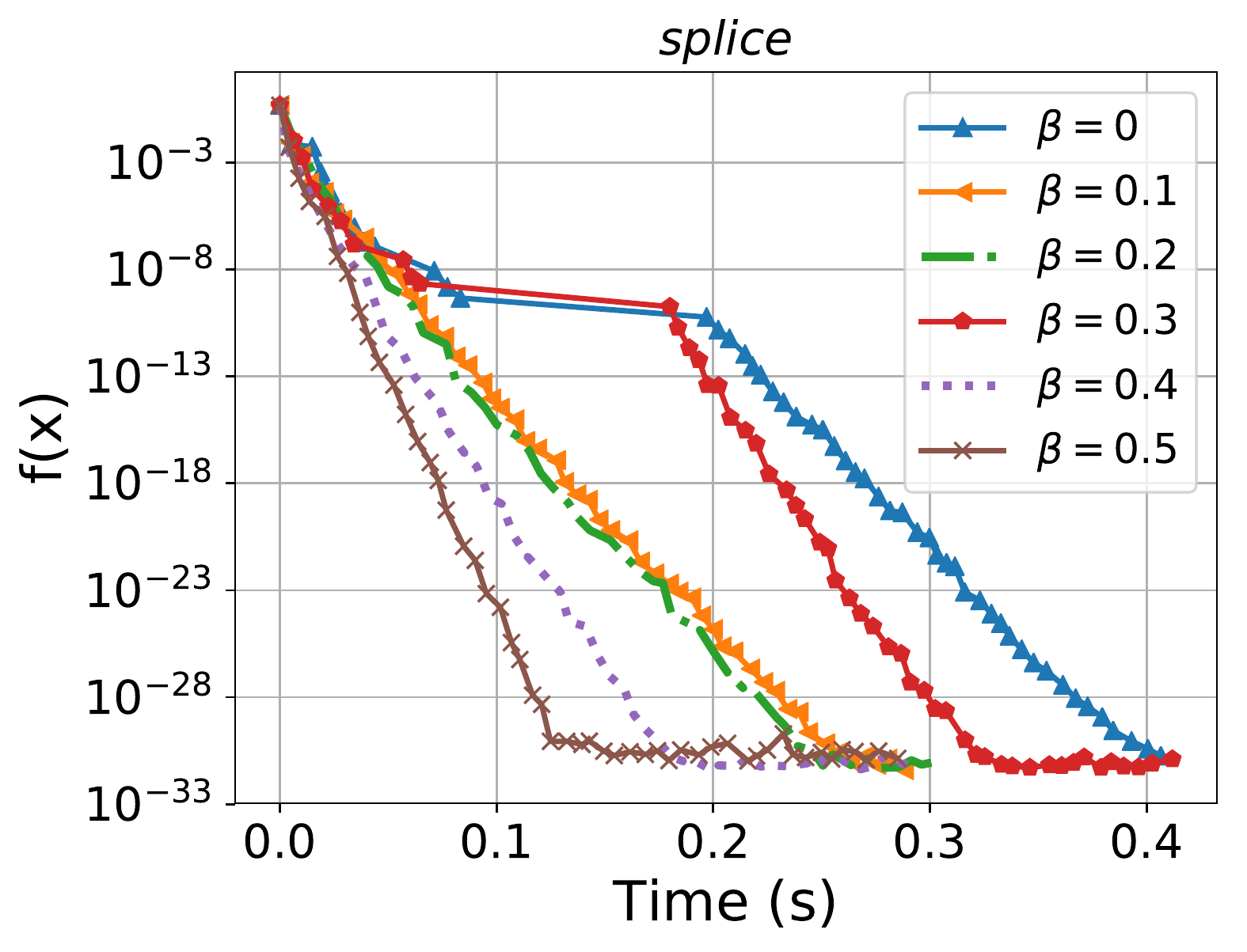}
\end{subfigure}\\
\begin{subfigure}{.23\textwidth}
  \centering
  \includegraphics[width=1\linewidth]{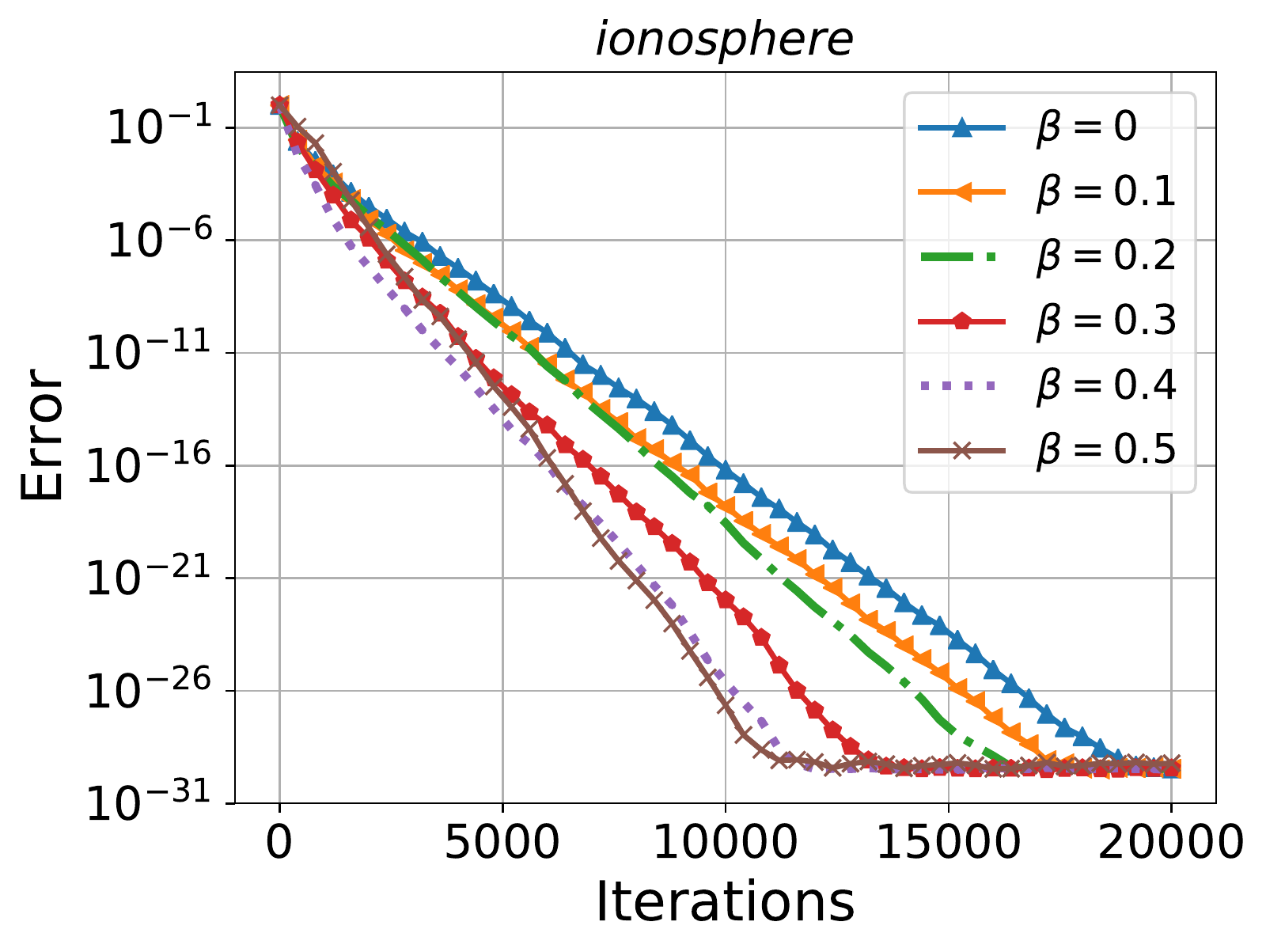}
\end{subfigure}%
\begin{subfigure}{.23\textwidth}
  \centering
  \includegraphics[width=1\linewidth]{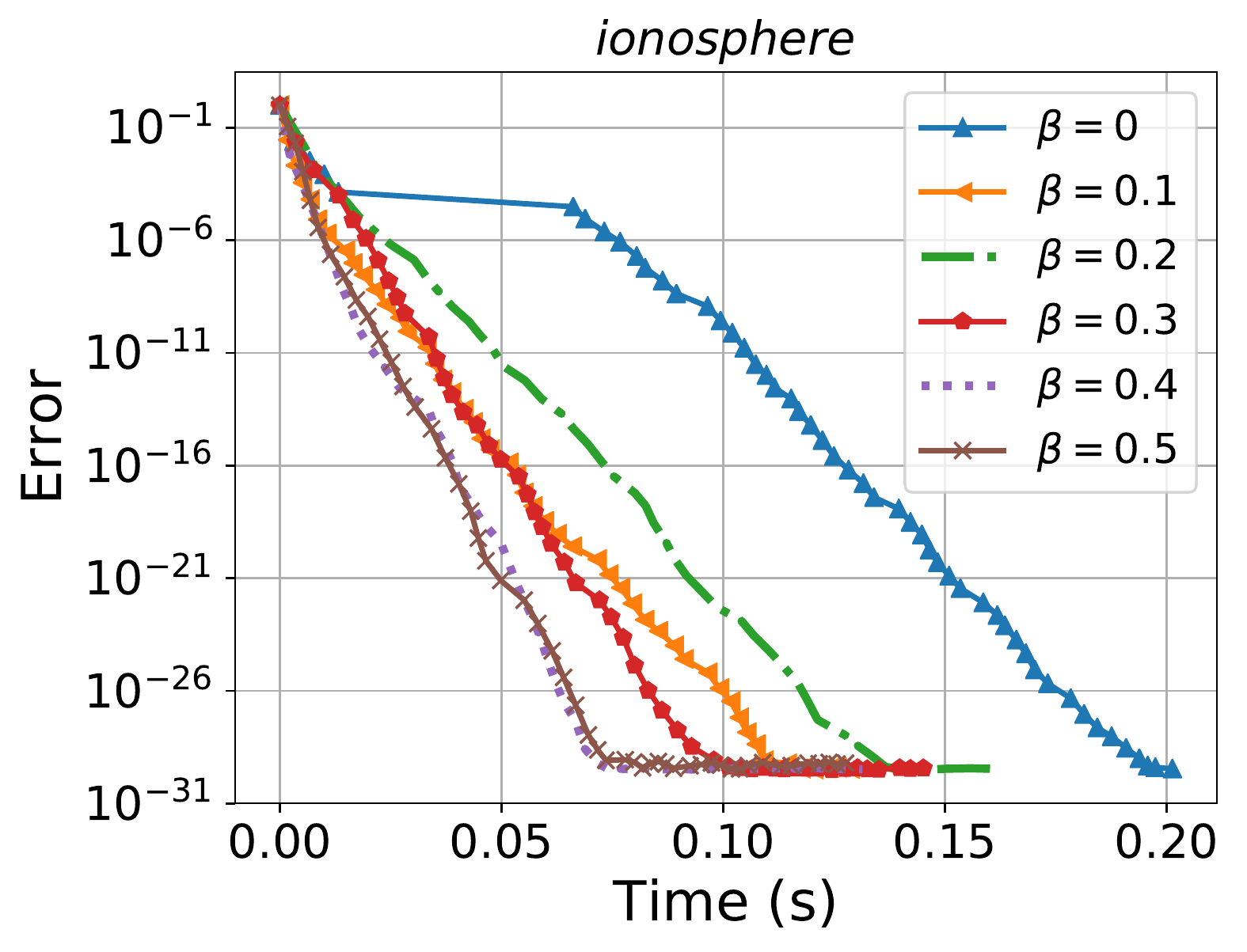}
\end{subfigure}
\begin{subfigure}{.23\textwidth}
  \centering
  \includegraphics[width=1\linewidth]{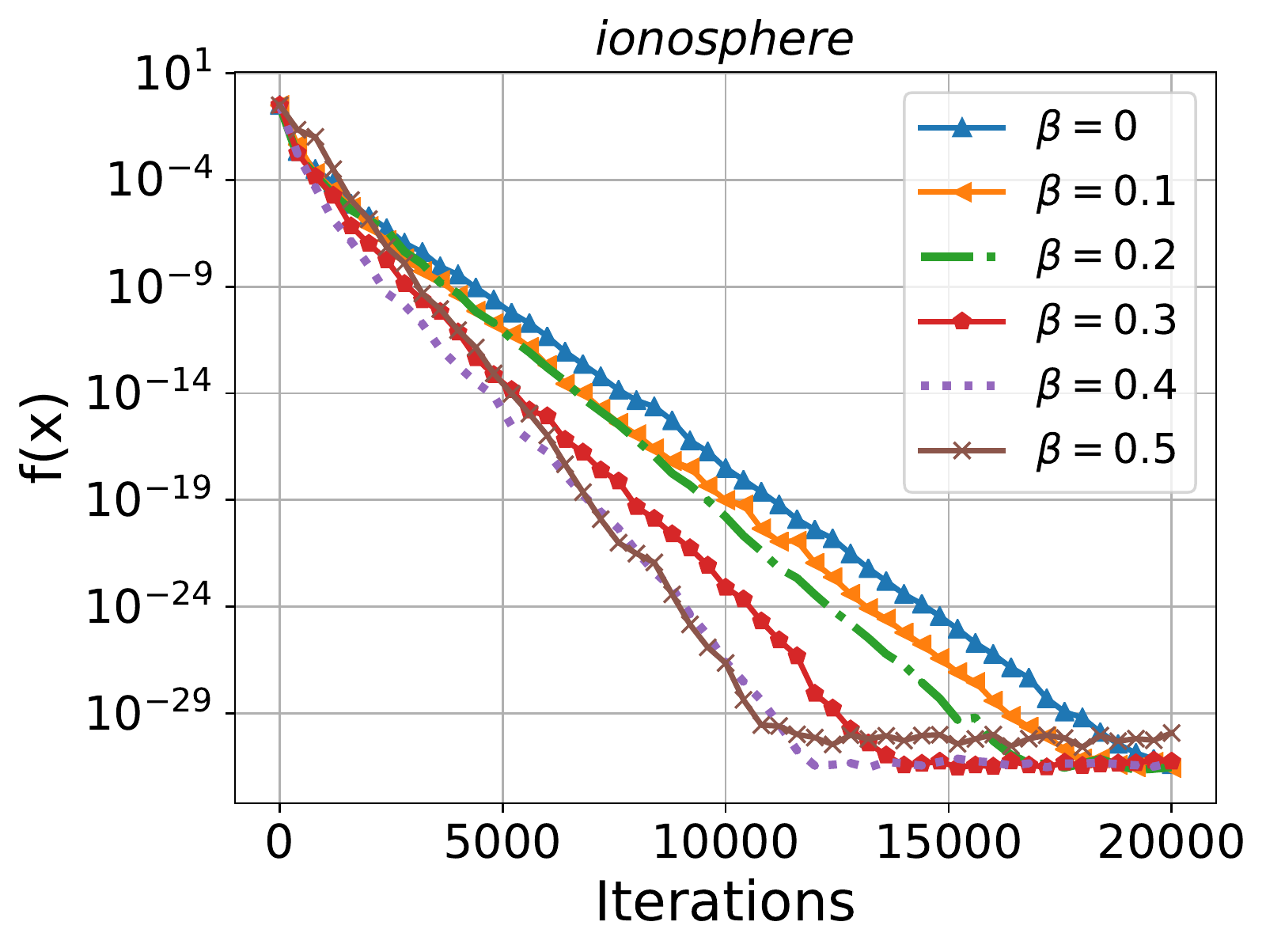}
\end{subfigure}
\begin{subfigure}{.23\textwidth}
  \centering
  \includegraphics[width=1\linewidth]{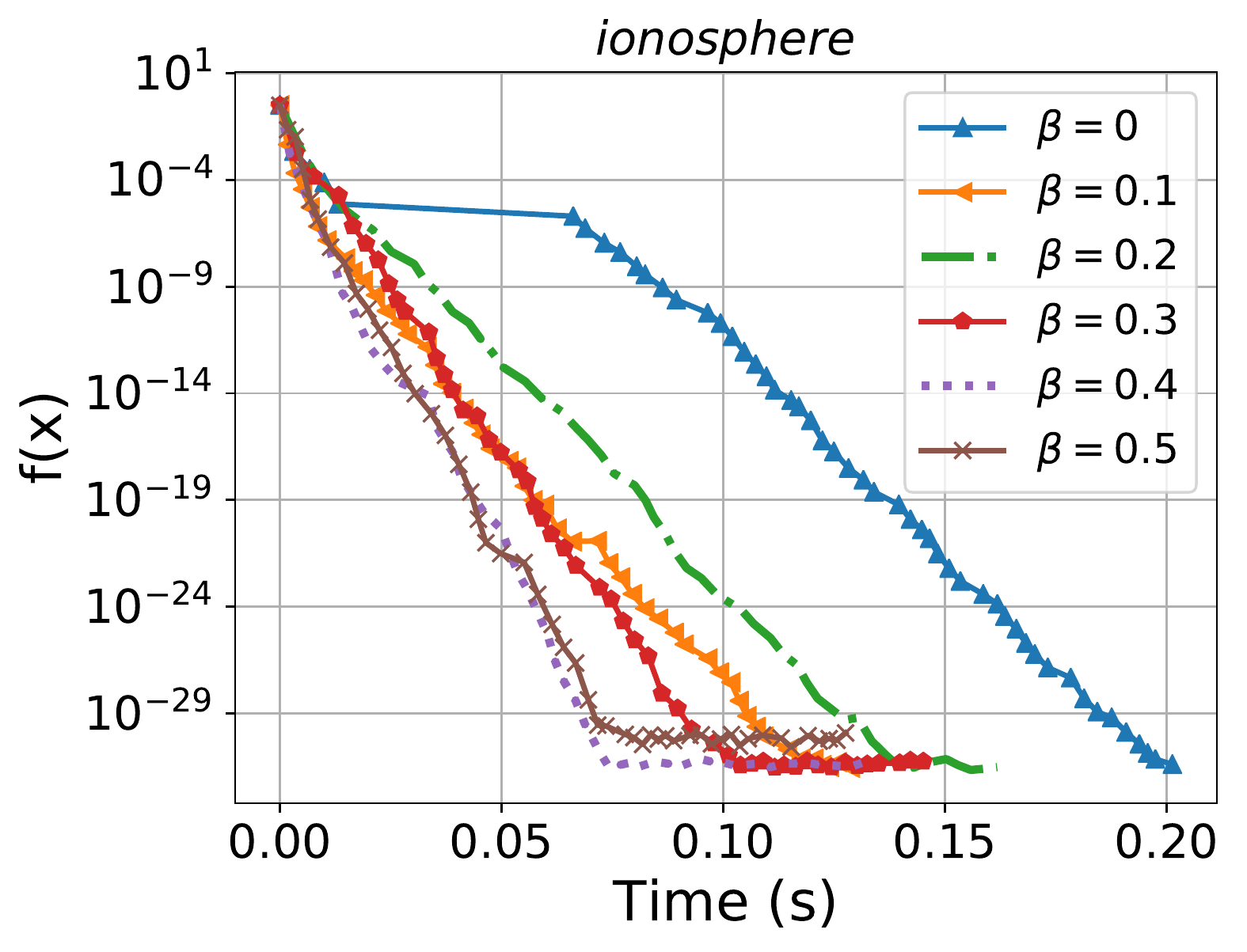}
\end{subfigure}\\
\caption{The performance of mRK for several momentum parameters $\beta$ on real data from LIBSVM \cite{chang2011libsvm}. a9a: $(m,n)=(32561,123)$, mushrooms: $(m,n)=(8124,112)$,  australian: $(m,n)=(690,14)$, gisette: $(m,n)=(6000,5000)$,  madelon: $(m,n)=(2000,500)$, splice: $(m,n)=(1000,60)$, ionosphere: $(m,n)=(351,34)$. The graphs in the first (second) column plot iterations (time) against residual error while those in the third (forth) column plot iterations (time) against function values. The ``Error" on the vertical axis represents the relative error $\|x^k-x^*\|^2_\bB / \|x^0-x^*\|^2_\bB \overset{\bB=\bI, x^0=0}{=}\|x^k-x^*\|^2 / \|x^*\|^2_\bB$ and the function values $f(x^k)$ refer to function~\eqref{functionRK}.}
\label{RealDataplots}
\end{figure}

\subsection{Comparison of momentum \& stochastic momentum}
\label{DSHB comp}
In Theorem \ref{thm:DSHBspeedup}, the total complexities (number of operations needed to achieve a given accuracy) of mSGD and smSGD have been compared and it has been shown that for small momentum parameter $\beta$,
\[C_{\beta}=\frac{C_{\text{mSGD}}(\beta)}{C_{\text{smSGD}}(\beta n)}  \approx  1 + \frac{n}{g},\]  where $C_{\text{mSGD}}$ and $C_{\text{smSGD}}$ represent the total costs of the two methods. The goal of this experiment is to show that this relationship holds also in practice. 

For this experiment we assume that the non-zeros of matrix $\bA$ are not concentrated in certain rows but instead that each row has the same number of non-zero coordinates. We denote by $g$ the number the non-zero elements per row. Having this assumption it can be shown that for the RK method the cost of one projection is equal to $4g$ operations while 
the cost per iteration of the mRK and of the smRK are $4g+3n$ and $4g+1$ respectively. For more details about the cost per iteration of the general mSGD and smSGD check Table~\ref{tableSHBandDSHB}.

As a first step a Gaussian matrix $\bA \in \R^{m \times n}$ is generated.
Then using this matrix several consistent linear systems are obtained as follows. Several values for $g \in [1,n]$ are chosen and for each one of these a matrix $\bA_g \in \R^{m \times n}$ with the same elements as $\bA$ but with $n-g$ zero coordinates per row is produced. 
For every matrix $\bA_g$, a Gaussian vector $z_g \in \R^n$ is drawn and to ensure consistency of the linear system, the right hand side is set to $b_g=\bA_g z$.

We run both mSGD and smSGD with small momentum parameter $\beta=0.0001$ for solving the linear systems $\bA_g x=b_g$ for all selected values of $g \in [1,n]$. The starting point for each run is taken to be $x^0=0 \in \R^n$.  The methods run until $\epsilon=\|x^k-x^*\|<0.001$, where $x^*=\Pi_{\cL_g}(x^0)$\footnote{\blue{To pre-compute the solution $x^*$ for each linear system $\bA_g x=b_g$ we use the closed form expression of the projection \eqref{Projection_IntroThesis}.} } and $\cL_g$ is the solution set of the linear system $\bA_g x=b_g$. In each run the number of operations needed to achieve the accuracy $\epsilon$ have been counted. 
For each linear system the average after $10$ trials of the value $\frac{C_{\text{mSGD}}(\beta)}{C_{\text{smSGD}}(\beta n)}$ is computed. 

\begin{table}[H]
\begin{center}
{\footnotesize
\begin{tabular}{ | p{4cm} | p{4cm} | p{3cm} | p{4cm} | }
 \hline
 Algorithm & Cost per iteration & Cost per Iteration (RK, mRK, smRK) \\
 \hline
 Basic Method ($\beta=0$) & $\cO(g)$ & $4g$  \\
  \hline
 mSGD & $\cO(g) +\cO(n) =\cO(n)$ &  $4g+3n$  \\
  \hline
smSGD & $\cO(g) +\cO(1) =\cO(g)$ & $4g+1$ \\
 \hline
\end{tabular}
}
\end{center}
\caption{Cost per iteration of the basic, mSGD and smSGD in the general setting and in the special cases of RK,mRK and smRK.}
\label{tableSHBandDSHB}
\end{table}

In Figure \ref{comparisonFigure} the actual ratio $\frac{C_{\text{mSGD}}(\beta)}{C_{\text{smSGD}}(\beta n)}$ and the theoretical approximation $1 + \frac{n}{g}$ are plot and it is shown that they have similar behavior. Thus the theoretical prediction of Theorem~\ref{thm:DSHBspeedup} is numerically confirmed.
In particular in the implementations we use the Gaussian matrices $\bA \in \R^{200 \times 100}$ and $\bA \in \R^{1000 \times 300}$.

\begin{figure}[!]
\centering
\begin{subfigure}{.4\textwidth}
  \centering
  \includegraphics[width=1\linewidth]{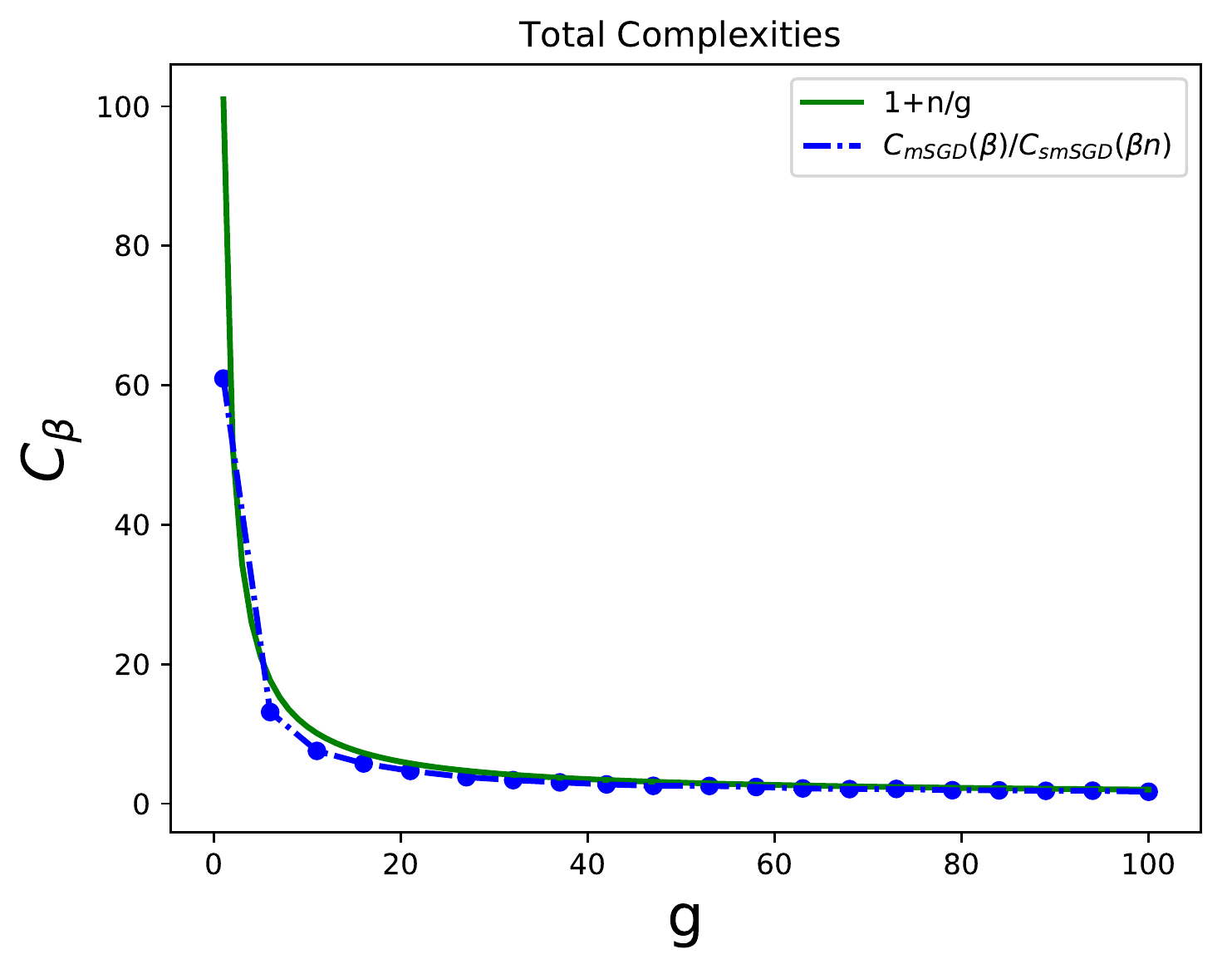}
  \caption{$\bA \in \R^{200 \times 100}$}
\end{subfigure}
\begin{subfigure}{.4\textwidth}
  \centering
  \includegraphics[width=1\linewidth]{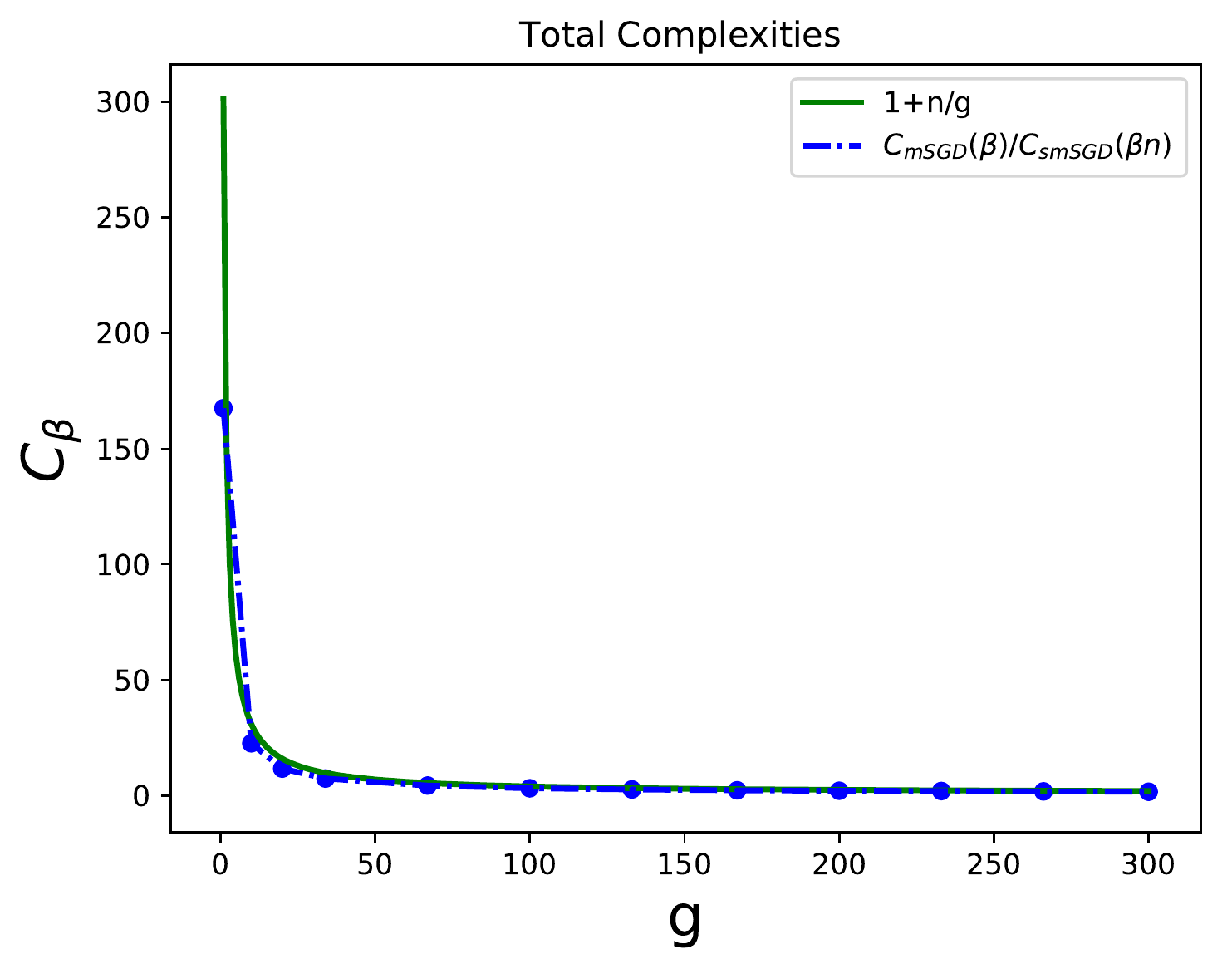}
  \caption{$\bA \in \R^{1000 \times 300}$}
\end{subfigure}
\caption{Comparison of the total complexities of mRK and smRK. The green continuous line denotes the theoretical relationship $1 + \frac{n}{g}$ that we predict in Theorem~\ref{thm:DSHBspeedup}. The blue dotted line shows the ratio of the total complexities $\frac{C_{\text{mSGD}}(\beta)}{C_{\text{smSGD}}(\beta n)} $ for several linear systems $\bA_g x=b_g$ where $g \in [1,n]$. The momentum parameter $\beta=0.0001$ is used for both methods.}
\label{comparisonFigure}
\end{figure}
\subsection{Faster method for average consensus}
\label{consensus}
\subsubsection{Background}
Average consensus (AC) is a fundamental problem in distributed computing and multi-agent systems \cite{dimakis2010gossip, boyd2006randomized}. Consider a connected undirected network $\cG=(\cV,\cE)$ with node set $\cV=\{1,2,\dots,n\}$ and edges $\cE$, ($|\cE|=m$), where each node $i \in \cV$ owns a private value $c_i \in \R$. The goal of the AC problem is \blue{for} each node of the network to compute the average of these private values, $\bar{c}\eqdef\frac{1}{n}\sum_i c_i$, via a protocol which allows communication between neighbours only. The problem comes up in many real world applications such as coordination of autonomous agents, estimation, rumour spreading in social networks, PageRank and distributed data fusion on ad-hoc networks and decentralized optimization. 

It was shown recently that several randomized methods for solving linear systems  can be interpreted as randomized gossip algorithms for solving the AC problem when applied to a special system encoding the underlying network \cite{gower2015stochastic, LoizouRichtarik}. As we have already explained both basic method \cite{ASDA} and basic method with momentum (this chapter) find the solution of the linear system that is closer to the starting point of the algorithms. That is, both methods converge linearly to $x^*=\Pi_{\cL, \mB}(x^0)$; the projection of the initial iterate onto the solution set of the linear system. \blue{As a result (check the Introduction)}, they can be interpreted as methods for solving the best approximation problem~\eqref{BestApproximation_IntroThesis}.
In the special case that
\begin{enumerate}
\item the linear system in the constraints of \eqref{BestApproximation_IntroThesis} is the homogeneous linear system ($\bA x=0$) with matrix $\bA\in \R^ {m \times n}$ being the incidence matrix of the undirected graph $\cG=(\cV,\cE)$, and 
\item the starting point of the method are the initial values of the nodes $x^0=c$,
\end{enumerate} 
it  is straightforward to see that the solution of the best approximation problem is a vector with all components equal to the consensus value $\bar{c}\eqdef\frac{1}{n}\sum_i c_i$. Under this setting, the famous randomized pairwise gossip algorithm (randomly pick an edge $e\in E$ and replace the private values of its two nodes to their average) that was first proposed and analyzed in \cite{boyd2006randomized}, is equivalent \blue{to} the RK method without relaxation ($\omega=1$) \cite{gower2015stochastic, LoizouRichtarik}.

\begin{rem}
In the gossip framework, the condition number of the linear system when RK is used has a simple structure and it depends on the characteristics of the network under study. More specifically, it depends on the number of the edges $m$ and \blue{on the Laplacian} matrix of the network\footnote{Matrix $\bA$ of the linear system is the incidence matrix of the graph and it is known that the Laplacian matrix is equal to $\bL=\bA^\top \bA$, where $\|\bA\|^2_F=2m$.}:
\begin{equation}
\label{algebconeec}
\frac{1}{\lambda_{\min}^+(\bW)}\overset{\eqref{matrixW}}{=}\frac{1}{\lambda_{\min}^+(\bA^\top \bA/ \|\bA\|^2_F)}\overset{\|\bA\|^2_F=2m}{=}\frac{2m}{\lambda_{\min}^+(\bA^\top \bA)}=\frac{2m}{\lambda_{\min}^+(\bL)},
\end{equation}
where $\bL=\bA^\top \bA$ is the Laplacian matrix of the network and the quantity $\lambda_{\min}^+(\bL)$ is the very well studied \textit{algebraic connectivity} of the graph \cite{de2007old}. 
\end{rem}

\begin{rem} The convergence analysis in this chapter holds for any consistent linear system $\bA x= b$ without any assumption on the rank of the matrix $\bA$. The lack of any assumption on the form of matrix $\bA$ allows us to solve the homogeneous linear system $\bA x=0$ where $\bA$ is the incidence matrix of the network which by construction is rank deficient. More specifically, it can be shown that ${\rm rank}(\bA)=n-1$ \cite{LoizouRichtarik}. Note that many existing methods for solving linear systems make the assumption that the matrix $\bA$ of the linear systems is full rank \cite{RK, needell2010randomized, RBK} and as a result can not be used to solve the AC problem.
\end{rem}

\subsubsection{Numerical Setup}
Our goal in this experiment is to show that the addition of the momentum term to the randomized pairwise gossip algorithm (RK in the gossip setting) can lead to faster gossip algorithms and as a result the nodes of the network will converge to the average consensus faster both in number of iterations and in time. We do not intend to analyze the distributed behavior of the method (this is \blue{an ongoing} research work). In our implementations we use three of the most popular graph topologies in the literature of wireless sensor networks. These are the line graph, cycle graph and the random geometric graph $G(n,r)$. In practice, $G(n,r)$ consider ideal for modeling wireless sensor networks, because of their particular formulation. In the experiments the $2$-dimensional $G(n,r)$ is used which is formed by placing $n$ nodes uniformly at random in a unit square with edges only between nodes that have euclidean distance less than the given radius $r$. To preserve the connectivity of $G(n, r)$ a radius $r = r(n) =  \log(n)/n$ is used \cite{penrose2003random}. 
The AC problem is solved for the three aforementioned networks for both $n=100$ and $n=200$ number of nodes. We run mRK with several momentum parameters $\beta$ for 10 trials and we plot their average. Our results are available in Figures~\ref{consensus100} and \ref{consensus200}.

Note that the vector of the initial values of the nodes can be chosen arbitrarily, and the proposed algorithms will  find the average  of these values. In Figures~\ref{consensus100} and \ref{consensus200} the initial value of each node is chosen independently at random from the uniform distribution in the interval $(0,1)$.

\subsubsection{Experimental Results}

By observing Figures~\ref{consensus100} and \ref{consensus200}, it is clear that the addition of the momentum term improves the performance of the popular pairwise randomized gossip (PRG) method \cite{boyd2006randomized}.  The choice $\beta=0.4$ as the momentum parameter improves the performance of the vanilla PRG for all networks under study and $\beta=0.5$ is a good choice for the cases of the cycle and line graph. Note that for networks such as the cycle and line graphs there are known closed form expressions for the algebraic connectivity \cite{de2007old}. Thus, using equation \eqref{algebconeec}, we can compute the exact values of the condition number $1/\lambda_{\min}^+$ for these networks. Interestingly, as we can see in Table~\ref{Algebraic Connectvity} for $n=100$ and $n=200$ (number of nodes), the condition number $1/\lambda_{\min}^+$ appearing in the iteration complexity of our methods is not very large. This is in contrast with experimental  observations from Section~\ref{gaussiansyntheric} where it was shown that the choice $\beta=0.5$ is good  for very ill conditioned problems only ($1/\lambda_{\min}^+$ very large). 

\begin{table}[H]
\begin{center}
{\footnotesize
\begin{tabular}{ | p{2cm} | p{4cm} | p{3cm} | p{3cm}|  }
 \hline
Network & Formula for $\lambda_{\min}^+(\bL)$ & $1/\lambda_{\min}^+$ for $n=100$ & $1/\lambda_{\min}^+$ for $n=200$\\
   \hline
Line & $2\left(1-\cos ({\pi}/{n})\right)$& 1013 &4052\\
 \hline
Cycle &  $2\left(1-\cos ({2\pi}/{n})\right)$&253 &1013\\
 \hline
\end{tabular}
}
\end{center}
\caption{Algebraic connectivity of cycle and line graph for $n=100$ and $n=200$}
\label{Algebraic Connectvity}
\end{table}
 
\begin{figure}[!]
\centering
\begin{subfigure}{.23\textwidth}
  \centering
  \includegraphics[width=1\linewidth]{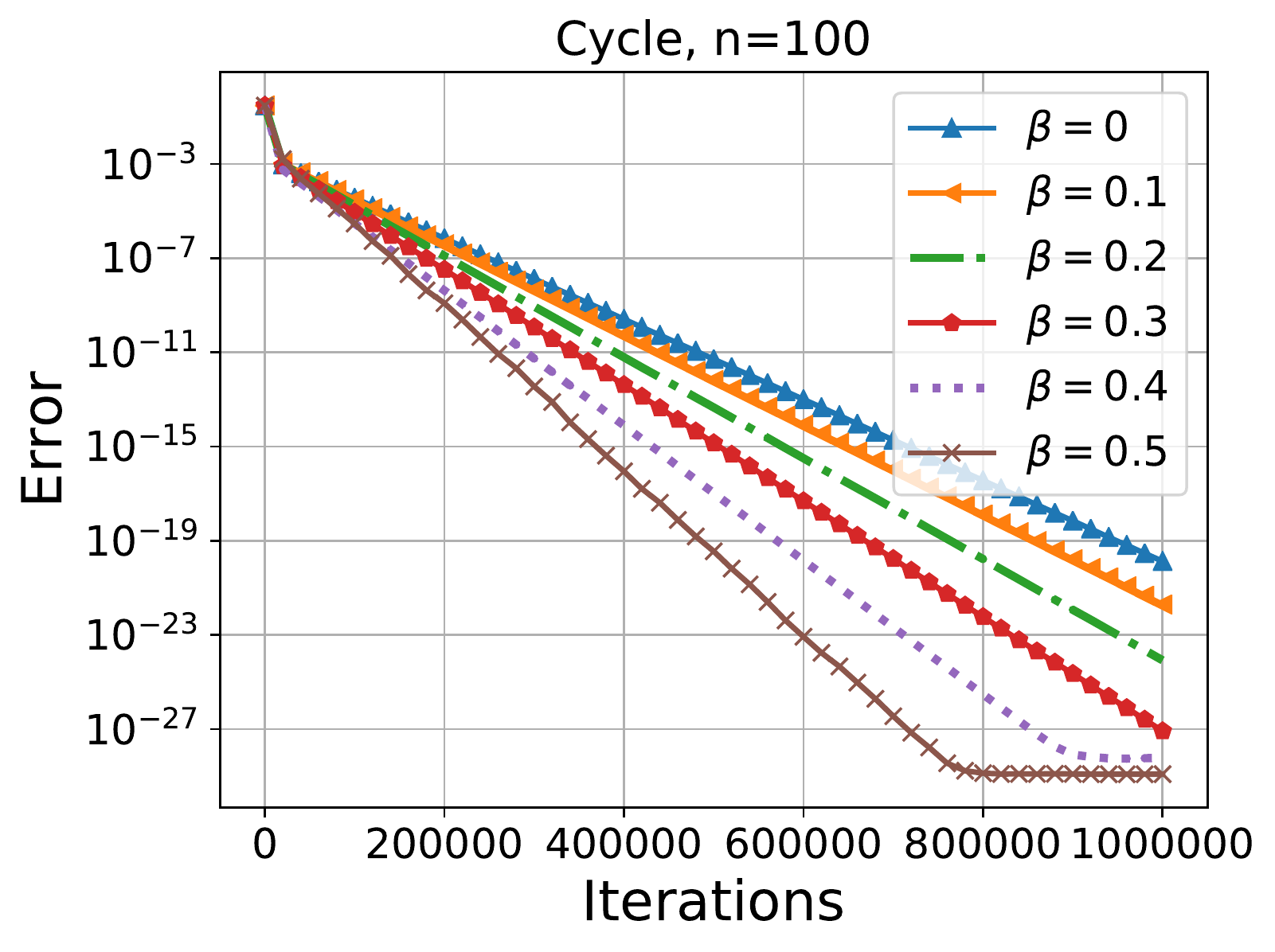}
\end{subfigure}%
\begin{subfigure}{.23\textwidth}
  \centering
  \includegraphics[width=1\linewidth]{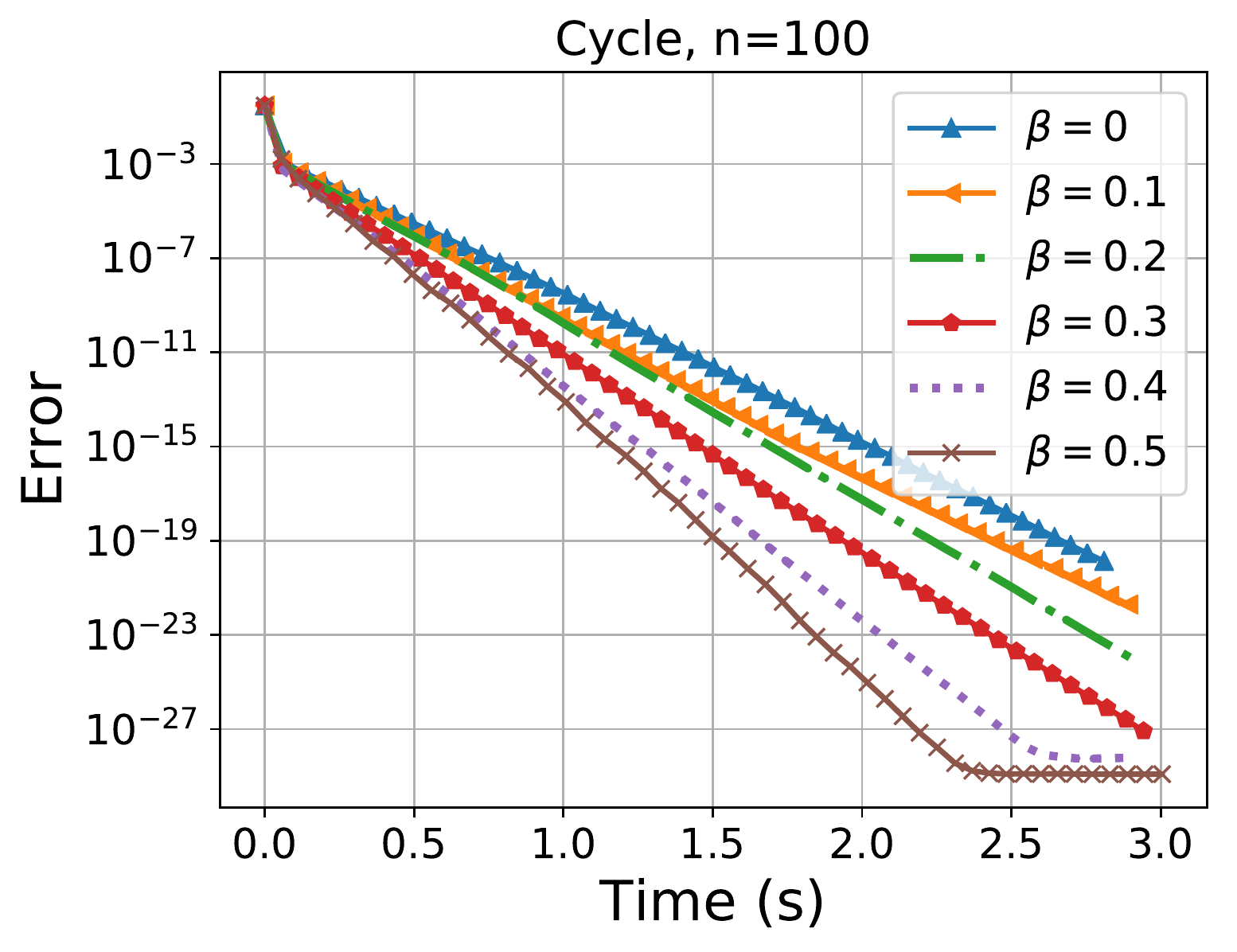}
\end{subfigure}
\begin{subfigure}{.23\textwidth}
  \centering
  \includegraphics[width=1\linewidth]{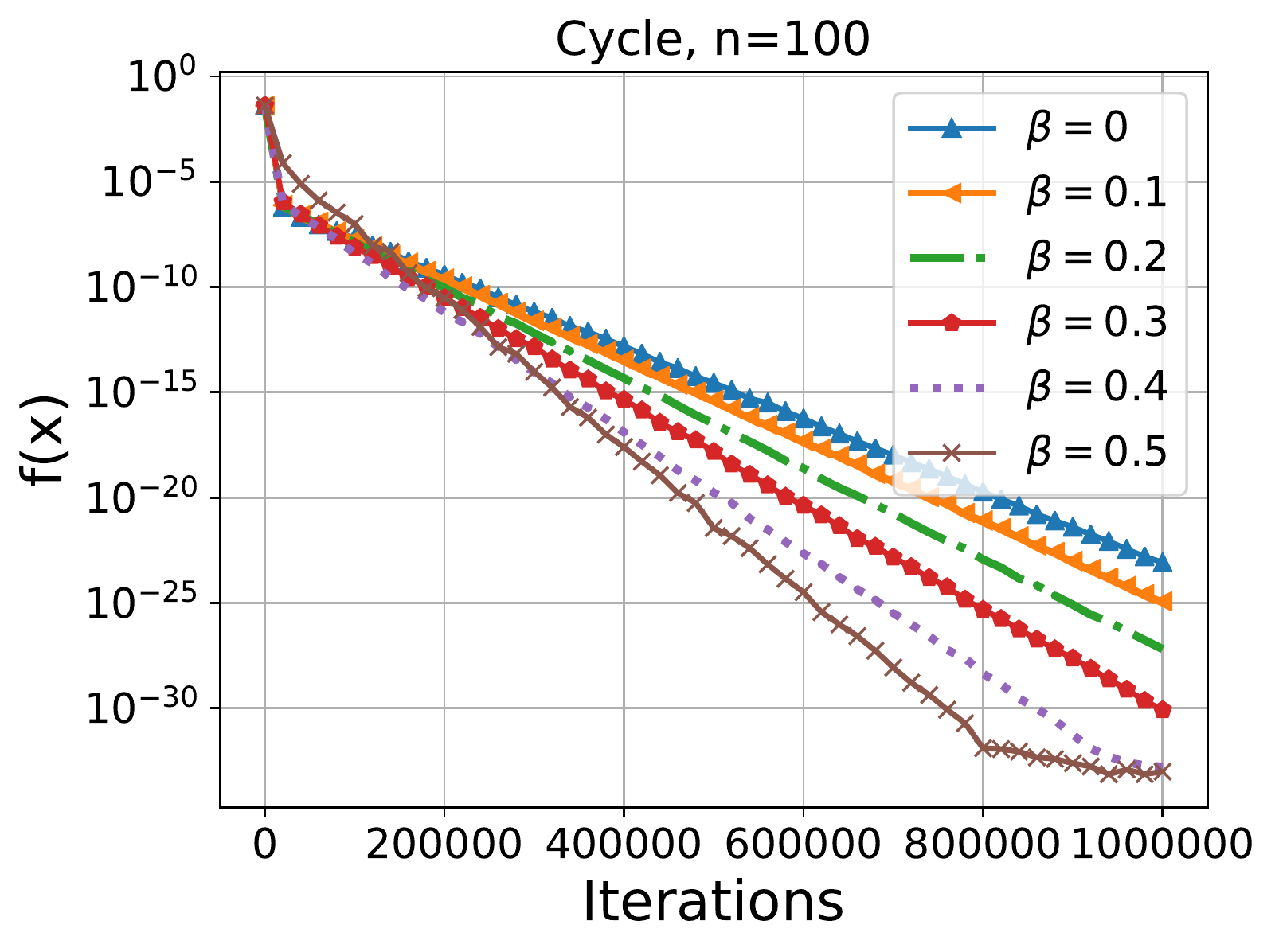}
\end{subfigure}
\begin{subfigure}{.23\textwidth}
  \centering
  \includegraphics[width=1\linewidth]{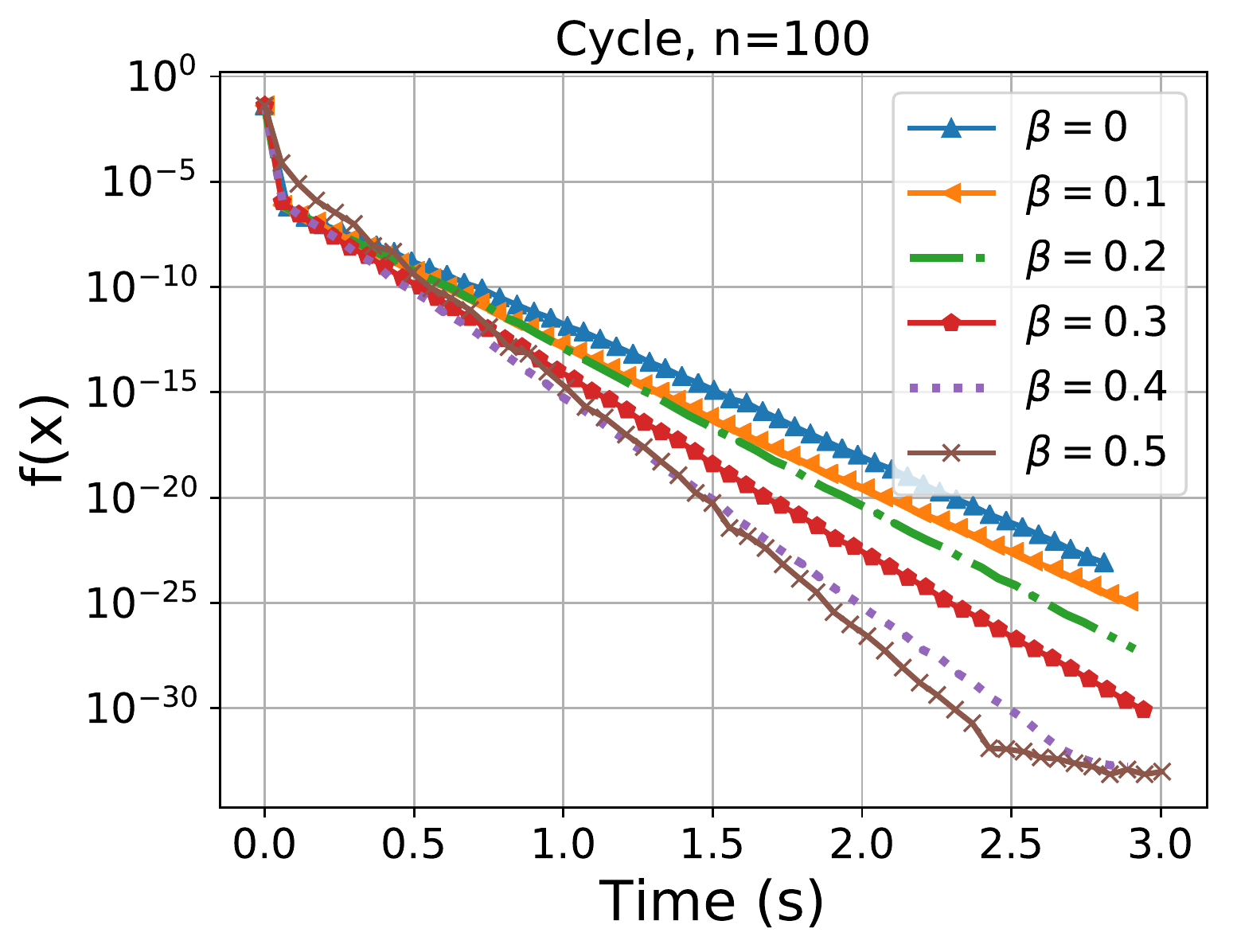}
\end{subfigure}\\
\begin{subfigure}{.23\textwidth}
  \centering
  \includegraphics[width=1\linewidth]{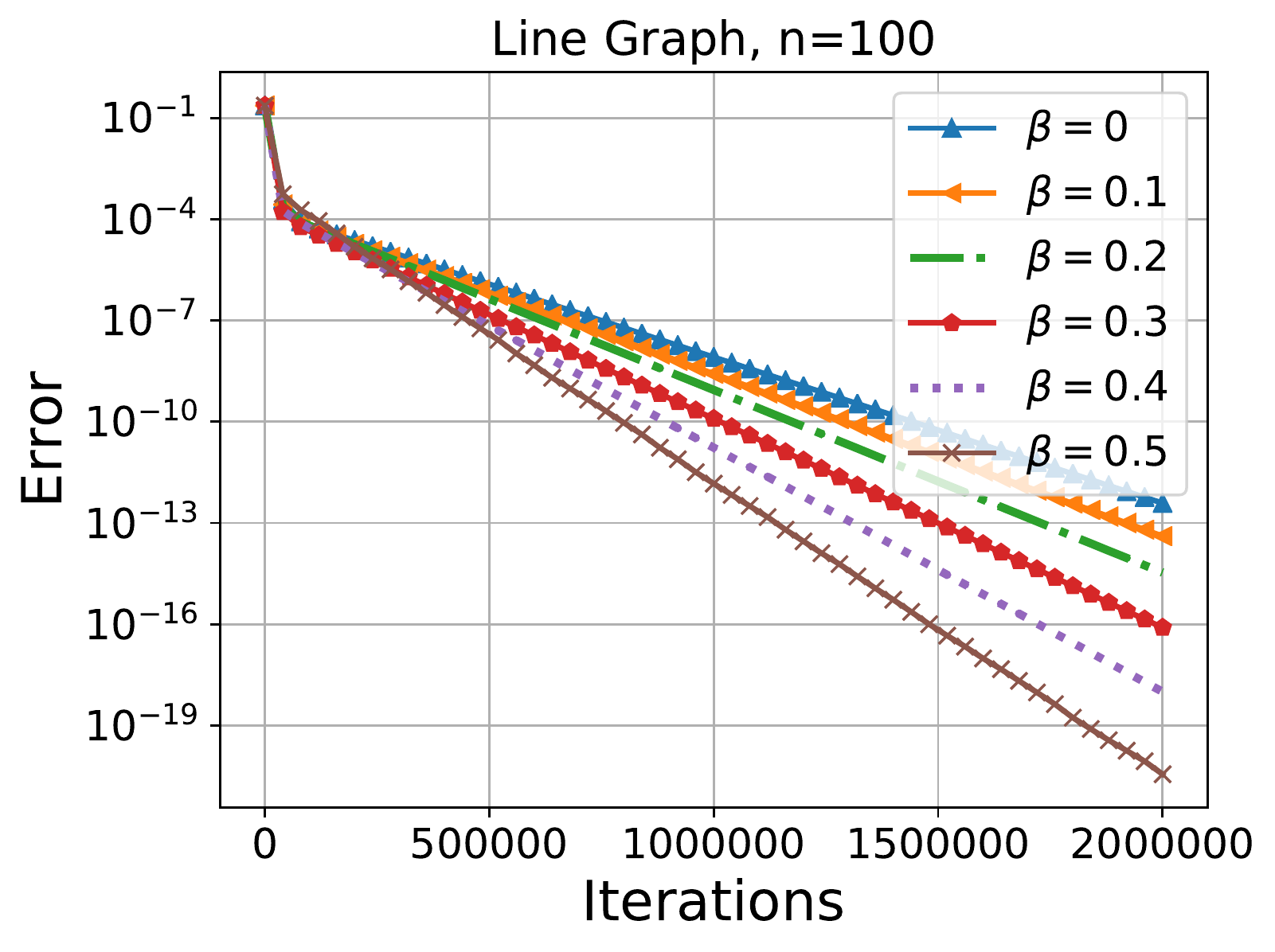}
\end{subfigure}%
\begin{subfigure}{.23\textwidth}
  \centering
  \includegraphics[width=1\linewidth]{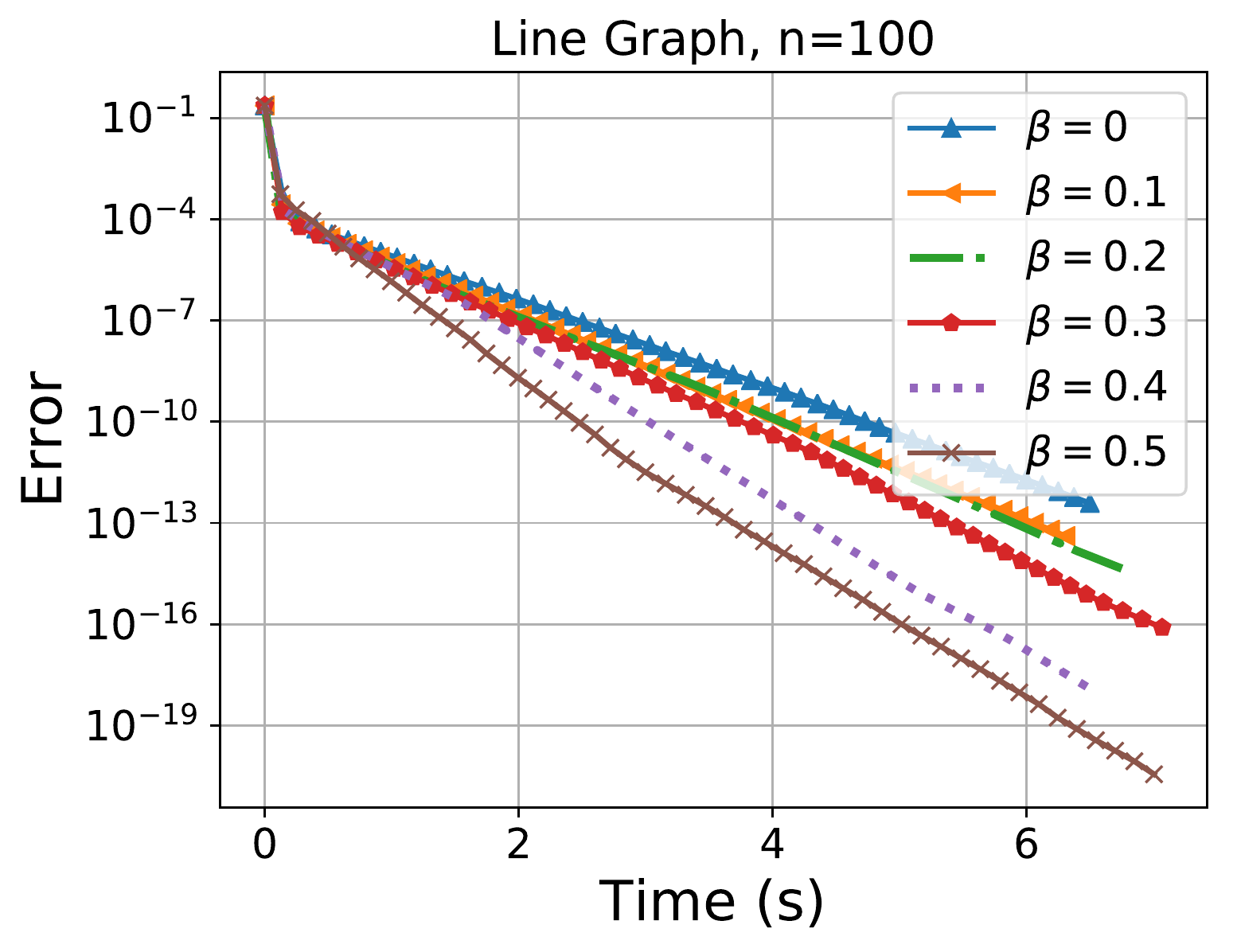}
\end{subfigure}
\begin{subfigure}{.23\textwidth}
  \centering
  \includegraphics[width=1\linewidth]{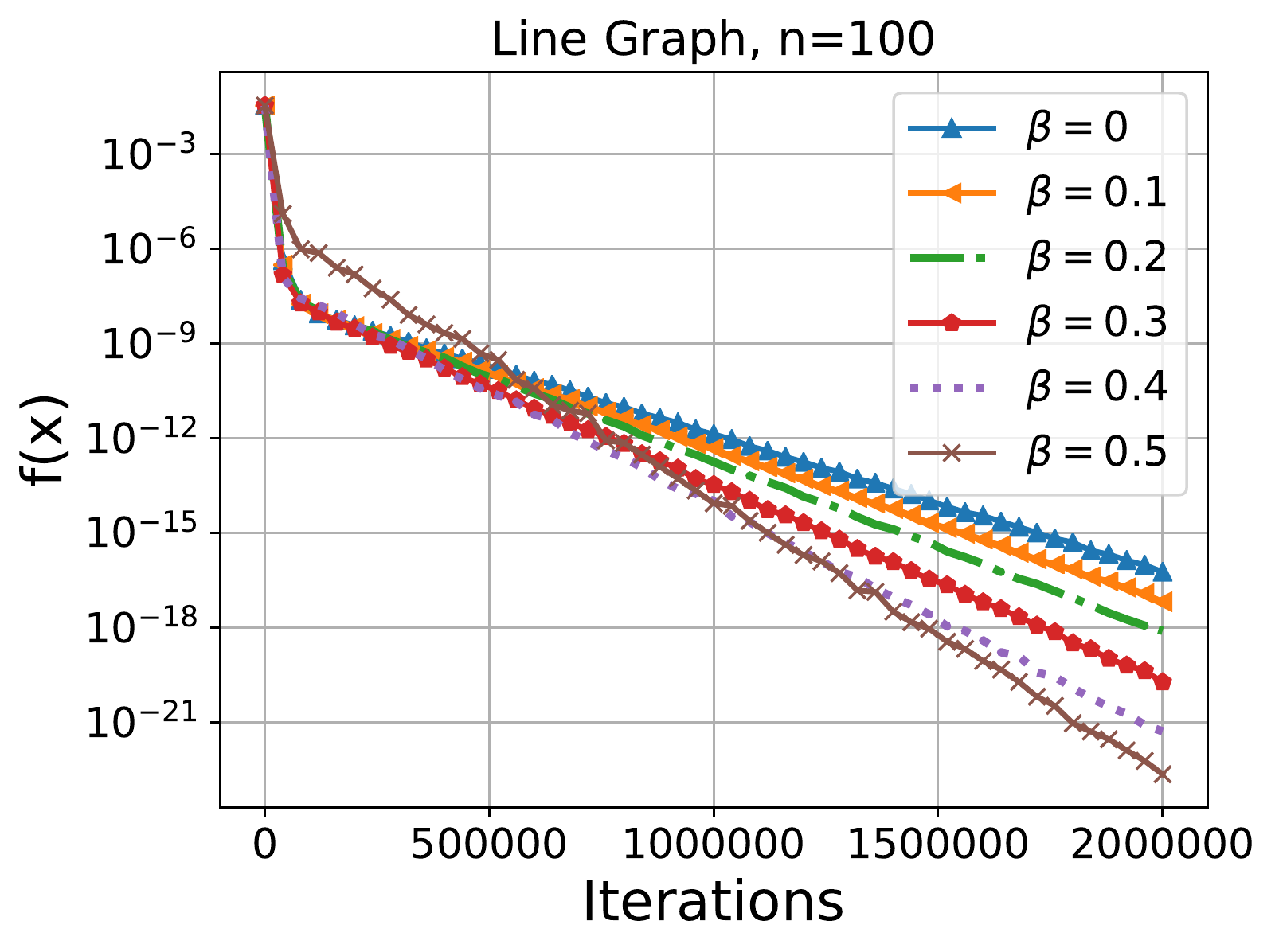}
\end{subfigure}
\begin{subfigure}{.23\textwidth}
  \centering
  \includegraphics[width=1\linewidth]{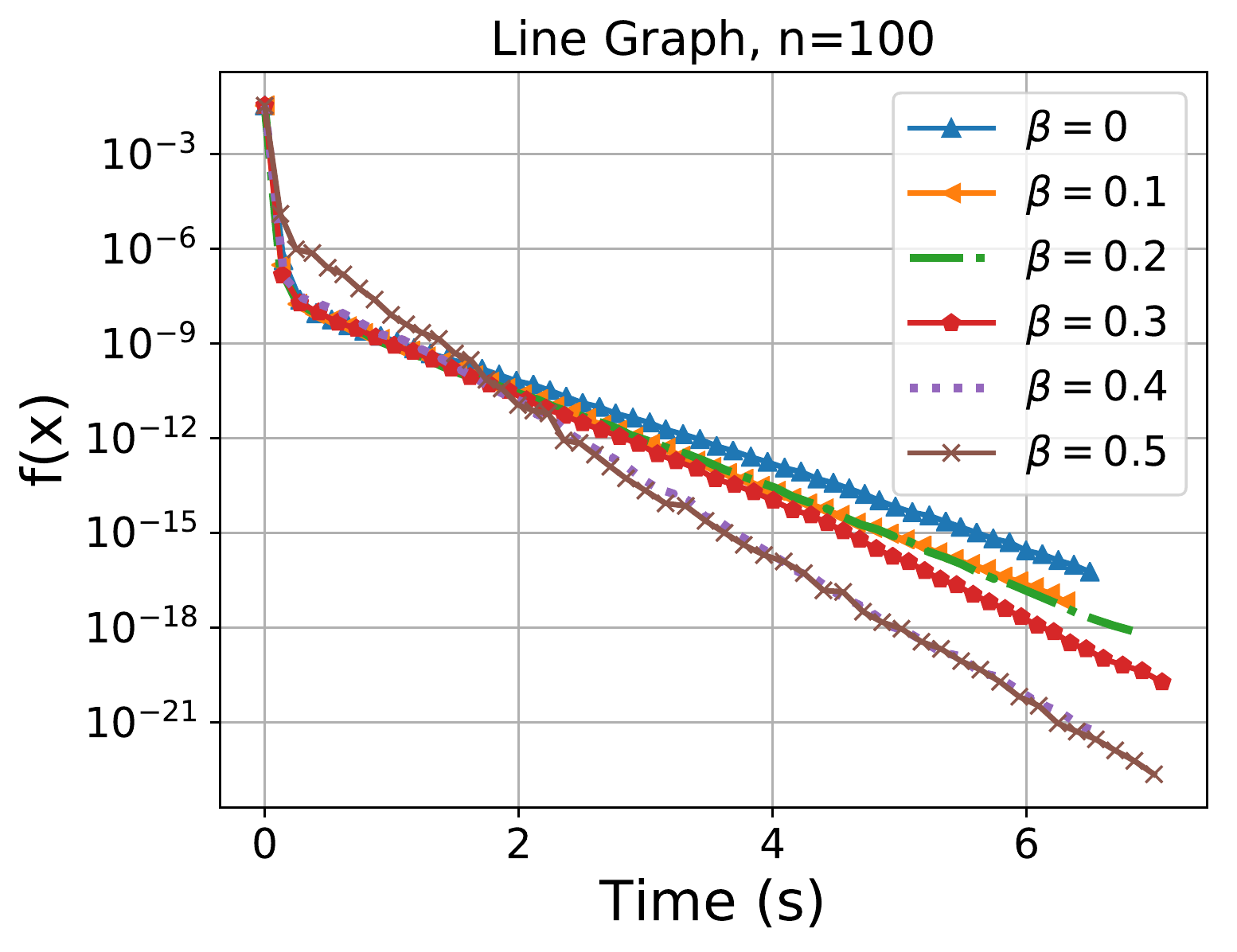}
\end{subfigure}\\
\begin{subfigure}{.23\textwidth}
  \centering
  \includegraphics[width=1\linewidth]{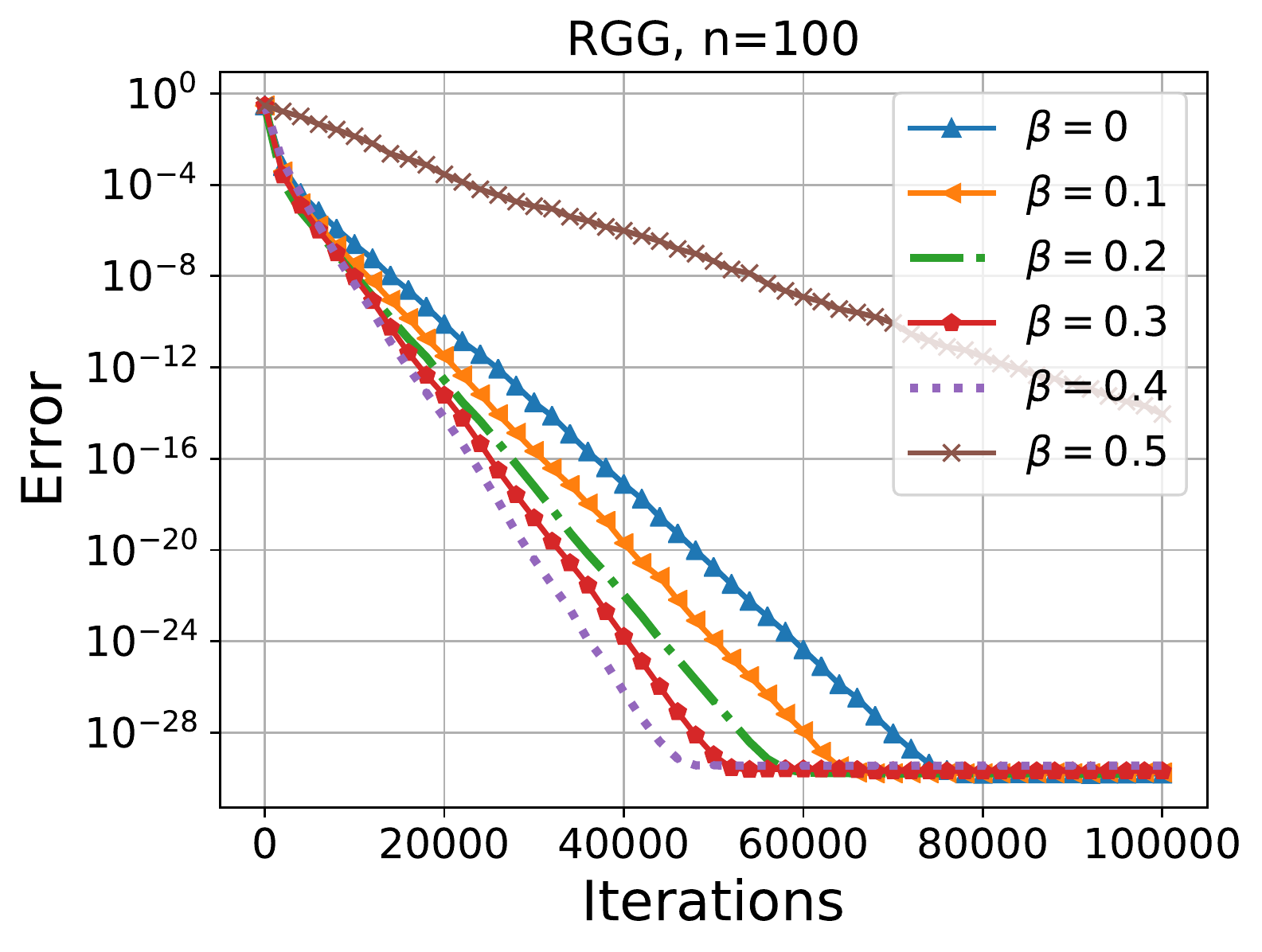}
\end{subfigure}%
\begin{subfigure}{.23\textwidth}
  \centering
  \includegraphics[width=1\linewidth]{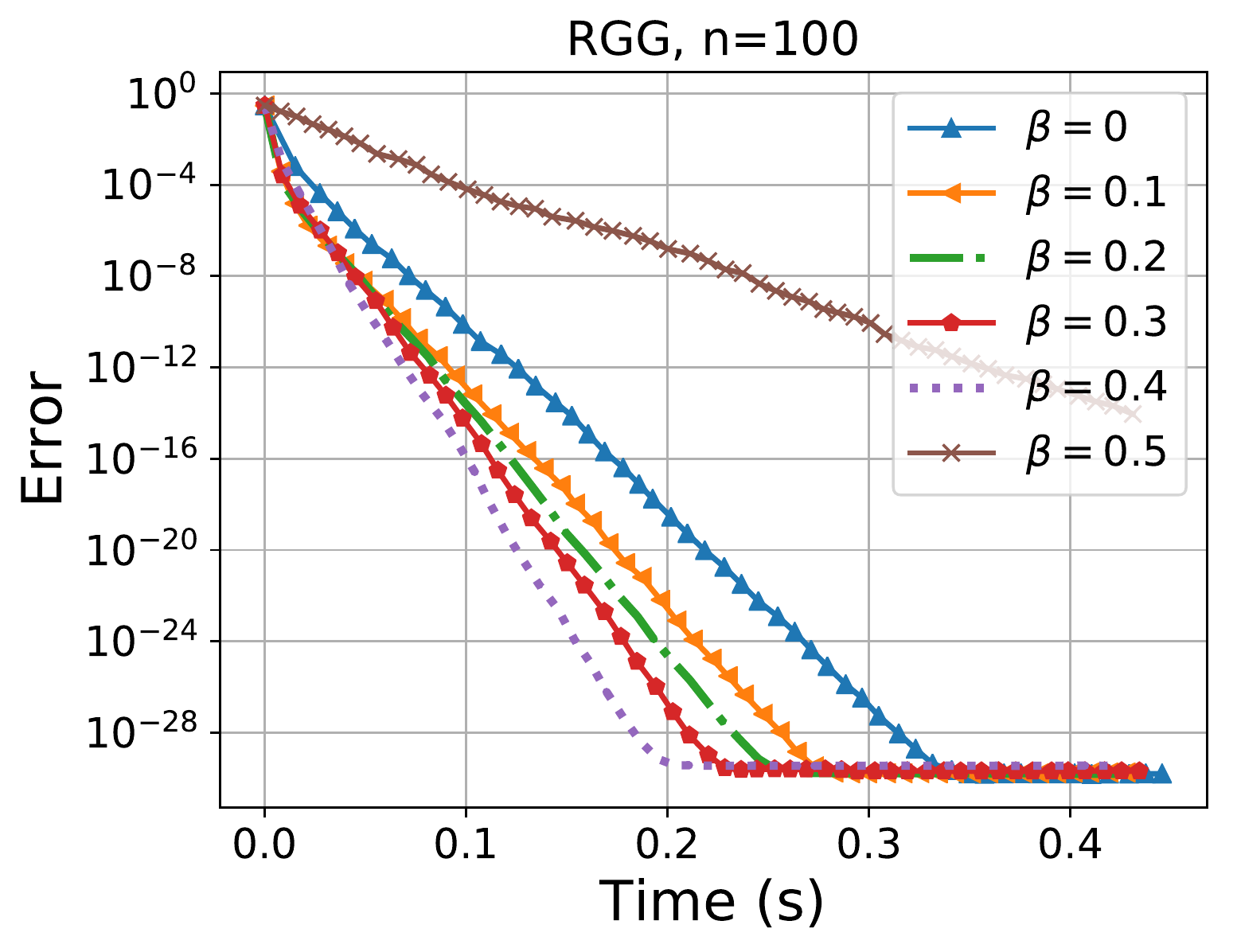}
\end{subfigure}
\begin{subfigure}{.23\textwidth}
  \centering
  \includegraphics[width=1\linewidth]{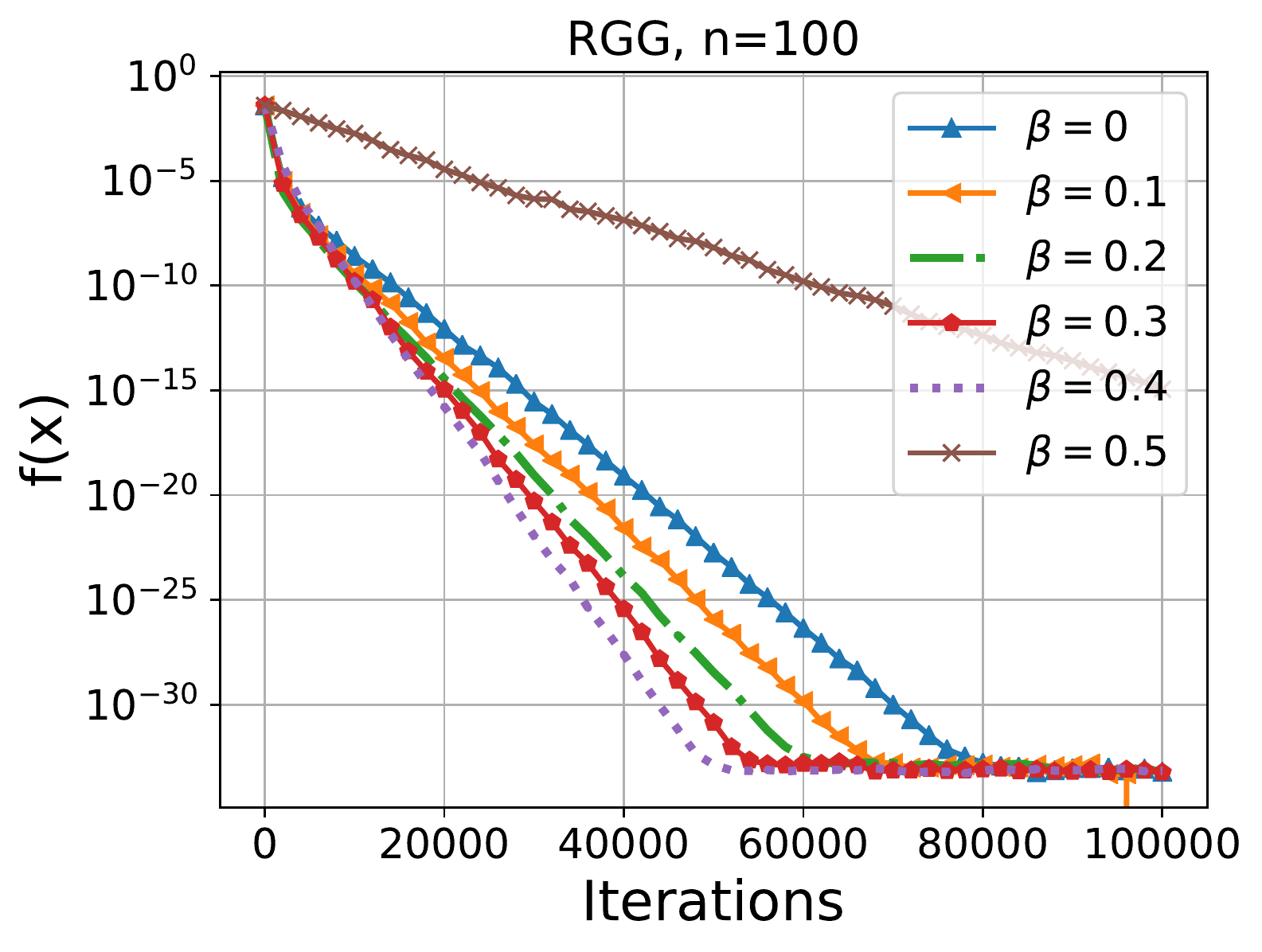}
\end{subfigure}
\begin{subfigure}{.23\textwidth}
  \centering
  \includegraphics[width=1\linewidth]{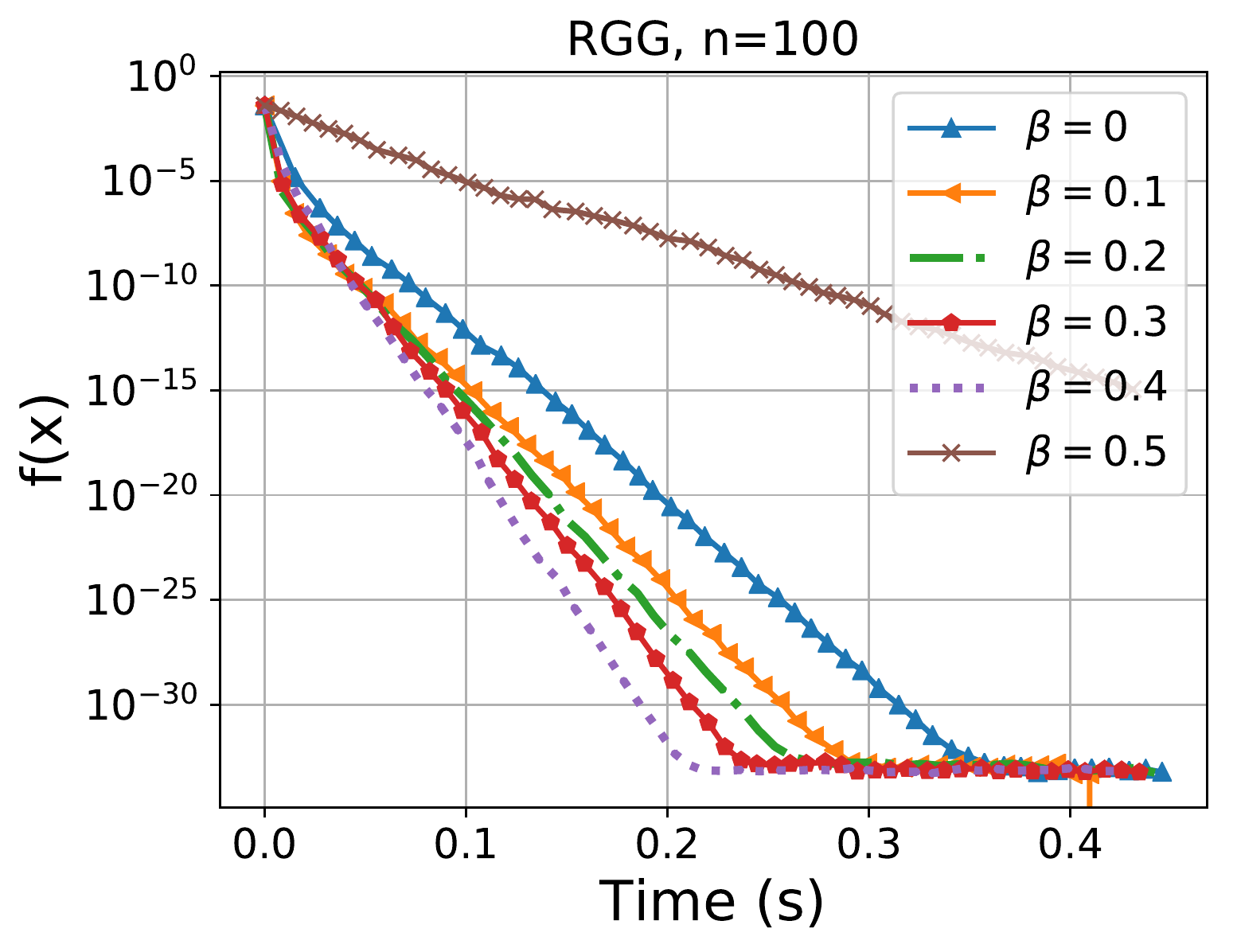}
\end{subfigure}\\
\caption{Performance of mPRG for several momentum parameters $\beta$ for solving the average consensus problem in a cycle graph, line graph and random geometric graph $G(n,r)$ with $n=100$ nodes. For the $G(n,r)$ to ensure connectivity of the network a radius $r=\sqrt{\log(n)/n}$ is used. The graphs in the first (second) column plot iterations (time) against residual error while those in the third (forth) column plot iterations (time) against function values. The ``Error" in the vertical axis represents the relative error $\|x^k-x^*\|^2_\bB / \|x^0-x^*\|^2_\bB \overset{\bB=\bI, x^0=c}{=}\|x^k-x^*\|^2 / \|c-x^*\|^2_\bB$ and the function values $f(x^k)$ refer to function~\eqref{functionRK}.}
\label{consensus100}
\end{figure}

\begin{figure}[!]
\centering
\begin{subfigure}{.23\textwidth}
  \centering
  \includegraphics[width=1\linewidth]{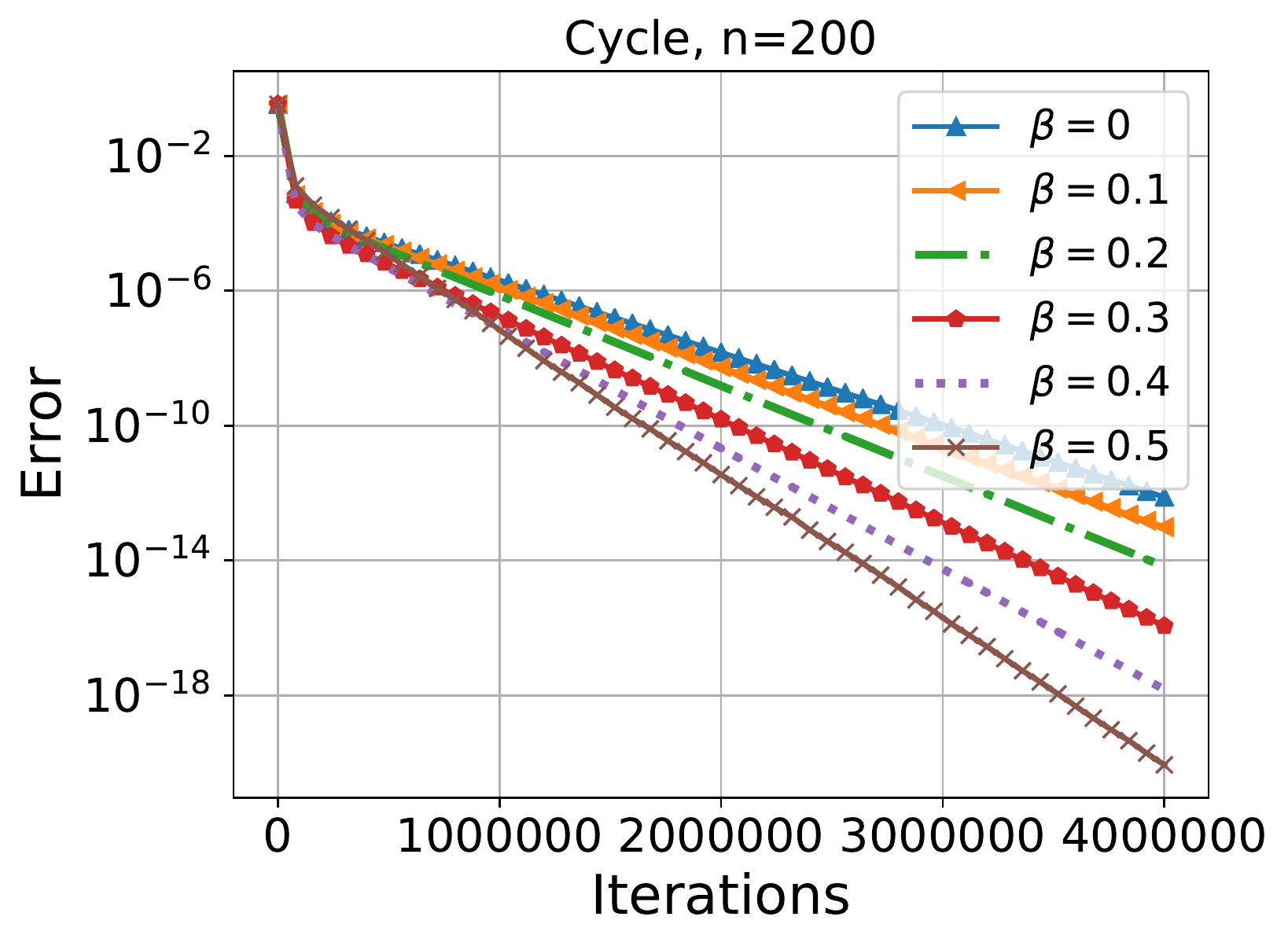}
\end{subfigure}%
\begin{subfigure}{.23\textwidth}
  \centering
  \includegraphics[width=1\linewidth]{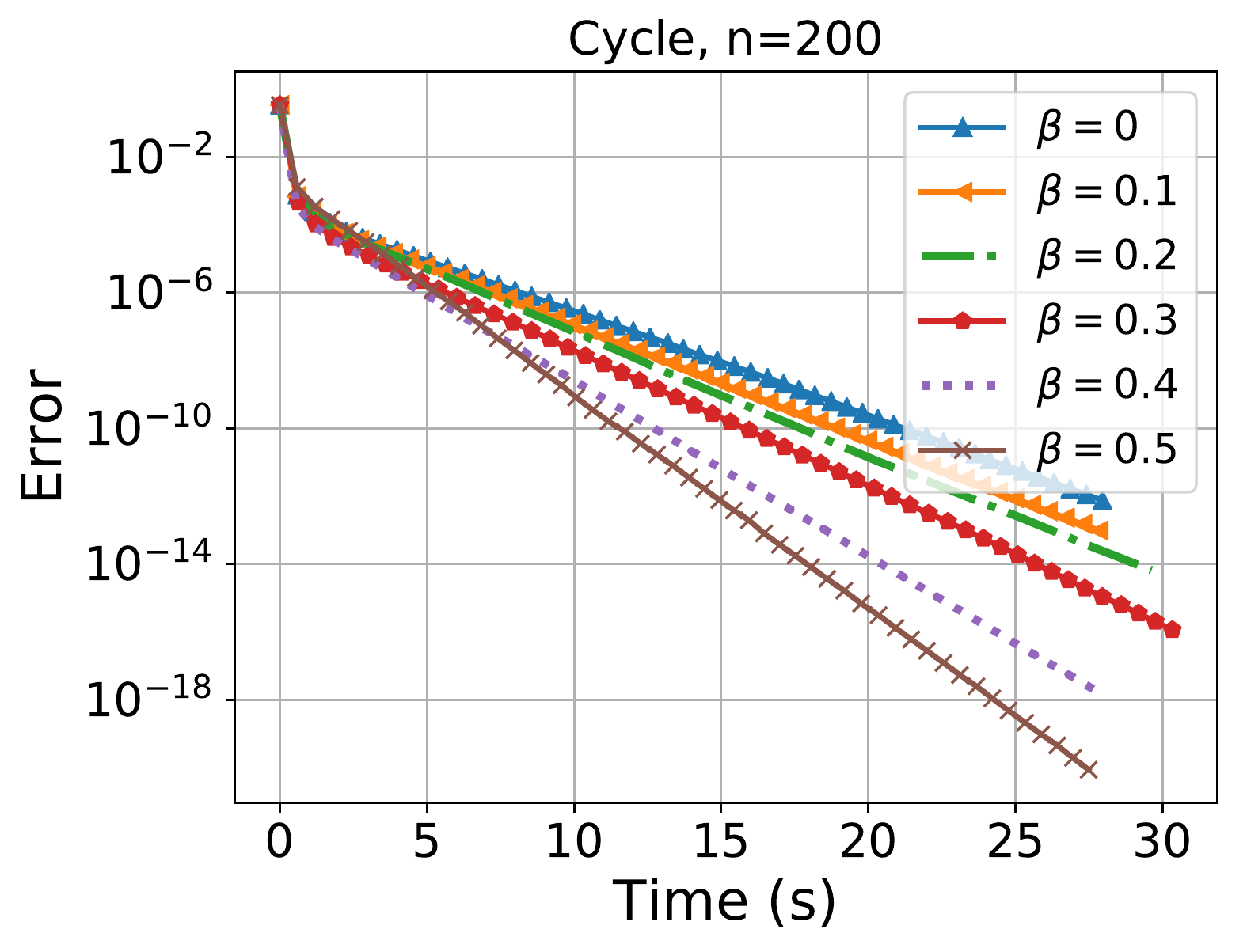}
\end{subfigure}
\begin{subfigure}{.23\textwidth}
  \centering
  \includegraphics[width=1\linewidth]{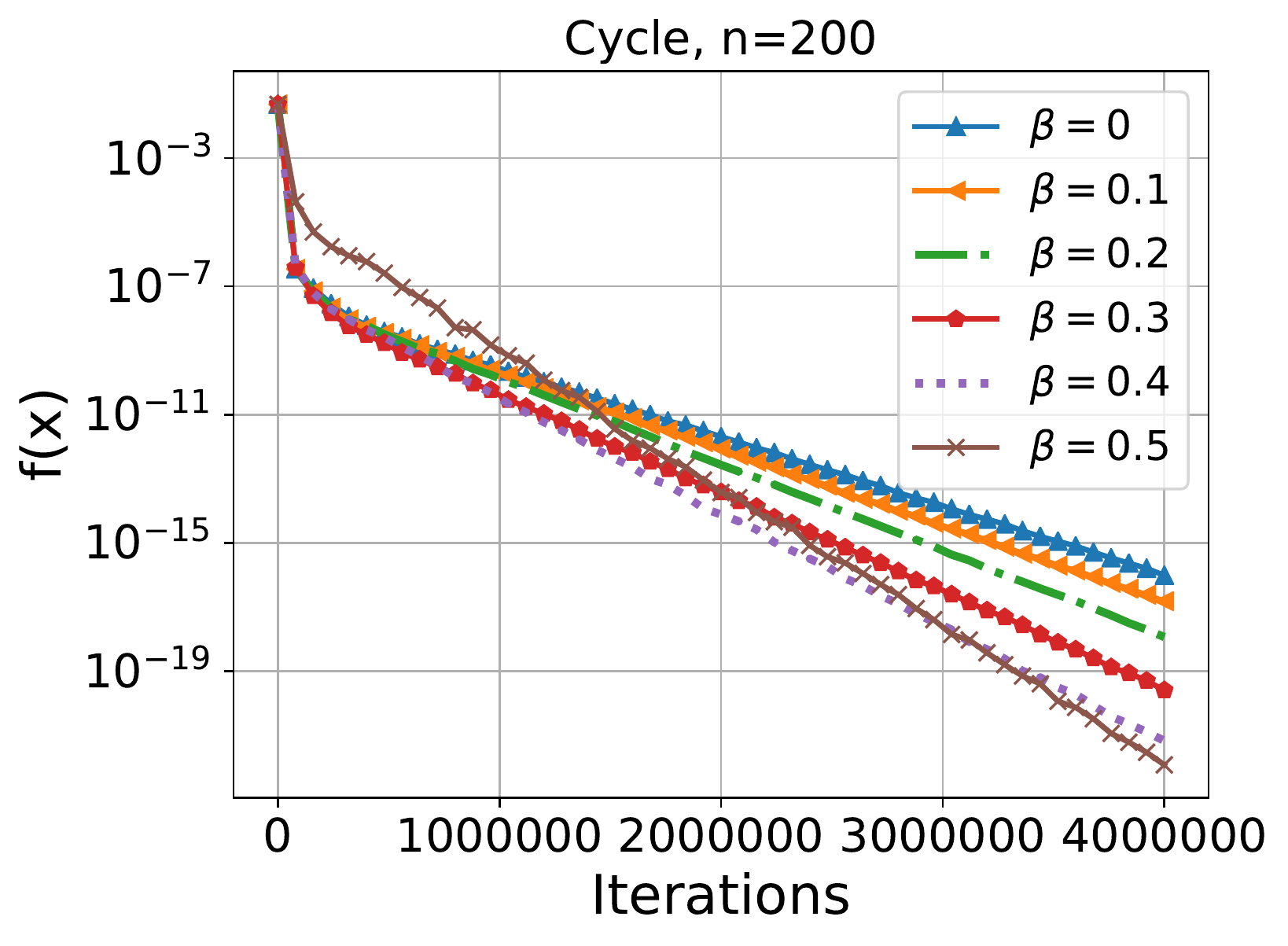}
\end{subfigure}
\begin{subfigure}{.23\textwidth}
  \centering
  \includegraphics[width=1\linewidth]{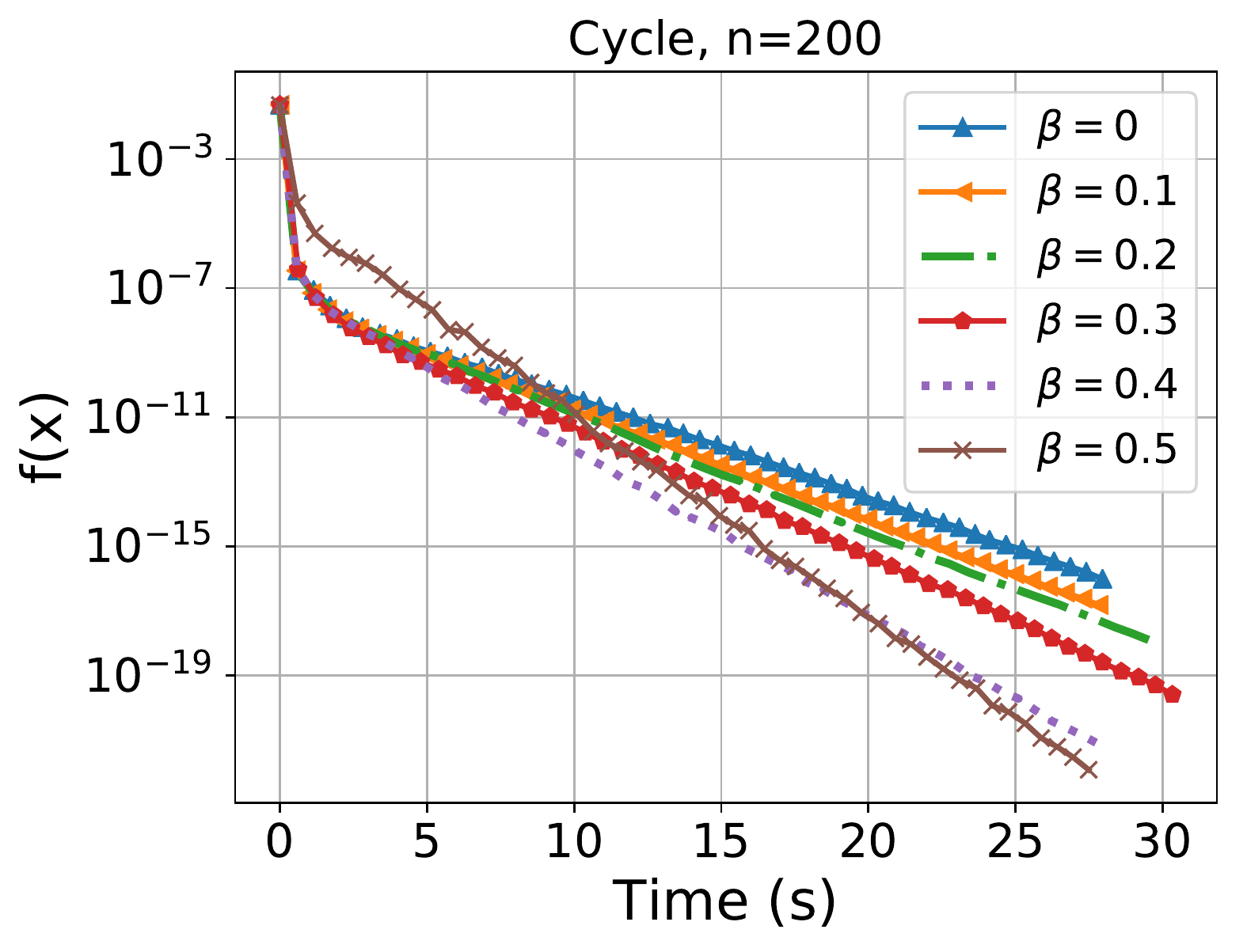}
\end{subfigure}\\
\begin{subfigure}{.23\textwidth}
  \centering
  \includegraphics[width=1\linewidth]{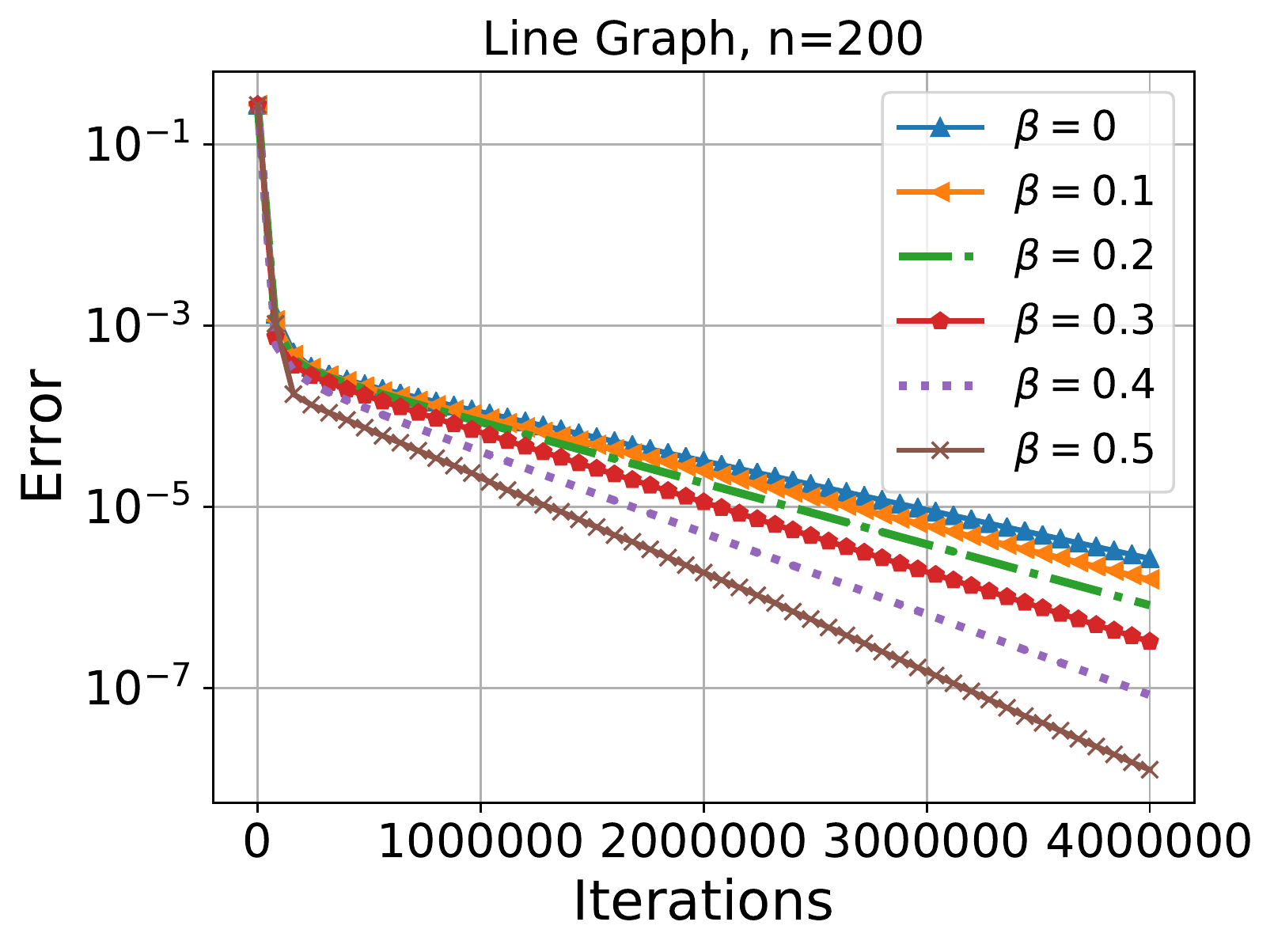}
\end{subfigure}%
\begin{subfigure}{.23\textwidth}
  \centering
  \includegraphics[width=1\linewidth]{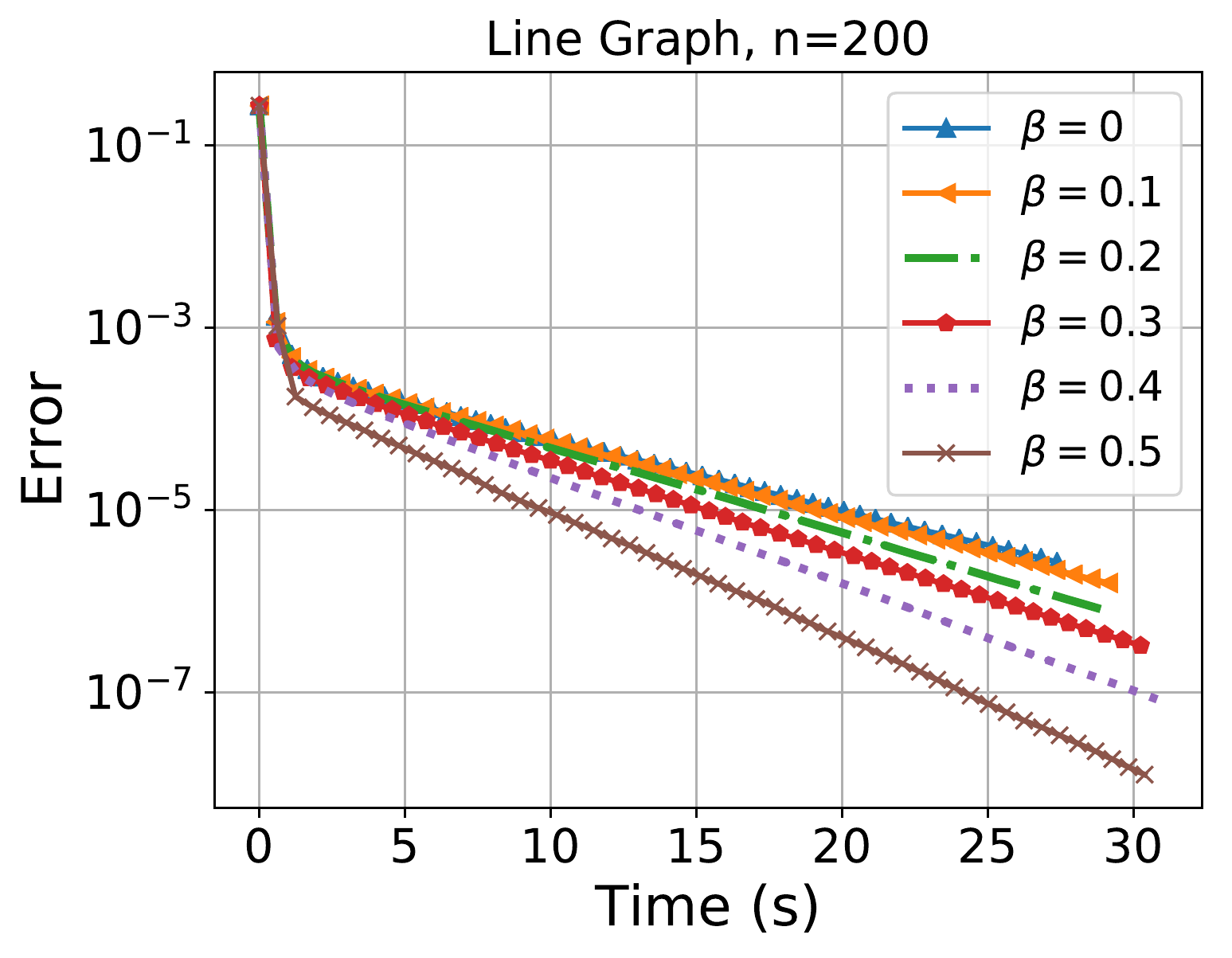}
\end{subfigure}
\begin{subfigure}{.23\textwidth}
  \centering
  \includegraphics[width=1\linewidth]{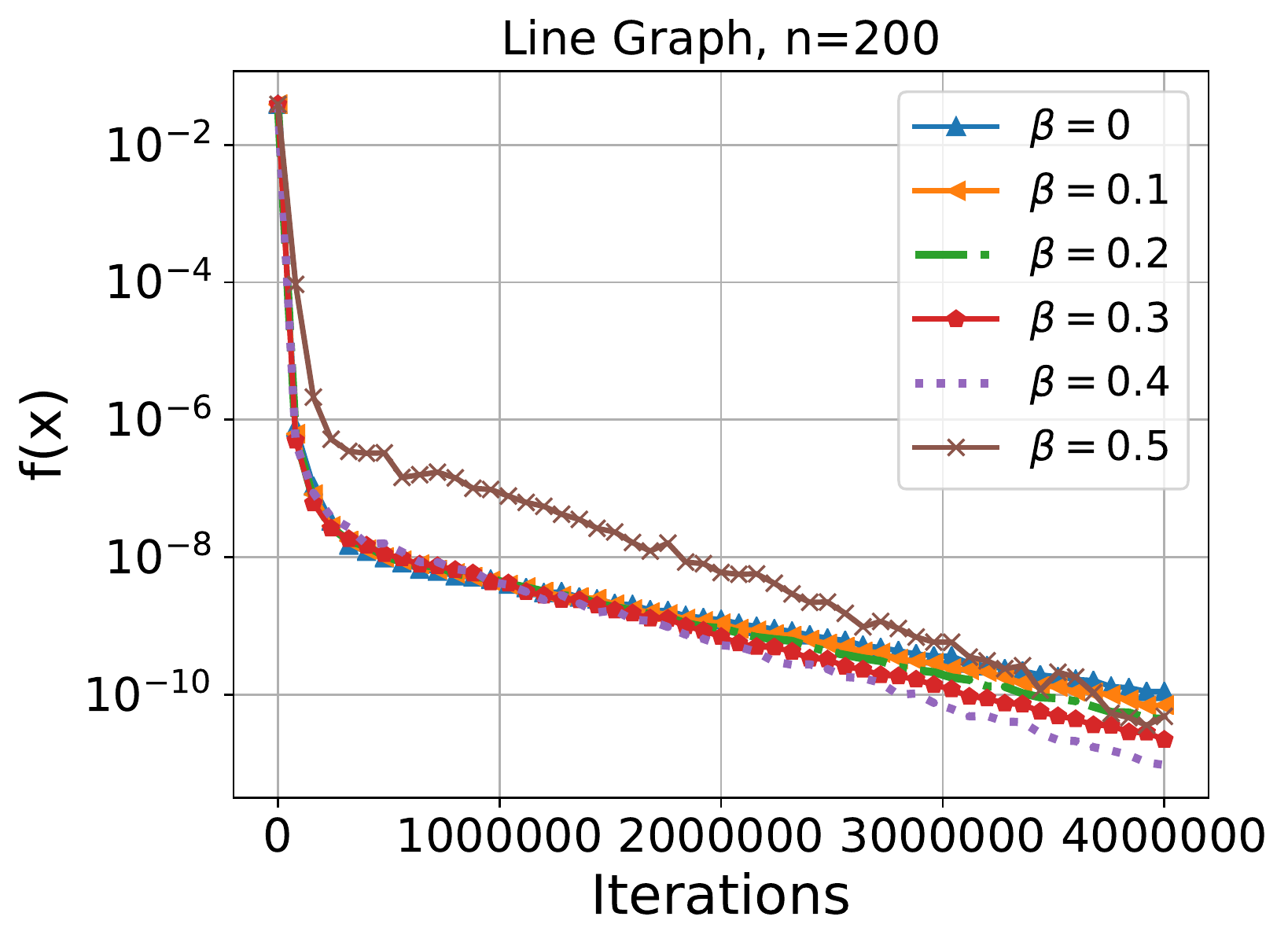}
\end{subfigure}
\begin{subfigure}{.23\textwidth}
  \centering
  \includegraphics[width=1\linewidth]{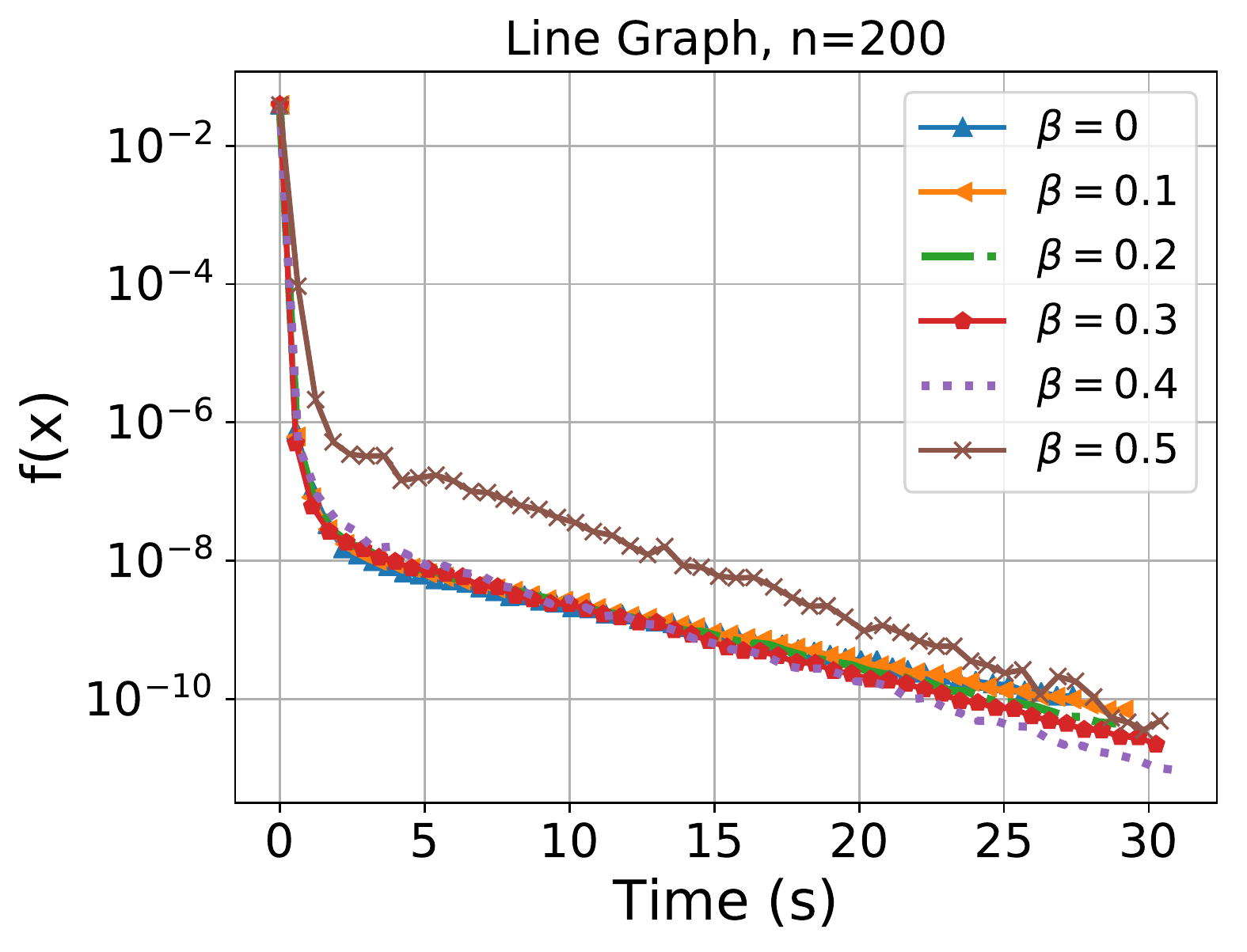}
\end{subfigure}\\
\begin{subfigure}{.23\textwidth}
  \centering
  \includegraphics[width=1\linewidth]{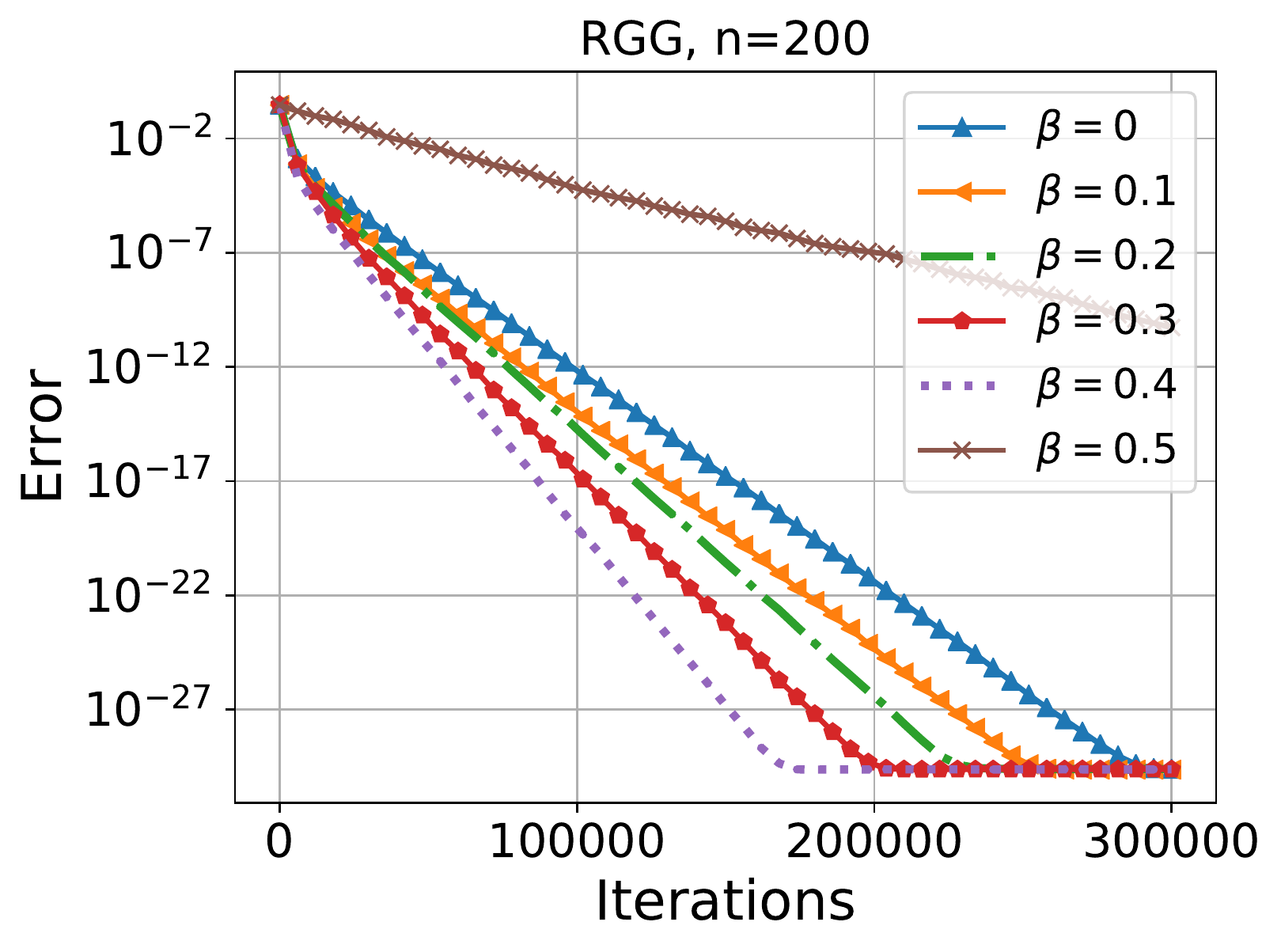}
\end{subfigure}%
\begin{subfigure}{.23\textwidth}
  \centering
  \includegraphics[width=1\linewidth]{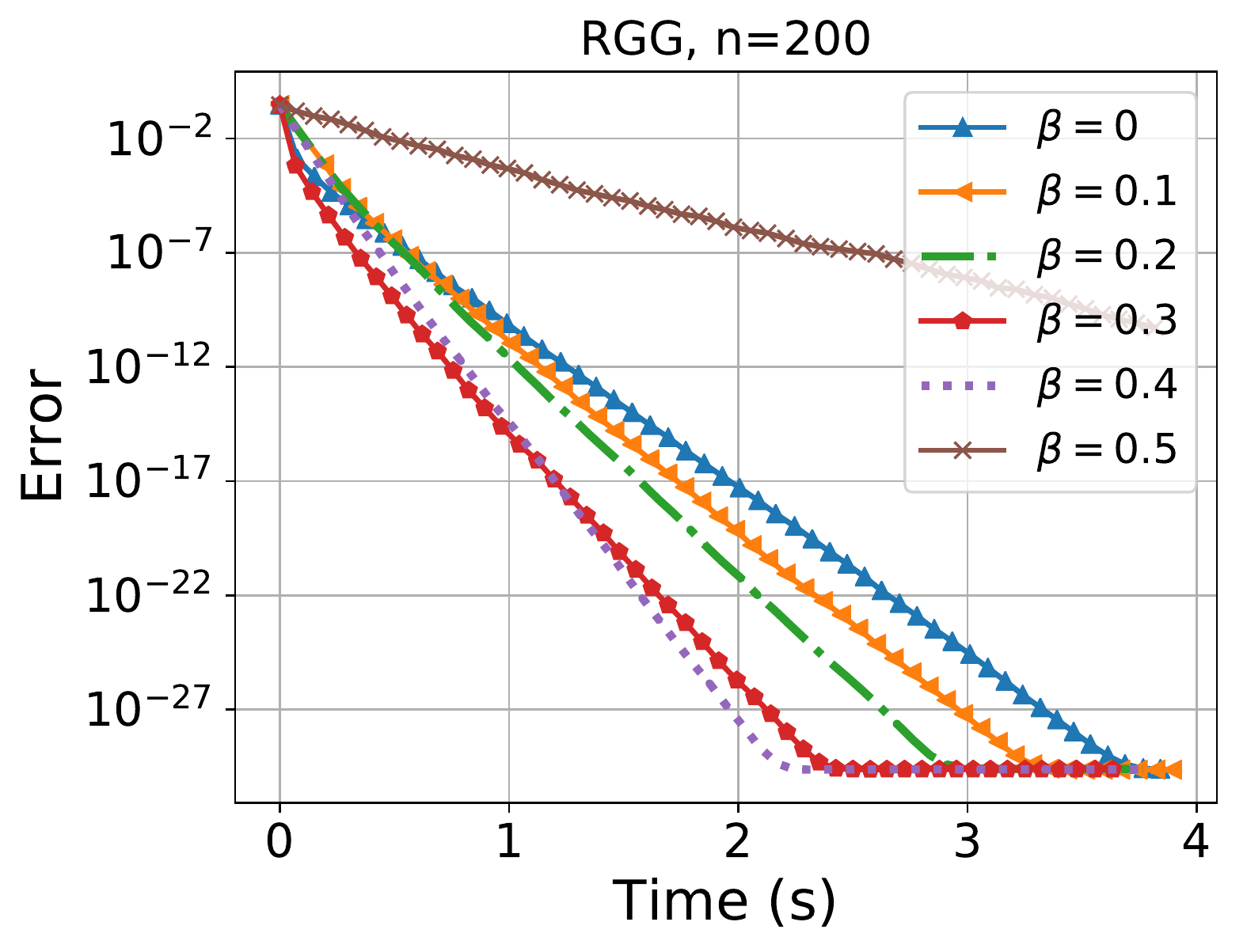}
\end{subfigure}
\begin{subfigure}{.23\textwidth}
  \centering
  \includegraphics[width=1\linewidth]{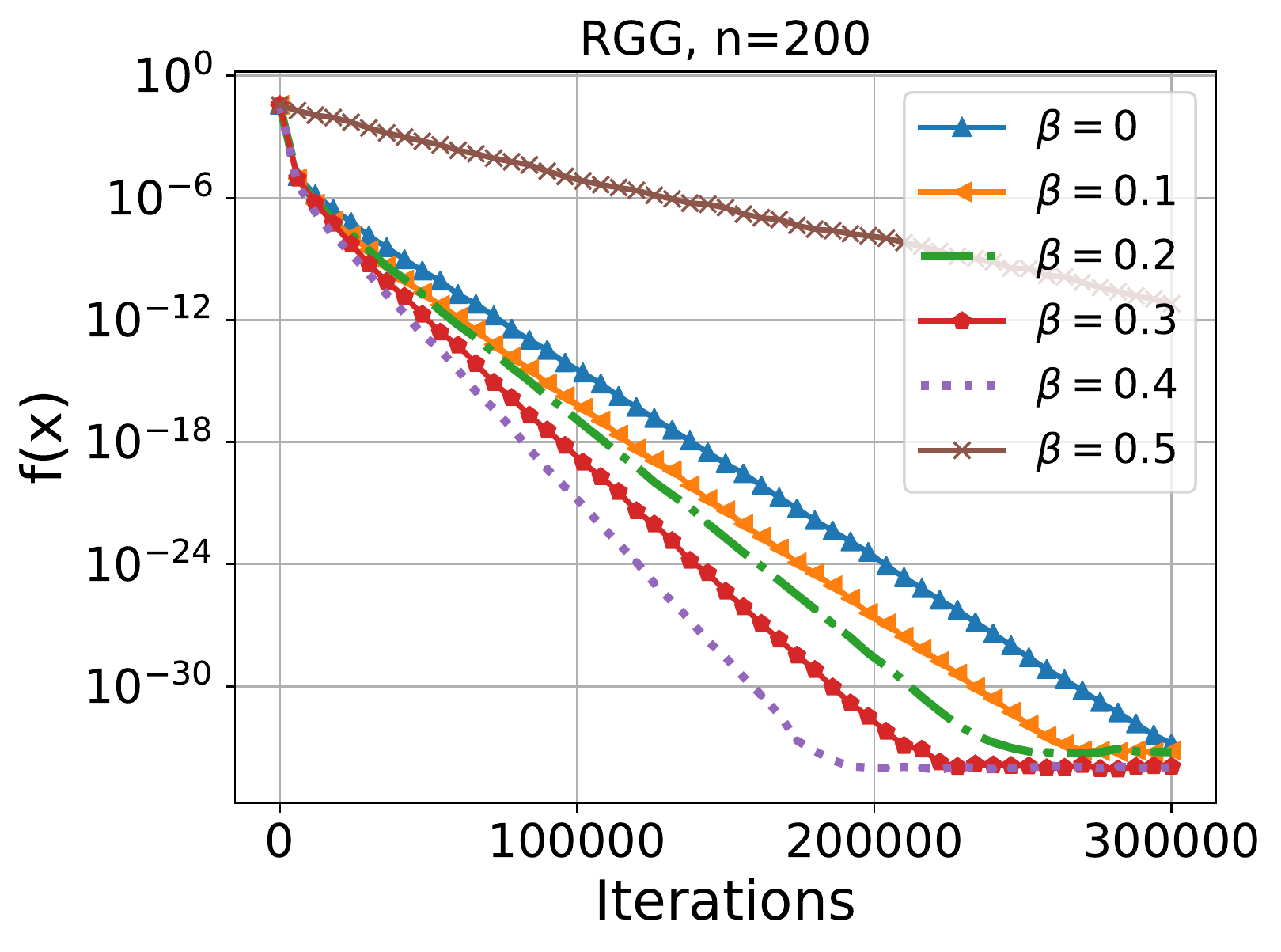}
\end{subfigure}
\begin{subfigure}{.23\textwidth}
  \centering
  \includegraphics[width=1\linewidth]{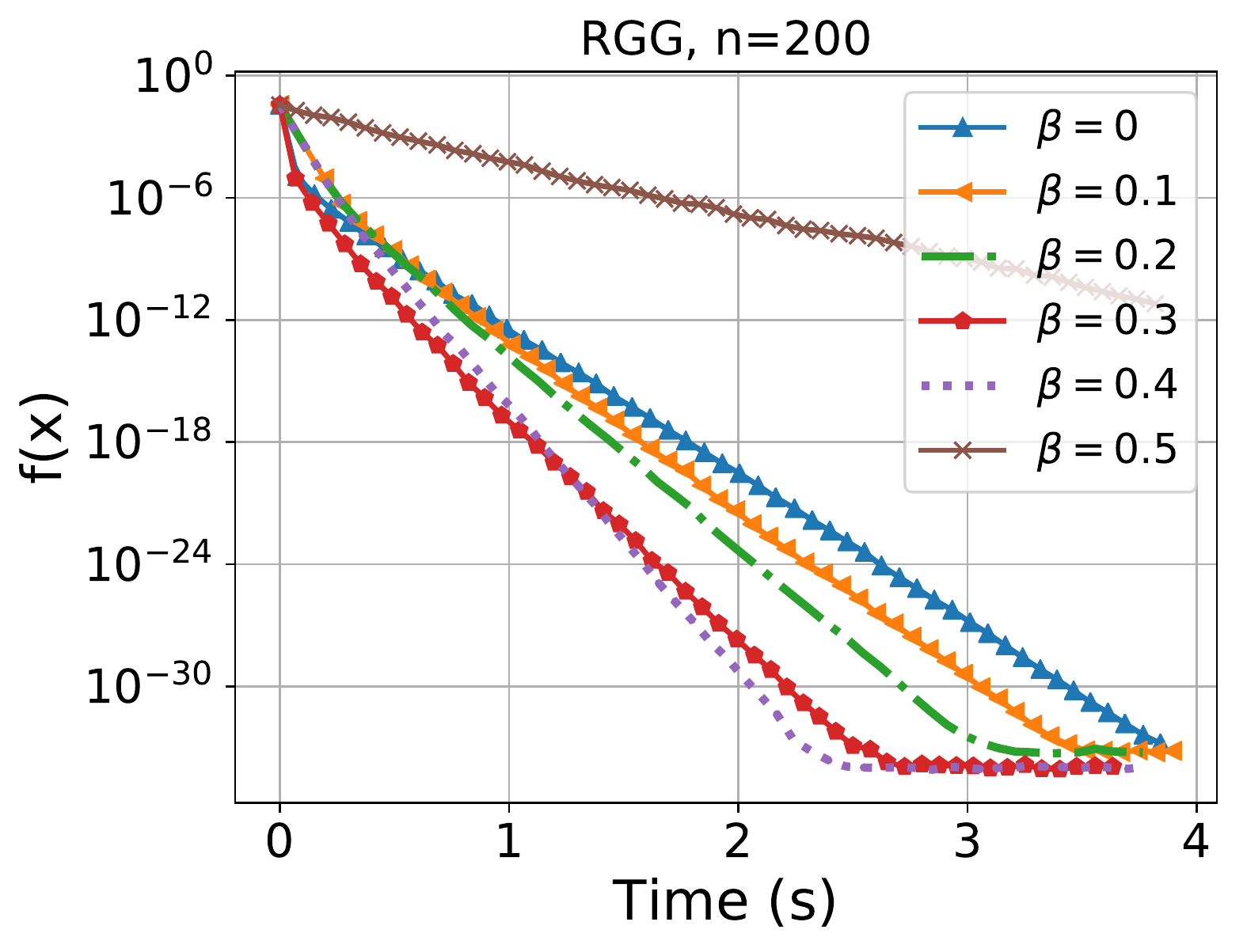}
\end{subfigure}\\
\caption{Performance of mPRG for several momentum parameters $\beta$ for solving the average consensus problem in a cycle graph, line graph and random geometric graph $G(n,r)$ with $n=200$ nodes. For the $G(n,r)$ to ensure connectivity of the network a radius $r=\sqrt{\log(n)/n}$ is used. The graphs in the first (second) column plot iterations (time) against residual error while those in the third (forth) column plot iterations (time) against function values. The ``Error" in the vertical axis represents the relative error $\|x^k-x^*\|^2_\bB / \|x^0-x^*\|^2_\bB \overset{\bB=\bI, x^0=c}{=}\|x^k-x^*\|^2 / \|c-x^*\|^2$ and the function values $f(x^k)$ refer to function~\eqref{functionRK}.}
\label{consensus200}
\end{figure}

\section{Conclusion}

\blue{In this chapter, we studied the convergence analysis of several stochastic optimization algorithms enriched with heavy ball momentum for solving  stochastic optimization problems of special structure.  We proved global, non-asymptotic linear convergence rates of all of these methods as well as accelerated linear rate for the case of the norm of expected iterates. We also introduced a new momentum strategy called \textit{stochastic momentum} which is beneficial for the case of sparse data and proved linear convergence in this setting. We corroborated our theoretical results with extensive experimental testing.

Our work is amenable to further extensions. A natural extension of our results is the analysis of heavy ball momentum variants of our proposed methods (SGD,SN, SPP, etc) in the case of general convex or strongly convex functions. While we have shown that the expected iterates converge in an accelerated manner, it is an open problem whether an accelerated rate can be established for the expected distance, i.e., for $\E{\|x^k-x^*\|_\mB^2}$.    In our analysis we also focus on the case of fixed constant step-size and momentum parameters. A study of the effect of decreasing or adaptive choice of the parameters might provide novel insights. 

The obtained results hold under the exactness condition which as we explain is very weak, allowing for virtually arbitrary distributions $\cD$ from which the random matrices are drawn. One may wish to design optimized distributions in terms of the  convergence rates or overall complexity. 

Finally, we show how the addition of momentum on top of gossip algorithms can lead to faster convergence. An interesting question is to interpret the distributed nature of these algorithms and try to understand how we can improve the analysis using the properties of the underlying network. This is precisely what we are doing later in Chapter~\ref{ChapterGossip}.} 

\section{Proofs of Main Results}
\label{ProofsMomentum}

\subsection{Technical lemmas}

\begin{lem}
\label{LemmaGlobal}
Fix $F^1=F^0\geq 0$ and let $\{F^k\}_{k\geq 0}$ be a  sequence of nonnegative  real numbers satisfying the relation 
\begin{equation} \label{eq:809h9s94nbd}F^{k+1}\leq a_1F^k +a_2 F^{k-1}, \qquad \forall k\geq 1,\end{equation} where  $a_2 \geq 0 $,  $ a_1 + a_2 <1$ and at least one of the coefficients $a_1,a_2$ is positive. Then the sequence satisfies the relation
$ F^{k+1}\leq q^{k} (1+ \delta)  F^0$ for all $k\geq 1,$
where $q=\frac{a_1+\sqrt{a_1^2+4a_2}}{2}$  and $\delta=q-a_1\geq 0$. Moreover,
\begin{equation}\label{eq:98sh8hgbf93nd} q \geq a_1 + a_2,\end{equation}
with equality if and only if $a_2=0$ (in which case $q=a_1$ and $\delta=0$).
\end{lem}

\begin{proof}
Choose   $\delta = \frac{-a_1+\sqrt{a_1^2+4a_2}}{2}$. We claim $\delta \geq 0$ and $a_2 \leq (a_1+\delta)\delta$.
Indeed, non-negativity of $\delta$  follows from $a_2\geq 0$, while the second   relation follows from the fact that $\delta$ satisfies \begin{equation}\label{eq:sbuf78b38bc}(a_1+\delta)\delta - a_2 = 0.\end{equation} 

In view of these two relations,  adding $\delta F_k$  to both sides of \eqref{eq:809h9s94nbd},  we get 
\begin{equation}\label{eq:09h9sh9089ns}
F^{k+1} + \delta F^k  \leq   (a_1+\delta)F^k + a_2 F^{k-1}\\
\leq (a_1+\delta)(F^k + \delta F^{k-1}) = q(F^k+\delta F^{k-1}).
\end{equation}

Let us now argue that  $0<q<1$. Non-negativity of $q$ follows from non-negativity of $a_2$.  Clearly, as long as $a_2>0$, $q$ is positive. If $a_2=0$, then $a_1>0$ by assumption, which implies that $q$ is positive. The inequality $q<1$ follows directly from the assumption $a_1+a_2<1$.
By unrolling the recurrence \eqref{eq:09h9sh9089ns}, we obtain
$F^{k+1} \leq F^{k+1} + \delta F^k \leq q^k (F^1+ \delta F^0) = q^{k}(1+\delta) F^{0}.$

Finally, let us establish \eqref{eq:09h9sh9089ns}. Noting that $a_1 = q-\delta$, and since  in view of \eqref{eq:sbuf78b38bc} we have $a_2=q\delta$, we conclude that $a_1+a_2 = q + \delta(q-1) \leq q$, where the inequality follows from $q <1$.
\end{proof}

\blue{Finally, let us present a simple lemma of an identity that we use in our main proofs. This preliminary result is known to hold for the case of Euclidean norms ($\bB=\bI$). We provide the proof for the more general $\bB-$norm for completeness.}

\begin{lem}
\label{formulaWithB}
Let $a,b,c $ be arbitrary vectors in $\R^n$ and let $\bB$ be a positive definite matrix. Then the following identity holds: 
$2 \langle a-c,c-b \rangle_{\bB}=\|a-b\|^2_{\bB}-\|c-b\|^2_{\bB}-\|a-c\|^2_{\bB},$
\end{lem}
\begin{proof}
\begin{eqnarray*}
LHS &=& 2 \langle a-c,c-b \rangle_{\bB} = 2(a-c)^\top \bB (c-b)\notag\\
& =& 2a^\top \bB c-2a^\top \bB b-2c^\top \bB c+2c^\top \bB b
\end{eqnarray*}
and 
\begin{eqnarray*}
RHS &=& \|a-b\|^2_{\bB}-\|c-b\|^2_{\bB}-\|a-c\|^2_{\bB}\notag\\
& =& (a-b)^\top \bB (a-b)- (c-b)^\top \bB (c-b)-(a-c)^\top \bB (a-c)\notag\\
&=& a^\top\bB a-a^\top\bB b-b^\top\bB a+b^\top\bB b-c^\top\bB c+c^\top\bB b+b^\top\bB c-b^\top\bB b\notag\\
&& - a^\top\bB a+a^\top\bB c+c^\top\bB a-c^\top\bB c \notag\\
& =& 2a^\top \bB c-2a^\top \bB b-2c^\top \bB c+2c^\top \bB b
\end{eqnarray*}
LHS (left-hand side)=RHS (right-hand side) and this completes the proof.
\end{proof}

\subsection{Proof of Theorem~\ref{L2}} \label{app:1}
First, we decompose
\begin{eqnarray}
\|x^{k+1}-x^*\|^2_{\bB} & =& \|x^k-\omega \nabla f_{\bS_k}(x^k)+\beta(x^k-x^{k-1})-x^*\|^2_{\bB} \notag\\
& =& \underbrace{\|x^k-\omega \nabla f_{\bS_k}(x^k)-x^*\|^2_{\bB}}_{\encircle{1}} \notag\\
&& \quad +\underbrace{2\langle x^k-\omega \nabla f_{\bS_k}(x^k)-x^*, \beta(x^k-x^{k-1})  \rangle_{\bB}}_{\encircle{2}} \notag\\
&& \quad +\underbrace{\beta^2\|x^k-x^{k-1}\|^2_{\bB}}_{\encircle{3}}. \label{n0}
\end{eqnarray}
We will now analyze the three expressions \encircle{1}, \encircle{2}, \encircle{3} separately. The first expression can be written as
\begin{eqnarray}
\encircle{1} &=& \|x^k-x^*\|^2_{\bB}-2 \omega \langle x^k-x^*,\nabla f_{\bS_k}(x^k) \rangle_{\bB} +\omega^2 \|\nabla f_{\bS_k}(x^k)\|^2_{\bB} \notag\\
& \overset{\eqref{normbound},\eqref{functionequivalence}}{=} & \|x^k-x^*\|^2_{\bB}-4\omega f_{\bS_k}(x^k)+2\omega^2 f_{\bS_k}(x^k) \notag\\
&=& \|x^k-x^*\|^2_{\bB}-2\omega(2-\omega)f_{\bS_k}(x^k).\label{n1}
\end{eqnarray}
We will now bound the second expression. First, we have
\begin{equation}\label{eq:ibnodh90h}
\begin{aligned}
\encircle{2}
&= 2\beta \langle x^k-x^*, x^k-x^{k-1} \rangle_{\bB} +2\omega \beta \langle \nabla f_{\bS_k}(x^k), x^{k-1} - x^k  \rangle_{\bB}\\ 
&= 2\beta \langle x^k-x^*, x^k-x^{*} \rangle_{\bB} + 2\beta \langle x^k-x^*, x^*-x^{k-1} \rangle_{\bB} +2\omega \beta \langle \nabla f_{\bS_k}(x^k),x^{k-1}- x^k  \rangle_{\bB}\\ 
&= 2\beta \|x^k-x^*\|^2_{\bB}  + 2\beta \langle x^k-x^*, x^*-x^{k-1} \rangle_{\bB} +2\omega \beta \langle \nabla f_{\bS_k}(x^k),x^{k-1}- x^k  \rangle_{\bB}. 
\end{aligned}
\end{equation}
\blue{Using the identity from Lemma~\ref{formulaWithB} for the vectors $x^k, x^*$ and $x^{k-1}$ we obtain:}
$$2 \langle x^k-x^*, x^*-x^{k-1} \rangle_{\bB}=  \|x^k-x^{k-1}\|^2_{\bB}- \|x^{k-1}-x^*\|^2_{\bB}-\|x^k-x^*\|^2_{\bB}.$$
Substituting this into \eqref{eq:ibnodh90h} gives
\begin{equation}
\label{n2}
\begin{aligned}
\encircle{2}& = \beta \|x^k-x^*\|^2_{\bB}+\beta \|x^k-x^{k-1}\|^2_{\bB}-\beta \|x^{k-1}-x^*\|^2_{\bB} + 2\omega \beta \langle \nabla f_{\bS_k}(x^k),x^{k-1}- x^k  \rangle_{\bB}. 
\end{aligned}
\end{equation}

The third expression can be \blue{bounded} as
\begin{equation}
\label{n3}
\encircle{3} =\beta^2\|(x^{k}-x^*)+(x^*-x^{k-1})\|^2_{\bB}  \leq  2\beta^2\|x^{k}-x^*\|^2_{\bB}+2
\beta^2\|x^{k-1}-x^*\|^2_{\bB}.
\end{equation}

By substituting the  bounds \eqref{n1}, \eqref{n2}, \eqref{n3} into  \eqref{n0} we obtain
\begin{eqnarray*}
\|x^{k+1}-x^*\|^2_{\bB} 
& \leq & \|x^k-x^*\|^2_{\bB}-2\omega(2-\omega)f_{\bS_k}(x^k)\\
&& \quad + \beta \|x^k-x^*\|^2_{\bB}+\beta \|x^{k}-x^{k-1}\|^2_{\bB}-\beta \|x^{k-1}-x^*\|^2_{\bB} \\
&& \quad + 2\omega \beta \langle \nabla f_{\bS_k}(x^k),x^{k-1}- x^k  \rangle_{\bB}  +  2\beta^2\|x^{k}-x^*\|^2_{\bB}+2\beta^2\|x^{k-1}-x^*\|^2_{\bB}\\
& \leq & (1+3\beta + 2\beta^2)\|x^k-x^*\|^2_{\bB}+ (\beta + 2\beta^2 )\|x^{k-1}-x^*\|^2_{\bB}-2\omega(2-\omega)f_{\bS_k}(x^k)\\
&& \quad + 2\omega \beta \langle \nabla f_{\bS_k}(x^k),x^{k-1}- x^k  \rangle_{\bB}.
\end{eqnarray*}
Now by first taking expectation with respect to $\mS_k$, we obtain:
\begin{eqnarray*}
\Exp_{\mS_k}[\|x^{k+1}-x^*\|^2_{\bB}] & \leq & (1+3\beta+2\beta^2)\|x^k-x^*\|^2_{\bB}+ (\beta +2\beta^2)\|x^{k-1}-x^*\|^2_{\bB} \\
&& \quad -2\omega(2-\omega)f(x^k) + 2\omega \beta \langle \nabla f(x^k),x^{k-1}- x^k  \rangle_{\bB}\\
 & \leq & (1+3\beta+2\beta^2)\|x^k-x^*\|^2_{\bB}+ (\beta +2\beta^2)\|x^{k-1}-x^*\|^2_{\bB} \\
&& \quad -2\omega(2-\omega)f(x^k) + 2\omega \beta(f(x^{k-1})-f(x^k))\\
 & = & (1+3\beta+2\beta^2)\|x^k-x^*\|^2_{\bB}+ (\beta +2\beta^2)\|x^{k-1}-x^*\|^2_{\bB} \\
&& \quad - (2\omega(2-\omega) +2\omega\beta)f(x^k) + 2\omega \beta f(x^{k-1}).
\end{eqnarray*}
where in the second step we used the inequality
 $\langle \nabla f(x^k),x^{k-1}- x^k \rangle  \leq f(x^{k-1})-f(x^k)$ and the fact that $\omega \beta \geq 0$, which follows from the assumptions. We now apply  inequalities  \eqref{b2} and \eqref{b3}, obtaining

\begin{eqnarray*}
\Exp_{\mS_k}[\|x^{k+1}-x^*\|^2_{\bB}] & \leq &
 \underbrace{(1+3\beta+2\beta^2 - (\omega(2-\omega) +\omega\beta)\lambda_{\min}^+)}_{a_1}\|x^k-x^*\|^2_{\bB} \\
 && \quad + \underbrace{(\beta +2\beta^2 + \omega \beta \lambda_{\max})}_{a_2}\|x^{k-1}-x^*\|^2_{\bB}.
\end{eqnarray*}

By taking expectation again, and letting $F^k\eqdef \Exp[\|x^{k}-x^*\|^2_{\bB}]$, we get  the relation
\begin{equation}
\label{recur}
F^{k+1}  \leq a_1 F^k  + a_2 F^{k-1} .
\end{equation}

It suffices to apply  Lemma~\ref{LemmaGlobal} to the relation  \eqref{recur}. The conditions of the lemma are satisfied. Indeed, $a_2\geq 0$, and if $a_2=0$, then $\beta=0$ and hence $a_1=1-\omega(2-\omega)\lambda_{\min}^+>0$. The condition $a_1+a_2<1$ holds by assumption.

The convergence result in function values, $\Exp[f(x^k)]$, follows as a corollary by applying inequality \eqref{b2} to  \eqref{eq:nfiug582}.

\subsection{Proof of Theorem~\ref{cesaro}}
\label{app:acc212}

Let $p^t=\frac{\beta}{1-\beta}(x^t-x^{t-1})$ 
and $d^t = \|x^t + p^t -x^*\|_{\mB}^2$. In view of \eqref{eq:SHB-intro},  we can write 

\begin{eqnarray}
\label{casna}
x^{t+1}+p^{t+1} & =& x^{t+1}+\frac{\beta}{1-\beta}(x^{t+1}-x^{t}) \notag\\
&  \overset{\eqref{eq:SHB-intro}}{=} & x^{t}-\omega \nabla f_{\bS_t}(x^t)+\beta(x^t-x^{t-1})+\frac{\beta}{1-\beta}\left(-\omega \nabla f_{\bS_t}(x^t)+\beta(x^t-x^{t-1})\right) \notag\\
& =& x^{t}-\left[\omega +\frac{\beta}{1-\beta}\omega \right] \nabla f_{\bS_t}(x^t)+\left[\beta+\frac{\beta^2}{1-\beta}\right](x^t-x^{t-1})\notag\\
& =& x^{t}-\frac{\omega}{1-\beta}\nabla f_{\bS_t}(x^t)+\frac{\beta}{1-\beta}(x^t-x^{t-1})\notag\\
& =&  x^t+p^t-\frac{\omega}{1-\beta} \nabla f_{\bS_t}(x^t)
\end{eqnarray}

and therefore
\begin{eqnarray*}
d^{t+1} & \overset{\eqref{casna}}{=} &  \left\|x^t+p^t-\frac{\omega}{1-\beta} \nabla f_{\bS_t}(x^t) -x^* \right\|^2_{\bB} \  \\
& =& d^t -2 \frac{\omega}{1-\beta} \langle x^t+p^t-x^*,  \nabla f_{\bS_t}(x^t) \rangle_{\bB} + \frac{\omega^2}{(1-\beta)^2} \|\nabla f_{\bS_t}(x^t)\|^2_{\bB}\\
& =& d^t -\frac{2\omega}{1-\beta} \langle x^t-x^*,  \nabla f_{\bS_t}(x^t) \rangle_{\bB} - \frac{2 \omega \beta}{(1-\beta)^2}   \langle x^t-x^{t-1},  \nabla f_{\bS_t}(x^t) \rangle_{\bB}\\ 
& & \quad + \frac{\omega^2}{(1-\beta)^2} \|\nabla f_{\bS_t}(x^t)\|^2_{\bB}.
\end{eqnarray*}

Taking expectation with respect to the random matrix $\bS_t$ we obtain:
\begin{eqnarray*}
\Exp_{\mS_t}[d^{t+1}] & =& d^t -\frac{2\omega}{1-\beta} \langle x^t-x^*,  \nabla f(x^t) \rangle_{\bB} - \frac{2 \omega \beta}{(1-\beta)^2}   \langle x^t-x^{t-1},  \nabla f(x^t) \rangle_{\bB} \\
&& \quad + \frac{\omega^2}{(1-\beta)^2} 2 f(x^t) \notag\\
& \overset{\eqref{asnda}}{=} & d^t -\frac{4\omega}{1-\beta}  f(x^t) - \frac{2 \omega \beta}{(1-\beta)^2}    \langle x^t-x^{t-1},  \nabla f(x^t) \rangle_{\bB} + \frac{\omega^2}{(1-\beta)^2} 2 f(x^t)\\
& \leq & d^t -\frac{4\omega}{1-\beta}  f(x^t) - \frac{2 \omega \beta}{(1-\beta)^2}   [f(x^t)-f(x^{t-1})] + \frac{\omega^2}{(1-\beta)^2} 2 f(x^t)\\
& = & d^t + \left[ -\frac{4\omega}{1-\beta} - \frac{2 \omega \beta}{(1-\beta)^2} +\frac{2 \omega^2}{(1-\beta)^2}\right] f(x^t)  +  \frac{2 \omega \beta}{(1-\beta)^2} f(x^{t-1}),
\end{eqnarray*}
where the inequality follows from convexity of $f$.  After rearranging the terms we get
\[
\Exp_{\mS_t}[d^{t+1}] +   \frac{2 \omega \beta}{(1-\beta)^2} f(x^t) + \alpha f(x^t)  \leq d^t + \frac{2 \omega \beta}{(1-\beta)^2} f(x^{t-1}),
\]
where $\alpha =   \frac{4\omega}{1-\beta} -\frac{2 \omega^2}{(1-\beta)^2} > 0$. Taking expectations again and using the tower property, we get
\begin{equation}\label{eq:oih89hd8}
\theta^{t+1} + \alpha \Exp[f(x^t)]  \leq \theta^t, \qquad t=1,2,\dots,
\end{equation}
where $\theta^t = \Exp[d^t] + \frac{2 \omega \beta}{(1-\beta)^2}\Exp[ f(x^{t-1})]$. By summing up \eqref{eq:oih89hd8} for $t=1,\dots, k$ we get
\begin{equation}\label{eq:s098h89hffdss}\sum_{t=1}^k \Exp[f(x^t)] \leq \frac{\theta^1-\theta^{k-1}}{\alpha} \leq \frac{\theta^1}{\alpha}.\end{equation}
Finally, using Jensen's inequality, we get
\[\Exp[f(\hat{x}^k)] = \Exp \left[f\left(\frac{1}{k}\sum_{t=1}^k x^t\right)\right] \leq \Exp \left[\frac{1}{k}\sum_{t=1}^k f(x^t)\right] =  \frac{1}{k}\sum_{t=1}^k \Exp[f(x^t)] \overset{\eqref{eq:s098h89hffdss}}{\leq} \frac{\theta^1}{\alpha k}.\]

It remains to note that $\theta^1 = \|x^0-x^*\|_{\mB}^2 + \frac{2\omega \beta}{(1-\beta)^2 }f(x^0).$

\subsection{Proof of Theorem~\ref{theoremheavyball}} \label{app:acc}

In the proof of Theorem~\ref{theoremheavyball} the following two lemmas are used.

\begin{lem}[\cite{ASDA}]
\label{Forweakconvergence}
Assume exactness. Let $x\in \R^n$ and $x^* = \Pi_{\mathcal{L},\bB} (x)$. If $\lambda_i=0$, then $u_i^\top \bB^{1/2} (x-x^*)=0$.

\end{lem}

\begin{lem}[\cite{elaydi2005introduction, fillmore1968linear}]
\label{recurence}
Consider the second degree linear homogeneous recurrence relation:
\begin{equation}
\label{asdasdasd}
 r^{k+1}= a_1r^k+a_2 r^{k-1}
\end{equation}
with initial conditions $r^0,r^1 \in \R$.
Assume that the constant coefficients $a_1$ and $a_2$ satisfy the inequality $a_1^2 +4a_2<0$ (the roots of the characteristic equation $t^2-a_1t-a_2=0$ are imaginary). Then there are complex constants $C_0$ and $ C_1$ (depending on the initial conditions $r^0$ and $r^1$) such that:
$$r^k=2 M^k (C_0 \cos( \theta k) + C_1 \sin(\theta k))$$
where $M= \bigg(\sqrt{\frac{a_1^2}{4}+\frac{(-a_1^2-4a_2)}{4}} \bigg)=\sqrt{-a_2}$ and $\theta$ is such that $a_1=2 M \cos(\theta)$ and $\sqrt{-a_1^2-4a_2}=2 M \sin(\theta)$.
\end{lem}

We can now turn to the proof of Theorem~\ref{theoremheavyball}.
Plugging in the expression for the stochastic gradient, mSGD can be written in the form
\begin{eqnarray}
x^{k+1} & =& x^k -\omega \nabla f_{\bS_k}(x^k) + \beta(x^k - x^{k-1}) \notag \\
&\overset{\eqref{Gradf_S_IntroThesis}}{=}& x^k- \omega {\bB}^{-1} \bZ_k(x^k-x^*) + \beta(x^k - x^{k-1}).\label{difernt expad}
\end{eqnarray}

Subtracting  $x^*$ from both sides of \eqref{difernt expad}, we get
\begin{eqnarray*}
x^{k+1}-x^* 
& = & (\bI - \omega {\bB}^{-1} \bZ_k)(x^k-x^*) + \beta(x^k -x^* +x^* - x^{k-1})\\
& =& \left((1+\beta)\bI - \omega {\bB}^{-1} \bZ_k\right)(x^k-x^*) - \beta(x^{k-1}-x^*).
\end{eqnarray*}

Multiplying the last identity from the left by $\bB^{1/2}$, we get
\begin{eqnarray*}
\bB^{1/2} (x^{k+1}-x^*)
&=& \left((1+\beta)\bI - \omega \bB^{-1/2} \bZ_k \bB^{-1/2}\right) \bB^{1/2}(x^{k} -x^*) - \beta \bB^{1/2}(x^{k-1}-x^*).
\end{eqnarray*}

Taking expectations, conditioned on $x^k$ (that is, the expectation is with respect to $\bS_k$):
\begin{equation}
\label{eq:98g8gfsssd}
\bB^{1/2} \Exp[x^{k+1} -x^* \;|\; x^k]=\left((1+\beta)\bI - \omega \bB^{-1/2}  \Exp[\bZ]  \bB^{-1/2}\right) \bB^{1/2}(x^{k} -x^*) - \beta \bB^{1/2}(x^{k-1}-x^*) . 
\end{equation}

Taking expectations again, and using the tower property, we get
\begin{eqnarray*}
\bB^{1/2} \Exp[x^{k+1} -x^*] & = & \bB^{1/2}\Exp\left[\Exp[x^{k+1} -x^* \;|\; x^k]\right]\\
&\overset{\eqref{eq:98g8gfsssd}}{=}& \left((1+\beta)\bI - \omega \bB^{-1/2}  \Exp[\bZ] \bB^{-1/2} \right) \bB^{1/2} \Exp[x^{k} -x^*] - \beta \bB^{1/2} \Exp[x^{k-1}-x^*].
\end{eqnarray*}

Plugging the eigenvalue decomposition ${\bU}\bm{\Lambda} {{\bU}}^\top$ of the matrix $\bW=\bB^{-1/2}  \Exp[\bZ] \bB^{-1/2}$ into the above, and multiplying both sides from the left by ${{\bU}}^\top$, we obtain
\begin{equation}
\label{asdnaskdn}
{{\bU}}^\top \bB^{1/2} \Exp[x^{k+1} -x^*]  = {{\bU}}^\top \left((1+\beta)\bI - \omega {\bU}\bm{\Lambda} {{\bU}}^\top \right)\bB^{1/2} \Exp[x^{k} -x^*] - \beta {{\bU}}^\top \bB^{1/2} \Exp[x^{k-1}-x^*].
\end{equation}

Let us define $s^k\eqdef {{\bU}}^\top \bB^{1/2} \Exp[x^{k} -x^*]  \in \R^n$. Then relation \eqref{asdnaskdn} takes the form of the recursion
$$ s^{k+1}= [(1+\beta)\mI - \omega \bm{\Lambda} ] s^k - \beta s^{k-1},$$
which can be written in a coordinate-by-coordinate form as follows:
\begin{equation}
\label{coordinate}
 s^{k+1}_i= [(1+\beta) - \omega \lambda_i ] s^k_i - \beta s^{k-1}_i   \quad \text{for all} \quad i= 1,2,3,...,n,
\end{equation}
where $s^k_i$ indicates the $i$th coordinate of $s^k$.

We will now fix $i$ and analyze recursion \eqref{coordinate} using Lemma~\ref{recurence}. Note that \eqref{coordinate} is a second degree linear homogeneous recurrence relation of the form \eqref{asdasdasd} with $a_1=1+\beta - \omega \lambda_i $ and $a_2=- \beta$. Recall  that $0\leq\lambda_i \leq1$ for all $i$. Since we assume that $0< \omega \leq 1/\lambda_{\max}$, we know that $0\leq \omega \lambda_i \leq 1$ for all $i$. We now consider two  cases:

\begin{enumerate}
\item $ \lambda_i =0$.   

In this case, \eqref{coordinate} takes the form: \begin{equation} \label{eq:9g898fg9sy}s^{k+1}_i=(1+\beta)s^k_i-\beta s^{k-1}_i .\end{equation} Applying Proposition~\ref{prop:projections}, we know that $x^*=\Pi_{\mathcal{L},\bB}(x^0)=\Pi_{\mathcal{L},\bB}(x^1)$. Using Lemma \ref{Forweakconvergence} twice, once for $x=x^0$ and then for $x=x^1$, we observe that $s^0_i=u_i^\top \bB^{1/2} (x^0-x^*)=0$ and $s^1_i=u_i^\top \bB^{1/2} (x^1-x^*)=0$. Finally, in view of \eqref{eq:9g898fg9sy} we conclude that \begin{equation}\label{eq:8g98gdu9hhOOh} s^k_i=0 \quad \text{for all} \quad k\geq 0 .\end{equation}

\item $ \lambda_i >0$. 

Since $0<\omega \lambda_i \leq 1$ and $\beta\geq 0$, we have $1+\beta - \omega \lambda_i \geq 0$ and hence
\[a_1^2+4a_2=(1+\beta -\omega \lambda_i)^2-4\beta \leq  (1+\beta -\omega \lambda_{\min}^+)^2-4\beta  < 0,\]
where the last inequality can be shown to hold\footnote{The lower bound on $\beta$ is tight. However,  the upper bound is not. However, we do not care much about the regime of large $\beta$ as $\beta$ is the convergence rate, and hence is only interesting if smaller than 1.} for  $\left(1-\sqrt{\omega \lambda_{\min}^+}\right)^2 < \beta < 1 $. Applying Lemma~\ref{recurence} the following bound can be deduced
\begin{eqnarray}
s^k_i  &=& 2(-a_2)^{k/2} (C_0  \cos(\theta k) +C_1 \sin(\theta k)) \; \leq \; 2 \beta^{k/2} P_i, \label{eq:uibd880s-pO}
\end{eqnarray}
where $P_i$ is a constant depending  on the initial conditions (we can simply choose $P_i = |C_0| + |C_1|$).

\end{enumerate}

Now putting the two cases together, for all $k\geq0$ we have
\begin{eqnarray*}
\|\Exp[x^{k} -x^*]\|_{\bB}^2&=& \Exp[x^{k} -x^*]^\top \bB \Exp[x^{k} -x^*] \; = \; \Exp[x^{k} -x^*] \bB^{1/2} \bU {\bU}^\top \bB^{1/2} \Exp[x^{k} -x^*]  \\
&=& \|{\bU}^\top \bB^{1/2} \Exp[x^{k} -x^*] \|^2 \; = \; \|s^k\|^2 \;= \; \sum_{i=1}^{n} (s^k_i)^2 \\
&= &  \sum_{i: \lambda_i=0} (s^k_i)^2  +  \sum_{i:  \lambda_i >0} (s^k_i)^2 \; \overset{\eqref{eq:8g98gdu9hhOOh}}{=}\;    \sum_{i:  \lambda_i >0} (s^k_i)^2\\
& \overset{\eqref{eq:uibd880s-pO}}{\leq} & \sum_{i:  \lambda_i >0} 4 \beta ^k P_i^2 \\
&=& \beta^k C,
\end{eqnarray*}
where $C=4\sum_{i:  \lambda_i >0}  P_i^2$.

\subsection{Proof of Theorem~\ref{thm:DSHB-L2} } \label{app:7}

The proof follows a similar pattern to that of Theorem~\ref{L2}. However, stochasticity in the momentum term introduces an additional layer of complexity, which we shall tackle by utilizing   a more involved version of the tower property.

For simplicity, let $i=i_k$ and $r^{k}_i  \eqdef e_i^\top(x^k-x^{k-1})e_i$. First, we decompose
\begin{eqnarray}
\|x^{k+1}-x^*\|^2 
& =& \|x^k-\omega \nabla f_{\bS_k}(x^k)+\gamma r^k_i -x^*\|^2 \notag\\
& =& \|x^k-\omega \nabla f_{\bS_k}(x^k)-x^*\|^2+2\langle x^k-\omega \nabla f_{\bS_k}(x^k)-x^*, \gamma r^k_i  \rangle + \gamma^2\| r^k_i\|^2. \label{eq:iugh98d894}
\end{eqnarray}
 We shall use the tower property in the form
\begin{equation}\label{eq:tower3SHB}\Exp[\Exp[\Exp[ X  \;|\; x^k,  \mS_k] \;|\; x^k]] = \Exp[X],\end{equation}
where $X$ is some random variable. We shall perform the three expectations in order, from the innermost to the outermost.  Applying the inner expectation to the identity \eqref{eq:iugh98d894}, we get
\begin{eqnarray}\Exp[\|x^{k+1}-x^*\|^2 \;|\; x^k, \mS_k ] 
&=&
 \underbrace{\Exp[\|x^k-\omega \nabla f_{\bS_k}(x^k)-x^*\|^2\;|\; x^k, \mS_k]}_{\encircle{1}} \notag\\
 && \quad +\underbrace{\Exp[2\langle x^k-\omega \nabla f_{\bS_k}(x^k)-x^*, \gamma r^k_i  \rangle \;|\; x^k, \mS_k]}_{\encircle{2}} \notag\\
&& \quad +\underbrace{\Exp[\gamma^2\| r^k_i\|^2\;|\; x^k, \mS_k]}_{\encircle{3}}. \label{eq:098j}
\end{eqnarray}

We will now analyze the three expressions \encircle{1}, \encircle{2}, \encircle{3} separately. The first expression is constant under the expectation, and hence we can write
\begin{eqnarray}
\encircle{1} &=& \|x^k-\omega \nabla f_{\bS_k}(x^k)-x^*\|^2\notag\\
&=& \|x^k-x^*\|^2-2 \omega \langle x^k-x^*,\nabla f_{\bS_k}(x^k) \rangle +\omega^2 \|\nabla f_{\bS_k}(x^k)\|^2\notag\\
& \overset{\eqref{normbound} + \eqref{functionequivalence}}{=} & \|x^k-x^*\|^2-4\omega f_{\bS_k}(x^k)+2\omega^2 f_{\bS_k}(x^k)\notag\\
&=& \|x^k-x^*\|^2-2\omega(2-\omega)f_{\bS_k}(x^k). \label{n1X}
\end{eqnarray}

We will now bound the second expression. Using the identity
\begin{equation}\label{eq:98gf8g8e09}\Exp[r^k_i  \;|\; x^k, \mS_k] = \Exp_i [r^k_i] = \sum_{i=1}^n \frac{1}{n}r^k_i = \frac{1}{n}(x^k-x^{k-1}),\end{equation}
we can write
\begin{eqnarray}
\encircle{2}
&=& \Exp[2\langle x^k-\omega \nabla f_{\bS_k}(x^k)-x^*, \gamma r^k_i  \rangle \;|\; x^k, \mS_k]\notag\\
&=& 2\langle x^k-\omega \nabla f_{\bS_k}(x^k)-x^*, \gamma \Exp[r^k_i  \;|\; x^k, \mS_k] \rangle\notag\\
&\overset{\eqref{eq:98gf8g8e09}}{=}& 2\langle x^k-\omega \nabla f_{\bS_k}(x^k)-x^*, \frac{\gamma}{n} (x^k - x^{k-1})  \rangle\notag \\
& =&
2\frac{\gamma}{n} \langle x^k-x^*, x^k-x^{k-1} \rangle +2\omega \frac{\gamma}{n} \langle \nabla f_{\bS_k}(x^k), x^{k-1} - x^k  \rangle\notag\\ 
&=& 2\frac{\gamma}{n} \langle x^k-x^*, x^k-x^{*} \rangle + 2\frac{\gamma}{n} \langle x^k-x^*, x^*-x^{k-1} \rangle +2\omega \frac{\gamma}{n} \langle \nabla f_{\bS_k}(x^k),x^{k-1}- x^k  \rangle \notag\\ 
&=& 2\frac{\gamma}{n} \|x^k-x^*\|^2 + 2\frac{\gamma}{n} \langle x^k-x^*, x^*-x^{k-1} \rangle +2\omega \frac{\gamma}{n} \langle \nabla f_{\bS_k}(x^k),x^{k-1}- x^k  \rangle. \label{eq:ibnodh90hX}
\end{eqnarray}
Using the fact that for arbitrary vectors $a,b,c \in \R^n$ we have the identity
$2 \langle a-c,c-b \rangle =\|a-b\|^2-\|c-b\|^2-\|a-c\|^2,$
we obtain
$$2 \langle x^k-x^*, x^*-x^{k-1} \rangle=  \|x^k-x^{k-1}\|^2- \|x^{k-1}-x^*\|^2-\|x^k-x^*\|^2.$$
Substituting this into \eqref{eq:ibnodh90hX} gives
\begin{equation}
\label{n2X}
\begin{aligned}
\encircle{2}& = \frac{\gamma}{n} \|x^k-x^*\|^2+\frac{\gamma}{n}\|x^k-x^{k-1}\|^2-\frac{\gamma}{n} \|x^{k-1}-x^*\|^2 + 2\omega \frac{\gamma}{n} \langle \nabla f_{\bS_k}(x^k),x^{k-1}- x^k  \rangle. 
\end{aligned}
\end{equation}

The third expression can be bound as
\begin{eqnarray}
\encircle{3} &=& \Exp[\gamma^2\| r^k_i\|^2\;|\; x^k, \mS_k] \notag\\
&=& \gamma^2 \Exp_i[\| r^k_i\|^2] \notag \\
&=& \gamma^2 \sum_{i=1}^n \frac{1}{n} (x^k_i -x^{k-1}_i)^2 \notag\\
&=& \frac{\gamma^2}{n} \|x^k-x^{k-1}\|^2 \notag \\ 
&=&\frac{\gamma^2}{n}\|(x^{k}-x^*)+(x^*-x^{k-1})\|^2 \notag \\
& \leq &  \frac{2\gamma^2}{n}\|x^{k}-x^*\|^2+ \frac{2\gamma^2}{n}\|x^{k-1}-x^*\|^2. \label{n3X}
\end{eqnarray}

By substituting the  bounds \eqref{n1X}, \eqref{n2X}, \eqref{n3X} into  \eqref{eq:098j} we obtain
\begin{eqnarray}
\Exp[\|x^{k+1}-x^*\|^2 \;|\; x^k, \mS_k ] 
& \leq & \|x^k-x^*\|^2-2\omega(2-\omega) f_{\bS_k}(x^k)\notag\\
&& \quad + \frac{\gamma}{n}  \|x^k-x^*\|^2+ \frac{\gamma}{n} \|x^{k}-x^{k-1}\|^2 -\frac{\gamma}{n}  \|x^{k-1}-x^*\|^2 \notag\\
&& \quad + 2\omega \frac{\gamma}{n}  \langle \nabla f_{\bS_k}(x^k), x^{k-1}- x^k  \rangle  +  2\frac{\gamma^2}{n} \|x^{k}-x^*\|^2 \\
&& \quad + 2\frac{\gamma^2}{n}\|x^{k-1}-x^*\|^2 \notag\\
&\overset{\eqref{n3}}{\leq}& \left(1+3\frac{\gamma}{n} + 2\frac{\gamma^2}{n}\right)\|x^k-x^*\|^2+ \left(\frac{\gamma}{n} + 2\frac{\gamma^2}{n} \right) \|x^{k-1}-x^*\|^2 \notag\\
&& \quad - 2\omega(2-\omega)f_{\bS_k}(x^k) + 2\omega \frac{\gamma}{n} \langle \nabla f_{\bS_k}(x^k),x^{k-1}- x^k  \rangle.\label{eq:iohih638ygbdd}
\end{eqnarray}
We now take the middle expectation (see \eqref{eq:tower3SHB}) and apply it to inequality \eqref{eq:iohih638ygbdd}:
\begin{eqnarray*}
\Exp[\Exp[\|x^{k+1}-x^*\|^2 \;|\; x^k, \mS_k ] \;|\; x^k]  &\leq& \left(1+3\frac{\gamma}{n} + 2\frac{\gamma^2}{n}\right)\|x^k-x^*\|^2+ \left(\frac{\gamma}{n} + 2\frac{\gamma^2}{n} \right) \|x^{k-1}-x^*\|^2 \\
&& \quad -2\omega(2-\omega)f(x^k) + 2\omega \frac{\gamma}{n} \langle \nabla f(x^k),x^{k-1}- x^k  \rangle\\
&\leq& \left(1+3\frac{\gamma}{n} + 2\frac{\gamma^2}{n}\right)\|x^k-x^*\|^2+ \left(\frac{\gamma}{n} + 2\frac{\gamma^2}{n} \right) \|x^{k-1}-x^*\|^2 \\
&& \quad -2\omega(2-\omega)f(x^k) + 2\omega \frac{\gamma}{n}(f(x^{k-1})-f(x^k))\\
&=&  \left(1+3\frac{\gamma}{n} + 2\frac{\gamma^2}{n}\right)\|x^k-x^*\|^2+ \left(\frac{\gamma}{n} + 2\frac{\gamma^2}{n} \right) \|x^{k-1}-x^*\|^2 \\
&& \quad - \left(2\omega(2-\omega) +2\omega\frac{\gamma}{n}\right)f(x^k) + 2\omega \frac{\gamma}{n} f(x^{k-1}).
\end{eqnarray*}
where in the second step we used the inequality
 $\langle \nabla f(x^k),x^{k-1}- x^k \rangle  \leq f(x^{k-1})-f(x^k)$ and the fact that $\omega \gamma \geq 0$, which follows from the assumptions. We now apply  inequalities  \eqref{b2} and \eqref{b3}, obtaining
\begin{eqnarray*}
\Exp[\Exp[\|x^{k+1}-x^*\|^2 \;|\; x^k, \mS_k ] \;|\; x^k]& \leq &
 \underbrace{\left(1+3\frac{\gamma}{n}+2\frac{\gamma^2}{n} - \left(\omega(2-\omega) +\omega\frac{\gamma}{n}\right)\lambda_{\min}^+ \right)}_{a_1}\|x^k-x^*\|^2 \\
 && \quad + \underbrace{\frac{1}{n}\left(\gamma +2\gamma^2 + \omega \gamma \lambda_{\max}\right)}_{a_2}\|x^{k-1}-x^*\|^2.
\end{eqnarray*}

By taking expectation again (outermost expectation in the tower rule \eqref{eq:tower3SHB}), and letting $F^k\eqdef \Exp[\|x^{k}-x^*\|^2_{\bB}]$, we get  the relation
\begin{equation}
\label{recurX}
F^{k+1}  \leq a_1 F^k  + a_2 F^{k-1} .
\end{equation}

It suffices to apply  Lemma~\ref{LemmaGlobal} to the relation  \eqref{recur}. The conditions of the lemma are satisfied. Indeed, $a_2\geq 0$, and if $a_2=0$, then $\gamma=0$ and hence $a_1=1-\omega(1-\omega)\lambda_{\min}^+>0$. The condition $a_1+a_2<1$ holds by assumption.

The convergence result in function values follows as a corollary by applying inequality \eqref{b2} to  \eqref{eq:nfiug582X}.

\chapter{Inexact Randomized Iterative Methods}
\label{ChapterInexact}

\section{Introduction}
In the era of big data where data sets become continuously larger, randomized iterative methods become very popular and are increasingly playing a major role in areas such as numerical linear algebra, scientific computing and optimization. They are preferred mainly because of their cheap per-iteration cost which leads to improvements in terms of complexity upon classical results by orders of magnitude. In addition they can easily scale to extreme dimensions.
However, a common feature of these methods is that in their update rule a particular subproblem needs to be solved \emph{exactly}. In a large scale setting, often this step is computationally very expensive. The purpose of this work is to reduce the cost of this step by allowing for \emph{inexact updates} in the stochastic methods under study.

\subsection{The setting}
In this chapter we are interested to solve the three closely related problems described in the previous chapters. As a reminder, these are:
\begin{enumerate}
\item[(i)] stochastic quadratic optimization \eqref{StochReform_IntroThesis}, 
\item[(ii)] best approximation \eqref{BestApproximation_IntroThesis}, and 
\item[(iii)] (bounded) concave quadratic maximization \eqref{DualProblem_IntroThesis}.
\end{enumerate} 

In particular we propose and analyze {\em inexact} variants of the exact algorithms presented in the introduction of this thesis for solving the above problems. Among the methods studied are: stochastic gradient descent (SGD), stochastic Newton (SN), stochastic proximal point (SPP),  sketch and project method (SPM) and stochastic subspace ascent (SDSA). 
In all of these methods, a certain potentially expensive calculation/operation needs to be performed in each step; it is this operation that we propose to be performed \emph{inexactly}.  For instance, in the case of SGD, it is the computation of the stochastic gradient $\nabla f_{\mS_k}(x^k)$, in the case of  SPM is the computation of the projection $\Pi_{\cL_{\mS}, \bB}(x^k)$, and in the case of SDSA it is the computation of the dual update $\mS_k \lambda^k$.

We perform an iteration complexity analysis under an abstract notion of inexactness and also under a more structured form of inexactness appearing in practical scenarios. An inexact solution of these subproblems can be obtained much more quickly than the exact solution. Since in practical applications the savings thus obtained are larger than the increase in the number of iterations needed for convergence, our inexact methods can be dramatically faster.

\subsection{Structure of the chapter and main contributions}

Let us now outline the main contribution and the structure of this chapter.

In Section~\ref{SecondSection} we describe the subproblems and introduce two notions of inexactness (abstract and structured) that will be used in the rest of this chapter. The Inexact Basic Method (iBasic) is also presented. iBasic is a method that simultaneously captures inexact variants of the algorithms \eqref{SGD_IntroThesis}, \eqref{SNM_IntroThesis}, \eqref{SPPM_IntroThesis} for solving the stochastic optimization problem \eqref{StochReform_IntroThesis} and algorithm \eqref{SPM_IntroThesis} for solving the best approximation problem \eqref{BestApproximation_IntroThesis}. It is an inexact variant of the \emph{Basic Method}, first presented in \cite{ASDA}, where the inexactness is introduced by the addition of an inexactness error $\epsilon^k$ in the original update rule. We illustrate the generality of  iBasic  by presenting popular algorithms that can be cast as special cases. 

In Section~\ref{gental assumption} we establish convergence results of iBasic under general assumptions on the inexactness error $\epsilon^k$ of its update rule (see Algorithm~\ref{inexact_basic}). In this part we do not focus on any specific mechanisms which lead to inexactness; we treat the problem abstractly. However, such errors appear often in practical scenarios and can be associated with inaccurate numerical solvers, quantization,  sparsification and compression mechanisms. In particular, we introduce several abstract assumptions on the inexactness level  and describe our generic convergence results. For all assumptions we establish linear rate of decay of the quantity $\Exp[\|x^{k}-x^*\|_{\mB}^2]$ (i.e. L2 convergence)\footnote{As we explain later, a convergence of the expected function values of problem \ref{StochReform_IntroThesis} can be easily obtained as a corollary of L2 convergence.}. 

Subsequently, in Section~\ref{InexactSolvers} we apply our general convergence results to a more structured notion of inexactness error and  propose a concrete mechanisms leading to such errors.  We provide theoretical guarantees for this method  in situations when  a linearly convergent iterative method (e.g., Conjugate Gradient)  is used to solve the subproblem inexactly. We also highlight the importance of the dual viewpoint through a sketch-and-project interpretation.

In Section~\ref{InexactDualMethods} we study an inexact variant of SDSA, which we called iSDSA, for directly solving the dual problem \eqref{DualProblem_IntroThesis}.  We provide a correspondence between  iBasic and iSDSA and we show that the random iterates of iBasic arise as affine images of iSDSA. We consider both abstract and structured inexactness errors and provide linearly convergent rates in terms of the dual function suboptimality  $\E{D(y^*) - D(y^0)}$.

Finally, in Section~\ref{experiments} we evaluate the performance of the proposed inexact methods through numerical experiments and show the benefits of our approach on both synthetic and real datasets. Concluding remarks are given in Section~\ref{conlcusion}. 

A summary of the convergence results of iBasic under several assumptions on the inexactness error  with pointers to the relevant theorems is available in Table~\ref{OurResultsInexact}. We highlight that similar convergence results can be also obtained for  iSDSA in terms of the dual function suboptimality  $\E{D(y^*) - D(y^0)}$ (check Section~\ref{InexactDualMethods} for more details on iSDSA).

\begin{table}[t!]
\begin{center}
\scalebox{0.85}{
\begin{tabular}{ |c|c|c|c| }
 \hline
 \begin{tabular}{c} Assumption on \\ the Inexactness error $\epsilon^k$  \end{tabular}  & $\omega$& \begin{tabular}{c}Upper Bounds\end{tabular} & Theorem \\
 \hline
 \hline
Assumption \ref{Assumption1} &  $(0,2)$ & $\rho^{k/2} \|x^{0}-x^*\|_{\mB} +  \sum_{i=0}^{k-1} \rho^{\frac{k-1-i}{2}}\sigma_i$ &\ref{InexactSGDConstant}\\
 \hline
Assumption \ref{Assumption3} &  $(0,2)$ & $\left(\sqrt{\rho}+q\right)^{2k} \|x^0-x^*\|_{\mB}^2$ &\ref{ISGDwithq}\\
 \hline
Assumptions \ref{Assumption2},\ref{Assumption5} & $(0,2)$  &  $\rho^{k} \|x^{0}-x^*\|^2_{\mB} +  \sum_{i=0}^{k-1} \rho^{k-1-i}{\bar{\sigma}}^2_i$ & \ref{InexactSGDrandom}(i) \\
 \hline 
 Assumptions \ref{Assumption3},\ref{Assumption5} & $(0,2)$  &  $\left(\rho + q^2 \right)^k \|x^0-x^*\|_{\mB}^2$ & \ref{InexactSGDrandom}(ii) \\
 \hline 
Assumptions \ref{Assumption4},\ref{Assumption5} & $(0,2)$  &  $\left(\rho+q^2 \lambda_{\min}^+ \right)^k \|x^0-x^*\|_{\mB}^2$ & \ref{InexactSGDrandom}(iii) \\
 \hline 

\end{tabular}}
\end{center}
\caption{\small Summary of the iteration complexity results obtained in this chapter. $\omega$ denotes the stepsize (relaxation parameter) of the method. In all cases, $x^*=\Pi_{\cL, \bB}(x^0)$ and $\rho=1- \omega (2-\omega)\lambda_{\min}^+ \in (0,1)$ are the quantities appear in the convergence results (here $\lambda_{\min}^+$ denotes the minimum non zero eigenvalue of matrix $\bW$, see equation \eqref{matrixW}). Inexactness parameter $q$ is chosen always in such a way to obtain linear convergence and it can be seen as the quantity that controls the inexactness. In all theorems the quantity of convergence is $\Exp[\|x^k-x^*\|^2_{\bB}]$ (except in Theorem~\ref{InexactSGDConstant} where we analyze $\Exp[\|x^k-x^*\|_{\bB}])$. As we show in Section~\ref{InexactDualMethods}, under similar assumptions, iSDSA has exactly the same convergence with iBasic but the upper bounds of the third column are related to the dual function values $\E{D(y^*) - D(y^0)}$. }
\label{OurResultsInexact}
\end{table}

\subsection{Notation}
Following the rest of this thesis, with boldface upper-case letters we denote matrices and $\bI$ is the identity matrix. By $\cL$ we denote the solution set of the linear system $\bA x=b$. By $\cL_{\bS}$, where $\mS$ is a random matrix, we denote the solution set of the {\em sketched} linear system $\bS^\top \bA x= \bS^\top b$.  In general, we use $*$ to express the exact solution of a sub-problem and $\approx$ to indicate its inexact variant. 
Unless stated otherwise, throughout the chapter, $x^*$ is the projection of $x^0$ onto $\cL$ in the $\mB$-norm: $x^*=\Pi_{\cL, \bB}(x^0)$. 

\section{Inexact Update Rules}
\label{SecondSection}
In this section we start by explaining the key sub-problems that need to be solved exactly in the update rules of the previously described methods. We present iBasic, a method that solves problems \eqref{StochReform_IntroThesis} and \eqref{BestApproximation_IntroThesis} and we show how by varying the main parameters of the method we recover inexact variants of popular algorithms as special cases. Finally closely related work on inexact algorithms for solving different problems is also presented.
\subsection{Expensive sub-problems in update rules}
\label{subproblems}
Let us devote this subsection on explaining how the inexactness can be introduced in the current exact update rules of SGD\footnote{Note that SGD has identical updates to the Stochastic Newton and Stochastic proximal point method. Thus the inexactness can be added to these updates in similar way.} \eqref{SGD_IntroThesis}, Sketch and Project \eqref{SPM_IntroThesis} and SDSA \eqref{SDSA_IntroThesis} for solving the stochastic optimization, best approximation  and the dual problem respectively. As we have shown these methods solve closely related problems and the key subproblems in their update rule are similar. However the introduction of inexactness in the update rule of each one of them can have different interpretation. 

For example for the case of SGD for solving the stochastic optimization problem \eqref{StochReform_IntroThesis} (see also Section~\ref{sectionlinesyst} and \ref{SFPinterpretation} for more details), if we define $\lambda^k_*=(\mS_{k}^\top \mA \mB^{-1} \mA^\top \mS_{k})^{\dagger} \mS_{k}^\top(b-\mA x^k) $  then the stochastic gradient of function $f$ becomes  $\nabla f_{\mS_k}(x^k)\overset{\eqref{Gradf_S_IntroThesis}}{=}-\bB^{-1}\bA ^\top \bS_k \lambda^k_*$ and the update rule of SGD takes the form: $x^{k+1} = x^k+\omega \mB^{-1}\mA^\top \mS_{k} \lambda^k_*$.  Clearly in this update the expensive part is the computation of the quantity $\lambda^k_*$ that can be equivalently computed to be the least norm solution of the smaller (in comparison to $\bA x=b$) linear system $\mS_{k}^\top \mA \mB^{-1} \mA^\top \mS_{k} \lambda =\mS_{k}^\top (b-\mA x^k)$. In our work we are suggesting to use an approximation $\lambda^k_{\approx}$ of the exact solution and with this way avoid executing the possibly expensive step of the update rule.  Thus the inexact update is taking the following form:
$$x^{k+1} =x^k+\omega \mB^{-1}\mA^\top \mS_{k} \lambda^k_{\approx}= x^k - \omega \nabla f_{\mS_k}(x^k)+\underbrace{\omega \bB^{-1}\bA^\top \bS_k (\lambda^k_{\approx}-\lambda^k_*)}_{\epsilon^k}.$$
Here $\epsilon^k$ denotes a more abstract notion of inexactness and it is not necessary to be always equivalent to the quantity $\omega \bB^{-1}\bA^\top \bS_k(\lambda^k_{\approx}-\lambda^k_*)$. It can be interpreted as an expression that acts as an perturbation of the exact update. In the case that $\epsilon^k$ has the above form we say that the notion of inexactness is structured.
In our work we are interested in both the \emph{abstract} and more \emph{structured} notions of inexactness. We first present general convergence results where we require the error $\epsilon^k$ to satisfy general assumptions (without caring how this error is generated) and later we analyze the concept of structured inexactness by presenting algorithms where $\epsilon^k= \omega \bB^{-1}\bA^\top \bS_k(\lambda^k_{\approx}-\lambda^k_*)$.

In similar way, the expensive operation of SPM \eqref{SPM_IntroThesis} is the exact computation of the projection $\Pi_{\cL_{\mS_k},\mB}^*(x^k)$. Thus we are suggesting to replace this step with an inexact variant and compute an approximation of this projection. The inexactness here can be also interpreted using both, the abstract $\epsilon^k$ error and its more structured version $\epsilon^k=\omega \left( \Pi_{\cL_{\mS_k},\mB}^{\approx}(x^k)- \Pi_{\cL_{\mS_k},\mB}^*(x^k) \right)$. At this point, observe that, by using the expression \eqref{Projection_IntroThesis} the structure of the $\epsilon^k$ in SPM and SGD has the same form.

In the SDSA the expensive subproblem in the update rule is the computation of the $\lambda^k_*$ that satisfy $\lambda^k_* \in \arg\max_\lambda D(y^k + \mS_k \lambda)$. Using the definition of the dual function \eqref{DualProblem_IntroThesis} this value can be also computed by evaluating the least norm solution of the linear system $ \mS_{k}^\top \mA \mB^{-1} \mA^\top \mS_{k} \lambda  =\mS_{k}^\top \left(b-\bA(x^0 + \bB^{-1}\bA^\top y^k \right))$. Later in Section~\ref{InexactDualMethods} we analyze both notions of inexactness (abstract and more structured) for inexact variants of SDSA.

Table~\ref{KeySubproblems} presents the key sub-problem that needs to be solved in each algorithm as well as the part where the inexact error is appeared in the update rule.  

\begin{table}[t!]
\begin{center}
\scalebox{0.6}{
\begin{tabular}{ |c|c|c| }
 \hline
Exact Algorithms &\begin{tabular}{c} Key Subproblem \\ (problem that we solve inexactly) \end{tabular}  & \begin{tabular}{c} Inexact Update Rules \\ (abstract and structured inexactness error) \end{tabular} \\
 \hline
 \hline
 SGD \eqref{SGD_IntroThesis} & \begin{tabular}{c}Exact computation of $\lambda^k_*$, \\where $ \lambda^k_*=\arg\min_{ \lambda : \bM_k  \lambda =d_k} \| \lambda\|$. \\ Appears in the computation of $\nabla f_{\mS_k}(x^k)=-\bB^{-1}\bA ^\top \bS_k \lambda^k_*$  \end{tabular}&\begin{tabular}{c} \\$x^{k+1} = x^k+\omega \mB^{-1}\mA^\top \mS_{k} \lambda^k_{\approx}$\\$\quad = x^k - \omega \nabla f_{\mS_k}(x^k)+\underbrace{\omega \bB^{-1}\bA^\top \bS_k (\lambda^k_{\approx}-\lambda^k_*)}_{\epsilon^k}.$\end{tabular}\\
 \hline
SPM \eqref{SPM_IntroThesis} & \begin{tabular}{c}Exact computation of the projection \\$ \Pi_{\cL_{\mS_k},\mB}^*(x^k)=\arg\min_{x' \in \cL_{\mS_k}} \|x'-x^k\|_{\bB} $ \\ \end{tabular}& \begin{tabular}{c}\\$x^{k+1} =  \omega \Pi_{\cL_{\mS_k},\mB}^{\approx}(x^k)+ (1-\omega) x^k$\\ $\quad=  \omega \Pi_{\cL_{\mS_k},\mB}(x^k) + (1-\omega) x^k+ \underbrace{\omega \left(  \Pi_{\cL_{\mS_k},\mB}^{\approx}(x^k)- \Pi_{\cL_{\mS_k},\mB}^*(x^k) \right)}_{\epsilon^k} $\end{tabular}\\
 \hline
SDSA \eqref{SDSA_IntroThesis} & \begin{tabular}{c}Exact computation of $\lambda^k_*$, \\ where $\lambda^k_* \in \arg\max_\lambda D(y^k + \mS_k \lambda)$. \end{tabular}& \begin{tabular}{c}\\$y^{k+1} = y^k + \omega  \mS_k \lambda^k_{\approx}=y^k + \omega  \mS_k \lambda^k_*+\underbrace{\omega \bS_k (\lambda^k_{\approx} - \lambda^k_*) }_{\epsilon^k_d}$ \end{tabular} \\
 \hline 
\end{tabular}}
\end{center}
\caption{\small The exact algorithms under study with the potentially expensive to compute key sub-problems of their update rule. The inexact update rules are presented in the last column for both notions of inexactness (abstract and more structured). We use $*$ to define the important quantity that needs to be computed exactly in the update rule of each method and $\approx$ to indicate the proposed inexact variant.}
\label{KeySubproblems}
\end{table}

\subsection{The inexact basic method}
In each iteration of the all aforementioned exact methods a sketch matrix $\bS \sim {\cal D}$ is drawn from a given distribution and then a certain subproblem is solved exactly to obtain the next iterate. The sketch matrix $\bS \in \R^{m \times q}$ requires to have $m$ rows but no assumption on the number of columns is made which means that the number of columns $q$ allows to vary through the iterations and it can be very large.  The setting that we are interested in is precisely that of having such large random matrices $\bS$. In these cases we expect that having approximate solutions of the subproblems will be beneficial.

Recently randomized iterative algorithms that requires to solve large subproblems in each iteration have been extensively studied and it was shown that are really beneficial when they compared to their single coordinates variants ($\bS \in \R^{m \times 1}$) \cite{RBK,l2015randomized,richtarik2014iteration,LoizouRichtarik}. However, in theses cases the evaluation of an exact solution for the suproblem in the update rule can be computationally very expensive.
In this work we propose and analyze inexact variants by allowing to solve the  subproblem that appear in the update rules of the stochastic methods, inexactly. In particular, following the convention established in \cite{ASDA} of naming the main algorithm of the paper \emph{Basic method} we propose the \emph{inexact Basic method (iBasic)} (Algorithm~\ref{inexact_basic}). 

\begin{algorithm}[H]
  \caption{Inexact Basic Method (iBasic)
    \label{inexact_basic}}
  \begin{algorithmic}[1]
    \Require{Distribution $\cD$ from which we draw random matrices $\bS$, positive definite matrix $\bB\in\R^{n\times n}$, stepsize $\omega>0$.}
    \Ensure{$x^0\in\R^n$}
 \For{$k=0,1,2,\cdots$}
 \State Generate a fresh sample $\bS_k \sim {\cal D}$
 \State Set $x^{k+1}=x^k-\omega \mB^{-1}\mA^\top \mS_{k} (\mS_{k}^\top \mA \mB^{-1} \mA^\top \mS_{k})^{\dagger} \mS_{k}^\top(\mA x^k-b)+ \epsilon^k$
 \EndFor
 \end{algorithmic}
\end{algorithm}

The $\epsilon^k$ in the update rule of the method represents the abstract inexactness error described in Subsection~\ref{subproblems}. Note that, iBasic can have several equivalent interpretations. This allow as to study the methods \eqref{SGD_IntroThesis},\eqref{SNM_IntroThesis},\eqref{SPPM_IntroThesis} for solving the stochastic optimization problem and the sketch and project method \eqref{SPM_IntroThesis} for the best approximation problem in a single algorithm only. 
In particular iBasic can be seen as inexact stochastic gradient descent (iSGD) with fixed stepsize applied to \eqref{StochReform_IntroThesis}. From \eqref{Gradf_S_IntroThesis}, $\nabla f_{\mS_k}(x^k) = \mB^{-1} \mA^\top \mH_k (\mA x^k - b) $ and as a result the update rule of iBasic can be equivalently written as: $x^{k+1}=x^k-\omega \nabla f_{\mS_k}(x^k)  + \epsilon^{k}.$ In the case of the best approximation problem \eqref{BestApproximation_IntroThesis}, iBasic can be interpreted as inexact Sketch and Project method (iSPM) as follows:
\begin{eqnarray}
\label{anska}
x^{k+1} & = & x^k-\omega \mB^{-1}\mA^\top \mS_{k} (\mS_{k}^\top \mA \mB^{-1} \mA^\top \mS_{k})^{\dagger} \mS_{k}^\top(\mA x^k-b)+\epsilon^k  \notag\\
& = &  \omega \left[ x^k- \bB^{-1}(\bS_k^\top \bA)^\top (\bS_k^\top \bA B^{-1}(\bS_k^\top \bA) ^ \top )^\dagger (\bS_k^\top \bA x^k-\bS_k^\top b)\right] + (1-\omega) x^k   +\epsilon^k \notag\\
&\overset{\eqref{Projection_IntroThesis}} =&\omega \Pi_{\cL_{\mS_k},\mB}(x^k) + (1-\omega) x^k +\epsilon^k
\end{eqnarray}
For the dual problem \eqref{DualProblem_IntroThesis} we devote Section~\ref{InexactDualMethods} for presenting an inexact variant of the SDSA (iSDSA) and analyze its convergence using the rates obtained for the iBasic in Sections~\ref{gental assumption} and \ref{InexactSolvers}.

\subsection{General framework and further special cases}
\label{Special Cases}

The proposed inexact methods, iBasic  (Algorithm~\ref{inexact_basic}) and iSDSA (Section~\ref{InexactDualMethods}),  belong in the general \emph{sketch and project} framework, first proposed from Gower and Richtarik in \cite{gower2015randomized} for solving consistent linear systems and where a unified analysis of several randomized methods was studied. This interpretation of the algorithms allow us to recover a comprehensive array of well-known methods as special cases by choosing carefully the combination of the main parameters of the algorithms.

In particular, the iBasic has two main parameters (besides the stepsize $\omega>0$ of the update rule). These are the distribution $\cD$ from which we draw random matrices $\bS$ and the positive definite matrix $\bB\in\R^{n\times n}$. By choosing carefully combinations of the parameters $\cD$ and $\bB$ we can recover several existing popular algorithms as special cases of the general method.  For example, special cases of the exact Basic method are the Randomized Kaczmarz, Randomized Gaussian Kaczmarz\footnote{Special case of the iBasic, when the random matrix $\bS$ is chosen to be a Gaussian vector with mean $0 \in R^m$ and a positive definite covariance matrix $\Sigma \in \R^{m\times m}$. That is $\bS \sim N(0,\Sigma)$ \cite{gower2015randomized}.}, Randomized Coordinate Descent and their block variants. For more details about the generality of the sketch and project framework and further algorithms that can be cast as special cases of the analysis we refer the interested reader to Section 3 of \cite{gower2015randomized}. Here we present only the inexact update rules of two special cases that we will later use in the numerical evaluation. 

\emph{Special Cases: }
Let us define with $\bI_{:C}$ the column concatenation of the $m \times m$ identity matrix indexed by a random subset $C$ of $[m]$. 
\begin{itemize}
\item \emph{Inexact Randomized Block Kaczmarz (iRBK)}:
Let $\bB= \bI$ and let pick in each iteration the random matrix $\bS=\bI_{:C} \sim \cD$. In this setup the update rule of the iBasic simplifies to 
\begin{equation}
\label{iRBK}
x^{k+1}=x^k -\omega \bA_{C:}^\top (\bA_{C:}\bA_{C:}^\top)^\dagger (\bA_{C:}x^k-b_C) + \epsilon^k.
\end{equation}
\item \emph{Inexact Randomized Block Coordinate Descent (iRBCD)}\footnote{In the setting of solving linear systems Randomized Coordinate Descent is known also as Gauss-Seidel method. Its block variant can be also interpret as randomized coordinate Newton method (see \cite{qu2015sdna}).}:
If the matrix $\bA$ of the linear system is positive definite then we can choose $\bB= \bA$. Let also pick in each iteration the random matrix $\bS=\bI_{:C} \sim \cD$. In this setup the update rule of the iBasic simplifies to 
\begin{equation}
\label{iRBCD}
x^{k+1}=x^k -\omega \bI_{:C} (\bI_{:C}^\top\bA \bI_{:C})^\dagger \bI_{:C}^\top (\bA x^k-b) + \epsilon^k. 
\end{equation}
\end{itemize}

For more papers related to Kaczmarz method (randomized, greedy, cyclic update rules) we refer the interested reader to \cite{kaczmarz1937angenaherte, loizou2017linearly, popa1995least,byrne2008applied, nutini2016convergence, popa2017convergence, CsibaPL, needell2010randomized, RBK, eldar2011acceleration, MaConvergence15, zouzias2013randomized, l2015randomized, schopfer2016linear}. 
For the coordinate descent method (a.k.a Gauss-Seidel for linear systems) and its block variant, Randomized Block Coordinate Descent we suggest \cite{leventhal2010randomized, nesterov2012efficiency, richtarik2014iteration, richtarik2016parallel, qu2016coordinate,qu2016coordinate2, qu2015quartz, SCP, lee2013efficient, fercoq2015accelerated, allen2016even, tu2017breaking}. 

\subsection{Other related work on inexact methods}
\label{OtherWork}
One of the current trends in the large scale optimization problems is the introduction of inexactness in the update rules of popular deterministic and stochastic methods. The rational behind this is that an approximate/inexact step can often computed very efficiently and can have significant computational gains compare to its exact variants. 

In the area of deterministic algorithms, the inexact variant of the full gradient descent method, $x^{k+1} = x^k - \omega_k [\nabla f(x^k)+\epsilon^k]$, has received a lot of attention \cite{schmidt2011convergence, devolder2014first,so2017non, friedlander2012hybrid,necoara2014rate}. It has been analyzed for the cases of convex and strongly convex functions under several meaningful assumptions on the inexactness error $\epsilon^k$ and its practical benefit  compared to the exact gradient descent is apparent.  For further deterministic inexact methods check \cite{dembo1982inexact} for Inexact Newton methods, \cite{solodov2001unified, salzo2012inexact} for Inexact Proximal Point methods and \cite{birken2015termination} for Inexact Fixed point methods.

In the recent years, with the explosion that happens in areas like machine learning and data science inexactness enters also the updating rules of several stochastic optimization algorithms and many new methods have been proposed and analyzed. 

In the large scale setting, stochastic optimization methods are preferred mainly because of their cheap per iteration cost (compared to their deterministic variants), their property to scale to extreme dimensions and their improved theoretical complexity bounds. In areas like machine learning and data science, where the datasets become larger rapidly, the development of faster and efficient stochastic algorithms is crucial. For this reason, inexactness has recently introduced to the update rules of several stochastic optimization algorithms and new methods have been proposed and analyzed. One of the most interesting work on inexact stochastic algorithms appears in the area of second order methods. In particular on inexact variants of the Sketch-Newton method and subsampled Newton Method for minimize convex and non-convex functions \cite{schmidt201111, berahas2017investigation, bollapragada2016exact,xu2017newton,xu2016sub, yao2018inexact}. Note that our results are related also with this literature since our algorithm can be seen as inexact stochastic Newton method (see equation \eqref{SNM_IntroThesis}). To the best or our knowledge our work is the first that provide convergence analysis of inexact stochastic proximal point methods (equation \eqref{SPPM_IntroThesis}) in any setting. From numerical linear algebra viewpoint inexact sketch and project methods for solving the best approximation problem and its dual problem where also never analyzed before. 

As we already mentioned our framework is quite general and many algorithms, like iRBK \eqref{iRBK} and iRBCD \eqref{iRBCD} can be cast as special cases. As a result, our general convergence analysis includes the analysis of inexact variants of all of these more specific algorithms as special cases. In \cite{RBK} an analysis of the exact randomized block Kacmzarz method has been proposed and in the experiments an inexact variant was used to speedup the method. However, no iteration complexity results were presented for the inexact variant and both the analysis and numerical evaluation have been made for linear systems with full rank matrices that come with natural partition of the rows (this is a much more restricted case than the one analyzed in our setting). For inexact variants of the randomized block coordinate descent algorithm in different settings than ours we suggest  \cite{tappenden2016inexact, fountoulakis2018flexible, cassioli2013convergence, dvurechensky2017randomized}.

Finally an analysis of approximate stochastic gradient descent for solving the empirical risk minimization problem using quadratic constraints and sequential semi-definite programs has been presented in \cite{hu2017analysis}. 

\section{Convergence Results Under General Assumptions}
\label{gental assumption}
In this section we consider scenarios in which the inexactness error $\epsilon^k$  can be controlled, by specifying a per iteration bound $\sigma_k$ on the norm of the error.  In particular, by making different assumptions on the bound $\sigma_k$ we derive general convergence rate results. Our focus is on the abstract notion of inexactness described in Section~\ref{subproblems} and we make no assumptions on how this error is generated. 

An important assumption that needs to be hold in all of our results is \emph{exactness}. A formal presentation of exactness was presented in the introduction of this thesis. We highlight that is a requirement for all of the convergence results of this chapter (It is also required in the analysis of the exact algorithms; see Theorems~\ref{BasicMethodConvergence} and \ref{TheoremSDSA_IntroThesis} in the introduction).

\subsection{Assumptions on inexactness error}
\label{asssad}
In the convergence analysis of iBasic the following assumptions on the inexactness error are used.  We note that Assumptions \ref{Assumption1}, \ref{Assumption3} and \ref{Assumption4} are special cases of Assumption \ref{Assumption2}. Moreover Assumption \ref{Assumption5} is algorithmic dependent and can hold in addition of any of the other four assumptions. In our analysis, depending on the result we aim at, we will require either one of the first four Assumptions to hold by itself, or to hold together with Assumption \ref{Assumption5}. We will always assume exactness.
  
In all assumptions the expectation on the norm of error ($\|\epsilon^k\|^2$) is conditioned on the value of the current iterate $x^k$ and the random matrix $\bS_k$. Moreover it is worth to mention that for the convergence analysis we never assume that the inexactness error has zero mean, that is $\Exp[\epsilon^k]=0$. 

\begin{assumption}{1}{}
\label{Assumption2}
 \begin{equation}
 \label{Assumption2serial}
  \Exp[\|\epsilon^k\|^2_{\bB}\;|\;x^k,\bS_k]\leq\sigma_k^2,
  \end{equation}
where the upper bound $\sigma_k$ is a sequence of random variables (that can possibly depends on both the value of the current iterate $x^k$ and the choice of the random $\bS_k$ at the $k^{th}$ iteration).
\end{assumption}

The following three assumptions on the sequence of upper bounds are more restricted however as we will later see allow us to obtain stronger and more controlled results.

 \begin{assumption}{1}{a}
 \label{Assumption1}
 \begin{eqnarray}
\label{Assumption1serial}
 \Exp[\|\epsilon^k\|^2_{\bB}\;|\;x^k,\bS_k]\leq\sigma_k^2,
\end{eqnarray}
where the upper bound $\sigma_k\in \R$ is a sequence of real numbers. 
\end{assumption}

\begin{assumption}{1}{b}
\label{Assumption3}
 \begin{equation}
 \label{Assumption3serial}
  \Exp[\|\epsilon^k\|^2_{\bB}\;|\;x^k,\bS_k]\leq\sigma_k^2 = q^2\|x^k-x^*\|^2_{\bB},
  \end{equation}
where the upper bound is a special sequence that depends on a non-negative inexactness parameter $q$ and the distance to the optimal value $\|x^k-x^*\|^2_{\bB}$.
\end{assumption}

\begin{assumption}{1}{c}
\label{Assumption4}
 \begin{equation}
 \label{Assumption4serial}
  \Exp[\|\epsilon^k\|^2_{\bB}\;|\;x^k,\bS_k]\leq\sigma_k^2 = 2 q^2 f_{\bS_k}(x^k),
  \end{equation}
where the upper bound is a special sequence that depends on a non-negative inexactness parameter $q$ and the value of the stochastic function $f_{\bS_k}$ computed at the iterate $x^k$. Recall, that in our setting $f_{\bS}(x) = \frac{1}{2}\|\nabla f_{\bS}(x)\|^2_{\bB}$  \eqref{normbound}. Hence, the upper bound can be equivalently expressed as $\sigma_k^2=q^2 \|\nabla f_{\bS}(x)\|^2_{\bB}$.

\end{assumption}

Finally the next assumption is more algorithmic oriented. It holds in cases where the inexactness error $\epsilon^k$ in the update rule is chosen to be orthogonal with respect to the $\bB$-inner product to the vector  $\Pi_{\cL_{\bS_k},{\bB}} (x^k) - x^* = (\mI - \omega \mB^{-1}\mZ_k) (x^k-x^*)$. This statement may seem odd at this point but its usefulness will become more apparent in the next section where inexact algorithms with structured inexactness error will be analyzed. As it turns out, in the case of structured inexactness error (Algorithm~\ref{inexact_solver_algorithm}) this assumption is satisfied.
\begin{assumption}{2}{}
 \label{Assumption5}
 \begin{equation}
\Exp[\left\langle (\mI - \omega \mB^{-1}\mZ_k) (x^k-x^*), \epsilon^k \right \rangle_{\mB}]=0.
  \end{equation}

\end{assumption}

\subsection{Convergence results}
In this section we present the analysis of the convergence rates of iBasic by assuming several combination of the previous presented assumptions. 

All convergence results are described only in terms of convergence of the iterates $x^k$, that is $\|x^k-x^*\|^2_{\bB}$, and not the objective function values $f(x^k)$. This is sufficient, because by $f(x)\leq\frac{\lambda_{\rm max}}{2}\|x-x^*\|^2_{\bB}$ (see Lemma \ref{bounds}) we can directly deduce a convergence rate for the function values. 

The exact Basic method (Algorithm~\ref{inexact_basic} with $\epsilon^k =0$), has been analyzed in \cite{ASDA} and it was shown to converge with $\Exp[\|x^{k}-x^*\|^2_{\mB}] \leq \rho^{k} \|x^{0}-x^*\|^2_{\mB}$ where $\rho=1-\omega(2-\omega) \lambda_{\min}^+$. Our analysis of iBasic is more general and includes the convergence of the exact Basic method as special case when we assume that the upper bound is $\sigma_k=0, \quad \forall k\geq0$. For brevity, in he convergence analysis results of this chapter we also use $$\rho=1-\omega(2-\omega) \lambda_{\min}^+.$$

Let us start by presenting the convergence of iBasic when only Assumption~\ref{Assumption1} holds for the inexactness error.
\begin{thm}
\label{InexactSGDConstant}
Let assume exactness and let $\{x^k\}_{k=0}^\infty$ be the iterates produced by iBasic with $\omega\in(0,2)$.  Set $x^*= \Pi_{\cL,\mB}(x^0)$ and consider the error $\epsilon^k$ be such that it satisfies Assumption \ref{Assumption1}. Then,
\begin{eqnarray}
\label{Theorem1}
\Exp[\|x^{k}-x^*\|_{\mB}] \leq \rho^{k/2} \|x^{0}-x^*\|_{\mB} +  \sum_{i=0}^{k-1} \rho^{\frac{k-1-i}{2}}\sigma_i.
\end{eqnarray}
\end{thm}
\begin{proof}
See Section~\ref{Appendix1}.
\end{proof}
\begin{cor}
\label{FirstCorollary}
 In the special case that the upper bound $\sigma_k$ in Assumption \ref{Assumption1} is fixed, that is $\sigma_k=\sigma$ for all $k>0$ then inequality \eqref{Theorem1} of Theorem \ref{InexactSGDConstant} takes the following form:
\begin{eqnarray}
\Exp[\|x^{k}-x^*\|_{\mB}] \leq  \rho^{k/2} \|x^{0}-x^*\|_{\mB} +  \sigma \frac{\rho^{1/2}}{1-\rho}.
\end{eqnarray}
This means that we obtain a linear convergence rate up to a solution level that is proportional to the upper bound $\sigma$\footnote{Several similar more specific assumptions can be made for the upper bound $\sigma_k$. For example if the upper bound satisfies $\sigma_k=\sigma^k$ with $\sigma \in (0,1)$ for all $k>0$ then it can be shown that $C\in (0,1)$ exist such that inequality \eqref{Theorem1} of Theorem \ref{InexactSGDConstant} takes the form: $\Exp[\|x^{k}-x^*\|_{\mB}] \leq   \cO(C^k) $ (see \cite{so2017non, friedlander2012hybrid} for similar results). }.
\end{cor}
\begin{proof}
See Section~\ref{Appendix2}.
\end{proof}

Inspired from \cite{friedlander2012hybrid}, let us now analyze iBasic using the sequence of upper bounds that described in Assumption~\ref{Assumption3}. This construction of the upper bounds  allows us to obtain stronger and more controlled results. In particular using the upper bound of Assumption~\ref{Assumption3} the sequence of expected errors converge linearly to the exact $x^*$ (not in a potential neighborhood like the previous result). In addition Assumption~\ref{Assumption3} guarantees that the distance to the optimal solution reduces with the increasing of the number of iterations. However for this stronger convergence a bound for $\lambda_{\min}^+$ is required, a quantity that in many problems is unknown to the user or intractable to compute. Nevertheless, there are cases that this value has a closed form expression and can be computed before hand without any further cost. See for example \cite{LoizouRichtarik,loizou2018accelerated,loizou2018provably,hanzely2019privacy} where methods for solving the average consensus were presented and the value of $\lambda_{\min}^+$ corresponds to the algebraic connectivity of the network under study.
\begin{thm}
\label{ISGDwithq}
Assume exactness. Let $\{x^k\}_{k=0}^\infty$ be the iterates produced by iBasic with $\omega\in(0,2)$. Set $x^*= \Pi_{\cL, \mB}(x^0)$ and consider the inexactness error $\epsilon^k$ be such that it satisfies Assumption \ref{Assumption3},
with $0\leq q < 1-\sqrt{\rho}$. Then 
\begin{eqnarray}\label{Theorem2withq}
\Exp[\|x^{k}-x^*\|_{\mB}^2]\leq \left(\sqrt{\rho}+q\right)^{2k} \|x^0-x^*\|_{\mB}^2.
\end{eqnarray}
\end{thm}
\begin{proof}
See Section~\ref{Appendix3}.
\end{proof}

At Theorem~\ref{ISGDwithq}, to guarantee linear convergence the \emph{inexact parameter} $q$ should live in the interval $\left[0,1-\sqrt{\rho}\right)$. In particular, $q$ is the parameter that controls the level of inexactness of Algorithm \ref{inexact_basic}. Not surprisingly the fastest convergence rate is obtained when $q=0$; in such case the method becomes equivalent with its exact variant and the convergence rate simplifies to $\rho=1- \omega (2-\omega)\lambda_{\min}^+$.  Note also that similar to the exact case the optimal convergence rate is obtained for $\omega=1$ \cite{ASDA}.

Moreover, the upper bound $\sigma_k$ of Assumption~\ref{Assumption3} depends on two important quantities, the $\lambda_{\min}^+$ (through the upper bound of the inexactness parameter $q$) and the distance to the optimal solution $\|x^k-x^*\|^2_{\bB}$. Thus, it can have natural interpretation. In particular the inexactness error is allowed to be large either when the current iterate is far from the optimal solution ($\|x^k-x^*\|^2_{\bB}$ large) or when the problem is well conditioned and $\lambda_{\min}^+$ is large. In the opposite scenario, when we have ill conditioned problem or we are already close enough to the optimum $x^*$ we should be more careful and allow less errors to the updates of the method. 

In the next theorem we provide the complexity results of iBasic in the case that the Assumption~\ref{Assumption5} is satisfied combined with one of the previous assumptions.
\begin{thm}
\label{InexactSGDrandom}
Let assume exactness and let $\{x^k\}_{k=0}^\infty$ be the iterates produced by iBasic with $\omega\in(0,2)$.  Set $x^*= \Pi_{\cL,\mB}(x^0)$. Let also assume that the inexactness error $\epsilon^k$ be such that it satisfies Assumption~\ref{Assumption5}. Then:
\begin{enumerate}
\item[(i)] If Assumption \ref{Assumption2} holds: 
\begin{eqnarray}
\label{klasnaso}
\Exp[\|x^{k}-x^*\|_{\mB}^2] \leq  \rho^{k} \|x^{0}-x^*\|^2_{\mB} +  \sum_{i=0}^{k-1} \rho^{k-1-i}{\bar{\sigma}}^2_i,
\end{eqnarray}
where $\bar{\sigma}_i^2=\Exp[\sigma_i^2], \forall i \in [k-1].$
\item[(ii)] If Assumption \ref{Assumption3} holds with $q \in \left(0,\sqrt{\rho}\right)$: 
\begin{eqnarray}
\label{jaksxal}
\Exp[\|x^{k}-x^*\|_{\mB}^2] & \leq& (\rho + q^2)^k \|x^0-x^*\|_{\mB}^2. 
\end{eqnarray}
\item[(iii)] If Assumption \ref{Assumption4} holds with $q \in \left(0, \sqrt{\omega(2-\omega)}\right)$: 
\begin{equation}
\Exp[\|x^{k}-x^*\|_{\mB}^2]\leq(1- (\omega (2-\omega)-q^2 )\lambda_{\min}^+ )^k \|x^0-x^*\|_{\mB}^2=(\rho+q^2\lambda_{\min}^+ )^k \|x^0-x^*\|_{\mB}^2.
\end{equation}
\end{enumerate} 
\begin{proof}
See Section~\ref{Appendix4}.
\end{proof}
\end{thm}

\begin{rem}
In the case that Assumptions~\ref{Assumption1} and \ref{Assumption5} hold simultaneously, the convergence of iBasic is similar to \eqref{klasnaso} but in this case ${\bar{\sigma}}^2_i=\sigma_i^2,\, \forall i \in [k-1]$ (due to Assumption~\ref{Assumption1}, $\sigma_k \in \R$ is a sequence of real numbers). In addition, note that for $q \in (0, \min\{\sqrt{\rho}, 1-\sqrt{\rho}\})$ having Assumption~\ref{Assumption5} on top of Assumption~\ref{Assumption3} leads to improvement of the convergence rate. In particular, from Theorem~\ref{ISGDwithq}, iBasic converges with rate $(\sqrt{\rho}+q)^{2}= \rho+q^2+2\sqrt{\rho}q$ while having both assumptions this is simplified to the faster $\rho + q^2$ \eqref{jaksxal}. 
\end{rem}

\section{iBasic with Structured Inexactness Error}
\label{InexactSolvers}
Up to this point, the analysis of iBasic was focused in more general abstract cases where the inexactness error $\epsilon^k$ of the update rule satisfies several general assumptions.  In this section we are focusing on a more structured form of inexactness error and we provide convergence analysis in the case that a linearly convergent algorithm is used for the computation of the expensive key subproblem of the method. 

\subsection{Linear system in the update rule}
\label{sectionlinesyst}

As we already mentioned in Section~\ref{subproblems} the update rule of the exact Basic method (Algorithm~\ref{inexact_basic} with $\epsilon^k=0$) can be expressed as $ x^{k+1} = x^k+\omega \mB^{-1}\mA^\top \mS_{k} \lambda^k_*$, where $\lambda^k_*= (\mS_{k}^\top \mA \mB^{-1} \mA^\top \mS_{k})^{\dagger} \mS_{k}^\top(b-\mA x^k)$.

Using this expression the exact Basic method can be equivalently interpreted as the following two step procedure:
\begin{enumerate}
\item Find the least norm solution\footnote{We are precisely looking for the least norm solution of the linear system $\bM_k  \lambda =d_k$ because this solution can be written down in a compact way using the Moore-Penrose pseudoinverse.  This is equivalent with the expression that appears in our update: $\lambda^k_*=(\mS_{k}^\top \mA \mB^{-1} \mA^\top \mS_{k})^{\dagger} \mS_{k}^\top(b-\mA x^k)=\bM_k^{\dagger} d_k $. However it can be easily shown that the method will still converge with the same rate of convergence even if we choose any other solution of the linear system $\bM_k  \lambda =d_k$.} of $
\underbrace{\mS_{k}^\top \mA \mB^{-1} \mA^\top \mS_{k}}_{\bM_k} \lambda = \underbrace{\mS_{k}^\top(b-\mA x^k)}_{d_k}
$. That is find $ \lambda^k_*=\arg\min_{ \lambda \in \cQ_k} \| \lambda\|$ where $\cQ_k= \left\{ \lambda \in \R^q : \bM_k  \lambda =d_k\right\}$.
\item Compute the next iterate:  $x^{k+1}=x^k + \omega \mB^{-1} \mA ^\top \mS_{k}  \lambda^k_*.$
\end{enumerate}

In the case that the random matrix $\bS_k$ is large (this is the case that we are interested in), solving exactly the linear system $\bM_k  \lambda=d_k$ in each step can be prohibitively expensive. To reduce this cost we allow the inner linear system $\bM_k  \lambda=d_k$ to be solved inexactly using an iterative method. In particular we propose and analyze the following inexact algorithm: 

\begin{algorithm}[H]
  \caption{iBasic with structured inexactness error
    \label{inexact_solver_algorithm}}
  \begin{algorithmic}[1]
    \Require{Distribution $\cD$ from which we draw random matrices $\bS$, positive definite matrix $\bB\in\R^{n\times n}$, stepsize $\omega>0$.}
    \Ensure{$x^0\in\R^n$}
 \For{$k=0,1,2,\cdots$}
 \State Generate a fresh sample $\bS_k \sim {\cal D}$
 \State  Using an iterative method compute an approximation $\lambda^k_{\approx}$ of the least norm solution of the linear system:
 \begin{equation}
 \label{linearsysteminCode}
\underbrace{\mS_{k}^\top \mA \mB^{-1} \mA^\top \mS_{k}}_{\bM_k} \lambda = \underbrace{\mS_{k}^\top(b-\mA x^k)}_{d_k}.
\end{equation}
 \State Set $x^{k+1}=x^k + \omega \mB^{-1} \mA ^\top \mS_{k} \lambda^k_{\approx}$.
 \EndFor
 \end{algorithmic}
\end{algorithm}

For the computation of the inexact solution of the linear system \eqref{linearsysteminCode} any known iterative method for solving general linear systems can be used. In our analysis we focus on linearly convergent methods. For example based on the properties of the linear system \eqref{linearsysteminCode},  conjugate gradient (CG) or sketch and project method (SPM) can be used for the execution of step 3. In these cases, we name Algorithm~\ref{inexact_solver_algorithm},  \emph{InexactCG} and \emph{InexactSP } respectively.

It is known that the classical CG can solve linear systems with positive definite matrices. In our approach matrix $\bM_k $ is positive definite only when the original linear system $\bA x= b$ has full rank matrix $\bA$. On the other side SPM can solve any consistent linear system and as a result can solve the inner linear system $\bM_k  \lambda^k=d_k$ without any further assumption on the original linear system. In this case, one should be careful because the system has no unique solution. We are interested to find the least norm solution of $\bM_k  \lambda^k=d_k$ which means that the starting point of the sketch and project at the $k^{th}$ iteration should be always $\lambda^k_0=0$. Recall that any special case of the sketch and project method (Section~\ref{Special Cases}) solves the best approximation problem.

Let us now define $\lambda^k_r$ to be the approximate solution $\lambda^k_{\approx}$ of the $q \times q$ linear system \eqref{linearsysteminCode} obtained after $r$ steps of the linearly convergent iterative method.  Using this, the update rule of Algorithm~\ref{inexact_solver_algorithm}, takes the form:
\begin{equation}
\label{inesacdwdad}
x^{k+1}=x^k+ \omega \mB^{-1} \mA ^\top \mS_{k} \lambda^k_r.
\end{equation}

\begin{rem}
\label{corresINexacterror}
The update rule \eqref{inesacdwdad} of Algorithm~\ref{inexact_solver_algorithm} is equivalent to the update rule of iBasic (Algorithm~\ref{inexact_basic}) when the error $\epsilon^k$ is chosen to be,
\begin{equation}
\label{specEpsilon}
\epsilon^k=\omega \bB^{-1}\bA^\top \bS_k (\lambda^k_r-\lambda^k_*).
\end{equation}
This is precisely the connection between the abstract and more concrete/structured notion of inexactness that first presented in Table~\ref{KeySubproblems}.
\end{rem}

Let us now define a Lemma that is useful for the analysis of this section and it verifies that Algorithm \ref{inexact_solver_algorithm} with unit stepsize satisfies the general Assumption \ref{Assumption5} presented in Section~\ref{asssad}. 

\begin{lem}
\label{lemmaFigure}
Let us denote $x^k_*=\Pi_{\cL_{\bS_k},\bB}(x^k)$ the projection of $x^k$ onto $\cL_{\bS_k}$ in the $\mB$-norm and $x^*= \Pi_{\cL,\mB}(x^0)$. Let also assume that $\omega=1$ (unit stepsize). Then for the updates of Algorithm~\ref{inexact_solver_algorithm} it holds that:
\begin{equation}
\label{iandia}
 \left\langle x^k_*-x^*, \epsilon^k \right \rangle_{\mB}=\left\langle (\mI - \omega \mB^{-1}\mZ_k) (x^k-x^*), \epsilon^k \right \rangle_{\mB}=0, \quad \forall k \geq 0.
\end{equation}
\end{lem}
\begin{proof}
Note that $x^k_*-x^*=x^k-\nabla f_{\mS_k}(x^k)-x^* \in Null (\bS_k^\top \bA)$ . Moreover $\epsilon^k\overset{\eqref{specEpsilon}}=\bB^{-1}\bA^\top \bS_k (\lambda^k_r - \lambda^k_*) \in Range(\bB^{-1} \bA^\top \bS_k)$. From the knowledge that the null space of an arbitrary matrix is the orthogonal complement of the range space of its transpose we have that $Null (\bS_k^\top \bA)$ is orthogonal with respect to the $\bB$-inner product to $Range(\bB^{-1} \bA^\top \bS_k)$. This completes the proof (see Figure~\ref{SketchProject} for the graphical interpretation).
\end{proof}

\begin{figure}[t!]
  \centering
\includegraphics[width=8cm]{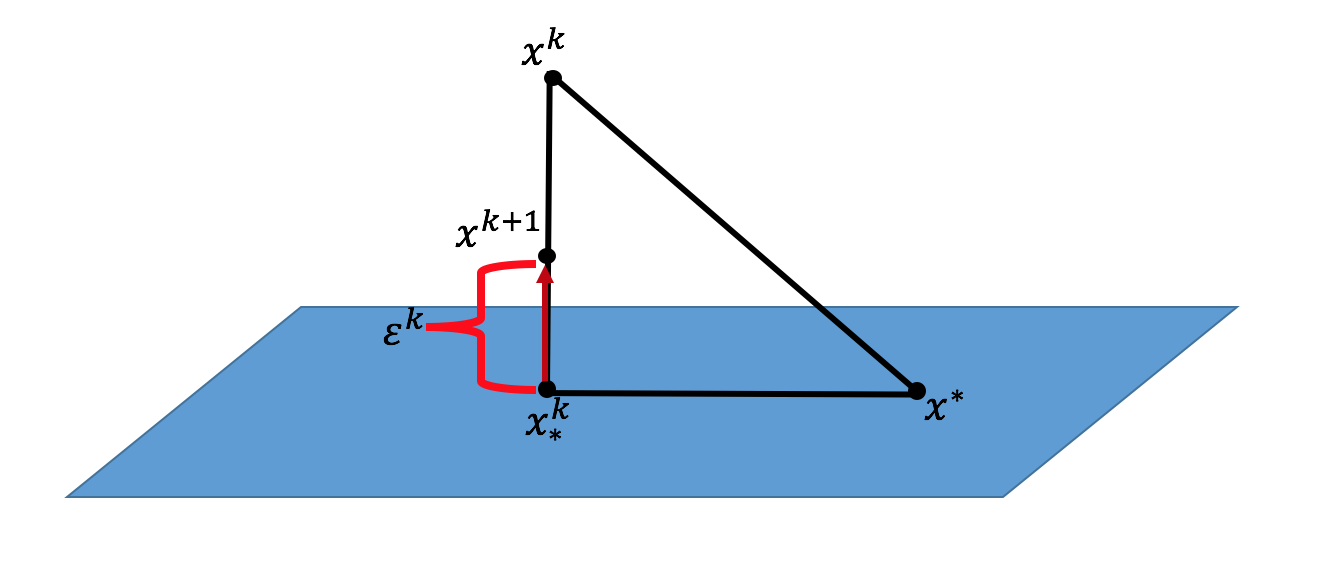}
\caption{\small Graphical interpretation of orthogonality (justifies equation \eqref{iandia}). It shows that the two vectors, $x^k_*-x^*$ and $\epsilon^k$, are orthogonal complements of each other with respect to the $\bB$-inner product. $x^{k+1}$ is the point that Algorithm \ref{inexact_solver_algorithm} computes in each step. The colored region represents the $Null (\bS_k^\top \bA)$. $x^k_*=\Pi_{\cL_{\bS_k}, \bB}(x^k)$, $x^*= \Pi_{\cL, \mB}(x^0)$ and $\epsilon^k$ is the inexactness error. }
\label{SketchProject}
\end{figure}

\subsection{Sketch and project interpretation}
\label{SFPinterpretation}
Let us now give a different interpretation of the inexact update rule of Algorithm~\ref{inexact_solver_algorithm} using the sketch and project approach. That will make us appreciate more the importance of the dual viewpoint and make clear the connection between the primal and dual methods.



In general, execute a projection step is one of the most common task in numerical linear algebra/optimization literature. However in the large scale setting even this task can be prohibitively expensive and it can be difficult to execute inexactly.  For this reason we suggest to move to the dual space where the inexactness can be easily controlled. 

Observe that the update rule of the exact sketch and project method \eqref{Sket} has the same structure as the best approximation problem \eqref{BestApproximation_IntroThesis} where the linear system under study is the sketched system $\bS_k^\top \mA x = \bS_k^\top b$ and the starting point is the current iterate $ x^k$. Hence we can easily compute its dual:
\begin{equation}
\label{innerdualprobalme}
\max_{\lambda \in \R^q} D_k(\lambda) \eqdef (\bS_k^\top b - \bS_k^\top \bA x^k)^\top \lambda - \frac{1}{2}\|\bA^\top \bS_k \lambda\|^2_{\bB^{-1}}.
\end{equation}
where $\lambda \in \R^q$ is the dual variable. The $\lambda^k_*$ (possibly more than one) that solves the dual problem in each iteration $k$, is the one that satisfies $\nabla D_k (\lambda^k_*)=0$. By computing the derivative this is equivalent with finding the $\lambda$ that satisfies the linear system $\mS_{k}^\top \mA \mB^{-1} \mA^\top \mS_{k} \lambda = \mS_{k}^\top(b-\mA x^k)$. This is the same linear system we desire to solve inexactly in Algorithm \ref{inexact_solver_algorithm}. Thus, computing an inexact solution $\lambda^k_{\approx}$ of the linear system is equivalent with computing an inexact solution of the dual problem \eqref{innerdualprobalme}. Then by using the affine mapping \eqref{Corresp_IntroThesis1} that connects the primal and the dual spaces we can also evaluate an inexact solution of the original primal problem \eqref{Sket}.

The following result relates the inexact levels of these quantities. In particular it shows that dual suboptimality of $\lambda^k$ in terms of dual function values is equal to the distance of the dual values $\lambda^k$ in the $\bM_k$-norm.
\begin{lem}
\label{lemmaD(y)}
Let us define $\lambda^k_* \in \R^q$ be the exact solution of the linear system $\mS_{k}^\top \mA \mB^{-1} \mA^\top \mS_{k} \lambda = \mS_{k}^\top(b-\mA x^k)$ or equivalently of dual problem \eqref{innerdualprobalme}. Let us also denote with $\lambda^k_{\approx} \in \R^q$ the inexact solution. Then:
$$D_k(\lambda^k_*)-D_k(\lambda^k_{\approx})= \frac{1}{2}\|\lambda^k_{\approx}-\lambda^k_*\|^2_{\mS_{k}^\top \mA \mB^{-1} \mA^\top \mS_{k}}.$$
\end{lem}

\begin{proof}
\begin{eqnarray*}
D_k(\lambda^k_*)-D_k(\lambda^k_{\approx}) & \overset {\eqref{innerdualprobalme}}{=} &  [\bS_k^\top b - \bS_k^\top \bA x^k]^\top [\lambda^k_*-\lambda^k_{\approx}]-\frac{1}{2} (\lambda^k_*)^\top \mS_{k}^\top \mA \mB^{-1} \mA^\top \mS_{k} \lambda^k_*\notag\\ 
&&+\frac{1}{2} (\lambda^k_{\approx})^\top \mS_{k}^\top \mA \mB^{-1} \mA^\top \mS_{k} \lambda^k_{\approx} \notag\\
& \overset {\eqref{Corresp_IntroThesisOptimal}}{=} & (\lambda^k_*)^\top \mS_{k}^\top \mA \mB^{-1} \mA^\top \mS_{k}  [\lambda^k_*-\lambda^k_{\approx}]-\frac{1}{2} (\lambda^k_*)^\top \mS_{k}^\top \mA \mB^{-1} \mA^\top \mS_{k} \lambda^k_*\notag\\ 
&&+ \frac{1}{2} (\lambda^k_{\approx})^\top \mS_{k}^\top \mA \mB^{-1} \mA^\top \mS_{k} \lambda^k_{\approx} \notag\\
& = & \frac{1}{2}(\lambda^k_{\approx}-\lambda^k_*)^\top \mS_{k}^\top \mA \mB^{-1} \mA^\top \mS_{k}(\lambda^k_{\approx}-\lambda^k_*) \notag\\
& = & \frac{1}{2}\|\lambda^k_{\approx}-\lambda^k_*\|^2_{\mS_{k}^\top \mA \mB^{-1} \mA^\top \mS_{k}} 
\end{eqnarray*}
where in the second equality we use equation \eqref{Corresp_IntroThesisOptimal} to connect the optimal solutions of \eqref{Sket} and \eqref{innerdualprobalme} and obtain $[\bS_k^\top b - \bS_k^\top \bA x^k]^\top=(\lambda^k_*)^\top \mS_{k}^\top \mA \mB^{-1} \mA^\top \mS_{k}.$
\end{proof}

\subsection{Complexity results}
\label{MainTheory}
In this part we analyze the performance of Algorithm \ref{inexact_solver_algorithm} when a linearly convergent iterative method is used for solving inexactly the linear system \eqref{linearsysteminCode} in step 3 of Algorithm \ref{inexact_solver_algorithm} . We denote with $\lambda^k_r$ the approximate solution of the linear system after we run the iterative method for $r$ steps.

Before state the main convergence result let us present a lemma that summarize some observations that are true in our setting.
\begin{lem}
Let $\lambda^k_*=(\mS_{k}^\top \mA \mB^{-1} \mA^\top \mS_{k})^{\dagger} \mS_{k}^\top(b-\mA x^k)$  be the exact solution and $\lambda^k_r$ be approximate solution of the linear system \eqref{linearsysteminCode}. Then, $\|\lambda^k_*\|^2_{\bM_k}=2 f_{\bS_k}(x^k) $ and $\|\epsilon^k\|_{\mB}^2= \|\lambda^k_r -\lambda^k_*\|_{\bM_k}^2$.
\begin{proof}
\begin{eqnarray}
\label{asdasfarf}
\|\lambda^k_*\|^2_{\bM_k} &=& \|\bM_k^\dagger \bS_k^\top \bA (x^*-x^k) \|^2_{\bM_k} = (x^k-x^*)^\top \bA^\top \bS_k \underbrace{\bM_k^{\dagger} \bM_k \bM_k^{\dagger}}_{ \bM_k^{\dagger}} \mS_{k}^\top \mA  (x^k-x^*) \notag\\
 &\overset{\eqref{ZETA_Intro}}=& (x^k-x^*)^\top \bZ_k (x^k-x^*) \overset{\eqref{f_sZeta_IntroThesis}}= 2 f_{\bS_k}(x^k). 
\end{eqnarray}
Moreover,
\begin{eqnarray}
\label{noasdnk}
\|\epsilon^k\|_{\mB}^2 &\overset{Remark~\ref{corresINexacterror}}=& \|\bB^{-1}\bA^\top \bS_k (\lambda^k_r -\lambda^k_*)\|_{\mB}^2 
=  \|\lambda^k_r -\lambda^k_*\|_{\mS_{k}^\top \mA \mB^{-1} \mA ^\top \mS_{k}}^2
=   \|\lambda^k_r -\lambda^k_*\|_{\bM_k}^2.
\end{eqnarray}
\end{proof}
\end{lem}

\begin{thm}
\label{MainTheorem}
Let us assume that for the computation of the inexact solution of the linear system \eqref{linearsysteminCode} in step 3 of Algorithm~\ref{inexact_solver_algorithm}, a linearly convergent iterative method is chosen such that \footnote{In the case that deterministic iterative method is used, like CG, we have that $\|\lambda^k_r-\lambda^k_*\|_{\bM_k}^2 \leq \rho_{\bS_k}^r \|\lambda^k_0-\lambda^k_*\|^2_{\bM_k}$ which is also true in expectation.}: 
\begin{equation}
\label{linaelr}
\Exp[\|\lambda^k_r-\lambda^k_*\|_{\bM_k}^2 \;|\;x^k,\bS_k] \leq \rho_{\bS_k}^r \|\lambda^k_0-\lambda^k_*\|^2_{\bM_k}, 
\end{equation}
where $\lambda^k_0=0$ for any $k>0$ and $\rho_{\bS_k} \in (0,1)$ for every choice of $\bS_k \sim \cD$.
Let exactness hold and let $\{x^k\}_{k=0}^\infty$ be the iterates produced by Algorithm~\ref{inexact_solver_algorithm} with unit stepsize ($\omega=1$).  Set $x^*= \Pi_{\cL,\mB}(x^0)$. 
Suppose further that there exists a scalar $\theta <1$ such that with
probability 1,  $\rho_{\bS_k} \leq \theta $. Then,  Algorithm \ref{inexact_solver_algorithm} converges linearly with:
$$\Exp[\|x^{k}-x^*\|_{\mB}^2] \leq \left[ 1- \left(1 -\theta^r  \right) \lambda_{\min}^+ \right] ^k \|x^0-x^*\|_{\mB}^2. $$
\end{thm}
\begin{proof}
Theorem~\ref{MainTheorem} can be interpreted as corollary of the general Theorem~\ref{InexactSGDrandom}(iii). Thus, it is sufficient to show that Algorithm \ref{inexact_solver_algorithm} satisfies the two Assumptions~\ref{Assumption4} and \ref{Assumption5}. Firstly, note that from Lemma~\ref{lemmaFigure}, Assumption \ref{Assumption5} is true. Moreover, 
\begin{eqnarray*}
\Exp[\|\epsilon^k\|_{\bM_k}^2 \;|\;x^k,\bS_k] &\overset{\eqref{noasdnk}}=&\Exp[\|\lambda^k_r-\lambda^k_*\|_{\bM_k}^2 \;|\;x^k,\bS_k] \overset{\eqref{linaelr}}\leq \rho_{\bS_k}^r \|\lambda^k_0-\lambda^k_*\|^2_{\bM_k} \notag\\
&\leq &\theta^r \|\lambda^k_0-\lambda^k_*\|^2_{\bM_k} \overset{\lambda^k_0=0}{=}  \theta^r \|\lambda^k_*\|^2_{\bM_k}\overset{\eqref{asdasfarf}}{=}  2 \theta^r f_{\bS_k}(x^k) 
\end{eqnarray*}
which means that Assumption~\ref{Assumption4} also holds with $q = \theta^{r/2} \in(0,1)$. This completes the proof.
\end{proof}

Having present the main result of this section let us now state some remarks that will help understand the convergence rate of the last Theorem.
\begin{rem}
From its definition $\theta^r \in (0,1)$ and as a result $\left(1 -\theta^r  \right) \lambda_{\min}^+ \leq \lambda_{\min}^+ $. This means that the method converges linearly but always with worst rate than its exact variant. 
\end{rem}

\begin{rem}
Let as assume that $\theta$ is fixed. Then as the number of iterations in step 3 of the algorithm ($r\rightarrow\infty$) increasing $(1-\theta^r) \rightarrow 1$ and as a result the method behaves similar to the exact case.
\end{rem}

\begin{rem}
The $\lambda_{\min}^+$ depends only on the random matrices $\bS \sim \cD$ and to the positive definite matrix $\bB$ and is independent to the iterative process used in step 3. The iterative process of step 3 controls only the parameter $\theta$ of the convergence rate.
\end{rem}

\begin{rem}
Let us assume that we run Algorithm \ref{inexact_solver_algorithm} two separate times for two different choices of the linearly convergence iterative method of step 3. Let also assume that the distribution $\cD$ of the random matrices and the positive definite matrix $\bB$ are the same for both instances and that for step 3 the iterative method run  for $r$ steps for both algorithms. Let assume that $\theta_1< \theta_2$ then we have that $\rho_1=1- \left(1 -\theta_1^r  \right) \lambda_{\min}^+< 1- \left(1 -\theta_2^r  \right) \lambda_{\min}^+=\rho_2$. This means in the case that $\theta$ is easily computable, we should always prefer the inexact method with smaller $\theta$.
\end{rem}

The convergence of Theorem~\ref{MainTheorem} is quite general and it holds for any linearly convergent methods that can inexactly solve  \eqref{linearsysteminCode}. However, in case that the iterative method is known we can have more concrete results. See below the more specified results for the cases of Conjugate gradient (CG) and Sketch and project method (SPM). 

\paragraph{Convergence of InexactCG:}
CG is deterministic iterative method for solving linear systems $\bA x=b$ with symmetric and positive definite matrix $\bA \in \R^{n\times n}$ in finite number of iterations. In particular, it can be shown that converges to the unique solution in at most $n$ steps. The worst case behavior of CG is given by \cite{wright1999numerical,golub2012matrix} \footnote{A sharper convergence rate of CG \cite{wright1999numerical} for solving $\bA x=b$ can be also used 
$$
\|x^k-x^*\|^2_{\bA}\leq \left( \frac{\lambda_{n-k}-\lambda_1}{\lambda_{n-k}+\lambda_1}\right)^{2} \|x^0-x^*\|^2_{\bA},
$$
where matrix $\bA\in \R^{n\times n}$ has $\lambda_1\leq \lambda_2 \leq \dots \leq \lambda_n$ eigenvalues.}:
\begin{equation}
\label{conjgrad}
\|x^k-x^*\|_{\bA}\leq \left( \frac{\sqrt{\kappa(\bA)}-1}{\sqrt{\kappa(\bA)}+1}\right)^{2k} \|x^0-x^*\|_{\bA},
\end{equation}
where $ x^k$ is the $k^{th}$ iteration of the method and $\kappa(\bA)$ the condition number of matrix $\bA$.

Having present the convergence of CG for general linear systems, let us now return back to our setting. We denote $\lambda^k_r \in \R^q$ to be the approximate solution of the inner linear system \eqref{linearsysteminCode} after $r$ conjugate gradient steps. Thus using \eqref{conjgrad} we know that
$\|\lambda^k_r-\lambda^k_*\|_{\bM_k}^2 \leq \rho_{\bS_k}^{4r} \|\lambda^k_0-\lambda^k_*\|^2_{\bM_k},$
where $\rho_{\bS_k}= \left( \frac{\sqrt{\kappa(\bM_k)}-1}{\sqrt{\kappa(\bM_k)}+1}\right)$. Now by making the same assumption as the general Theorem~\ref{MainTheorem} the InexactCG converges with
$\Exp[\|x^{k}-x^*\|_{\mB}^2] \leq \left[ 1- \left(1 -\theta_{CG}^r  \right) \lambda_{\min}^+ \right] ^k \|x^0-x^*\|_{\mB}^2, $
where $\theta_{CG} <1$ such that $\rho_{\bS_k}=\left( \frac{\sqrt{\kappa(\bM_k)}-1}{\sqrt{\kappa(\bM_k)}+1}\right)^{4} \leq \theta_{CG}$ with probability 1.

\paragraph{Convergence of InexactSP:}
In this setting we suggest to run the sketch and project method (SPM) for solving inexactly the linear system \eqref{linearsysteminCode}. This allow us to have no assumptions on the structure of the original system $\bA x=b$ and as a result we are able to solve more general problems compared to what problems InexactCG can solve\footnote{Recall that InexactCG requires the matrix $\bM_k$ to be positive definite (this is true when matrix $\bA$ is a full rank matrix)}. Like before, by making the same assumptions as in Theorem~\ref{MainTheorem} the more specific convergence $\Exp[\|x^{k}-x^*\|_{\mB}^2] \leq \left[ 1- \left(1 -\theta_{SP}^r  \right) \lambda_{\min}^+ \right] ^k \|x^0-x^*\|_{\mB}^2, $ for the InexactSP can be obtained. Now the quantity $\rho_{\bS_k}$ denotes the convergence rate of the exact Basic method\footnote{Recall that iBasic and its exact variant ($\epsilon^k=0$) can be expressed as sketch and project methods \eqref{anska}.} when this applied to solve linear system \eqref{linearsysteminCode} and $\theta_{SP} <1$ is a scalar such that $\rho_{\bS_k} \leq \theta_{SP}$ with probability 1.

\section{Inexact Dual Method}
\label{InexactDualMethods}
In the previous sections we focused on the analysis of inexact stochastic methods for solving the stochastic optimization problem \eqref{StochReform_IntroThesis} and the best approximation \eqref{BestApproximation_IntroThesis}.  In this section we turn into the dual of the best approximation \eqref{DualProblem_IntroThesis} and we propose and analyze an inexact variant of the SDSA \eqref{SDSA_IntroThesis}. We call the new method iSDSA and is formalized as Algorithm \ref{inexact_dual}. In the update rule $\epsilon_d^k$ indicates the dual inexactness error that appears in the $k^{th}$ iteration of iSDSA.

\begin{algorithm}[H]
  \caption{Inexact Stochastic Dual Subspace Ascent (iSDSA)
    \label{inexact_dual}}
  \begin{algorithmic}[1]
    \Require{Distribution $\cD$ from which we draw random matrices $\bS$, positive definite matrix $\bB\in\R^{n\times n}$, stepsize $\omega>0$.}
    \Ensure{$y^0=0 \in \R^m$, $x^0 \in \R^n$}
 \For{$k=0,1,2,\cdots$}
 \State Draw a fresh sample $\bS_k \sim \cD$
 \State  Set $y^{k+1}=y^k+\omega \mS_k \left(\mS_k^\top \bA \bB^{-1}\bA^\top \mS_k \right)^\dagger\bS_k^\top \left(b-\bA(x^0 + \bB^{-1}\bA^\top y^k) \right) + \epsilon_d^k$
 \EndFor
 \end{algorithmic}
\end{algorithm}

\subsection{Correspondence between the primal and dual methods}
With the sequence of the dual iterates $\{y^k\}_{k=0}^\infty$ produced by the iSDSA we can associate a sequence of primal iterates  $\{x^k\}_{k=0}^\infty$ using the affine mapping \eqref{Corresp_IntroThesis}.  In our first result we show that the random iterates produced by iBasic arise as an affine image of iSDSA under this affine mapping.

\begin{thm}(Correspondence between the primal and dual methods)
\label{LemmaCorrespondance}
Let $\{x^k\}_{k=0}^\infty$  be the iterates produced by iBasic (Algorithm~\ref{inexact_basic}). Let $y^0=0$, and $\{y^k\}_{k=0}^\infty$ the iterates of the iSDSA. Assume that the two methods use the same stepsize $\omega >0 $ and the same sequence of random matrices $\bS_k$. Assume also that $\epsilon^k=\bB^{-1}\bA^\top \epsilon_d^k$ where $\epsilon^k$ and $\epsilon_d^k$ are the inexactness errors appear in the update rules of iBasic and iSDSA respectively. Then $$x^k=\phi(y^k)=x^0+\bB^{-1}\bA^\top y^k.$$ for all $k\geq0$. That is, the primal iterates arise as affine images of the dual iterates.
\end{thm}
\begin{proof}
\begin{eqnarray*}
\phi(y^{k+1}) &\overset{\eqref{Corresp_IntroThesis}}{=}& x^0  + \bB^{-1}\bA^\top y^{k+1} \overset{ \eqref{lambdak_IntroThesis},\text{Alg.}\ref{inexact_dual}}{=} x^0+\bB^{-1}\bA^\top\left[y^k+\omega \bS_k \lambda^k +\epsilon_d^k \right]\\
&\overset{\eqref{ZETA_Intro}, \eqref{lambdak_IntroThesis} }{=}&  \underbrace{x^0+\bB^{-1}\bA^\top y^k}_{\phi(y^k)}+\omega \bB^{-1} \bZ_k \left(x^*-( \underbrace{x^0+\bB^{-1}\bA^\top y^k}_{\phi(y^k)}) \right)+\bB^{-1}\bA^\top \epsilon_d^k\\
&=& \phi(y^k)- \omega \bB^{-1}\bZ_k( \phi(y^k)-x^*) +\bB^{-1}\bA^\top \epsilon_d^k\\
\end{eqnarray*}
Thus by choosing the inexactness error of the primal method to be $\epsilon^k=\bB^{-1}\bA^\top \epsilon_d^k$ the sequence of vectors  $\{\phi(y^k)\}$ satisfies the same recursion as the sequence $\{x^k\}$ defined by iBasic.  It remains to check that the first element of both recursions coincide. Indeed, since $y^0=0$, we have $x^0 = \phi(0) =  \phi(y^0)$.
\end{proof}

\subsection{iSDSA with structured inexactness error}
In this subsection we present Algorithm~\ref{inexact_solver_Dualalgorithm}. It can be seen as a special case of iSDSA but with a more structured inexactness error. 

\begin{algorithm}[H]
  \caption{iSDSA with structured inexactness error
    \label{inexact_solver_Dualalgorithm}}
  \begin{algorithmic}[1]
    \Require{Distribution $\cD$ from which we draw random matrices $\bS$, positive definite matrix $\bB\in\R^{n\times n}$, stepsize $\omega>0$.}
    \Ensure{$y^0=0\in\R^m$, $x^0 \in \R^n$}
 \For{$k=0,1,2,\cdots$}
 \State Generate a fresh sample $\bS_k \sim {\cal D}$
 \State Using an Iterative method compute an approximation $\lambda^k_{\approx}$ of the least norm solution of the linear system:
 \begin{equation}
\underbrace{\mS_{k}^\top \mA \mB^{-1} \mA^\top \mS_{k}}_{\bM_k} \lambda =  \underbrace{\mS_{k}^\top(b-\bA(x^0 + \bB^{-1}\bA^\top y^k)}_{d_k}
\end{equation}
 \State Set $y^{k+1}=y^k + \omega \mS_{k} \lambda^k_{\approx}$
 \EndFor
 \end{algorithmic}
\end{algorithm}

Similar to their primal variants, it can be easily checked that Algorithm~\ref{inexact_solver_Dualalgorithm} is a special case of the iSDSA ( Algorithm~\ref{inexact_dual}) when the dual inexactness error is chosen to be $\epsilon^k_d= \bS_k(\lambda^k_r-\lambda^k_*)$. Note that, using the observation of Remark~\ref{corresINexacterror} that $\epsilon^k=\omega \bB^{-1}\bA^\top \bS_k (\lambda^k_r-\lambda^k_*)$ and the above expression of $\epsilon^k_d$ we can easily verify that the expression
$\epsilon^k=\bB^{-1}\bA^\top \epsilon_d^k$ holds.
This is precisely the connection between the primal and dual inexactness errors that have already been used in the proof of Theorem~\ref{LemmaCorrespondance}.

\subsection{Convergence of dual function values}
We are now ready to state a linear convergence result describing the behavior of the inexact dual method in terms of the function values $D(y^k)$. The following result is focused on the convergence of iSDSA by making similar assumption to Assumption~\ref{Assumption3}. Similar convergence results can be obtained using any other assumption of Section~\ref{asssad}. The convergence of Algorithm~\ref{inexact_solver_Dualalgorithm}, can be also easily derived using similar arguments with the one presented in Section~\ref{InexactSolvers} and the convergence guarantees of Theorem~\ref{MainTheorem}.

\begin{thm}(Convergence of dual objective).
Assume exactness. Let $y^0=0$ and let $\{y^k\}_{k=0}^\infty$ to be the dual iterates of iSDSA (Algorithm~\ref{inexact_dual}) with $\omega \in (0,2)$.  Set $x^*= \Pi_{\cL,\mB}(x^0)$ and let $y^*$ be any dual optimal solution. Consider the inexactness error $\epsilon^k_d$ be such that it satisfies 
 $\Exp[\|\bB^{-1}\bA^\top\epsilon^k_d\|^2_{\bB}\;|\;y^k,\bS_k]\leq\sigma_k^2= q^2 2 \left[D(y^*)-D(y^k)\right]$
where $0\leq q < 1-\sqrt{\rho}$. Then 
\begin{eqnarray}
\Exp[D(y^*)-D(y^k)] \leq \left(\sqrt{\rho}+q\right)^{2k} \left[ D(y^*)-D(y^0)\right].
\end{eqnarray}
\end{thm}
\begin{proof}
The proof follows by applying Theorem~\ref{ISGDwithq} together with Theorem~\ref{LemmaCorrespondance} and the identity $\frac{1}{2}\|x^k-x^*\|^2_\mB = D(y^*) - D(y^k)$ \eqref{identity}.
\end{proof}

Note that in the case that $q=0$, iSDSA simplifies to its exact variant SDSA and the convergence rate coincide with the one presented in Theorem~\ref{TheoremSDSA_IntroThesis}. Following similar arguments to those in \cite{gower2015stochastic}, the same rate can be proved for the duality gap $\Exp[P(x^k)-D(y^k)]$.

\section{Numerical Evaluation}
\label{experiments}
In this section we perform preliminary numerical tests for studying the computational behavior of iBasic with structured inexactness error when is used to solve the best approximation problem \eqref{BestApproximation_IntroThesis} or equivalently the stochastic optimization problem \eqref{StochReform_IntroThesis}\footnote{Note that from Section~\ref{InexactDualMethods} and the correspondence between the primal and dual methods, iSDSA will have similar behavior when is applied to the dual problem \eqref{DualProblem_IntroThesis}.}. As we have already mentioned, iBasic can be interpreted as sketch-and-project method, and as a result a comprehensive array of well-known algorithms can be recovered as special cases by varying the main parameters of the methods (Section \ref{Special Cases}).  In particular, in our experiments we focus on the evaluation of two popular special cases, the inexact Randomized Block Kaczmarz (iRBK) (equation \eqref{iRBK}) and inexact randomized block coordinate descent method (iRBCD) (equation \eqref{iRBCD})
We implement Algorithm~\ref{inexact_solver_algorithm} presented in Section~\ref{InexactSolvers} using CG \footnote{Recall that in order to use CG, the matrix $\bM_k$ that appears in linear system \eqref{linearsysteminCode} should be positive definite. This is true in the case that the matrix $\bA$ of the original system has full column rank matrix. Note however that the analysis of Section~\ref{InexactSolvers} holds for any consistent linear system $\bA x=b$ and without making any further assumption on its structure or the linearly convergence methods.} to inexactly solve the linear system of the update rule (equation \eqref{linearsysteminCode}). Recall that in this case we named the method InexactCG.

The convergence analysis of previous sections is quite general and holds for several combinations of the two main parameters of the method, the positive definite matrix $\bB$ and the distribution $\cD$ of the random matrices $\bS$. For obtaining iRBK as special case we have to choose $\bB=\bI \in \R^{n \times n}$ (Identity matrix) and for the iRBCD the given matrix $\bA $ should be positive definite and choose $\bB=\bA$. For both methods the distribution $\cD$ should be over random matrices $\bS=\bI_{:C}$ where $\bI_{:C}$ is the column concatenation of the $m \times m$ identity matrix indexed by a random subset $C$ of $[m]$. In our experiments we choose to have one specific distribution over these matrices. In particular, we assume that the random matrix in each iteration is chosen uniformly at random to be $\bS=\bI_{:d}$ with the subset $d$ of $[m]$ to have fixed pre-specified cardinality.

The code for all experiments is written in the Julia 0.6.3 programming language and run on a Mac laptop computer (OS X El Capitan), 2.7 GHz Intel Core i5 with 8 GB of RAM.

To coincide with the theoretical convergence results of Algorithm~\ref{inexact_solver_algorithm} the relaxation parameter (stepsize) of the methods study in our experiments is chosen to be $\omega=1$ (no relaxation). In all implementations, we use $x^0=0 \in \R^n$ as an initial point and in comparing the methods with their inexact variants we use the relative error measure $\|x^k-x^*\|^2_\bB / \|x^0-x^*\|^2_\bB \overset{x^0=0}{=}\|x^k-x^*\|^2_\bB / \|x^*\|^2_\bB $. We run each method (exact and inexact) until the relative error is below $10^{-5}$. For the horizontal axis we use either the number of iterations or the wall-clock time measured using the tic-toc Julia function. In the exact variants, the linear system \eqref{linearsysteminCode} in Algorithm~\ref{inexact_solver_algorithm} needs to be solved exactly. In our experiments we follow the implementation of \cite{gower2015randomized} for both exact RBCD and exact RBK where the built-in direct solver (sometimes referred to as "backslash") is used. 

\paragraph{Experimental setup:}
For the construction of consistent linear systems $\bA x=b$ we use the setup described in Section~\ref{subsectionEvaluation}. In particular, the linear systems used for the numerical evaluation of iRBK and iRBCD have been generated as described in Section~\ref{subsectionEvaluation} for algorithms mRK and mRCD, respectively.

\subsection{Importance of large block size}
Many recent works have shown that using larger block sizes can be very beneficial for the performance of randomized iterative algorithms \cite{gower2015randomized, richtarik2014iteration, RBK, LoizouRichtarik}. In Figure~\ref{BlockSizeRBKRBCD} we numerically verify this statement. We show that both RBK and RBCD (no inexact updates) outperform in number of iterations and wall clock time their serial variants where only one coordinate is chosen (block of size $d=1$) per iteration. This justify the necessity of choosing methods with large block sizes. Recall that this is precisely the class of algorithms that could have an expensive subproblem in their update rule which is required to be solved exactly and as a result can benefit the most from the introduction of inexactness.

\begin{figure}[t!]
\centering
\begin{subfigure}{.34\textwidth}
  \centering
  \includegraphics[width=1\linewidth]{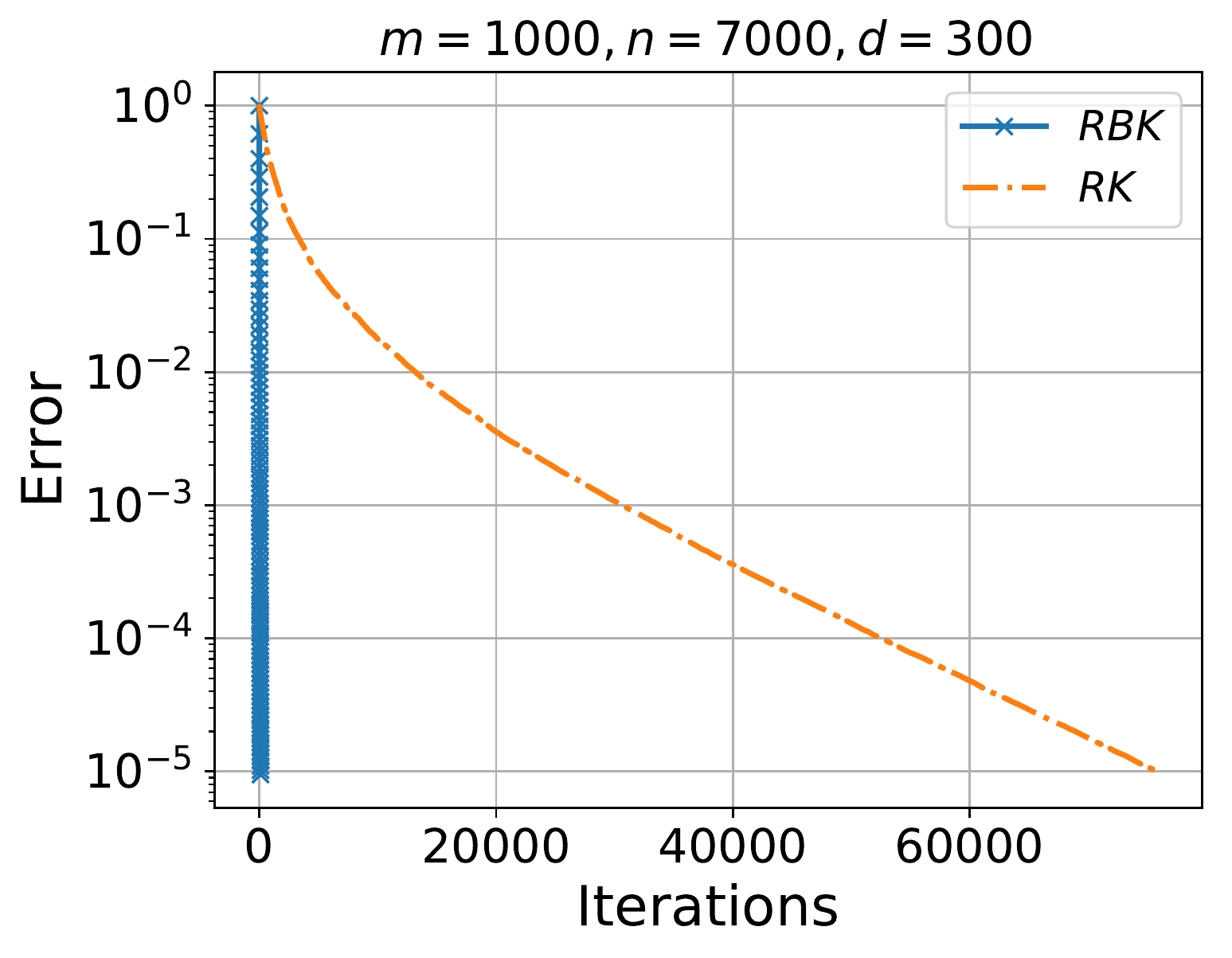}
\end{subfigure}%
\begin{subfigure}{.34\textwidth}
  \centering
  \includegraphics[width=1\linewidth]{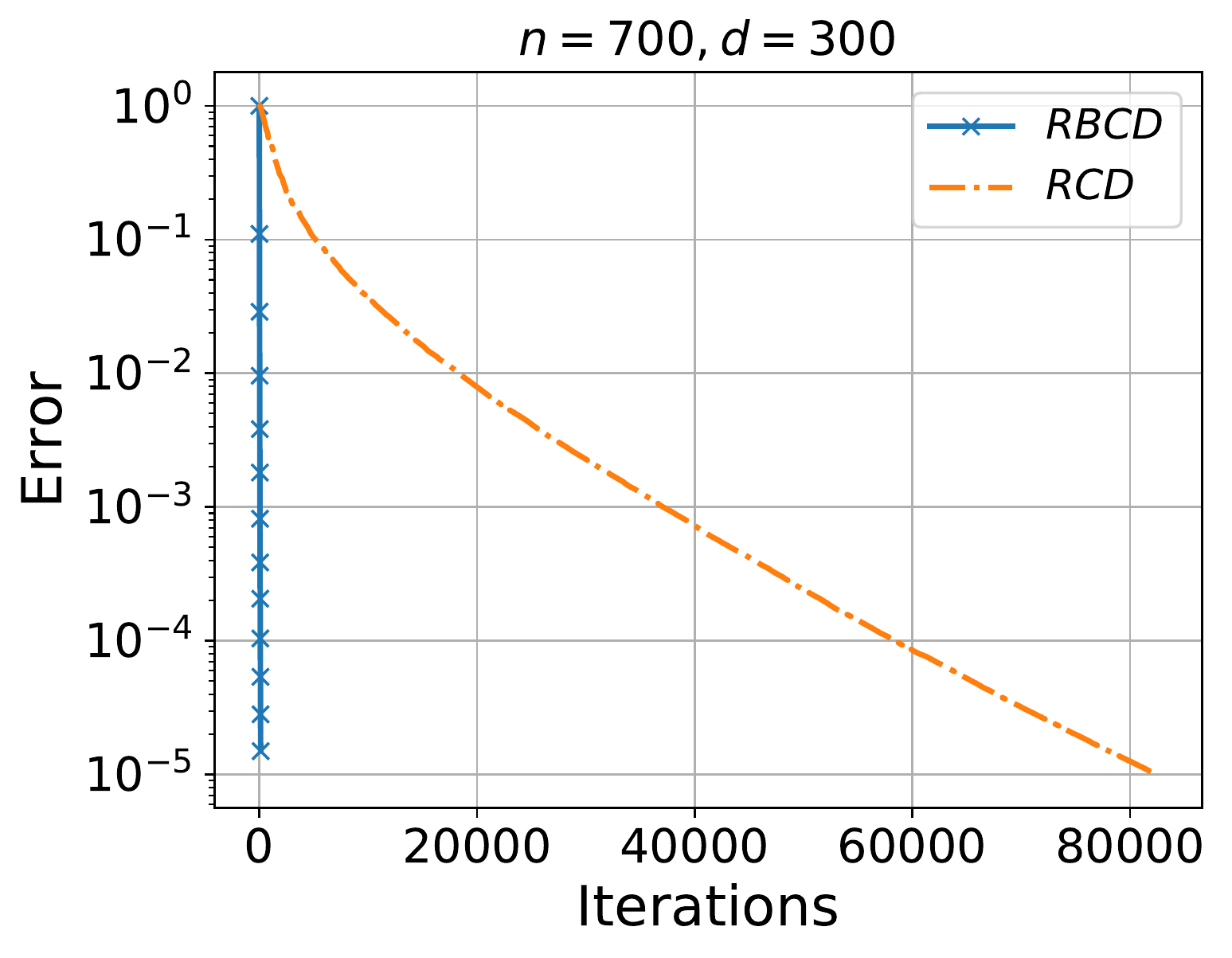}
\end{subfigure}\\
\begin{subfigure}{.34\textwidth}
  \centering
  \includegraphics[width=\linewidth]{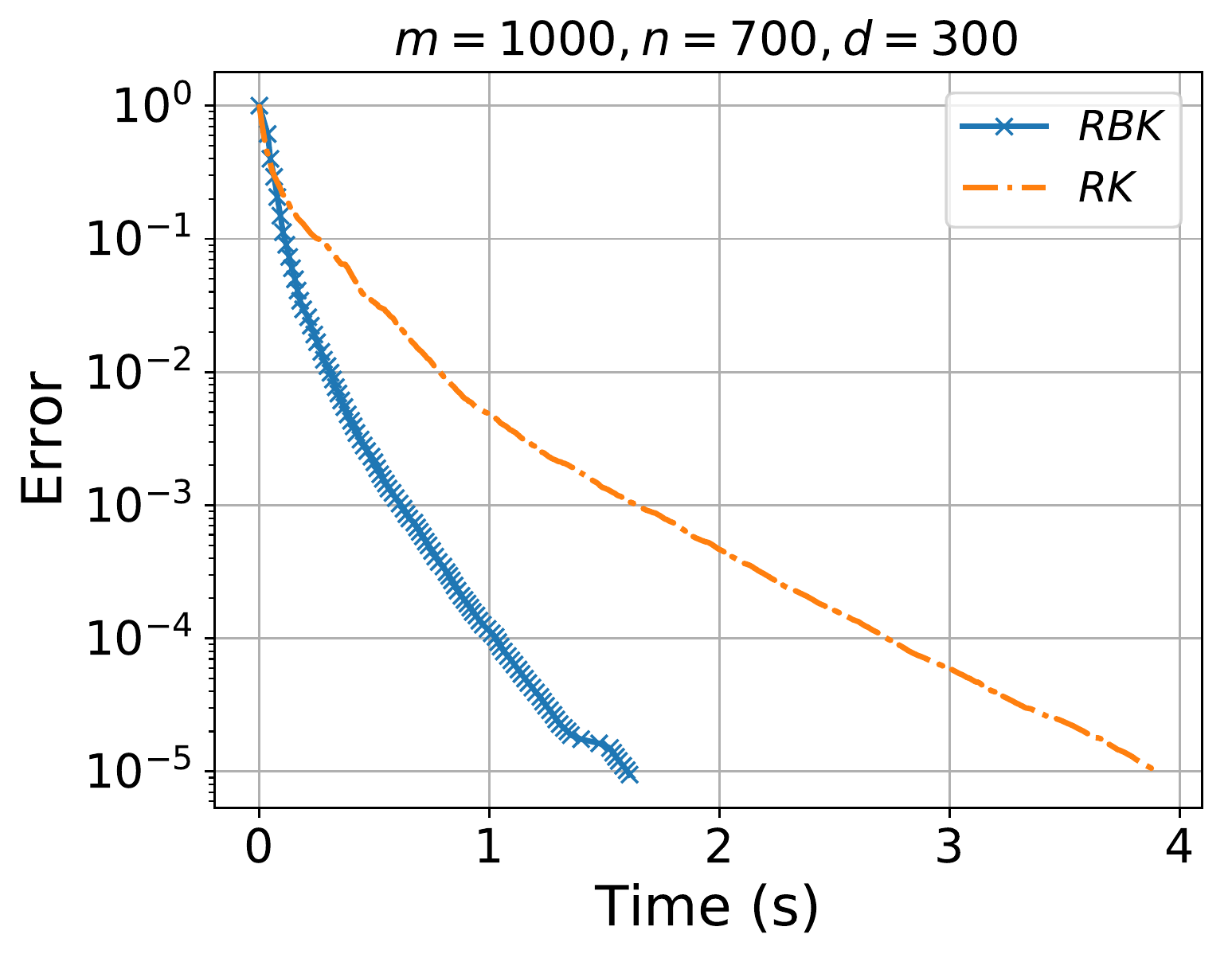}
  \caption{RK vs RBK}
\end{subfigure}
\begin{subfigure}{.34\textwidth}
  \centering
  \includegraphics[width=1\linewidth]{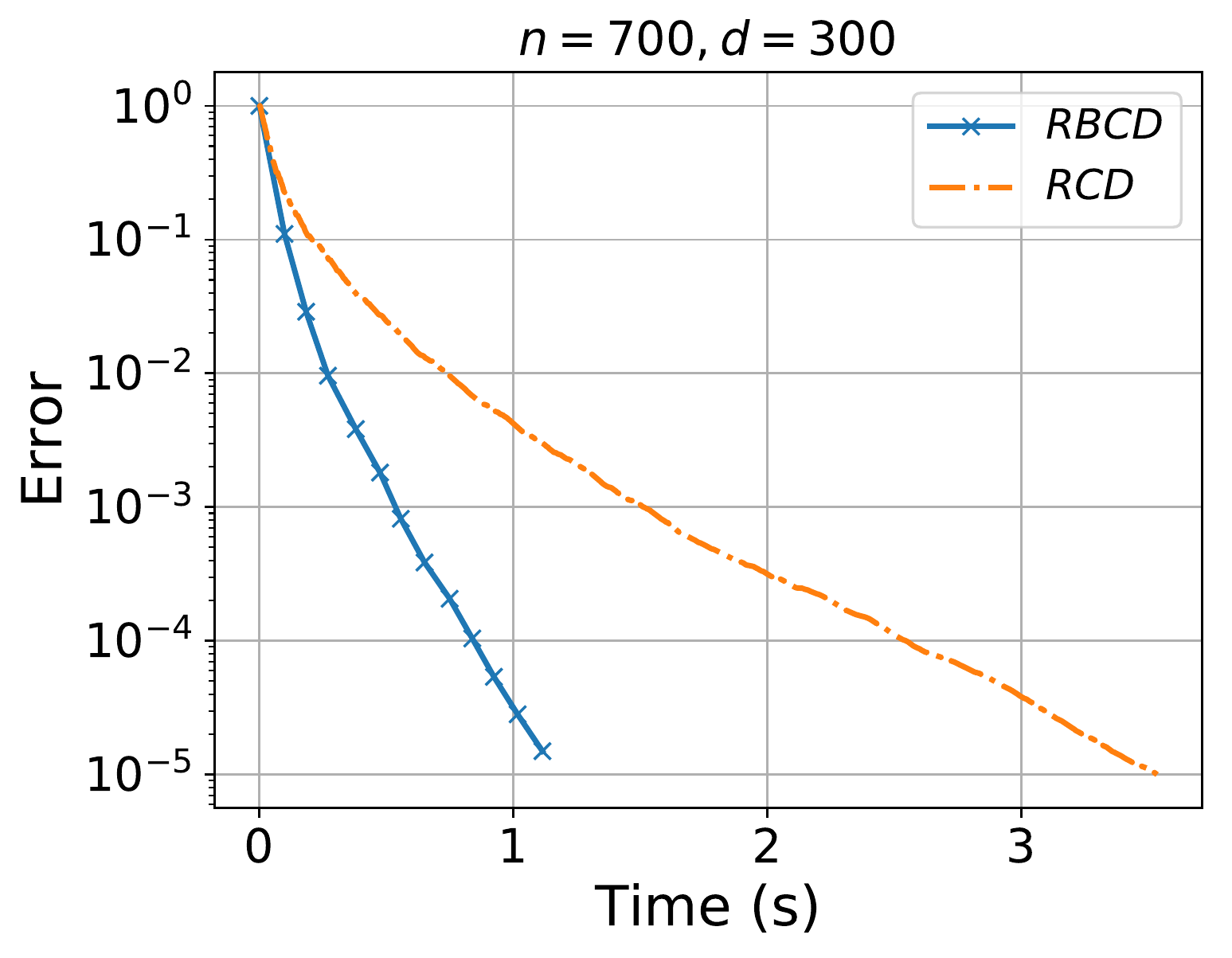}
  \caption{RCD vs RBCD }
\end{subfigure}\\
\caption{\small Comparison of the performance of the exact RBK and RBCD with their non-block variants RK and RCD. For the Kaczmarz methods (first column)  $\bA \in \R^{1000,700}$ is a Gaussian matrix and for the Coordinate descent methods (second column)  $\bA = \bP^\top \bP \in \R^{700 \times 700}$ where $\bP \in \R^{1000 \times 700}$ is Gaussian matrix.  To guarantee consistency $b=\bA z$ where $z$ is also Gaussian vector. The block size that chosen for the block variants is $d=300$.}
\label{BlockSizeRBKRBCD}
\end{figure}

\subsection{Inexactness and block size (iRBCD)}
In this experiment, we first construct a positive definite linear system following the previously described procedure for iRBCD. We first generate a Gaussian matrix $\bP \in \R^{10000 \times 7000}$ and then the positive definite matrix $\bA = \bP^\top \bP \in \R^{7000 \times 7000}$ is used to define a consistent liner system. We run iRBCD in this specific linear system and compare its performance with its exact variance for several block sizes $d$ (numbers of column of matrix $\bS$). For evaluating the inexact solution of the linear system in the update rule we run CG for either 2, 5 or 10 iterations. In Figure~\ref{iRBCDfigure}, we plot the evolution of the relative error in terms of both the number of iterations and the wall-clock time.

We observe that for any block size the inexact methods are always faster in terms of wall clock time than their exact variants even if they require (as is expected) equal or larger number of iterations. Moreover it is obvious that the performance of the inexact method becomes much better than the exact variant as the size $d$ increases and as a results the sub-problem that needs to be solved in each step becomes more expensive. It is worth to highlight that for the chosen systems, the exact RBCD behaves better in terms of wall clock time as the size of block increases (this coincides with the findings of the previous experiment).

\begin{figure}[t!]
\centering
\begin{subfigure}{.24\textwidth}
  \centering
  \includegraphics[width=1\linewidth]{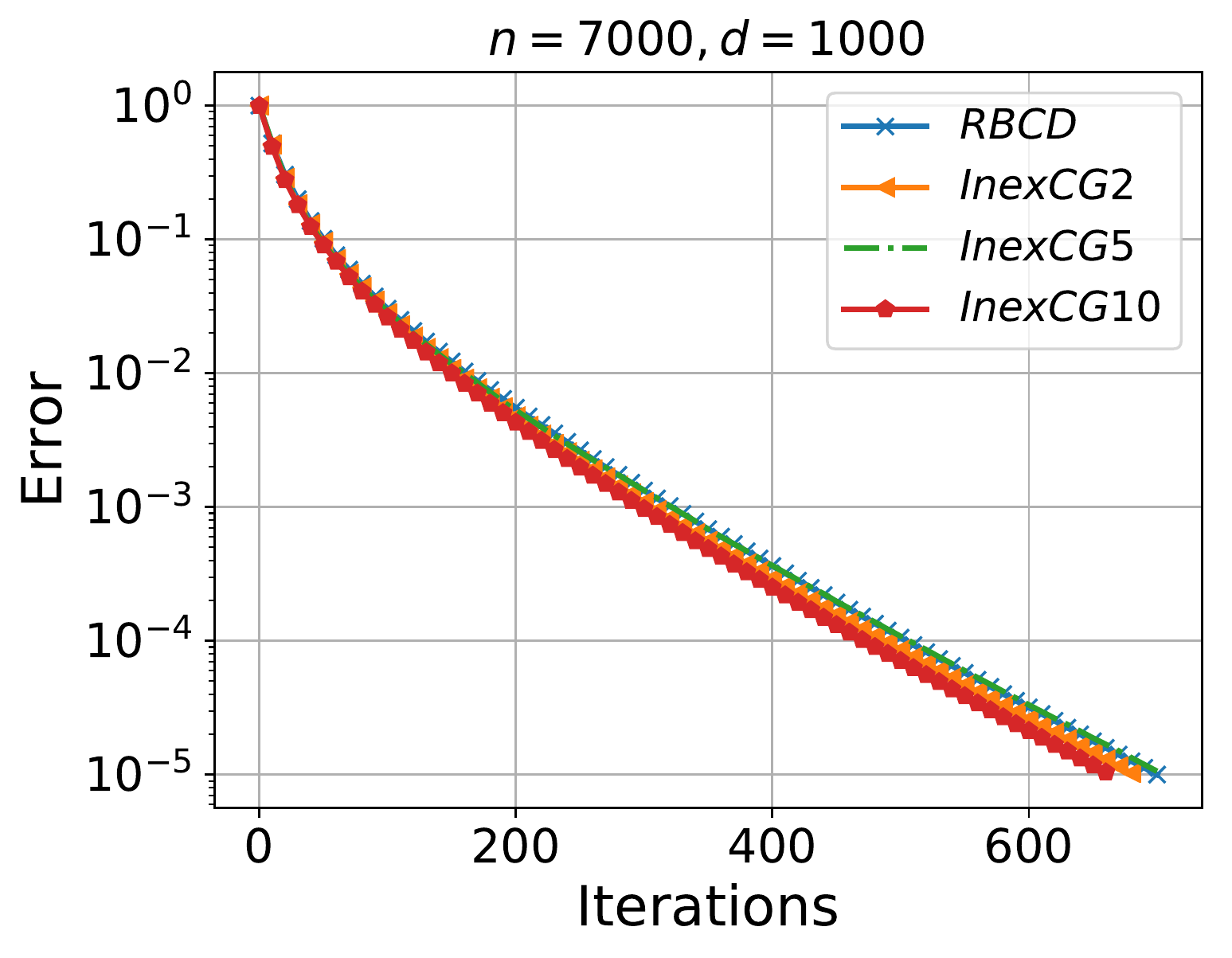}
\end{subfigure}%
\begin{subfigure}{.24\textwidth}
  \centering
  \includegraphics[width=1\linewidth]{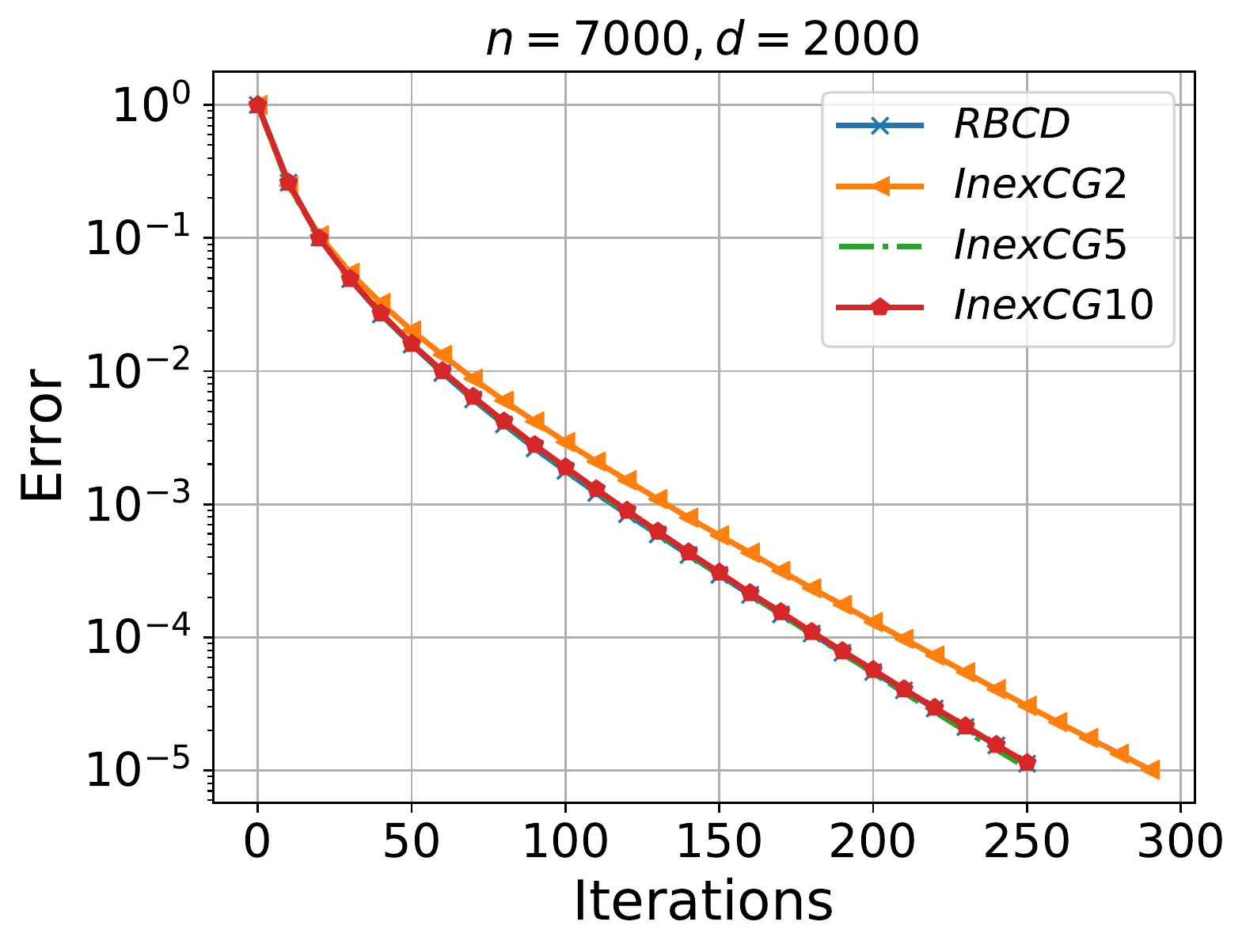}
\end{subfigure}%
\begin{subfigure}{.24\textwidth}
  \centering
  \includegraphics[width=1\linewidth]{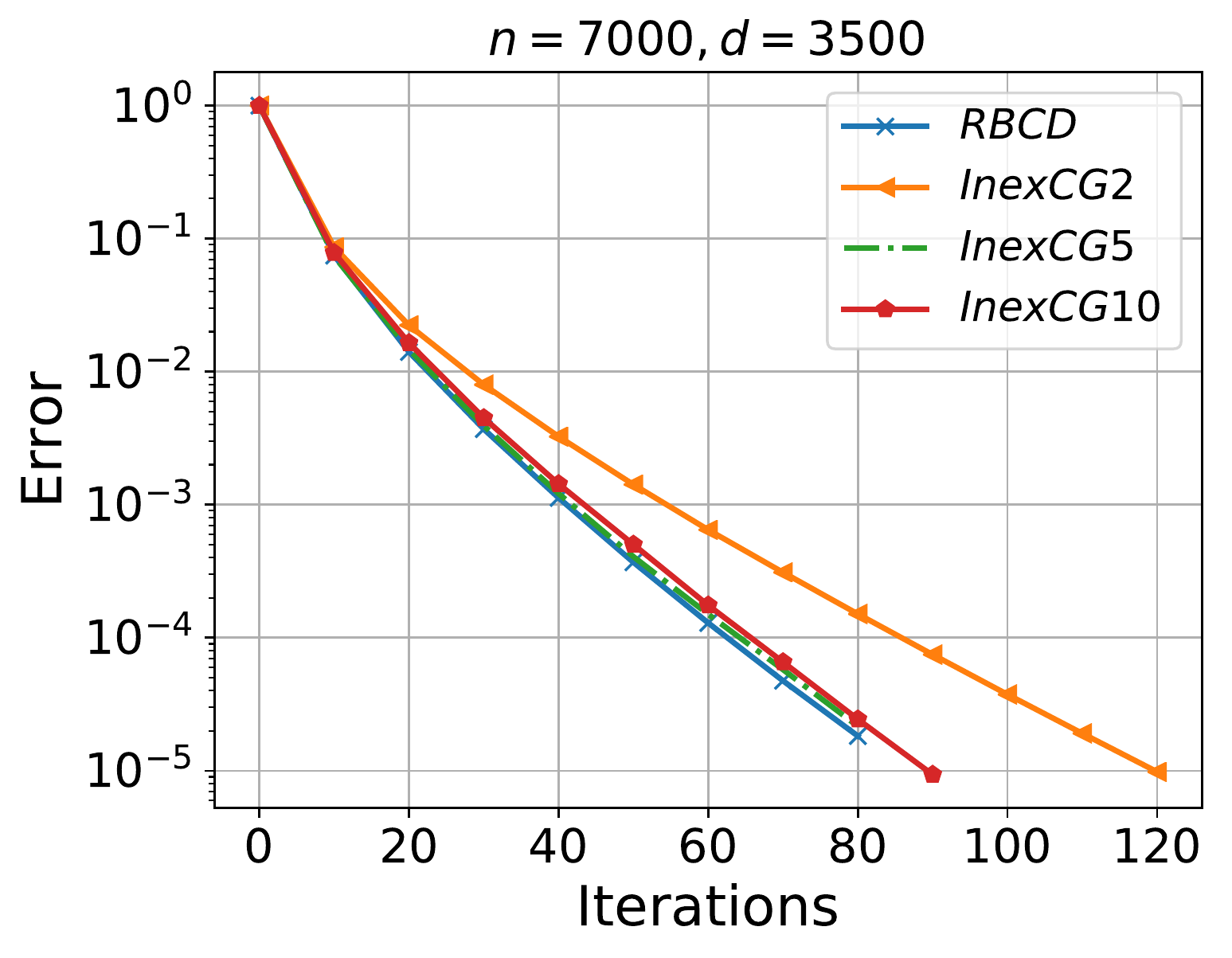}
\end{subfigure}%
\begin{subfigure}{.24\textwidth}
  \centering
  \includegraphics[width=1\linewidth]{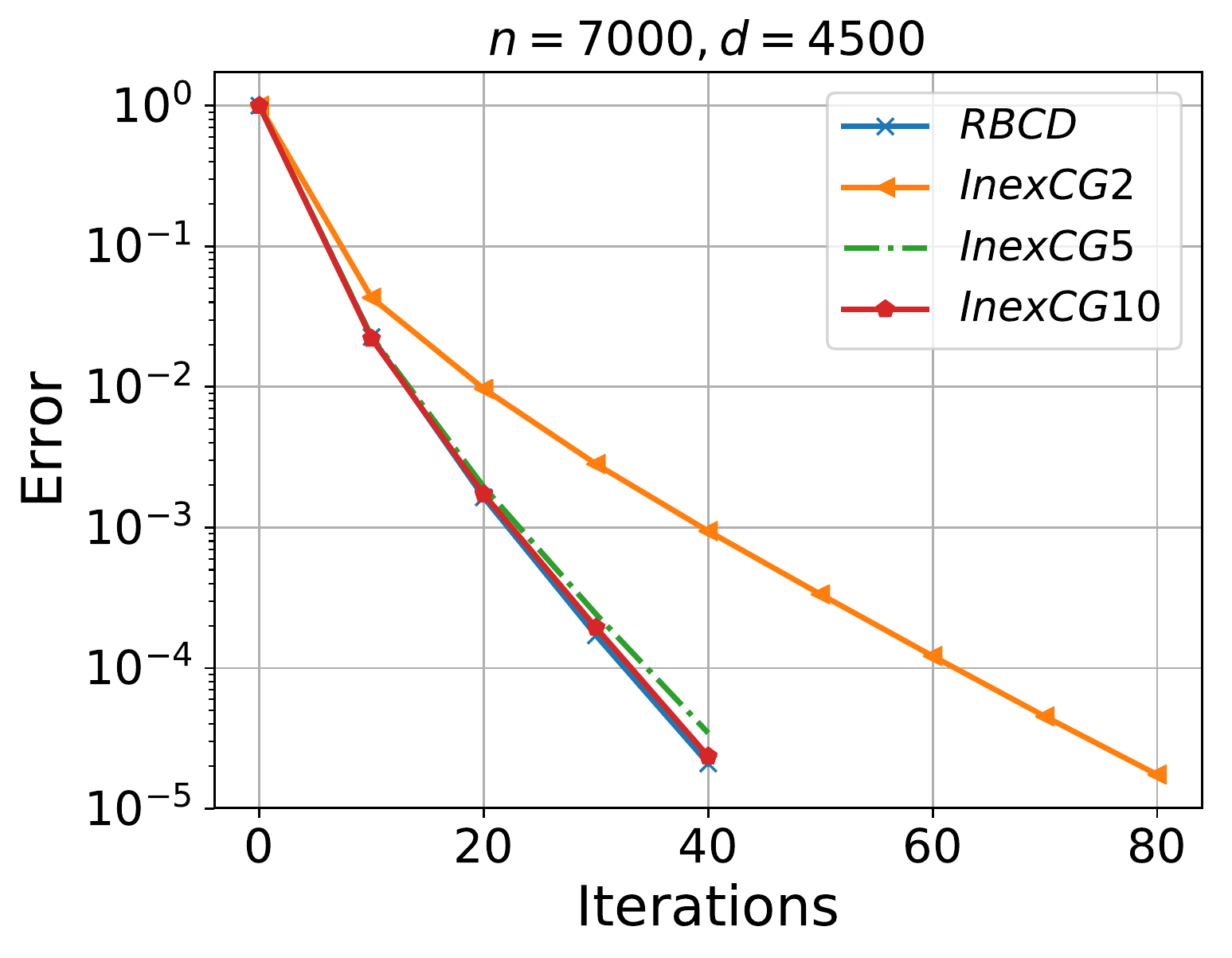}
\end{subfigure}\\
\begin{subfigure}{.24\textwidth}
  \centering
  \includegraphics[width=1\linewidth]{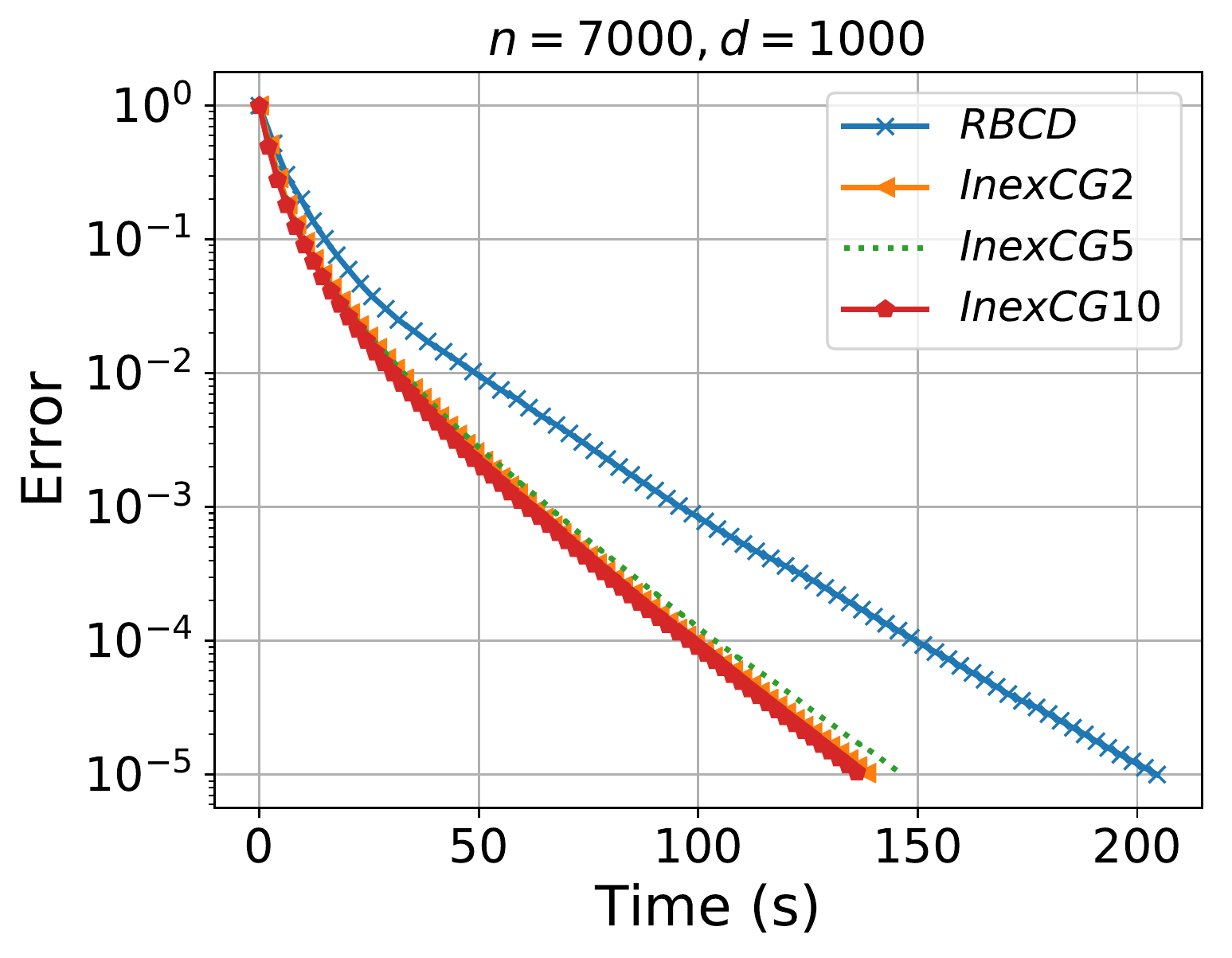}
  \caption{d=1000}
\end{subfigure}
\begin{subfigure}{.24\textwidth}
  \centering
  \includegraphics[width=1\linewidth]{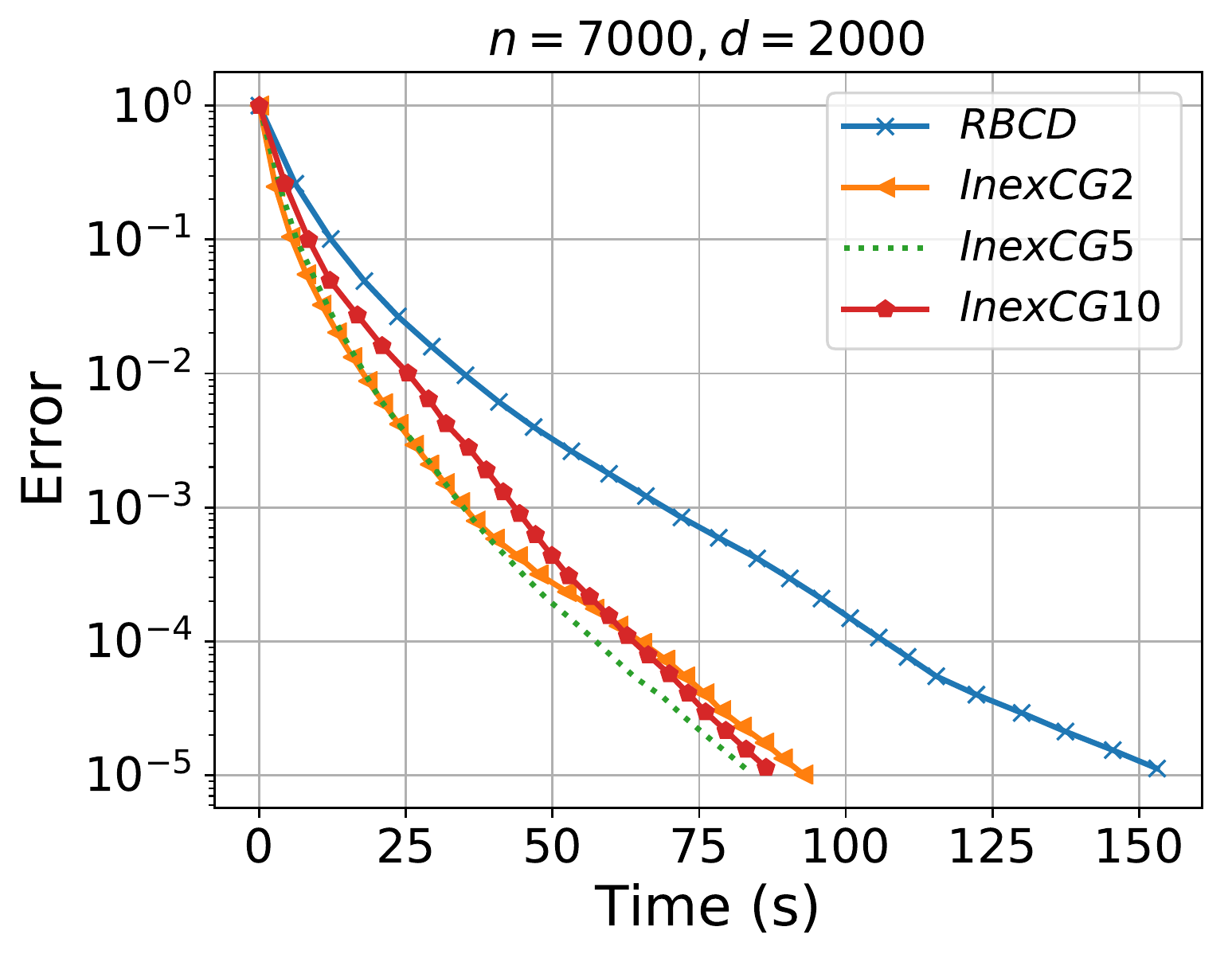}
  \caption{d=2000}
\end{subfigure}
\begin{subfigure}{.24\textwidth}
  \centering
  \includegraphics[width=1\linewidth]{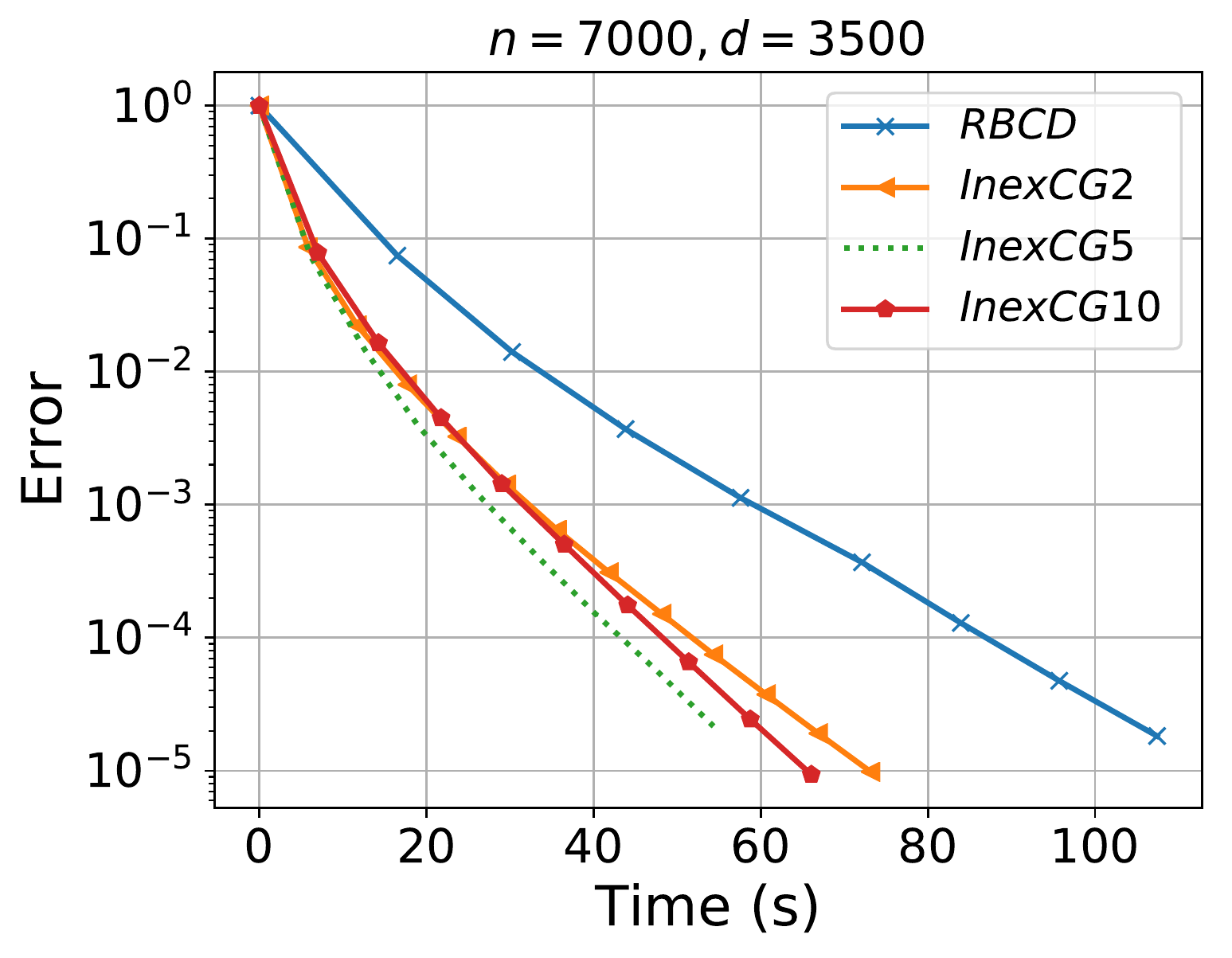}
  \caption{d=3500}
\end{subfigure}
\begin{subfigure}{.24\textwidth}
  \centering
  \includegraphics[width=1\linewidth]{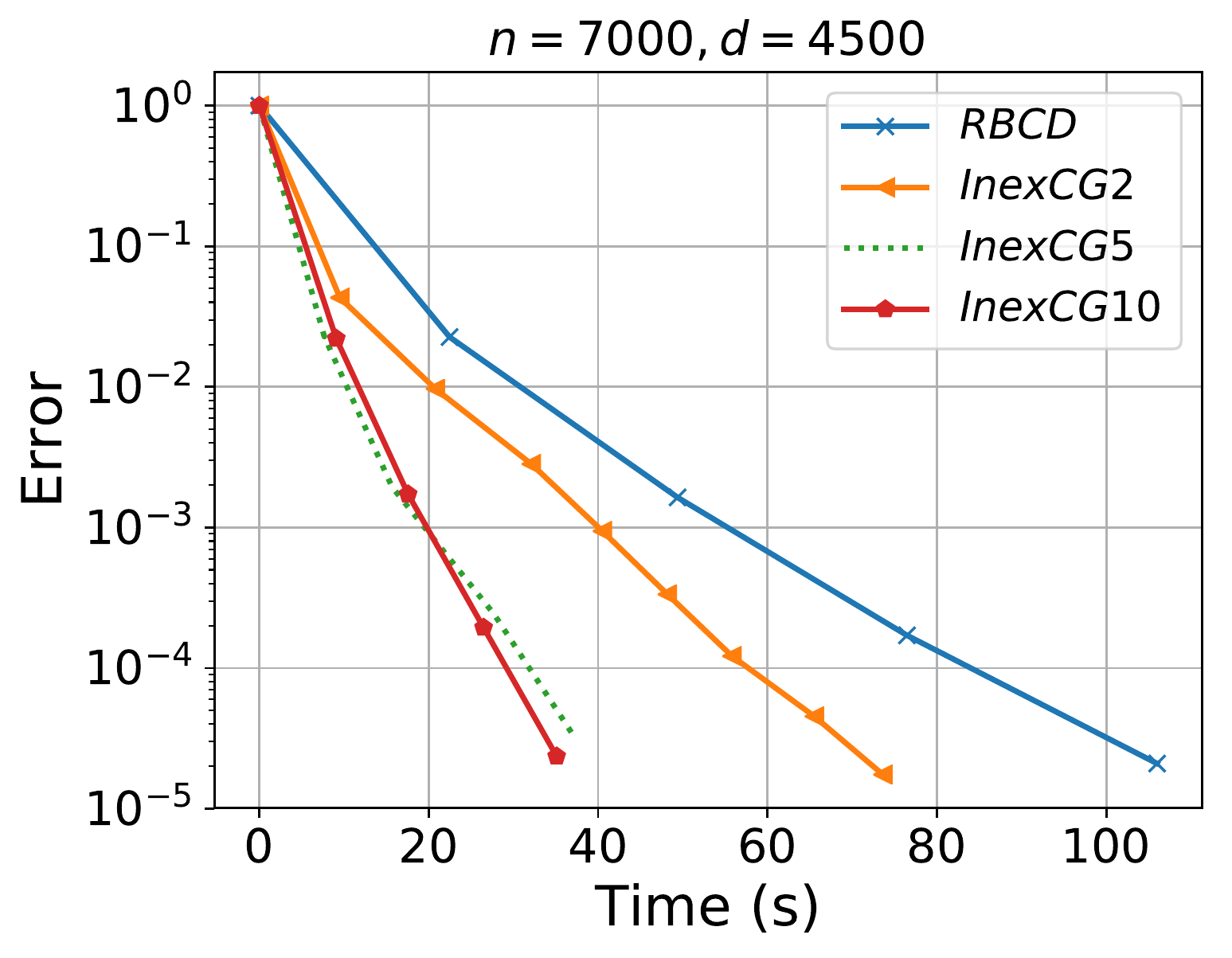}
  \caption{d=4500}
\end{subfigure}\\
\caption{\small Performance of iRBCD (InexactCG) and exact RBCD for solving a consistent linear systems with $\bA = \bP^\top \bP \in \R^{7000 \times 7000}$, where $\bP \in \R^{10000 \times 7000}$ is a Gaussian matrix. The right hand side for the system is chosen to be $b=\bA z$ where $z$ is also a Gaussian vector. Several block sizes are used: $d=1000, 2000, 3500, 4500.$ The graphs in the first (second) row plot the iterations (time) against relative error $\|x^k-x^*\|^2_\bA / \|x^*\|^2_\bA $.}
\label{iRBCDfigure}
\end{figure}

\subsection{Evaluation of iRBK}
In the last experiment we evaluate the performance of iRBK in both synthetic and real datasets. For computing the inexact solution of the linear system in the update rule we run CG for pre-specified number of iterations that can vary depending the datasets.  In particular, we compare iRBK and RBK on synthetic linear systems generated with the Julia Gaussian matrix functions ``randn(m,n)" and ``sprandn(m,n,r)" (input $r$ of sprandn function indicates the density of the matrix). For the real datasets, we test the performance of iRBK and RBK using real matrices from the library of support vector machine problems LIBSVM \cite{chang2011libsvm}. Each dataset of the LIBSVM consists of a matrix $\bA \in \R^{m \times n}$ ($m$ features and $n$ characteristics) and a vector of labels $b \in \R^m$. In our experiments we choose to use only the matrices of the datasets and ignore the label vectors \footnote{Note that the real matrices of the Splice and Madelon datasets are full rank matrices.}. As before, to ensure consistency of the linear system, we choose a Gaussian vector $z\in\R^n$ and the right hand side of the linear system is set to $b = \mA z$ (for both the synthetic and the real matrices). By observing Figure~\ref{RKinsideRBKreal} it is clear that for all problems under study the performance of iRBK in terms of wall clock time is much better than its exact variant RBK.

\begin{figure}[t!]
\centering
\begin{subfigure}{.24\textwidth}
  \centering
  \includegraphics[width=1\linewidth]{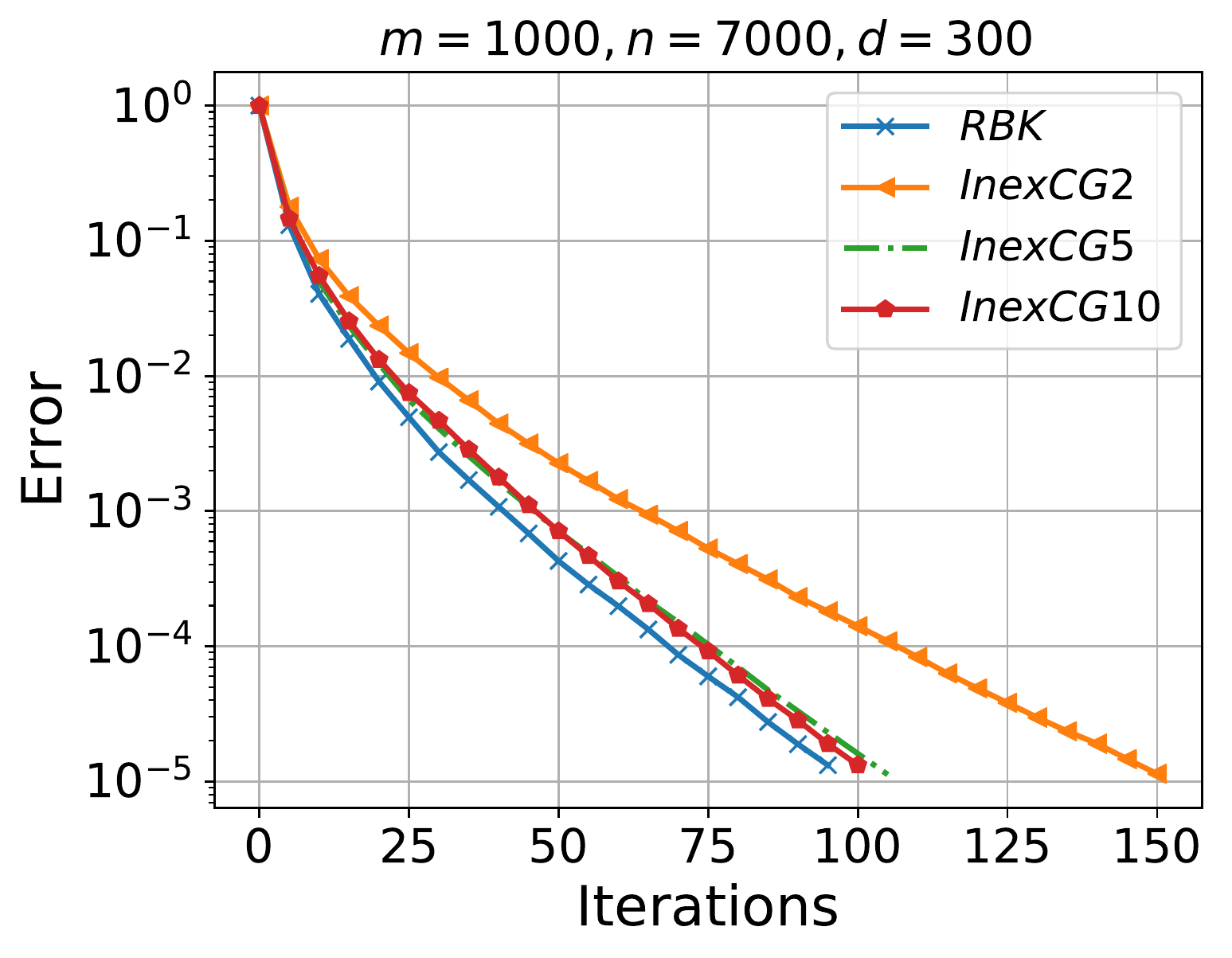}
\end{subfigure}%
\begin{subfigure}{.24\textwidth}
  \centering
  \includegraphics[width=1\linewidth]{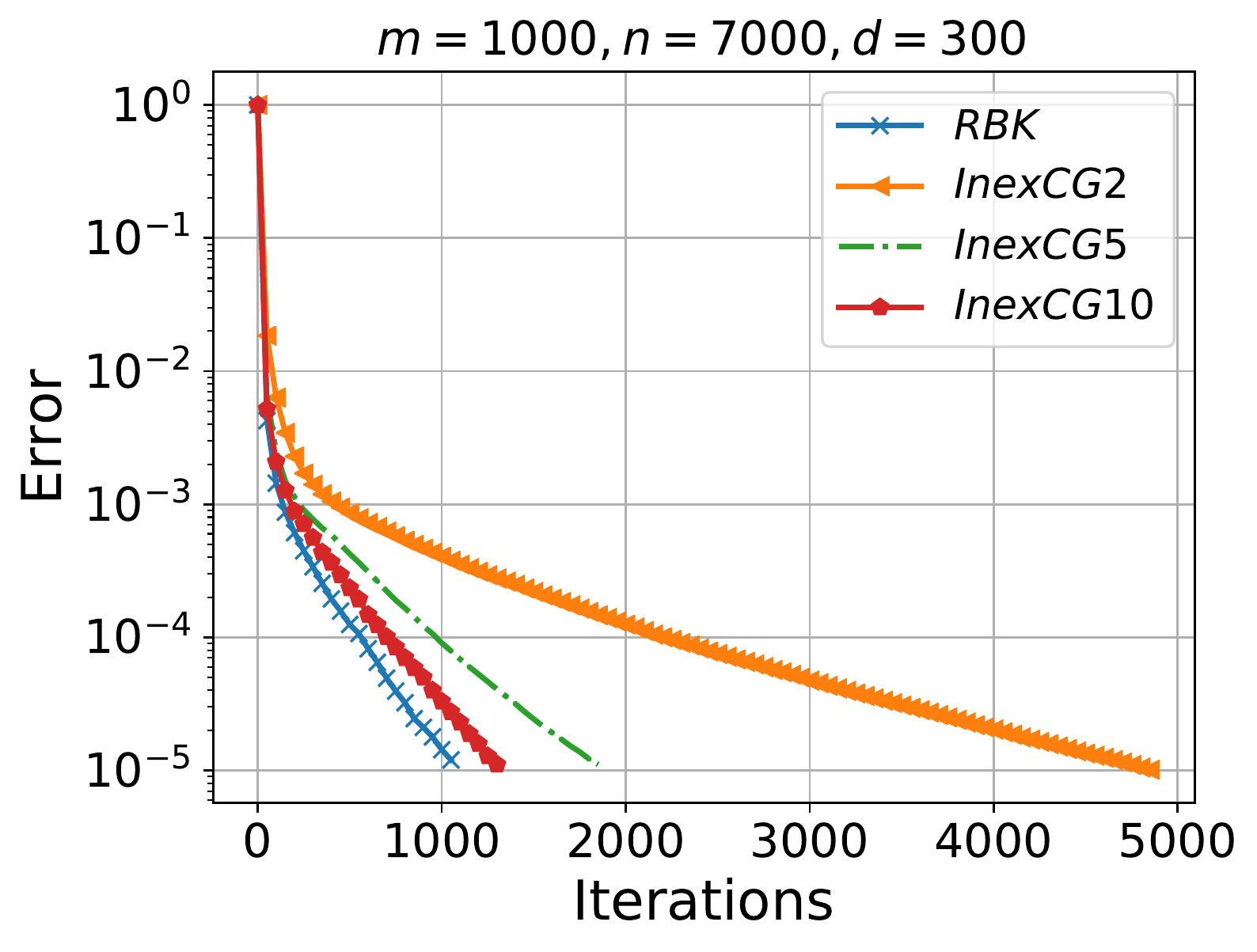}
\end{subfigure}
\begin{subfigure}{.24\textwidth}
  \centering
  \includegraphics[width=\linewidth]{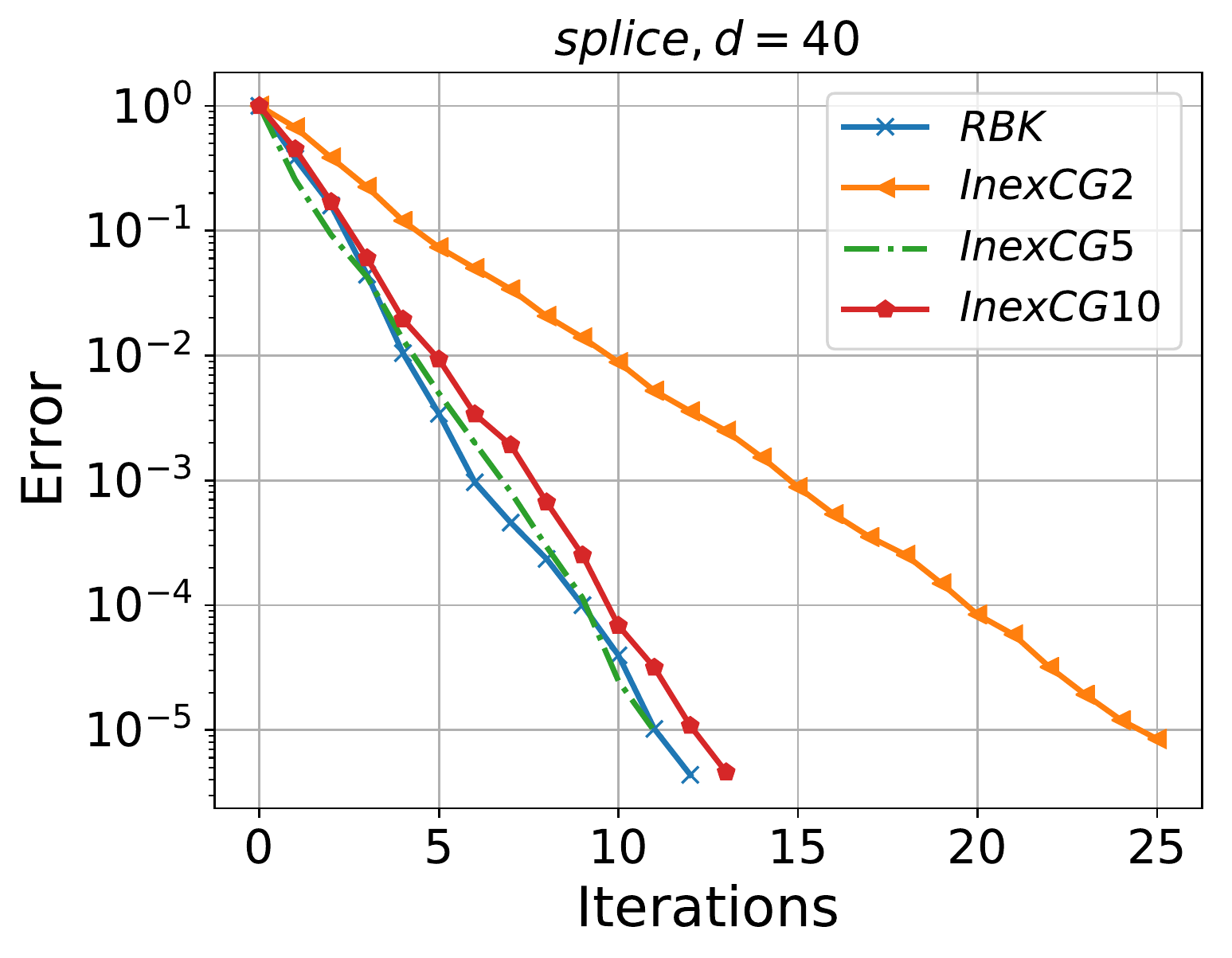}
\end{subfigure}
\begin{subfigure}{.24\textwidth}
  \centering
  \includegraphics[width=1\linewidth]{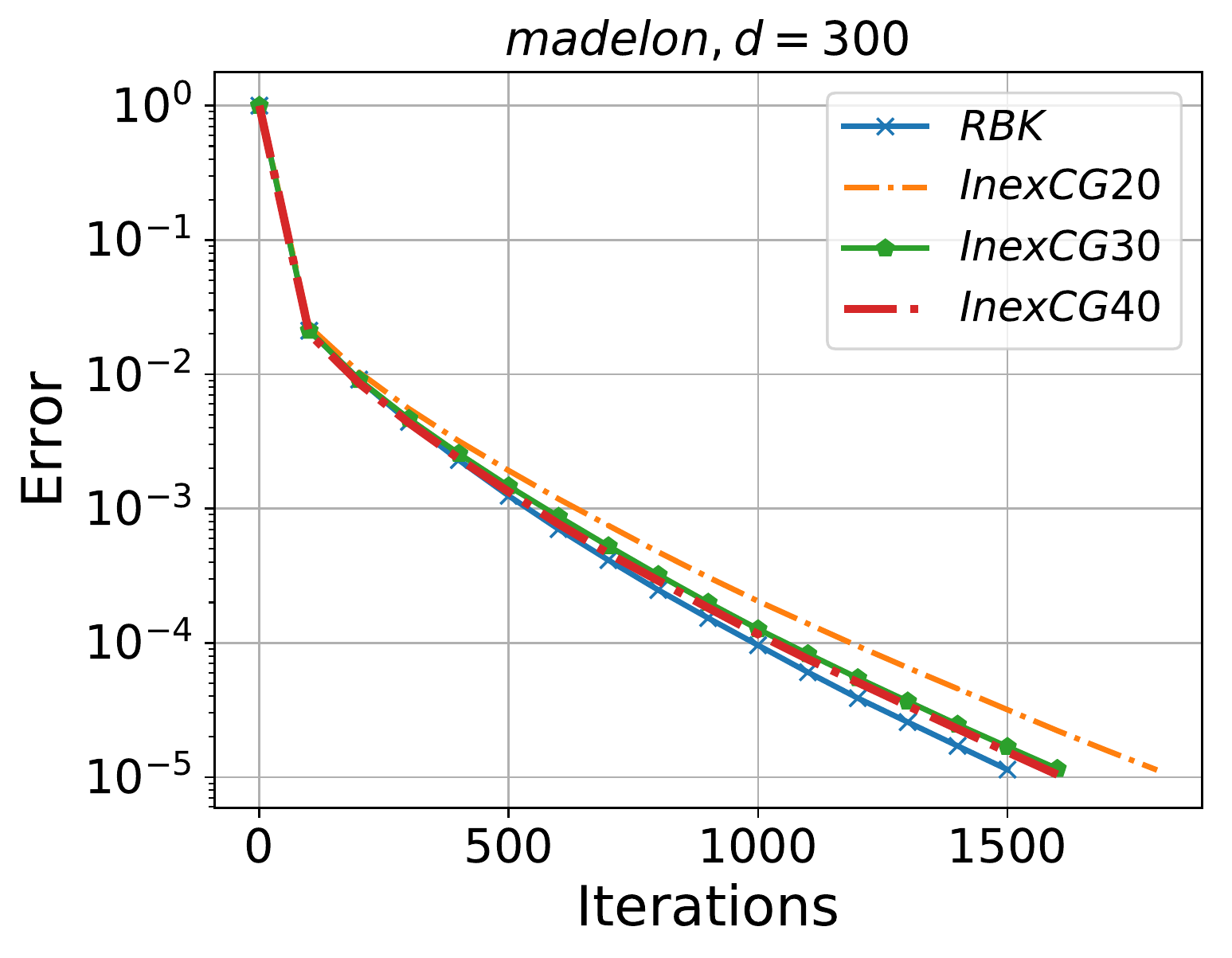}
\end{subfigure}\\
\begin{subfigure}{.24\textwidth}
  \centering
  \includegraphics[width=1\linewidth]{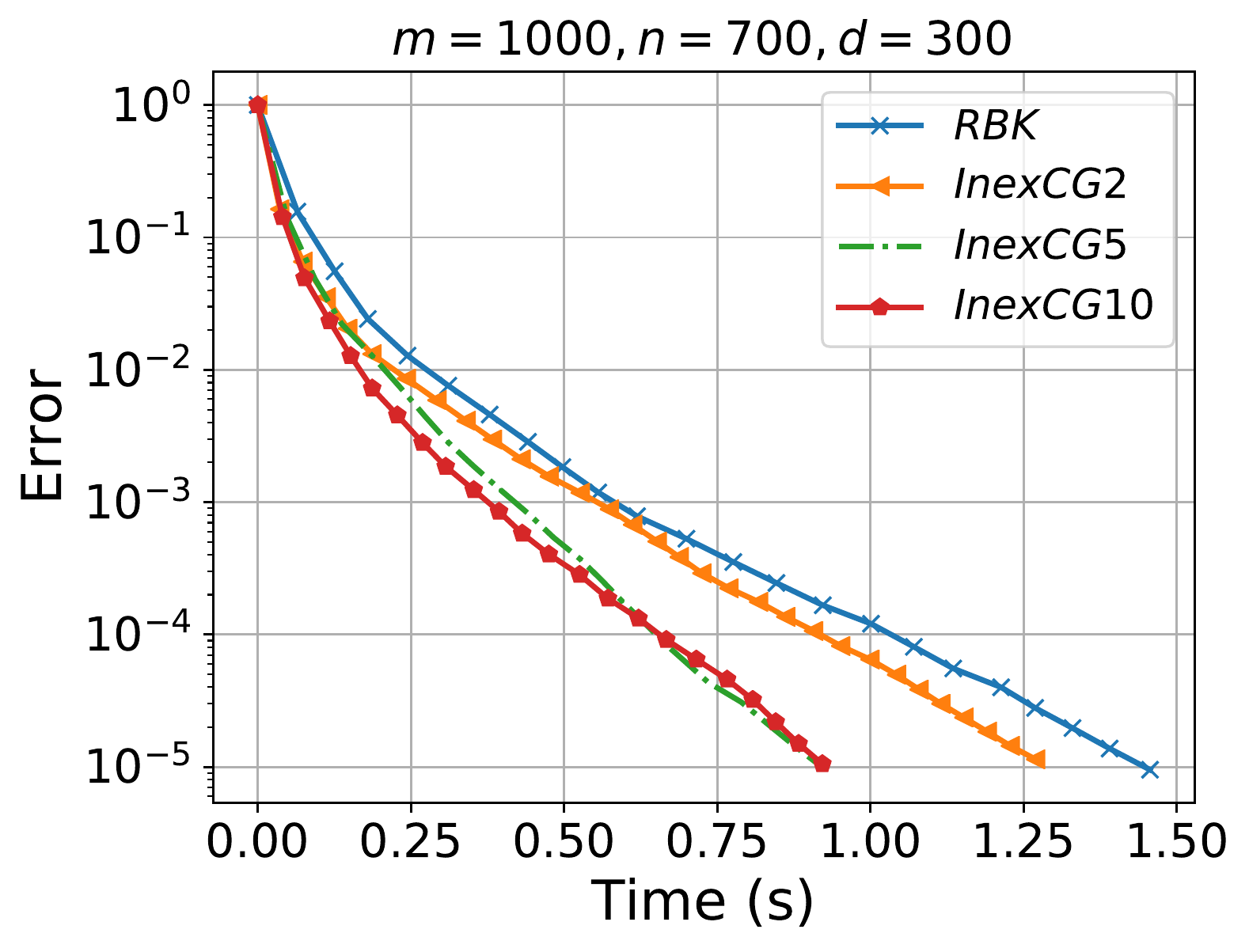}
  \caption{randn(m,n)}
\end{subfigure}%
\begin{subfigure}{.24\textwidth}
  \centering
  \includegraphics[width=1\linewidth]{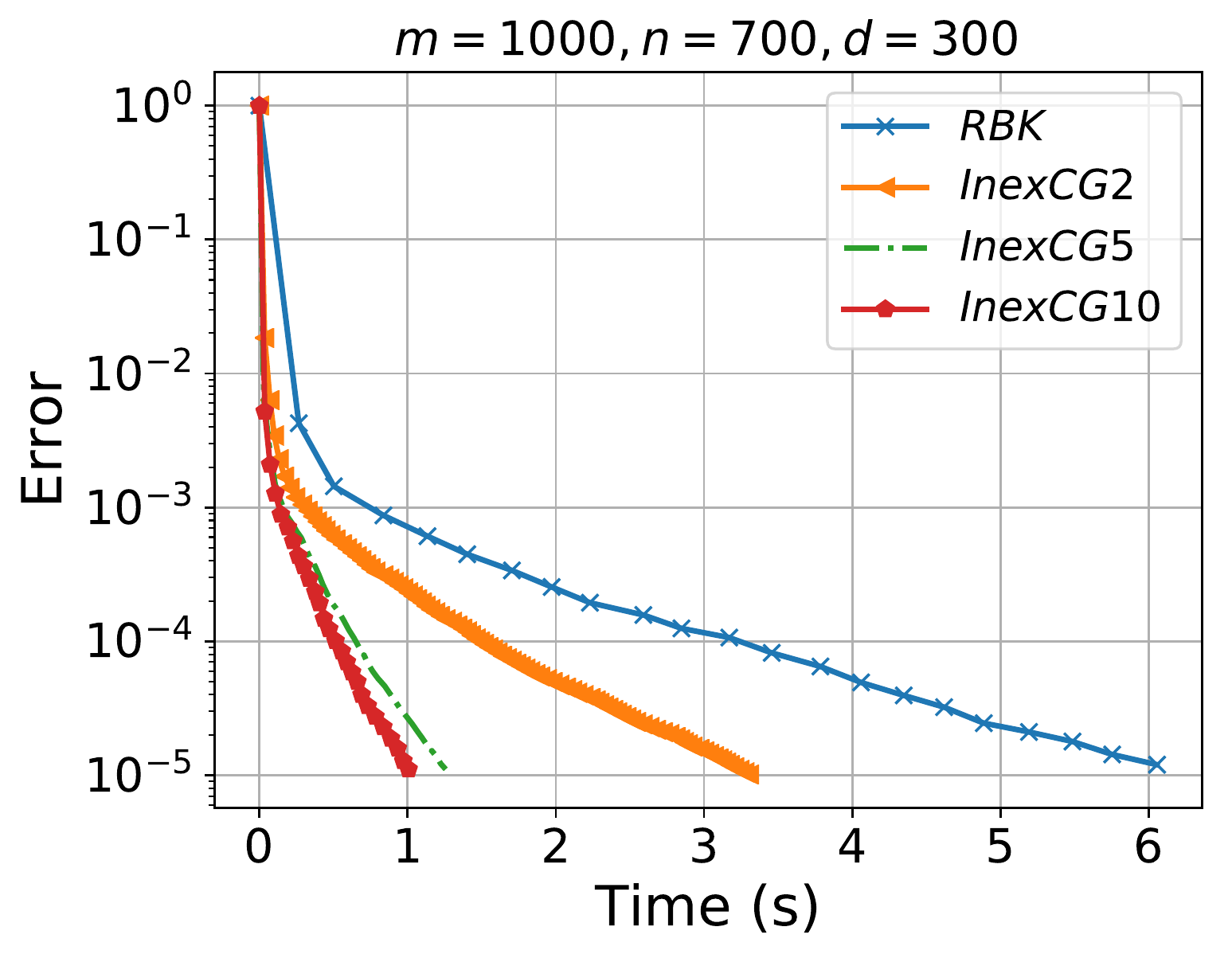}
  \caption{sprandn(m,n,0.01)}
\end{subfigure}
\begin{subfigure}{.24\textwidth}
  \centering
  \includegraphics[width=\linewidth]{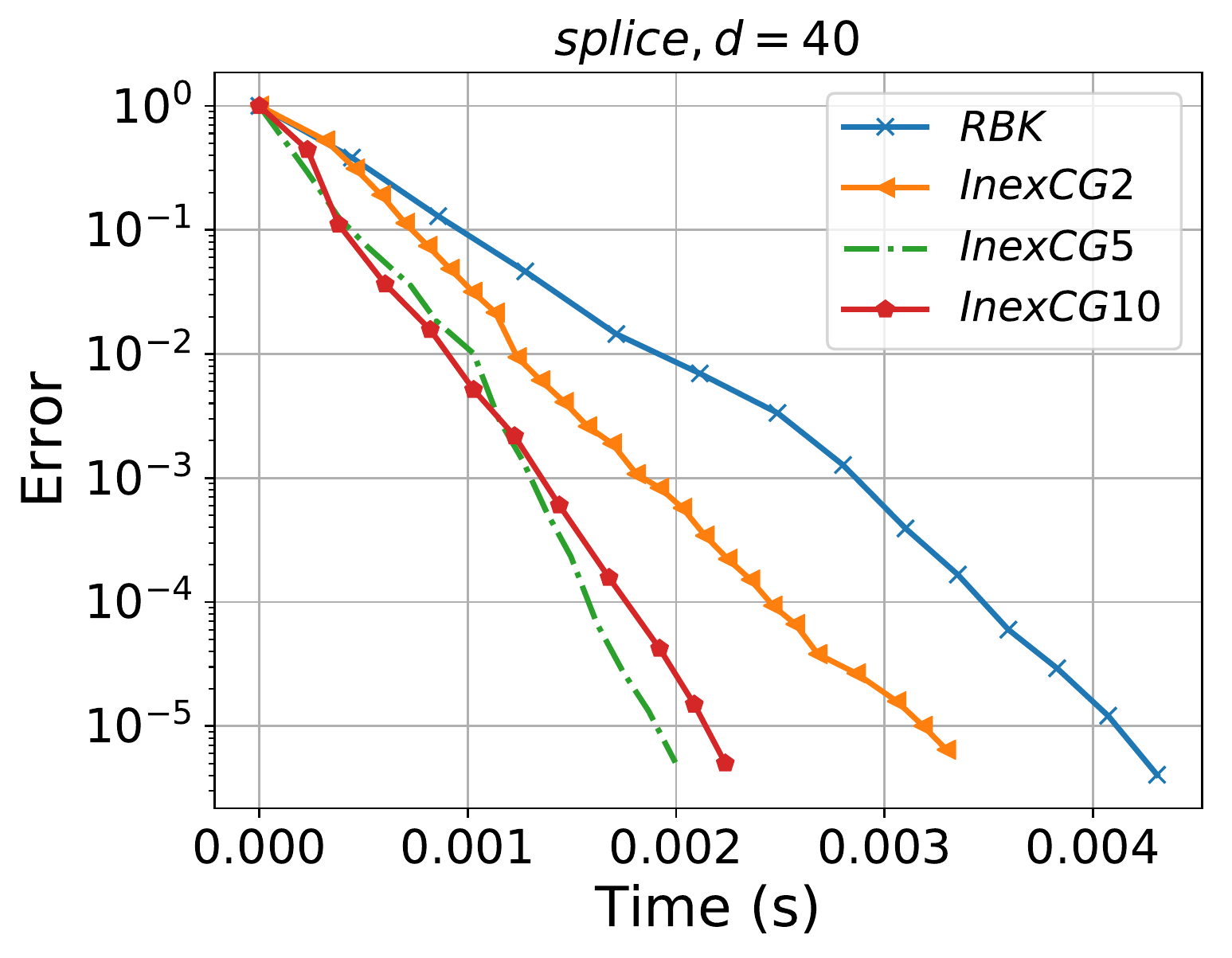}
  \caption{splice}
\end{subfigure}
\begin{subfigure}{.24\textwidth}
  \centering
  \includegraphics[width=1\linewidth]{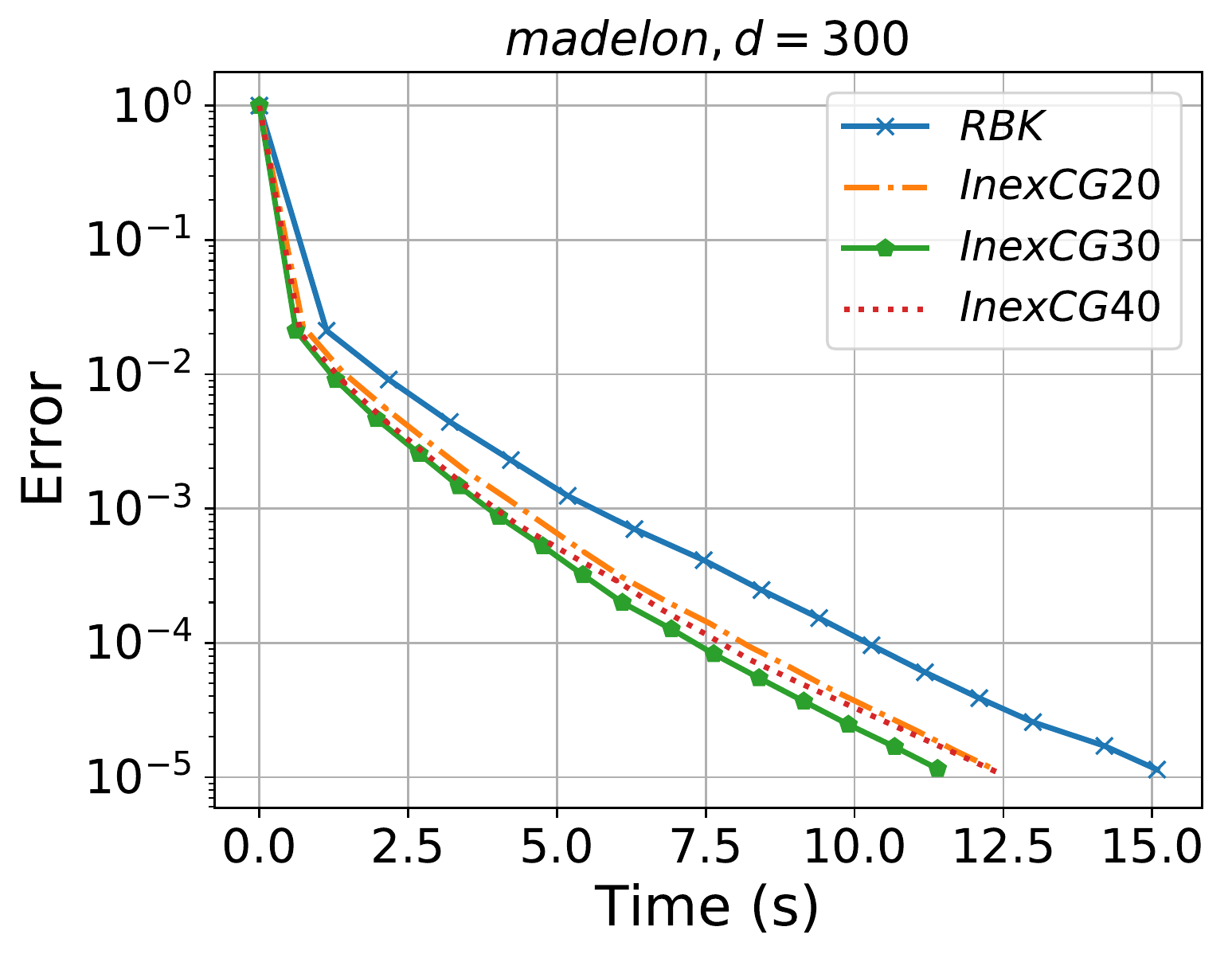}
  \caption{madelon}
\end{subfigure}\\
\caption{\small The performance of iRBK (InexactCG) and RBK on synthetic and real datasets. Synthetic matrices: (a) randn(m,n) with (m,n)=(1000,700), (b) sprandn(m,n,0.01) with (m,n)=(1000,700). Real Matrices from LIBSVM \cite{chang2011libsvm} :  (c) splice : (m,n)=(1000,60), (d) madelon: (m,n)=(2000,500). The graphs in the first (second) row plot the iterations (time) against relative error $\|x^k-x^*\|^2 / \|x^*\|^2$. The quantity $d$ in the title of each plot indicates the size of the block size for both iRBK and RBK.}
\label{RKinsideRBKreal}
\end{figure}

\section{Conclusion}
\label{conlcusion}

In this chapter we propose and analyze inexact variants of several stochastic algorithms for solving quadratic optimization problems and linear systems. We provide linear convergence rate under several assumptions on the inexactness error. The proposed methods require more iterations than their exact variants to achieve the same accuracy. However, as we show through our numerical evaluations, the inexact algorithms require significantly less time to converge. 

With the continuously increasing size of datasets, inexactness should  definitely be a tool that practitioners should use in their implementations even in the case of stochastic methods that have much cheaper-to-compute iteration complexity than their deterministic variants.  Recently, accelerated and parallel stochastic optimization methods \cite{loizou2017momentum, ASDA, tu2017breaking} have been proposed for solving linear systems. We speculate that the addition of inexactness to these update rules will lead to methods faster in practice. We also believe that our approach and complexity results can be extended to the more general case of minimization of convex and non-convex functions in the stochastic setting.  

\section{Proofs of Main Results}
In our convergence analysis we use several popular inequalities. Look Table~\ref{inequalities}  for the abbreviations and the relevant formulas.

A key step in the proofs of the theorems is to use the tower property of the expectation. We use it in the form
\begin{eqnarray}\label{eq:tower3}
\Exp[\Exp[\Exp[ X  \;|\; x^k,  \mS_k] \;|\; x^k]] = \Exp[X],
\end{eqnarray}
where $X$ is some random variable. In all proofs we perform the three expectations in order, from the innermost to the outermost. Similar to the main part of this chapter we use $\rho=1 - \omega(2-\omega)\lambda_{\rm min}^+$.

The following remark on random variables is also used in our proofs.
 \begin{rem}\label{variance_ineq}
Let $x$ and $y$ be random vectors and let $\sigma$ positive constant. If we assume $\Exp[\|x\|_{\mB}^2\;|\;y] \leq \sigma^2$ then by using the variance inequality (check Table~\ref{inequalities}) we obtain $\Exp[\|x\|_{\mB}\;|\;y] \leq \sigma$.
In our setting if we assume $\Exp[\|\epsilon^k\|_{\mB}^2\;|\;x^k,\bS_k] \leq \sigma_k^2$ where $\epsilon^k$ is the inexactness error and $x^k$ is the current iterate then by the variance inequality it holds that $\Exp[\|\epsilon^k\|_{\mB}\;|\;x^k,\bS_k] \leq \sigma_k$.
\end{rem}

 \begin{table}[H]
\begin{center}
 \begin{tabular}{ |p{3cm}||p{2.3cm}||p{4cm}||p{2.9cm}|  }
 \hline
 \multicolumn{4}{|c|}{Useful inequalities} \\
 \hline
\begin{tabular}{c}  Inequalities \\ (Full names) \end{tabular} & Abbreviations & Formula  & Assumptions\\
 \hline
 \hline
Jensen Inequality & \textit{ Jensen } & $f[\Exp(x)]\leq \Exp[f(x)]$ &  $f$ is convex\\
 \hline
 Conditioned Jensen &   \textit{ C.J. }  & $f(\Exp[x \;|\; s]) \leq \Exp[f(x) \;|\; s]$ &  $f$ is convex\\
  \hline
\begin{tabular}{c}    Cauchy-Swartz \\ (B-norm) \end{tabular} & \textit{ C.S. } &  $| \langle a,b \rangle_{\bB}| \leq \|a\|_{\bB} \|b\|_{\bB}$ & $a,b \in \R^n$ \\
   \hline
 Variance Inequality &  \textit{ V.I } & $(\Exp[X])^2\leq \Exp[X^2]$ & $X$ random variable\\
  \hline
\end{tabular}
\end{center}
 \caption{Popular inequalities with abbreviations and formulas.}
\label{inequalities}
\end{table}

\subsection{Proof of Theorem~\ref{InexactSGDConstant}}
\label{Appendix1}
\begin{proof}
First we decompose:
\begin{eqnarray}\label{main_equality}
\|x^{k+1}-x^*\|_{\mB}^2 &=& \| (\mI - \omega \mB^{-1}\mZ_k) (x^k-x^*) +  \epsilon^k\|_{\mB}^2\notag\\
&=& \|(\mI - \omega \mB^{-1}\mZ_k) (x^k-x^*)\|_{\mB}^2+\|\epsilon^k\|_{\mB}^2 \notag\\
&&+2 \left\langle (\mI - \omega \mB^{-1}\mZ_k) (x^k-x^*), \epsilon^k \right \rangle.
\end{eqnarray}
Applying the innermost expectation of \eqref{eq:tower3} to \eqref{main_equality}, we get:
\begin{eqnarray}\label{mainequation}
\Exp\left[\|x^{k+1}-x^*\|_{\mB}^2\;|\;x^k, \bS_k\right] &=&\underbrace{\Exp \left[ \|(\mI - \omega \mB^{-1}\mZ_k) (x^k-x^*)\|_{\mB}^2\;|\;x^k,\bS_k \right]}_{T1} + \underbrace{\Exp \left[ \|\epsilon^k\|_{\mB}^2\;|\;x^k,\bS_k \right]}_{T2}\notag\\
&&+ 2\underbrace{\Exp \left[\left\langle (\mI - \omega \mB^{-1}\mZ_k) (x^k-x^*), \epsilon^k \right \rangle_{\mB}\;|\;x^k,\bS_k \right]}_{T3}.
\end{eqnarray}
We now analyze the three expression T1,T2,T3 separately. 

Note that an upper bound for the expression T2 can be directly obtained from the assumption 
\begin{equation}
\label{ForT2}
T2= \Exp[\|\epsilon^k\|^2_{\bB}\;|\;x^k,\bS_k]\leq\sigma_k^2.
 \end{equation}
The first expression can be written as: 
\begin{eqnarray}\label{ForT1}
T1= \Exp[ \|(\mI - \omega \mB^{-1}\mZ_k) (x^k-x^*)\|_{\mB}^2\;|\;x^k,\bS_k]& = &\|(\mI - \omega \mB^{-1}\mZ_k) (x^k-x^*)\|_{\mB}^2\notag\\
&\overset{\eqref{x_k_omega}}=& \|x^k-x^*\|_{\mB}^2 - 2\omega (2-\omega)f_{\bS_k}(x^k).
\end{eqnarray}
For expression T3: 
\begin{eqnarray}\label{ForT3}
\Exp \left[\left\langle (\mI - \omega \mB^{-1}\mZ_k) (x^k-x^*), \epsilon^k \right \rangle_{\mB}\;|\;x^k,\bS_k \right] &=&\left\langle (\mI - \omega \mB^{-1}\mZ_k) (x^k-x^*), \Exp[\epsilon^k\;|\;x^k,\bS_k] \right \rangle_{\mB}\notag\\
& \overset{{\rm \;C.S.}}{\leq}&  \|(\mI - \omega \mB^{-1}\mZ_k) (x^k-x^*)\|_{\bB} \|\Exp[\epsilon^k\;|\;x^k,\bS_k]\|_{\bB}\notag\\
&\overset{{\rm C.J.}}{\leq}&  \|(\mI - \omega \mB^{-1}\mZ_k) (x^k-x^*)\|_{\bB} \Exp[\|\epsilon^k \|_{\bB}\;|\;x^k,\bS_k]\notag\\
&\overset{(*)}\leq& \|(\mI - \omega \mB^{-1}\mZ_k) (x^k-x^*)\|_{\bB} \sigma_k,
\end{eqnarray}
where in the inequality $(*)$ we use Remark~\ref{variance_ineq} and  \eqref{Assumption1serial}.

By substituting the bounds \eqref{ForT2}, \eqref{ForT1}, and \eqref{ForT3} into \eqref{mainequation} we obtain:
\begin{eqnarray}\label{afterdba}
\Exp \left[\|x^{k+1}-x^*\|_{\mB}^2\;|\;x^k, \bS_k \right] &\leq&  \|x^k-x^*\|_{\mB}^2 - 2\omega (2-\omega)f_{\bS_k}(x^k)+ \sigma_k^2 \notag\\
&&+2 \|(\mI - \omega \mB^{-1}\mZ_k) (x^k-x^*)\|_{\bB} \sigma_k.
\end{eqnarray}
We now take the middle expectation~(see \eqref{eq:tower3}) and apply it to inequality \eqref{afterdba}:
\begin{eqnarray}\label{4}
\Exp \left[\Exp[\|x^{k+1}-x^*\|_{\mB}^2\;|\;x^k, \bS_k]\;|\;x^k \right] &\leq&  \|x^k-x^*\|_{\mB}^2 - 2\omega (2-\omega)f(x^k)+\sigma_k^2 \notag\\
&&+2 \Exp[\|(\mI - \omega \mB^{-1}\mZ_k) (x^k-x^*)\|_{\bB}\;|\;x^k]\sigma_k.
\end{eqnarray}
Now let us find a bound on the quantity $\Exp\left[\|(\mI - \omega \mB^{-1}\mZ_k) (x^k-x^*)\|_{\bB}\;|\;x^k\right]$.
Note that from \eqref{exp_x_k_omega} and \eqref{x_k_omega} we have that $\Exp\left[\|(\mI - \omega \mB^{-1}\mZ_k) (x^k-x^*)\|^2_{\bB}\;|\;x^k\right] \leq \rho \|x^k-x^*\|^2_{\bB}$. By using Remark~\ref{variance_ineq} in the last inequality we obtain:
\begin{eqnarray}\label{5} 
\Exp\left[\|(\mI - \omega \mB^{-1}\mZ_k) (x^k-x^*)\|_{\bB}\;|\;x^k\right] & = &\sqrt{\rho}\|x^k-x^*\|_{\mB}.
\end{eqnarray} 

By substituting \eqref{5} in \eqref{4}:
\begin{eqnarray}
\label{eq19}
\Exp[\Exp[\|x^{k+1}-x^*\|_{\mB}^2\;|\;x^k, \bS_k]\;|\;x^k] &\leq &  \|x^k-x^*\|_{\mB}^2 - 2\omega (2-\omega)f(x^k)+ \sigma_k^2 \notag\\
&&+2\sigma_k \sqrt{\rho}\|x^k-x^*\|_{\mB} \notag\\
&\overset{\eqref{exp_x_k_omega}}\leq& \rho \|x^k-x^*\|_{\mB}^2 + \sigma_k^2 + 2\sigma_k \sqrt{\rho}\|x^k-x^*\|_{\mB}
\end{eqnarray}
We take the final expectation (outermost expectation in the tower rule \eqref{eq:tower3}) on the above expression to find:
\begin{eqnarray}
\label{asjcakd}
\Exp[\|x^{k+1}-x^*\|_{\mB}^2]&=&\Exp[\Exp[\Exp[\|x^{k+1}-x^*\|_{\mB}^2\;|\;x^k, \bS_k]\;|\;x^k]]\notag\\
& \leq& \rho \Exp[\|x^k-x^*\|_{\mB}^2] + \sigma_k^2 + 2 \sigma_k \sqrt{\rho} \,\ \Exp[\|x^k-x^*\|_{\mB}] \notag\\
&\overset{V.I}\leq& \rho \Exp[\|x^k-x^*\|_{\mB}^2] + \sigma_k^2 + 2 \sigma_k \sqrt{\rho} \sqrt{\Exp[\|x^k-x^*\|_{\mB}^2]} 
\end{eqnarray}
Using $r_k = \E{\|x^{k}-x^*\|_{\mB}^2}$ equation \eqref{asjcakd} takes the form: 
\begin{eqnarray*}
r_{k+1}\leq & \rho r_k  +  \sigma_k^2 + 2\sigma_k \sqrt{\rho} \sqrt{r_k} = \left( \sqrt{\rho r_k}+\sigma_k \right)^2
 \end{eqnarray*}
 If we further substitute $p_k=\sqrt{r_k}$ and $\ell=\sqrt{\rho}$ the recurrence simplifies to:
 \begin{eqnarray*}
p_{k+1}\leq & \ell p_k +\sigma_k 
 \end{eqnarray*}
By unrolling the final inequality:
\begin{eqnarray*}
p_k \leq \ell^k r_0 + (\ell^0\sigma_{k-1}+\ell\sigma_{k-2}+\cdots + \ell^{k-1}\sigma_0) = \ell^k p_0 +  \sum_{i=0}^{k-1} \ell^{k-1-i}\sigma_i.
\end{eqnarray*}
Hence,  $$\sqrt{\Exp[\|x^{k}-x^*\|_{\mB}^2]} \leq  \rho^{k/2} \|x^{0}-x^*\|_{\mB} +  \sum_{i=0}^{k-1} \rho^{\frac{k-1-i}{2}}\sigma_i.$$
The result is obtained by using V.I in the last expression.
\end{proof}

\subsection{Proof of Corollary~\ref{FirstCorollary}}
\label{Appendix2}
By denoting $r_k=\Exp[\|x^{k}-x^*\|_{\mB}]$ in \eqref{Theorem1} we obtain:
\[r_k \leq \rho^{k/2} r_0 + \rho^{1/2} \sigma \sum_{i=0}^{k-1} \rho^{k-1-i} =\rho^{k/2} r_0 + \rho^{1/2} \sigma \sum_{i=0}^{k-1} \rho^{i}=\rho^{k/2} r_0 + \rho^{1/2} \sigma  \frac{1-\rho^k}{1-\rho}.\]
Since $1-\rho^k\le 1$ the result is obtained. 

\subsection{Proof of Theorem~\ref{ISGDwithq}}
\label{Appendix3}
In order to prove Theorem~\ref{ISGDwithq} we need to follow similar steps to the proof of Theorem~\ref{InexactSGDConstant}. The main  differences of the two proofs appear at the points that we need to upper bound the norm of the inexactness error ($\|\epsilon^k\|^2$). In particular instead of using the general sequence $\sigma_k^2 \in \R$ we utilize the bound $q^2\|x^k-x^*\|^2_{\bB}$ from Assumption~\ref{Assumption3}. Thus, it is sufficient to focus at the parts of the proof that these bound is used.

Similar to the proof of Theorem~\ref{InexactSGDConstant} we first decompose to obtain the equation \eqref{mainequation}. There, the expression T1 can be upper bounded from \eqref{ForT1} but now using the Assumption~\ref{Assumption3} the expression T2  and T3 can be upper bounded as follows:
\begin{equation}
\label{For2T2}
T2= \Exp[\|\epsilon^k\|^2_{\bB}\;|\;x^k,\bS_k]\leq q^2\|x^k-x^*\|^2_{\bB}.
 \end{equation}
\begin{eqnarray}
\label{For2T3}
T3&=&\Exp[\left\langle (\mI - \omega \mB^{-1}\mZ_k) (x^k-x^*), \epsilon^k \right \rangle_{\mB}\;|\;x^k,\bS_k] \notag\\
&\overset{{\rm \;Remark\;}\ref{variance_ineq}\;{\rm and\;} \eqref{ForT3}}\leq &\|(\mI - \omega \mB^{-1}\mZ_k) (x^k-x^*)\|_{\bB} q\|x^k-x^*\|
\end{eqnarray}

As a result by substituting the bounds \eqref{ForT1}, \eqref{For2T2}, and \eqref{For2T3} into \eqref{mainequation} we obtain:
\begin{eqnarray}
\Exp[\|x^{k+1}-x^*\|_{\mB}^2\;|\;x^k, \bS_k] &\overset{\eqref{mainequation}}\leq&  \|x^k-x^*\|_{\mB}^2 - 2\omega (2-\omega)f_{\bS_k}(x^k)+ q^2\|x^k-x^*\|^2_{\bB} \notag\\
&&+2 \|(\mI - \omega \mB^{-1}\mZ_k) (x^k-x^*)\|_{\bB} q\|x^k-x^*\|_{\bB}.
\end{eqnarray}

By following the same steps to the proof of Theorem~\ref{InexactSGDConstant} the equation \eqref{eq19} takes the form:
\begin{eqnarray}
\Exp[\Exp[\|x^{k+1}-x^*\|_{\mB}^2\;|\;x^k, \bS_k]\;|\;x^k] 
&\leq& \rho \|x^k-x^*\|_{\mB}^2 + q^2\|x^k-x^*\|^2_{\bB} \notag\\ && + \quad 2q\|x^k-x^*\|_{\bB} \sqrt{\rho}\|x^k-x^*\|_{\mB} \notag\\
&=& \left(\rho+2q \sqrt{\rho}+q^2\right)\|x^k-x^*\|_{\mB}^2.\notag\\
& = &\left(\sqrt{\rho}+q\right)^2 \|x^k-x^*\|_{\mB}^2
\end{eqnarray}
We take the final expectation (outermost expectation in the tower rule \eqref{eq:tower3}) on the above expression to find:
\begin{eqnarray}
\Exp[\|x^{k+1}-x^*\|_{\mB}^2]&=&\Exp \left[\Exp[\Exp[\|x^{k+1}-x^*\|_{\mB}^2\;|\;x^k, \bS_k]\;|\;x^k] \right]\notag\\
& \leq& \left(\sqrt{\rho}+q\right)^2 \Exp[\|x^k-x^*\|_{\mB}^2].
\end{eqnarray}
The final result follows by unrolling the recurrence. 

\subsection{Proof of Theorem \ref{InexactSGDrandom}}
\label{Appendix4}
\begin{proof}
Similar to the previous two proofs by decomposing the update rule and using the innermost expectation of \eqref{eq:tower3} we obtain equation \eqref{mainequation}.  An upper bound of expression T1 is again given by inequality \eqref{ForT1}. For the expression T2 depending the assumption that we have on the norm of the inexactness error different upper bounds can be used. In particular, 
\begin{enumerate}
\item[(i)] If Assumption \ref{Assumption2} holds then:  
 $ T2=\Exp[\|\epsilon^k\|^2_{\bB}\;|\;x^k,\bS_k]\leq\sigma_k^2.$
\item[(ii)] If Assumption \ref{Assumption3} holds then: $T2=\Exp[\|\epsilon^k\|^2_{\bB}\;|\;x^k,\bS_k]\leq\sigma_k^2 = q^2\|x^k-x^*\|^2_{\bB}.$
\item[(iii)] If Assumption \ref{Assumption4} holds then:  
$ T2=\Exp[\|\epsilon^k\|^2_{\bB}\;|\;x^k,\bS_k]\leq\sigma_k^2 = 2 q^2 f_{S_k}(x^k).$
\end{enumerate}

The main difference from the previous proofs, is that due to the Assumption \ref{Assumption5} and tower property \eqref{eq:tower3} the expression T3 will eventually be equal to zero. More specifically, we have that:
\begin{eqnarray}
\Exp \left[\Exp \left[\Exp \left[ \left\langle (\mI - \omega \mB^{-1}\mZ_k) (x^k-x^*), \epsilon^k \right \rangle_{\mB}  \;|\; x^k,  \mS_k \right] \;|\; x^k\right] \right] &=& \Exp \left[\left\langle (\mI - \omega \mB^{-1}\mZ_k) (x^k-x^*), \epsilon^k \right \rangle_{\mB} \right] \notag\\
&=& T3=0,
\end{eqnarray}
Thus, in this case equation \eqref{afterdba} takes the form:
\begin{eqnarray}\label{mainequationansdjkad}
\Exp[\|x^{k+1}-x^*\|_{\mB}^2\;|\;x^k, \bS_k] &\leq & \|x^k-x^*\|_{\mB}^2 - 2\omega (2-\omega)f_{\bS_k}(x^k) + \sigma_k^2.
\end{eqnarray}
Using the above expression depending the assumption that we have we obtain the following results:
\begin{enumerate}
\item[(i)] 
By taking the middle expectation~(see \eqref{eq:tower3}) and apply it to the above inequality:
\begin{eqnarray}
\Exp\left[\Exp\left[\|x^{k+1}-x^*\|_{\mB}^2\;|\;x^k, \bS_k \right]\;|\;x^k\right] &\leq&  \|x^k-x^*\|_{\mB}^2 - 2\omega (2-\omega)f(x^k)+\Exp[\sigma_k^2 \;|\;x^k] \notag\\
&\overset{\eqref{exp_x_k_omega}}\leq&  \rho \|x^k-x^*\|_{\mB}^2 +\Exp[\sigma_k^2 \;|\;x^k] 
\end{eqnarray}

We take the final expectation (outermost expectation in the tower rule \eqref{eq:tower3}) on the above expression to find:
\begin{eqnarray}
\Exp\left[\|x^{k+1}-x^*\|_{\mB}^2 \right]&=&\Exp \left[\Exp \left[\Exp \left[\|x^{k+1}-x^*\|_{\mB}^2\;|\;x^k, \bS_k\right]\;|\;x^k\right] \right]\notag\\
& \leq& \rho\Exp[\|x^k-x^*\|_{\mB}^2] + \Exp[\Exp[\sigma_k^2 \;|\;x^k]] \notag\\
& =& \rho \Exp[\|x^k-x^*\|_{\mB}^2] + \Exp[\sigma_k^2] \notag\\
& =& \rho \Exp[\|x^k-x^*\|_{\mB}^2] + {\bar{\sigma}}_k^2 
\end{eqnarray}
Using $r_k = \E{\|x^{k}-x^*\|_{\mB}^2}$ the last inequality takes the form
$r_{k+1}\leq \rho r_k  +   {\bar{\sigma}}_k^2$. By unrolling the last expression:
$r_k \leq \rho^k r_0 + (\rho^0 {\bar{\sigma}}_{k-1}^2+\rho {\bar{\sigma}}^2_{k-2}+\cdots + \rho^{k-1} {\bar{\sigma}}^2_0) =\rho^k r_0 +  \sum_{i=0}^{k-1} \rho^{k-1-i} {\bar{\sigma}}^2_i.$
Hence,  $$\Exp[\|x^{k}-x^*\|_{\mB}^2] \leq  \rho^{k} \|x^{0}-x^*\|^2_{\mB} +  \sum_{i=0}^{k-1} \rho^{k-1-i}{\bar{\sigma}}^2_i.$$
\item[(ii)]
For the case (ii) inequality \eqref{mainequationansdjkad} takes the form:
$$
\Exp[\|x^{k+1}-x^*\|_{\mB}^2\;|\;x^k, \bS_k] \leq  \|x^k-x^*\|_{\mB}^2 - 2\omega (2-\omega)f_{\bS_k}(x^k) + q^2\|x^k-x^*\|^2_{\bB},
$$
and by taking the middle expectation~(see \eqref{eq:tower3}) we obtain:
\begin{eqnarray}
\Exp[\Exp[\|x^{k+1}-x^*\|_{\mB}^2\;|\;x^k, \bS_k]\;|\;x^k]  &\leq & \|x^k-x^*\|_{\mB}^2 - 2\omega (2-\omega)f(x^k) + q^2\|x^k-x^*\|^2_{\bB}\notag\\
& \overset{\eqref{exp_x_k_omega}}\leq & \rho \|x^k-x^*\|_{\mB}^2 + q^2\|x^k-x^*\|^2_{\bB} \notag\\
& =& (\rho + q^2) \|x^k-x^*\|_{\mB}^2. 
\end{eqnarray}
By taking the final expectation of the tower rule \eqref{eq:tower3} and apply it to the above inequality:
\begin{eqnarray}
\Exp[\|x^{k+1}-x^*\|_{\mB}^2] & \leq& (\rho + q^2) \Exp[\|x^k-x^*\|_{\mB}^2]. 
\end{eqnarray}
and the result is obtain by unrolling the last expression.
\item[(iii)] 
For the case (iii) inequality \eqref{mainequationansdjkad} takes the form:
\begin{eqnarray}
\Exp[\|x^{k+1}-x^*\|_{\mB}^2\;|\;x^k, \bS_k] &\leq & \|x^k-x^*\|_{\mB}^2 - 2(\omega (2-\omega)-q^2 ) f_{\bS_k}(x^k),
\end{eqnarray}
and by taking the middle expectation~(see \eqref{eq:tower3}) we obtain:
\begin{eqnarray}
\Exp[\Exp[\|x^{k+1}-x^*\|_{\mB}^2\;|\;x^k, \bS_k]\;|\;x^k]  &\leq &\|x^k-x^*\|_{\mB}^2 - 2(\omega (2-\omega)-q^2 ) f(x^k) \notag\\
& \overset{\eqref{b3}}\leq & \|x^k-x^*\|_{\mB}^2 - (\omega (2-\omega)-q^2 )\lambda_{\min}^+ \|x^k-x^*\|_{\mB}^2 \notag\\
& =& (1- (\omega (2-\omega)-q^2 )\lambda_{\min}^+ ) \|x^k-x^*\|_{\mB}^2.
\end{eqnarray}
By taking the final expectation of the tower rule \eqref{eq:tower3} to the above inequality:
\begin{eqnarray}
\Exp[\|x^{k+1}-x^*\|_{\mB}^2] & \leq& (1- (\omega (2-\omega)-q^2 )\lambda_{\min}^+ ) \Exp[\|x^k-x^*\|_{\mB}^2].
\end{eqnarray}
and the result is obtain by unrolling the last expression.
\end{enumerate}
\end{proof}

\chapter{Revisiting Randomized Gossip Algorithms}
\label{ChapterGossip}

\section{Introduction}
\label{sec:ACP}
Average consensus is a fundamental problem in distributed computing and multi-agent systems. It comes up in many real world applications such as coordination of autonomous agents, estimation, rumour spreading in social networks, PageRank and distributed data fusion on ad-hoc networks and decentralized optimization. Due to its great importance there is much classical \cite{tsitsiklis1986distributed,degroot1974reaching} and recent   \cite{xiao2005scheme, xiao2004fast, boyd2006randomized} work on the design of efficient algorithms/protocols for solving it.

In the average consensus (AC) problem we are given an undirected connected network $\cG=(\cV,\cE)$ with node set $\cV=\{1,2,\dots,n\}$ and edges $\cE$. Each node $i \in \cV$ ``knows'' a private value $c_i \in \R$. The goal of AC is for every node to compute the average of these private values, $\bar{c}\eqdef\frac{1}{n}\sum_i c_i$, in a decentralized fashion. That is, the exchange of information can only occur between connected nodes (neighbors). 

Among the most attractive protocols for solving the average consensus problem are gossip algorithms. The development and design of gossip algorithms was studied extensively in the last decade. The seminal 2006 paper of Boyd et al.\ \cite{boyd2006randomized} motivated a fury of subsequent research and gossip algorithms now appear in many applications, including distributed data fusion in sensor networks \cite{xiao2005scheme}, load balancing \cite{cybenko1989dynamic} and clock synchronization \cite{freris2012fast}.  For a survey of selected relevant work prior to 2010, we refer the reader to the work of  Dimakis et al.\ \cite{dimakis2010gossip}. For more recent results on randomized gossip algorithms we suggest \cite{zouzias2015randomized, liu2013analysis,olshevsky2014linear,liu2018privacy, nedic2018network, aybat2017decentralized}. See also  \cite{dimakis2008geographic, aysal2009broadcast, olshevsky2009convergence}. 

\subsection{Main contributions} 
In this chapter, we connect two areas of research which until now have remained remarkably disjoint in the literature: randomized iterative (projection) methods for solving linear systems and randomized gossip protocols for solving the average consensus. This connection enables us to make contributions by borrowing from each body of literature to the other and using it we propose a new framework for the design and analysis of novel efficient randomized gossip protocols.

The main contributions of our work include:

\begin{itemize}
\item \textbf{RandNLA.} We show how classical randomized iterative methods for solving linear systems can be interpreted as gossip algorithms when applied to special systems encoding the underlying network and explain in detail their decentralized nature.
Through our general framework we recover a comprehensive array of well-known gossip protocols as special cases. In addition our approach allows for the development of novel block and dual variants of all of these methods. From a numerical analysis viewpoint our work is the first that explores in depth, the decentralized nature of randomized iterative methods for solving linear systems and proposes them as efficient methods for solving the average consensus problem (and its weighted variant). 

\item \textbf{Weighted AC.}  The methods presented in this chapter solve the more general \textit{weighted} average consensus (Weighted AC) problem (Section~\ref{weightedAC}) popular in the area of distributed cooperative spectrum sensing networks. The proposed protocols are the first randomized gossip algorithms that directly solve this problem with finite-time convergence rate analysis. In particular, we prove linear convergence of the proposed protocols and explain how we can obtain further acceleration using momentum. To the best of our knowledge, the existing decentralized protocols that solve the weighted average consensus problem show convergence but without convergence analysis. 

\item \textbf{Acceleration.} We present novel and provably \textit{accelerated} randomized gossip protocols. In each step, of the proposed algorithms, all nodes of the network update their values using their own information but only a subset of them exchange messages. The protocols are inspired by the recently proposed accelerated variants of randomized Kaczmarz-type methods and use momentum terms on top of the sketch and project update rule (gossip communication) to obtain better theoretical and practical performance.
To the best of our knowledge, our accelerated protocols are the first randomized gossip algorithms that converge to a consensus with a provably accelerated linear rate without making any further assumptions on the structure of the network. Achieving an accelerated linear rate in this setting using randomized gossip protocols was an open problem.

\item \textbf{Duality.} We reveal a hidden duality of randomized gossip algorithms, with the dual iterative process maintaining variables attached to the edges of the network. We show how the randomized coordinate descent and randomized Newton methods work as edge-based dual randomized gossip algorithms.

\item \textbf{Experiments.} We corroborate our theoretical results with extensive experimental testing on typical wireless network topologies.
We numerically verify the linear convergence of the our protocols for solving the weighted AC problem. We explain the benefit of using block variants in the gossip protocols where more than two nodes update their values in each iteration. We explore the performance of the proposed provably accelerated gossip protocols and show that they significantly outperform the standard pairwise gossip algorithm and existing fast pairwise gossip protocols with momentum. An experiment showing the importance of over-relaxation in the gossip setting is also presented.
\end{itemize}

We believe that this work could potentially open up new avenues of research in the area of decentralized gossip protocols. 

\subsection{Structure of the chapter}
This chapter is organized as follows. Section~\ref{background} introduces the necessary background on basic randomized iterative methods for linear systems that will be used for the development of randomized gossip protocols.  Related work on the literature of linear system solvers, randomized gossip algorithms for averaging and gossip algorithms for consensus optimization is presented.   In Section~\ref{skecthsection} the more general weighted average consensus problem is described and the connections between the two areas of research (randomized projection methods for linear systems and gossip algorithms) is established. In particular we explain how methods for solving linear systems can be interpreted as gossip algorithms when applied to special systems encoding the underlying network and elaborate in detail their distributed nature. Novel block gossip variants are also presented. In Section~\ref{AccelerateGossip} we describe and analyze fast and provably accelerated randomized gossip algorithms. In each step of these protocols all nodes of the network update their values but only a subset of them exchange their private values. Section~\ref{DualBlock} describes dual randomized gossip algorithms that operate with values that are associated to the edges of the network and Section~\ref{FurtherConnections} highlights further connections between methods for solving linear systems and gossip algorithms.
Numerical evaluation of the new gossip protocols is presented in Section~\ref{experimentsGossip}. Finally, concluding remarks are given in Section~\ref{conclusion}.

\subsection{Notation}
For convenience, a table of the most frequently used notation of this chapter is included in Section~\ref{NotationTable}. 
In particular, with boldface upper-case letters denote matrices; $\bI$ is the identity matrix. By $\|\cdot \|$ and $\|\cdot \|_F$ we denote the Euclidean norm and  the Frobenius norm, respectively.  For a positive integer number $n$, we write $[n]\eqdef \{1,2, \dots ,n\}$.  By $\cL$ we denote the solution set of the linear system $\bA x=b$, where $\bA \in \R^{m\times n}$ and $b\in \R^m$.

Vector $x^k = (x^k_1,\dots,x^k_n) \in \R^n$ represents the vector with the private values of the $n$ nodes of the network at the $k^{th}$ iteration while with $x_i^{k}$ we denote the value of node $i \in [n]$ at the $k^{th}$ iteration. $\cN_i \subseteq \cV$ denotes the set of nodes that are neighbors of node $i \in \cV$.  By $\ac(\cG)$ we denote the algebraic connectivity of graph $\cG$. Throughout the chapter, $x^*$ is the projection of $x^0$ onto $\cL$ in the $\bB$-norm. We write $x^*=\Pi_{\cL,\bB}(x^0)$.

The complexity of all gossip protocols presented in this chapter is described by the spectrum of matrix 
\begin{equation}
\label{MatrixW}
\bW=\bB^{-1/2}\mA^\top \Exp[\bH] \mA\bB^{-1/2}\overset{\eqref{ZETA_Intro}}{=}\bB^{-1/2}\Exp[\bZ]\bB^{-1/2},
\end{equation}
where the expectation is taken over $\bS\sim \cD$.  
With $\lambda_{\min}^+$ and $\lambda_{\max}$ we indicate the smallest nonzero and the largest eigenvalue of matrix $\bW$, respectively. Recall that this is exactly the same matrix used in the previous chapters of this thesis.

Finally, with $\bQ \in \R^{|\cE| \times n}$ we define the incidence matrix and with $\bL \in \R^{n\times n} $ the Laplacian matrix of the network. Note that it holds that $\bL=\bQ^\top \bQ$. Further, with $\bD$ we denote the degree matrix of the graph. That is, $\bD=\text{\textbf{Diag}}(d_1,d_2,\dots, d_n) \in \R^{n \times n}$ where $d_i$ is the degree of node $i \in \cV$. 

\section{Background - Technical Preliminaries}
\label{background}

As we have already mentioned in this thesis, solving linear systems is a central problem in numerical linear algebra and plays an important role in computer science, control theory, scientific computing, optimization, computer vision, machine learning, and many other fields. With the advent of the age of big data, practitioners are looking for ways to solve linear systems of unprecedented sizes. In this large scale setting, randomized iterative methods are preferred mainly because of their cheap per iteration cost and because they can easily scale to extreme dimensions.

\subsection{Randomized iterative methods for linear systems}

Recall that in the introduction of this thesis we presented the sketch and project method \eqref{SPM_IntroThesis} and we explained how this algorithm is identical to SGD \eqref{SGD_IntroThesis}, SN \eqref{SNM_IntroThesis} and SPP \eqref{SPPM_IntroThesis} for solving the stochastic quadratic optimization problem \eqref{StochReform_IntroThesis}. For the benefit of the reader and for the easier comparison to the gossip algorithms, a formal presentation of the Sketch and Project method for solving a consistent linear system $\bA x=b$ is presented in Algorithm~\ref{FullSkecth}. 

\begin{algorithm}[H]
	\caption{Sketch and Project Method \cite{ASDA}}
	\label{FullSkecth}
	\small \small
	\begin{algorithmic}[1]
		\State {\bf Parameters:} Distribution $\mathcal{D}$ from which method samples matrices; stepsize/relaxation parameter $\omega \in \R$; momentum parameter $\beta$.
		\State {\bf Initialize:} $x^0,x^1 \in \R^n$
		\For{$k=0,1,2,\dots$} 
		\State Draw a fresh $\bS_k \sim \cD$
		\State   Set 
$x^{k+1}=x^k -\omega \bB^{-1}\bA^\top \bS_k (\bS_k^\top \bA \bB^{-1} \bA^\top \bS_k)^\dagger \bS_k^\top (\bA x^k-b)$
		\EndFor
		\State {\bf Output:} The last iterate $x^k$
	\end{algorithmic}
\end{algorithm}

In this chapter, we are mostly interested in two special cases of the sketch and project framework--- the randomized Kaczmarz (RK) method and its block variant, the randomized block Kaczmarz (RBK) method. In addition, in the following sections we present novel scaled and accelerated variants of these two selected cases and interpret their gossip nature. In particular, we focus on explaining how these methods can solve the average consensus problem and its more general version, the weighted average consensus (subsection~\ref{weightedAC}).

Let  $e_i \in \R^m$  be the $i^{\text{th}}$ unit coordinate vector in $ \R^m$ and let $\bI_{:C}$ be column submatrix of the $m \times m$ identity matrix with columns indexed by  $C\subseteq [m]$.  Then RK and RBK methods can be obtained as special cases of Algorithm~\ref{FullSkecth} as follows:
\begin{itemize}
\item RK: Let $\bB=\bI$ and $\bS_k=e_i$, where $i \in [m]$ is chosen independently at each iteration, with probability $p_i>0$. In this setup the update rule of Algorithm~\ref{FullSkecth} simplifies to  
\begin{equation}
\label{RK}
x^{k+1}=x^k - \omega \frac{\bA_{i :} x^k -b_{i}}{\|\bA_{i :}\|^2} \bA_{i :}^ \top  .
\end{equation}
\item RBK: Let $\bB=\bI$ and $\bS=\bI_{:C}$, where set $C\subseteq [m]$ is chosen independently at each iteration, with probability $p_C\geq 0$. In this setup the update rule of Algorithm~\ref{FullSkecth} simplifies to  
\begin{equation}
\label{RBK}
x^{k+1}=x^k - \omega \bA_{C:}^\top (\bA_{C:}\bA_{C:}^\top)^\dagger (\bA_{C:}x^k-b_C).
\end{equation}
\end{itemize}

As we explained in Chapter~\ref{ChapterIntroduction}, the sketch and project method, converges linearly to one particular solution of the linear system: the projection (on $\bB$-norm) of the initial iterate $x^0$ onto the solution set of the linear system, $x^*=\Pi_{\cL,\bB}(x^0)$. Therefore, the method solve the best approximation problem \eqref{BestApproximation_IntroThesis}.

The convergence performance of the Sketch and Project method (Algorithm~\ref{FullSkecth}) for solving the best approximation problem is described by the following theorem\footnote{For the proof of Theorem~\ref{ConvergenceSketchProject} check Theorem~\ref{BasicMethodConvergence} in Section~\ref{IntroductoryAnalysis} and recall that SGD and the sketch and project method are identical in this setting.}.

\begin{thm}
\label{ConvergenceSketchProject}
Let assume exactness and let $\{x^k\}_{k=0}^\infty$ be the iterates produced by the sketch and project method (Algorithm~\ref{FullSkecth}) with step-size $\omega \in (0,2)$. Set, $x^*=\Pi_{\cL,\bB}(x^0)$. Then,
 \begin{equation}
 \label{ConvergenceBasic}
 \Exp[\|x^k-x^*\|_{\bB}^2]\leq \rho^k \|x^0-x^*\|_{\bB}^2,
 \end{equation}
where
\begin{equation}
\label{RateRho}
\rho \eqdef 1 - \omega (2-\omega) \lambda_{\min}^+ \in [0,1].
\end{equation}
\end{thm}

In other words, using standard arguments, from Theorem~\ref{ConvergenceSketchProject} we observe that for a given $\epsilon>0$ we have that:
$$k \geq \frac{1}{1-\rho} \log \left(\frac{1}{\epsilon} \right) \quad \Rightarrow \quad \Exp[\|x^k-x^*\|_{\bB}^2]\leq \epsilon \|x^0-x^*\|_{\bB}^2$$

\subsection{Other related work}


\paragraph{Gossip algorithms for average consensus} The problem of average consensus has been extensively studied in the automatic control and signal processing literature for the past two decades~\cite{dimakis2010gossip}, and was first introduced for decentralized processing in the seminal work~\cite{tsitsiklis1986distributed}. A clear connection between the rate of convergence and spectral characteristics of the underlying network topology over which message passing occurs was first established in~\cite{boyd2006randomized} for pairwise randomized gossip algorithms. 

Motivated by network topologies with salient properties of wireless networks (e.g., nodes can communicate directly only with other nearby nodes), several methods were proposed to accelerate the convergence of gossip algorithms. For instance, \cite{benezit2010order} proposed averaging among a set of nodes forming a path in the network (this protocol can be seen as special case of our block variants in Section~\ref{BlockGossip}). Broadcast gossip algorithms have also been analyzed \cite{aysal2009broadcast} where the nodes communicate with more than one of their neighbors by broadcasting their values.

While the gossip algorithms studied in~\cite{boyd2006randomized,benezit2010order,aysal2009broadcast} are all first-order (the update of $x^{k+1}$ only depends on $x^k$), a faster randomized pairwise gossip protocol was proposed in~\cite{cao2006accelerated} which suggested to incorporate additional memory to accelerate convergence. The first analysis of this protocol was later proposed in~\cite{liu2013analysis} under strong condition. It is worth to mention that in the setting of deterministic gossip algorithms theoretical guarantees for accelerated convergence were obtained in \cite{oreshkin2010optimization,kokiopoulou2009polynomial}.
In Section~\ref{AccelerateGossip} we propose fast and provably accelerated randomized gossip algorithms with memory and compare them in more detail with the fast randomized algorithm proposed in~\cite{cao2006accelerated, liu2013analysis}.

\paragraph{Gossip algorithms for multiagent consensus optimization.} 
In the past decade there has been substantial interest in consensus-based mulitiagent optimization methods that use gossip updates in their update rule \cite{nedic2018network, yuan2016convergence, shi2015extra}. In multiagent consensus optimization setting , $n$ agents or nodes, cooperate to solve an optimization problem. In particular, a local objective function $f_i:\R^d \rightarrow \R$ is associated with each node $i \in [n]$ and the goal is for all nodes to solve the optimization problem \begin{equation}
\label{optMulti}
\min_{x \in \R^d} \frac{1}{n}\sum_{i=1}^n f_i(x)
\end{equation} by communicate only with their neighbors. In this setting gossip algorithms works in two steps by first executing some local computation followed by communication over the network \cite{nedic2018network}. Note that the average consensus problem with $c_i$ as node $i$ initial value can be case as a special case of the optimization problem \eqref{optMulti} when the function values are $f_i(x)=(x- c_i)^2$.

Recently there has been an increasing interest in applying mulitagent optimization  methods to solve convex and non-convex optimization problems arising in machine learning~\cite{tsianos2012communication, lian2018asynchronous, assran2018stochastic, assran2018asynchronous,colin2016gossip, koloskova2019decentralized, hendrikx2018accelerated}. In this setting most consensus-based optimization methods make use of standard, first-order gossip, such as those described in~\cite{boyd2006randomized}, and incorporating momentum into their updates to improve their practical performance. 

\section{Sketch and Project Methods as Gossip Algorithms}
\label{skecthsection}
In this section we show how by carefully choosing the linear system in the constraints of the best approximation problem \eqref{BestApproximation_IntroThesis} and the combination of the parameters of the Sketch and Project method (Algorithm~\ref{FullSkecth}) we can design efficient randomized gossip algorithms. We show that the proposed protocols can actually solve the weighted average consensus problem, a more general version of the average consensus problem described in Section~\ref{sec:ACP}. In particular we focus, on a scaled variant of the RK method \eqref{RK} and on the RBK \eqref{RBK} and understand the convergence rates of these methods in the consensus setting, their distributed nature and how they are connected with existing gossip protocols. 

\subsection{Weighted average consensus}
\label{weightedAC}
In the \emph{weighted average consensus} (Weighted AC) problem we are given an undirected connected network $\cG=(\cV,\cE)$ with node set $\cV=\{1,2,\dots,n\}$ and edges $\cE$. Each node $i \in \cV$ holds a private value $c_i \in \R$ and its weight $w_i$. The goal of this problem is for every node to compute the weighted average of the private values, 
$$\bar{c}\eqdef \frac{\sum_{i=1}^n w_i c_i}{\sum_{i=1}^n w_i},$$
in a distributed fashion. That is, the exchange of information can only occur between connected nodes (neighbors). 

Note that in the special case when the weights of all nodes are the same ($w_i=r$ for all $i \in [n]$) the weighted average consensus is reduced to the standard average consensus problem. However, there are more special cases that could be interesting. For instance the weights can represent the degree of the nodes ($w_i= d_i$) or they can denote a probability vector and satisfy $\sum_i^n w_i=1$ with $w_i>0$. 

It can be easily shown that the weighted average consensus problem can be expressed as optimization problem as follows:
\begin{equation}
\label{weightedCOnsensus}
\min_{x = (x_1,\dots, x_n) \in \R^n} \frac{1}{2} \|x-c\|_{\bB}^2 
\quad \text{subject to}  \quad x_1=x_2=\dots=x_n
\end{equation}
where matrix $\bB=\textbf{Diag}(w_1, w_2,\dots, w_n)$ is a diagonal positive definite matrix (that is $w_i>0$ for all $i \in [n]$) and $c=(c_1,\dots,c_n)^\top$ the vector with the initial values $c_i$ of all nodes $i \in \cV$. The optimal solution of this problem is $x^*_i=\frac{\sum_{i=1}^n w_i c_i}{\sum_{i=1}^n w_i}$ for all $i \in [n]$ which is exactly the solution of the weighted average consensus.

As we have explained, the standard average consensus problem can be cast as a special case of weighted average consensus. However, in the situation when the nodes have access to global information related to the network, such as the size of the network (number of nodes $n=|\cV|$) and the sum of the weights $\sum_{i=1}^n w_i$, then any algorithm that solves the standard average consensus can be used to solve the weighted average consensus problem with the initial private values of the nodes changed from $c_i$ to $\frac{n w_i c_i}{\sum_{i=1}^n w_i}$. 

The weighted AC problem is popular in the area of distributed cooperative spectrum sensing networks \cite{hernandes2018improved, pedroche2014convergence, zhang2015distributed, zhang2011distributed}. In this setting, one of the goals is to develop decentralized protocols for solving the cooperative sensing problem in cognitive radio systems. The weights in this case represent a ratio related to the channel conditions of each node/agent \cite{hernandes2018improved}. The development of methods for solving the weighted AC problem is an active area of research (check \cite{hernandes2018improved} for a recent comparison of existing algorithms). However, to the best of our knowledge, existing analysis for the proposed algorithms focuses on showing convergence and not on providing convergence rates. Our framework allows us to obtain novel randomized gossip algorithms for solving the weighted AC problem. In addition, we provide a tight analysis of their convergence rates. In particular, we show convergence with a linear rate. See Section~\ref{Weightexperiments} for an experiment confirming linear convergence of one of our proposed protocols on typical wireless network topologies.

\subsection{Gossip algorithms through sketch and project framework}
We propose that randomized gossip algorithms should be viewed as special case of the Sketch and Project update to a particular problem of the form \eqref{BestApproximation_IntroThesis}. In particular, we let $c=(c_1,\dots,c_n)$ be the initial values stored at the nodes of $\cG$, and choose $\bA$ and $b$ so that the constraint $\bA x = b$ is equivalent to the requirement that $x_i=x_j$ (the value stored at node $i$ is equal to the value stored at node $j$) for all $(i,j)\in \cE$.

\begin{defn}
\label{defACsystem}
 We say that  $\bA x = b$ is an ``average consensus (AC) system'' when $\bA x = b$ iff $x_i = x_j$ for all $(i,j) \in \cE$.
\end{defn}

It is easy to see that $\bA x = b$ is an AC system precisely when $b=0$ and the nullspace of $\bA$ is $\{t 1_n : t\in \R\}$, where $1_n$ is the vector of all ones in $\R^n$.  Hence, $\bA$ has rank $n-1$. Moreover in the case that $x^0=c$, it is easy to see that for any AC system, the solution of \eqref{BestApproximation_IntroThesis}  necessarily is $x^* = \bar{c} \cdot 1_n$ --- this is why we singled out AC systems. In this sense, {\em any} algorithm for solving \eqref{BestApproximation_IntroThesis} will ``find'' the (weighted) average $\bar{c}$. However, in order to obtain a distributed algorithm we need to make sure that only ``local'' (with respect to $\cG$) exchange of information is allowed.  

\paragraph{Choices of AC systems.} It can be shown that many linear systems satisfy the above definition. 

For example, we can choose:
\begin{enumerate}
\item $ b=0$ and $\bA=\bQ \in \R^{|\cE| \times n}$ to be the incidence matrix of  $\cG$. That is, $\bQ \in \R^{|\cE| \times n}$ such that  $\bQ x = 0$ directly encodes the constraints $x_i=x_j$ for $(i,j)\in \cE$. That is, row  $e=(i,j) \in \cE$ of matrix $\bQ$ contains value $1$ in column $i$, value $-1$ in column $j$ (we use an arbitrary but fixed order of nodes defining each edge in order to fix $\bQ$) and zeros elsewhere. 
\item  A different choice is to pick $b=0$ and  $\mA =\mL=\bQ^\top \bQ$, where $\mL$ is the Laplacian matrix of network $\cG$. 
\end{enumerate}
Depending on what AC system is used, the sketch and project methods can have different interpretations as gossip protocols.  

In this work we mainly focus on the above two AC systems but we highlight that other choices are possible\footnote{Novel gossip algorithms can be proposed by using different AC systems to formulate the average consensus problem. For example one possibility is using the random walk normalized Laplacian $\bL^{rw}=\bD^{-1} \bL$. For the case of degree-regular networks the symmetric normalized Laplacian matrix $\bL^{sym}=\bD^{-1/2} \bL \bD^{-1/2}$ can also being used.}. In Section~\ref{accSubsection} for the provably accelerated gossip protocols we also use a normalized variant ($\|\bA_{i:}\|^2=1$) of the Incidence matrix. 

\subsubsection{Standard form and mass preservation}
Assume that $\bA x = b$ is an AC system. Note that since $b=0$, the update rule of Algorithm~\ref{FullSkecth} simplifies to:
\begin{equation}
\label{updateSkProj}
x^{k+1}= \left[ \bI- \omega\bA^\top \bH_k\bA \right] x^k=\left[ \bI- \omega  \bZ_k \right] x^k.
\end{equation}
This is the standard form in which randomized gossip algorithms are written. What is new here is that the iteration matrix $\bI- \omega  \bZ_k$ has a specific structure which guarantees convergence to $x^*$ under very weak assumptions (see Theorem~\ref{ConvergenceSketchProject}). Note that if $x^0=c$, i.e., the starting primal iterate is the vector of private values (as should be expected from any gossip algorithm), then the iterates of \eqref{updateSkProj} enjoy a mass preservation property (the proof follows the fact that $\bA 1_n = 0$):
\begin{thm}[Mass preservation]
If $\bA x =b$ is an AC system, then the iterates produced by \eqref{updateSkProj} satisfy:  $\frac{1}{n}\sum_{i=1}^{n}x_i^k=\bar{c}$, for all  $k \geq 0$.
 \end{thm}
 \begin{proof} Let fix $k\geq0$ then,
 $$\frac{1}{n} 1_n^\top x^{k+1}=\frac{1}{n} 1_n^\top ( \bI - \omega\bA^\top \bH_k \bA) x^k=\frac{1}{n} 1_n^\top  \bI x^k - \frac{1}{n} 1_n^\top \omega\bA^\top \bH_k \bA x^k\overset{\bA 1_n = 0}{=}\frac{1}{n} 1_n^\top x^k.$$
 \end{proof}
 
\subsubsection{$\varepsilon$-Averaging time}

Let $z^k\eqdef \|x^k - x^*\|$. The typical measure of convergence speed employed in the randomized gossip literature, called $\varepsilon$-averaging time and here denoted by $T_{ave}(\varepsilon)$, represents the smallest time $k$ for which  $x^{k}$ gets within $\varepsilon z^0$ from $x^*$, with probability greater than $1-\varepsilon$, uniformly over all starting values $x^0=c$. More formally, we define
\[
T_{ave}(\varepsilon)\eqdef \sup_{c\in \R^n} \inf  \left\{k\;:\; \Prob \left(z^k > \varepsilon z^0 \right)\leq\varepsilon \right\}.
\]
This definition differs slightly from the standard one in that we use $z^0$ instead of $\|c\|$.

Inequality \eqref{ConvergenceBasic}, together with Markov inequality, can be used  to give a bound on $K(\varepsilon)$, formalized next:

\begin{thm}\label{thm:complexity_standard} Assume $\bA x=b$ is an AC system. Let $x^0=c$ and $\bB$ be positive definite diagonal matrix. Assume exactness. Then for any $0<\varepsilon < 1$ we have
$$T_{ave}(\epsilon) \leq 3 \frac{\log(1/\varepsilon)}{\log(1/\rho)} \leq 3\frac{\log(1/\epsilon)}{1-\rho},$$
where $\rho$ is defined in \eqref{RateRho}. 
\end{thm}
\begin{proof}
See Section~\ref{ProofTave}.
\end{proof}

Note that under the assumptions of the above theorem, $\bW=\bB^{-1/2} \Exp[\bZ] \bB^{-1/2}$ only has a single zero eigenvalue, and hence $\lambda_{\min}^+ (\bW)$ is the second smallest eigenvalue of $\bW$. Thus, $\rho$ is the second largest eigenvalue of $\bI - \bW$. The bound on $K(\varepsilon)$ appearing in Thm~\ref{thm:complexity_standard} is often written with $\rho$ replaced by $\lambda_2(\bI - \bW)$ \cite{boyd2006randomized}.

In the rest of this section we show how two special cases of the sketch and project framework, the randomized Kaczmarz (RK) and its block variant, randomized block Kaczmatz (RBK) work as gossip algorithms for the two AC systems described above. 

\subsection{Randomized Kaczmarz method as gossip algorithm}
As we described before the sketch and project update rule of Algorithm~\ref{FullSkecth} has several parameters that should be chosen in advance by the user. These are the stepsize $\omega$ (relaxation parameter), the positive definite matrix $\bB$ and the distribution $\cD$ of the random matrices $\bS$. 

In this section we focus on one particular special case of the sketch and project framework, a scaled/weighted variant of the randomized Kaczmarz method (RK) presented in \eqref{RK},  and we show how this method works as gossip algorithm when applied to special systems encoding the underlying network. In particular, the linear systems that we solve are the two AC systems described in the previous section where the matrix is either the incidence matrix $\bQ$ or the Laplacian matrix $\bL$ of the network.

As we described in \eqref{RK} the standard RK method can be cast as special case of Algorithm~\ref{FullSkecth} by choosing $\bB=\bI$ and $\bS=e_i$. In this section, we focus on a small modification of this algorithm and we choose the positive definite matrix $\bB$ to be  $\bB=\textbf{Diag}(w_1, w_2,\dots, w_n)$, the diagonal matrix of the weights presented in the weighted average consensus problem. 

\paragraph{Scaled RK:} Let us have a general consistent linear system $\bA x= b$ with $\bA \in \R^{m \times n}$. Let us also choose $\bB=\textbf{Diag}(w_1, w_2,\dots, w_n)$ and $\bS_k=e_i$, where $i\in [m]$ is chosen in each iteration independently, with probability $p_i>0$. In this setup the update rule of Algorithm~\ref{FullSkecth} simplifies to  
\begin{equation}
\label{scaledRK}
x^{k+1}=x^k - \omega \frac{e_i^\top (\bA x^k -b) }{e_i^\top \bA \bB^{-1} \bA^\top e_i} \bB^{-1}\bA^ \top e_i  =x^k - \omega \frac{\bA_{i :} x^k -b_{i}}{\|\bB^{-1/2} \bA_{i :}^\top\|^2_{2}} \bB^{-1}\bA_{i :}^ \top.
\end{equation}
This small modification of RK allow us to solve the more general weighted average consensus presented in Section~\ref{weightedAC} (and at the same time the standard average consensus problem if $\bB= r \bI$ where $r \in \R$). To the best of our knowledge, even if this variant is special case of the general Sketch and project update, was never precisely presented before in any setting.

\subsubsection{AC system with incidence matrix $\bQ$}
Let us represent the constraints of problem \eqref{weightedCOnsensus} as linear system with matrix $\bA=\bQ \in \R^{|\cE| \times n}$ be the Incidence matrix of the graph and right had side $b=0$. Lets also assume that the random matrices $\bS \sim \cD$ are unit coordinate vectors in $\R^{m}=\R^{|\cE|}$. 

Let $e=(i,j) \in \cE$ then from the definition of matrix $\bQ$ we have that $\bQ_{e:}^\top=f_i-f_j$ where $f_i, f_j$ are unit coordinate vectors in $\R^n$. In addition, from the definition the diagonal positive definite matrix $\bB$ we have that 
\begin{equation}
\label{ansdl}
\|\bB^{-1/2} \bQ_{e :}^\top\|^2=\|\bB^{-1/2} (f_i-f_j)\|^2=\frac{1}{w_1}+\frac{1}{w_j}.
\end{equation}

Thus in this case the update rule \eqref{scaledRK} simplifies:
\begin{eqnarray}
\label{updateSRKQ}
x^{k+1}&\overset{b=0, \bA=\bQ, \eqref{scaledRK}}{=}& x^k - \omega \frac{\bQ_{e :} x^k}{\|\bB^{-1/2} \bQ_{e :}^\top\|^2} \bB^{-1}\bQ_{e :}^ \top \notag \\
&\overset{\eqref{ansdl}}{=}&x^k - \omega \frac{\bQ_{e :} x^k}{\frac{1}{w_1}+\frac{1}{w_j}} \bB^{-1}\bQ_{e :}^ \top \notag \\
&=&x^k-\frac{\omega (x^k_i-x^k_j)}{\frac{1}{w_i}+\frac{1}{w_j}}  \left(\frac{1}{w_i} f_i- \frac{1}{w_j } f_j \right). 
\end{eqnarray}

From \eqref{updateSRKQ} it can be easily seen that only the values of coordinates $i$ and $j$ update their values. These coordinates correspond to the private values $x^k_i$ and $x^k_j$ of the nodes of the selected edge $e=(i,j)$. In particular the values of $x^k_i $ and $x^k_j$ are updated as follows:

\begin{equation}
\label{ScRKwithQ}
x^{k+1}_i=\left(1-\omega \frac{w_j}{w_j+w_i} \right) x^{k}_i + \omega \frac{ w_j}{w_j+w_i} x^k_j \quad \text{and} \quad x^{k+1}_j= \omega \frac{ w_i}{w_j+w_i} x^{k}_i + \left(1-\omega\frac{ w_i}{w_j+w_i} \right) x^k_j.
\end{equation}

\begin{rem}
In the special case that $\bB=r \bI$ where $r\in \R$ (we solve the standard average consensus problem) the update of the two nodes is simplified to
$$x^{k+1}_i=\left(1-\frac{\omega}{2} \right) x^{k}_i + \frac{\omega }{2} x^k_j \quad \text{and} \quad x^{k+1}_j= \frac{\omega}{2} x^{k}_i + \left(1-\frac{\omega}{2} \right) x^k_j.$$
If we further select $\omega=1$ then this becomes:
\begin{equation}
\label{pairwiseUpdate}
x^{k+1}_i=x^{k+1}_j= \frac{x^{k}_i +x^k_j}{2},
\end{equation}
which is the update of the standard pairwise randomized gossip algorithm first presented and analyzed in \cite{boyd2006randomized}.
\end{rem}

\subsubsection{AC system with Laplacian matrix $\bL$}
The AC system takes the form $\bL x=0$, where matrix $\bL \in \R^{n \times n}$ is the Laplacian matrix of the network. In this case, each row of the matrix corresponds to a node. Using the definition of the Laplacian, we have that $\bL_{i:}^\top=d_if_i-\sum_{j \in \cN_i}f_j$, where $f_i, f_j$ are unit coordinate vectors in $\R^n$ and $d_i$ is the degree of node $i \in \cV$. 

Thus, by letting $ \bB=\textbf{Diag}(w_1, w_2,\dots, w_n)$ to be the diagonal matrix of the weights we obtain:

\begin{equation}
\label{ansdl2}
\|\bB^{-1/2} \bL_{i :}^\top\|^2=\left\|\bB^{-1/2} (d_if_i-\sum_{j \in \cN_i}f_j) \right\|^2=\frac{d_i^2}{w_i}+\sum_{j \in \cN_i} \frac{1}{w_j}.
\end{equation}

In this case, the update rule \eqref{scaledRK} simplifies to:
\begin{eqnarray}
\label{updateSRKL}
x^{k+1}&\overset{b=0, \bA=\bL, \eqref{scaledRK}}{=}& x^k - \omega \frac{\bL_{i :} x^k}{\|\bB^{-1/2} \bL_{i :}^\top\|^2_{2}} \bB^{-1}\bL_{i :}^ \top \notag \\
&\overset{\eqref{ansdl2}}{=}&x^k - \omega \frac{\bL_{i :} x^k}{\frac{d_i^2}{w_i}+\sum_{j \in \cN_i} \frac{1}{w_j}} \bB^{-1}\bL_{i :}^ \top \notag \\
&=&x^k-\frac{\omega \left(d_i x^k_i- \sum_{j\in \cN_i} x^k_j \right) }{\frac{d_i^2}{w_i}+\sum_{j \in \cN_i} \frac{1}{w_j}}  \left(\frac{d_i}{w_i} f_i- \sum_{j\in \cN_i} \frac{1}{w_j } f_j \right).
\end{eqnarray}

From \eqref{updateSRKL}, it is clear that only coordinates $ \{i\} \cup \cN_i$ update their values. All the other coordinates remain unchanged. In particular, the value of the selected node $i$ (coordinate $i$) is updated as follows:

\begin{eqnarray}
x^{k+1}_i = x^k_i-\frac{\omega \left(d_i x^k_i- \sum_{j\in \cN_i} x^k_j \right) }{\frac{d_i^2}{w_i}+\sum_{j \in \cN_i} \frac{1}{w_j }}  \frac{d_i}{w_i}, 
\end{eqnarray}
while the values of its neighbors $ j \in \cN_i$ are updated as:
\begin{eqnarray}
x^{k+1}_j =x^k_j+\frac{\omega  \left(d_i x^k_i- \sum_{\ell \in \cN_i} x^k_\ell \right)  }{\frac{d_i^2}{w_i}+\sum_{\ell \in \cN_i} \frac{1}{w_\ell}} \frac{1}{w_j }. 
\end{eqnarray}

\begin{rem}
\label{naosjkl}
Let  $\omega=1$ and $\bB=r \bI$ where $r\in \R$ then the selected nodes update their values as follows:
\begin{eqnarray}
\label{aoskmpas}
x^{k+1}_i = \frac{ \sum_{\ell \in \{i \cup \cN_i\}} x^k_{\ell} }{d_i+1} \quad \text{and} \quad
x^{k+1}_j =x^k_j+\frac{ (d_i x^k_i- \sum_{\ell \in \cN_i} x^k_\ell)  }{d_i^2+d_i}.
\end{eqnarray}
That is, the selected node $i$ updates its value to the average of its neighbors and itself, while all the nodes $j \in \cN_i$ update their values using the current value of node $i$ and all nodes in $\cN_i$.
\end{rem}

In a wireless network, to implement such an update, node $i$ would first broadcast its current value to all of its neighbors. Then it would need to receive values from each neighbor to compute the sums over $\cN_i$, after which node $i$ would broadcast the sum to all neighbors (since there may be two neighbors $j_1, j_2 \in \cN_i$ for which $(j_1, j_2) \notin \cE$). In a wired network, using standard concepts from the MPI library, such an update rule could be implemented efficiently by defining a process group consisting of $\{i\} \cup \cN_i$, and performing one \texttt{Broadcast} in this group from $i$ (containing $x_i$) followed by an \texttt{AllReduce} to sum $x_\ell$ over $\ell \in \cN_i$. Note that the terms involving diagonal entries of $\bB$ and the degrees $d_i$ could be sent once, cached, and reused throughout the algorithm execution to reduce communication overhead.

\subsubsection{Details on complexity results}
Recall that the convergence rate of the sketch and project method (Algorithm~\ref{FullSkecth}) is equivalent to:
$$\rho \eqdef 1 - \omega (2-\omega) \lambda_{\min}^+(\bW),$$
where $\omega \in (0,2)$ and $\bW= \bB^{-1/2}\bA^\top \Exp[\bH] \bA\bB^{-1/2}$ (from Theorem~\ref{ConvergenceSketchProject}).  In this subsection we explain how the convergence rate of the scaled RK method \eqref{scaledRK} is modified for different choices of the main parameters of the method.

Let us choose $\omega=1$ (no over-relaxation). In this case, the rate is simplified to $\rho=1 - \lambda_{\min}^+$. 

Note that the different ways of modeling the problem (AC system) and the selection of the main parameters (weight matrix $\bB$ and distribution $\cD$) determine the convergence rate of the method through the spectrum of matrix $\bW$.

Recall that in the $k^{th}$ iterate of the scaled RK method \eqref{scaledRK} a random vector $\bS_k=e_i$ is chosen with probability $p_i>0$. 
For convenience, let us choose\footnote{Similar probabilities have been chosen in \cite{gower2015randomized} for the convergence of the standard RK method ($\bB=\bI$). The distribution $\cD$ of the matrices $\bS$ used in equation \eqref{convProb} is common in the area of randomized iterative methods for linear systems and is used to simplify the analysis and the expressions of the convergence rates. For more choices of distributions we refer the interested reader to \cite{gower2015randomized}. It is worth to mention that the probability distribution that optimizes the convergence rate of the RK and other projection methods can be expressed as the solution to a convex semidefinite program \cite{gower2015randomized, dai2014randomized}.}:
\begin{equation}
\label{convProb}
p_i= \frac{\|\bB^{-1/2} \bA_{i:}^\top\|^2}{\|\bB^{-1/2}\bA^\top\|^2_F} \;.
\end{equation}
Then we have that:
\begin{eqnarray}
\label{ajskmxa}
\Exp[ \bH] &=&\Exp[ \mS (\mS^\top \mA \bB^{-1} \mA^\top \mS)^\dagger \mS^\top] \notag\\
&=&\sum_{i=1}^m p_i \frac{e_i e_i^\top}{e_i^\top \mA \bB^{-1} \mA^\top e_i}=\sum_{i=1}^m p_i \frac{e_i e_i^\top}{\|\bA_{i:}^\top\|_{\bB^{-1} }^2}=\sum_{i=1}^m p_i \frac{e_i e_i^\top}{\|\bB^{-1/2} \bA_{i:}^\top\|^2}\notag\\
& \overset{\eqref{convProb}}{=} & \sum_{i=1}^m \frac{e_i e_i^\top}{\|\bB^{-1/2}\bA^\top\|^2_F}=\frac{1}{\|\bB^{-1/2}\bA^\top\|^2_F} \bI,
\end{eqnarray}
and
\begin{equation}
\label{Wconvinience}
\bW\overset{\eqref{MatrixW}, \eqref{ajskmxa}}{=} \frac{\bB^{-1/2}\bA^\top \bA\bB^{-1/2}}{\|\bB^{-1/2}\bA^\top\|^2_F}.
\end{equation}

\paragraph{Incidence Matrix:} Let us choose the AC system to be the one with the incidence matrix $\bA=\bQ$. Then $\|\bB^{-1/2}\bQ^\top\|^2_F= \sum_{i=1}^n \frac{d_i}{b_i}$ and we obtain 
$$\bW\overset{\bA=\bQ, \eqref{Wconvinience}}{=}\frac{\bB^{-1/2}\bL \bB^{-1/2}}{\|\bB^{-1/2}\bQ^\top\|^2_F}=\frac{\bB^{-1/2}\bL \bB^{-1/2}}{\sum_{i=1}^n \frac{d_i}{\bB_{ii}}}.$$

If we further have $\bB=\bD$, then $\bW=\frac{\bD^{-1/2}\bL \bD^{-1/2}}{n}$ and the convergence rate simplifies to:
$$\rho=1-\frac{\lambda_{\min}^+\left(\bD^{-1/2}\bL \bD^{-1/2}\right)}{n}=1-\frac{\lambda_{\min}^+\left(\bL^{sym} \right)}{n}.$$

If $\bB=r \bI$ where $r\in \R$ (solve the standard average consensus problem), then $\bW=\frac{\bL}{\|\bQ\|^2_F}= \frac{\bL}{\sum_{i=1}^n d_i}=\frac{\bL}{2m}$ and the convergence rate simplifies to 
\begin{equation}
\label{ratePairwise}
\rho=1-\frac{\lambda_{\min}^+(\bL)}{2m}=1-\frac{\ac(\cG)}{2m},
\end{equation}
which is exactly the same convergence rate of the pairwise gossip algorithm presented in \cite{boyd2006randomized}. This was expected, since the gossip protocol in this case works exactly the same as the one proposed in \cite{boyd2006randomized}, see equation \eqref{pairwiseUpdate}.

\paragraph{Laplacian Matrix:} If we choose to formulate the AC system using the Laplacian matrix $\bL$, that is $\bA=\bL$, then $\|\bB^{-1/2}\bL^\top\|^2_F= \sum_{i=1}^n \frac{d_i(d_i+1)}{\bB_{ii}}$ and we have:
$$\bW\overset{\bA=\bL, \eqref{Wconvinience}}{=}\frac{\bB^{-1/2}\bL^\top \bL \bB^{-1/2}}{\sum_{i=1}^n \frac{d_i(d_i+1)}{\bB_{ii}}}.$$

If $\bB=\bD$, then the convergence rate simplifies to:
$$\rho=1-\frac{\lambda_{\min}^+(\bD^{-1/2}\bL^\top \bL\bD^{-1/2})}{\sum_{i=1}^n (d_i+1)}=1-\dfrac{\lambda_{\min}^+\left(\bD^{-1/2}\bL^2 \bD^{-1/2}\right)}{n+ \sum_{i=1}^n d_i}\overset{\sum_{i=1}^n d_i=2m}{=}1-\dfrac{\lambda_{\min}^+\left(\bD^{-1/2}\bL^2 \bD^{-1/2}\right)}{n+ 2m}.$$

If $\bB=r \bI$, where $r\in \R$, then $\bW=\frac{\bL^2}{\|\bL\|^2_F}= \frac{\bL^2}{\sum_{i=1}^n d_i(d_i+1)}$ and the convergence rate simplifies to $$\rho=1-\frac{\lambda_{\min}^+(\bL^2)}{\sum_{i=1}^n d_i(d_i+1)}=1-\frac{\ac(\cG)^2}{\sum_{i=1}^n d_i(d_i+1)}.$$

\subsection{Block gossip algorithms}
\label{BlockGossip}
Up to this point we focused on the basic connections between the convergence analysis of the sketch and project methods and the literature of randomized gossip algorithms. We show how specific variants of the randomized Kaczmarz method (RK) can be interpreted as gossip algorithms for solving the weighted and standard average consensus problems. 

In this part we extend the previously described methods to their block variants related to randomized block Kaczmarz (RBK) method \eqref{RBK}. In particular, in each step of Algorithm~\ref{FullSkecth}, the random matrix $\bS$ is selected to be a random column submatrix of the $m \times m$ identity matrix corresponding to columns indexed by a random subset $C  \subseteq [m]$. That is, $\bS=\bI_{:C}$, where a set $C\subseteq [m]$ is chosen in each iteration independently, with probability $p_C\geq 0$ (see equation \eqref{RBK}). Note that in the special case that set $C$ is a singleton with probability 1 the algorithm is simply the randomized Kaczmarz method of the previous section.

To keep things simple, we assume that $\bB=\bI$ (standard average consensus, without weights) and choose the stepsize $\omega=1$. In the next section, we will describe gossip algorithms with heavy ball momentum and explain in detail how the gossip interpretation of RBK change in the more general case of $\omega \in (0,2)$.

Similar to the previous subsections, we formulate the consensus problem using either $\bA=\bQ$ or $\bA=\bL$ as the matrix in the AC system.
In this setup, the iterative process of Algorithm~\ref{FullSkecth} has the form:
\begin{eqnarray}
\label{RBKgossip}
x^{k+1} &\overset{\eqref{RBK},\eqref{updateSkProj}}{=}& x^k - \bA^ \top \bI_{:C}(\bI_{:C}^\top \bA\bA^\top \bI_{:C})^{\dagger}\bI_{:C}^\top \bA x^k= x^k - \bA_{C:}^\top (\bA_{C:}\bA_{C:}^\top)^\dagger \bA_{C:}x^k,
\end{eqnarray}
which, as explained in the introduction, can be equivalently written as: 
\begin{equation}
\label{RBKaLgorithm}
x^{k+1}=\underset{x \in \R^n}{\operatorname{argmin}} \{\|x-x^k\|^2 \;:\; \bI_{:C}^\top \bA x=0\}.
\end{equation} 
Essentially in each step of this method the next iterate is evaluated to be the projection of the current iterate $x^k$ onto the solution set of a row subsystem of $\bA x=0$. 

\paragraph{AC system with Incidence Matrix:} In the case that $\bA=\bQ$ the selected rows correspond to a random subset $C \subseteq \cE$ of selected edges. While \eqref{RBKgossip} may seem to be a complicated algebraic (resp.\ variational) characterization of the method, due to our choice of $\bA=\bQ$ we have the following result which gives a natural interpretation of RBK as a gossip algorithm (see also Figure~\ref{fig:RBK}).

\begin{thm}[RBK as Gossip algorithm: RBKG]
\label{TheoremRBK}
Consider the AC system with the constraints being expressed using the Incidence matrix $\bQ$.  
Then each iteration of RBK (Algorithm~\eqref{RBKgossip}) works as gossip algorithm as follows:
\begin{enumerate}
\item Select a random set of edges $C \subseteq \cE$, 
\item Form subgraph $\cG_k$ of $\cG$ from the selected edges 
\item For each connected component of $\cG_k$, replace node values with their average.
\end{enumerate}
\end{thm}
\begin{proof}
See Section~\ref{ProofRBK}.
\end{proof}
Using the convergence result of general Theorem~\ref{ConvergenceSketchProject} and the form of matrix $\bW$ (recall that in this case we assume $\bB=\bI$, $\bS=\bI_{:C}\sim \cD$ and $\omega=1$), we obtain the following complexity for the algorithm:
\begin{equation}
\label{anisojxalksda}
\Exp[\|x^k-x^*\|^2]\leq \left[1 - \lambda_{\min}^+ \left( \Exp\left[\bQ_{C:}^\top (\bQ_{C:}\bQ_{C:}^\top)^\dagger \bQ_{C:}\right] \right) \right]^k \|x^0-x^*\|^2.
\end{equation}
For more details on the above convergence rate of randomized block Kaczmarz method with meaningfully bounds on the rate in a more general setting we suggest the papers \cite{RBK,l2015randomized}.

There is a very closed relationship between the gossip interpretation of RBK explained in Theorem~\ref{TheoremRBK} and several existing randomized gossip algorithms that in each step update the values of more than two nodes. 
For example the {\em path averaging} algorithm porposed in \cite{benezit2010order} is a special case of RBK, when set $C$ is restricted to correspond to a path of vertices. That is, in path averaging, in each iteration a path of nodes is selected and the nodes that belong to it update their values to their exact average.
A different example is the recently proposed clique gossiping \cite{liu2017clique} where the network is already divided into cliques and through a random procedure a clique is activated and the nodes of it update their values to their exact average.
In \cite{boyd2006randomized} a synchronous variant of gossip algorithm is presented where in each step multiple node pairs communicate exactly at the same time with the restriction that these simultaneously active node pairs are disjoint.

It is easy to see that all of the above algorithms can be cast as special cases of RBK if the distribution $\cD$ of the random matrices is chosen carefully to be over random matrices $\bS$ (column sub-matrices of Identity) that update specific set of edges in each iteration. As a result our general convergence analysis can recover the complexity results proposed in the above works.

Finally, as we mentioned, in the special case in which set $C$ is always a singleton, Algorithm~\eqref{RBKgossip} reduces to the standard randomized Kaczmarz method. This means that only a random edge is selected in each iteration and the nodes incident with this edge replace their local values with their average. This is the pairwise gossip algorithm of Boyd er al. \cite{boyd2006randomized} presented in equation \eqref{pairwiseUpdate}. Theorem~\ref{TheoremRBK} extends this interpretation to the case of the RBK method. 

\begin{figure}[htb]
\begin{minipage}[b]{1.0\linewidth}
  \centering
  \centerline{\includegraphics[scale=0.4]{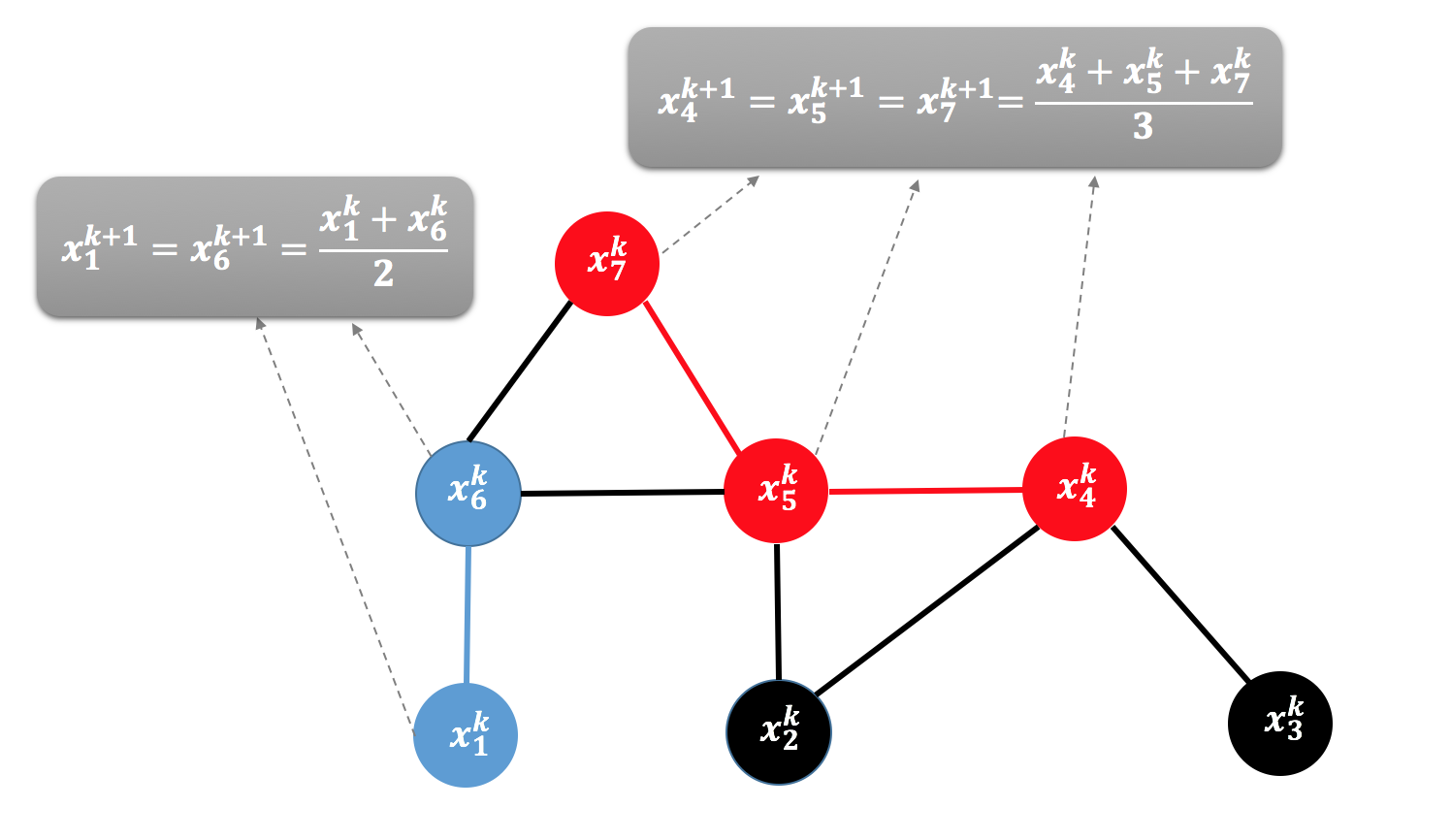}}
  \caption{\footnotesize Example of how the RBK method works as gossip algorithm in case of AC system with Incidence matrix. In the presented network 3 edges are randomly selected and a subgraph of two connected components (blue and red) is formed. Then the nodes of each connected component update their private values to their average.}
  \label{fig:RBK}
\end{minipage}
\end{figure}

\paragraph{AC system with Laplacian Matrix:}
For this choice of AC system the update is more complicated. To simplify the way that the block variant work as gossip we make an extra assumption. We assume that the selected rows of the constraint $\bI_{:C}^\top \bL x=0$ in update \eqref{RBKaLgorithm} have no-zero elements at different coordinates. This allows to have a direct extension of the serial variant presented in Remark~\ref{naosjkl}. Thus, in this setup, the RBK update rule \eqref{RBKgossip} works as gossip algorithm as follows:

\begin{enumerate}
\item $|C|$ nodes are activated (with restriction that the nodes are not neighbors and they do not share common neighbors)
\item For each node $i \in C$ we have the following update:
\begin{eqnarray}
x^{k+1}_i = \frac{ \sum_{\ell \in \{i \cup \cN_i\}} x^k_{\ell} }{d_i+1} \quad \text{and} \quad
x^{k+1}_j =x^k_j+\frac{ \left(d_i x^k_i- \sum_{\ell \in \cN_i} x^k_\ell \right)  }{d_i^2+d_i} .
\end{eqnarray}
\end{enumerate}
The above update rule can be seen as a parallel variant of update \eqref{aoskmpas}. Similar to the convergence in the case of Incidence matrix, the RBK for solving the AC system with a Laplacian matrix converges to $x^*$ with the following rate (using result of Theorem~\ref{ConvergenceSketchProject}):
$$\Exp[\|x^k-x^*\|^2]\leq \left[1 - \lambda_{\min}^+ \left( \Exp\left[\bL_{C:}^\top (\bL_{C:}\bL_{C:}^\top)^\dagger \bL_{C:}\right] \right) \right]^k \|x^0-x^*\|^2.$$

\section{Faster and Provably Accelerated Randomized Gossip Algorithms}
\label{AccelerateGossip}

The main goal in the design of gossip protocols is for the computation and communication to be done as quickly and efficiently as possible. In this section, our focus is precisely this. We design randomized gossip protocols which converge to consensus fast with provable accelerated linear rates. To the best of our knowledge, the proposed protocols are the first randomized gossip algorithms that converge to consensus with an accelerated linear rate. 

In particular, we present novel protocols for solving the average consensus problem where in each step all nodes of the network update their values but only a subset of them exchange their private values. The protocols are inspired from the recently developed accelerated variants of randomized Kaczmarz-type methods for solving consistent linear systems where the addition of momentum terms on top of the sketch and project update rule provides better theoretical and practical performance.

In the area of optimization algorithms, there are two popular ways to accelerate an algorithm using momentum. The first one is using the Polyak's heavy ball momentum \cite{polyak1964some} and the second one is using the theoretically much better understood momentum introduced by Nesterov \cite{nesterov1983method, nesterov2013introductory}.  Both momentum approaches have been recently proposed and analyzed to improve the performance of randomized iterative methods for solving linear systems.

To simplify the presentation, the accelerated algorithms and their convergence rates are presented for solving the standard average consensus problem ($\bB=\bI$). Using a similar approach as in the previous section, the update rules and the convergence rates can be easily modified to solve the more general weighted average consensus problem. For the protocols in this section we use the incidence matrix $\bA=\bQ$ or its normalized variant to formulate the AC system.

\subsection{Gossip algorithms with heavy ball momentum}

In Chapter~\ref{ChapterMomentum} of this thesis we have analyzed heavy ball momentum variants of several algorithms for solving the stochastic optimization problem \eqref{StochReform_IntroThesis} and as we explained the best approximation problem \eqref{BestApproximation_IntroThesis}. In this section we revisit Algorithm~\ref{anjsdnaodala} of Chapter~\ref{ChapterMomentum} and we focus on its sketch and project viewpoint. In particular, we explain how it works as gossip algorithm when is applied to the AC system with the incidence matrix.

\subsubsection{Sketch and project with heavy ball momentum}

The sketch and project method with heavy ball momentum is formally presented in the following algorithm. 

\begin{algorithm}[H]
	\caption{Sketch and Project with Heavy Ball Momentum}
	\label{SHBgradient}
	\small \small
	\begin{algorithmic}[1]
		\State {\bf Parameters:} Distribution $\mathcal{D}$ from which method samples matrices; stepsize/relaxation parameter $\omega \in \R$; momentum parameter $\beta$.
		\State {\bf Initialize:} $x^0,x^1 \in \R^n$
		\For{$k=1,2,\dots$} 
		\State Draw a fresh $\bS_k \sim \cD$
		\State   Set 
\begin{equation}
\label{SPmomentum}
x^{k+1}=x^k -\omega \bB^{-1}\bA^\top \bS_k (\bS_k^\top \bA \bB^{-1} \bA^\top \bS_k)^\dagger \bS_k^\top (\bA x^k-b) + \beta(x^k - x^{k-1}).
\end{equation}
		\EndFor
		\State {\bf Output:} The last iterate $x^k$
	\end{algorithmic}
\end{algorithm}

Using, $\bB=\bI$ and the same choice of distribution $\cD$ as in equations \eqref{RK} and \eqref{RBK} we can now obtain momentum variants of the RK and RBK as special case of the above algorithm as follows: 
\begin{itemize}
\item RK with momentum (mRK): \\
\begin{equation}
\label{ncajsolkmal}
x^{k+1}=x^k -\omega \frac{\bA_{i :} x^k -b_{i}}{\|\bA_{i :}\|^2} \bA_{i :}^ \top + \beta(x^k - x^{k-1}).
\end{equation}
\item RBK with momentum (mRBK):
\begin{equation}
\label{nacsklals}
x^{k+1}=x^k -\omega \bA_{C:}^\top (\bA_{C:}\bA_{C:}^\top)^\dagger (\bA_{C:}x^k-b_C) + \beta(x^k - x^{k-1}). 
\end{equation}
\end{itemize}

For more details on the convergence analysis of Algorithm~\ref{SHBgradient} see Section~\ref{sec:primal} and recall that in our setting the sketch and project update rule is identical to the SGD (Chapter~\ref{ChapterIntroduction}). As a result Algorithm~\ref{SHBgradient} is identical to Algorithm~\ref{anjsdnaodala} (mSGD/mSN/mSPP).

Having presented Algorithm~\ref{SHBgradient}, let us now describe its behavior as a randomized gossip protocol when applied to the AC system $\bA x=0$ with $\bA=\bQ \in |\cE| \times n$ (incidence matrix of the network).  

Note that since $b=0$ (from the AC system definition), the update rule \eqref{SPmomentum} of Algorithm~\ref{SHBgradient} is simplified to (and by having $\bB=\bI$):
\begin{equation}
\label{momentumupdateb0}
x^{k+1}=\left[\bI - \omega \bA^\top \bS_k (\bS_k^\top \bA \bA^\top \bS_k)^\dagger \bS_k^\top \bA \right] x^k + \beta(x^k - x^{k-1}).
\end{equation}

In the rest of this section we focus on  two special cases of \eqref{momentumupdateb0}: RK with heavy ball momentum (equation \eqref{ncajsolkmal} with $b_i=0$) and RBK with heavy ball momentum (equation \eqref{nacsklals} with $b_C=0$).

\subsubsection{Randomized Kaczmarz gossip with heavy ball momentum}
As we have seen in previous section when the standard RK is applied to solve the AC system $\bQ x=0$, one can recover the famous pairwise gossip algorithm~\cite{boyd2006randomized}.  Algorithm~\ref{RKmomentum} describes how a relaxed variant of randomized Kaczmarz with heavy ball momentum ($0<\omega < 2$ and $ 0 \leq \beta <1$) behaves as a gossip algorithm. See also Figure~\eqref{fig:mRK} for a graphical illustration of the method.

\begin{algorithm}[t!]
	\caption{mRK: Randomized Kaczmarz with momentum as a gossip algorithm}
	\label{RKmomentum}
	\small \small
	\begin{algorithmic}[1]
		\State {\bf Parameters:} Distribution $\mathcal{D}$ from which method samples matrices; stepsize/relaxation parameter $\omega \in \R$; heavy ball/momentum parameter $\beta$.
		\State {\bf Initialize:} $x^0 ,x^1 \in \R^n$
		\For{$k=1,2,\dots$} 
		\State Pick an edge $e=(i,j)$ following the distribution $\cD$
		\State The values of the nodes are updated as follows:
\begin{itemize}
\item Node $i$: $x_i^{k+1}= \frac{2-\omega}{2}x_i^k+ \frac{\omega}{2}x_j^k+\beta (x_i^k - x_i^{k-1})$
\item Node $j$: $x_j^{k+1}= \frac{2-\omega}{2} x_j^k+\frac{\omega}{2}x_i^k+\beta (x_j^k - x_j^{k-1})$
\item Any other node $\ell$: $x_\ell^{k+1}=x_\ell^k+\beta (x_\ell^k - x_\ell^{k-1})$
\end{itemize}
		\EndFor
		\State {\bf Output:} The last iterate $x^k$
	\end{algorithmic}
\end{algorithm}

\begin{rem}
In the special case that $\beta=0$ (zero momentum) only the two nodes of edge $e=(i,j)$ update their values. In this case the two selected nodes do not update their values to their exact average but to a convex combination that depends on the stepsize $\omega \in (0,2)$. To obtain the pairwise gossip algorithm of \cite{boyd2006randomized}, one should further choose $\omega=1$.
\end{rem}

\textbf{Distributed Nature of the Algorithm:} Here we highlight a few ways to implement mRK in a distributed fashion.
\begin{itemize}
\item \emph{Pairwise broadcast gossip:}  In this protocol each node $i \in \cV$ of the network $\cG$ has a clock that ticks at the times of a rate 1 Poisson process. The inter-tick times are exponentially distributed, independent across nodes, and independent across time. This is equivalent to a global clock ticking at a rate $n$ Poisson process which wakes up an edge of the network at random. In particular, in this implementation mRK works as follows:  In the $k^{th}$ iteration (time slot) the clock of node $i$ ticks and node $i$ randomly contact one of its neighbors and simultaneously broadcast a signal to inform the nodes of the whole network that is updating (this signal does not contain any private information of node $i$). The two nodes $(i,j)$ share their information and update their private values following the update rule of Algorithm~\ref{RKmomentum} while all the other nodes update their values using their own information. In each iteration only one pair of nodes exchange their private values. 
 
\item \emph{Synchronous pairwise gossip:} In this protocol a single global clock is available to all nodes. The time is assumed to be slotted commonly across nodes and in each time slot only a pair of nodes of the network is randomly activated and exchange their information following the update rule of Algorithm~\ref{RKmomentum}.  The remaining not activated nodes update their values using their own last two private values.  Note that this implementation of mRK comes with the disadvantage that it requires a central entity which in each step requires to choose the activated pair of nodes\footnote{We speculate that a completely distributed synchronous gossip algorithm that finds pair of nodes in a distributed manner without any additional computational burden can be design following the same procedure proposed in Section III.C of  \cite{boyd2006randomized}.}.

\item  \emph{Asynchronous pairwise gossip with common counter:} 
Note that the update rule of the selected pair of nodes $(i,j)$ in Algorithm~\ref{RKmomentum} can be rewritten as follows:
$$ x_i^{k+1}= x_i^k + \beta (x_i^k - x_i^{k-1}) + \frac{\omega}{2} (x_j^k -x_i^k),$$
$$x_j^{k+1}= x_j^k + \beta (x_j^k - x_j^{k-1}) + \frac{\omega}{2} (x_i^k -x_j^k).$$
In particular observe that the first part of the above expressions $x_i^k + \beta (x_i^k - x_i^{k-1})$ (for the case of node $i$) is exactly the same with the update rule of the non activate nodes at $k^{th}$ iterate (check step 5 of Algorithm~\ref{RKmomentum}) . Thus, if we assume that all nodes share a common counter that keeps track of the current iteration count and that each node $i \in \cV$ remembers the iteration counter $k_i$ of when it was last activated, then step 5 of Algorithm~\ref{RKmomentum} takes the form:
\begin{itemize}
\item $ x_i^{k+1}= i_k \left[x_i^k + \beta (x_i^k - x_i^{k-1}) \right]+ \frac{\omega}{2} (x_j^k -x_i^k),$
\item $x_j^{k+1}= j_k \left[x_j^k + \beta (x_j^k - x_j^{k-1}) \right] + \frac{\omega}{2} (x_i^k -x_j^k),$
\item $k_i = k_j =k+1,$
\item Any other node $\ell$: $x_\ell^{k+1}=x_\ell^k,$
\end{itemize}
where $i_k=k-k_{i}$ ($j_k=k-k_{j}$) denotes the number of iterations between the current iterate and the last time that the $i^{th}$ ($j^{th}$) node is activated. In this implementation only a pair of nodes communicate and update their values in each iteration (thus the justification of asynchronous), however it requires the nodes to share a common counter that keeps track the current iteration count in order to be able to compute the value of $i_k=k-k_{i}$.
\end{itemize}

\begin{figure}[t!]
\vspace{6pt}
\begin{minipage}[b]{1.0\linewidth}
  \centering
  \centerline{\includegraphics[scale=0.55]{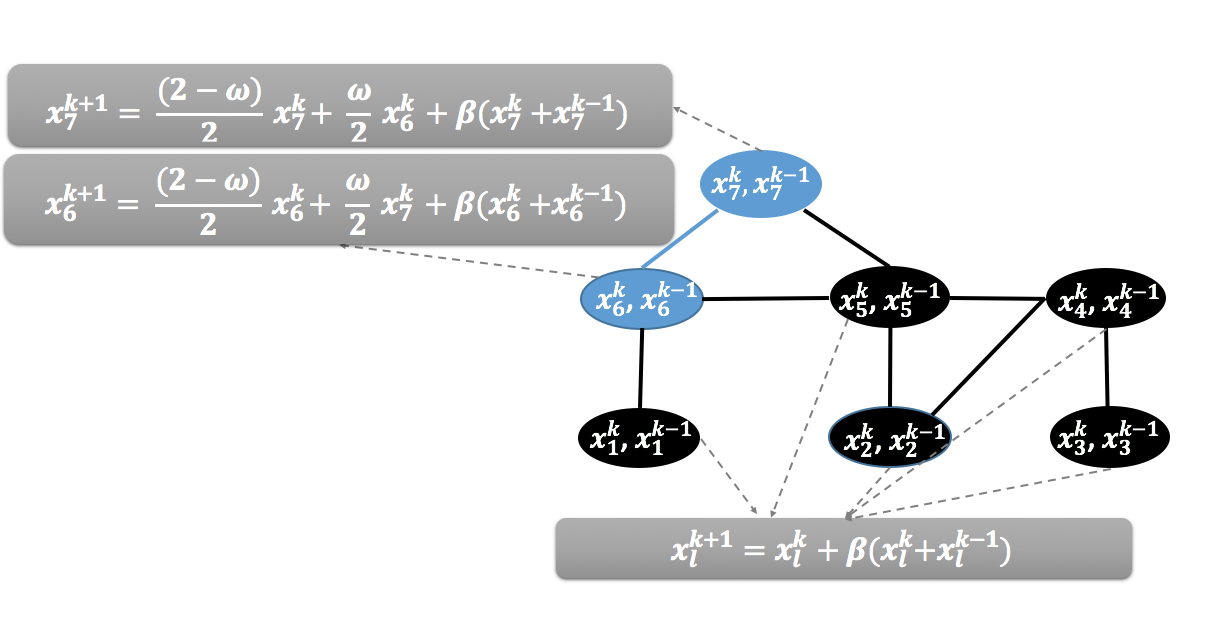}}
  \caption{\footnotesize Example of how mRK works as gossip algorithm. In the presented network the edge that connects nodes $6$ and $7$ is randomly selected. The pair of nodes exchange their information and update their values following the update rule of the Algorithm~\ref{RKmomentum} while the rest of the nodes, $\ell \in [5]$, update their values using only their own previous private values.}
  \label{fig:mRK}
\end{minipage}
\end{figure}

\subsubsection{Connections with existing fast randomized gossip algorithms}
\label{connectionOfAcceleratedMethods}

In the randomized gossip literature there is one particular method closely related to our approach. It was first proposed in \cite{cao2006accelerated} and its analysis under strong conditions was presented in \cite{liu2013analysis}. In this work local memory is exploited by installing shift registers at each agent. In particular we are interested in the case of two registers where the first stores the agent's current value and the second the agent's value before the latest update. The algorithm can be described as follows. Suppose that edge $e=(i,j)$ is chosen at time $k$. Then,
\begin{itemize}
\item Node $i$: $x_i^{k+1}= \omega(\frac{x_i^k+x_j^k}{2})+(1-\omega)x_i^{k-1},$
\item Node $j$: $x_i^{k+1}=  \omega(\frac{x_i^k+x_j^k}{2})+(1-\omega)x_j^{k-1},$
\item Any other node $\ell$: $x_\ell^{k+1}=x_\ell^k,$
\end{itemize}
where $\omega \in [1,2)$. The method was analyzed in \cite{liu2013analysis} under a strong assumption on the probabilities of choosing the pair of nodes, that as the authors mentioned, is unrealistic in practical scenarios, and for networks like the random geometric graphs. At this point we should highlight that the results presented in Chapter~\ref{ChapterMomentum} hold for essentially any distribution $\cD$ \footnote{The only restriction is the exactness condition to be satisfied. See Theorem~\ref{L2}.} and as a result in the proposed gossip variants with heavy ball momentum such problem cannot occur.

Note that, in the special case that we choose $\beta=\omega-1$ in the update rule of Algorithm~\ref{RKmomentum} is simplified to:
\begin{itemize}
\item Node $i$: $x_i^{k+1}= \omega(\frac{x_i^k+x_j^k}{2})+(1-\omega)x_i^{k-1},$
\item Node $j$: $x_i^{k+1}=  \omega(\frac{x_i^k+x_j^k}{2})+(1-\omega)x_j^{k-1},$
\item Any other node $\ell$: $x_\ell^{k+1}=\omega x_\ell^k+(1-\omega)x_\ell^{k-1}.$
\end{itemize}

Recall that in order to apply Theorem~\ref{L2}, we need to assume that $0< \omega < 2$ and $\beta=\omega-1 \geq 0$ which also means that $\omega \in [1,2)$. Thus for $\omega \in [1,2)$ and momentum parameter $\beta=\omega-1$ it is easy to see that our approach is very similar to the shift-register algorithm. Both methods update the selected pair of nodes in the same way. However, in Algorithm~\ref{RKmomentum} the not selected nodes of the network do not remain idle but instead update their values using their own previous information.

By defining the momentum matrix $\bM=\text{\textbf{Diag}}(\beta_{1},\beta_{2},\dots, \beta_{n})$, the above closely related algorithms can be expressed, in vector form, as:
\begin{equation}
\label{updatewithB}
x^{k+1} = x^k - \frac{\omega}{2} (x_i^k-x_j^k)(e_i-e_j) + \bM (x^k-x^{k-1}).
\end{equation}

In particular, in mRK every diagonal element of matrix $\bM$ is equal to $\omega-1$, while in the algorithm of \cite{cao2006accelerated, liu2013analysis} all the diagonal elements are zeros except the two values that correspond to nodes $i$ and $j$ that are equal to $\beta_{i}=\beta_{j}=\omega-1$.

\begin{rem}
The shift register algorithm of \cite{liu2013analysis} and Algorithm~\ref{RKmomentum} of this work can be seen as the two limit cases of the update rule \eqref{updatewithB}. As we mentioned, the shift register method \cite{liu2013analysis} uses only two non-zero diagonal elements in $\bM$, while our method has a full diagonal. We believe that further methods can be developed in the future by exploring the cases where more than two but not all elements of the diagonal matrix $\bM$ are non-zero. It might be possible to obtain better convergence if one carefully chooses these values based on the network topology. We leave this as an open problem for future research.
\end{rem}

\subsubsection{Randomized block Kaczmarz gossip with heavy ball momentum}
Recall that Theorem~\ref{TheoremRBK} explains how RBK (with no momentum and no relaxation) can be interpreted as a gossip algorithm. In this subsection by using this result we explain how relaxed RBK with momentum works. Note that the update rule of RBK with momentum can be rewritten as follows:
\begin{equation}
\label{updateRBK2}
x^{k+1} \overset{\eqref{momentumupdateb0},\eqref{nacsklals}}{=} \omega \left(\bI- \bA_{C:}^\top (\bA_{C:}\bA_{C:}^\top)^\dagger \bA_{C:} \right) x^k+(1-\omega)x^k +\beta(x^k-x^{k-1}),
\end{equation}
and recall that $x^{k+1} =\left(\bI - \bA_{C:}^\top (\bA_{C:}\bA_{C:}^\top)^\dagger \bA_{C:} \right) x^k$ is the update rule of the standard RBK  \eqref{RBKgossip}.

Thus, in analogy to the standard RBK, in the $k^{th}$ step, a random set of edges is selected and $q \leq n$ connected components are formed as a result. This includes the connected components that belong to both sub-graph $\cG_k$ and also the singleton connected components (nodes outside the $\cG_k$). Let us define the set of the nodes that belong in the $r \in [q]$ connected component at the $k^{th}$ step $\cV_r^k$, such that $\cV= \cup_{r\in [q]} \cV_r^k$ and $|\cV|=\sum_{r=1}^{q} |\cV_r^k|$ for any $k>0$. 

Using the update rule \eqref{updateRBK2}, Algorithm~\ref{RBKmomentum} shows how mRBK is updating the private values of the nodes of the network (see also Figure~\ref{fig:mRBK} for the graphical interpretation).

\begin{algorithm}[t!]
	\caption{mRBK: Randomized Block Kaczmarz Gossip with momentum}
	\label{RBKmomentum}
	\small \small
	\begin{algorithmic}[1]
		\State {\bf Parameters:} Distribution $\mathcal{D}$ from which method samples matrices;  stepsize/relaxation parameter $\omega \in \R$;  heavy ball/momentum parameter $\beta$.
		\State {\bf Initialize:} $x^0,x^1 \in \R^n$
		\For{$k=1,2,...$} 
		\State Select a random set of edges $\cS \subseteq \cE$
		\State Form subgraph $\cG_k$ of $\cG$ from the selected  edges 
		\State Node values  are updated as follows:
\begin{itemize}
\item For each connected component $\cV_r^k$ of $\cG_k$, replace the values of its nodes with: 
\begin{equation}
\label{updateruelblock}
x_i^{k+1}=\omega \frac{\sum_{j \in \cV_r^k} x_j^{k}}{|\cV_r^k|} +(1-\omega)x_i^k+\beta (x_i^k-x_i^{k-1}).
\end{equation}
\item Any other node $\ell$: $x_\ell^{k+1}=x_\ell^k+\beta (x_\ell^k - x_\ell^{k-1})$
\end{itemize}
		\EndFor
		\State {\bf Output:} The last iterate $x^k$
	\end{algorithmic}
\end{algorithm}

Note that in the update rule of mRBK the nodes that are not attached to a selected edge (do not belong in the sub-graph $\cG_k$) update their values via $x_\ell^{k+1}=x_\ell^k+\beta (x_\ell^k - x_\ell^{k-1})$. By considering these nodes as singleton connected components their update rule is exactly the same with the nodes of sub-graph $\cG_k$. This is easy to see as follows:
\begin{eqnarray}
x_\ell^{k+1}&\overset{\eqref{updateruelblock}}{=}&\omega \frac{\sum_{j \in \cV_r^k} x_j^{k}}{|\cV_r^k|} +(1-\omega)x_\ell^k+\beta (x_\ell^k-x_\ell^{k-1})\notag\\
&\overset{|\cV_r^k|=1}{=}&\omega x_\ell^k +(1-\omega)x_\ell^k+\beta (x_\ell^k-x_\ell^{k-1})\notag\\
&=& x_\ell^k+\beta (x_\ell^k - x_\ell^{k-1}).
\end{eqnarray}

\begin{rem}
In the special case that only one edge is selected in each iteration ($\bS_k \in \R^{m \times 1}$)  the update rule of mRBK is simplified to the update rule of mRK. In this case the sub-graph $\cG_k$ is the pair of the two selected edges.  
\end{rem}

\begin{rem}
In previous section we explained how several existing gossip protocols for solving the average consensus problem are special cases of the RBK (Theorem~\ref{TheoremRBK}). For example two gossip algorithms that can be cast as special cases of the standard RBK are the path averaging proposed in \cite{benezit2010order} and the clique gossiping \cite{liu2017clique}. In path averaging, in each iteration a path of nodes is selected and its nodes update their values to their exact average ($\omega=1$). In clique gossiping, the network is already divided into cliques and through a random procedure a clique is activated and the nodes of it update their values to their exact average ($\omega=1$). Since mRBK contains the standard RBK as a special case (when $\beta=0$), we expect that these special protocols can also be accelerated with the addition of momentum parameter $\beta \in (0,1)$.
\end{rem}

\begin{figure}[t!]
\begin{minipage}[b]{1.0\linewidth}
  \centering
  \vspace{6pt}
  \centerline{\includegraphics[scale=0.6]{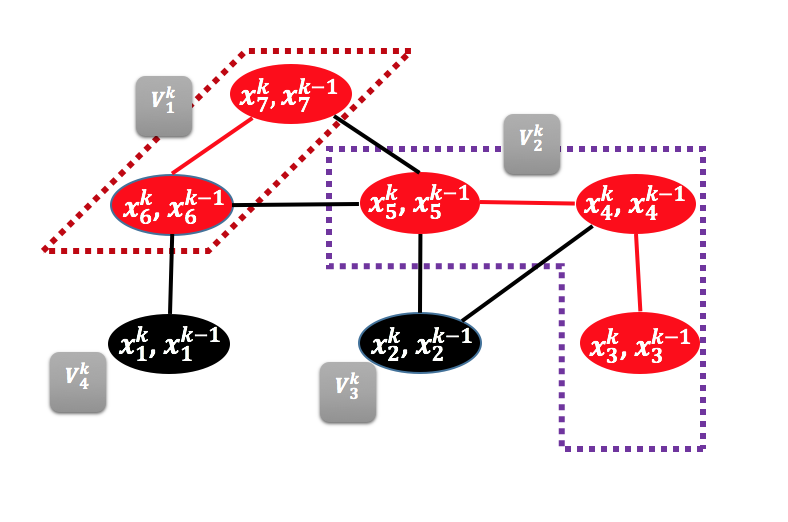}}
  \caption{\footnotesize Example of how the mRBK method works as gossip algorithm. In the presented network the red edges are randomly chosen in the $k^{th}$ iteration, and they form subgraph $\cG_k$ and four connected component. In this figure $V_1^k$ and $V_2^k$ are the two connected components that belong in the subgraph $\cG_k$ while $V_3^k$ and $V_4^k$ are the singleton connected components. Then the nodes update their values by communicate with the other nodes of their connected component using the update rule \eqref{updateruelblock}. For example the node number 5 that belongs in the connected component $V_2^k$ will update its value using the values of node 4 and 3 that also belong in the same component as follows:
$x_5^{k+1}=\omega \frac{x_3^k+x_4^k+x_5^k}{3} +(1-\omega)x_5^k+\beta (x_5^k-x_5^{k-1})$.}
  \label{fig:mRBK}
\end{minipage}
\end{figure}

\subsubsection{Mass preservation}
One of the key properties of some of the most efficient randomized gossip algorithms is mass preservation. That is, the sum (and as a result the average) of the private values of the nodes remains fixed during the iterative procedure ($\textstyle \sum_{i=1}^{n}x_i^{k}=\sum_{i=1}^{n}x_i^{0}, \quad \forall k \geq 1$). The original pairwise gossip algorithm proposed in \cite{boyd2006randomized} satisfied the mass preservation property, while exisiting fast gossip algorithms \cite{cao2006accelerated,liu2013analysis}  preserving a scaled sum.  
In this subsection we show that mRK and mRBK gossip protocols presented above satisfy the mass preservation property. In particular, we prove mass preservation for the case of the block randomized gossip protocol (Algorithm~\ref{RBKmomentum}) with momentum. This is sufficient since the randomized Kaczmarz gossip with momentum (mRK), Algorithm~\ref{RKmomentum} can be cast as special case.

\begin{thm}
Assume that $x^0=x^1=c$. That is, the two registers of each node have the same initial value.  Then for the Algorithms~\ref{RKmomentum} and \ref{RBKmomentum} we have $\sum_{i=1}^{n}x_i^k=\sum_{i=1}^{n}c_i$ for any $k\geq 0$ and as a result, $\frac{1}{n}\sum_{i=1}^{n}x_i^k=\bar{c}$. 
\end{thm}
\begin{proof}
We prove the result for the more general Algorithm~\ref{RBKmomentum}. Assume that in the $k^{th}$ step of the method $q$ connected components are formed.  Let the set of the nodes of each connected component be $\cV_r^k$ so that $\cV= \cup_{r=\{1,2,...q\}} \cV_r^k$ and $|\cV|=\sum_{{r}=1}^{q} |\cV_r^k|$ for any $k>0$.  Thus:
\begin{equation}
\label{generalsum}
\textstyle \sum_{i=1}^{n}x_i^{k+1}=\sum_{i \in \cV_1^k} x_i^{k+1} +\dots + \sum_{i \in \cV_q^k} x_i^{k+1}.
\end{equation}
Let us first focus, without loss of generality, on  connected component $r \in [q]$ and simplify the expression for the sum of its nodes:
\begin{eqnarray}
 \sum_{i\in \cV_r^k} x_i^{k+1}
&\overset{\eqref{updateruelblock}}=& \textstyle \sum_{i \in \cV_r^k} \omega \frac{\sum_{j \in \cV_r^k} x_j^{k}}{|\cV_r^k|} +   (1-\omega) \sum_{i \in \cV_r^k} x_i^k  +\beta \sum_{i \in \cV_r^k}  (x_i^k-x_i^{k-1})\notag\\
 &=&|\cV_r^k| \frac{\omega \sum_{j \in \cV_r^k} x_j^{k}}{|\cV_r^k|}+ (1-\omega) \sum_{i \in \cV_r^k} x_i^k 
 +\beta \sum_{i \in \cV_r^k}  (x_i^k-x_i^{k-1})\notag\\
 &=&(1+\beta) \sum_{i \in \cV_r^k}x_i^k-\beta \sum_{i \in \cV_r^k}x_i^{k-1}.
\end{eqnarray}
 By substituting this for all $r \in [q]$ into the right hand side of \eqref{generalsum} and from the fact that $\cV= \cup_{r\in [q]}
 \cV_r^k$, we obtain:
$$ \sum_{i=1}^{n}x_i^{k+1}= (1+\beta) \sum_{i=1}^{n}x_i^k-\beta \sum_{i=1}^{n} x_i^{k-1}.$$
Since $x^0=x^1$, we have $\sum_{i=1}^{n}x_i^{0}=\sum_{i=1}^{n}x_i^{1}$, and as a result $
\sum_{i=1}^{n}x_i^{k} = \sum_{i=1}^{n}x_i^{0}$ for all $ k \geq 0$.
\end{proof}
\subsection{Provably accelerated randomized gossip algorithms}
\label{accSubsection}

In this subsection we focus on one specific case of the Sketch and Project framework, the RK method \eqref{RK}. We present two accelerated variants of RK where the Nesterov's momentum is used, for solving consistent linear systems and we describe their theoretical convergence results. Based on these methods we propose two provably accelerated gossip protocols, along with some remarks on their implementation.

\subsubsection{Accelerated Kaczmarz methods using Nesterov's momentum}
\label{AcceleratedVariants}
There are two different but very similar ways to provably accelerate the randomized Kaczmarz method using Nesterov's acceleration. The first paper that proves \emph{asymptotic convergence }with an accelerated linear rate is \cite{liu2016accelerated}. The proof technique is similar to the framework developed by Nesterov in \cite{nesterov2012efficiency} for the acceleration of coordinate descent methods.  In \cite{tu2017breaking,gower2018accelerated} a modified version for the selection of the parameters was proposed and a \emph{non-asymptotic} accelerated linear rate was established. In Algorithm~\ref{alg:AccKaczmarz}, pseudocode of the Accelerated Kaczmarz method (AccRK) is presented where both variants can be cast as special cases, by choosing the parameters with the correct way. 
\begin{algorithm}[!h]
\begin{algorithmic}[1]
\State Data: Matrix $\mA\in \R^{m\times n}$; vector $b\in \R^m$
\State Choose $x^0\in \R^n$ and set $v^0 = x^0$
\State Parameters: 
Evaluate the sequences of the scalars $\alpha_k, \beta_k ,\gamma_k$ following one of two possible options.
\For {$k = 0, 1, 2, \dots, K$}
 \State $y^k = \alpha_k v^k + (1-\alpha_k) x^k$ 
\State Draw a fresh sample $i_k \in [m]$ with equal probability 
\State $x^{k+1} = y^k - \frac{\bA_{i_k :} y^k -b_{i_k}}{\|\bA_{i_k :}\|^2} \bA_{i_k :}^ \top.$ 
\State $v^{k+1} = \beta_k v^k + (1-\beta_k) y^k -\gamma_k \frac{\bA_{i_k :} y^k -b_{i_k}}{\|\bA_{i_k :}\|^2} \bA_{i_k :}^ \top.$
\EndFor
\end{algorithmic}
\caption{Accelerated Randomized Kaczmarz Method (AccRK)}
\label{alg:AccKaczmarz}
\end{algorithm}

There are two options for selecting the parameters of the AccRK for solving consistent linear systems with normalized matrices, which we describe next. 
\begin{enumerate}
\item From \cite{liu2016accelerated}: 
Choose $\lambda \in [0,\lambda_{\min}^+(\bA^\top \bA)]$ and set $\gamma_{-1}=0$.
Generate the sequence $\{\gamma_k: k=0,1,\dots, K+1\}$ by choosing $\gamma_k$ to be the largest root of $$\gamma_k^2-\frac{\gamma_k}{m}=(1-\frac{\gamma_k}\lambda{m})\gamma_{k-1}^2,$$ and generate the sequences $\{\alpha	_k :  k=0,1,\dots,K+1\}$ and $\{\beta_k :  k=0,1,\dots,K+1\}$ by setting $$\alpha_k=\frac{m-\gamma_k\lambda}{\gamma_k(m^2-\lambda)}, \quad \beta_k=1-\frac{\gamma_k \lambda}{m}.$$
\item From \cite{gower2018accelerated}: Let 
\begin{equation}
\label{thenu}
\nu= \max_{u \in \range{\mA^\top}   } \frac{ u^\top \left[\sum_{i=1}^m  \mA_{i:}^\top \mA_{i:} (\mA^\top \mA)^\dagger \mA_{i:}^\top \mA_{i:} \right]u }{ u^\top \frac{\bA^\top \bA }{m}u }.
\end{equation}
Choose the three sequences to be fixed constants as follows:
$\beta_k =\beta = 1-\sqrt{ \frac{\lambda_{\min}^+(\bW)}{\nu}} $, \;$\gamma_k=\gamma = \sqrt{ \frac{1}{\lambda_{\min}^+(\bW) \nu}} $, \; $\alpha_k=\alpha =  \frac{1}{1+\gamma \nu} \in (0,1)$ where $\bW=\frac{\bA^\top \bA}{m}$.
\end{enumerate}
\subsubsection{Theoretical guarantees of AccRK}
The two variants (Option 1 and Option 2) of AccRK are closely related, however their convergence analyses are different. Below we present the theoretical guarantees of the two options as presented in \cite{liu2016accelerated} and \cite{gower2018accelerated}.
\begin{thm}[\cite{liu2016accelerated}]
\label{causjoakls}
Let $\{x^k\}_{k=0}^\infty$ be the sequence of random iterates produced by Algorithm~\ref{alg:AccKaczmarz} with the Option 1 for the parameters. Let $\bA$ be normalized matrix and let $\lambda \in [0,\lambda_{\min}^+(\bA^\top \bA)]$. Set $\sigma_1=1+\frac{\sqrt{\lambda}}{2m}$ and $\sigma_2=1-\frac{\sqrt{\lambda}}{2m}$. Then for any $k\geq 0$ we have that:
$$\Exp[\|x^{k}-x^*\|^2 ] \leq \frac{4 \lambda}{(\sigma_1^{k}-\sigma^{k}_2)^2}\|x^0-x^*\|^2_{(\bA^\top \bA)^\dagger}.$$
\end{thm}
\begin{cor}[\cite{liu2016accelerated}]
\label{basjda}
Note that as $k\rightarrow\infty$, we have that $\sigma^k_2\rightarrow 0$. This means that the decrease of the right hand side is governed mainly by the behavior of the term $\sigma_1$ in the denominator and as a result the method converge \emph{asymptotically} with a decrease factor per iteration:
$\sigma_1^{-2}=(1+\frac{\sqrt{\lambda}}{2m})^{-2}\approx 1-\frac{\sqrt{\lambda}}{m}.$
That is, as $k\rightarrow\infty$:
$$\Exp[\|x^{k}-x^*\|^2 ] \leq \left(1- \sqrt{\lambda}/m \right)^k 4 \lambda \|x^0-x^*\|^2_{(\bA^\top \bA)^\dagger}$$
\end{cor}

Thus, by choosing $\lambda=\lambda_{\min}^+$ and for the case that $\lambda_{\min}^+$ is small, Algorithm~\ref{alg:AccKaczmarz} will have significantly faster convergence rate than RK. Note that the above convergence results hold only for normalized matrices $\bA \in \R^{m \times n}$, that is matrices that have $\|\bA_{i:}\|=1$ for any $i \in m$. 

Using Corollary~\ref{basjda}, Algorithm~\ref{alg:AccKaczmarz} with the first choice of the parameters converges linearly with rate $\left(1- \sqrt{\lambda}/m \right)$. That is, it requires $O \left( m/ \sqrt{\lambda} \log(1/\epsilon) \right)$ iterations to obtain accuracy $\Exp[\|x^{k}-x^*\|^2 ]\leq   \epsilon 4 \lambda \|x^0-x^*\|^2_{(\bA^\top \bA)^\dagger} $.

\begin{thm}[\cite{ gower2018accelerated}]
\label{theorem2}
Let $\bW=\frac{\bA^\top \bA}{m}$ and let assume exactness\footnote{Note that in this setting $\bB=\bI$, which means that $\bW=\Exp[\bZ]$, and the exactness assumption takes the form ${\rm Null}(\mW) = {\rm Null}(\mA)$.}. Let $\{x^k,y^k,v^k\}$ be the iterates  of Algorithm~\ref{alg:AccKaczmarz} with the Option 2 for the parameters.  Then 
$$\Psi^k \leq \left(1-\sqrt{\lambda_{\min}^+(\bW)/\nu}\right)^k \Psi^0,$$
where $\Psi^k =\Exp \left[\| v^k - x^*\|_{\mW^\dagger}^2+  \frac{1}{\mu  } \|x^{k}-x^*\|^2 \right]$.
\end{thm}

The above result implies that Algorithm~\ref{alg:AccKaczmarz} converges linearly with rate $1-\sqrt{\lambda_{\min}^+(\bW)/\nu}$, which translates to a total of $\cO\left(\sqrt{\nu/\lambda_{\min}^+(\bW)}\log(1/\epsilon)\right)$ iterations to bring the quantity $\Psi^k$ below $\epsilon>0$. It can be shown that $1 \leq \nu \leq 1/\lambda_{\min}^+(\bW)$, (Lemma 2 in \cite{gower2018accelerated}) where $\nu$ is as defined in \eqref{thenu}. Thus,
$\sqrt{\frac{1}{\lambda_{\min}^+(\bW)}} \leq \sqrt{\frac{\nu}{\lambda_{\min}^+(\bW)}} \leq \frac{1}{\lambda_{\min}^+(\bW)},$
which means that the rate of AccRK (Option 2) is always better than that of the RK with unit stepsize which is equal to $\cO\left(\frac{1}{\lambda_{\min}^+(\bW)} \log(1/\epsilon)\right)$ (see Theorem~\ref{ConvergenceSketchProject}). 

In \cite{ gower2018accelerated}, Theorem~\ref{theorem2} has been proposed for solving more general consistent linear systems (the matrix $\bA$ of the system is not assumed to be normalized). In this case $\bW=\Exp[\bZ]$ and the parameter $\nu$ is slightly more complicated than the one of equation \eqref{thenu}. We refer the interested reader to \cite{ gower2018accelerated} for more details.

\paragraph{Comparison of the convergence rates:}
Before describe the distributed nature of the AccRK and explain how it can be interpreted as a gossip algorithm, let us compare the convergence rates of the two options of the parameters for the case of general normalized consistent linear systems ($\|\bA_{i:}\|=1$ for any $i \in [m]$).

Using Theorems~\ref{causjoakls} and \ref{theorem2}, it is clear that the iteration complexity of AccRK is 
\begin{equation}
\label{IterCompleOption1}
O \left(\frac{m}{\sqrt{\lambda}} \log(1/\epsilon) \right)\overset{\lambda=\lambda_{\min}^+(\bA^\top \bA)}{=} O \left( \frac{m}{\sqrt{\lambda_{\min}^+(\bA^\top \bA)}} \log(1/\epsilon) \right),
\end{equation} and 
\begin{equation}
\label{IterCompleOption2}
\cO\left(\sqrt{\frac{\nu m}{\lambda_{\min}^+(\bA^\top \bA)}}\log(1/\epsilon)\right),
\end{equation} for the Option 1 and Option 2 for the parameters, respectively.

In the following derivation we compare the iteration complexity of the two methods.

\begin{lem}
\label{naoiskma}
Let matrices $\bC \in \R^{n \times n}$  and $\bC_i \in \R^{n \times n}$ where $i \in [m]$ be positive semidefinite, and satisfying $\sum_{i=1}^m \bC_i =\bC$. Then 
$$\sum_{i=1}^m \bC_i  \bC^\dagger \bC_i \preceq \bC.$$
\end{lem}

\begin{proof}
From the definition of the matrices it holds that $\bC_i \preceq \bC$ for any $i \in [m]$. Using the properties of Moore-Penrose pseudoinverse,  this implies that
\begin{equation}
\label{cnaosklasdpoad}
\bC_i^\dagger \succeq \bC^\dagger.
\end{equation}
Therefore
\begin{equation}
\label{cbaiusjkalks}
\bC_i= \bC_i \bC_i^\dagger \bC_i \overset{\eqref{cnaosklasdpoad}}{\succeq} \bC_i \bC^\dagger \bC_i.
\end{equation} 
From the definition of the matrices by taking the sum over all $i \in [m]$ we obtain:
$$\bC= \sum_{i=1}^m \bC_i \overset{\eqref{cbaiusjkalks}}{\succeq} \sum_{i=1}^m \bC_i \bC^\dagger \bC_i,$$
which completes the proof.
\end{proof}

Let us now choose $\bC_i = \mA_{i:}^\top \mA_{i:}$ and $\bC=\mA^\top \mA$. Note that from their definition the matrices are positive semidefinite and satisfy $\sum_{i=1}^m \mA_{i:}^\top \mA_{i:}= \mA^\top \mA$.
Using Lemma~\ref{naoiskma} it is clear that:
$$\sum_{i=1}^m  \mA_{i:}^\top \mA_{i:} (\mA^\top \mA)^\dagger \mA_{i:}^\top \mA_{i:} \preceq \bA^\top \bA,$$
or in other words, for any vector $v \notin {\rm Null}{(\bA)}$ we set the inequality
$$ \frac{v^\top \left[\sum_{i=1}^m  \mA_{i:}^\top \mA_{i:} (\mA^\top \mA)^\dagger \mA_{i:}^\top \mA_{i:}\right] v}{v^\top [\bA^\top \bA] v} \leq 1.$$
Multiplying both sides by $m$, we set:
$$ \frac{v^\top \left[\sum_{i=1}^m  \mA_{i:}^\top \mA_{i:} (\mA^\top \mA)^\dagger \mA_{i:}^\top \mA_{i:}\right] v}{v^\top [\frac{\bA^\top \bA}{m}] v} \leq m.$$

Using the above derivation, it is clear from the definition of the parameter $\nu$ \eqref{thenu}, that $\nu \leq m.$
By combining our finding with the bounds already obtained in  \cite{ gower2018accelerated} for the parameter $\nu$, we have that:
\begin{equation}
\label{acsnklasda}
1 \leq \nu \leq \min \left\{ m, \frac{1}{\lambda_{\min}^+(\bW)} \right\}.
\end{equation}
Thus, by comparing the two iteration complexities of equations \eqref{IterCompleOption1} and \eqref{IterCompleOption2} it is clear that Option 2 for the parameters \cite{ gower2018accelerated} is always faster in theory than Option 1 \cite{liu2016accelerated}. To the best of our knowledge, such comparison of the two choices of the parameters for the AccRK was never presented before.

\subsubsection{Accelerated randomized gossip algorithms}
\label{sec:AccGossip}
Having presented the complexity analysis guarantees of AccRK for solving consistent linear systems with normalized matrices, let us now explain how the two options of AccRK behave as gossip algorithms when they are used to solve the linear system $\bA x=0$ where $\bA \in \R^{|\cE| \times n}$ is the normalized incidence matrix of the network. That is, each row $e=(i,j)$ of $\bA$ can be represented as $ (\bA_{e:})^\top= \frac{1}{\sqrt{2}}(e_i -e_j)$ where $e_i$ (resp.$e_j$) is the $i^{th}$ (resp. $j^{th}$) unit coordinate vector in $\R^{n}$.

By using this particular linear system, the expression $\frac{\bA_{i :} y^k -b_{i}}{\|\bA_{i :}\|^2} \bA_{i :}^ \top$ that appears in steps 7 and 8 of AccRK takes the following form when the row $e=(i,j) \in \cE$ is sampled:
$$\frac{\bA_{e :} y^k -b_{i}}{\|\bA_{e :}\|^2} \bA_{e :}^ \top \overset{b=0}{=} \frac{\bA_{e :} y^k}{\|\bA_{e :}\|^2} \bA_{e :}^ \top \overset{\text{form of A}}{=} \frac{ y_i^k - y_j^k}{2}(e_i-e_j).$$

Recall that with $\bL$ we denote the Laplacian matrix of the network. For solving the above AC system (see Definition~\ref{defACsystem}), the standard RK requires $\cO\left( \left(\frac{2m}{\lambda_{\min}^+(\bL)}\right)\log(1/\epsilon)\right)$ iterations to achieve expected accuracy $\epsilon>0$. To understand the acceleration in the gossip framework this should be compared to the $$\cO\left(m \sqrt{\frac{2}{\lambda_{\min}^+(\bL)}} \log(1/\epsilon)\right)$$ of AccRK (Option 1) and the $$\cO\left(\sqrt{\frac{2m\nu}{\lambda_{\min}^+(\bL)} } \log(1/\epsilon)\right)$$ of AccRK (Option 2).

Algorithm~\ref{alg:acceleratedNew} describes in a single framework how the two variants of AccRK of Section~\ref{AcceleratedVariants} behave as gossip algorithms when are used to solve the above linear system. Note that each node $\ell \in \cV$ of the network has two local registers to save the quantities $v^k_\ell$ and $x^k_\ell$. In each step using these two values every node $\ell \in \cV$ of the network (activated or not) computes the quantity $y^k_\ell =\alpha_k v^k_\ell + (1-\alpha_k) x^k_\ell$. Then in the $k^{th}$ iteration the activated nodes $i$ and $j$ of the randomly selected edge $e=(i,j)$ exchange their values $y^k_i$ and $y^k_j$ and update the values of $x^k_i$,  $x^k_j$ and $v^k_i$, $v^k_j$ as shown in Algorithm~\ref{alg:acceleratedNew}. The rest of the nodes use only their own $y^k_\ell$ to update the values of $v^k_i$ and $x^k_i$ without communicate with any other node.

The parameter $\lambda^+_{\min}(\bL)$ can be estimated by all nodes in a decentralized manner using the method described in~\cite{charalambous2016distributed}. In order to implement this algorithm, we assume that all nodes have synchronized clocks and that they know the rate at which gossip updates are performed, so that inactive nodes also update their local values. This may not be feasible in all applications, but when it is possible (e.g., if nodes are equipped with inexpensive GPS receivers, or have reliable clocks) then they can benefit from the significant speedup achieved.

\begin{algorithm}[!h]
\begin{algorithmic}[1]
\State \textbf{Data:} Matrix $\mA\in \R^{m\times n}$ (normalized incidence matrix); vector $b=0\in \R^m$
\State Choose $x^0\in \R^n$ and set $v^0 = x^0$
\State\textbf{ Parameters:} 
Evaluate the sequences of the scalars $\alpha_k, \beta_k ,\gamma_k$ following one of two possible options.
\For {$k = 0, 1, 2, \dots, K$}
\State Each node $\ell \in \cV$ evaluate $y^k_\ell = \alpha_k v^k_\ell + (1-\alpha_k) x^k_\ell$.
\State Pick an edge $e=(i,j)$ uniformly at random.
\State Then the nodes update their values as follows:
\begin{itemize}
\item The selected node $i$ and node $j$: 
$$x^{k+1}_i = x^{k+1}_j  = (y^k_i+y^k_j)/2$$ 
$$v^{k+1}_i = \beta_k v^k_i + (1-\beta_k) y^k_i -\gamma_k (y^k_i-y^k_j)/2$$
$$v^{k+1}_j = \beta_k v^k_j + (1-\beta_k) y^k_j -\gamma_k (y^k_j-y^k_i)/2$$
\item Any other node $\ell \in \cV$:
$$x^{k+1}_\ell = y^k_\ell \quad,\quad v^{k+1}_\ell = \beta_k v^k_\ell + (1-\beta_k) y^k_\ell$$
\end{itemize}
\EndFor
\end{algorithmic}
\caption{Accelerated Randomized Gossip Algorithm (AccGossip)}
\label{alg:acceleratedNew}
\end{algorithm}

\section{Dual Randomized Gossip Algorithms}
\label{DualBlock}
An important tool in optimization literature is duality. In our setting, instead of solving the original minimization problem (primal problem) one may try to develop dual in nature methods that have as a goal to directly solve the dual maximization problem.  Then the primal solution can be recovered through the use of optimality conditions and the development of an affine mapping between the two spaces (primal and dual).

In this section, using existing dual methods and the connection already established between the two areas of research (methods for linear systems and gossip algorithms), we present a different viewpoint that allows the development of novel dual randomized gossip algorithms. 

Without loss of generality we focus on the case of $\bB=\bI$ (no weighted average consensus). For simplicity, we formulate the AC system as the one with the incidence matrix of the network ($\bA=\bQ$) and focus on presenting the distributed nature of dual randomized gossip algorithms with no momentum. While we focus only on no-momentum protocols, we note that accelerated variants of the dual methods could be easily obtained using tools from Section~\ref{AccelerateGossip}.

\subsection{Dual problem and SDSA}
As we have already presented in Section~\ref{BestaprooximationSection_INtro}, the Lagrangian dual of the best approximation problem \eqref{BestApproximation_IntroThesis} is the (bounded) unconstrained concave quadratic maximization problem:
\begin{equation}\label{eq:Dual}
\max_{y\in \R^m} D(y) \eqdef (b-\bA x^0)^\top y - \frac{1}{2}\|\bA^\top y\|_{\bB^{-1}}^2.
\end{equation}

A direct method for solving the dual problem is Stochastic Dual Subspace Accent (SDSA), a randomized iterative algorithm first proposed in \cite{gower2015stochastic}, which updates the dual vectors $y^k$ as follows: 
\begin{equation}
\label{alg:dual}
 y^{k+1}= y^k + \omega \bS_k \left( \bS_k^\top \bA \bB^{-1} \bA^\top  \bS_k \right)^\dagger \bS_k^\top \left( b-\bA(x^0+ \bB^{-1} \bA^\top y^k) \right).
\end{equation}

In Section~\ref{BestaprooximationSection_INtro} we showed that the iterates $\{x^k\}_{k\geq0}$ of the sketch and project method (Algorithm~\ref{FullSkecth}) can be arised as affine images of the iterates $\{y^k\}_{k\geq0}$ of the dual method \eqref{alg:dual} through the mapping: 
\begin{equation}
\label{eq:dual-corresp}
x^{k} = \phi(y^k) =  x^0 + \bB^{-1} \mA^\top y^k, 
\end{equation}
and we provided a proof for the linear convergence of SDSA (see Theorem~\ref{TheoremSDSA_IntroThesis}). Recall that SDSA and the sketch and project method (Algorithm~\ref{FullSkecth}) converge to a solution of the dual problem and primal problem, respectively, with exactly the same convergence rate. 

Let us choose $\bB=\bI$. In the special case that the random matrix $\bS_k$ is chosen randomly from the set of unit coordinate/basis vectors in $\R^m$,  the dual method \eqref{alg:dual} is the randomized coordinate descent \cite{leventhal2010randomized,  richtarik2014iteration}, and the corresponding primal method is RK \eqref{RK}.  More generally, if $\bS_k$ is a random column submatrix of the $m \times m$ identity matrix, the dual method is the randomized Newton method \cite{qu2015sdna}, and the corresponding primal method is RBK \eqref{RBK}. Next we shall describe the more general block case in more detail.

\subsection{Randomized Newton method as a dual gossip algorithm}
\label{RNMdual}
In this subsection we bring a new insight into the randomized gossip framework by presenting how the dual iterative process that is associated to RBK method solves the AC problem with $\bA=\bQ$ (incidence matrix). Recall that the right hand side of the linear system is $b=0$. For simplicity, we focus on the case of $\bB=\bI$ and $\omega=1$.

Under this setting ($\bA=\bQ$, $\bB=\bI$ and $\omega=1$) the dual iterative process \eqref{alg:dual} takes the form:
\begin{equation}
\label{SDSA}
 y^{k+1}= y^k - \bI_{C:}(\bI_{C:}^\top \bQ \bQ^\top \bI_{C:})^\dagger \bQ(x^0+\bQ^\top y^k),
\end{equation}
and from Theorem~\ref{TheoremSDSA_IntroThesis} converges to a solution of the dual problem as follows:
$$\Exp \left[D(y^*)-D(y^k) \right] \leq \left[1 - \lambda_{\min}^+ \left( \Exp\left[\bQ_{C:}^\top (\bQ_{C:}\bQ_{C:}^\top)^\dagger \bQ_{C:}\right] \right) \right]^k \left[D(y^*)-D(y^0)\right].$$
Note that the convergence rate is exactly the same with the rate of the RBK under the same assumptions (see \eqref{anisojxalksda}).

This algorithm is a randomized variant of the Newton method applied to the problem of maximizing the quadratic function $D(y)$ defined in \eqref{eq:Dual}. Indeed, in each iteration we perform the update $y^{k+1} = y^k +\bI_{C:} \lambda^k$, where $\lambda^k$ is chosen greedily so that $D(y^{k+1})$ is maximized. In doing so, we invert a random principal submatrix of the Hessian of $D$, whence the name.  

 \textit{Randomized Newton Method} (RNM) was first proposed by Qu et al.\ \cite{qu2015sdna}. RNM was first analyzed as an algorithm for minimizing \emph{smooth strongly convex functions}. In \cite{gower2015stochastic} it was also extended to the case of a \emph{smooth but weakly convex quadratics}.  This method was not previously associated with any gossip algorithm.

The most important distinction of RNM compared to  existing gossip algorithms is that it operates with values that are associated to the {\em edges} of the network. To the best of our knowledge, it is the first {\em randomized dual gossip method}. In particular, instead of iterating over values  stored at  the nodes, RNM uses these values to update ``dual weights'' $y^k \in \R^m$ that correspond to the edges $\cE$ of the network. However, deterministic dual  distributed averaging algorithms were proposed before  \cite{rabbat2005generalized, ghadimi2014admm}. Edge-based methods have also been proposed before; in particular  in \cite{wei20131} an asynchronous distributed ADMM algorithm presented for solving the more general consensus optimization problem with convex functions. 

\textbf{Natural Interpretation.} In  iteration $k$, \emph{RNM} (Algorithm~\eqref{SDSA}) executes the following steps:
1) Select a random set of edges $\cS_k \subseteq \cE$, 2)  Form a subgraph $\cG_k$ of $\cG$ from the selected edges, 3) The values of the edges in each connected component of $\cG_k$ are updated: their new values are a linear combination of the private values of the nodes belonging to the connected component and of the adjacent edges of their connected components. (see also example of Figure~\ref{figureRNM}).

\textbf{Dual Variables as Advice.} The weights $y^k$ of the edges  have a natural interpretation as \textit{advice} that each selected node receives from the network in order to update its value (to one that will eventually converge to the desired average).

Consider RNM performing the $k^{th}$ iteration and let $\cV_r$ denote the set of nodes of the selected connected component that node $i$ belongs to. Then, from Theorem~\ref{TheoremRBK} we know that $x_i^{k+1}=\sum_{i \in \cV_r} x_i^k / |\cV_r|$. Hence, by using \eqref{eq:dual-corresp}, we obtain the following identity:
\begin{equation}
\label{advice}
\textstyle
 (\bA^\top y^{k+1})_i=\frac{1}{|\cV_r|} \sum_{i \in \cV_{r}}(c_i+(\bA^\top y^{k})_i)- c_i.
\end{equation}
Thus in each step $(\bA^\top y^{k+1})_i$ represents the term (advice) that must be added to the initial value $c_i$ of node $i$ in order to update its value to the average of the values of the nodes of the connected component  $i$ belongs to.

\textbf{Importance of the dual perspective:} It was shown  in  \cite{qu2015sdna} that when RNM (and as a result, RBK, through the affine mapping \eqref{eq:dual-corresp}) is viewed as a family of methods indexed by the size $\tau = |\cS|$ (we choose $\cS$ of fixed size in the experiments),   then $\tau \to 1/(1-\rho)$, where $\rho$ is defined in \eqref{RateRho}, decreases {\em superlinearly} fast in $\tau$. That is, as $\tau$ increases by some factor, the iteration complexity  drops by a factor that is at least as large. Through preliminary numerical experiments in Section~\ref{cakslas} we experimentally show that this is true for the case of AC systems as well.
\begin{figure}[htb]
\begin{minipage}[b]{1.0\linewidth}
  \centering
  \centerline{\includegraphics[scale=0.4]{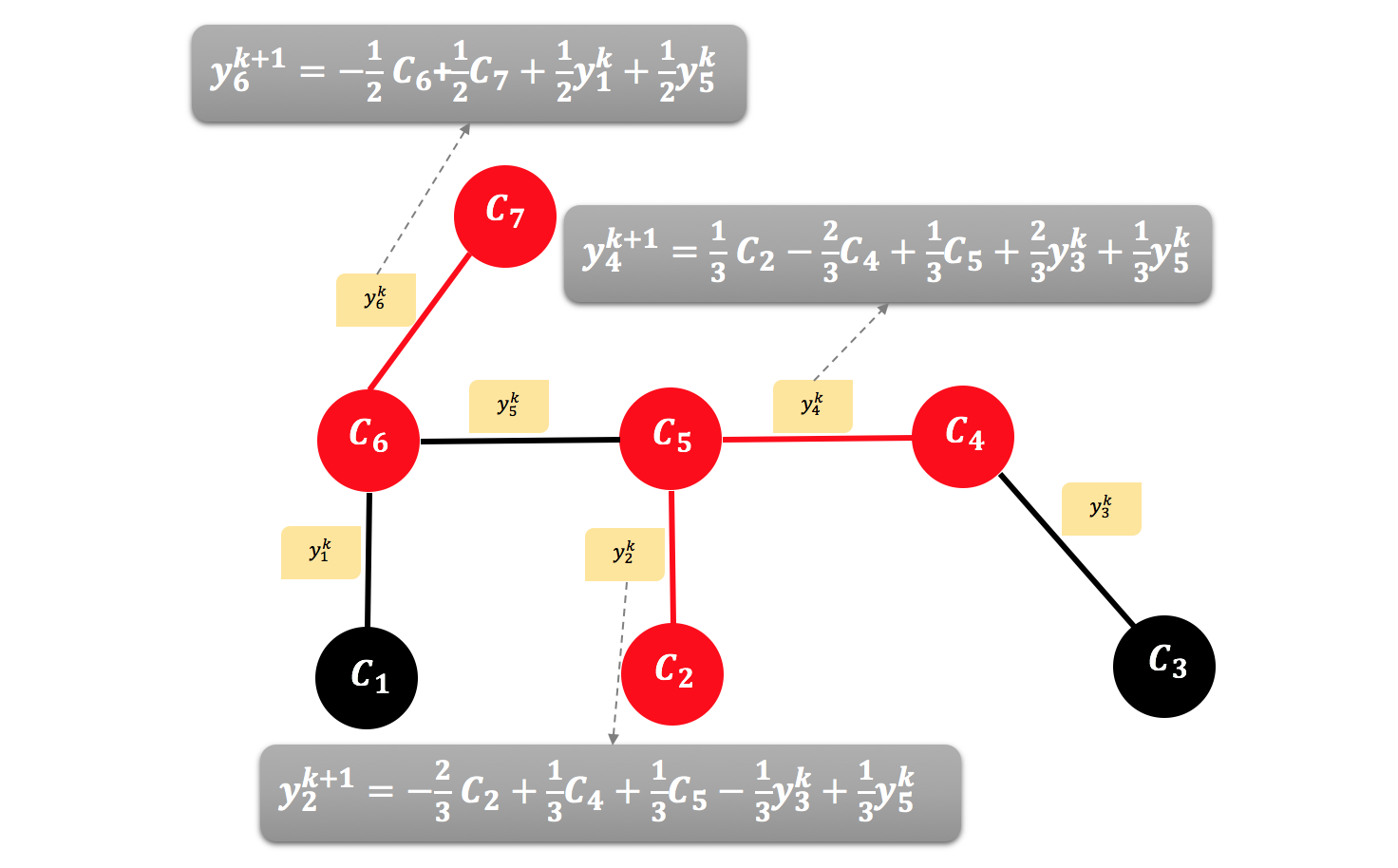}}
  \caption{\footnotesize Example of how the RNM method works as gossip algorithm. In this specific case 3 edges are selected and form a sub-graph with two connected components. Then the values at the edges update their values using the private values of the nodes belonging to their connected component and the values associate to the adjacent edges of their connected components.}
  \label{figureRNM}
\end{minipage}
\end{figure}

\section{Further Connections Between Methods for Solving Linear Systems and Gossip Algorithms}
\label{FurtherConnections}
In this section we highlight some further interesting connections between linear systems solvers and gossip protocols for average consensus:
\begin{itemize}
\item \textbf{Eavesdrop gossip as special case of Kaczmarz-Motzkin method.}
In \cite{ustebay2008greedy} greedy gossip with eavesdropping (GGE), a novel randomized gossip algorithm for distributed computation of the average consensus problem was proposed and analyzed. 
In particular it was shown that that greedy updates of GGE lead to rapid convergence. In this protocol, the greedy updates are made possible by exploiting the broadcast nature of wireless communications. During the operation of GGE, when a node decides to gossip, instead of choosing one of its neighbors at random, it makes a greedy selection, choosing the node which has the value most different from its own. In particular the method behaves as follows:

At the $k^{th}$ iteration of GGE, a node $i_k$ is chosen uniformly at random from $[n]$. Then, $i_k$ identifies a neighboring node $j_k \in \cN_i$ satisfying:
$$j_k \in \max_{j \in \cN_i} \left\{ \frac{1}{2} (x_i^k-x_j^k)^2 \right\}$$
which means that the selected node $i_k$ identifies a neighbor that currently has the most different value from its own. This choice is possible because each node $i\in \cV$  maintains not only its own local variable $x_i^k$, but also a copy of the current values at its neighbors $x_j^k$ for $j \in \cN_i$. In the case that node $i_k$ has multiple neighbors whose values are all equally (and maximally) different from  its current value, it chooses one of these neighbors at random. Then node $i_k$ and $j_k$ update their values to:
$$x_i^{k+1}=x_j^{k+1}=\frac{1}{2} (x_i^k+x_j^k).$$

In the area of randomized methods for solving large linear system there is one particular method,  the Kaczmarz-Motzkin algorithm \cite{de2017sampling, haddock2018motzkin} that can work as gossip algorithm with the same update as the GGE when is use to solve the homogeneous linear system with matrix the Incidence matrix of the network.

Update rule of Kaczmarz-Motzkin algorithm (KMA) \cite{de2017sampling, haddock2018motzkin}:
\begin{enumerate}
\item Choose sample of $d_k$ constraints, $P_k$, uniformly at random from among the rows of matrix $\bA$.
\item From among these $d_k$ constraints, choose $t_k=\text{argmax}_{i \in P_k} \bA_{i:} x^k-b_i.$
\item Update the value: $x^{k+1}=x^k - \frac{\bA_{t_k :} x^k -b_{i}}{\|\bA_{t_k :}\|^2} \bA_{t_k :}^ \top.$
\end{enumerate}

It is easy to verify that when the Kaczmarz-Motzkin algorithm is used for solving the AC system with $\bA=\bQ$ (incidence matrix) and in each step of the method the chosen constraints $d_k$ of the linear system correspond to edges attached to one node it behaves exactly like the GGE. From numerical analysis viewpoint an easy way to choose the constraints $d_k$ that are compatible to the desired edges is in each iteration to find the indexes of the non-zeros of a uniformly at random selected column (node) and then select the rows corresponding to these indexes.

Therefore, since GGE \cite{ustebay2008greedy} is a special case of the KMA (when the later applied to special AC system with Incidence matrix) it means that we can obtain the convergence rate of GGE by simply use the tight conergence analysis presented in \cite{de2017sampling, haddock2018motzkin} \footnote{Note that the convergence theorems of \cite{de2017sampling, haddock2018motzkin} use $d_k=d$. However, with a small modification in the original proof the theorem can capture the case of different $d_k$.}. In \cite{ustebay2008greedy} it was mentioned that analyzing the convergence behavior of GGE is non-trivial and not an easy task. By establishing the above connection the convergence rates of GGE can be easily obtained as special case of the theorems presented in \cite{de2017sampling}. 

In Section~\ref{AccelerateGossip} we presented provably accelerated variants of the pairwise gossip algorithm and of its block variant. Following the same approach one can easily develop accelerated variants of the GGE using the recently proposed analysis for the accelerated Kaczmarz-Motzkin algorithm presented in \cite{morshed2019accelerated}. 

\item \textbf{Inexact Sketch and Project Methods:}

In Chapter~\ref{ChapterInexact}, several inexact variants of the sketch and project method \eqref{FullSkecth} have been proposed. As we have already mentioned the sketch and project method is a two step procedure algorithm where first the sketched system is formulated and then the last iterate $x^k$ is \textit{exactly} projected into the solution set of the sketched system. In Chapter~\ref{ChapterInexact}, we replace the exact projection with an inexact variant and we suggest to run a different algorithm (this can be the sketch and project method itself) in the sketched system to obtain an approximate solution.  It was shown that in terms of time the inexact updates can be faster than their exact variants.

In the setting of randomized gossip algorithms for the AC system with Incidence matrix ($\bA=\bQ$) , $\bB=\bI$ and $\omega =1$ a variant of the inexact sketch and project method will work as follows (similar to the update proved in Theorem~\ref{TheoremRBK}):
\begin{enumerate}
\item Select a random set of edges $C \subseteq \cE$.
\item Form subgraph $\cG_k$ of $\cG$ from the selected edges.
\item Run the pairwise gossip algorithm of \cite{boyd2006randomized} (or any variant of the sketch and project method) on the subgraph $\cG_k$ until an accuracy $\epsilon$ is achieved (reach a neighborhood of the exact average).
\end{enumerate}

\item \textbf{Non-randomized gossip algorithms as special cases of Kaczmarz methods:}

In the gossip algorithms literature there are efficient protocols that are not randomized\cite{mou2010deterministic, he2011periodic, liu2011deterministic, yu2017distributed}. Typically, in these algorithms the pairwise exchanges between nodes it happens in a deterministic, such as predefined cyclic, order. For example, $T$-periodic gossiping is a protocol which stipulates that each node must interact with each of its neighbours exactly once every $T$ time units. It was shown that under suitable connectivity assumptions of the network $\mathcal{G}$, the $T$-periodic gossip sequence will converge at a rate determined by the magnitude of the second largest eigenvalue of the stochastic matrix determined by the sequence of pairwise exchanges which occurs over a period. It has been shown that if the underlying graph is a tree, the mentioned eigenvalue is constant for all possible $T$-periodic gossip protocols.

In this work we focus only on randomized gossip protocols. However we speculate that the above non-randomized gossip algorithms would be able to express as special cases of popular non-randomized projection methods for solving linear systems   \cite{popa2017convergence, nutini2016convergence, dutight}.  Establishing connections like that is an interesting future direction of research and can possibly lead to the development of novel block and accelerated variants of many non-randomized gossip algorithms, similar to the protocols we present in Sections~\ref{skecthsection} and \ref{AccelerateGossip}.

\end{itemize}

\section{Numerical Evaluation}
\label{experimentsGossip}
In this section, we empirically validate our theoretical results and evaluate the performance of the proposed randomized gossip algorithms. The section is divided into four main parts, in each of which we highlight a different aspect of our contributions. 

In the first experiment, we numerically verify the linear convergence of the Scaled RK algorithm (see equation \eqref{scaledRK})  for solving the weighted average consensus problem presented in Section~\ref{weightedAC}. In the second part, we explain the benefit of using block variants in the gossip protocols where more than two nodes update their values in each iteration (protocols presented in Section~\ref{BlockGossip}).  In the third part, we explore the performance of the faster and provably accelerated gossip algorithms proposed in Section~\ref{AccelerateGossip}. In the last experiment, we numerically show that relaxed variants of the pairwise randomized gossip algorithm converge faster than the standard randomized pairwise gossip with unit stepsize (no relaxation). This gives a specific setting where the phenomenon of over-relaxation of iterative methods for solving linear systems is beneficial.

In the comparison of all gossip algorithms we use the relative error measure $\|x^k-x^*\|_{\bB}^2 / \|x^0-x^*\|_{\bB}^2 $ where $x^0 =c \in \R^n$ is the starting vector of the values of the nodes and matrix $\bB$ is the positive definite diagonal matrix with weights in its diagonal (recall that in the case of standard average consensus this can be simply $\bB=\bI$). Depending on the experiment, we choose the values of the starting vector $c \in \R^n$ to follow either a Gaussian distribution or uniform distribution or to be integer values such that $c_i=i \in \R$.  In the plots, the horizontal axis represents the number of iterations except in the figures of subsection~\ref{cakslas}, where the horizontal axis represents the block size.

In our implementations we use three popular graph topologies from the area of wireless sensor networks. These are the cycle (ring graph), the 2-dimension grid and the random geometric graph (RGG) with radius $r=\sqrt{\log(n)/n}$. In all experiments we formulate the average consensus problem (or its weighted variant) using the incidence matrix. That is, $\bA=\bQ$ is used as the AC system. Code was written in Julia 0.6.3. 

\subsection{Convergence on weighted average consensus}
\label{Weightexperiments}
As we explained in Section~\ref{skecthsection}, the sketch and project method (Algorithm~\ref{FullSkecth}) can solve the more general weighted AC problem. In this first experiment we numerically verify the linear convergence of the Scaled RK algorithm \eqref{scaledRK} for solving this problem in the case of $\bB=\bD$.  That is, the matrix $\bB$ of the weights is the degree matrix $\bD$ of the graph ($\bB_{ii}=d_i$,  $\forall i \in [n]$). In this setting the exact update rule of the method is given in equation \eqref{ScRKwithQ}, where in order to have convergence to the weighted average the chosen nodes are required to share not only their private values but also their weight $\bB_{ii}$ (in our experiment this is equal to the degree of the node $d_i$).  In this experiment the starting vector of values $x^0=c \in \R^n$ is a Gaussian vector. The linear convergence of the algorithm is clear in Figure~\ref{scaledRKFigure}.

\begin{figure}[t]
\centering
\begin{subfigure}{.3\textwidth}
  \centering
  \includegraphics[width=1\linewidth]{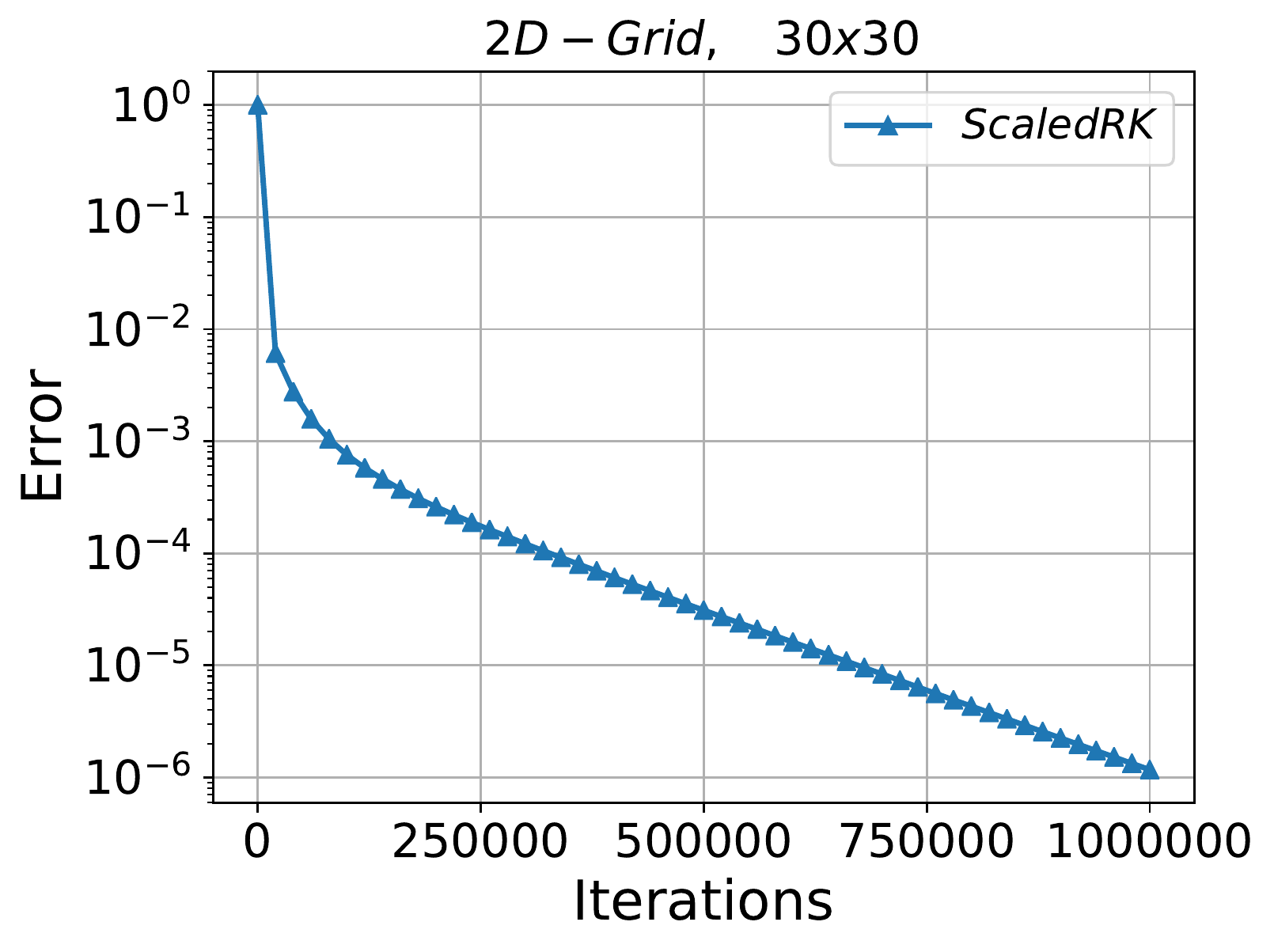}
\end{subfigure}%
\begin{subfigure}{.3\textwidth}
  \centering
  \includegraphics[width=1\linewidth]{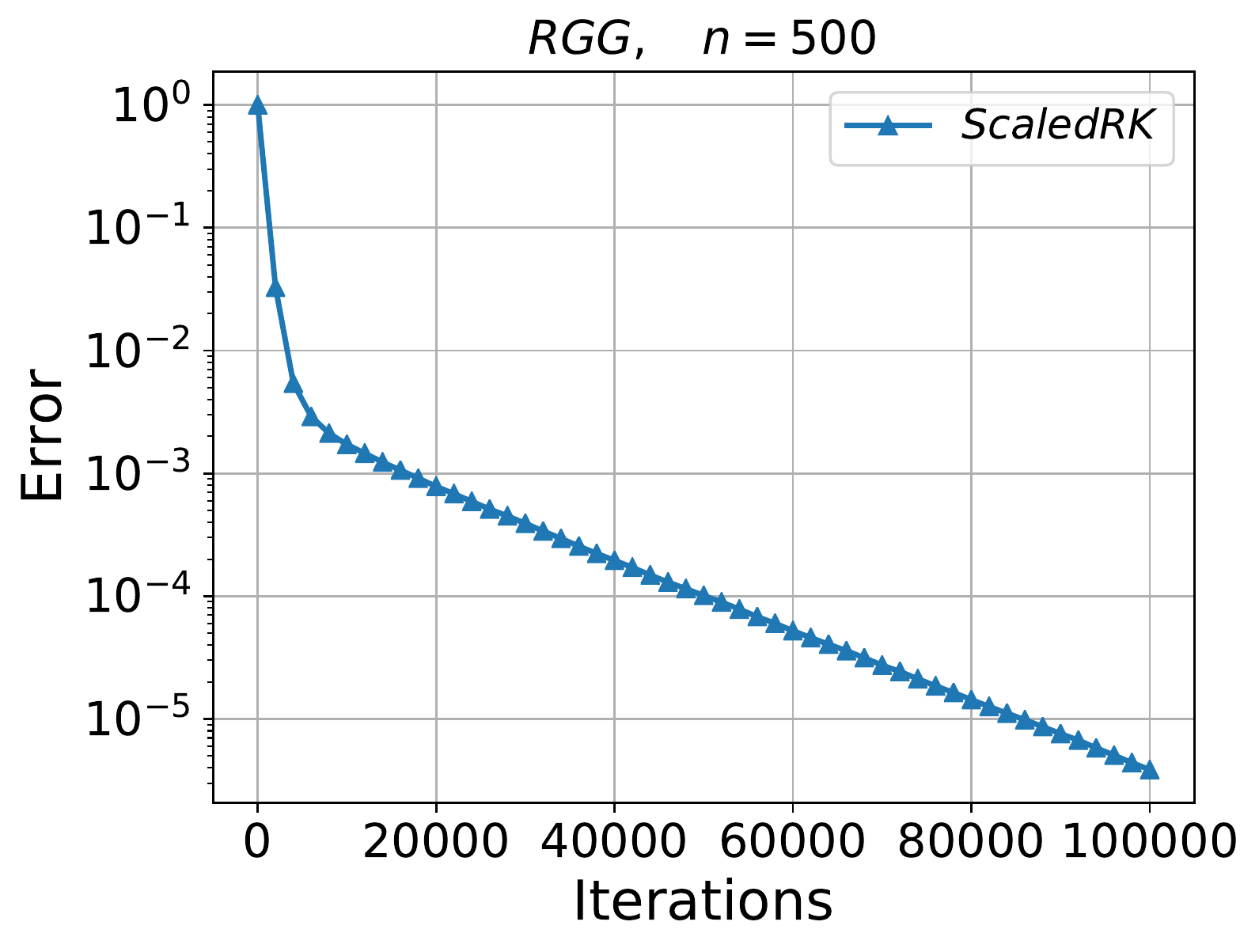}
\end{subfigure}
\begin{subfigure}{.3\textwidth}
  \centering
  \includegraphics[width=1\linewidth]{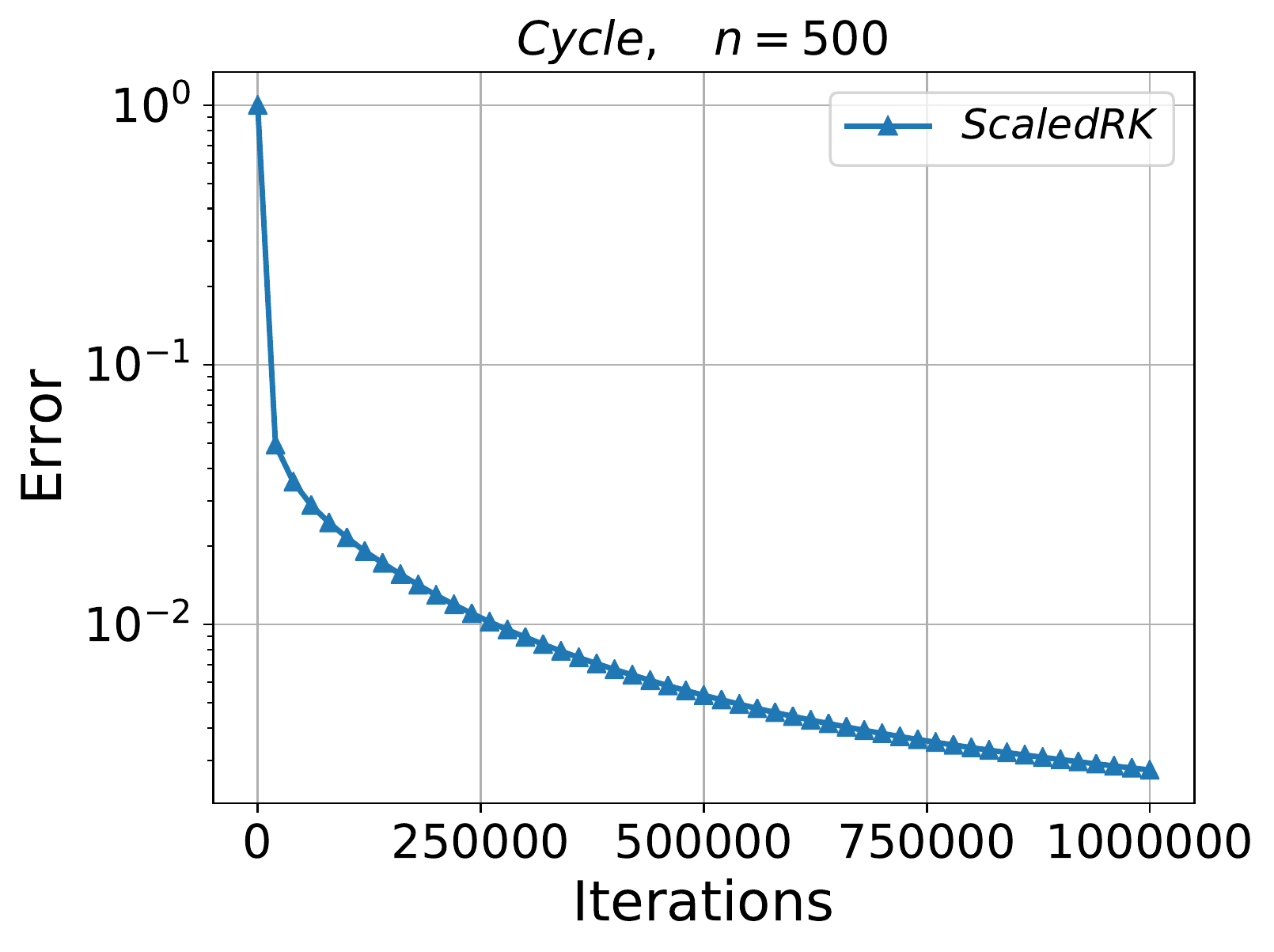}
\end{subfigure}\\
\caption{\footnotesize Performance of ScaledRK  in a 2-dimension grid, random geometric graph (RGG) and a cycle graph for solving the weighted average consensus problem. The weight matrix is chosen to be $\bB=\bD$, the degree matrix of the network. The $n$ in the title of each plot indicates the number of nodes of the network. For the grid graph this is $n \times n$.}
\label{scaledRKFigure}
\end{figure}

\subsection{Benefit of block variants}
\label{cakslas}
We devote this experiment to evaluate the performance of the randomized block gossip algorithms presented in Sections \ref{BlockGossip} and \ref{DualBlock}.  In particular, we would like to highlight the benefit of using larger block size in the update rule of randomized Kaczmarz method and as a result through our established connection of the randomized pairwise gossip algorithm \cite{boyd2006randomized} (see equation \eqref{pairwiseUpdate}).

Recall that in Section~\ref{DualBlock} we show that both RBK and RNM converge to the solution of the primal and dual problems respectively with the same rate and that their iterates are related via a simple affine transform \eqref{eq:dual-corresp}. In addition note that an interesting feature of the RNM \cite{qu2015sdna}, is that when the method viewed as algorithm indexed by the size $ \tau=|C|$, it enjoys superlinear speedup in $\tau$. That is, as $\tau$ (block size) increases by some factor, the iteration complexity drops by a factor that is at least as large (see Section~\ref{RNMdual}). Since RBK and RNM share the same rates this property naturally holds for RBK as well.

We show that for a  connected network $\cG$, the complexity improves superlinearly in $\tau = |C|$, where $C$ is chosen as a subset of $\cE$ of size $\tau$, uniformly at random (recall the in the update rule of RBK the random matrix is $\bS=\bI_{:C}$). Similar to the rest of this section in comparing the number of iterations for different values of $\tau$, we use the relative error  $\varepsilon=\|x^k-x^*\|^2 / \|x^0-x^*\|^2$. We let $x^0_i=c_i=i$ for each node $i \in \cV$ (vector of integers). We run RBK until the relative error becomes smaller than $0.01$.  The blue solid line in the figures denotes the actual number of iterations (after running the code) needed in order to achieve $\varepsilon\leq 10 ^{-2}$ for different values of $\tau$. The green dotted line represents the function $f(\tau)\eqdef \frac{\ell}{\tau}$, where $\ell$ is the number of iterations of RBK with $\tau=1$ (i.e., the pairwise gossip algorithm).  The green line depicts linear speedup;  the fact that the blue line (obtained through  experiments) is below the green line points to superlinear speedup.  In this experiment we use the Cycle graph with $n=30$ and $n=100$ nodes (Figure~\ref{fig:test3}) and  the $4 \times 4$ two dimension grid graph (Figure~\ref{fig:test4}). Note that, when $|C|=m$ the convergence rate of the  method becomes $\rho=0$ and as a result it converges in one step.

\begin{figure}[!h]
\label{RingGraph}
\centering
\begin{subfigure}{.3\textwidth}
  \centering
  \includegraphics[width=1\linewidth]{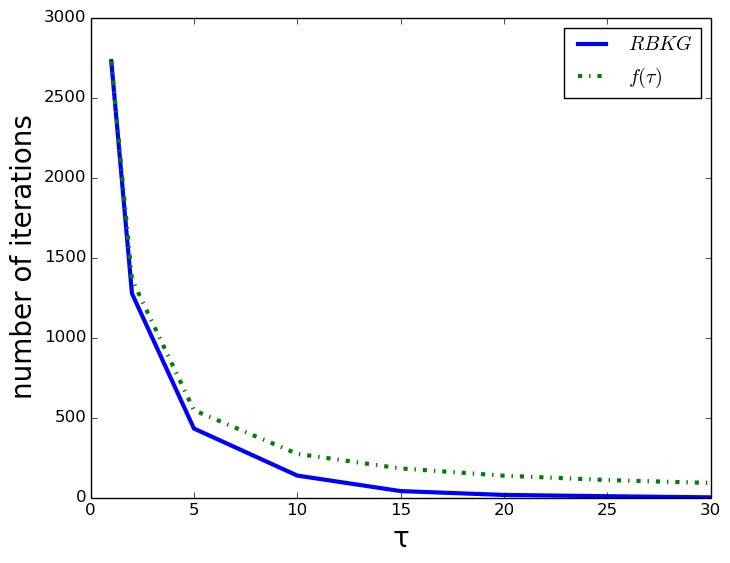}
  \caption{\footnotesize Cycle, $n=30$}
  \label{fig:sub1}
\end{subfigure}%
\begin{subfigure}{.3\textwidth}
  \centering
  \includegraphics[width=1\linewidth]{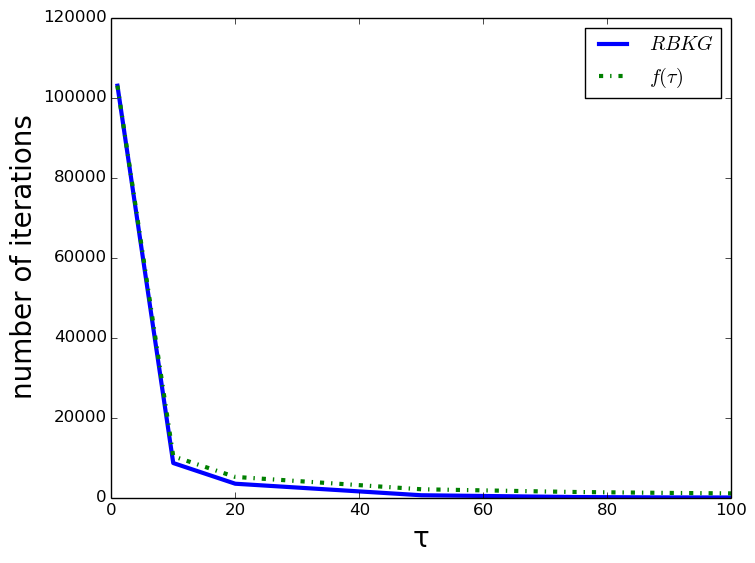}
  \caption{\footnotesize Cycle, $n=100$}
  \label{fig:sub2}
\end{subfigure}
\caption{\footnotesize Superlinear speedup of RBK on cycle graphs.}
\label{fig:test3}
\end{figure}
\begin{figure}[!h]
\label{gridGraph}
\centering
\begin{subfigure}{.3\textwidth}
  \centering
  \includegraphics[width=1\linewidth]{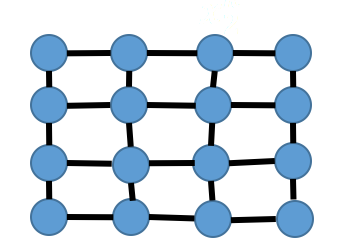}
  \caption{\footnotesize 2D-Grid, $4 \times 4$}
  \label{fig:sub3}
\end{subfigure}%
\begin{subfigure}{.3\textwidth}
  \centering
  \includegraphics[width=1\linewidth]{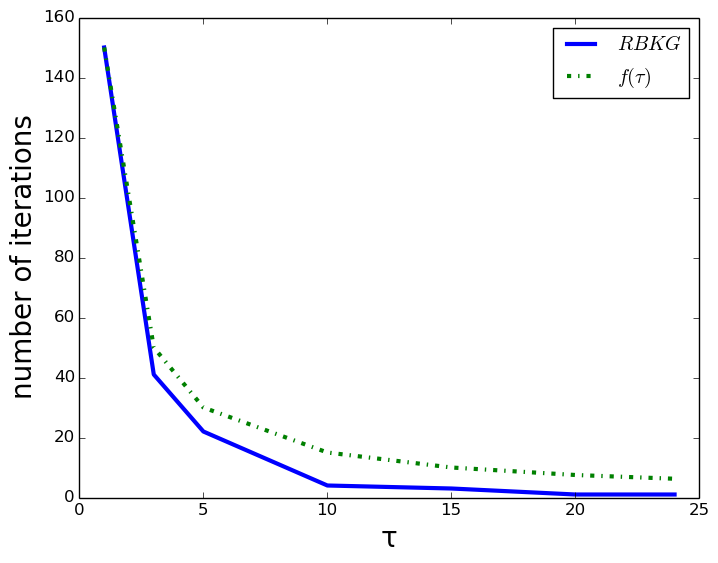}
  \caption{\footnotesize Speedup in $\tau$}
  \label{fig:sub4}
\end{subfigure}
\caption{\footnotesize Superlinear speedup of RBK on a $4 \times 4$ two dimension grid graph.}
\label{fig:test4}
\end{figure}

\subsection{Accelerated gossip algorithms}
We devote this subsection to experimentally evaluate the performance of the proposed accelerated gossip algorithms: mRK (Algorithm~\ref{RKmomentum}),  mRBK (Algorithm~\ref{RBKmomentum}) and AccGossip with the two options of the parameters (Algorithm~\ref{alg:acceleratedNew}). In particular we perform four experiments. In the first two we focus on the performance of the mRK and how the choice of stepsize (relaxation parameter) $\omega$ and heavy ball momentum parameter $\beta$ affect the performance of the method. In the next experiment we show that the addition of heavy ball momentum can be also beneficial for the performance of the block variant mRBK. In the last experiment we compare the standard pairwise gossip algorithm (baseline method) from \cite{boyd2006randomized}, the mRK and the AccGossip and show that the probably accelerated gossip algorithm, AccGossip outperforms the other algorithms and converge as predicted from the theory with an accelerated linear rate.

\subsubsection{Impact of momentum parameter on mRK}
As we have already presented in the standard pairwise gossip algorithm (equation \eqref{pairwiseUpdate}) the two selected nodes that exchange information update their values to their exact average while all the other nodes remain idle. In our framework this update can be cast as special case of mRK when $\beta=0$ and $\omega=1$.

In this experiment we keep the stepsize fixed and equal to $\omega=1$ which means that the pair of the chosen nodes update their values to their exact average and we show that by choosing a suitable momentum parameter $\beta \in (0,1)$ we can obtain faster convergence to the consensus for all networks under study. The momentum parameter $\beta$ is chosen following the suggestions made in Chapter~\ref{ChapterMomentum} for solving general consistent linear systems. See Figure~\ref{mRKomega1} for more details. It is worth to point out that for all networks under study the addition of a heavy ball momentum term is beneficial in the performance of the method.

\begin{figure}[t]
\centering
\begin{subfigure}{.3\textwidth}
  \centering
  \includegraphics[width=1\linewidth]{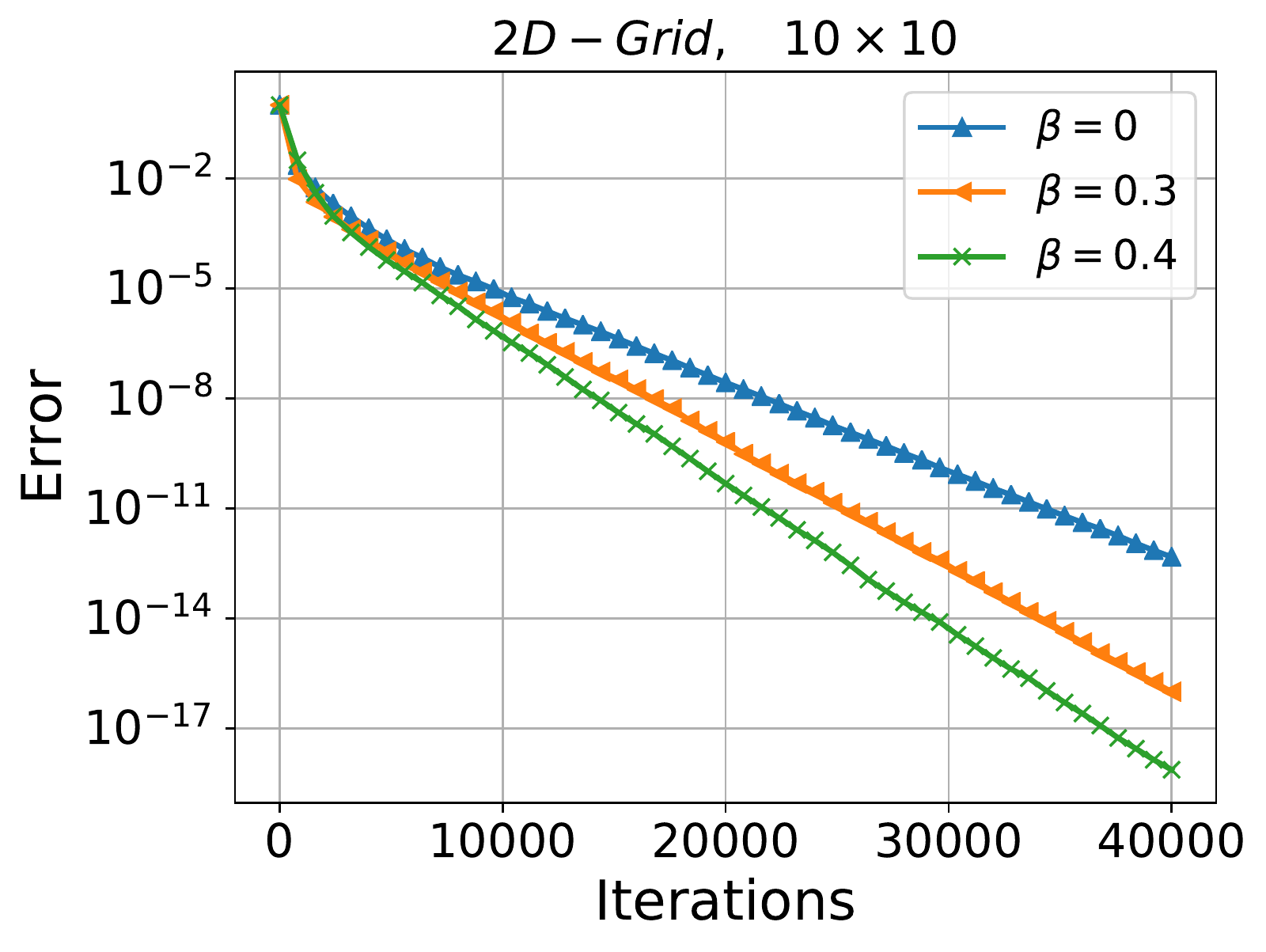}
\end{subfigure}%
\begin{subfigure}{.3\textwidth}
  \centering
  \includegraphics[width=1\linewidth]{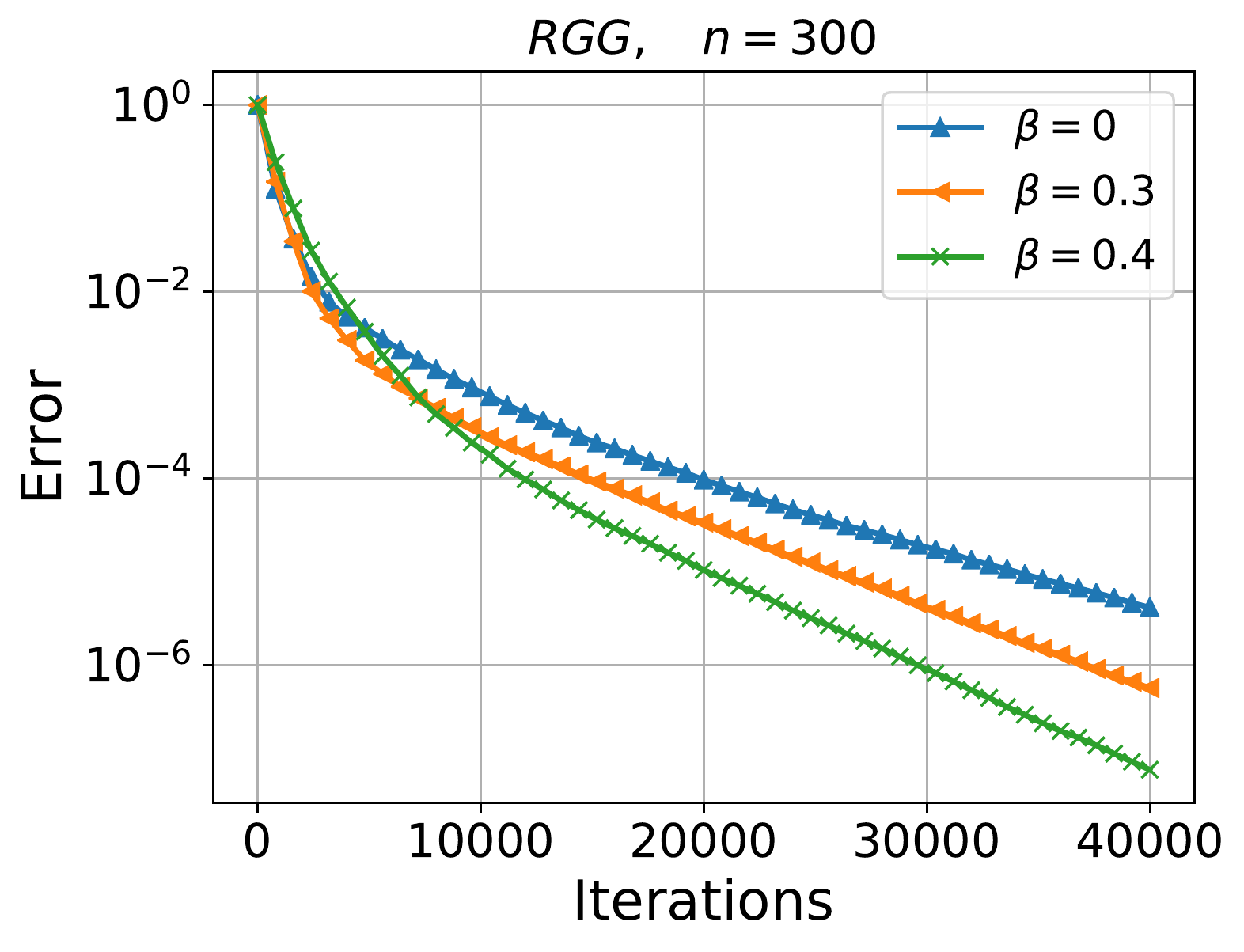}
\end{subfigure}%
\begin{subfigure}{.3\textwidth}
  \centering
  \includegraphics[width=1\linewidth]{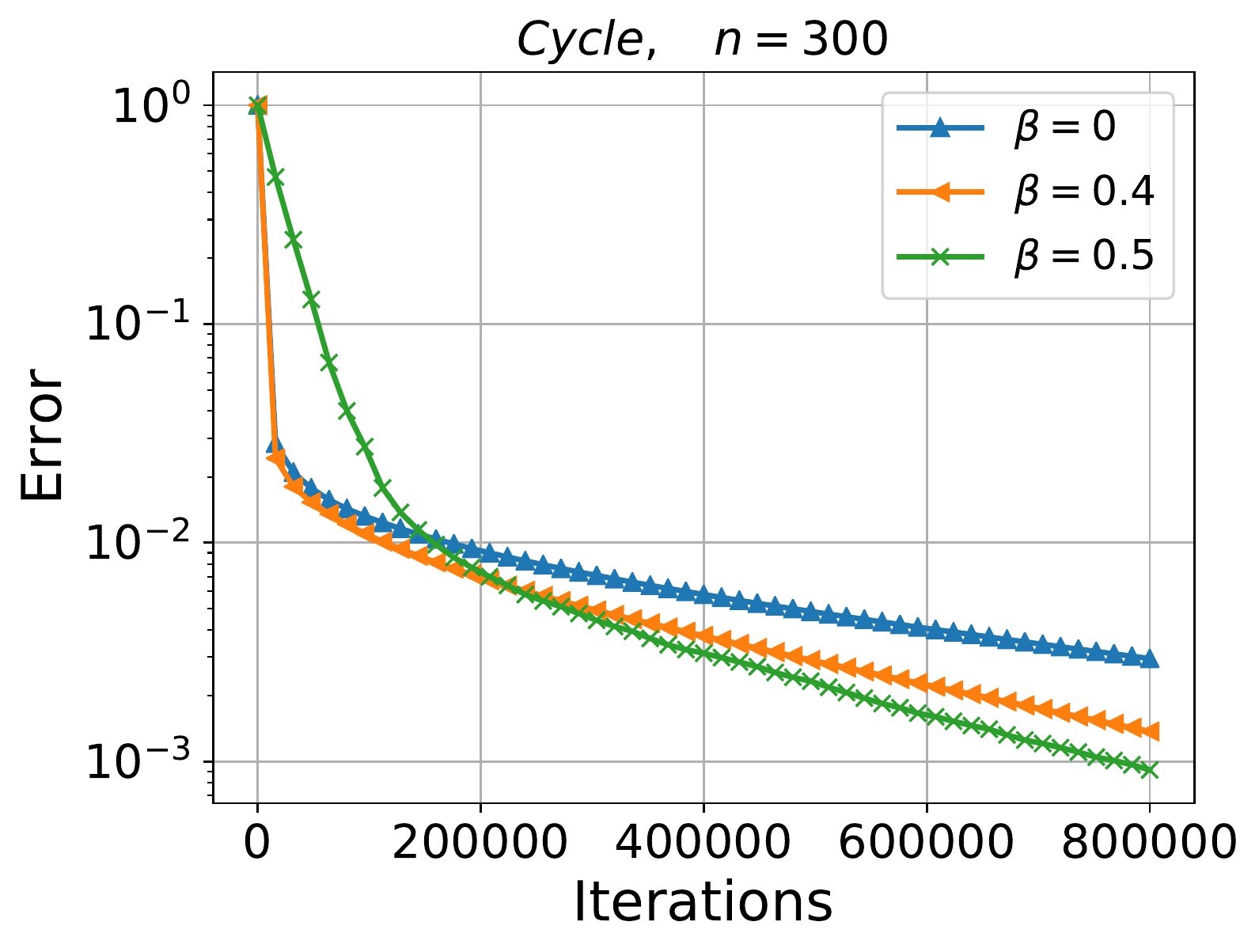}
\end{subfigure}\\
\begin{subfigure}{.3\textwidth}
  \centering
  \includegraphics[width=1\linewidth]{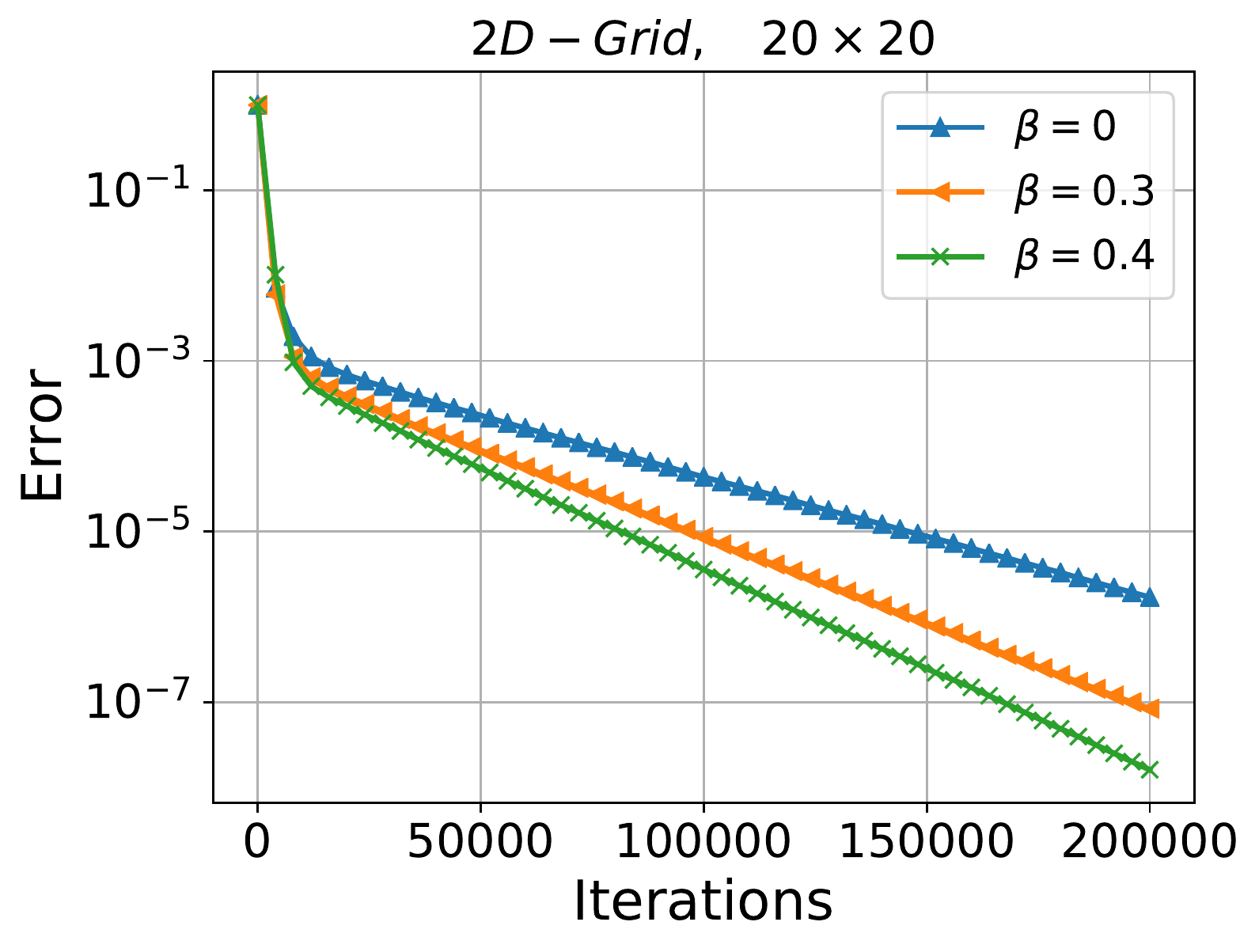}
\end{subfigure}
\begin{subfigure}{.3\textwidth}
  \centering
  \includegraphics[width=1\linewidth]{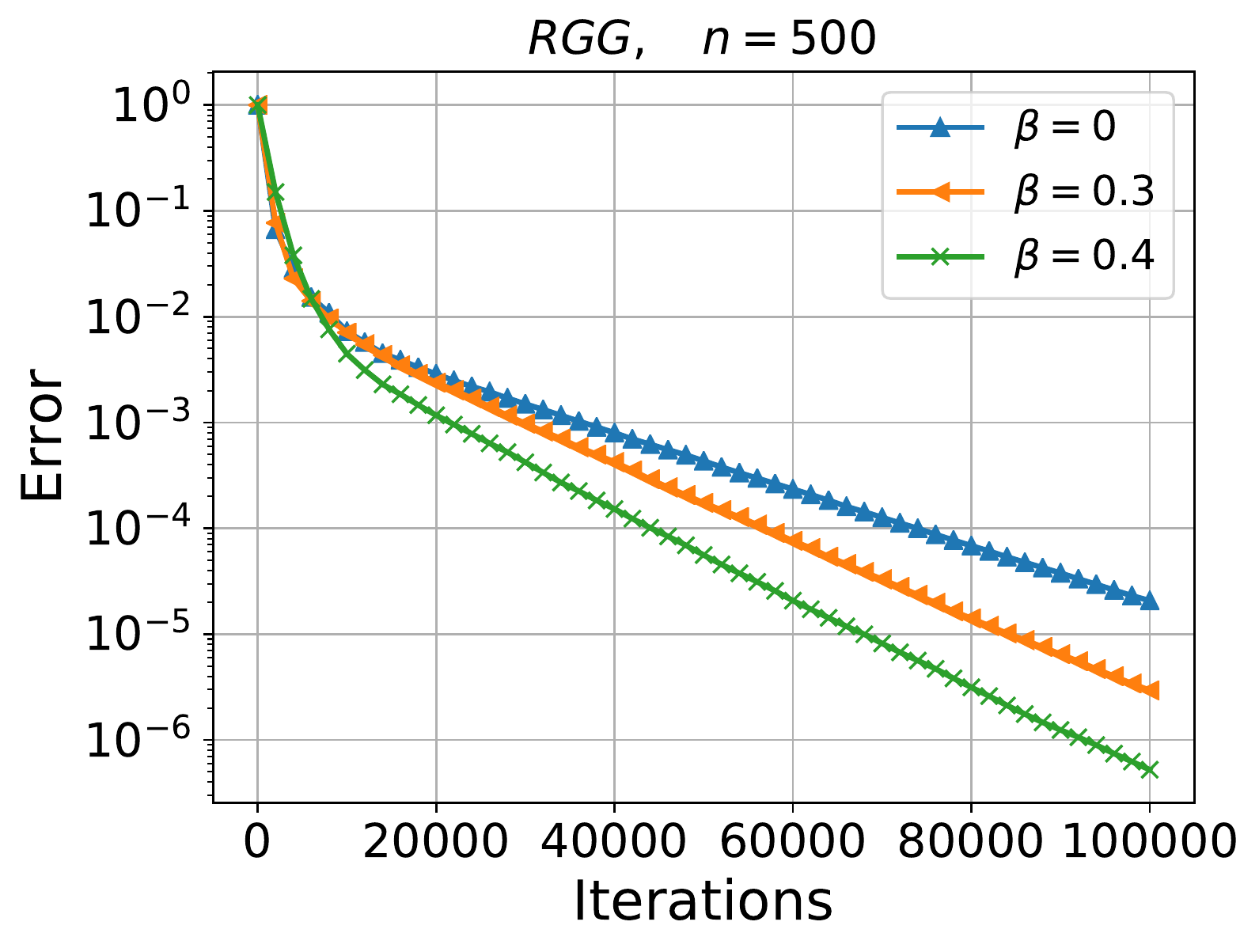}
\end{subfigure}
\begin{subfigure}{.3\textwidth}
  \centering
  \includegraphics[width=1\linewidth]{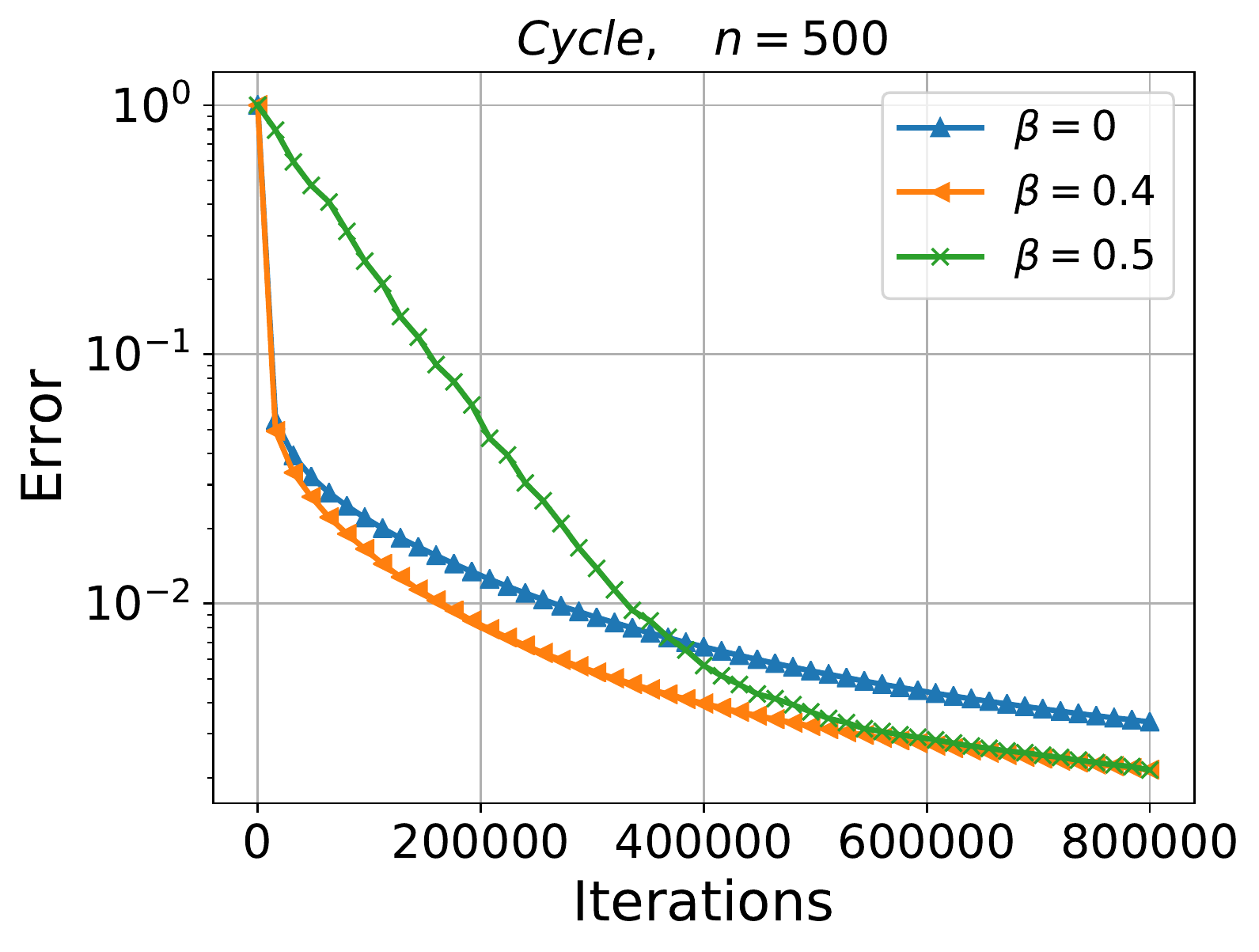}
\end{subfigure}\\
\caption{\footnotesize Performance of mRK for fixed step-size $\omega=1$ and several momentum parameters $\beta$ in a 2-dimension grid, random geometric graph (RGG) and a cycle graph. The choice $\beta=0$ corresponds to the randomized pairwise gossip algorithm proposed in \cite{boyd2006randomized}. The starting vector $x^0=c \in \R^n$ is a Gaussian vector. The $n$ in the title of each plot indicates the number of nodes of the network. For the grid graph this is $n \times n$. }
\label{mRKomega1}
\end{figure}

\subsubsection{Comparison of mRK and shift-Register algorithm \cite{liu2013analysis}}
In this experiment we compare mRK with the shift register gossip algorithm (pairwise momentum method, abbreviation: Pmom) analyzed in \cite{liu2013analysis}. We choose the parameters $\omega$ and $\beta$ of mRK in such a way in order to satisfy the connection established in Section~\ref{connectionOfAcceleratedMethods}. That is, we choose $\beta=\omega-1$ for any choice of $\omega \in (1,2)$. Observe that in all plots of Figure~\ref{shiftregister} mRK outperforms the corresponding shift-register algorithm. 

\begin{figure}[t]
\centering
\begin{subfigure}{.3\textwidth}
  \centering
  \includegraphics[width=1\linewidth]{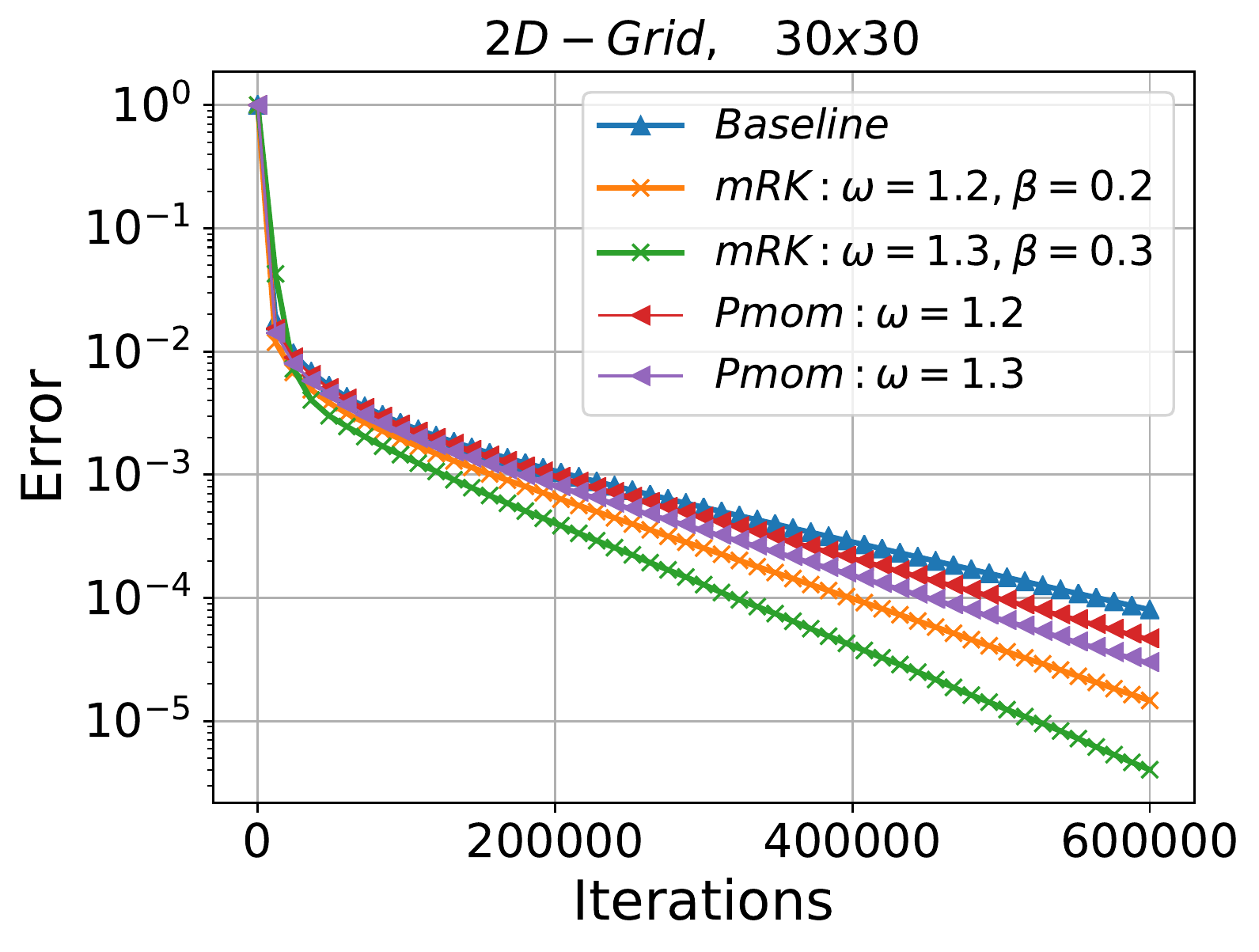}
\end{subfigure}%
\begin{subfigure}{.3\textwidth}
  \centering
  \includegraphics[width=1\linewidth]{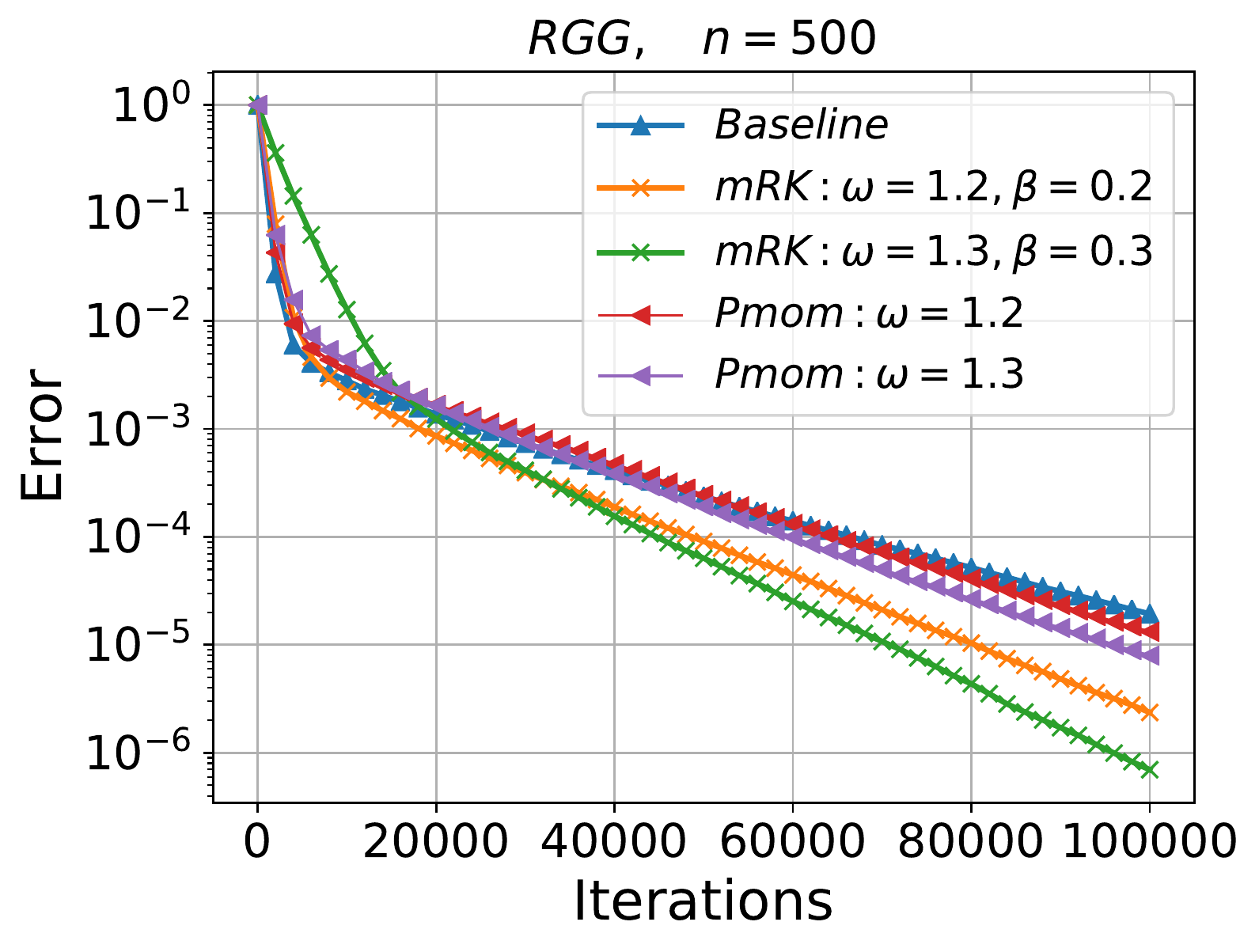}
\end{subfigure}%
\begin{subfigure}{.3\textwidth}
  \centering
  \includegraphics[width=1\linewidth]{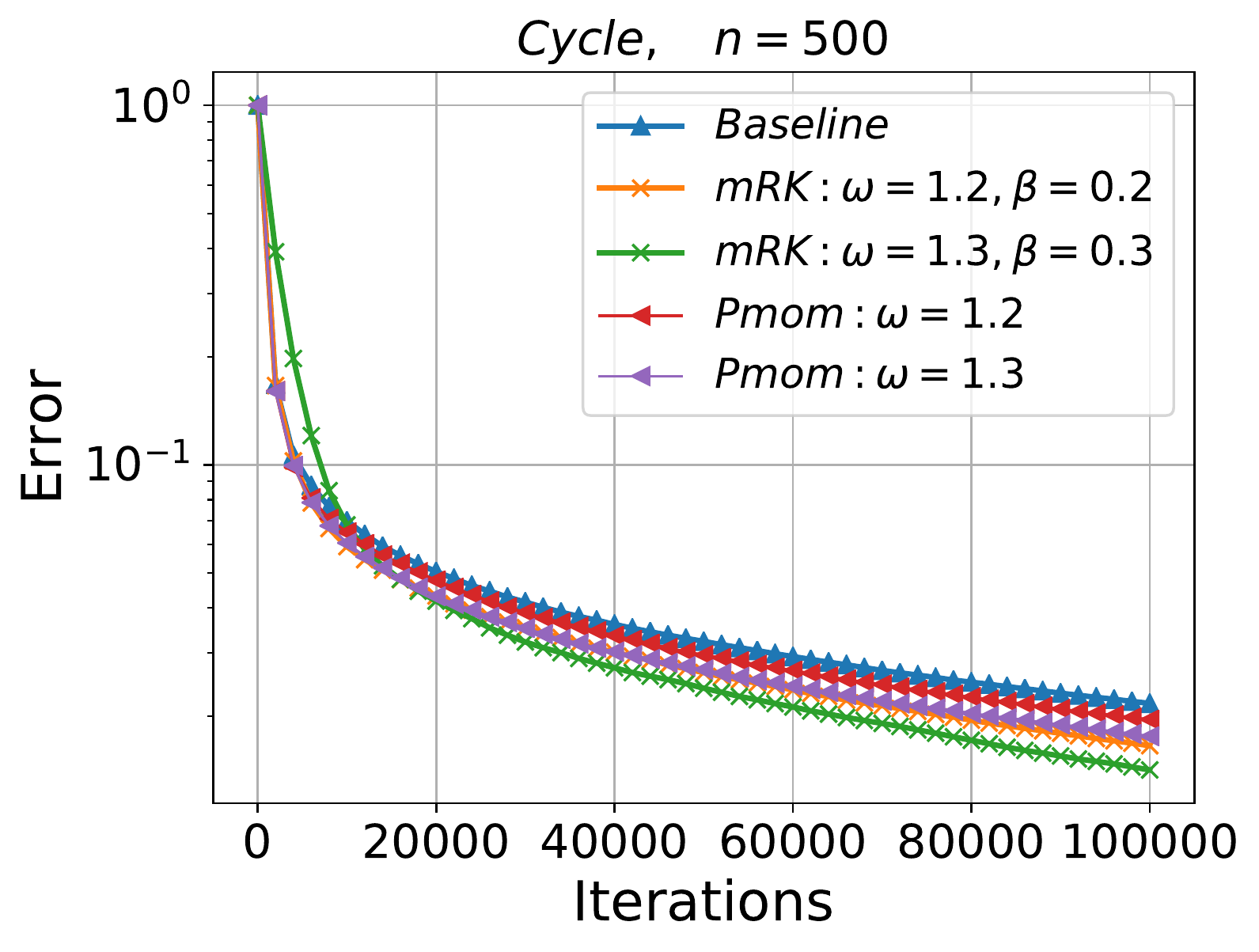}
\end{subfigure}\\
\begin{subfigure}{.3\textwidth}
  \centering
  \includegraphics[width=1\linewidth]{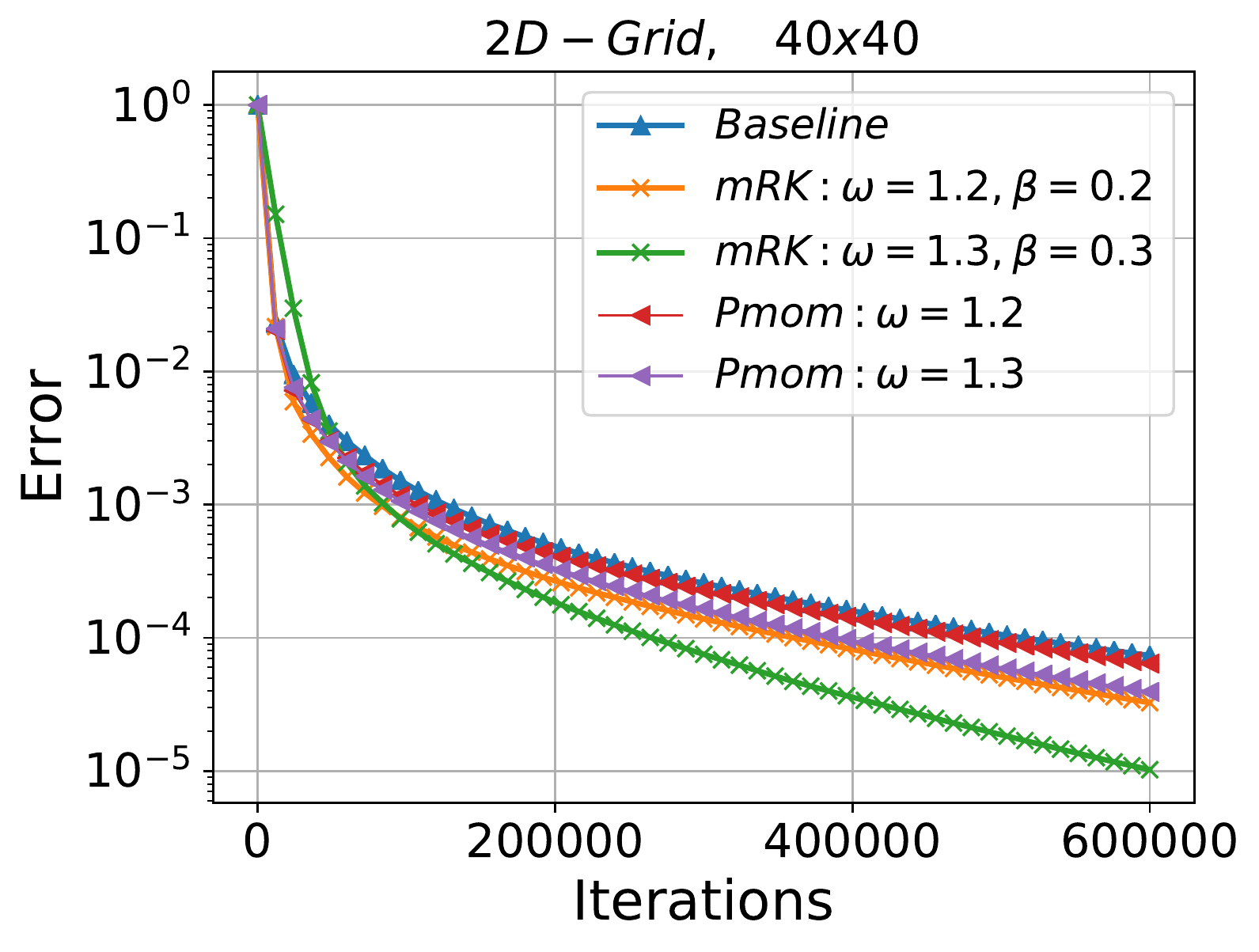}
\end{subfigure}
\begin{subfigure}{.3\textwidth}
  \centering
  \includegraphics[width=1\linewidth]{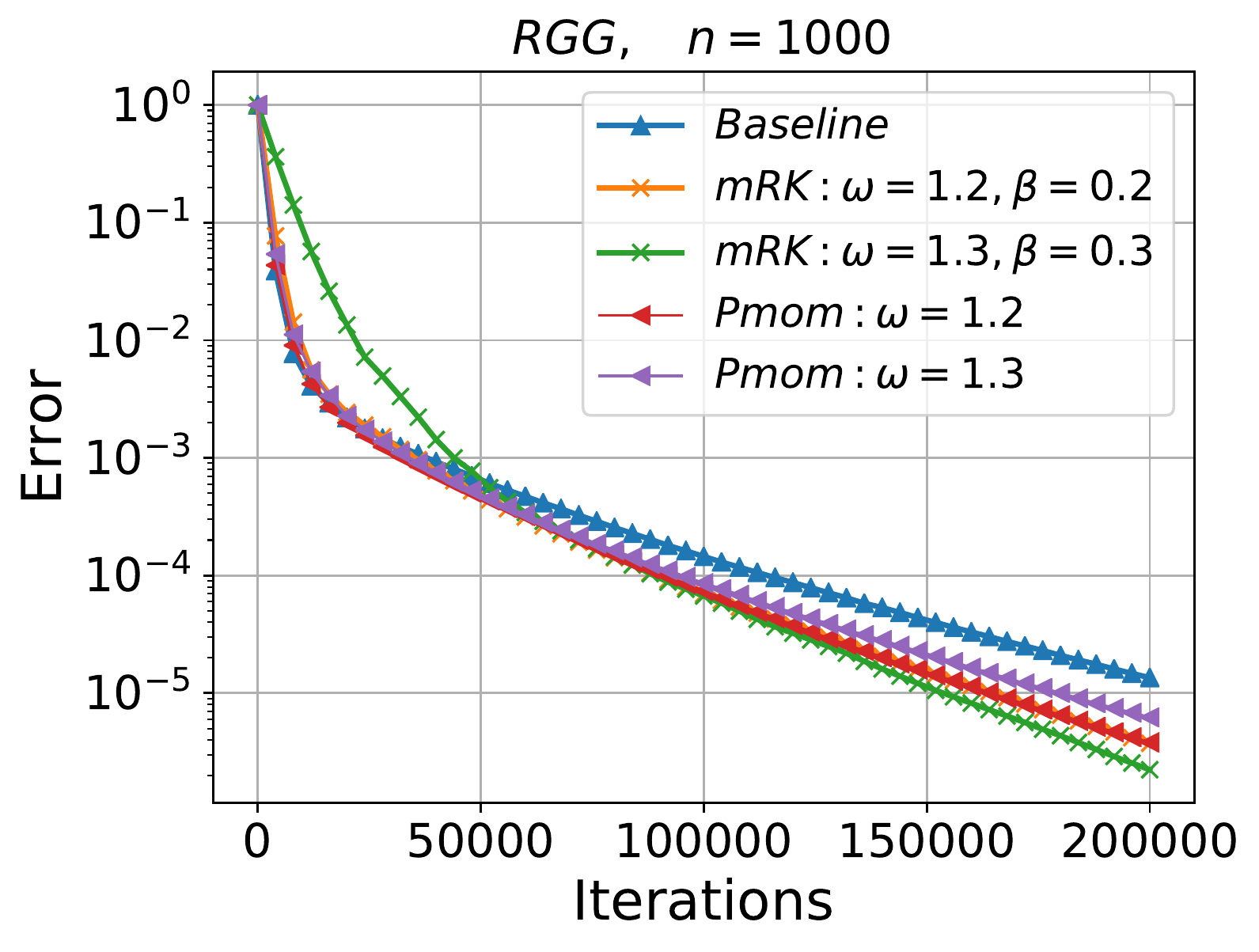}
\end{subfigure}
\begin{subfigure}{.3\textwidth}
  \centering
  \includegraphics[width=1\linewidth]{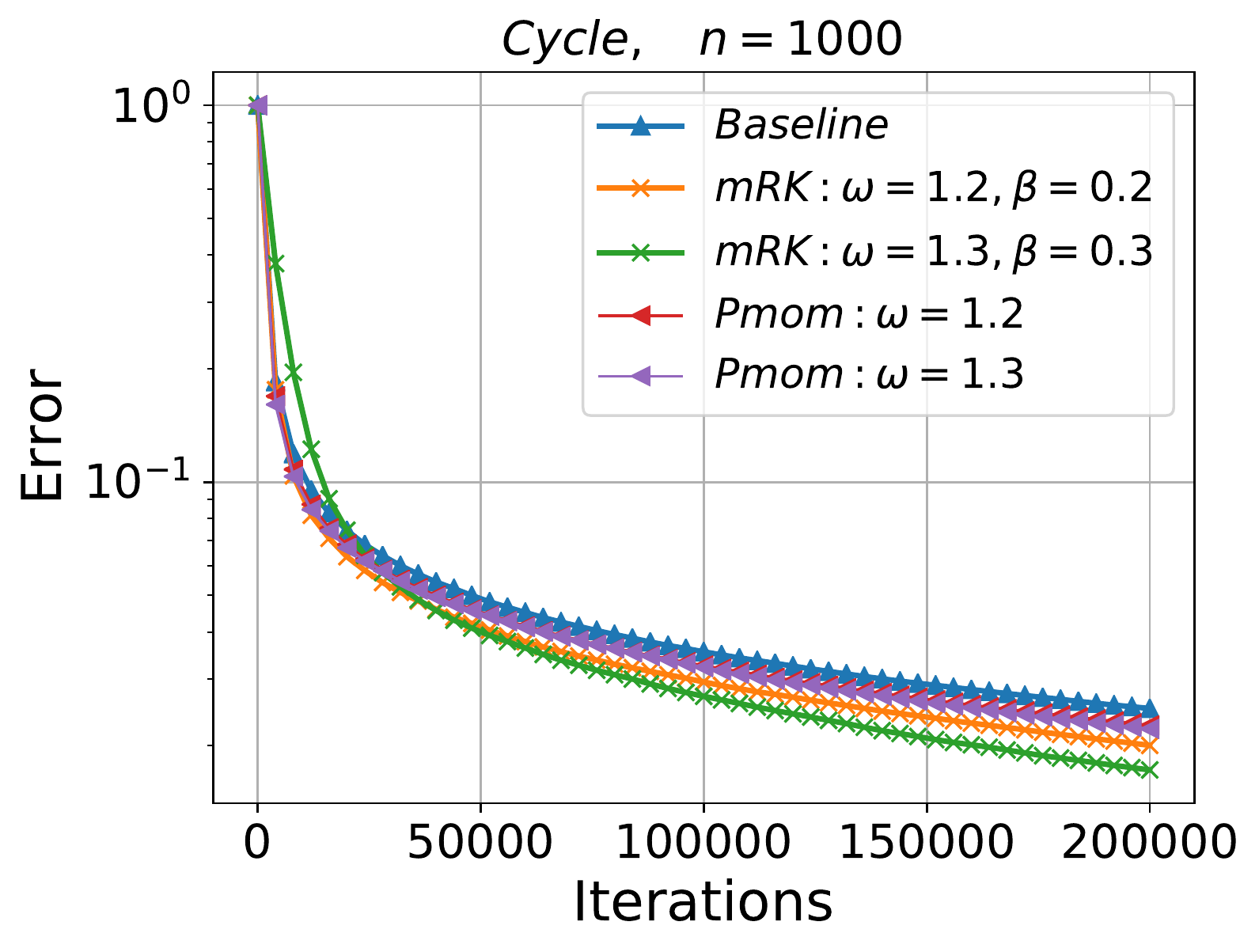}
\end{subfigure}\\
\caption{\footnotesize Comparison of mRK and the pairwise momentum method (Pmom), proposed in \cite{liu2013analysis} (shift-register algorithm of Section~\ref{connectionOfAcceleratedMethods}). Following the connection between mRK and Pmom established in Section~\ref{connectionOfAcceleratedMethods} the momentum parameter of mRK is chosen to be $\beta=\omega-1$ and the stepsizes are selected to be either $\omega= 1.2$ or $\omega=1.3$.  The baseline method is the standard randomized pairwise gossip algorithm from \cite{boyd2006randomized}. The starting vector $x^0=c \in \R^n$ is a Gaussian vector. The $n$ in the title of each plot indicates the number of nodes of the network. For the grid graph this is $n \times n$.}
\label{shiftregister}
\end{figure}

\subsubsection{Impact of momentum parameter on mRBK}
In this experiment our goal is to show that the addition of heavy ball momentum accelerates the RBK gossip algorithm presented in Section~\ref{BlockGossip}. Without loss of generality we choose the block size to be  equal to $\tau=5$. That is, the random matrix $\bS_k\sim \cD$ in the update rule of mRBK is a $m \times 5$ column submatrix of the indetity $m \times m$ matrix. Thus, in each iteration $5$ edges of the network are chosen to form the subgraph $\cG_k$ and the values of the nodes are updated according to Algorithm~\ref{RBKmomentum}. Note that similar plots can be obtained for any choice of block size. We run all algorithms with fixed stepsize $\omega=1$. From Figure~\ref{RBKfigures}, it is obvious that for all networks under study, choosing a suitable momentum parameter $\beta \in (0,1)$ gives faster convergence than having no momentum, $\beta =0$. 

\begin{figure}[t]
\centering
\begin{subfigure}{.3\textwidth}
  \centering
  \includegraphics[width=1\linewidth]{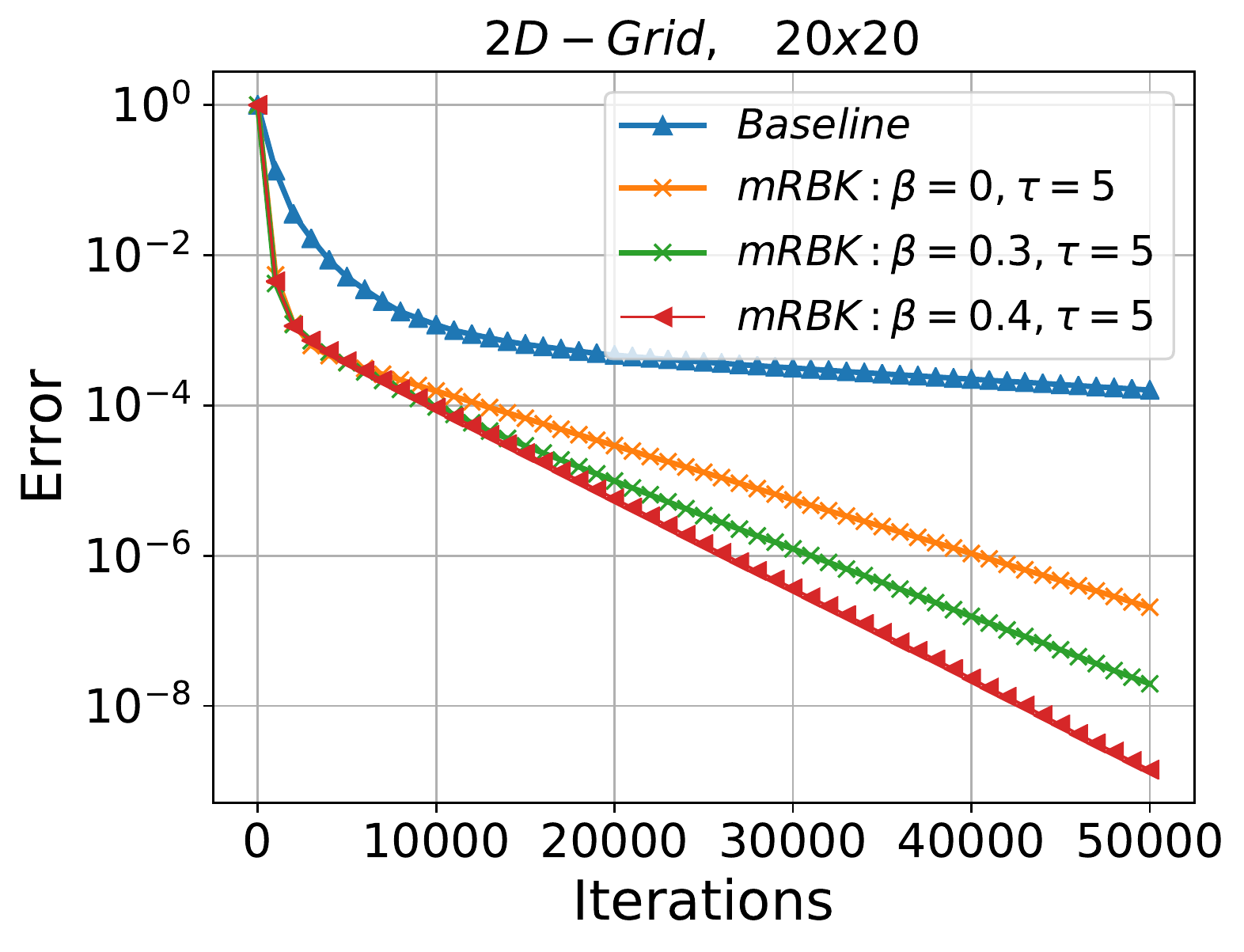}
\end{subfigure}%
\begin{subfigure}{.3\textwidth}
  \centering
  \includegraphics[width=1\linewidth]{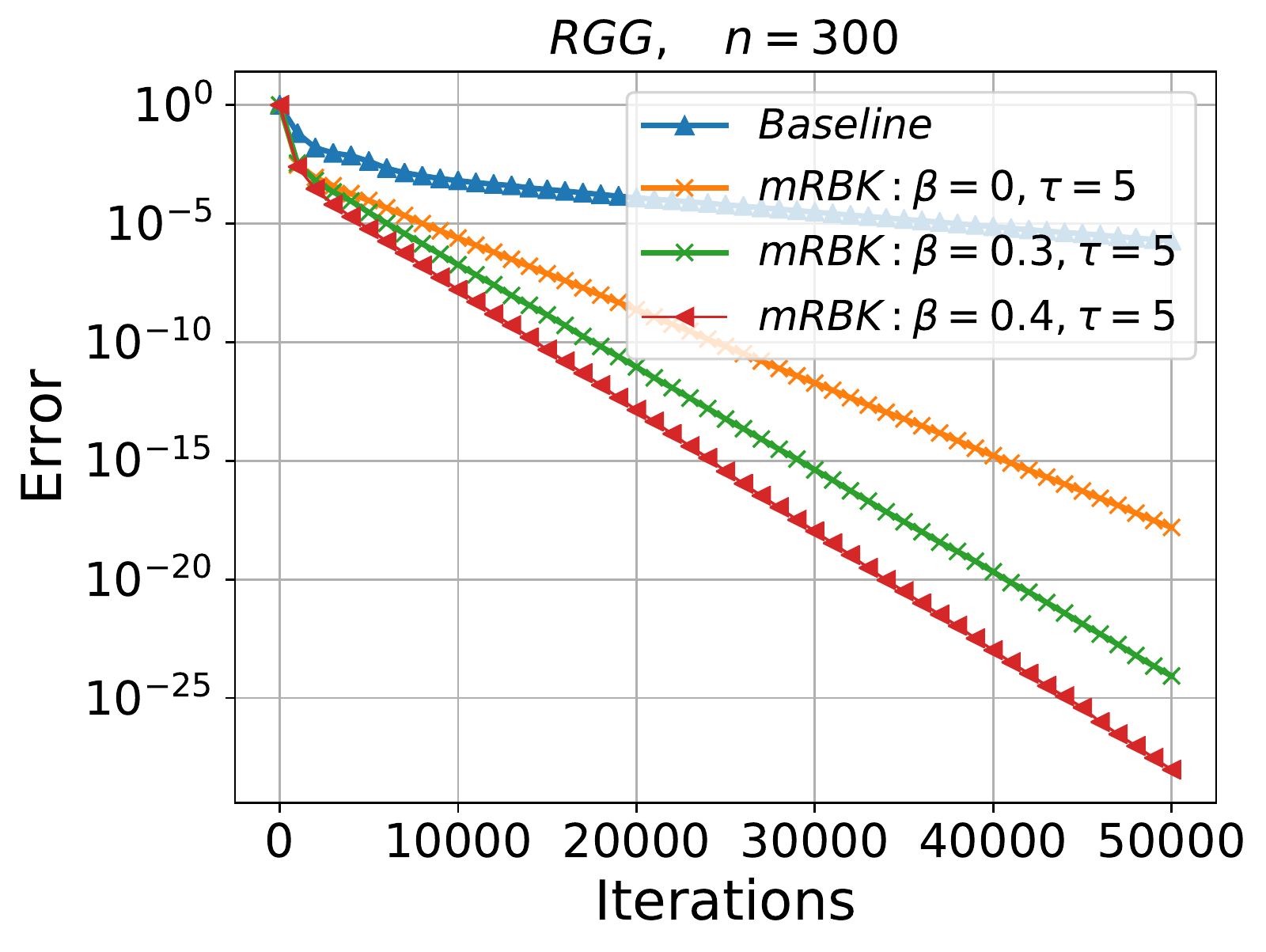}
\end{subfigure}%
\begin{subfigure}{.3\textwidth}
  \centering
  \includegraphics[width=1\linewidth]{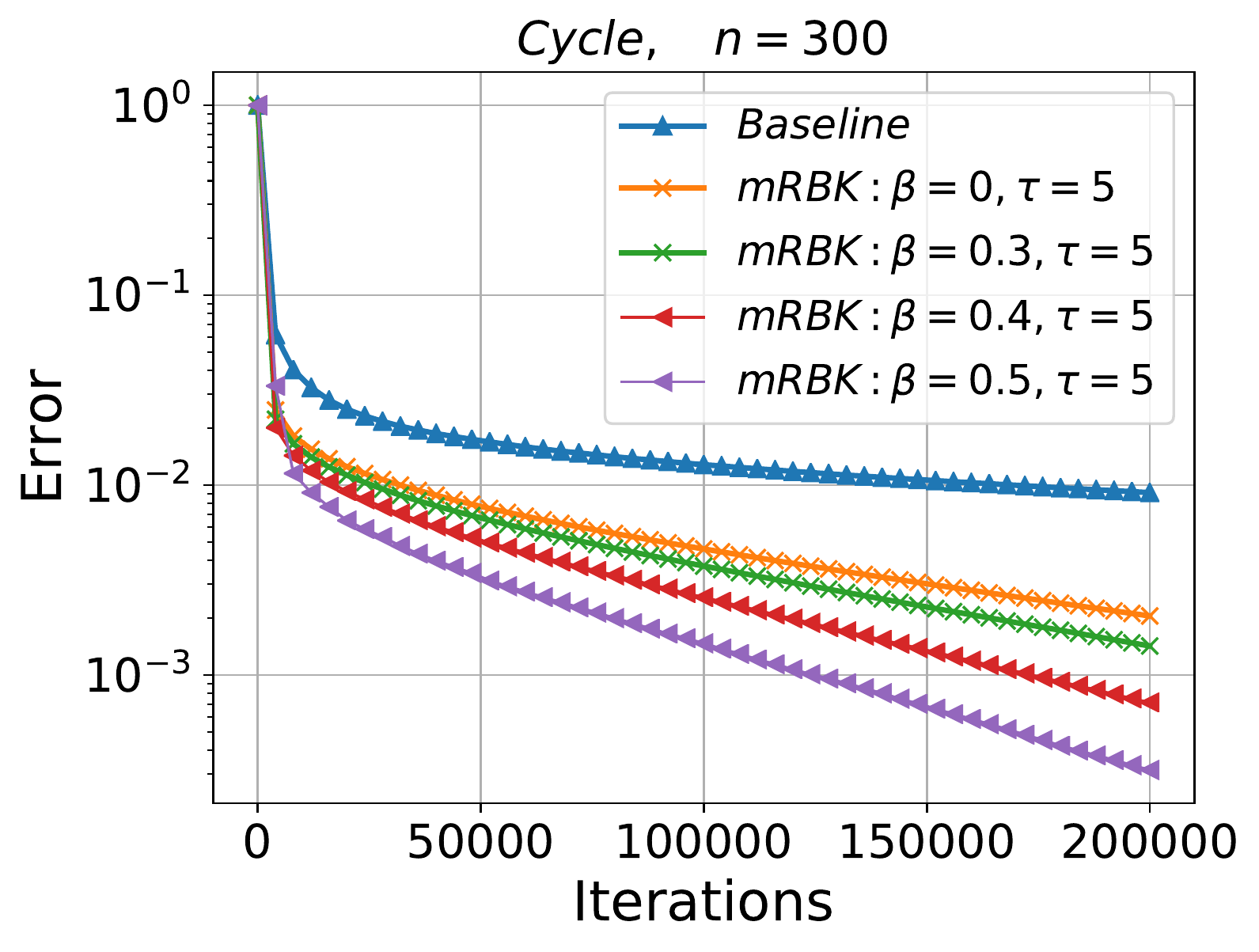}
\end{subfigure}\\
\begin{subfigure}{.3\textwidth}
  \centering
  \includegraphics[width=1\linewidth]{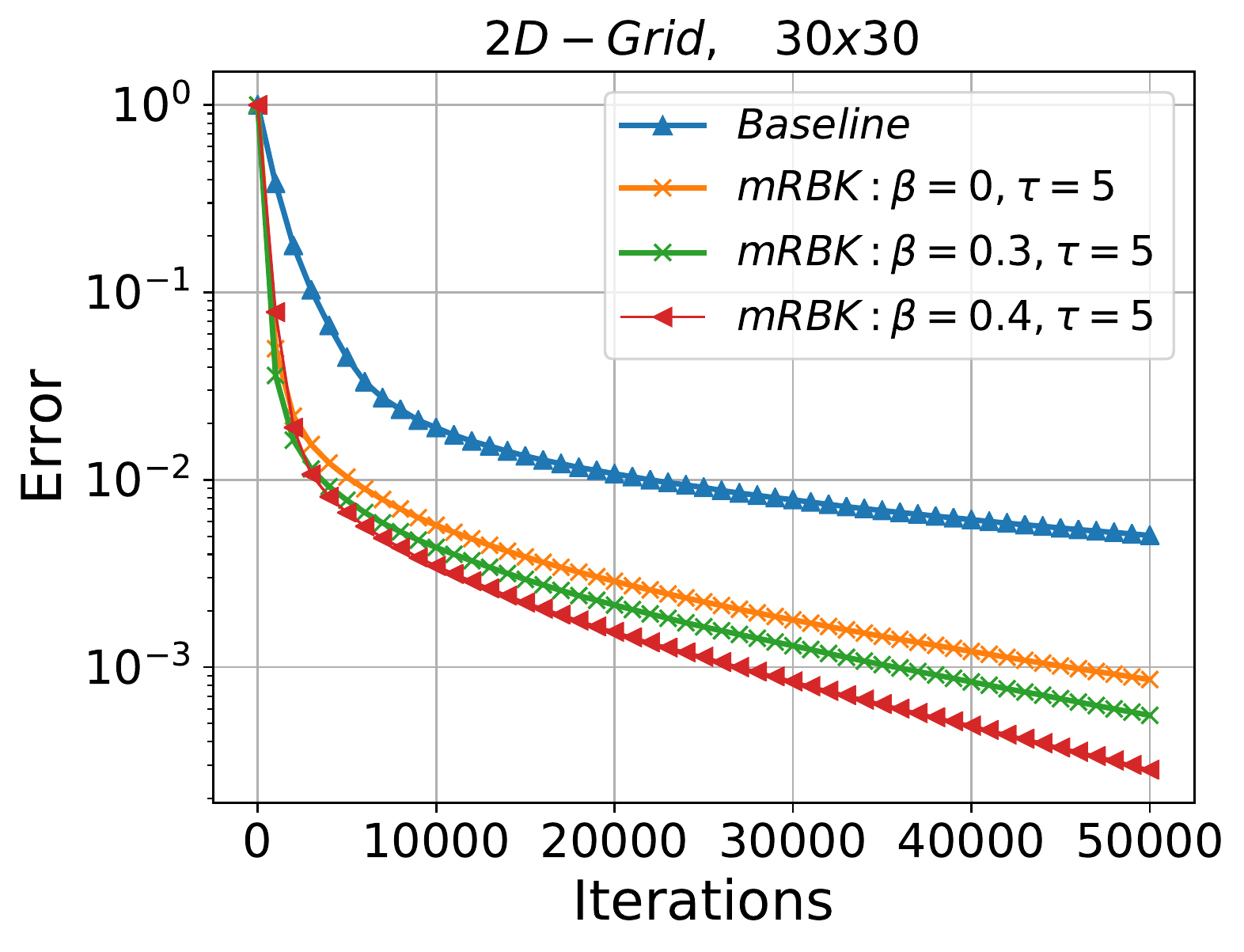}
\end{subfigure}
\begin{subfigure}{.3\textwidth}
  \centering
  \includegraphics[width=1\linewidth]{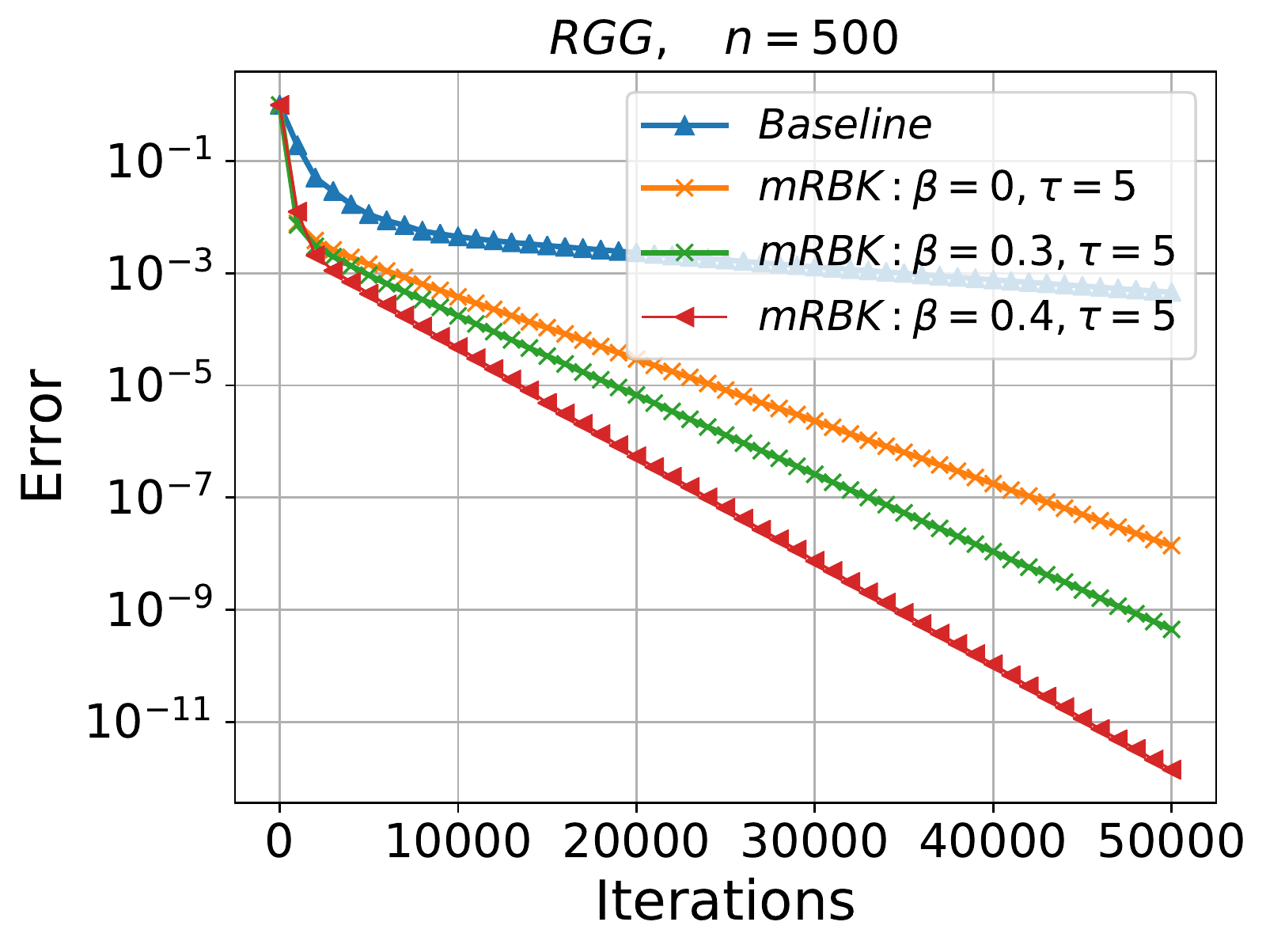}
\end{subfigure}
\begin{subfigure}{.3\textwidth}
  \centering
  \includegraphics[width=1\linewidth]{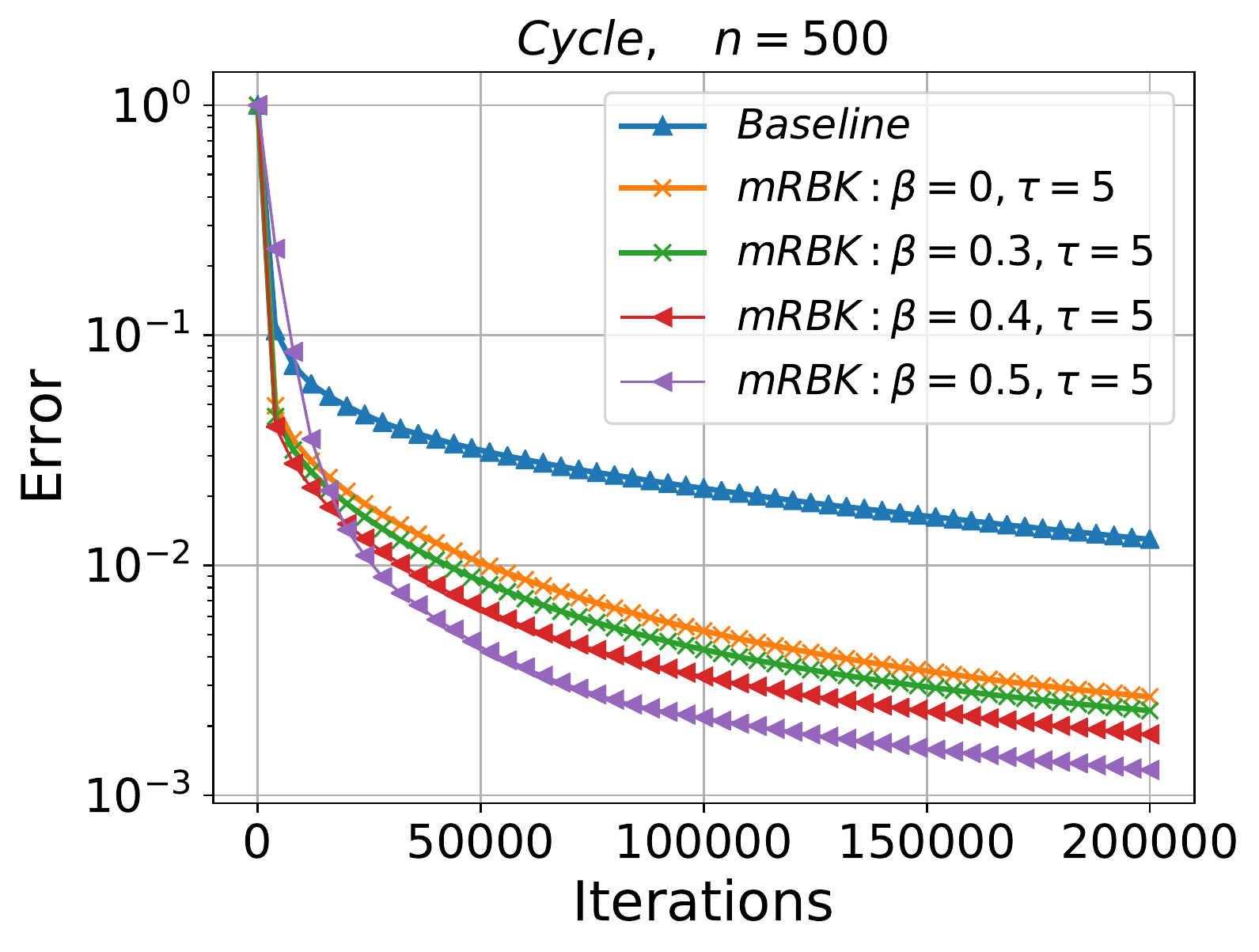}
\end{subfigure}\\
\caption{\footnotesize Comparison of mRBK with its no momentum variant RBK ($\beta=0$).  The stepsize for all methods is $\omega=1$ and the block size is $\tau=5$. The baseline method in the plots denotes the standard randomized pairwise gossip algorithm (block $\tau=1$) and is plotted to highlight the benefits of having larger block sizes (at least in terms of iterations). The starting vector $x^0=c \in \R^n$ is a Gaussian vector. The $n$ in the title of each plot indicates the number of nodes. For the grid graph this is $n \times n$.}
\label{RBKfigures}
\end{figure}

\subsubsection{Performance of AccGossip}
In the last experiment on faster gossip algorithms we evaluate the performance of the proposed provably accelerated gossip protocols of Section~\ref{accSubsection}. In particular we compare the standard RK (pairwise gossip algorithm of \cite{boyd2006randomized}) the mRK (Algorithm~\ref{RKmomentum}) and the AccGossip (Algorithm~\ref{alg:acceleratedNew}) with the two options for the selection of the parameters presented in Section~\ref{AcceleratedVariants}. 

The starting vector of values $x^0=c$ is taken to be a Gaussian vector.  For the implementation of mRK we use the same parameters with the ones suggested in the stochastic heavy ball (SGB) setting in Chapter~\ref{ChapterMomentum}. For the AccRK (Option 1) we use $\lambda=\lambda_{\min}^+(\bA^\top\bA)$ and for AccRK (Option 2) we select $\nu=m$\footnote{For the networks under study we have
 $m < \frac{1}{\lambda_{\min}^+(\bW)}$. Thus, by choosing $\nu=m$ we select the pessimistic upper bound of the parameter \eqref{acsnklasda} and not its exact value \eqref{thenu}. As we can see from the experiments, the performance is still accelerated and almost identical to the performance of AccRK (Option 1) for this choice of $\nu$.}. From Figure~\ref{AccGossipPlots} it is clear that for all networks under study the two randomized gossip protocols with Nesterov momentum are faster than both the pairwise gossip algorithm of \cite{boyd2006randomized} and the mRK/SHB (Algorithm~\ref{RKmomentum}). To the best of our knowledge Algorithm~\ref{alg:acceleratedNew} (Option 1 and Option 2) is the first randomized gossip protocol that converges with provably accelerated linear rate and as we can see from our experiment its faster convergence is also obvious in practice.

\begin{figure}[t]
\centering
\begin{subfigure}{.3\textwidth}
  \centering
  \includegraphics[width=1\linewidth]{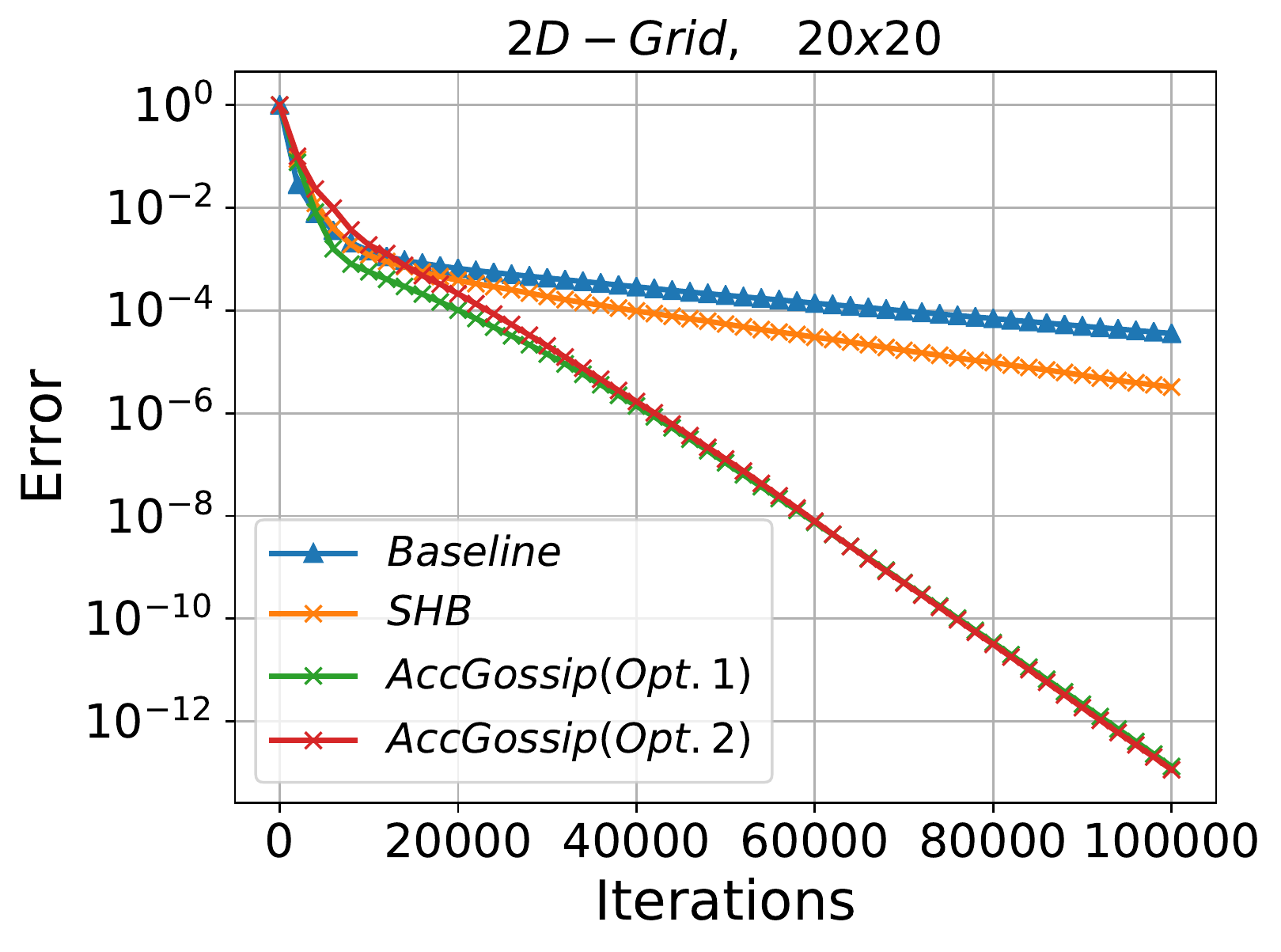}
\end{subfigure}%
\begin{subfigure}{.3\textwidth}
  \centering
  \includegraphics[width=1\linewidth]{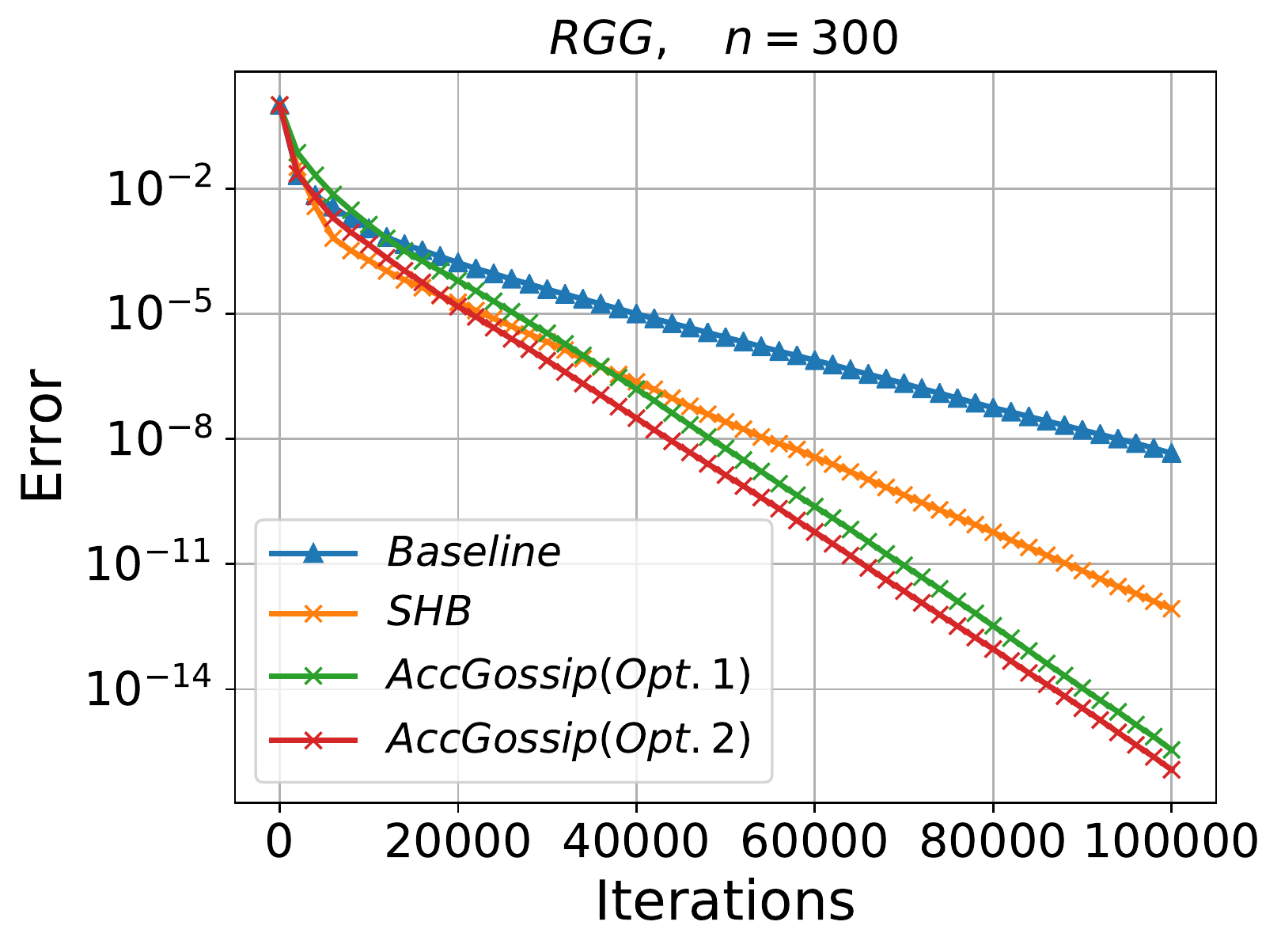}
\end{subfigure}%
\begin{subfigure}{.3\textwidth}
  \centering
  \includegraphics[width=1\linewidth]{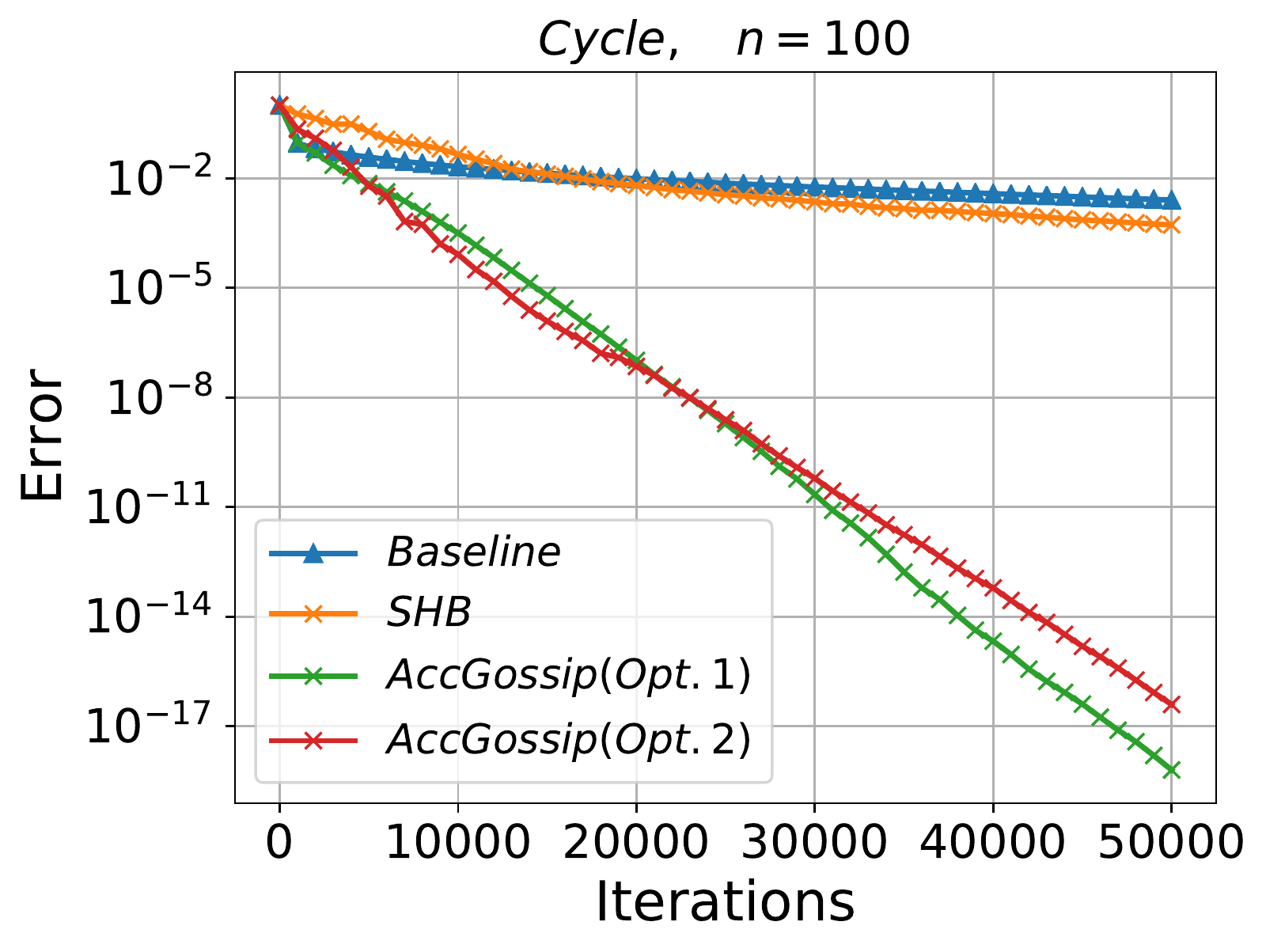}
\end{subfigure}\\
\begin{subfigure}{.3\textwidth}
  \centering
  \includegraphics[width=1\linewidth]{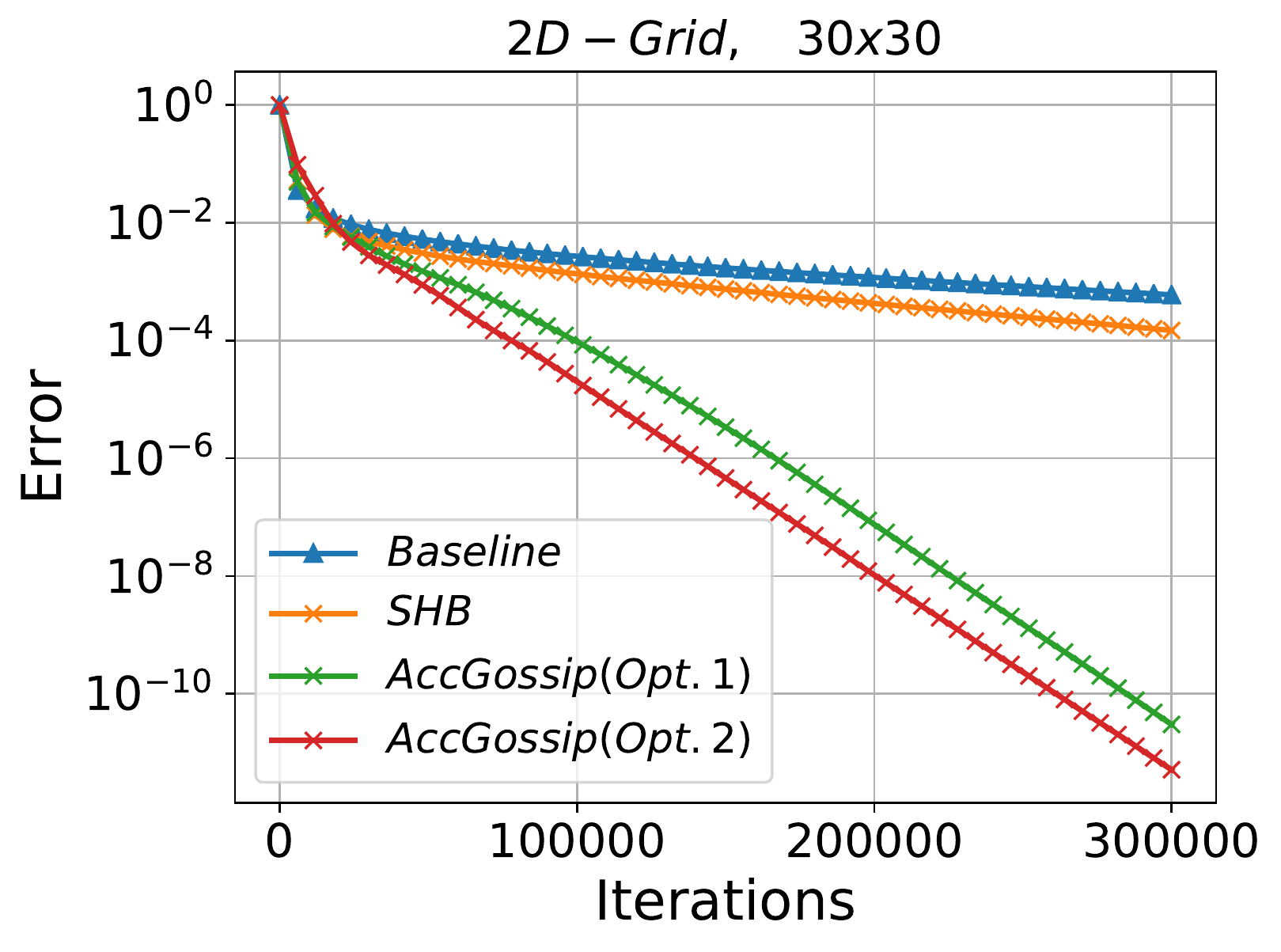}
\end{subfigure}
\begin{subfigure}{.3\textwidth}
  \centering
  \includegraphics[width=1\linewidth]{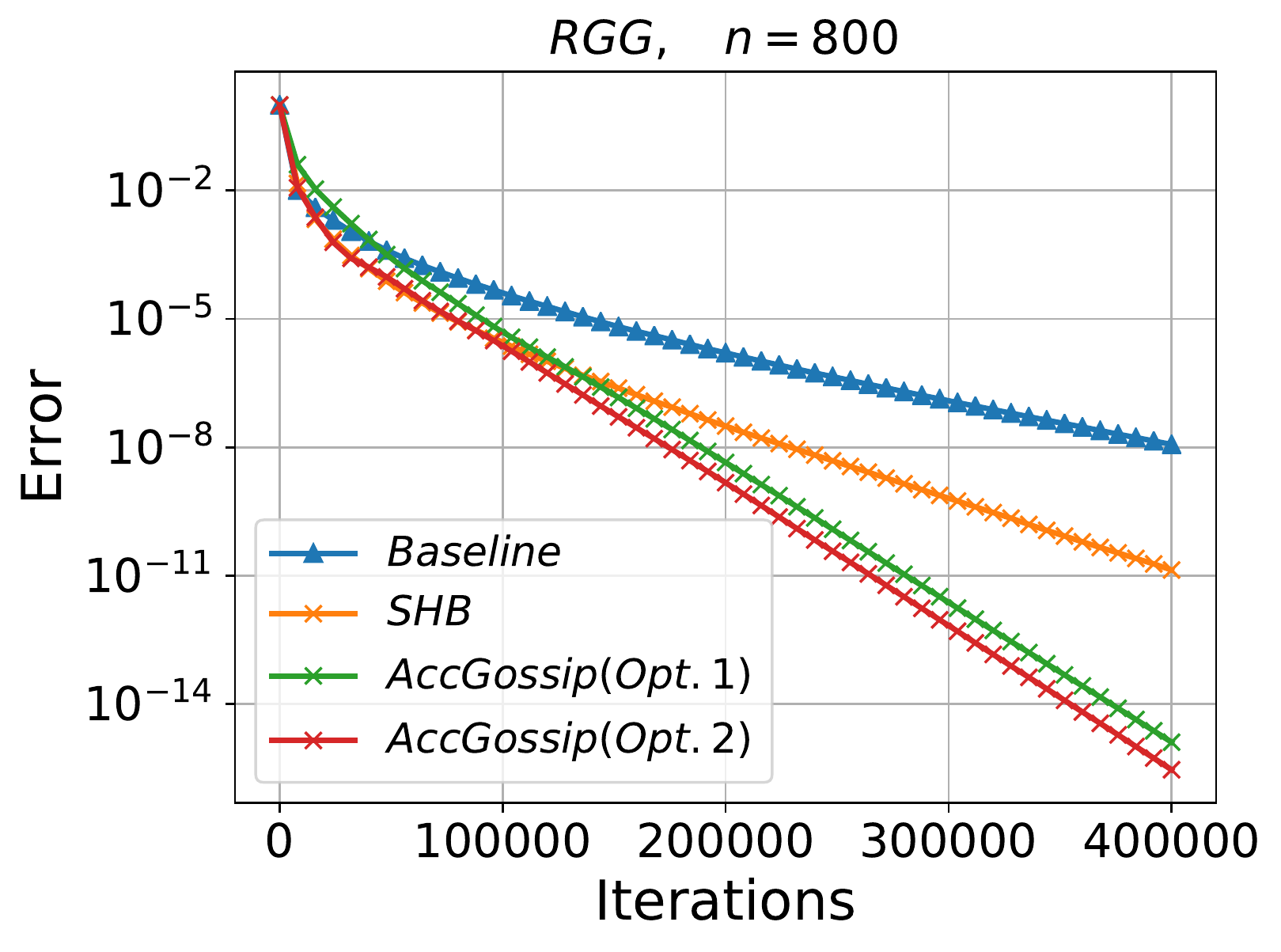}
\end{subfigure}
\begin{subfigure}{.3\textwidth}
  \centering
  \includegraphics[width=1\linewidth]{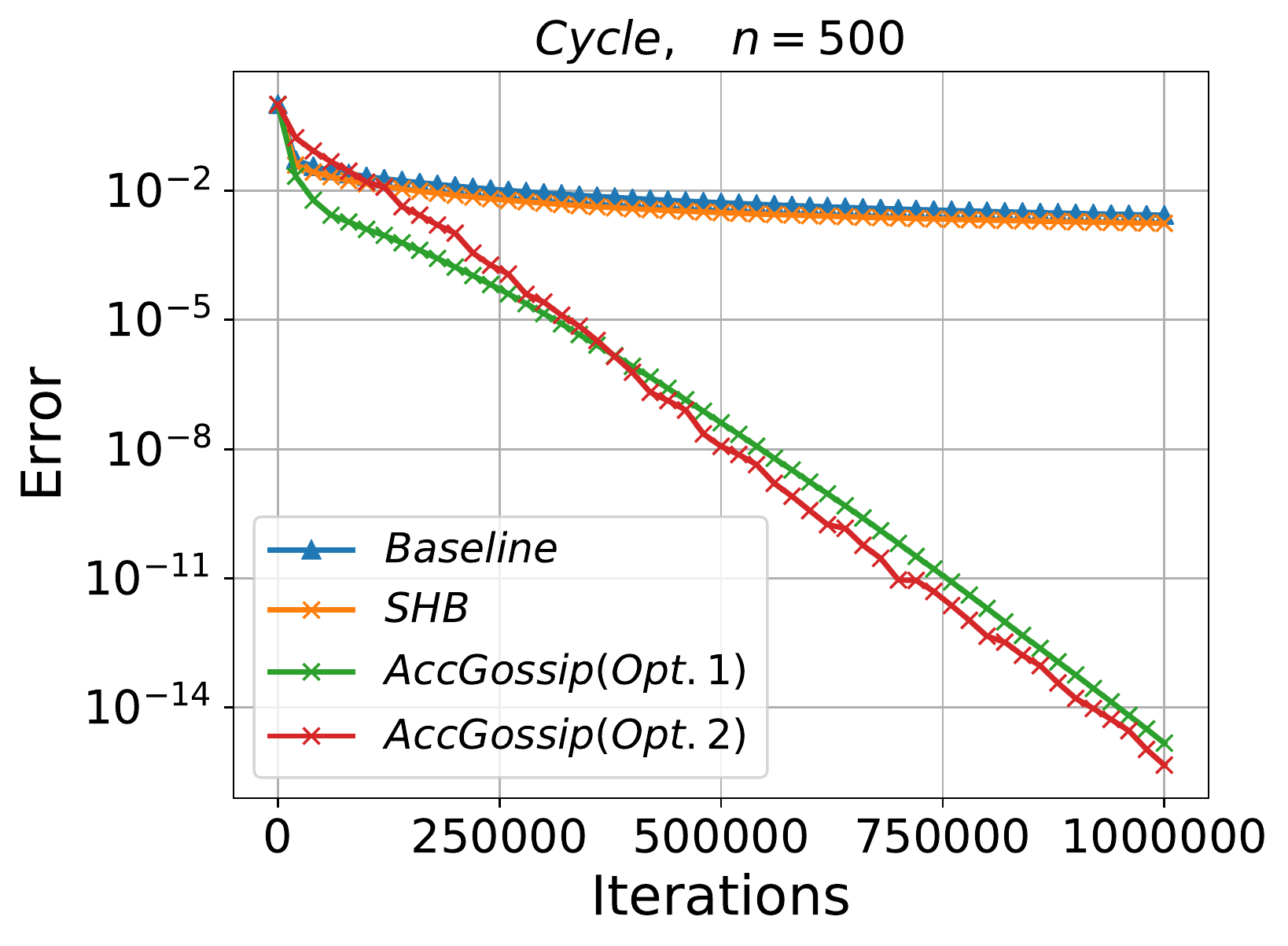}
\end{subfigure}\\
\caption{\footnotesize Performance of AccGossip (Option 1 and Option 2 for the parameters) in a 2-dimension grid, random geometric graph (RGG) and a cycle graph. The Baseline method corresponds to the randomized pairwise gossip algorithm proposed in \cite{boyd2006randomized} and the SHB represents the mRK (Algorithm~\ref{RKmomentum}) with the best choice of parameters as proposed in Chapter~\ref{ChapterMomentum} ; The $n$ in the title of each plot indicates the number of nodes of the network. For the grid graph this is $n \times n$.}
\label{AccGossipPlots}
\end{figure}

\subsection{Relaxed randomized gossip without momentum}
In the area of randomized iterative methods for linear systems it is know that over-relaxation (using of larger step-sizes) can be particularly helpful in practical scenarios. However, to the best of our knowledge there is not theoretical justification of why this is happening.

In our last experiment we explore the performance of relaxed randomized gossip algorithms ($\omega \neq 1$) without momentum and show that in this setting having larger stepsize can be particularly beneficial.

As we mentioned before (see Theorem~\ref{ConvergenceSketchProject}) the sketch and project method (Algorithm~\ref{FullSkecth}) converges with linear rate when the step-size (relaxation parameter) of the method is $\omega \in (0,2)$ and the best theoretical rate is achieved when $\omega=1$.  In this experiment we explore the performance of the standard pairwise gossip algorithm when the step-size of the update rule is chosen in $(1,2)$. Since there is no theoretical proof of why over-relaxation can be helpful we perform the experiments using different starting values of the nodes. In particular we choose the values of vector $c \in R^n$ to follow (i) Gaussian distribution, (ii) Uniform Distribution and (iii) to be integers values such that $c_i=i \in R$. Our findings are presented in Figure~\ref{RelaxedGossipFigure}. Note that for all networks under study and for all choices of starting values having larger stepsize, $\omega \in (1,2)$ can lead to better performance. Interesting observation from Figure~\ref{RelaxedGossipFigure} is that the stepsizes $\omega =1.8$ and $\omega=1.9$ give the best performance (among the selected choices of stepsizes) for all networks and for all choices of starting vector $x^0=c$.

\begin{figure}[t]
\centering
\begin{subfigure}{.3\textwidth}
  \centering
  \includegraphics[width=1\linewidth]{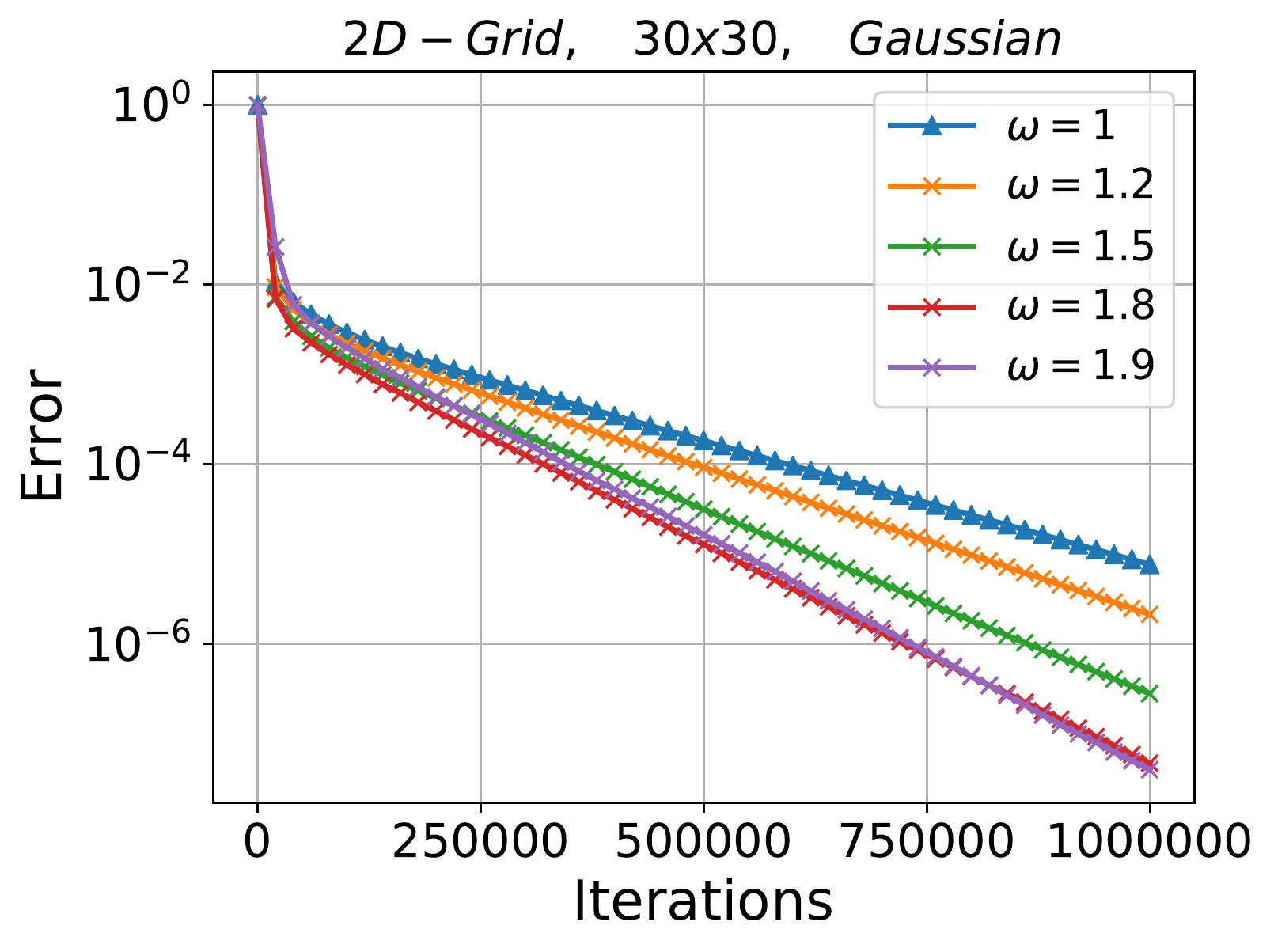}
\end{subfigure}%
\begin{subfigure}{.3\textwidth}
  \centering
  \includegraphics[width=1\linewidth]{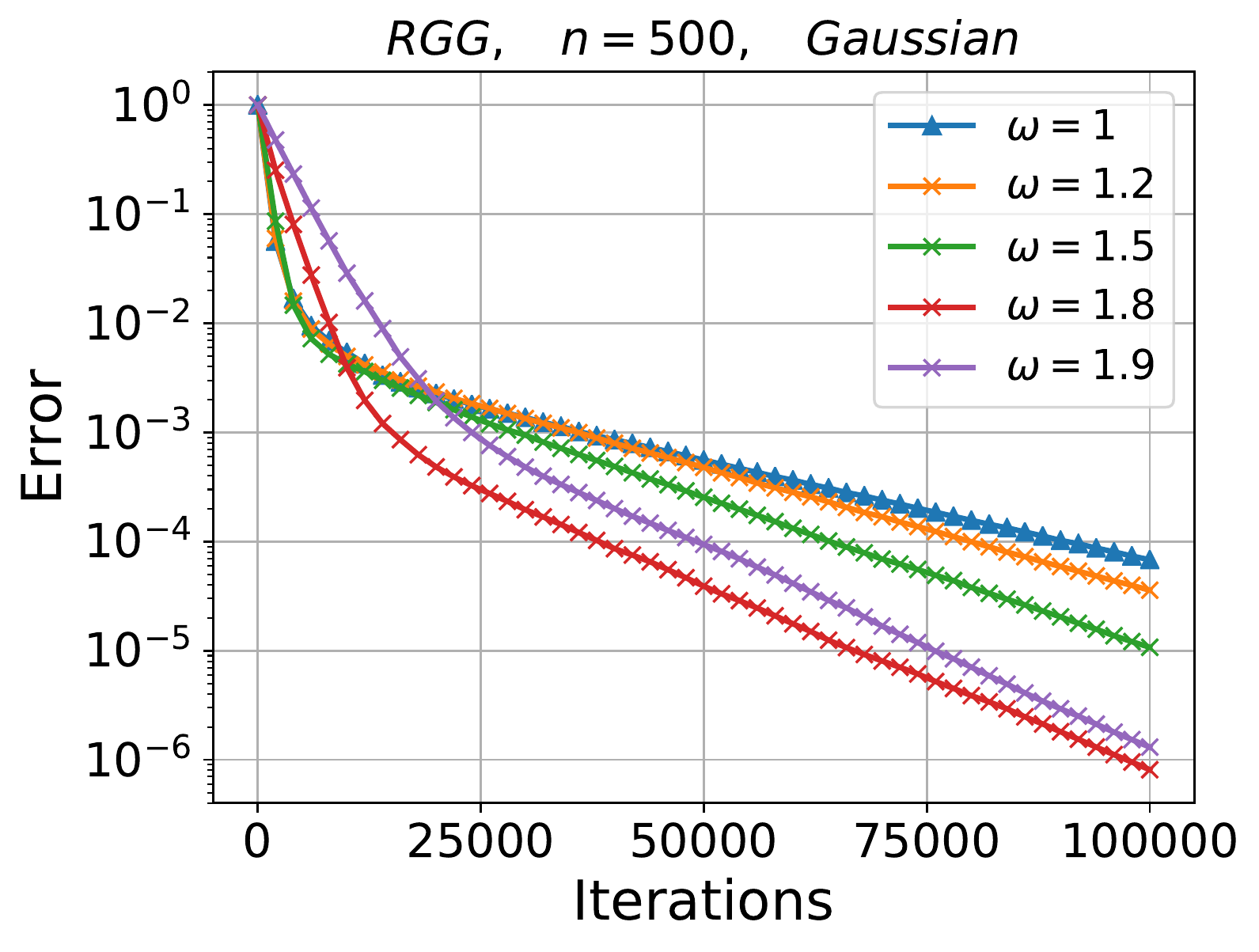}
\end{subfigure}
\begin{subfigure}{.3\textwidth}
  \centering
  \includegraphics[width=1\linewidth]{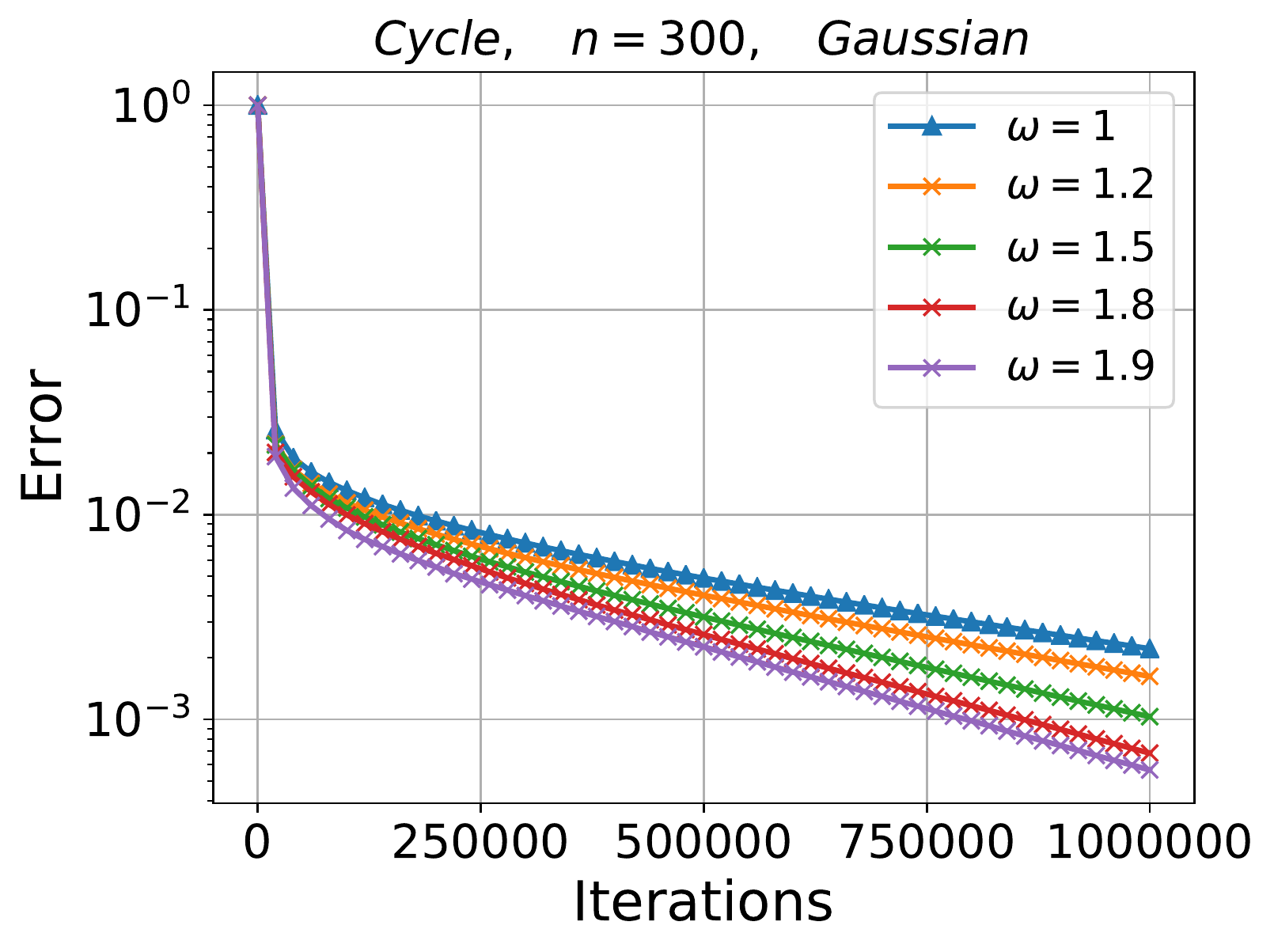}
\end{subfigure}\\
\begin{subfigure}{.3\textwidth}
  \centering
  \includegraphics[width=1\linewidth]{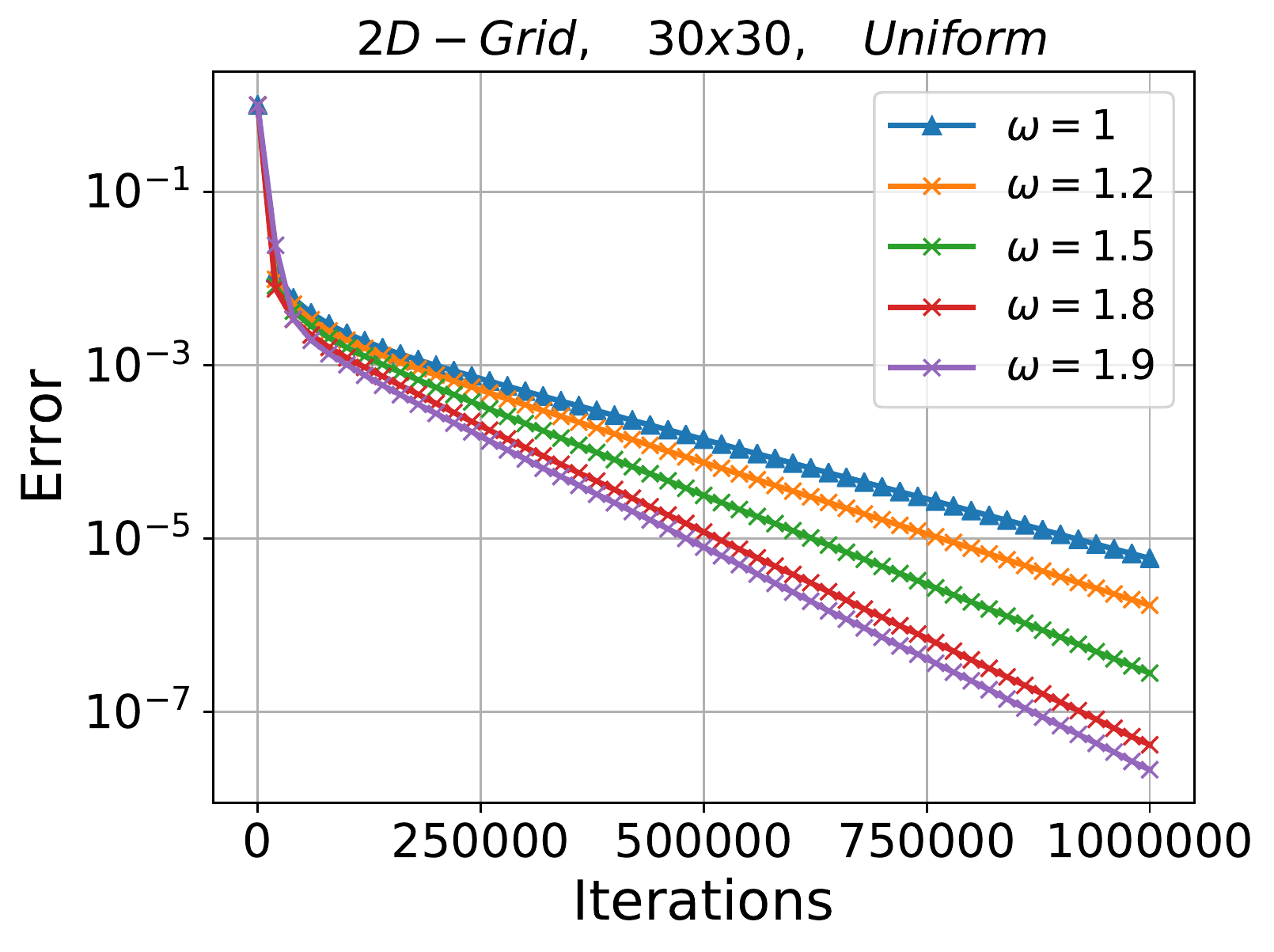}
\end{subfigure}%
\begin{subfigure}{.3\textwidth}
  \centering
  \includegraphics[width=1\linewidth]{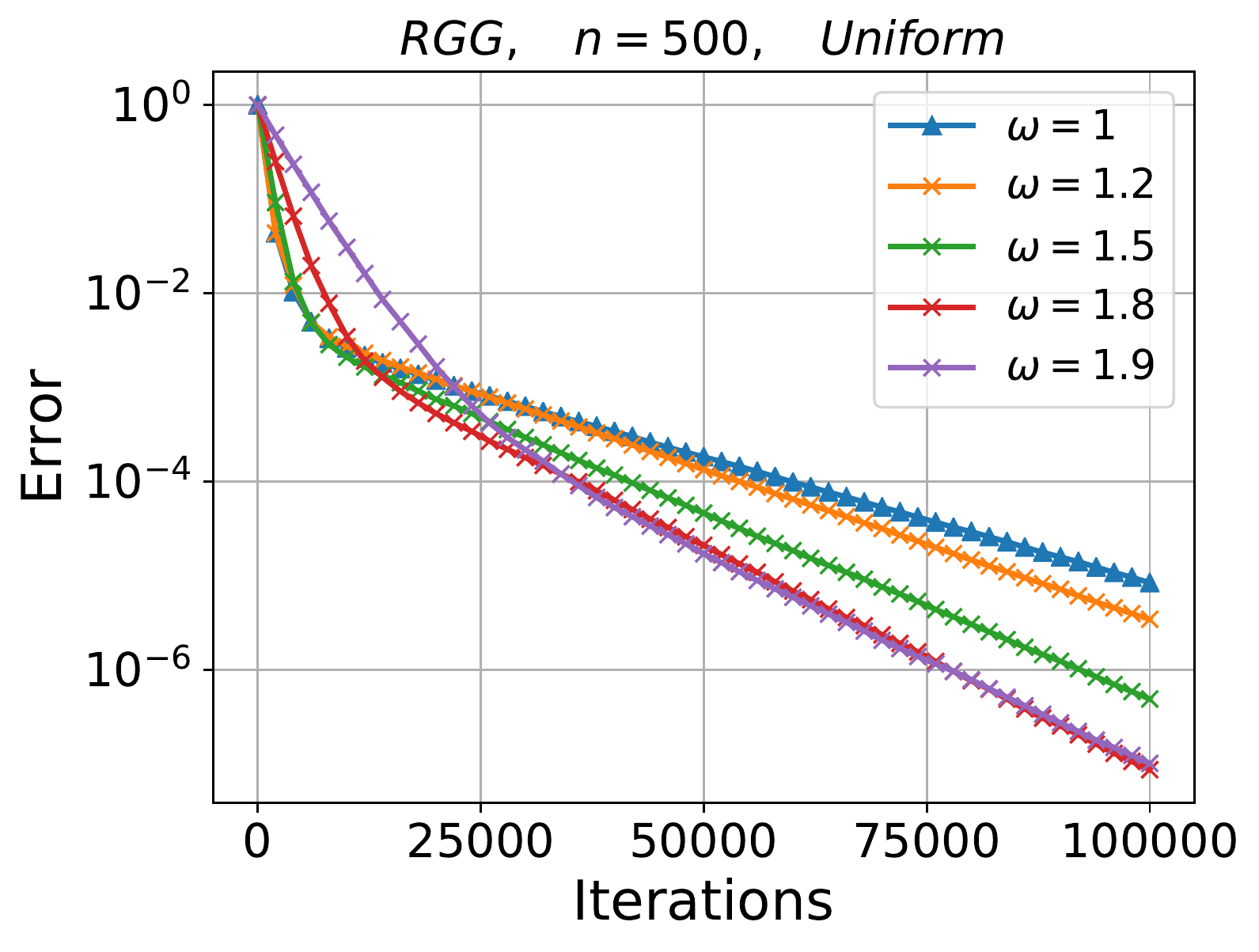}
\end{subfigure}
\begin{subfigure}{.3\textwidth}
  \centering
  \includegraphics[width=1\linewidth]{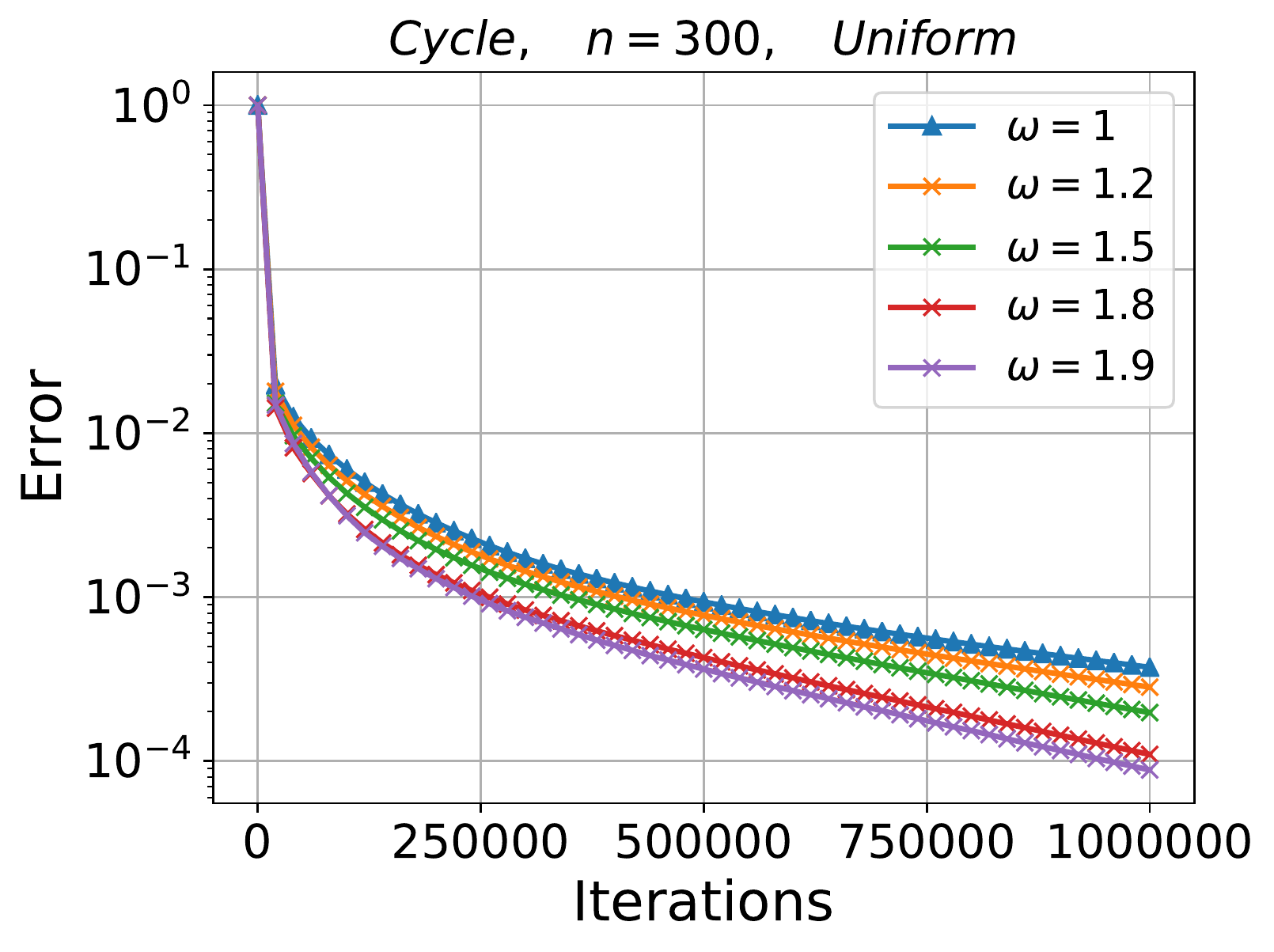}
\end{subfigure}\\
\begin{subfigure}{.3\textwidth}
  \centering
  \includegraphics[width=1\linewidth]{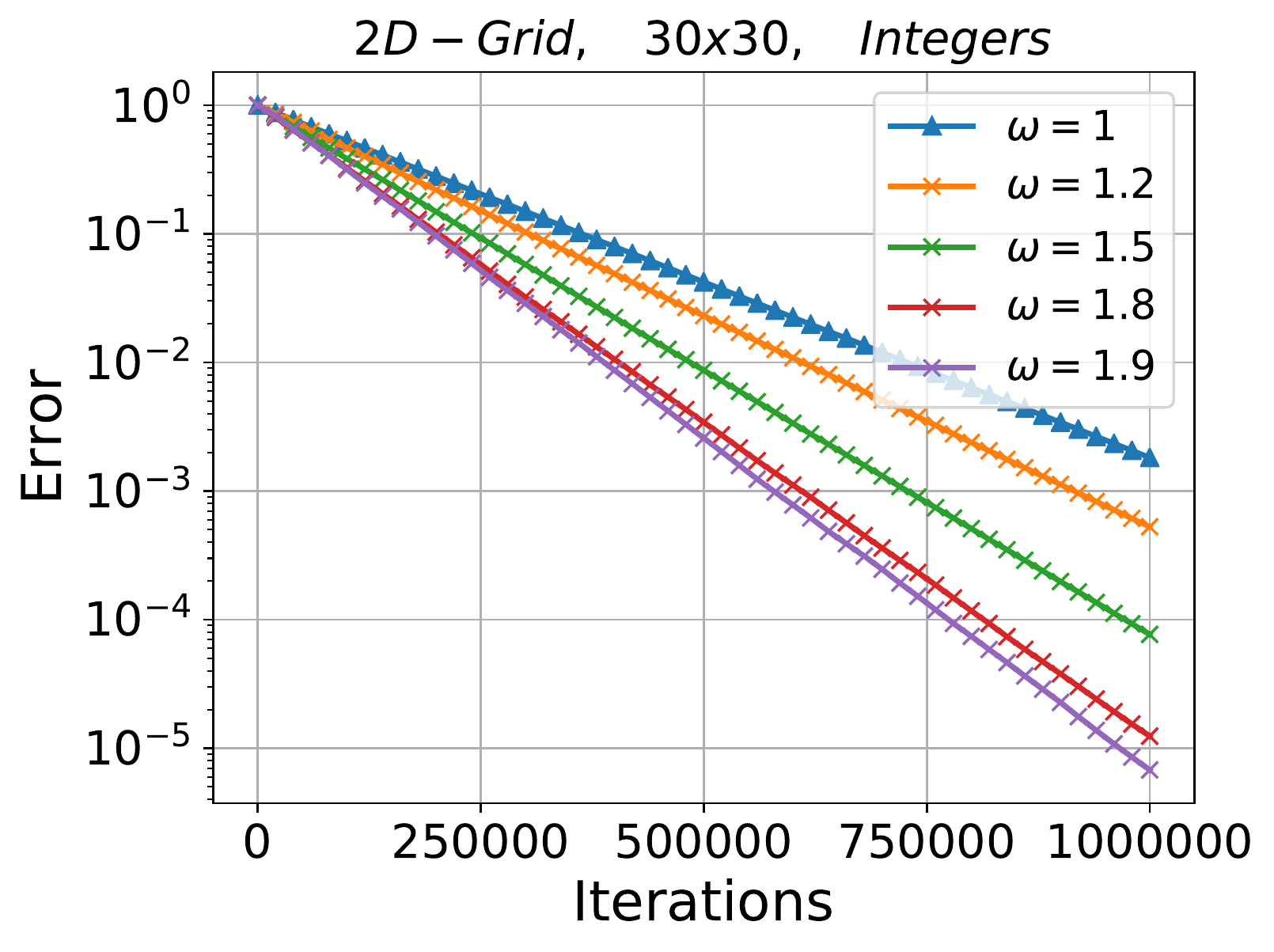}
\end{subfigure}%
\begin{subfigure}{.3\textwidth}
  \centering
  \includegraphics[width=1\linewidth]{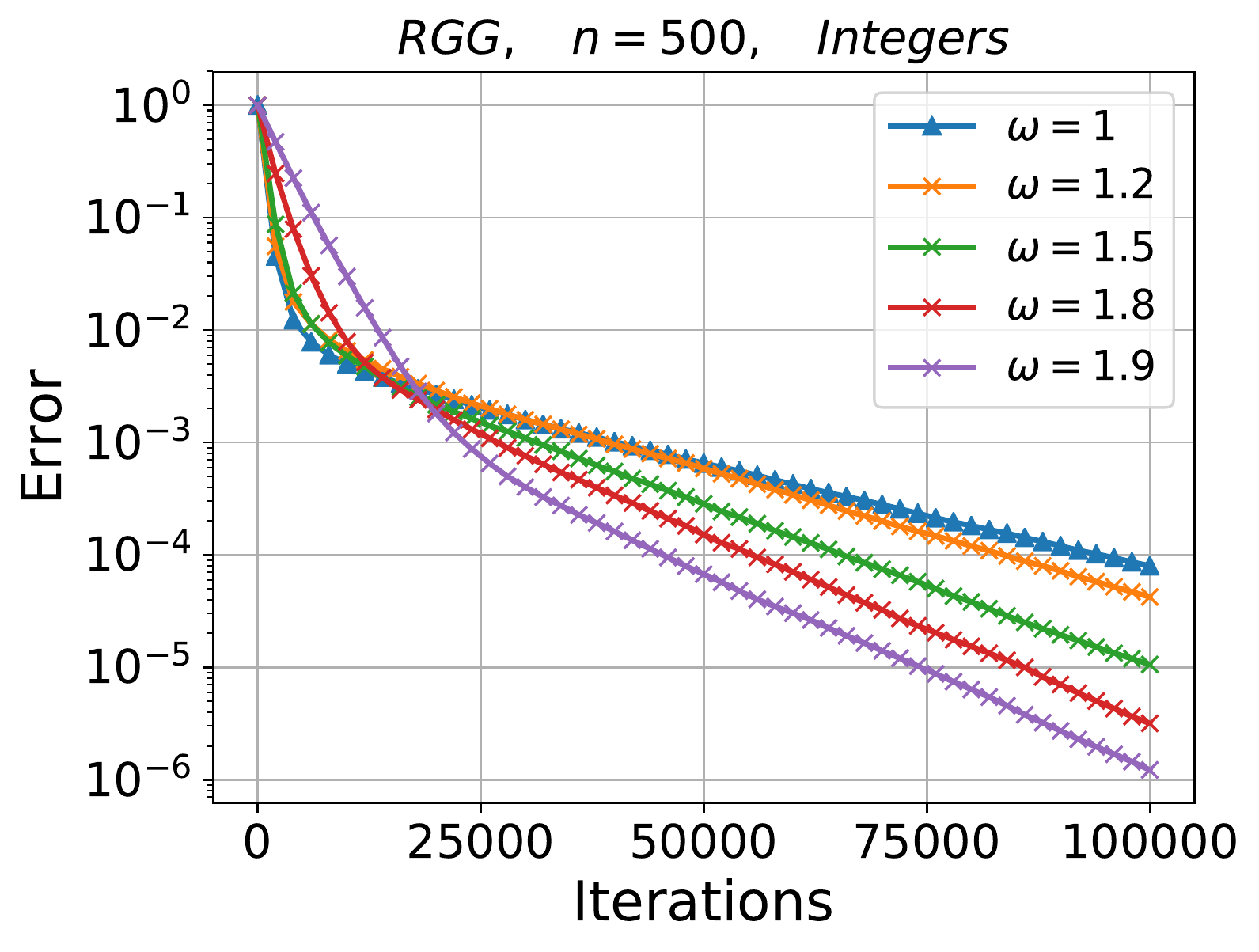}
\end{subfigure}
\begin{subfigure}{.3\textwidth}
  \centering
  \includegraphics[width=1\linewidth]{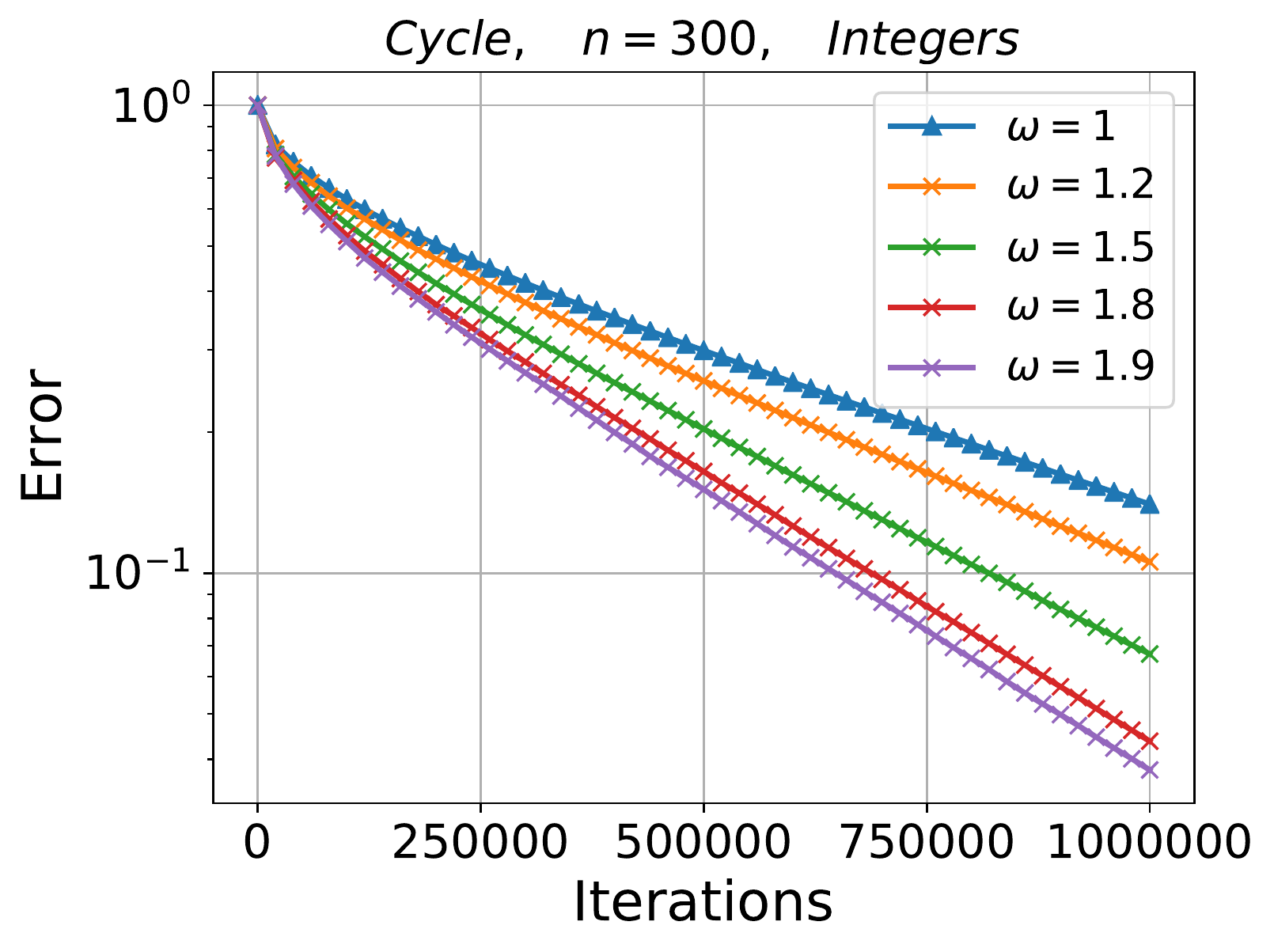}
\end{subfigure}
\caption{\footnotesize Performance of Relaxed randomized pairwise Gossip algorithm in a 2-dimension grid, random geometric graph (RGG) and a cycle graph. The case of $\omega=1$ corresponds to the randomized pairwise gossip algorithm proposed in \cite{boyd2006randomized} ; The $n$ in the title of each plot indicates the number of nodes of the network. For the grid graph this is $n \times n$. The title of each plot indicates the vector of starting values that is used. }
\label{RelaxedGossipFigure}
\end{figure}

\section{Conclusion}
\label{conclusion}
In this chapter, we present a general framework for the analysis and design of randomized gossip algorithms.  Using tools from numerical linear algebra and the area of randomized projection methods for solving linear systems we propose novel serial, block and accelerated gossip protocols for solving the average consensus and weighted average consensus problems. 

We believe that this work could open up several future avenues for research. Using similar approach with the one presented in this manuscript, many popular projection methods can be interpreted as gossip algorithms when used to solve linear systems encoding the underlying network. This can lead to the development of novel distributed protocols for average consensus.  

In addition, we speculate that the gossip protocols presented in this work can be extended to the more general setting of multi-agent consensus optimization where the goal is to minimize the average of convex or non-convex functions $\frac{1}{n} \sum_{i=1}^n f_i(x)$ in a decentralized way \cite{nedic2018network}. 

\section{Missing Proofs} 
 \subsection{Proof of Theorem \ref{thm:complexity_standard}}
 \label{ProofTave}
 \begin{proof}
Let $z^k \eqdef \|x^k-x^*\|$ , $x^0=c$ is the starting point and $\rho$ is as defined in \eqref{RateRho}.
From Theorem~\ref{ConvergenceSketchProject} we know that sketch and project method converges with\begin{equation}
 \label{ajcn}
 \Exp[\|x^k-x^*\|_{\bB}^2]\leq \rho^k \|x^0-x^*\|_{\bB}^2,
 \end{equation}
where $x^*$ is the solution of \eqref{BestApproximation_IntroThesis}. Inequality \eqref{ajcn}, together with Markov inequality can be used to give the following bound
\begin{equation} 
\Prob(z^k / z^0 \geq \varepsilon^2) \leq \frac{\Exp(z^k/z^0)}{\varepsilon^2} \leq \frac{\rho^k}{\varepsilon^2}.
\end{equation}
Therefore, as long as $k$ is large enough so that $\rho^k \leq \varepsilon^3$, we have
$\Prob \left(z^k / z^0 \geq \varepsilon^2 \right) \leq \varepsilon$. 
That is, if 
$$\rho^k \leq \varepsilon^3 \Leftrightarrow k \geq \frac{3\log\varepsilon}{\log\rho}\Leftrightarrow k \geq \frac{3\log(1/\varepsilon)}{\log(1/ \rho)}, $$
then:$$\Prob \left(\frac{\|x^{k}-\bar{c}1\|}{\|x^{0}-\bar{c}1\|}\geq \varepsilon \right)\leq\varepsilon.$$

Hence, an upper bound for value $T_{ave}(\epsilon)$ can be obtained as follows,
\begin{eqnarray}
T_{ave}(\epsilon)&=&\sup_{c\in \R^n} \inf  \left\{k\;:\; \Prob \left(z^k > \varepsilon z^0 \right)\leq\varepsilon \right\} \leq \sup_{c\in \R^n} \inf  \left\{k\;:\; k \geq \frac{3\log(1/\varepsilon)}{\log(1/ \rho)}  \right\} \notag\\
 & =&\sup_{c\in \R^n} \frac{3 \log(1/\varepsilon)}{\log(1/ \rho)}= \frac{3 \log(1/\varepsilon)}{\log(1/ \rho)}
  \leq  \frac{3 \log(1/\varepsilon)}{1-\rho},
\end{eqnarray}
where in last inequality we use $1/\log(1/ \rho) \leq 1/1-\rho$ which is true because $\rho \in (0,1)$.  
\end{proof}

\subsection{Proof of Theorem \ref{TheoremRBK}}
\label{ProofRBK}

\begin{proof}
The following notation conventions are used in this proof.
With $q_k$ we indicate the number of connected components of subgraph $\cG_k$, while with $\cV_r$ we denote the set of nodes of each connected component $q_k$ $(r \in \{1,2,\dots,q_k\})$. Finally, $|\cV_r|$ shows the cardinality of set $\cV_r$. Notice that, if $\cV$ is the set of all nodes of the graph then $\cV= \underset{r=\{1,2,...q\}}{\cup} \cV_r$ and $|\cV|=\sum\limits_{{r}=1}^{q} |\cV_r|$. 

Note that from equation \eqref{RBKaLgorithm}, the update of RBK for $\bA=\bQ$(Incidence matrix) can be expressed as follows:
\begin{equation}
\label{constrained}
\begin{aligned}
& \underset{x}{\text{minimize}}
& & \phi^k(x)\eqdef \|x-x^k\|^2\\
& \text{subject to}
& & \bI_{:C}^\top \bQ x=0
\end{aligned}
\end{equation}
Notice that $\bI_{:C}^\top \bQ$ is a row submatrix of matrix $\bQ$ with rows those that correspond to the random set $C \subseteq \cE$ of the edges. From the expression of matrix $\bQ$ we have that
$$(\bI_{:C}^\top \bQ)_{e:}^\top=f_i-f_j , \quad \forall e=(i,j) \in C \subseteq \cE.$$

Now, using this, it can be seen that the constraint $\bI_{:C}^\top \bQ x=0$ of problem \eqref{BestApproximation_IntroThesis} is equivalent to $q$ equations (number of connected components) where each one of them forces the values $x_i^{k+1}$ of the nodes $i \in \cV_r$ to be equal. That is, if we use $z_r$ to represent the value of all nodes that belong in the connected component $r$ then:
\begin{equation}
\label{zr}
 x_i^{k+1}= z_r \quad \forall i \in \cV_r,
\end{equation}
and the constrained optimization problem \eqref{BestApproximation_IntroThesis} can expressed as unconstrained as follows:
\begin{equation}
\label{fin}
 \underset{z}{\text{minimize}}  \quad \phi^k(z)=\sum\limits_{i \in \cV_1}^{}(z_1-x_i^{k})^2+...+\sum\limits_{i \in \cV_q}^{}(z_q-x_i^{k})^2,
\end{equation}
where $z=(z_1,z_2,\dots,z_q) \in \R^q$ is the vector of all values $z_r$ when $r \in \{1,2,\dots,q\}$.
Since our problem is unconstrained the minimum of equation \eqref{fin} is obtained when $\nabla \phi^k(z)=0$.

By evaluating partial derivatives of \eqref{fin} we obtain:
$$\frac{\partial \phi^k(z)}{\partial z_r}=0 \Longleftrightarrow \sum\limits_{i \in \cV_r}^{} 2(z_r- x_i^k)=0.$$
As a result, 
$$z_r=\frac{\sum\limits_{i \in \cV_r}^{} x_i^k}{|\cV_r|}, \quad \forall r \in \{1,2,\dots,q\}. $$
Thus from \eqref{zr}, the value of each node $i \in \cV_r$ will be updated to 
$$x_i^{k+1}= z_r=\frac{\sum\limits_{i \in \cV_r}^{} x_i^k}{|\cV_r|}.$$ .
\end{proof}
\newpage

\chapter{Privacy Preserving Randomized Gossip Algorithms}
\label{ChapterPrivacy}

\section{Introduction}
\label{sec:introduction}
In this chapter, similar to Chapter~\ref{ChapterGossip}, we consider the average consensus (AC) problem. 
In particular, we focus on randomized gossip algorithms for solving the AC problem and propose techniques for protecting the information of the initial values $c_i$, as these may be sensitive. We develop and analyze three privacy preserving variants of the randomized pairwise gossip algorithm (``randomly pick an edge $(i,j)\in \cE$ and then replace the values stored at vertices $i$ and $j$ by their average'') first proposed in \cite{boyd2006randomized} for solving the average consensus problem.  While we shall not formalize the notion of privacy preservation in this work, it will be intuitively clear that our methods indeed make it harder for nodes to infer information about the private values of other nodes, {\it which might be useful in practice.} 

\subsection{Background}
The literature on decentralized protocols for solving the average consensus problem is vast and has long history \cite{tsitsiklis1984problems, tsitsiklis1986distributed, bertsekas1989parallel, kempe2003gossip}. In particular, the algorithms for solving this problem can be divided into two broad categories: the average consensus algorithms \cite{xiao2004fast} which work in a synchronous setting and the gossip algorithms \cite{boyd2006randomized, shah2009gossip} which they consider ideal protocols for the asynchronous time model \cite{boyd2006randomized}.  
In the average consensus algorithms, all nodes of the network update their values simultaneously by communicating with a set of their neighbours and in each iteration of the algorithmic procedure the same update occurs. On the other hand, in gossip protocols only one edge of the whole network is selected at each iteration and only the nodes, that this edge connects, exchange their private information and update their values to their average. 

In this chapter, we focus on modifying the basic algorithm of \cite{boyd2006randomized}, which we refer to as ``Standard Gossip'' algorithm. In the following, we review some of the most important gossip protocols for solving the average consensus proposed in the last decade. While we do not address any privacy considerations for these protocols, they can serve as inspiration for further work. For a survey of relevant work, we refer the interested reader to \cite{dimakis2010gossip, olfati2007consensus, ren2007information, nedic2018network}.

The \emph{Geographic Gossip algorithm} was proposed in \cite{dimakis2008geographic}, in which the authors combine the gossip approach with a geographic routing towards a randomly chosen location with the main goal to improve the convergence rate of Standard Gossip algorithm. In each step, a node is activated, assuming that it is aware of its geographic location and some additional assumptions on the network topology, it chooses another node from the rest of the network (not necessarily one of its neighbours) and performs a pairwise averaging with this node. Later, using the same assumptions, this algorithm was extended into \emph{Geographic Gossip Algorithm with Path Averaging} \cite{benezit2010order}, in which connected sequences of nodes were chosen in each step and they averaged their values. More recently, in \cite{freschi2016accelerating} and \cite{freschi2017} authors propose a geographic and path averaging methods which converge to the average consensus without the assumption that nodes are aware of their geographic location. Recall that, in Section~\ref{BlockGossip} we show how the path averaging gossip algorithm can be seen as special case of the Randomized Block Kaczmarz method for solving consistent linear systems.

Another important randomized gossip algorithm is the \emph{Broadcast Gossip algorithm}, first proposed in \cite{aysal2009broadcast} and then extended in \cite{franceschelli2011distributed, wu2013, jun2013performance}. The idea of this algorithm is simple: In each step, a node in the network is activated uniformly at random, following the asynchronous time model, and broadcasts its value to its neighbours. The neighbours receive this value and update their own values. It was experimentally shown that this method converges faster than the pairwise and geographic randomized gossip algorithms.

Alternative gossip protocols are the so-called \emph{non-randomized Gossip algorithms} \cite{mou2010deterministic, he2011periodic, liu2011deterministic, yu2017distributed}. Typically, this class of algorithms executes the pairwise exchanges between nodes in a deterministic, such as predefined cyclic, order. $T$-periodic gossiping is a protocol which stipulates that each node must interact with each of its neighbours exactly once every $T$ time units. Under suitable connectivity assumptions of the network $\mathcal{G}$, the $T$-periodic gossip sequence will converge at a rate determined by the magnitude of the second largest eigenvalue of the stochastic matrix determined by the sequence of pairwise exchanges which occurs over a period. It has been shown that if the underlying graph is a tree, the mentioned eigenvalue is constant for all possible $T$-periodic gossip protocols.

 \emph{Accelerated Gossip algorithms} have also been proposed for solving the average consensus problem. In this setting, the nodes of the network incorporate additional memory to accelerate convergence. In particular,  the nodes update their value using an update rule that involves not only the current values of the sampled nodes but also their previous values. This idea is closely related to the shift register methods studied in numerical linear algebra for improving the convergence rate of linear system solvers. The works \cite{cao2006accelerated, liu2013analysis} have shown theoretically and numerically, that under specific assumptions this idea can improve the performance of the Standard Gossip algorithm. For more details on these gossip protocols check also Section~\ref{connectionOfAcceleratedMethods} of this thesis.

\emph{Randomized Kaczmarz-type Gossip algorithms.} In Chapter~\ref{ChapterGossip} of this thesis we presented how popular randomized Kaczmarz-type methods for solving large linear systems can also solve the AC problem. We explained how these methods can be interpreted as randomized gossip algorithms when applied to special systems encoding the underlying network structure and present in detail their decentralized nature. 

\paragraph{Asynchronous Time Model:} In this chapter, we are interested in the asynchronous time model \cite{boyd2006randomized, bertsekas1989parallel}. More precisely, we assume that each node of our network has a clock which ticks at a rate of $1$ Poisson process. This is equivalent of having available a global clock which ticks according to a rate $n$ Poisson process and selects an edge of the network uniformly at random. In general, the synchronous setting (all nodes update the values of their nodes simultaneously using information from a set of their neighbours) is convenient for theoretical considerations but is not representative of some practical scenarios, such as the distributed nature of sensor networks. For more details on clock modeling we refer the reader to \cite{boyd2006randomized}, as the contribution of this chapter is orthogonal to these considerations.

\paragraph{Privacy and Average Consensus:} 
Finally, the introduction of notions of privacy within the AC problem is relatively recent in the literature, and the existing works consider two different ideas.
\begin{enumerate}
\item In \cite{huang2012differentially}, the concept of differential privacy \cite{dwork2014algorithmic} is used to protect the output value $\bar{c}$ computed by all nodes. In this work, an exponentially decaying Laplacian noise is added to the consensus computation. This notion of privacy refers to protection of the {\it final average}, and formal guarantees are provided. 
\item A different approach with a more stricter goal is the design of privacy-preserving average consensus protocols that guarantee protection of the {\it initial values} $c_i$ of the nodes \cite{nozari2017differentially, manitara2013privacy, mo2017privacy}. In this setting each node should be unable to infer a lot about the initial values $c_i$ of any other node. In the existing works, this is mainly achieved with the clever addition of noise through the iterative procedure that guarantees preservation of privacy and at the same time converges to the exact average. We shall however mention, that none of these works address any specific notion of privacy (no clear measure of privacy is presented) and it is still not clear how the formal concept of differential privacy \cite{dwork2014algorithmic} can be applied in this setting.
\end{enumerate}

It is worth to highlight that all of the above-mentioned privacy preserving average consensus papers propose protocols which work on the synchronous setting (all nodes update their values simultaneously). To the best of our knowledge our proposed protocols are the first that solve the AC problem and at the same time protect the initial values of the nodes using the asynchronous time model (by having gossip updates).

\subsection{Main contributions}
In this chapter, we present three different approaches for solving the Average Consensus problem while at the same time protecting the information about the initial values. To the best of our knowledge, this work is the first which combines the \emph{gossip framework} with the privacy concept of protection of the initial values. It is important to stress that, we provide tools for protection of the initial values, but we do not address any specific notion of privacy or a threat model, nor how these quantitatively translate to any explicit measure. These would be highly application dependent, and we only provide theoretical convergence rates for the techniques we propose.

The methods we propose are all dual in nature. The dual setting of this chapter will be explained in detail in Section~\ref{sec:duality}. Recall that in Chapters~\ref{ChapterIntroduction}, \ref{ChapterMomentum} and \ref{ChapterInexact} of this thesis, we have shown how duality and dual algorithms can be used for solving consistent linear systems. In addition, in Chapter~\ref{ChapterGossip}, the dual viewpoint was extended to the concept of the average consensus problem and the first dual gossip algorithms were presented. As we have seen, the dual updates correspond to updates of the primal variables, via an affine mapping. Using this relationship of the primal and the dual spaces the convergence analysis of the dual methods can be easily obtained once the analysis of the primal methods is available (see for example the proof of Theorem~\ref{TheoremSDSA_IntroThesis} in the introduction of this thesis). In this chapter, one of our contributions is a novel dual analysis of randomized pairwise gossip (without the use of rates that obtain first through a primal analysis) which exactly recovers existing convergence rates for the primal iterates. 

We now outline the three different techniques we propose in this chapter, which we refer to as ``Binary Oracle'', ``$\epsilon$-Gap Oracle'' and ``Controlled Noise Insertion''. The first two are, to best of our knowledge, the first proposals of weakening the oracle used in the gossip framework. Privacy preservation is attained implicitly, as the nodes do not exchange the full information about their values. The last technique is inspired by the addition of noise proposed in \cite{le2014differentially, nozari2017differentially, manitara2013privacy, mo2017privacy} for the synchronous setting. We extend this technique by providing explicit finite time convergence guarantees.

{\bf Binary Oracle.}
We propose to reduce the amount of information transmitted in each iteration to a single bit\footnote{We do not refer to the size of the object being transmitted over the network, but the binary information that can be inferred from the exchange. In practice, this might be achieved using secure multiparty protocols \cite{cramer2015secure}, causing the overall network bandwidth to slightly increase compared to the usual implementation of standard gossip algorithm.}. More precisely, when an edge is selected, each corresponding node will only receive information whether the value on the other node is smaller or larger. Instead of setting the value on the selected nodes to their average, each node increases or decreases its value by a pre-specified step.

{\bf $\epsilon$-Gap Oracle.}
In this case, we have an oracle that returns one of three options and is parametrized by $\epsilon$. If the difference in values of sampled nodes is larger than $\epsilon$, an update similar to the one in Binary Oracle is taken. Otherwise, the values remain unchanged. An advantage compared to the Binary Oracle is that this approach will converge to a certain accuracy and stop there, determined by $\epsilon$ (Binary Oracle will oscillate around optimum for a fixed stepsize). However, in general, it will disclose more information about the initial values.

{\bf Controlled Noise Insertion.}
This approach is inspired by the works of \cite{manitara2013privacy, mo2017privacy}, and protects the initial values by inserting noise in the process. Broadly speaking, in each iteration, each of the sampled nodes first adds a noise to its current value, and an average is computed afterward. Convergence is guaranteed due to the correlation in the noise across iterations. Each node remembers the noise it added last time it was sampled, and in the following iteration, the previously added noise is first subtracted, and a fresh noise of smaller magnitude is added. Empirically, the protection of initial values is provided by first injecting noise into the system, which propagates across the network, but is gradually withdrawn to ensure convergence to the true average.

\begin{table}[t!]
\centering
\scalebox{0.9}{
\begin{tabular}{ |p{4.1cm}||M{3.5cm}|M{4.3cm}|M{0.6cm}|  }
 \hline
 \multicolumn{4}{|c|}{Main Results} \\
 \hline
 \hbox{Randomized Gossip Methods} & Convergence Rate & Success Measure & Thm \\
 \hline
 \hline
 Standard Gossip \cite{boyd2006randomized} & $\left( 1-\frac{\ac(\cG)}{2m} \right)^k$ & $\E{\frac{1}{2}\|\bar{c} \ones - x^k\|^2}$ & \ref{thm:G} \\
  \hline
 \hline
 {\bf New:} Private Gossip with Binary Oracle & $1 / \sqrt{k}$ & $\min_{t \leq k} \E{\frac{1}{m}\sum_{e} |x^t_i - x^t_j|}$ & \ref{thm:jhs988sh} \\
 \hline
 {\bf New:} Private Gossip with $\epsilon$-Gap Oracle & $ 1 / (k \epsilon^2)$ & $\E{\frac{1}{k}\sum_{t=0}^{k-1}\Delta^t (\epsilon)}$ & \ref{thm:09y09s9ffs} \\
 \hline
 {\bf New:} Private Gossip with Controlled Noise Insertion & $\left( 1-\min\left( \frac{\ac(\cG)}{2m},\frac{\gamma}{m}\right) \right)^k$ & $\E{ D(y^*)- D(y^{k}) }$& \ref{T: ng general convergence} \\
 \hline
\end{tabular}}
\caption{Complexity results of all proposed privacy preserving randomized gossip algorithms.}
\label{table1}
\end{table}

\paragraph{Convergence Rates of our Methods:} In Table~\ref{table1}, we present the summary of convergence guarantees for the above three techniques. By $\|\cdot\|$ we denote the standard Euclidean norm. 

The two approaches which restrict the amount of information disclosed, Binary Oracle and $\epsilon$-Gap Oracle, converge slower than the standard Gossip. In particular, these algorithms have sublinear convergence rate. At first sight, this should not be surprising, since we indeed use much less information. However, in Theorem~\ref{thm: stepsize_adaptive}, we show that if we had in a certain sense perfect global information, we could use it to construct a sequence of adaptive stepsizes, which would push the capability of the binary oracle to a linear convergence rate. However, this rate is still $m$-times slower than the standard rate of the binary gossip algorithm.  We note, however, that having the global information at hand is an impractical assumption. Nevertheless, this result highlights that there is a potentially large scope for improvement, which we leave for future work.

The approach of Controlled Noise Insertion yields a linear convergence rate which is driven by the minimum of two factors. Without going into details, which of these is bigger depends on the speed by which the magnitude of the inserted noise decays. If the noise decays fast enough, we recover the convergence rate of the standard the gossip algorithm. In the case of slow decay, the convergence is driven by this decay. By $\ac(\cG)$ we denote the {\em algebraic connectivity} of graph $\cG$ \cite{fiedler1973algebraic}. The parameter $\gamma$ controls the decay speed of the inserted noise, see Corollary~\ref{corolary}.

{\bf Measures of Success:} Note that the convergence of each randomized gossip algorithm in Table~\ref{table1} naturally depends on a different measure of suboptimality. All of them converge to $0$ as we approach the optimal solution. The details of these measures will be described later in the main body of this chapter. In particular a lemma that formally describes the key connections between these measures is presented in Section \ref{measures of success}. For now lets us give a brief description of these results.
The standard Gossip and Controlled Noise Insertion essentially depend on the same quantity, but we present the latter in terms of dual values as this is what our proofs are based on. Lemma~\ref{L: rel measures} formally specifies this equivalence. The binary oracle depends on the average difference among directly connected nodes. The measure for the $\epsilon$-Gap Oracle depends on quantities $\Delta^t(\epsilon)=\frac1m\left|\{ (i,j)\in  \cE: |x_i^t-x_j^t|\geq \epsilon\} \right|$, which is the number of edges that the values of their connecting nodes differ by more than $\epsilon$. 

\subsection{Structure of the chapter}
The remainder of this chapter is organized as follows: Section~\ref{sec:duality} introduces the basic setup that is used through the chapter. A detailed explanation of the duality behind the randomized pairwise gossip algorithm is given. We also include a novel and insightful dual analysis of this method as it will make it easier for the reader to parse later development.   In Section~\ref{sec:Private} we present our three private gossip algorithms as well as the associated iteration complexity results. Section~\ref{sec:experimentsPrivacy} is devoted to the numerical evaluation of our methods. Finally, conclusions are drawn in Section~\ref{sec:conclusion}. 

\section{Dual Analysis of Randomized Pairwise Gossip} \label{sec:duality}
As we outlined in the introduction of this chapter, our approach for extending the (standard) randomized pairwise gossip algorithm to privacy preserving variants utilizes duality. The purpose of this section is to formalize this duality. In addition, we provide a novel and self-contained dual analysis of randomized pairwise gossip. While this is of an independent interest, we include the proofs as their understanding aids in the understanding of the more involved proofs of our private gossip algorithms developed in the remainder of the chapter.

The main problems under study are the best approximation problem \eqref{BestApproximation_IntroThesis} and its dual \eqref{DualProblem_IntroThesis} that we have seen multiple times throughout the thesis. However, similar to Chapter~\ref{ChapterGossip} we focus on the more specific setting of the average consensus. To keep the chapter self-contained and for the benefit of the reader we present the definitions of these problems and we explain again how they are related to the average consensus problem. 

\subsection{Primal and dual problems}
\label{sec:primal_dual}

Consider solving the (primal) problem of projecting a given vector $c=x^0\in \R^n$ onto the solution space of a linear system: 
\begin{equation}\label{eq:primal}\min_{x\in \R^n} P(x) \eqdef \frac{1}{2}\|x-x^0\|^2 \quad \text{subject to} \quad \bA x=b,\end{equation} 
where $\bA\in \R^{m\times n}$, $b\in \R^m$, $x^0\in \R^n$. Note that this the best approximation problem \eqref{BestApproximation_IntroThesis} with $\bB=\bI$ (Identity matrix).
We assume the problem is feasible, i.e., that the system $\mA x = b$ is consistent. With the above optimization problem we associate the dual problem
\begin{equation}\label{eq:dual}\max_{y\in \R^m} D(y)\eqdef (b-\bA x^0)^\top y - \frac{1}{2}\|\bA^\top y\|^2.\end{equation}
As we have explained in the previous chapters, the dual is an unconstrained concave (but not necessarily strongly concave) quadratic maximization problem. It can be seen that as soon as the system $\mA x = b$ is feasible, the dual problem is bounded. Moreover, all bounded concave quadratics in $\R^m$ can be written in the as $D(y)$ for some matrix $\mA$ and vectors $b$ and $x^0$ (up to an additive constant).

With any dual vector $y$ we associate the primal vector via an affine transformation: $\phi(y) = x^0 + \mA^\top y.$ It can be shown that if $y^*$ is dual optimal, then $x^*=\phi(y^*)$ is primal optimal \cite{gower2015stochastic}. Hence, any dual algorithm producing a sequence of dual variables $y^t \to y^*$ gives rise to a corresponding primal algorithm producing the sequence $x^t \eqdef \phi(y^t) \to x^*$. We shall now consider one such dual algorithm.

\subsection{Stochastic dual subspace ascent}

Stochastic dual subspace ascent (SDSA) is a stochastic method for solving the dual problem \eqref{eq:dual}. In Section~\ref{BestaprooximationSection_INtro} we have already described how by choosing appropriately the main parameters of SDSA we can recover many known algorithms as special cases. In this chapter we focus only on one special case of the general algorithm. For the more general update rule of SDSA check equations \eqref{SDSA_IntroThesis} and \eqref{lambdak_IntroThesis}. In particular, following the notation of the rest of the thesis, we select $\omega=1$ (stepsize of the method) and $\bB=\bI$ (positive definite matrix that defines the geometry of the space). 
If we further use the fact that AC linear systems (see Definition~\ref{defACsystem}) have zero right hand side ($b=0$), then the update rule of SDSA takes the form:

\begin{equation}\label{eq:SDSA}y^{t+1} = y^t - \bS_t(\bS_t^\top \bA \bA^\top \bS_t)^\dagger \bS_t^\top \bA(x^0 + \bA^\top y^t),\end{equation}
where $\mS_t$ is a random matrix drawn independently at each iteration $t$ from an arbitrary but fixed distribution $\cD$, and $\dagger$ denotes the Moore-Penrose pseudoinverse. 
 
The corresponding primal iterates are defined via:
\begin{equation}
x^{t} \eqdef \phi(y^t) = x^0 + \bA^\top y^t.
\label{Eq: duality mapping}
\end{equation}

The relevance of this all to average consensus follows through the observation that for a specific choice of matrix $\mA$ and distribution $\cD$, the primal method produced by combining \eqref{Eq: duality mapping} and \eqref{eq:SDSA} is equivalent to the (standard) randomized pairwise gossip method (see discussion in Chapter~\ref{ChapterGossip}). In that case, SDSA is a dual variant of randomized pairwise gossip. In particular, in this chapter, we define $\cD$ as follows: $\bS_t$ is a  unit basis vector in $\R^m$, chosen uniformly at random from the collection of all such unit basis vectors, denoted $\{f_e \;|\; e\in \cE\}$. In this case, SDSA is the \textit{randomized coordinate ascent method} applied to the dual problem.

\subsection{Randomized gossip setup: choosing $\mA$}

We wish $(\mA,b)$ to be an average consensus (AC) system (see Definition~\ref{defACsystem}).  As we explained in Chapter~\ref{ChapterGossip} if $\mA x = b$ is an AC system, then the solution of the primal problem \eqref{eq:primal} is necessarily
$x^* = \bar{c} \cdot \ones$, where $\bar{c} = \frac{1}{n}\sum_{i=1}^n x^0_i$ is the value that each node needs to compute in the standard average consensus problem ($x^*_i = \bar{c}$ for all $i \in \cV$). 

In the rest of this chapter we focus on a specific AC system; the one  in which the matrix $\bA$  is the incidence matrix of the graph $\cG$.  In particular, we let $\mA\in \R^{m\times n}$ be the matrix defined as follows. Row $e=(i,j)\in \cE$  of $\mA$ is given by $\mA_{ei} = 1$, $\mA_{ej}=-1$ and $\mA_{el}=0$ if $l\notin \{i,j\}$.  Notice that the system $
\mA x=0$ encodes the constraints $x_i=x_j$ for all $(i,j)\in \cE$, as desired.

\subsection{Randomized pairwise gossip}
\label{sec:gossip}

We provide both primal and dual form of the (standard) randomized pairwise gossip algorithm. 

The primal form is standard and needs no lengthy commentary. At the beginning of the process, node $i$ contains private information $c_i = x^0_i$. In each iteration we sample a pair of connected nodes $(i,j)\in \cE$ uniformly at random, and update $x_i$ and $x_j$ to their average. We let the values at the remaining nodes intact.

\setcounter{algorithm}{\algG-1}
\begin{algorithm}[H]
  \caption{(Primal form)}
  \begin{algorithmic}[1]
    \Require{Vector of private values $c\in \R^n$.}
    \Ensure{Set $x^0=c$.}
 \For{$t= 0,1,\dots, k-1$}
 \State Choose edge $e = (i,j)\in \cE$ uniformly at random.
 \State Update the primal variable: 
$$x^{t+1}_l = \begin{cases} \frac{x^t_i+x^t_j}{2},\quad & l \in \{i,j\}\\ 
x^{t}_l, \quad & l \notin \{i,j\}.
 \end{cases}
$$
 \EndFor
 \State \textbf{Return} $x^k$
 \end{algorithmic}
\end{algorithm}

The dual form of the standard randomized pairwise gossip method is a specific instance of SDSA, as described in \eqref{eq:SDSA}, with $x^0=c$ and $\mS_t$ being a randomly chosen standard unit basis vector $f_e$ in $\R^m$ ($e$ is a randomly selected edge). It can be seen \cite{gower2015stochastic} that in that case, \eqref{eq:SDSA} takes the following form:

\setcounter{algorithm}{\algG-1}
\begin{algorithm}[H]
  \caption{(Dual form)}
  \begin{algorithmic}[1]
    \Require{Vector of private values $c\in \R^n$.}
    \Ensure{ Set $y^0=0\in\R^m$.}
 \For{$t= 0,1,\dots, k-1$}
 \State Choose edge $e = (i,j)\in \cE$ uniformly at random.
 \State Update the dual variable: 
$$y^{t+1} =  y^t + \lambda^t f_e \quad \mathrm{where} \quad \lambda^t =\mathrm{argmax}_{\lambda'} D(y^t + \lambda' f_e). $$
 \EndFor
 \State \textbf{Return} $y^k$
 \end{algorithmic}
\end{algorithm}

The following lemma is useful for the analysis of all our methods. It describes the increase in the dual function value after an arbitrary change to a single dual variable $e$.

\begin{lem} \label{lem:98y98yss}
Define $z=y^t + \lambda f_e$, where $e=(i,j)$ and $\lambda\in \R$. Then
 \begin{equation} \label{eq:89g9s8guffxx} D(z) - D(y^t) = -\lambda(x^t_i - x^t_j) - \lambda^2.\end{equation}
\end{lem}
\begin{proof}
The claim follows by direct calculation:
\begin{eqnarray*}
 D(y^{t}+ \lambda f_e) - D(y^t)
&=& -(\bA c)^\top (y^t + \lambda f_e) - \frac{1}{2}\|\bA^\top (y^t+ \lambda f_e)\|^2 + (\bA c)^\top y^t + \frac{1}{2}\| \bA^\top y^t\|^2\\
&=&  - \lambda f_e^\top \bA\underbrace{(c +  \bA^\top y^t)}_{x^t} - \frac{1}{2}\lambda^2 \underbrace{\| \bA^\top f_e\|^2}_{=2}  \quad = \quad  -\lambda (x^t_i-x^t_j) - \lambda^2.
\end{eqnarray*}
\end{proof}

The maximizer in $\lambda$ of the expression in \eqref{eq:89g9s8guffxx} leads to the exact line search formula
$\lambda^t = (x_j^t-x_i^t)/2$
used in the dual form of the method. 

\subsection{Complexity results}
\label{complexityResultsSection}

With graph $\cG = \{\cV,\cE\}$ we now associate a certain quantity, which we shall denote $\beta = \beta(\cG)$. It is the smallest nonnegative number $\beta$ such that the following inequality\footnote{We write $\sum_{(i,j)}$ to indicate sum over all {\em unordered} pairs of vertices. That is, we do not count $(i,j)$ and $(j,i)$ separately, only once. By $\sum_{(i,j)\in \cE}$ we denote a sum over all edges of $\cG$. On the other hand, by writing $\sum_i \sum_j $, we are summing over all (unordered) pairs of vertices twice.} holds for all $x\in \R^n$:

\begin{equation}\label{eq:hdgugvej}
 \sum_{(i,j)} (x_j  -x_i)^2 \leq \beta \sum_{(i,j)\in \cE} (x_j  -x_i)^2.
\end{equation}

The Laplacian matrix of graph $\cG$ is given by $\mL = \mA^\top \mA$. Let $\lambda_1(\mL)\geq \lambda_2(\mL) \geq \dots \geq \lambda_{n-1}(\mL)\geq \lambda_n(\mL)$ be the eigenvalues of $\mL$. The {\em algebraic connectivity} of $\cG$ is the second smallest eigenvalue of $\mL$:
\begin{equation} \label{eq:algebraic_connectivity} \ac(\cG) = \lambda_{n-1}(\mL).\end{equation}
We have $\lambda_{n}(\mL)=0$. Since we assume $\cG$ to be connected, we have $\ac(\cG)>0$. 
Thus,  $\ac(\cG)$ is the smallest nonzero eigenvalue of the Laplacian: $\ac(\cG)=\lambda_{\min}^+(\mL) = \lambda_{\min}^+(\mA^\top \mA ).$
As the next result states, the quantities $\beta(\cG)$ and $\ac(\cG)$ are inversely proportional.

\begin{lem}
\label{lem:beta} 
$\beta(\cG) = \frac{n}{\ac(\cG)}.$
\end{lem}
\begin{proof}
See Section~\ref{proofPrivacy1}.
\end{proof}

The following theorem gives a complexity result for (standard) randomized gossip. Our analysis is dual in nature.

\begin{thm}\label{thm:G} Consider the randomized gossip algorithm (Algorithm~\algG) with uniform edge-selection probabilities: $p_e=1/m$. Then:
\[\E{D(y^*) - D(y^{k}) } \leq \left(1-\frac{\ac(\cG)}{2m }\right)^k[D(y^*) - D(y^{0}) ]. \]
\end{thm}
\begin{proof}
See Section~\ref{proofPrivacy2}
\end{proof}

Theorem~\ref{thm:G}  yields the complexity estimate
${\cal O}\left(\frac{2m}{\ac(\cG)} \log(1/\epsilon)\right)$,
which exactly matches the complexity result obtained from the primal analysis (see \eqref{ratePairwise} in Chapter~\ref{ChapterGossip}). Hence, the primal and dual analyses give the same rate. 

Randomized coordinate descent methods were first analyzed in \cite{leventhal2010randomized, nesterov2012efficiency, richtarik2014iteration, richtarik2016parallel}. For a recent treatment, see \cite{qu2016coordinate, qu2016coordinate2}. Duality in randomized coordinate descent methods was studied in \cite{SDCA, qu2015quartz}. Acceleration was studied in \cite{lee2013efficient, fercoq2015accelerated, allen2016even}. These methods extend to nonsmooth problems of various flavours \cite{SPCDM, SCP}.

With all of this preparation, we are now ready to formulate and analyze our private gossip algorithms; we do so in Section~\ref{sec:Private}.

\section{Private Gossip Algorithms} \label{sec:Private}

In this section, we introduce three novel private gossip algorithms, complete with iteration complexity guarantees. In Section~\ref{measures of success} the key relationships between the measures of success (see Table~\ref{table1}) of all proposed algorithms are presented.  In Section~\ref{sec:B} the privacy is protected via a binary communication protocol. In Section~\ref{sec:E} we communicate more: besides binary information, we allow for the communication of a bound on the gap, introducing the $\epsilon$-gap oracle. In Section~\ref{sec: noise} we introduce a privacy-protection mechanism based on a procedure we call {\em controlled noise insertion}.

\subsection{Measures of success}
\label{measures of success}
We devote this subsection to present Lemma~\ref{L: rel measures} that formally specifies the connections between the different measures of suboptimality of the privacy preserving algorithms, firstly presented in Table~\ref{table1}. 
\begin{lem} 
\label{L: rel measures}
(Relationship between convergence measures)
Suppose that $x$ is primal variable corresponding to the dual variable $y$ as defined in \eqref{Eq: duality mapping}.
Dual suboptimality can be expressed as the following \cite{gower2015stochastic}:
\begin{equation}
\label{eq:99d8gds}
D(y^*)-D(y) = \frac{1}{2}\|\bar{c} \ones - x\|^2.
\end{equation}
Moreover, for any $x\in \R^{n}$ we have : 
\begin{eqnarray}
 \frac{1}{2n}\sum_{i=1}^n \sum_{j=1}^n  (x_j  -x_i)^2
&=& 
\|\bar{c}\ones - x\|^2 
\label{Eq: equality}
\\
\sum_{e=(i,j)\in \cE}|x_i-x_j|
&\leq& 
 \sqrt{mn}\|\bar{c} \ones -x \|,
\label{Eq: s upper bound}
\\
\sum_{e=(i,j)\in \cE}|x_i-x_j|
&\geq& 
\sqrt{\ac(\cG)}\|\bar{c} \ones-x \|,
\label{Eq: s lower bound}
\\
\sum_{e=(i,j)\in \cE}|x_i-x_j|
&\geq& \epsilon
  \left|\{ (i,j)\in  \cE: |x_i-x_j|\geq\epsilon\}\right|.
\label{Eq: delta bound}
\end{eqnarray}
\end{lem}
\begin{proof}
See Section~\ref{meslemma}.
\end{proof}

\subsection{Private gossip via binary oracle} 
\label{sec:B}

We now present the gossip algorithm with Binary Oracle in detail and provide theoretical convergence guarantee. The information exchanged between sampled nodes is constrained to a single bit, describing which of the nodes has the higher value. As mentioned earlier, we only present the conceptual idea, not how exactly would the oracle be implemented within a secure multiparty protocol between participating nodes \cite{cramer2015secure}.

We will first introduce the dual version of the algorithm.

\setcounter{algorithm}{\algB-1}
\begin{algorithm}[H]
  \caption{(Dual form)}
  \begin{algorithmic}[1]
    \Require{Vector of private values $c\in \R^n$, sequence of positive stepsizes $\{\lambda^t\}_{t=0}^{\infty}$}
    \Ensure{ Set $y^0=0\in \R^m$, $x^0=c$.}
 \For{$t= 0,1,\dots, k-1$}
 \State Choose edge $e = (i,j)\in \cE$ uniformly at random.
 \State Update the dual variable: 
$$y^{t+1} = \begin{cases} y^t + \lambda^t f_e, &\quad  x^t_i< x^t_j,\\
 y^t - \lambda^t f_e, & \quad x^t_i\geq x^t_j.
\end{cases}
 $$
 \State Set 
\begin{eqnarray*}
x^{t+1}_i &=& \begin{cases} x^t_i + \lambda^t, &\quad  x^t_i < x^t_j,\\
 x^t_i - \lambda^t , & \quad x^t_i\geq x^t_j.
\end{cases}
\\ 
 x^{t+1}_j &=& \begin{cases} x^t_j - \lambda^t, &\quad  x^t_i < x^t_j,\\
 x^t_j + \lambda^t  & \quad x^t_i\geq x^t_j.
\end{cases} 
\\
  x^{t+1}_l&=&x^t_l \qquad l\not\in\{i,j\}
\end{eqnarray*} 
 \EndFor
 \State \textbf{Return} $y^k$
 \end{algorithmic}
\end{algorithm}

The update of primal variables above is equivalent to set $x^{t+1}$ as primal point corresponding to dual iterate: $x^{t+1} = c+ \bA^\top y^{t+1} = x^t + \bA^\top (y^{t+1}-y^t)$. In other words, the primal iterates $\{x^t\}$ associated with the dual iterates $\{y^t\}$ can be written in the form:
\[x^{t+1} = \begin{cases} x^t + \lambda^t \bA_{e:}^\top, &\quad  x^t_i < x^t_j,\\
 x^t - \lambda^t \bA_{e:}^\top, & \quad x^t_i\geq x^t_j.
\end{cases} 
 \] 
It is easy to verify that due to the structure of $\bA$, this is equivalent to the updates above.

Since the evolution of dual variables $\{y^k\}$ serves only the purpose of the analysis, the method can be written in the primal-only form as follows:

\setcounter{algorithm}{\algB-1}
\begin{algorithm}[H]
  \caption{(Primal form)}
  \begin{algorithmic}[1]
    \Require{Vector of private values $c\in \R^n$, sequence of positive stepsizes $\{\lambda^t\}_{t=0}^{\infty}$}
    \Ensure{Set $x^0=c$.}
 \For{$t= 0,1,\dots, k-1$}
 \State Choose edge $e = (i,j)\in \cE$ uniformly at random.
 \State Set 
\begin{eqnarray*}
x^{t+1}_i &=& \begin{cases} x^t_i + \lambda^t, &\quad  x^t_i < x^t_j,\\
 x^t_i - \lambda^t , & \quad x^t_i\geq x^t_j.
\end{cases}
\\ 
 x^{t+1}_j &=& \begin{cases} x^t_j - \lambda^t, &\quad  x^t_i < x^t_j,\\
 x^t_j + \lambda^t  & \quad x^t_i\geq x^t_j.
\end{cases} 
\\
  x^{t+1}_l&=&x^t_l \qquad l\not\in\{i,j\}
\end{eqnarray*}
 \EndFor
 \State \textbf{Return} $x^k$
 \end{algorithmic}
\end{algorithm}

Given a sequence of stepsizes $\{\lambda^t\}$, it will be convenient to define $\alpha^k\eqdef \sum_{t=0}^{k}\lambda^t$ and $\beta^k \eqdef \sum_{t=0}^{k}\left(\lambda^t\right)^2$.
In the following theorem, we study the convergence of the quantity
\begin{equation} \label{eq: L def}
L^t\eqdef \frac{1}{m}\sum_{e=(i,j)\in \cE} |x^t_i - x^t_j|.
\end{equation} 

\begin{thm}\label{thm:jhs988sh}  For all $k\geq 1$ we have
\begin{equation}
\min_{t=0,1,\dots,k} \E{L^t} \leq \sum_{t=0}^{k} \frac{\lambda^t}{\alpha^k}\E{L^t} \leq U^k \eqdef \frac{D(y^*)-D(y^0)}{\alpha^k} + \frac{\beta^k}{\alpha^k}.\label{Eq: binary general rate}
\end{equation}
Moreover:
\begin{enumerate}
\item[(i)] If we set $\lambda^t = \lambda^0>0$ for all $t$, then  $U^k = \frac{D(y^*)-D(y^0)}{\lambda^0 (k+1)} + \lambda^0$.
\item[(ii)] Let $R$ be any constant such that $R\geq D(y^*)-D(y^0)$. If we fix  $k\geq 1$, then the choice of stepsizes $\{\lambda^0,\dots,\lambda^k\}$ which minimizes $U^k$ correspond to the constant stepsize rule $\lambda^t = \sqrt{\frac{R}{k+1}}$ for all $t=0,1,\dots,k$, and $U^k = 2\sqrt{\frac{R}{k+1}}$.
\item[(iii)] If we set $\lambda^t=a/\sqrt{t+1}$ for all $t=0,1,\dots,k$, then 
\begin{equation*}
U^k\leq \frac{D(y^*)-D(y^0)+a^2 \left(\log(k+3/2)+\log(2) \right)}{2a\left( \sqrt{k+2} -1\right)}=\cO\left(\frac{\log(k)}{\sqrt{k}}\right)
\end{equation*}
\end{enumerate}
\end{thm}
\begin{proof}
See Section~\ref{proofPrivacy3}
\end{proof}

The part \textit{(ii)} of Theorem~\ref{thm:jhs988sh} is useful in the case that we know exactly the number of iterations before running the algorithm, providing in a sense optimal stepsizes and rate $\cO(1/\sqrt{k})$. However, this might not be the case in practice. Therefore part \textit{(iii)} is also relevant, which yields the rate $\cO(\log(k)/\sqrt{k})$. These bounds are significantly weaker than the standard bound in Theorem~\ref{thm:G}. This should not be surprising though, as we use significantly less information than the Standard Gossip algorithm.

Nevertheless, there is a potential gap in terms of what rate can be practically achievable. The following theorem can be seen as a form of a bound on what convergence rate is possible to be attained by the Binary Oracle. However, this rate can be attained with access to very strong information. It requires a specific sequence of stepsizes $\lambda^t$ which is likely unrealistic in practical scenarios. This result points to a gap in the analysis which we leave open. We do not know whether the sublinear convergence rate in Theorem~\ref{thm:jhs988sh} is necessary or improvable without additional information about the system.

\begin{thm}
\label{thm: stepsize_adaptive}
For Algorithm~\algB\ with stepsizes chosen in iteration $t$ adaptively to the current values of $x^t$ as $\lambda^t = \frac{1}{2m}\sum_{e\in \cE}|x_i^t-x_j^t|$, we have
\begin{equation*}
\E{\|\overline{c}\ones-x^{k} \|^2}
\leq
\left(1-\frac{\ac(\cG)}{2m^2}\right)^k\|\overline{c}\ones-x^{0} \|^2
\end{equation*}
\end{thm}
\begin{proof}
See Section~\ref{proofPrivacy4}
\end{proof}
Comparing Theorem~\ref{thm: stepsize_adaptive} with the result for standard Gossip in Theorem~\ref{thm:G}, the convergence rate is worse by factor of $m$, which is the price we pay for the weaker oracle.

An alternative to choosing adaptive stepsizes is the use of adaptive probabilities \cite{AdaSDCA}. We leave such a study for future work.

\subsection{Private gossip via  $\epsilon$-gap oracle} 
\label{sec:E}

Here we present the gossip algorithm with $\epsilon$-Gap Oracle in detail and provide theoretical convergence guarantees. The information exchanged between the sampled nodes is restricted to be one of three cases, based on the difference of their values. As mentioned earlier, we only present the conceptual idea, not how exactly would the oracle be implemented within a secure multiparty protocol between participating nodes \cite{cramer2015secure}.

We will first introduce the dual version of the algorithm.

\setcounter{algorithm}{\algE-1}
\begin{algorithm}[H]
  \caption{(Dual form)}
  \begin{algorithmic}[1]
    \Require{Vector of private values $c\in \R^n$; error tolerance $\epsilon>0$}
    \Ensure{ Set $y^0 = 0 \in \R^m$; $x^0=c$.}
 \For{$t= 0,1,\dots, k-1$}
 \State Choose edge $e = (i,j)\in \cE$ uniformly at random.
 \State Update the dual variable: $$y^{t+1} = \begin{cases} y^t +\frac{ \epsilon}{2} f_e, &\quad  x^t_i-x^t_j < -\epsilon\\
 y^t - \frac{ \epsilon}{2} f_e, & \quad x^t_j-x^t_i < -\epsilon,\\
y^t, & \quad \text{otherwise.}
\end{cases}
$$
 \State If $x^t_i \leq x^t_j -  \epsilon$ then $x^{t+1}_i = x^t_i + \frac{ \epsilon}{2}$ and $x^{t+1}_j = x^t_j - \frac{ \epsilon}{2}$ 
\State  If $x^t_j \leq x^t_i -  \epsilon$ then $x^{t+1}_i = x^t_i - \frac{ \epsilon}{2}$ and $x^{t+1}_j = x^t_j + \frac{ \epsilon}{2}$ 
 \EndFor
 \State \textbf{Return} $y^k$
 \end{algorithmic}
\end{algorithm}

Note that the primal iterates $\{x^t\}$ associated with the dual iterates $\{y^t\}$ can be written in the form:
$$
x^{t+1} = \begin{cases} x^t + \frac{ \epsilon}{2} \bA_{e:}^\top, &\quad  x^t_i-x^t_j < -\epsilon\\
 x^t - \frac{ \epsilon}{2} \bA_{e:}^\top, & \quad x^t_j-x^t_i < -\epsilon,\\
x^t, & \quad \text{otherwise.}
\end{cases}
$$
The above is equivalent to setting
$x^{t+1} = x^t + \bA^\top (y^{t+1}-y^t)=c+ \bA^\top y^{t+1}$. 

Since the evolution of dual variables $\{y^t\}$ serves only the purpose of the analysis, the method can be written in the primal-only form as follows:

\setcounter{algorithm}{\algE-1}
\begin{algorithm}[H]
  \caption{(Primal form)}
  \begin{algorithmic}[1]
    \Require{Vector of private values $c\in \R^n$; error tolerance $\epsilon>0$}
    \Ensure{Set $x^0=c$.}
 \For{$t= 0,1,\dots, k-1$}
 \State Set $x^{t+1} = x^t$
 \State Choose edge $e = (i,j)\in \cE$ uniformly at random.
 \State  If $x^t_i \leq x^t_j -  \epsilon$ then $x^{t+1}_i = x^t_i + \frac{ \epsilon}{2}$ and $x^{t+1}_j = x^t_j - \frac{ \epsilon}{2}$ 
\State  If $x^t_j \leq x^t_i -  \epsilon$ then $x^{t+1}_i = x^t_i - \frac{ \epsilon}{2}$ and $x^{t+1}_j = x^t_j + \frac{ \epsilon}{2}$ 
 \EndFor
 \State \textbf{Return} $x^k$
 \end{algorithmic}
\end{algorithm}

Before stating the convergence result, let us define a quantity the convergence will naturally depend on. For each edge $e=(i,j)\in \cE$ and iteration $t \geq 0$ define the random variable \[\Delta^t_e(\epsilon) \quad \eqdef \quad \begin{cases} 1, \qquad |x^t_i-x^t_j|\geq \epsilon,\\
0, \qquad \text{otherwise.}\end{cases}\]
Moreover, let 
\begin{equation}
\label{eq:delta_k}
\Delta^t (\epsilon)\quad \eqdef \quad \frac{1}{m}\sum_{e\in \cE} \Delta^t_e(\epsilon).
\end{equation}

The following Lemma bounds the expected increase in dual function value in each iteration.

\begin{lem}\label{lem:09ys09y9ss} For all $t \geq 0$ we have
$\E{D(y^{t+1})-D(y^t) } \geq \frac{\epsilon^2}{4} \E{\Delta^t(\epsilon)}.$
\end{lem}
\begin{proof}
See Section~\ref{proofPrivacy5}
\end{proof}

Our complexity result will be expressed in terms of the quantity:
\begin{equation} \label{eq: delta def}
\delta^k(\epsilon)\quad \eqdef \quad  \E{\frac{1}{k}\sum_{t=0}^{k-1}\Delta^t(\epsilon)} = \frac{1}{k}\sum_{t=0}^{k-1}\E{\Delta^t (\epsilon)}.
\end{equation}

\begin{thm}\label{thm:09y09s9ffs} For all $k\geq 1$ we have \[\delta^k(\epsilon) \leq \frac{4\left(D(y^*)-D(y^0)\right)}{k\epsilon^2}.\]
\end{thm}
\begin{proof}
See Section~\ref{proofPrivacy6}
\end{proof}

Note that if $\Delta^k(\epsilon)=0$, it does not mean the primal iterate $x^k$ is optimal. This only implies that the values of all pairs of directly connected nodes differ by less than $\epsilon$.

\subsection{Private gossip via controlled noise insertion}
\label{sec: noise}

In this section, we present the gossip algorithm with Controlled Noise Insertion. As mentioned in the introduction of this chapter, the approach is similar to the technique proposed in \cite{manitara2013privacy, mo2017privacy}. Those works, however, address only algorithms in the synchronous setting, while our work is the first to use this idea in the asynchronous setting. Unlike the above, we provide finite time convergence guarantees and allow each node to add the noise differently, which yields a stronger result.

In our approach, each node adds noise to the computation independently of all other nodes. However, the noise added is correlated between iterations for each node. We assume that every node owns two parameters --- the initial magnitude of the generated noise $\sigma_i^2$ and rate of decay of the noise $\phi_i$. The node inserts noise $w_i^{t_i}$ to the system every time that an edge corresponding to the node was chosen, where variable $t_i$ carries an information how many times the noise was added to the system in the past by node $i$. Therefore, if we denote by $t$ the current number of iterations, we have $\sum_{i=1}^nt_i = 2t$.

In order to ensure convergence to the optimal solution, we need to choose a specific structure of the noise in order to guarantee the mean of the values $x_i$ converges to the initial mean. In particular, in each iteration a node $i$ is selected, we subtract the noise that was added last time, and add a fresh noise with smaller magnitude:
\begin{equation}
\label{thewdefinition}
w_i^{t_i} = \phi_i^{t_i}v_i^{t_i}-\phi_i^{t_i-1}v_i^{t_i-1},
\end{equation}
where $0 \leq \phi_i<1, v_i^{-1} = 0$ and $v_i^{t_i}\sim N(0,\sigma_i^2)$ for all iteration counters $k_i \geq 0$ is independent to all other randomness in the algorithm. This ensures that all noise added initially is gradually withdrawn from the whole network.

After the addition of noise, a standard Gossip update is made, which sets the values of sampled nodes to their average. Hence, we have
\begin{eqnarray*}
\lim_{t\rightarrow \infty}
\E{\left(
\overline{c}-
\frac1n\sum_{i=1}^n x_i^t
\right)^2}
&=&
\lim_{t\rightarrow \infty}
\E{\left(
\frac1n\sum_{i=1}^n \phi_i^{t_i-1}v_i^{t_i-1}
\right)^2}\\
&\leq&
\lim_{t\rightarrow \infty}
\E{\frac1n \sum_{i=1}^n \left(
 \phi_i^{t_i-1}v_i^{t_i-1}
\right)^2}
\\
&=&
\frac1n\lim_{t\rightarrow \infty}
\sum_{i=1}^n \E{  \left(
 \phi_i^{t_i-1}v_i^{t_i-1}
\right)^2} \\
&=&
\frac1n\lim_{t\rightarrow \infty}
\sum_{i=1}^n \E{  
 \phi_i^{2t_i-2}}\E{\left(v_i^{t_i-1}\right)^2
}
\\
&=&
\frac1n\lim_{t\rightarrow \infty}
\sum_{i=1}^n \E{  
 \phi_i^{2t_i-2}}\sigma_i^2
 =
 \frac1n \sum_{i=1}^n \sigma_i^2 \lim_{t\rightarrow \infty}
\E{  
 \phi_i^{2t_i-2}}
 \\
 &=&
 0,
\end{eqnarray*}
as desired. 

It is not the purpose of this work to define any quantifiable notion of protection of the initial values formally. However, we note that it is likely the case that the protection of private value $c_i$ will be stronger for bigger $\sigma_i$ and for $\phi_i$ closer to $1$.

For simplicity, we provide only the primal algorithm below.

\setcounter{algorithm}{\algN-1}
\begin{algorithm}[H]
  \caption{(Primal form)}
  \begin{algorithmic}[1]
    \Require{Vector of private values $c\in \R^n$; initial variances $\sigma^2_i \in \R_+$ and variance decrease rate $\phi_i$ such that $0\leq \phi_i < 1$ for all nodes $i$.}
    \Ensure{Set $x^0=c$; $t_1=t_2=\dots = t_n=0$, $v_1^{-1}=v_2^{-1}=\dots = v_n^{-1}=0$.}
 \For{$t= 0,1,\dots, k-1$}
 \State Choose edge $e = (i,j)\in \cE$ uniformly at random.
 \State Generate $v_i^{t_i}\sim N(0,\sigma^2_i)$ and $v_j^{t_j}\sim N(0,\sigma^2_j)$
 \State Set $$w_i^{t_i} = 
 \phi_i^{t_i}v_i^{t_i}-\phi_i^{t_i-1}v_i^{t_i-1} 
$$
 $$w_j^{t_j} = 
 \phi_j^{t_j}v_j^{t_j}-\phi_j^{t_j-1}v_j^{t_j-1} 
 $$
\State Update the primal variable: \[x^{t+1}_i =x^{t+1}_j= 
\frac{x^t_i+w^{t_i}_i+x^t_j+w^{t_j}_j}{2},\ \forall\, l \neq i,j:\,x^{t+1}_l=x^{t}_l
  \]
\State Set $t_i=t_i+1$, $t_j=t_j+1$
 \EndFor
 \State \textbf{Return} $x^k$
 \end{algorithmic}
\end{algorithm}

We now provide results of dual analysis of Algorithm~\algN. The following lemma provides us the expected decrease in dual suboptimality for each iteration.

\begin{lem}\label{L: noise exp iteration bound}
Let $d_i$ denote the number of neighbours of node $i$. Then, 
\begin{eqnarray}
\begin{split}
\E{ D(y^*)- D(y^{t+1}) } \leq & 
\left( 1-\frac{\ac(\cG)}{2m}\right)\E{D(y^*)- D(y^{t})} +\frac{1}{4m}\sum_{i=1}^n d_i\sigma^2_i \E{\phi_i^{2t_i}}
\\& \qquad-
\frac{1}{2m}\sum_{e\in \cE} \E{\left( \phi_i^{t_i-1}v_i^{t_i-1} x_j^t+ \phi_j^{t_j-1}v_j^{t_j-1} x_i^t
\right)}.
\end{split}
\label{Eq: noise gossip iteration bound final}
\end{eqnarray}
\end{lem}
\begin{proof}
See Section~\ref{proofPrivacy7}
\end{proof}

We use the lemma to prove our main result, in which we show linear convergence for the algorithm. For notational simplicity, we decided to have $\rho^t=(\rho)^t$, i.e. superscript of $\rho$ denotes its power, not an iteration counter.

\begin{thm}
Let us define the following quantities:
\begin{eqnarray*}
\rho&\eqdef& 1-\frac{\ac(\cG)}{2m},
\\
\psi^t&\eqdef&\frac{1}{\sum_{i=1}^n\left(d_i\sigma_i^2\right)}\sum_{i=1}^n d_i \sigma_i^2\left(1-\frac{d_i}{m}\left(1-\phi_i^2\right) \right)^{t}.
\end{eqnarray*}

Then for all $k\geq 1$ we have the following bound
\begin{equation*}
\E{ D(y^*)- D(y^{k}) } \leq  \rho^k \left( D(y^*)- D(y^{0}) \right)  + \frac{\sum\left(d_i\sigma_i^2\right)}{4m}\sum_{t=1}^k \rho^{k-t}\psi^{t}.
\end{equation*}
\label{T: ng general convergence}
\end{thm}
\begin{proof}
See Section~\ref{proofPrivacy8}
\end{proof}

Note that $\psi^t$ is a weighted sum of $t$-th powers of real numbers smaller than one.  For large enough $t$, this quantity will depend  on the largest of these numbers. This brings us to define $M$ as the set of indices $i$ for which the quantity $1-\frac{d_i}{m}\left(1-\phi_i^2\right)$ is maximized:  
$$
M=\arg\max_{i} \left\{ 1-\frac{d_i}{m}\left(1-\phi_i^2\right)\right\}. 
$$
Then for any $i_\mathrm{max}\in M$ we have
$$
\psi^t  
\approx
 \frac{1}{\sum_{i=1}^n \left(d_i\sigma_i^2\right)} \sum_{i\in M} d_i \sigma_i^2\left(1-\frac{d_i}{m}\left(1-\phi_i^2\right) \right)^{t}
 =
  \frac{\sum_{i\in M} d_i \sigma_i^2}{\sum_{i=1}^n \left(d_i\sigma_i^2\right)} \left(1-\frac{d_{i_\mathrm{max}}}{m}\left(1-\phi_{i_\mathrm{max}}^2\right) \right)^{t},
$$
which means that increasing $\phi_j$ for $j\not\in M$ will not substantially influence convergence rate. 

Note that as soon as we have
\begin{equation}
\rho>
 1-\frac{d_{i}}{m}\left(1-\phi_{i}^2\right)  \label{eq: treshold}
\end{equation} 
for all $i$, the rate from theorem \ref{T: ng general convergence} will be driven by $\rho^k$ (as $k \rightarrow \infty$) and we will have

\begin{equation}
\E{ D(y^*)- D(y^{k}) } = \tilde{O}\left(\rho^k\right)
\label{eq: tildeO}
\end{equation}

One can think of the above as a threshold: if there is $i$ such that $\phi_i$ is large enough so that the inequality \eqref{eq: treshold} does not hold, the convergence rate is driven by $\phi_{i_\mathrm{max}}$. Otherwise, the rate  is not influenced by the insertion of noise. Thus, in theory, we do not pay anything in terms of performance as long as we do not hit the threshold. One might be interested in choosing $\phi_i$ so that the threshold is attained for all $i$, and thus $M=\{1,\dots,n\}$. This motivates the following result:

\begin{cor}
\label{corolary}
Let us choose 
\begin{equation}
\phi_i \eqdef \sqrt{1-\frac{\gamma}{d_i}}
\label{Eq: phi_i def}
\end{equation}
for all $i$, where  $\gamma \leq d_{\mathrm{min}}$.
Then
\begin{eqnarray*}
\E{ D(y^*)- D(y^{k}) } & \leq &
\left( 1-\min\left( \frac{\ac(\cG)}{2m},\frac{\gamma}{m}\right) \right)^k \left( D(y^*)- D(y^{0})+\frac{\sum_{i=1}^n \left(d_i\sigma_i^2\right)}{4m}k \right). 
\end{eqnarray*}

As a consequence, $\phi_i=\sqrt{1-\frac{\ac(\cG)}{2d_i}}$ is the largest decrease rate of noise for node $i$ such that the guaranteed convergence rate of the algorithm is not violated.
\label{C: noisy gossip special}
\end{cor}
\begin{proof}
See Section~\ref{proofPrivacy9}
\end{proof}

While the above result clearly states the important threshold, it is not always practical as $\ac(\cG)$ might not be known. However, note that if we choose $\frac{nd_{\mathrm{min}}}{2(n-1)}\leq \gamma \leq d_{\mathrm{min}}$, we have 
$\min\left( \frac{\ac(\cG)}{2m},\frac{\gamma}{m}\right) = \frac{\ac(\cG)}{2m}$
since 
$\frac{\ac(\cG)}{2}\leq \frac{n}{n-1}\frac{d_{\mathrm{min}}}{2}\leq \gamma,$
where $e(\cG)$ denotes graph {\em edge connectivity}: the minimal number of edges to be removed so that the graph becomes disconnected.   Inequality $\ac(\cG) \leq \frac{n}{n-1}d_{\mathrm{min}}$ is a well known result in spectral graph theory \cite{fiedler1973algebraic}. As a consequence, if for all $i$ we have
$$
\phi_i\leq \sqrt{1-\frac{(n-1)d_\mathrm{min}}{2nd_i}},
$$
then the convergence rate is not driven by the noise.

\section{Numerical Evaluation}
\label{sec:experimentsPrivacy}
We devote this section to experimentally evaluate the performance of the Algorithms~\algB, \algE\ and \algN\ we proposed in the previous sections, applied to the Average Consensus problem. In the experiments, we used two popular graph topologies the cycle graph (ring network) and the random geometric graph (see Figure~\ref{Illustration} for an illustration of the two graphs). 
\begin{itemize}
\item {\em Cycle graph} with $n$ nodes: $\mathcal{C}(n)$. In our experiments we choose $n=10$. This small simple graph with regular topology is chosen for illustration purposes.
\item {\em Random geometric graph} with $n$ nodes and radius $r$:  $\mathcal{G}(n,r)$. Random geometric graphs \cite{penrose2003random} are very important in practice because of their particular formulation which is ideal for modeling wireless sensor networks \cite{gupta2000capacity, boyd2006randomized}. In our experiments we focus on a $2$-dimensional random geometric graph $\cG(n,r)$ which is formed by placing $n$ nodes uniformly at random in a unit square with edges between nodes which are having euclidean distance less than the given radius $r$. We set this to be to be $r = r(n) = \sqrt{\log(n)/n}$ --- it is well know that the connectivity is preserved in this case~\cite{gupta2000capacity}. We set $n = 100$.
\end{itemize}

\begin{figure}[t]
\centering
\begin{subfigure}{.45\textwidth}
  \centering
  \includegraphics[width=1\linewidth]{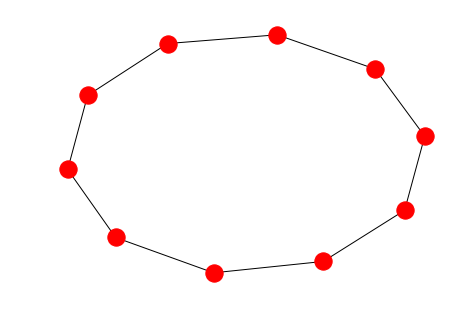}
  \caption{Cycle Graph: $\cC(10)$ }
\end{subfigure}%
\begin{subfigure}{.45\textwidth}
  \centering
  \includegraphics[width=1\linewidth]{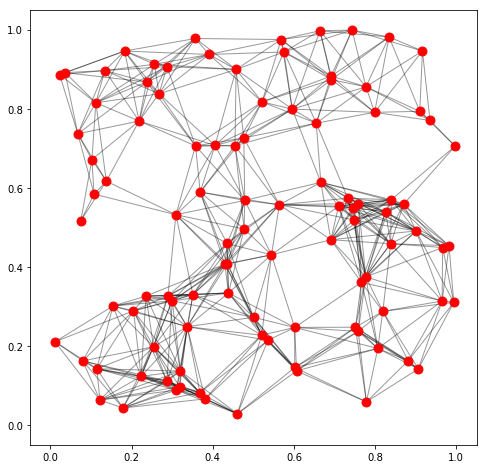}
  \caption{Random Geometric Graph: $\cG(n,r)$ }
\end{subfigure}
\caption{{Illustration of the two graph topologies we use in this section.}}
\label{Illustration}
\end{figure}

\textbf{Setup:}
In all experiments we generate a vector with of initial values $c_i$ from a uniform distribution over $[0,1]$. We run several experiments and present two kinds of figures that help us to understand how the algorithms evolve and verify the theoretical results of the previous sections. These figures are:
\begin{enumerate}
\item The evolution of the initial values of the nodes. In these figures, we plot how the trajectory of the values $x_i^t$ of each node $i$ evolves throughout iterations. The black dotted horizontal line represents the exact average consensus value which all nodes should approach, and thus all other lines should approach this level.
\item The evolution of the relative error measure $\|x^t-x^*\|^2 / \|x^0-x^*\|^2 $ where $x^0 =c \in \R^n$ is the starting vector of the values of the nodes. In these figures we choose to have the relative error, both in normal and logarithmic scale on the vertical axis and the number of iterations on the horizontal axis.
\end{enumerate}
For our evaluation we run each privacy preserving algorithm for several parameters and for a pre-specified number of iterations not necessarily the same for each experiment. 

To illustrate the first concept (trajectories of the values $x_i^t$) , we provide a simple example of the evolution of the initial values $x_i^t$ for the case of the Standard Gossip algorithm \cite{boyd2006randomized} in Figure~\ref{ExactMethod}. The horizontal black dotted line represents the average consensus value. It is the exact average of the initial values $c_i$ of the nodes in the network.
\begin{figure}[H]
\centering
\begin{subfigure}{.45\textwidth}
  \centering
  \includegraphics[width=1\linewidth]{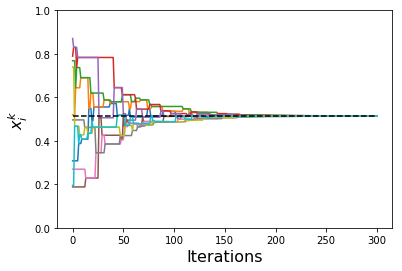}
  \caption{Cycle Graph}
\end{subfigure}%
\begin{subfigure}{.45\textwidth}
  \centering
  \includegraphics[width=1\linewidth]{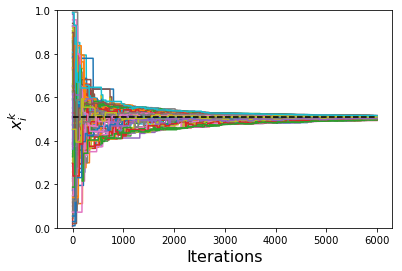}
  \caption{Random Geometric Graph}
\end{subfigure}\\
\caption{Trajectories of the values $x_i^t$ for the Standard Gossip algorithm for Cycle Graph and a random geometric graph. Each line corresponds to value $x_i$ of node $i \in \cV$.}
\label{ExactMethod}
\end{figure}

In the rest of this section we evaluate the performance of the three privacy preserving randomized gossip algorithms of Section~\ref{sec:Private}, and contrast with the above Standard Gossip algorithm, which we refer to as ``Baseline'' in the following figures labels.

\subsection{Private gossip via binary oracle}

In this section, we evaluate the performance of Algorithm~\algB\ presented in Section~\ref{sec:B}. In the algorithm, the input parameters are the positive stepsizes $\{\lambda^t\}_{t=0}^\infty$. The goal of the experiments is to compare the performance of the proposed algorithm using different choices of $\lambda^t$.

In particular, we use decreasing sequences of stepsizes $\lambda^t=1/t$ and $\lambda^t=1/ \sqrt{t}$, and three different fixed values for the stepsizes $\lambda^t=\lambda \in \{0.001, 0.01, 0.1\}$. We also include the adaptive choice $\lambda^t = \frac{1}{4m}\sum_{e\in \cE}|x_i^t-x_j^t|$ which we have proven to converge with linear rate in Theorem~\ref{thm: stepsize_adaptive}. We compare these choices in Figures \ref{Cycle10LamError} and \ref{RGG100LamErr}, along with the Standard Gossip algorithm for clear comparison.

\begin{figure}[H]
\centering
\begin{subfigure}{.3\textwidth}
  \centering
  \includegraphics[width=1\linewidth]{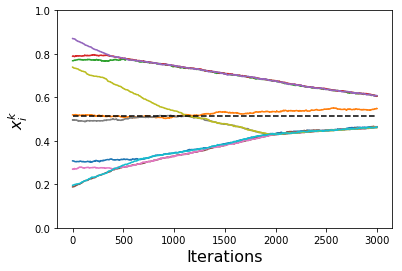}
  \caption{ $\lambda^t=\lambda=0.001$}
\end{subfigure}%
\begin{subfigure}{.3\textwidth}
  \centering
  \includegraphics[width=1\linewidth]{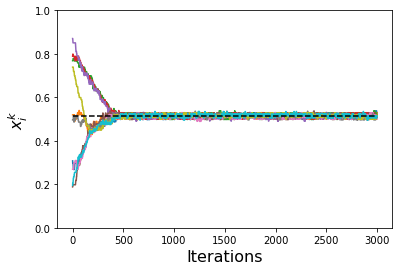}
  \caption{$\lambda^t=\lambda=0.01$}
\end{subfigure}
\begin{subfigure}{.3\textwidth}
  \centering
  \includegraphics[width=1\linewidth]{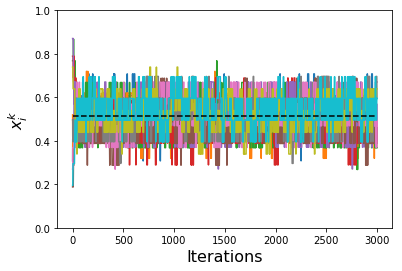}
  \caption{$\lambda^t=\lambda=0.1$}
\end{subfigure}\\
\begin{subfigure}{.3\textwidth}
  \centering
  \includegraphics[width=1\linewidth]{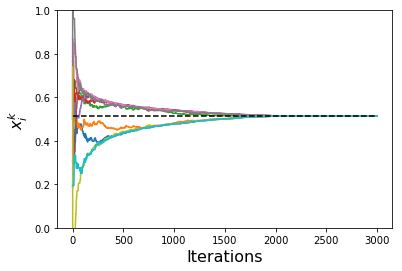}
  \caption{$\lambda^t=\frac{1}{t}$}
\end{subfigure}
\begin{subfigure}{.3\textwidth}
  \centering
  \includegraphics[width=1\linewidth]{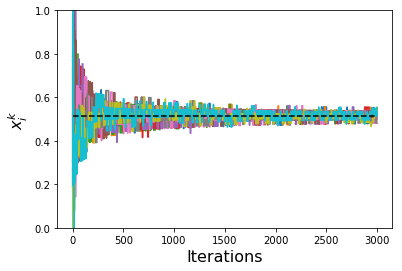}
  \caption{ $\lambda^t=\frac{1}{\sqrt{t}}$}
\end{subfigure}
\begin{subfigure}{.3\textwidth}
  \centering
  \includegraphics[width=1\linewidth]{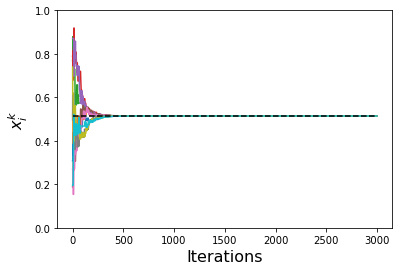}
  \caption{ $\lambda^t=\text{Adaptive}$}
\end{subfigure}
\caption{Trajectories of the values of $x_i^t$ for Binary Oracle run on the cycle graph.}
\label{Cycle10Lam}
\end{figure}

\begin{figure}[H]
\centering
\begin{subfigure}{.45\textwidth}
  \centering
  \includegraphics[width=1\linewidth]{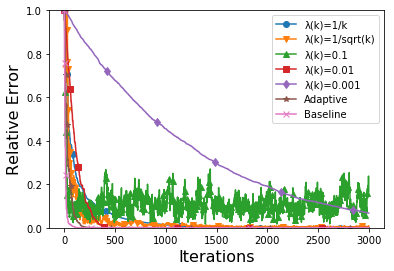}
  \caption{Linear Scale}
\end{subfigure}%
\begin{subfigure}{.45\textwidth}
  \centering
  \includegraphics[width=1\linewidth]{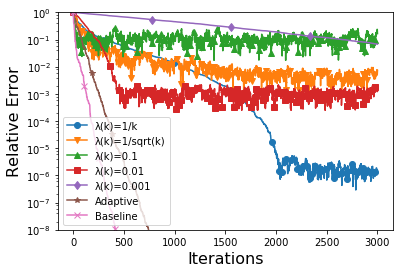}
  \caption{Logarithmic Scale}
\end{subfigure}
\caption{Convergence of the Binary Oracle run on the cycle graph.}
\label{Cycle10LamError}
\end{figure}
 
In general, we clearly see what is expected with the constant stepsizes --- that they converge to a certain neighbourhood and oscillate around optimum. With smaller stepsize, this neighbourhood is more accurate, but it takes longer to reach. With decreasing stepsizes, Theorem~\ref{thm:jhs988sh} suggests that $\lambda^t$ of order $1 / \sqrt{t}$ should be optimal. Figure~\ref{RGG100LamErr} demonstrates this, as the choice of $\lambda^t = 1/t$ decreases the stepsizes too quickly. However, this is not the case in Figure~\ref{Cycle10LamError} in which we observe the opposite effect. This is due to the cycle graph being small and simple, and hence the diminishing stepsize becomes a problem only after a relatively large number of iterations. With the adaptive choice of stepsizes, we recover the linear convergence rate as predicted by Theorem~\ref{thm: stepsize_adaptive}.

The results in Figure~\ref{RGG100LamErr} show one surprising comparison. The adaptive choice of stepsizes does not seem to perform better than $\lambda^t = 1 / \sqrt{t}$. However, we verified that when running for more iterations, the linear rate of adaptive stepsize is present and converges significantly faster to higher accuracies. We chose to present the results for $6000$ iterations since we found it overall cleaner.

\begin{figure}[H]
\centering
\begin{subfigure}{.3\textwidth}
  \centering
  \includegraphics[width=1\linewidth]{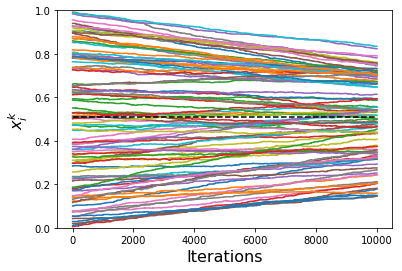}
  \caption{ $\lambda^t=\lambda=0.001$}
\end{subfigure}%
\begin{subfigure}{.3\textwidth}
  \centering
  \includegraphics[width=1\linewidth]{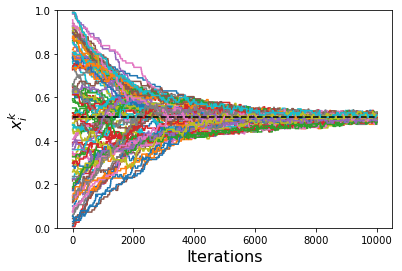}
  \caption{$\lambda^t=\lambda=0.01$}
\end{subfigure}
\begin{subfigure}{.3\textwidth}
  \centering
  \includegraphics[width=1\linewidth]{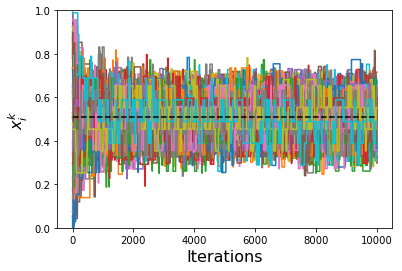}
  \caption{ $\lambda^t=\lambda=0.1$}
\end{subfigure}\\
\begin{subfigure}{.3\textwidth}
  \centering
  \includegraphics[width=1\linewidth]{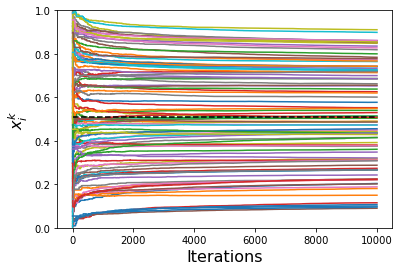}
  \caption{ $\lambda^t=\frac{1}{t}$}
\end{subfigure}
\begin{subfigure}{.3\textwidth}
  \centering
  \includegraphics[width=1\linewidth]{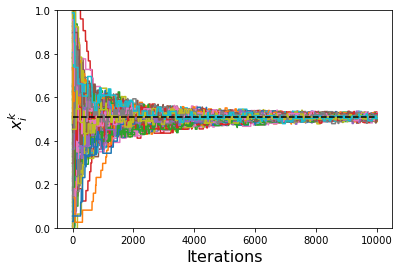}
  \caption{ $\lambda^t=\frac{1}{\sqrt{t}}$}
\end{subfigure}
\begin{subfigure}{.3\textwidth}
  \centering
  \includegraphics[width=1\linewidth]{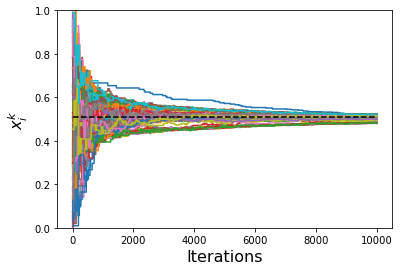}
  \caption{ $\lambda^t=\text{Adaptive}$}
\end{subfigure}
\caption{Trajectories of the values of $x_i^t$ for Binary Oracle run on the random geometric graph.}
\label{RGG100Lam}
\end{figure}

\begin{figure}[H]
\centering
\begin{subfigure}{.45\textwidth}
  \centering
  \includegraphics[width=1\linewidth]{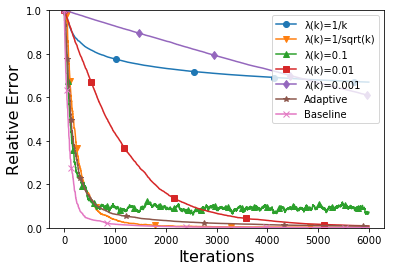}
  \caption{Linear Scale}
\end{subfigure}%
\begin{subfigure}{.45\textwidth}
  \centering
  \includegraphics[width=1\linewidth]{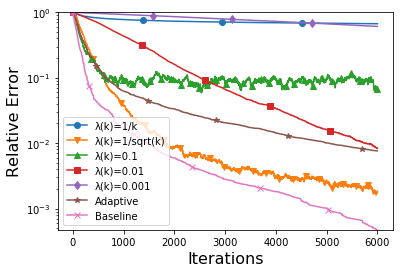}
  \caption{Logarithmic Scale}
\end{subfigure}
\caption{Convergence of the Binary Oracle run on the random geometric graph.}
\label{RGG100LamErr}
\end{figure}

\subsection{Private gossip via $\epsilon$-gap oracle}

In this section, we evaluate the performance of the Algorithm~\algE\ presented in Section~\ref{sec:E}. In the algorithm, the input parameter is the positive error tolerance variable $\epsilon$. For experimental evaluation. we choose three different values for the input, $\epsilon \in \{0.2, 0.02, 0.002\}$, and again use the same cycle and random geometric graphs. The trajectories of the values $x_i^t$ are presented in Figures~\ref{Cycle10Gap} and \ref{RGG100Gap}, respectively. The performance of the algorithm in terms of the relative error is presented in Figures~\ref{Cycle10GapErr} and \ref{RGG100GapErr}.

The performance is exactly matching the expectation --- with larger $\epsilon$, the method converges very fast to a wide neighbourhood of the optimum. For a small value, it converges much closer to the optimum, but it requires more iterations.

\begin{figure}[H]
\centering
\begin{subfigure}{.3\textwidth}
  \centering
  \includegraphics[width=1\linewidth]{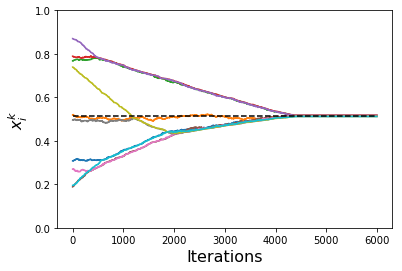}
  \caption{$\epsilon=0.002$}
\end{subfigure}
\begin{subfigure}{.3\textwidth}
  \centering
  \includegraphics[width=1\linewidth]{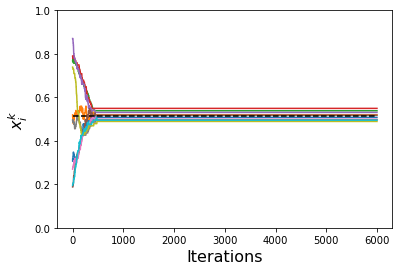}
  \caption{ $\epsilon=0.02$}
\end{subfigure}
\begin{subfigure}{.3\textwidth}
  \centering
  \includegraphics[width=1\linewidth]{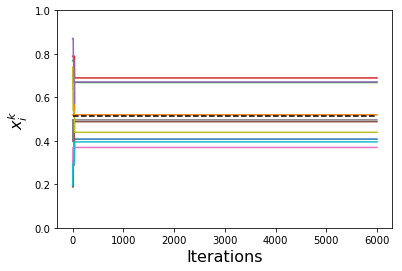}
  \caption{$\epsilon=0.2$}
\end{subfigure}
\caption{Trajectories of the values of $x_i^t$ for $\epsilon$-Gap Oracle run on the cycle graph.}
\label{Cycle10Gap}
\end{figure}

\begin{figure}[H]
\centering
\begin{subfigure}{.45\textwidth}
  \centering
  \includegraphics[width=1\linewidth]{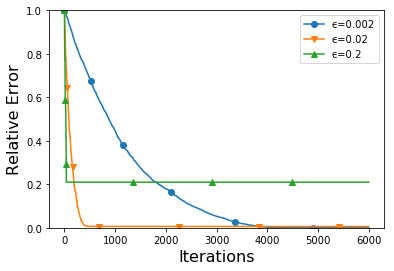}
  \caption{Linear Scale}
\end{subfigure}%
\begin{subfigure}{.45\textwidth}
  \centering
  \includegraphics[width=1\linewidth]{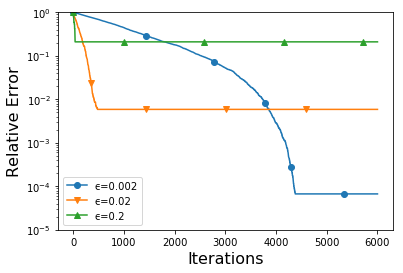}
  \caption{Logarithmic Scale}
\end{subfigure}
\caption{Convergence of the $\epsilon$-Gap Oracle run on the cycle graph.}
\label{Cycle10GapErr}
\end{figure}

\begin{figure}[H]
\centering
\begin{subfigure}{.3\textwidth}
  \centering
  \includegraphics[width=1\linewidth]{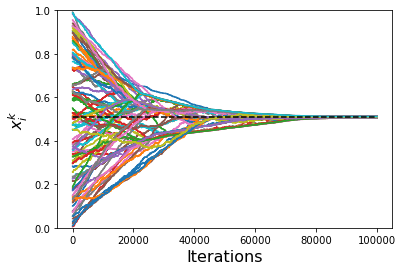}
  \caption{$\epsilon=0.002$}
\end{subfigure}
\begin{subfigure}{.3\textwidth}
  \centering
  \includegraphics[width=1\linewidth]{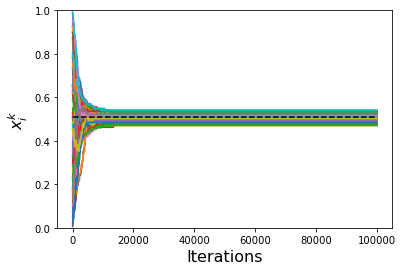}
  \caption{$\epsilon=0.02$}
\end{subfigure}
\begin{subfigure}{.3\textwidth}
  \centering
  \includegraphics[width=1\linewidth]{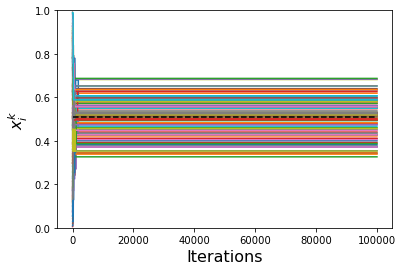}
  \caption{$\epsilon=0.2$}
\end{subfigure}
\caption{Trajectories of the values of $x_i^t$ for $\epsilon$-Gap Oracle run on the random geometric graph.}
\label{RGG100Gap}
\end{figure}

\begin{figure}[H]
\centering
\begin{subfigure}{.45\textwidth}
  \centering
  \includegraphics[width=1\linewidth]{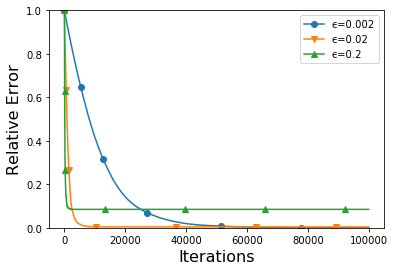}
  \caption{Linear Scale}
\end{subfigure}%
\begin{subfigure}{.45\textwidth}
  \centering
  \includegraphics[width=1\linewidth]{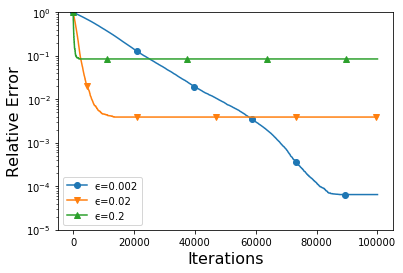}
  \caption{Logarithmic Scale}
\end{subfigure}
\caption{Convergence of the $\epsilon$-Gap Oracle run on the random geometric graph.}
\label{RGG100GapErr}
\end{figure}

\subsection{Private gossip via controlled noise insertion}

In this section, we evaluate the performance of Algorithm~\algN\ presented in Section~\ref{sec: noise}. This algorithm has two different parameters for each node $i$. These are the initial variance $\sigma_i^2 \geq 0$ and the rate of decay, $\phi_i$,  of the noise.

To evaluate the impact of these parameters, we perform several experiments. As earlier, we use the same graph structures for evaluation: cycle graph and random geometric graph. The algorithm converges with a linear rate depending on the minimum of two factors --- see Theorem~\ref{T: ng general convergence} and Corollary~\ref{C: noisy gossip special}. We will verify that this is indeed the case, and for values of $\phi_i$ above a certain threshold, the convergence is driven by the rate at which the noise decays. This is true for both identical values of $\phi_i$ for all $i$, and for varying values as per \eqref{Eq: phi_i def}. We further demonstrate the latter is superior in the sense that it enables insertion of more noise, without sacrificing the convergence speed. Finally, we study the effect of various magnitudes of the noise inserted initially.

\subsubsection{Fixed variance, identical decay rates}

In this part, we run Algorithm~\algN\ with $\sigma_i = 1$ for all $i$, and set $\phi_i = \phi$ for all $i$ and some $\phi$. We study the effect of varying the value of $\phi$ on the convergence of the algorithm. 

In both Figures~\ref{Cycle10NoiseErr}b and \ref{RGG100NoiseErr}b, we see that for small values of $\phi$, we eventually recover the same rate of linear convergence as the Standard Gossip algorithm. If the value of $\phi$ is sufficiently close to $1$ however, the rate is driven by the noise and not by the convergence of the Standard Gossip algorithm. This value is $\phi = 0.98$ for cycle graph, and $\phi=0.995$ for the random geometric graph in the plots we present.

Looking at the individual runs for small values of $\phi$ in Figure~\ref{RGG100NoiseErr}b, we see some variance in terms of when the asymptotic rate is realized. We would like to point out that this \emph{does not} provide additional insight into whether specific small values of $\phi$ are in general better for the following reason. The Standard Gossip algorithm is itself a randomized algorithm, with an inherent uncertainty in the convergence of any particular run. If we ran the algorithms multiple times, we observe variance in the evolution of the suboptimality of similar magnitude, just as what we see in the figure. Hence, the variance is expected, and not significantly influenced by the noise.

\begin{figure}[H]
\centering
\begin{subfigure}{.3\textwidth}
  \centering
  \includegraphics[width=1\linewidth]{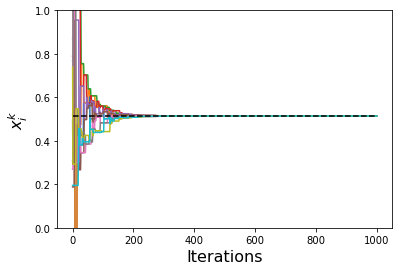}
  \caption{$\phi=0.001$}
\end{subfigure}
\begin{subfigure}{.3\textwidth}
  \centering
  \includegraphics[width=1\linewidth]{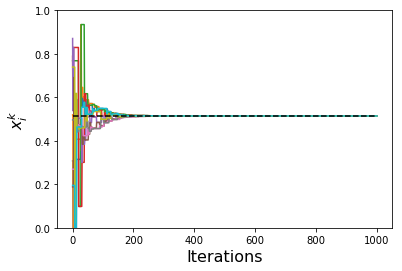}
  \caption{$\phi=0.01$}
\end{subfigure}
\begin{subfigure}{.3\textwidth}
  \centering
  \includegraphics[width=1\linewidth]{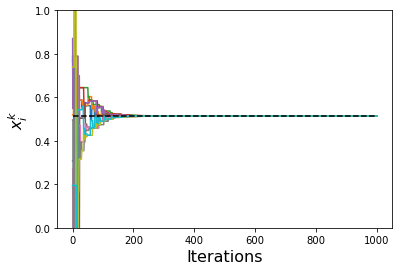}
  \caption{ $\phi=0.1$}
\end{subfigure}\\
\begin{subfigure}{.3\textwidth}
  \centering
  \includegraphics[width=1\linewidth]{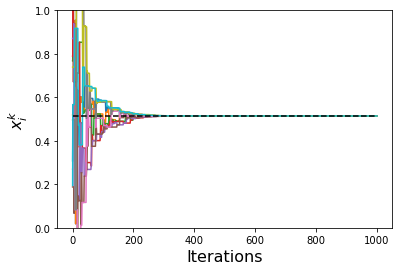}
  \caption{ $\phi=0.5$}
\end{subfigure}
\begin{subfigure}{.3\textwidth}
  \centering
  \includegraphics[width=1\linewidth]{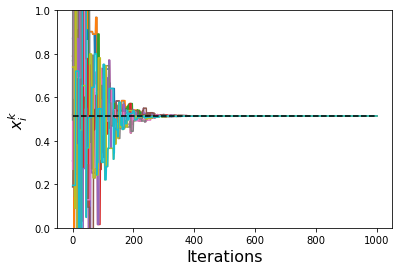}
  \caption{ $\phi=0.9$}
\end{subfigure}
\begin{subfigure}{.3\textwidth}
  \centering
  \includegraphics[width=1\linewidth]{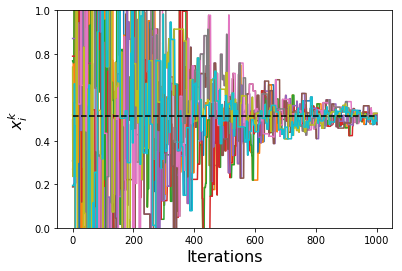}
  \caption{ $\phi=0.98$}
\end{subfigure}
\caption{Trajectories of the values of $x_i^t$ for Controlled Noise Insertion run on the cycle graph for different values of $\phi$.}
\label{Cycle10Noise}
\end{figure}

\begin{figure}[H]
\centering
\begin{subfigure}{.45\textwidth}
  \centering
  \includegraphics[width=1\linewidth]{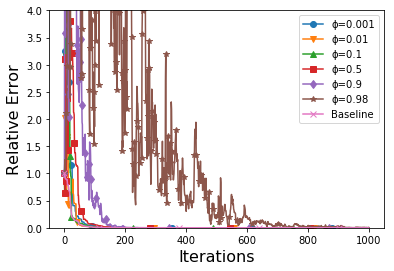}
  \caption{Linear Scale}
\end{subfigure}%
\begin{subfigure}{.45\textwidth}
  \centering
  \includegraphics[width=1\linewidth]{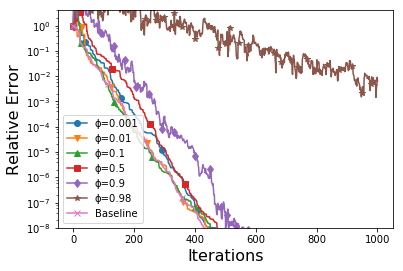}
  \caption{Logarithmic Scale}
\end{subfigure}
\caption{Convergence of the Controlled Noise Insertion run on the cycle graph for different values of $\phi$.}
\label{Cycle10NoiseErr}
\end{figure}

\begin{figure}[H]
\centering
\begin{subfigure}{.3\textwidth}
  \centering
  \includegraphics[width=1\linewidth]{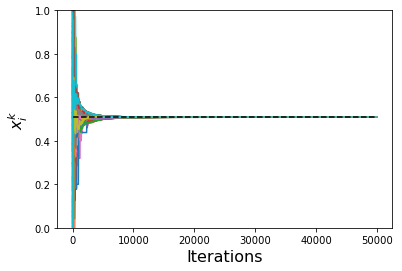}
  \caption{ $\phi=0.001$}
\end{subfigure}%
\begin{subfigure}{.3\textwidth}
  \centering
  \includegraphics[width=1\linewidth]{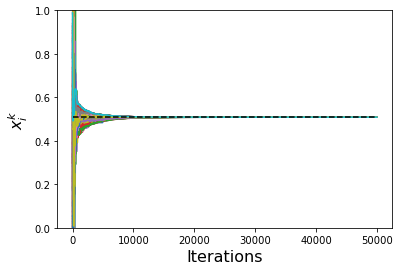}
  \caption{$\phi=0.01$}
\end{subfigure}
\begin{subfigure}{.3\textwidth}
  \centering
  \includegraphics[width=1\linewidth]{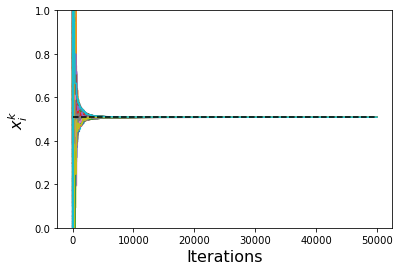}
  \caption{$\phi=0.1$}
\end{subfigure}\\
\begin{subfigure}{.3\textwidth}
  \centering
  \includegraphics[width=1\linewidth]{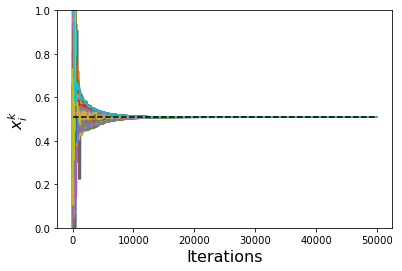}
  \caption{ $\phi=0.5$}
\end{subfigure}
\begin{subfigure}{.3\textwidth}
  \centering
  \includegraphics[width=1\linewidth]{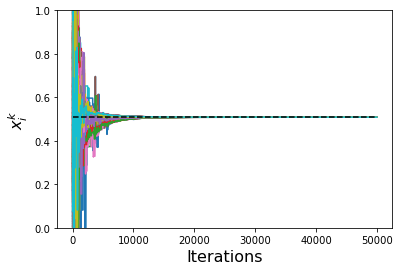}
  \caption{ $\phi=0.9$}
\end{subfigure}
\begin{subfigure}{.3\textwidth}
  \centering
  \includegraphics[width=1\linewidth]{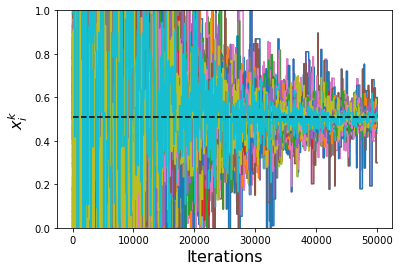}
  \caption{ $\phi=0.995$}
\end{subfigure}
\caption{Trajectories of the values of $x_i^t$ for Controlled Noise Insertion run on the random geometric graph for different values of $\phi$.}
\label{RGG100Noise}
\end{figure}

\begin{figure}[H]
\centering
\begin{subfigure}{.45\textwidth}
  \centering
  \includegraphics[width=1\linewidth]{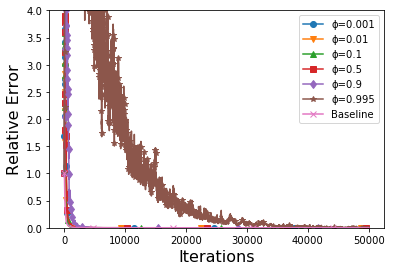}
  \caption{Linear Scale}
\end{subfigure}%
\begin{subfigure}{.45\textwidth}
  \centering
  \includegraphics[width=1\linewidth]{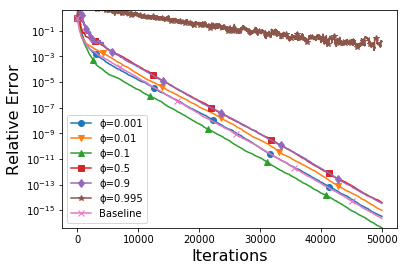}
  \caption{Logarithmic Scale}
\end{subfigure}
\caption{Convergence of the Controlled Noise Insertion run on the random geometric graph for different values of $\phi$.}
\label{RGG100NoiseErr}
\end{figure}

\subsubsection{Variance 1 and different decay rates}

In this section, we perform a similar experiment as above, but the values $\phi_i$ are not all the same. We rather control them by the choice of $\gamma$ as in \eqref{Eq: phi_i def}. Note that by decreasing $\gamma$, we increase $\phi_i$, and thus smaller $\gamma$ means the noise decays at a slower rate. Here, due to the regular structure of the cycle graph, we present only results for the random geometric graph.

It is not straightforward to compare this setting with the setting of identical $\phi_i$, and we return to it in the next section. Here we only remark that we again see the existence of a threshold predicted by theory, beyond which the convergence is dominated by the inserted noise. Otherwise, we recover the rate of the Standard Gossip algorithm.

\begin{figure}[H]
\centering
\begin{subfigure}{.3\textwidth}
  \centering
  \includegraphics[width=1\linewidth]{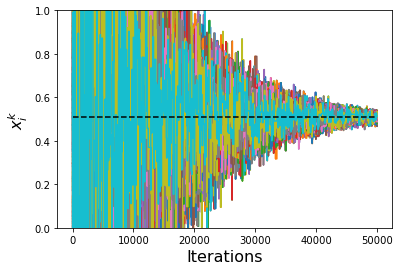}
  \caption{$\gamma=0.1$}
\end{subfigure}
\begin{subfigure}{.3\textwidth}
  \centering
  \includegraphics[width=1\linewidth]{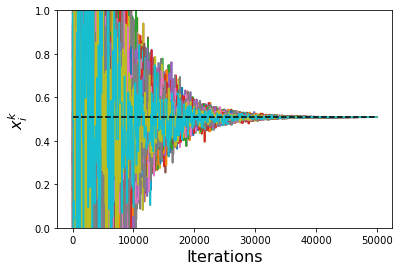}
  \caption{$\gamma=0.2$}
\end{subfigure}
\begin{subfigure}{.3\textwidth}
  \centering
  \includegraphics[width=1\linewidth]{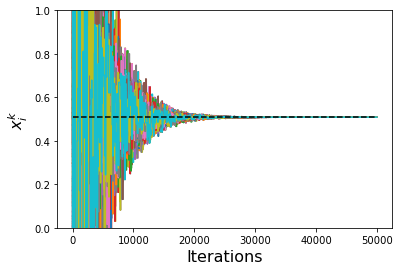}
  \caption{$\gamma=0.3$}
\end{subfigure}\\
\begin{subfigure}{.3\textwidth}
  \centering
  \includegraphics[width=1\linewidth]{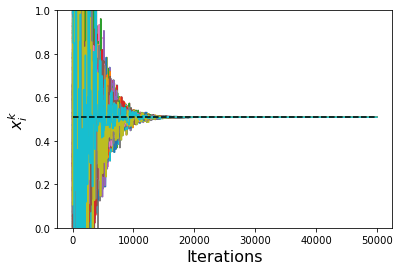}
  \caption{$\gamma=0.5$}
\end{subfigure}
\begin{subfigure}{.3\textwidth}
  \centering
  \includegraphics[width=1\linewidth]{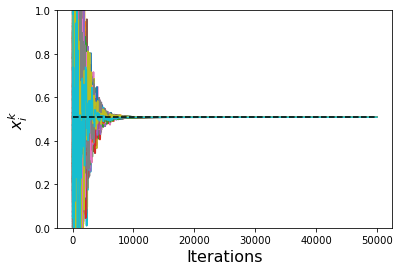}
  \caption{$\gamma=1$}
\end{subfigure}
\begin{subfigure}{.3\textwidth}
  \centering
  \includegraphics[width=1\linewidth]{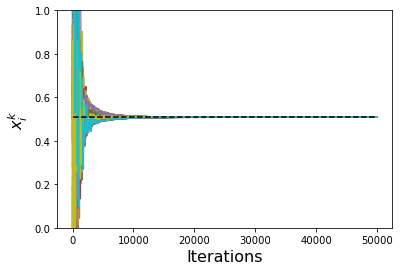}
  \caption{$\gamma=2$}
\end{subfigure}\\
\caption{Trajectories of the values of $x_i^t$ for Controlled Noise Insertion run on the random geometric graph for different values of $\phi_i$, controlled by $\gamma$.}
\label{Cycle10Gamma}
\end{figure}

\begin{figure}[H]
\centering
\begin{subfigure}{.45\textwidth}
  \centering
  \includegraphics[width=1\linewidth]{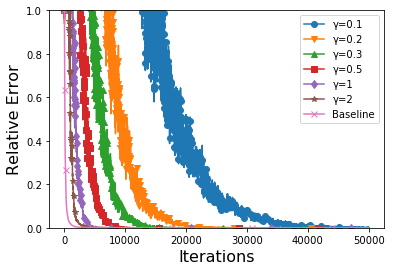}
  \caption{Normal Scale}
\end{subfigure}%
\begin{subfigure}{.45\textwidth}
  \centering
  \includegraphics[width=1\linewidth]{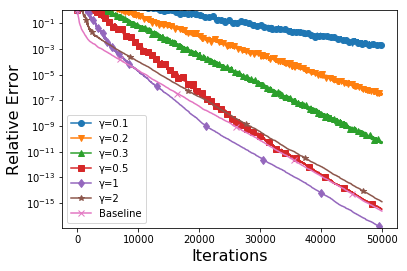}
  \caption{Logarithmic Scale}
\end{subfigure}
\caption{Convergence of the Controlled Noise Insertion run on the random geometric graph for different values of $\phi_i$, controlled by $\gamma$.}
\label{Cycle10GammaErr}
\end{figure}

\subsubsection{Impact of varying $\phi_i$}

In this experiment, we demonstrate the practical utility of letting the rate of decay $\phi_i$ to be different on each node $i$. In order to do so, we run the experiment on the random geometric graph and compare the settings investigated in the previous two sections --- the noise decay rate driven by $\phi$, or by $\gamma$.

In first place, we choose the values of $\phi_i$ such that that the two factors in Corollary~\ref{C: noisy gossip special} are equal. For the particular graph we used, this corresponds to $\gamma \approx 0.17$ with $\phi_i=\sqrt{1-\frac{\ac(\cG)}{2d_i}}$. Second, we make the factors equal, but with constraint of having $\phi_i$ to be equal for all $i$. This corresponds to $\phi_i \approx 0.983$ for all $i$.

The performance for a large number of iterations is displayed in the left side of Figure~\ref{fig:comparison_varying_phi}. We see that the above two choices indeed yield very similar practical performance, which also eventually matches the rate predicted by theory. For a complete comparison, we also include the performance of the Standard Gossip algorithm.

The important message is conveyed in the histogram in the right side of Figure~\ref{fig:comparison_varying_phi}. The histogram shows the distribution of the values of $\phi_i$ for different nodes $i$. The minimum of these values is what we needed in the case of identical $\phi_i$ for all $i$. However, most of the values are significantly higher. This means, that if we allow the noise decay rates to depend on the number of neighbours, we are able to increase the amount of noise inserted, without sacrificing practical performance. This is beneficial, as more noise will likely be beneficial for any formal notion of protection of the initial values.

\begin{figure}[H]
\centering
\begin{subfigure}{.45\textwidth}
  \centering
  \includegraphics[width=1\linewidth]{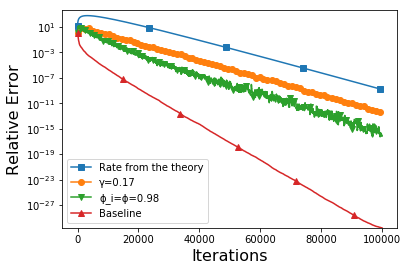}
\end{subfigure}%
\begin{subfigure}{.45\textwidth}
  \centering
  \includegraphics[width=1\linewidth]{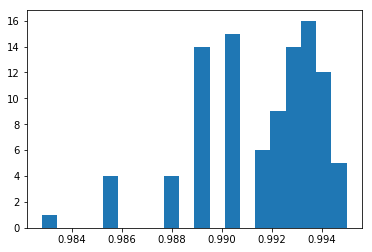}
\end{subfigure}

\caption{Left: Performance of the noise oracle with noise decrease rate chosen according to Corollary \ref{C: noisy gossip special}. Right: Histogram of of distribution of $\phi_i$}
\label{fig:comparison_varying_phi}
\end{figure}

\section{Conclusion}
\label{sec:conclusion}

In this chapter, we addressed the Average Consensus problem via novel asynchronous privacy preserving randomized gossip algorithms. In particular, we propose three different algorithmic tools for the protection of the initial private values of the nodes.

The first two proposed algorithms ``Private Gossip via Binary Oracle" and ``Private Gossip via  $\epsilon$-Gap Oracle" are based on the same idea of weakening the oracle used in the gossip update rule. In these two protocols the chosen pair of nodes of each gossip step instead of share their exact values they provide only categorical (or even binary) information to each other. 

In the third protocol ``Private Gossip via Controlled Noise Insertion", we systematically inject and withdraw noise throughout the iterations, so as to ensure convergence to the average consensus value and at the same time protect the private information of the nodes. 

In all cases, we provide explicit convergence rates and evaluate practical convergence on common simulated network topologies. 

Future work inludes the design of privacy preserving variants of several popular and fast gossip protocols  \cite{aysal2009broadcast, mou2010deterministic, he2011periodic, cao2006accelerated, liu2013analysis}.  One can also investigate more challenging types of consensus problems like the finite step consensus or consensus on
networks with time-varying topology, and design gossip protocols that preserve the privacy of the participating agents.  Designing the optimal network structure for information preservation is also an interesting research direction. 

As we have already mentioned the gossip algorithms of this chapter do not address any specific notion of privacy (no clear measure of privacy is presented) and it is still not clear how the formal concept of differential privacy \cite{dwork2014algorithmic} can be applied in protocols for solving the average consensus problem. Propose efficient differential privacy guarantees for gossip protocols in general graphs is an interesting open problem.


\section{Proofs of Main Results}
 \label{sec:App-Sec3}

\subsection{Proof of Lemma~\ref{lem:beta}}
\label{proofPrivacy1}
Let us first present a lemma that we use in the proof of Lemma~\ref{lem:beta}.
\begin{lem}
The eigenvalues of $\tilde{\mL}=n\mI - \ones \ones^\top$ are $\{0,n,n,\dots,n\}$
\label{L: eigenvalue of tilde L}
\end{lem}
\begin{proof}
Clearly, $\tilde{\mL} \ones=0$. Consider some vector $x$ such that $\langle x, \ones\rangle=0$. Then, $\tilde{\mL}x= n\mI x-\ones\ones^{\top} x=nx+\ones\ones^{\top} x=nx$ thus $x$ is an eigenvector corresponding to eigenvalue $n$. Thus, we can pick $n-1$ linearly independent eigenvectors of $\tilde{\mL}$ corresponding to eigenvalue $n$,  which concludes the proof. 
\end{proof} 
Having established the above lemma let us present the proof Lemma~\ref{lem:beta}.

The Laplacian matrix of $\cG$ is the matrix $\mL=\bA^\top \bA$. We have 
$\mL_{ii} = d_i$ (degree of vertex $i$), $\mL_{ij} = \bL_{ji} = -1$ if $(i,j)\in \cE$ and $\bL_{ij} = 0$ otherwise. A simple computation reveals that for any $x\in \R^n$ we have
\[x^\top \bL x = \sum_{e=(i,j)\in \cE} (x_i-x_j)^2.\]

Let $\tilde{\bA}$ be the $n(n-1)/2 \times n$  matrix corresponding to the complete graph $\tilde{\cG}$ on $\cV$. Let $\tilde{\bL} = \tilde{\bA}^\top \tilde{\bA}$ be its Laplacian. We have $\tilde{\bL}_{ii}=n-1$ for all $i$ and $\tilde{\bL}_{ij} =-1$ for $i\neq j$. So, $\tilde{\bL} = n \mI - \ones \ones^\top$.  Then
\[x^\top \tilde{\bL} x = n\|x\|^2 -\left(\sum_{i=1}^n x_i\right)^2=\sum_{(i,j)} (x_i - x_j)^2.\]

Inequality \eqref{eq:hdgugvej} can therefore be recast as follows:
\[x^\top (n \mI - \ones \ones^\top) x \leq  x^\top \beta(\cG) \bL x, \qquad x\in \R^n.\]

Let $\beta=\beta(\cG)$. Note that both $\tilde{\mL}$ and $\beta \mL$ are Hermitian thus have real eigenvalues and there exist an orthonormal basis of their eigenvectors. Suppose that $\{x_1,\dots x_n\}$ are eigenvectors of $\beta L$ corresponding to eigenvalues $\lambda_1(\beta \mL),\lambda_2(\beta \mL) \dots ,\lambda_n(\beta \mL)$. Without loss of generality assume that these eigenvectors form an orthonormal basis and $\lambda_1(\beta \mL)\geq \dots \geq \lambda_n(\beta \mL)$.

Clearly, $\lambda_n(\beta \mL)=0$, $x_n=\ones/\sqrt{n}$, and $\lambda_{n-1}(\beta \mL)=n$. Lemma~\ref{L: eigenvalue of tilde L}
states that eigenvalues of $\tilde{\mL}$ are $\{0,n,n,\dots,n\}$.

 One can easily see that eigenvector corresponding to zero eigenvalue of $\tilde{\mL}$ is $x_n$. Note that eigenvectors $x_1,\dots,x_{n-1}$ generate an eigenspace corresponding to eigenvalue $n$ of $\tilde{\mL}$. 

Consider some $x=\sum_{i=1}^n c_ix_i$, $c_i \in \R$ for all $i$. Then we have
$$
x^{\top}\tilde{\mL}x=
\sum_{i=1}^n \lambda_i\left(\tilde{\mL}\right)c_i^2\leq
\sum_{i=1}^n \lambda_i\left(\beta \mL \right)c_i^2=
 x^{\top}\beta \mL x,
$$
which concludes the proof. 
\subsection{Proof of Theorem~\ref{thm:G}}
\label{proofPrivacy2}
We first establish two lemmas which will be needed to prove Theorem~\ref{thm:G}.
\begin{lem} \label{lem:09u9djh9ffs}
Assume that edge $e=(i,j)$ is selected in iteration $t$ of Algorithm~\algG. Then \begin{equation} \label{eq:89g9s8guff} D(y^{t+1}) - D(y^t) = \frac{1}{4}(x^t_i-x^t_j)^2.\end{equation}
\end{lem}
\begin{proof} We have $y^{t+1} = y^t + \lambda^t f_e$ where $\lambda^t$ is chosen so that $D(y^{t+1})-D(y^t)$ is maximized. Applying Lemma~\ref{lem:98y98yss}, we have
\[
D(y^{t+1}) - D(y^t) =\max_{\lambda} -\lambda (x^t_i-x^t_j) - \lambda^2\\
=  \frac{1}{4}(x_i^t - x_j^t)^2.
\]
\end{proof}

\begin{lem} \label{L: D bound by beta} Let  $x\in \R^n$ such that $\frac{1}{n}\sum_i x_i = \bar{c}$. Then \begin{equation}
\label{eq:8998gd98gikd}\frac{1}{2}\|\bar{c}\ones - x\|^2 \leq \frac{1}{2\ac(\cG)} \sum_{e=(i,j)\in \cE} (x_i-x_j)^2.
\end{equation}
\end{lem}

\begin{proof}
\begin{eqnarray*}\frac{1}{2}\|\bar{c}\ones - x\|^2 
&\overset{\eqref{Eq: equality}}{=}& \frac{1}{4n} \sum_{i=1}^n \sum_{j=1}^n  (x_j  -x_i)^2
\quad = \quad \frac{1}{2n} \sum_{(i,j)} (x_j  -x_i)^2\\
&\overset{\eqref{eq:hdgugvej}}{\leq}& \frac{\beta(\cG)}{2n} \sum_{e=(i,j)\in \cE} (x_i-x_j)^2
\quad \overset{\text{Lemma}~\ref{lem:beta}}{=} \quad
\frac{1}{2\ac(\cG)} \sum_{e=(i,j)\in \cE} (x_i-x_j)^2
\end{eqnarray*}
\end{proof}

Having established Lemmas \ref{lem:09u9djh9ffs} and \ref{L: D bound by beta}, we can now proceed with the proof of Theorem~\ref{thm:G}:
\begin{eqnarray*}\E{D(y^*) - D(y^{t+1}) \;|\; y^t} &=& D(y^*) - D(y^{t}) - \E{D(y^{t+1}) - D(y^{t}) \;|\; y^t} 
\\
&\overset{\eqref{eq:89g9s8guff}}{=}&  D(y^*) - D(y^{t}) - \sum_{e=(i,j)\in \cE}\frac{1}{4m} (x^t_i-x^t_j)^2\\
&\overset{\eqref{eq:99d8gds}}{=}& \frac{1}{2}\|\bar{c}\ones - x^t\|^2- \sum_{e=(i,j)\in \cE}\frac{1}{4m} (x^t_i-x^t_j)^2\\
&\overset{\eqref{eq:8998gd98gikd}}{\leq} &  \left(1-\frac{\ac(\cG)}{2m }\right)\frac{1}{2}\|\bar{c}\ones - x^t\|^2\\
&\overset{\eqref{eq:99d8gds}}{=}& \left(1-\frac{\ac(\cG)}{2m}\right)\left(D(y^*) - D(y^{t}) \right) \ .
\end{eqnarray*}
Taking expectation again, we get the recursion
$$\E{D(y^*) - D(y^{t+1}) } \quad \leq \quad \left(1-\frac{\ac(\cG)}{2m}\right)\E{D(y^*) - D(y^{t}) }.$$
\subsection{Proof of Lemma \ref{L: rel measures}}
\label{meslemma}
Let us first present a Lemma that we use in the proof of Lemma~\ref{L: rel measures}.
\begin{lem}
\begin{equation}
\sum_{i=1}^n \left( \sum_{j=1}^n (x_j-x_i) \right)^2
=
\frac{n}{2}\sum_{i=1}^n \sum_{j=1}^n(x_j-x_i)^2 
\label{Eq: sum lemma}
\end{equation}
\end{lem}
\begin{proof}
Using simple algebra we have
\begin{eqnarray*}
\sum_{i=1}^n \left( \sum_{j=1}^n (x_j-x_i) \right)^2
&=&
\sum_{i=1}^n \left( \sum_{j=1}^n x_j-nx_i \right)^2
\\
&=&
\sum_{i=1}^n \left( \left( \sum_{j=1}^n x_j \right)^2 +n^2x_i^2-2nx_i \left( \sum_{j=1}^n x_j \right) \right)
\\
&=&
n\left(\sum_{j=1}^n x_j \right)^2
+
n^2\sum_{i=1}^nx_i^2
-2n\left(  \sum_{j=1}^n x_j \right)^2
\\
&=&
n^2\sum_{i=1}^nx_i^2
-n\left( \sum_{i=1}^n x_i \right)^2 .
\end{eqnarray*}
Manipulating right hand side of \eqref{Eq: sum lemma} we obtain
\begin{eqnarray*}
\frac{n}{2}\sum_{i=1}^n \sum_{j=1}^n(x_j-x_i)^2 &=&
\frac{n}{2}\sum_{i=1}^n \sum_{j=1}^n\left(x_j^2+x_i^2-2x_ix_j\right)
\\
&=& 
n^2\sum_{i=1}^nx_i^2 -n\sum_{i=1}^n \sum_{j=1}^nx_ix_j
\quad =\quad 
n^2\sum_{i=1}^nx_i^2
-n\left( \sum_{i=1}^n x_i \right)^2 .
\end{eqnarray*}
Clearly, LHS and RHS of \eqref{Eq: sum lemma} are equal.
\end{proof}
In order to show \eqref{Eq: equality} it is enough to notice that
\begin{eqnarray*}
\|\bar{c}\ones - x\|^2 
&=&  \sum_{i=1}^n (\bar{c}-x_i)^2 
\quad = \quad 
 \sum_{i=1}^n \left(\frac{1}{n}\sum_{j=1}^nx_j  -x_i\right)^2 
 \\
&=&
 \sum_{i=1}^n \left(\sum_{j=1}^n \frac{1}{n} (x_j  -x_i)\right)^2 
 \quad
\overset{\eqref{Eq: sum lemma}}{=} \quad
 \frac{1}{2} \sum_{i=1}^n \sum_{j=1}^n \frac{1}{n} (x_j  -x_i)^2 
 \\
&=& \frac{1}{2n} \sum_{i=1}^n \sum_{j=1}^n  (x_j  -x_i)^2.
\end{eqnarray*}
Note that we have 
$$
\frac{1}{nm}\left(\sum_{e=(i,j)\in \cE}|x_i-x_j|\right)^2 \leq \frac{1}{n}\sum_{e\in \cE}(x_i-x_j)^2\leq \frac{1}{n}\sum_{(i,j)}(x_i-x_j)^2 \stackrel{\eqref{Eq: equality}}{=} \|\bar{c} \ones -x\|^2,
$$
which proves \eqref{Eq: s upper bound}.  On the other hand, we have
$$
\frac{1}{\ac(\cG)}  \left(\sum_{e=(i,j)\in \cE}|x_i-x_j|\right)^2
\geq
\frac{1}{\ac(\cG)} \sum_{e\in \cE} \left(x_i-x_j \right)^2
\stackrel{\eqref{eq:8998gd98gikd}}{\geq}
\| \bar{c} \ones -x \|^2,
$$
which concludes \eqref{Eq: s lower bound}. Inequality \eqref{Eq: delta bound} holds trivially.  

\subsection{Proof of Theorem~\ref{thm:jhs988sh}}
\label{proofPrivacy3}
The following lemma is used in the proof of Theorem~\ref{thm:jhs988sh}.
\begin{lem} \label{L: optimal stepsizes binary}
 Fix $k\geq 0$ and let $R >0$. Then
\[\min_{\lambda = (\lambda^0,\dots,\lambda^k)\in \R^{k+1}} \frac{R + \beta^k}{\alpha^k} = 2\sqrt{\frac{R}{k+1}},\]
and the optimal solution is given by $\lambda^t = \sqrt{\frac{R}{k+1}}$ for all $t$.
\end{lem}
\begin{proof} Define $\phi(\lambda) = \frac{R+ \beta^k}{\alpha^k}$. If we write $\lambda = r x$, where $r=\|\lambda\|$ and $x$ is of unit norm, then $\phi(tx)=\frac{R+r^2}{r \langle \ones, x\rangle}$. Clearly, for any fixed $r$, the $x\in \R^{k+1}$  minimizing $x\mapsto \phi(rx)$ is $x=\ones/\|\ones\|$, where $\ones$ is the vector of ones in $\R^{k+1}$. It now only remains to minimize the function $r\mapsto \frac{R+r^2}{r \|\ones\|}$. This function is convex and differentiable. Setting the derivative to zero leads to $r = \sqrt{R}$. Combining the above, we get the optimal solution $\lambda = \frac{r }{\| \ones\|} \ones = \frac{\sqrt{R}}{\|  \ones \|} \ones$.
\end{proof}

Let $e=(i,j)$ be the edge selected at iteration $t\geq 0$. Applying Lemma~\ref{lem:98y98yss}, we see that
$
 D(y^{t+1}) - D(y^t)
= \lambda^t |x^t_i-x^t_j| - \left(\lambda^t\right)^2.
$ Taking expectation with respect to edge selection, we get 
\[\E{D(y^{t+1})-D(y^t) \;|\; y^t} = -\left(\lambda^t\right)^2 +  \lambda^t \cdot \frac{1}{m}\sum_{e=(i,j)\in \cE} |x^t_i - x^t_j|, \]
and taking expectation again and using the tower property, we get the identity
$$\E{D(y^{t+1})-D(y^t) } = -\left(\lambda^t\right)^2 +  \lambda^t\cdot \E{L^t}.$$
Therefore,
\begin{eqnarray*}
D(y^*) - D(y^0)& \geq & \E{D(y^{k+1})-D(y^0)} \\
&=& \E{\sum_{t=0}^{k}D(y^{t+1}) - D(y^t)} \\
&=& \sum_{t=0}^{k}\E{D(y^{t+1}) - D(y^t)} \quad = \quad -\sum_{t=0}^{k}\left(\lambda^t\right)^2 + \sum_{t=0}^{k}  \lambda^t \cdot \E{L^t}. \end{eqnarray*}
It remains to reshuffle the resulting inequality to obtain \eqref{Eq: binary general rate}.

We can see that part (i) follows directly. Optimality of stepsizes in (ii) is due to Lemma~\ref{L: optimal stepsizes binary}. To show (iii) we should state that
$$
\alpha^k=\sum_{t=0}^k \lambda^k=\sum_{1}^{k+1}\frac{a}{\sqrt{t}}\geq a \int_{t=1}^{k+2}t^{-1/2} dt=2a\left( \sqrt{k+2} -1\right)
$$
$$
\beta^k=\sum_{t=0}^k \left(\lambda^k\right)^2=\sum_{t=1}^{k+1}\frac{a^2}{t}\leq a^2 \int_{1/2}^{k+3/2}t^{-1} dt=a^2 \left(\log(k+3/2)+\log(2) \right) 
$$
The inequality above holds due to the fact that for $t>1/2$ we have $t^{-1}\leq \int_{t-1/2}^{t+1/2}x^{-1}dx $ since $x^{-1}$ is convex function. 
\subsection{Proof of Theorem \ref{thm: stepsize_adaptive}}
\label{proofPrivacy4}
Using Lemma~\ref{lem:98y98yss} with we have
\begin{eqnarray*}
\E{D(y^{t+1})-D(y^t)\;|\; y^t}&=&-\left(\lambda^t\right)^2+\lambda^t \frac{1}{m}\sum_{e\in \cE}|x_i^t-x_j^t|
\\
&=&
\frac{1}{4m^2}\left(\sum_{e\in \cE}|x_i^t-x_j^t|\right)^2
\quad \geq \quad \frac{1}{4m^2}\sum_{e\in \cE}\left(x_i^t-x_j^t\right)^2.
\end{eqnarray*}
Taking the expectation again we obtain
\begin{equation} \label{eq: adaptive binary first}
\E{D(y^{t+1})-D(y^t)}
 \geq \frac{1}{4m^2}\E{\sum_{e\in \cE}\left(x_i^t-x_j^t\right)^2}.
\end{equation}
On the other hand, we have
\begin{eqnarray*}
D(y^{t+1})-D(y^t) &=& \left (D(y^{t+1})-D(y^*) \right) +\left(D(y^*)-D(y^t) \right)\\
&=&
\frac12 \|\overline{c}\ones-x^{t} \|^2
-
\frac12 \|\overline{c}\ones-x^{t+1} \|^2 
\\
&=&
\frac{\ac(\cG)}{4m^2}\|\overline{c}\ones-x^{t} \|^2+
\left(1-\frac{\ac(\cG)}{2m^2}\right)\frac12 \|\overline{c}\ones-x^{t} \|^2
\\ &&  -  \quad
\frac12 \|\overline{c}\ones-x^{t+1} \|^2 
\\
&\stackrel{\eqref{eq:8998gd98gikd}}{\leq}&
\frac{1}{4m^2}\sum_{e=(i,j)\in \cE}\left(x_i^t-x_j^t\right)^2+
\left(1-\frac{\ac(\cG)}{2m^2}\right)\frac12 \|\overline{c}\ones-x^{t} \|^2
\\ &&  - \quad
\frac12 \|\overline{c}\ones -x^{t+1} \|^2 .
\end{eqnarray*}
Taking the expectation of the above and combining with \eqref{eq: adaptive binary first} we obtain the desired recursion
$$
\E{\|\overline{c}\ones-x^{t+1} \|^2} 
\quad \leq \quad
\left(1-\frac{\ac(\cG)}{2m^2}\right)
\E{\|\overline{c}\ones-x^{t} \|^2}
.$$

\subsection{Proof of Lemma~\ref{lem:09ys09y9ss}}
\label{proofPrivacy5}
Let $e=(i,j)$ be the edge selected at iteration $t$. Applying Lemma~\ref{lem:98y98yss}, we see that
\[
 D(y^{t+1}) - D(y^t)
=\begin{cases} -\frac{ \epsilon}{2} (x^t_i-x^t_j) - \frac{\epsilon^2}{4} , &\quad  x^t_i-x^t_j \leq -\epsilon\\
\frac{ \epsilon}{2} (x^t_i-x^t_j) - \frac{\epsilon^2}{4}, & \quad x^t_j-x^t_i \leq -\epsilon,\\
0, & \quad \text{otherwise.}
\end{cases}
\]

This implies that \[D(y^{t+1})-D(y^t) \begin{cases}\geq \frac{\epsilon^2}{4}, & \qquad \text{if}\;  \Delta^t_e = 1,\\
 = 0, & \qquad \text{if}\; \Delta^t_e = 0.
\end{cases} \]

Taking expectation in the selection of $e$, we get 
\[\E{D(y^{t+1})-D(y^t) \;|\; y^t} \geq \frac{\epsilon^2}{4} \cdot \Prob(\Delta^t_e = 1 \;|\; y^t) + 0  \cdot \Prob(\Delta^t_e = 0 \;|\; y^t)=\frac{\epsilon^2}{4} \Delta^t.\]
It remains to take expectation again.
\subsection{Proof of Theorem~\ref{thm:09y09s9ffs}}
\label{proofPrivacy6}
Since for all $k\geq 0$ we have $D(y^k) \leq D(y^*)$, it follows that \[D(y^*) - D(y^0) \geq \E{D(y^k)-D(y^0)} = \E{\sum_{t=0}^{k-1}D(y^{t+1}) - D(y^t)}= \sum_{t=0}^{k-1}\E{D(y^{t+1}) - D(y^t)}.\]
It remains to apply Lemma~\ref{lem:09ys09y9ss}.

\subsection{Proof of Lemma~\ref{L: noise exp iteration bound}}
\label{proofPrivacy7}
Let us first present three lemmas that we use in the proof of Lemma~\ref{L: noise exp iteration bound}.
\begin{lem}\label{Lm: independence}
Suppose that we run Algorithm \algN\ for $t$ iterations and $t_i$ denotes the number of times that some edge corresponding to node $i$ was selected during the algorithm.  
\begin{enumerate}
\item  $v_i^{t_i}$ and $t_j$ are independent  for all (i.e., not necessarily distinct) $i,j$.
\item $v_i^{t_i}$ and $\phi_j^{t_j}$ are independent for all (i.e., not necessarily distinct)  $i,j$.
\item $w_i^{t_i}$ and $w_j^{t_j}$ have zero correlation for all $i\neq j$.
\item $x_j^t$ and $\phi_i^{t_i}v_i^{t_i}$ have zero correlation for all (i.e., not necessarily distinct) $i,j$.
\end{enumerate}
\end{lem}
\begin{proof}
\begin{enumerate}
\item Follows from the definition of $v_i^{t}$.
\item Follows from the definition of $v_i^{t}$.
\item Note that we have $w_i^{t_i} = \phi_i^{t_i}v_i^{t_i} - \phi_i^{t_i-1}v_i^{t_i-1}$ and $w_j^{t_j} = \phi_j^{t_j}v_j^{t_j} - \phi_j^{t_j-1}v_j^{t_j-1}$. 
Clearly, $v_i^{t_i}$ and $w_j^{t_j}$ have zero correlation. Similarly $v_i^{t_i-1}$ and $w_j^{t_j}$ have zero correlation. Thus, $w_i^{t_i}$ and $w_j^{t_j}$ have zero correlation.
\item Clearly, $x_j^t$ is a function initial state and all instances of random variables up to the iteration $t$. Thus, $v_i^{t_i}$ is independent to $x_j^t$ from the definition. Thus, $x_j^t$ and $\phi_i^{t_i}v_i^{t_i}$ have zero correlation.
\end{enumerate}
\end{proof}
\begin{lem}\label{Lm: another_independence}
\begin{equation}
\E{\phi_i^{t_i-1}v_i^{t_i-1}x_i^t} = \frac12\E{\left(\phi_i^{t_i-1}v_i^{t_i-1}\right)^2}.
\end{equation}
\end{lem}
\begin{proof}
\begin{eqnarray*}
&&
\E{\phi_i^{t_i-1}v_i^{t_i-1}x_i^t}
\\
&&
=
 \E{\phi_i^{t_i-1}v_i^{t_i-1}\left(\left(x_i^t - \frac{\phi_i^{t_i-1}v_i^{t_i-1}}{2}\right)+\frac{\phi_i^{t_i-1}v_i^{t_i-1}}{2}\right)} 
 \\
&&
=
\E{\phi_i^{t_i-1}v_i^{t_i-1}\left(x_i^t - \frac{\phi_i^{t_i-1}v_i^{t_i-1}}{2}\right)}+\frac12\E{\left(\phi_i^{t_i-1}v_i^{t_i-1}\right)^2} 
\\
&&
 \stackrel{(*)}{=}
 \E{\phi_i^{t_i-1}v_i^{t_i-1}\left(\frac{x_i^{t_i-1} + x_l^{t_l^0} +w^{t_i-1}+w^{t_l^0} -\phi_i^{t_i-1}v_i^{t_i-1}}{2}\right)}
 +
 \frac12\E{\left(\phi_i^{t_i-1}v_i^{t_i-1}\right)^2}
 \\
&&
\stackrel{\eqref{thewdefinition}}{=}
\E{\phi_i^{t_i-1}v_i^{t_i-1}\left(\frac{x_i^{t_i-1} + x_l^{t_l^0} +\phi_i^{t_i-1}v_i^{t_i-1}- \phi_i^{t_i-2}v_i^{t_i-2}}{2}\right)}
\\
&&
 \qquad +
\E{\phi_i^{t_i-1}v_i^{t_i-1}\left(\frac{\phi_i^{t_l^0}v_l^{t_l^0} - \phi_i^{t_l^0-1}v_l^{t_l^0-1} -\phi_i^{t_i-1}v_i^{t_i-1}}{2}\right)}
\\
&&
 \qquad +
\frac12\E{\left(\phi_i^{t_i-1}v_i^{t_i-1}\right)^2}
\\
&&
=
\E{\phi_i^{t_i-1}v_i^{t_i-1}\left(\frac{x_i^{t_i-1} + x_l^{t_l^0} +\phi_i^{t_l^0}v_l^{t_l^0}- \phi_i^{t_l^0-1}v_l^{t_l^0-1}-\phi_i^{t_i-2}v_i^{t_i-2}}{2}\right)}
\\
&&
 \qquad 
 +\frac12\E{\left(\phi_i^{t_i-1}v_i^{t_i-1}\right)^2}
\\
&&
\stackrel{L.\ref{Lm: independence}}{=}\cancelto{0}{\E{\phi_i^{t_i-1}v_i^{t_i-1}}}\E{\left(\frac{x_i^{t_i-1} + x_l^{t_l^0} +\phi_i^{t_l^0}v_l^{t_l^0}-\phi_i^{t_l^0-1}v_l^{t_l^0-1}-\phi_i^{t_i-2}v_i^{t_i-2}}{2}\right)}
\\
&&
 \qquad +
\frac12\E{\left(\phi_i^{t_i-1}v_i^{t_i-1}\right)^2} 
\\
&&
=
\frac12\E{\left(\phi_i^{t_i-1}v_i^{t_i-1}\right)^2},
\end{eqnarray*}
where in the first equality we add and subtracting $\frac{\phi_i^{t_i-1}v_i^{t_i-1}}{2}$. In step $(*)$ we denote by $l$ a node such that that the noise $\phi_i^{t_i-1}v_i^{t_i-1}$ was added to the system when the edge $(i,l)$ was chosen (we do not consider $t_i=0$ since in this case the Lemma \ref{Lm: another_independence} trivially holds).
\end{proof}
\begin{lem}\label{Lm: squared_expetation}
\begin{equation}
\E{
\left(\phi_i^{t_i}v_i^{t_i}+
\phi_j^{t_j}v_j^{t_j}\right)^2
|x^t,e^t
} = \sigma^2_i\phi_i^{2t_i}+\sigma^2_j\phi_j^{2t_j}.
\label{Eq: squared expectation}
\end{equation}
\end{lem}
\begin{proof}
Since we have $\E{\left(\phi_i^{t_i}v_i^{t_i}+
\phi_j^{t_j}v_j^{t_j}\right)|x^t,e^t} = 0$, and also for any random variable $X$: $\E{X^2} = \Var{\left(X\right)}+\E{X}^2$, we only need to compute the variance:
\begin{eqnarray*}
\Var\left(\phi_i^{t_i}v_i^{t_i}+
\phi_j^{t_j}v_j^{t_j}\right)
 = 
 \Var{\left(\phi_i^{t_i}v_i^{t_i}\right)}+\Var{\left(
\phi_j^{t_j}v_j^{t_j}\right)} 
=
\left(\phi_i^{t_i}\right)^2\Var{\left(v_i^{t_i}\right)}+\left(\phi_j^{t_j}\right)^2\Var{\left(
v_j^{t_j}\right)} .
\end{eqnarray*}
\end{proof}

Having presented the above lemmas we can now proceed with the proof of Lemma~\ref{L: noise exp iteration bound}.

Firstly,  let us compute the increase of the dual function value at iteration $t$:
\begin{eqnarray}
D(y^{t+1})-D(y^t)&=&\frac12 \|\overline{c}\ones-x^t \|^2-\frac12 \|\overline{c}\ones-x^{t+1} \|^2
\nonumber
\\
&=&
\frac12 \left( 
(\overline{c}-x^t_j)^2+(\overline{c}-x^t_i)^2-(\overline{c}-x^{t+1}_j)^2-(\overline{c}-x^{t+1}_i)^2
\right)
\nonumber
\\
&=&
-\overline{c}\left(
 x^{t+1}_j+x^{t+1}_i-x^{t}_j-x^{t}_i
\right)
+\frac12 \left(
\left(x^{t}_j\right)^2+\left(x^{t}_i\right)^2
-\left(x^{t+1}_j\right)^2-\left(x^{t+1}_i\right)^2
\right)
\nonumber
\\
&=&
-\overline{c}\left( w_j^{t_j}+w_i^{t_i} \right)
+\frac12 
\left(
\left(x^{t}_j\right)^2+\left(x^{t}_i\right)^2
\right)
\nonumber
\\
&& \qquad 
-\frac14 \left( 
\left(x_j^t+x_i^t  \right)^2
+ 
2\left(x_j^t+x_i^t \right)\left( w_j^{t_j}+w_i^{t_i} \right)
+\left( w_j^{t_j}+w_i^{t_i} \right)^2
\right)
\nonumber
\\
&=&
\frac14 \left(x_j^t-x_i^t\right)^2-
\left(w_i^{t_i}+w_j^{t_j}\right)
\left(    
\overline{c}+\frac12\left(x_j^t+x_i^t\right)
\right)
-\frac14\left(w_i^{t_i}+w_j^{t_j}\right)^2.
\label{Eq: noise_no_exp}
\end{eqnarray}
Our goal is to estimate an upper bound of the quantity $\E{D(y^{t+1}) - D(y^t)}$. There are three terms in \eqref{Eq: noise_no_exp}. Since the expectation is linear, we will evaluate the expectations of these three terms separately and merge them at the end. 

Taking the expectation over the choice of edge and inserted noise in iteration $t$ we obtain
\begin{equation}
\E{\frac14\left(x_i^t - x_j^t\right)^2|x^t} = \frac{1}{4m}\sum\limits_{e\in \cE}\left(x_i^t-x_j^t\right)^2.\label{Eq: first_term}
\end{equation}
Thus we have
\begin{eqnarray*}
&&
\E{D(y^{t+1})-D(y^{t}) - \frac{1}{4}\left(x_i^t-x_j^t\right)^2 \;|\; x^t}
\\
&&
\qquad
\stackrel{\eqref{Eq: first_term}}{=}
\E{D(y^{t+1})-D(y^{t}) \;|\; x^t} - \frac{1}{4m}\sum\limits_{e\in \cE}\left(x_i^t-x_j^t\right)^2
\\
&&
\qquad
=
\E{D(y^*)-D(y^{t})\;|\; x^t} +\E{D(y^{t+1})-D(y^*)\;|\; x^t}  - \frac{1}{4m}\sum\limits_{e\in \cE}\left(x_i^t-x_j^t\right)^2
\\
&&
\qquad
\stackrel{\eqref{eq:99d8gds}}{=}
\frac12\|\overline{c}\ones-x^t \|^2
-\E{\frac12\|\overline{c}\ones-x^{t+1} \|^2 \;|\; x^t}
- \frac{1}{4m}\sum\limits_{e=(i,j)\in \cE}\left(x_i^t-x_j^t\right)^2
\\
&&
\qquad
\stackrel{\eqref{eq:8998gd98gikd}}{\leq}
\frac12\|\overline{c}\ones-x^t \|^2
-\E{\frac12\|\overline{c}\ones-x^{t+1} \|^2 \;|\; x^t}
- \frac{\ac(\cG)}{4m}\|\overline{c}\ones-x^t \|^2
\\
&&
\qquad
=
\left(1-\frac{\ac(\cG)}{2m}\right)\frac12\|\overline{c}\ones-x^t \|^2
-\E{\frac12\|\overline{c}\ones-x^{t+1} \|^2 \;|\; x^t}
\\
&&
\qquad
\stackrel{\eqref{eq:99d8gds}}{=}
\left(1-\frac{\ac(\cG)}{2m}\right)\left( D(y^*)- D(y^{t}) \right)
-\E{D(y^*)- D(y^{t+1})  \;|\; y^t}.
\end{eqnarray*}
Taking the full expectation of the above and using tower property, we get
\begin{equation}
\label{Eq: noisy gossip first full}
\E{D(y^{t+1})-D(y^{t}) - \frac{1}{4}\left(x_i^t-x_j^t\right)^2}
\leq
\left( 1-\frac{\ac(\cG)}{2m}\right)\E{D(y^*)- D(y^{t}) }
-\E{ D(y^*)- D(y^{t+1}) }.
\end{equation}

Now we are going to take the expectation of the second term of \eqref{Eq: noise_no_exp}. We will use the ``tower rule" of expectations  in the form
\begin{equation}
\label{TowerPropertyGossip}
\E{\E{\E{X\;| \; Y,Z}\;|\;Y}} = \E{X}
\end{equation}
where $X,Y,Z$ are random variables. In particular, we get
\begin{eqnarray}
\label{naoksxla}
&&\E{-\left(w_i^{t_i}+w_j^{t_j}\right)
\left(    
\overline{c}+\frac12\left(x_j^t+x_i^t\right)
\right)} \notag \\ && \qquad=\E{\E{\E{-\left(w_i^{t_i}+w_j^{t_j}\right)
\left(    
\overline{c}+\frac12\left(x_j^t+x_i^t\right)
\right) \;| \; e^t,x^t} \;| \; x^t}}.
\end{eqnarray}
In equation \eqref{naoksxla}, $e^t$ denotes the edge selected at iteration $t$.

Let us first calculate the inner most expectation of the right hand side of \eqref{naoksxla}:
\begin{eqnarray*}
&&
\E{-\left(w_i^{t_i}+w_j^{t_j}\right)
\left(    
\overline{c}+\frac12\left(x_j^t+x_i^t\right)
\right)\;|\; e^t,x^t}
\\
&&
\stackrel{\eqref{thewdefinition}}{=}
\E{\left(\phi_i^{t_i-1}v_i^{t_i-1}+
\phi_j^{t_j-1}v_j^{t_j-1}- \phi_i^{t_i}v_i^{t_i}-
\phi_j^{t_j}v_j^{t_j}
\right)\left(    
\overline{c}+\frac12\left(x_j^t+x_i^t\right)
\right)\;|\; e^t,x^t}
\\
&&
=
\E{ \overbrace{\left(\phi_i^{t_i-1}v_i^{t_i-1}+
\phi_j^{t_j-1}v_j^{t_j-1}
\right)}^{\text{constant}}
 \overbrace{\left(    
\overline{c}+\frac12\left(x_j^t+x_i^t
\right)
\right)}^{\text{constant}}\;|\; e^t,x^t}
\\
&&
\qquad 
+\E{\left(- \phi_i^{t_i}v_i^{t_i}-
\phi_j^{t_j}v_j^{t_j}
\right)
\overbrace{\left(    
\overline{c}+\frac12\left(x_j^t+x_i^t\right)
\right)}^{\text{constant}}\;|\; e^t,x^t}\\
&&
=
\left(\phi_i^{t_i-1}v_i^{t_i-1}+
\phi_j^{t_j-1}v_j^{t_j-1}
\right)\left(    
\overline{c}+\frac12\left(x_j^t+x_i^t\right)
\right)
\\
&&
\qquad
+\cancelto{0}{\E{\left(- \phi_i^{t_i}v_i^{t_i}-
\phi_j^{t_j}v_j^{t_j}
\right)|e,x^t}}\left(    
\overline{c}+\frac12\left(x_j^t+x_i^t\right)
\right)
\\
&&
\stackrel{L. \ref{Lm: independence}}{=}
\left(\phi_i^{t_i-1}v_i^{t_i-1}+
\phi_j^{t_j-1}v_j^{t_j-1}
\right)\left(    
\overline{c}+\frac12\left(x_j^t+x_i^t\right)
\right)
\\
&&
=
\left(\phi_i^{t_i-1}v_i^{t_i-1}+
\phi_j^{t_j-1}v_j^{t_j-1}
\right)   
\overline{c}
+\frac12\left(\phi_i^{t_i-1}v_i^{t_i-1}+
\phi_j^{t_j-1}v_j^{t_j-1}
\right)  \left(x_j^t+x_i^t
\right).
\end{eqnarray*}  

Now we take the expectation of the last expression above with respect to the choice of an edge at $t$-th iteration. We obtain
\begin{eqnarray*}
&&\E{\left(\phi_i^{t_i-1}v_i^{t_i-1}+
\phi_j^{t_j-1}v_j^{t_j-1}
\right)   
\overline{c}
+\frac12\left(\phi_i^{t_i-1}v_i^{t_i-1}+
\phi_j^{t_j-1}v_j^{t_j-1}
\right)  \left(x_j^t+x_i^t
\right)|x^t} 
\\
&&
=
\frac12
\E{\left(\phi_i^{t_i-1}v_i^{t_i-1}+
\phi_j^{t_j-1}v_j^{t_j-1}
\right)  \left(x_j^t+x_i^t\right) |x^t} \\
&&
\qquad
+ \E{\cancelto{0}{\left(\phi_i^{t_i-1}v_i^{t_i-1}+\phi_j^{t_j-1}v_j^{t_j-1}
\right)  \overline{c} |x^t}}
\\
&&
\stackrel{L. \ref{Lm: independence}}{=}
\frac12 \E{\left(\phi_i^{t_i-1}v_i^{t_i-1}+
\phi_j^{t_j-1}v_j^{t_j-1}
\right)  \left(x_j^t+x_i^t\right) |x^t}
\\
&&
=
\frac{1}{2m}\sum_{e\in \cE} \left(\phi_i^{t_i-1}v_i^{t_i-1}+
\phi_j^{t_j-1}v_j^{t_j-1}
\right)  \left(x_j^t+x_i^t\right) 
\\
&&
=
\frac{1}{2m}\sum_{e\in \cE} \left( \phi_i^{t_i-1}v_i^{t_i-1} x_i^t +\phi_j^{t_j-1}v_j^{t_j-1} x_j^t\right)
\\
&&
\qquad
+\frac{1}{2m} \sum_{e\in \cE} \left( \phi_i^{t_i-1}v_i^{t_i-1} x_j^t+ \phi_j^{t_j-1}v_j^{t_j-1} x_i^t
\right)
\\
&&
=
\frac{1}{2m}\sum_{i=1}^n d_i \phi_i^{t_i-1}v_i^{t_i-1} x_i^t
+\frac{1}{2m} \sum_{e\in \cE} \left( \phi_i^{t_i-1}v_i^{t_i-1} x_j^t+ \phi_j^{t_j-1}v_j^{t_j-1} x_i^t
\right),
\end{eqnarray*}
where in the last step we change the summation order.

Taking the expectation with respect to the algorithm we obtain
\begin{eqnarray}
&&
\E{-\left(w_i^{t_i}+w_j^{t_j}\right)
\left(   
\overline{c}+\frac12\left(x_j^t+x_i^t\right)
\right)} 
\nonumber
\\
&& 
\stackrel{\eqref{TowerPropertyGossip}}{=}\E{\E{\E{-\left(w_i^{t_i}+w_j^{t_j}\right)\left(   
\overline{c}+\frac12\left(x_j^t+x_i^t\right)
\right) |x^t,e}|x^t}}
\nonumber
\\
&& 
=
\E{
 \frac{1}{2m}\sum_{i=1}^n d_i \phi_i^{t_i-1}v_i^{t_i-1} x_i^t
+ \frac{1}{2m}\sum_{e\in \cE} \left( \phi_i^{t_i-1}v_i^{t_i-1} x_j^t+ \phi_j^{t_j-1}v_j^{t_j-1} x_i^t
\right)
}
\nonumber
\\
&& 
=
 \frac{1}{2m}\sum_{i=1}^n d_i \E{\phi_i^{t_i-1}v_i^{t_i-1} x_i^t}
+ \frac{1}{2m}\sum_{e\in \cE} \E{\left( \phi_i^{t_i-1}v_i^{t_i-1} x_j^t+ \phi_j^{t_j-1}v_j^{t_j-1} x_i^t
\right)
}
\nonumber
\\
&& 
\stackrel{L.\ref{Lm: another_independence}}{=}
 \frac{1}{4m}\sum_{i=1}^n d_i \E{\left(\phi_i^{t_i-1}v_i^{t_i-1}\right)^2 }
+ \frac{1}{2m}\sum_{e\in \cE} \E{\left( \phi_i^{t_i-1}v_i^{t_i-1} x_j^t+ \phi_j^{t_j-1}v_j^{t_j-1} x_i^t
\right)
}.
\label{Eq: noisy gossip second full}
\end{eqnarray}

Taking an expectation of the third term of \eqref{Eq: noise_no_exp} with respect to lastly added noise, the expression $\E{\left(w_i^{t_i}+w_j^{t_j}\right)^2 |x^t,e^t}$ is equal to
\begin{eqnarray*}
\E{\left(w_i^{t_i}+w_j^{t_j}\right)^2 |x^t,e^t}
&
\stackrel{\eqref{thewdefinition}}{=}&
\E{\left(
\phi_i^{t_i}v_i^{t_i}+
\phi_j^{t_j}v_j^{t_j}
-
\phi_i^{t_i-1}v_i^{t_i-1}-
\phi_j^{t_j-1}v_j^{t_j-1}
\right)^2 |x^t,e^t}
\\
&&  
=
\E{
\left(\phi_i^{t_i}v_i^{t_i}+
\phi_j^{t_j}v_j^{t_j}\right)^2
|x^t,e^t
}
\\
&&\quad -
2
\E{\overbrace{
\left(\phi_i^{t_i}v_i^{t_i}+\phi_j^{t_j}v_j^{t_j}\right)
}^{\E{\phi_i^{t_i}v_i^{t_i}+\phi_j^{t_j}v_j^{t_j}}=0}
\overbrace{
\left(\phi_i^{t_i-1}v_i^{t_i-1}+\phi_j^{t_j-1}v_j^{t_j-1}\right)
}^{\text{constant}}
|x^t,e^t
}
\\
&& \qquad+
 \overbrace{\E{
\left(    
\phi_i^{t_i-1}v_i^{t_i-1}+
\phi_j^{t_j-1}v_j^{t_j-1}
\right)^2
|x^t,e^t}}^{\text{constant}}
\\
&&
\stackrel{\eqref{Eq: squared expectation}}{=}
\sigma_i^2\phi_i^{2t_i}+\sigma^2_j\phi_j^{2t_j}+\left(
\phi_i^{t_i-1}v_i^{t_i-1}+
\phi_j^{t_j-1}v_j^{t_j-1}
\right)^2.
\end{eqnarray*}
Taking the expectation over $e^t$ we obtain:
\begin{eqnarray*}
\E{\left(w_i^{t_i}+w_j^{t_j}\right)^2 |x^t}&\stackrel{\eqref{TowerPropertyGossip}}{=}&\E{\E{\left(w_i^{t_i}+w_j^{t_j}\right)^2 |e^t, x^t}|x^t}
\\
&=&
\E{
\sigma_i^2\phi_i^{2t_i}+\sigma^2_j\phi_j^{2t_j}
+\left(
\phi_i^{t_i-1}v_i^{t_i-1}+
\phi_j^{t_j-1}v_j^{t_j-1}
\right)^2
|x^t}
\\
&=&
\frac{1}{m}\sum_{e \in \cE} \left( \sigma_i^2\phi_i^{2t_i}+\sigma^2_j\phi_j^{2t_j}\right)\\
&& \quad +
\frac1m \sum_{e\in \cE} \left(
\phi_i^{t_i-1}v_i^{t_i-1}+
\phi_j^{t_j-1}v_j^{t_j-1}
\right)^2
\\
&=&
\frac{1}{m}\sum_{i=1}^n d_i\sigma^2_i \phi_i^{2t_i}+
\frac1m \sum_{e\in \cE} \left(
\phi_i^{t_i-1}v_i^{t_i-1}+
\phi_j^{t_j-1}v_j^{t_j-1}
\right)^2.
\end{eqnarray*}
where in the last step we change the summation order.
 
Finally, taking the expectation with respect to the algorithm we get
\begin{eqnarray}
&&\E{\left(w_i^{t_i}+w_j^{t_j}\right)^2 }
\nonumber
\\
&&
\stackrel{\eqref{TowerPropertyGossip}}{=}
\E{\E{\left(w_i^{t_i}+w_j^{t_j}\right)^2|x^t }}
\nonumber
\\
&&
=
\frac{1}{m}\sum_{i=1}^n d_i\sigma^2_i \E{\phi_i^{2t_i}}+
\frac1m \sum_{e\in \cE}\E{ \left(
\phi_i^{t_i-1}v_i^{t_i-1}+
\phi_j^{t_j-1}v_j^{t_j-1}
\right)^2}
\nonumber
\\
&&
=
\frac{1}{m}\sum_{i=1}^n d_i\sigma^2_i \E{\phi_i^{2t_i}}
\nonumber
\\
&& \qquad+
\frac1m \sum_{e\in \cE}\E{ 
\left(\phi_i^{t_i-1}v_i^{t_i-1}\right)^2
+
\left(\phi_j^{t_j-1}v_j^{t_j-1}\right)^2
+ 2\phi_i^{t_i-1}v_i^{t_i-1}\phi_j^{t_j-1}v_j^{t_j-1}
}
\nonumber
\\
&&
\stackrel{(*)}{=}
\frac{1}{m}\sum_{i=1}^n d_i\sigma^2_i \E{\phi_i^{2t_i}}+
\frac1m \sum_{i=1}^n d_i\E{ 
\left(\phi_i^{t_i-1}v_i^{t_i-1}\right)^2}
+ \frac2m \sum_{e\in \cE}\E{\phi_i^{t_i-1}v_i^{t_i-1}\phi_j^{t_j-1}v_j^{t_j-1}
}
\nonumber
\\
\nonumber
&&
\stackrel{L.\ref{Lm: independence}}{=}
\frac{1}{m}\sum_{i=1}^n d_i\sigma^2_i \E{\phi_i^{2t_i}}+
\frac1m \sum_{i=1}^n d_i\E{ 
\left(\phi_i^{t_i-1}v_i^{t_i-1}\right)^2} 
\\
\nonumber
&& \qquad+ \frac2m \sum_{e\in \cE}\cancelto{0}{\E{\phi_i^{t_i-1}v_i^{t_i-1}}}\cancelto{0}{\E{\phi_j^{t_j-1}v_j^{t_j-1}}}
\nonumber
\\
&&
=
\frac{1}{m}\sum_{i=1}^n d_i\sigma^2_i \E{\phi_i^{2t_i}}+
\frac1m \sum_{i=1}^n d_i\E{ 
\left(\phi_i^{t_i-1}v_i^{t_i-1}\right)^2
\label{Eq: noisy gossip third full}
},
\end{eqnarray}
where in step (*) we change the summation order.

Combining \eqref{Eq: noise_no_exp} with \eqref{Eq: noisy gossip first full}, \eqref{Eq: noisy gossip second full} and \eqref{Eq: noisy gossip third full} we obtain
\begin{eqnarray*}
\E{ D(y^*)- D(y^{t+1}) } 
&\leq & 
\left( 1-\frac{\ac(\cG)}{2m}\right)\E{D(y^*)- D(y^{t})} 
-\frac{1}{4m}\sum_{i=1}^n d_i \E{\left(\phi_i^{t_i-1}v_i^{t_i-1}\right)^2 }
\\& &\qquad- \frac{1}{2m}\sum_{e\in \cE} \E{\left( \phi_i^{t_i-1}v_i^{t_i-1} x_j^t+ \phi_j^{t_j-1}v_j^{t_j-1} x_i^t
\right)}
\\
& & \qquad+\frac{1}{4m}\sum_{i=1}^n d_i \sigma^2_i \E{\phi_i^{2t_i}}+
\frac{1}{4m} \sum_{i=1}^n d_i\E{\left(\phi_i^{t_i-1}v_i^{t_i-1}\right)^2}
\\
&=&
\left( 1-\frac{\ac(\cG)}{2m}\right)\E{D(y^*)- D(y^{t})} +\frac{1}{4m}\sum_{i=1}^n d_i\sigma^2_i \E{\phi_i^{2t_i}}
\\& &\qquad-
\frac{1}{2m}\sum_{e\in \cE} \E{\left( \phi_i^{t_i-1}v_i^{t_i-1} x_j^t+ \phi_j^{t_j-1}v_j^{t_j-1} x_i^t
\right)},
\end{eqnarray*}
which concludes the proof.
\subsection{Proof of Theorem \ref{T: ng general convergence}}
\label{proofPrivacy8}
Let us present two lemmas that we use in the proof of Theorem~\ref{T: ng general convergence}.
\begin{lem}
After $t$ iterations of algorithm \algN\ we have
\begin{equation} \label{Eq: binomial}
\E{\phi_i^{2t_i}}=\left(1-\frac{d_i}{m} \left(1-\phi_i^2\right)\right)^{t}.
\end{equation}
\label{L: noisy binomial}
\end{lem}
\begin{proof}
\begin{eqnarray*}
\E{\phi_i^{2t_i}}&=&\sum_{j=0}^t \Prob(t_i=j)\phi_i^{2j}
\quad = \quad
\sum_{j=0}^t\binom{t}{j}\left(\frac{m-d_i}{m} \right)^{t-j}\left(\frac{d_i}{m}\phi_i^2\right)^j
\\
&=&
\left(\frac{m-d_i}{m} +\frac{d_i}{m}\phi_i^2\right)^{t}
\quad = \quad
\left(1-\frac{d_i}{m} \left(1-\phi_i^2\right)\right)^{t}.
\end{eqnarray*}
\end{proof}
\begin{lem}
Random variables $\phi_i^{t_i-1}v_i^{t_i-1} $ and $x_j^t$ are nonegatively correlated, i.e. 
\begin{equation}\label{Eq: positive corelation}
\E{\phi_i^{t_i-1}v_i^{t_i-1}x_j^t}\geq 0.
\end{equation}
\label{L: noisy positive correlation}
\end{lem}
\begin{proof}
Denote $R_{i,j}$ to be a random variable equal to 1 if the noise $w_i^{t_i}$ was added to the system when edge $(i,j)$ was chosen and equal to 0 otherwise. We can rewrite the expectation in the following way:
\begin{eqnarray*}
\E{\phi_i^{t_i-1}v_i^{t_i-1}x_j^t}&=&
\overbrace{\E{\phi_i^{t_i-1}v_i^{t_i-1}x_j^t \;|\; R_{i,j}=1}}^{\geq 0}\Prob(R_{i,j}=1)\\
\quad && +\overbrace{\E{\phi_i^{t_i-1}v_i^{t_i-1}x_j^t \;|\; R_{i,j}=0}}^{0}\Prob(R_{i,j}=0) \quad \geq \quad 0.
\end{eqnarray*}
The inequality  $
\E{\phi_i^{t_i-1}v_i^{t_i-1}x_j^t \;|\; R_{i,j}=1}\geq 0
$
holds due to the fact that $\phi_i^{t_i-1}v_i^{t_i-1}$ was added to $x_j$ with the positive sign.
\end{proof}
Combining \eqref{Eq: noise gossip iteration bound final} with the results of Lemmas \ref{L: noisy binomial} and \ref{L: noisy positive correlation} we obtain:
\begin{eqnarray*}
\E{ D(y^*)- D(y^{t+1}) } 
&\overset{\eqref{Eq: noise gossip iteration bound final}}{\leq}&
\left( 1-\frac{\ac(\cG)}{2m}\right)\E{D(y^*)- D(y^{t})} +\frac{1}{4m}\sum_{i=1}^n d_i \sigma^2_i \E{\phi_i^{2t_i}}
\\
& &\qquad-
\frac{1}{2m}\sum_{e\in \cE} \E{\left( \phi_i^{t_i-1}v_i^{t_i-1} x_j^t+ \phi_i^{t_j-1}v_j^{t_j-1} x_i^t
\right)}
\\
&\stackrel{\eqref{Eq: positive corelation}}{\leq} &
\left( 1-\frac{\ac(\cG)}{2m}\right)\E{D(y^*)- D(y^{t})} +\frac{1}{4m}\sum_{i=1}^n d_i\sigma^2_i  \E{\phi_i^{2t_i}}
\\
&\stackrel{\eqref{Eq: binomial}}{=} &
\left( 1-\frac{\ac(\cG)}{2m}\right)\E{D(y^*)- D(y^{t})} 
\\
&&\qquad+
\frac{1}{4m}\sum_{i=1}^n d_i\sigma^2_i  \left(1-\frac{d_i}{m}\left(1-\phi_i^2\right) \right)^{t}
\\
&=&
\left( 1-\frac{\ac(\cG)}{2m}\right)\E{D(y^*)- D(y^{t})} +\frac{\sum\left(d_i\sigma_i^2\right)}{4m}\psi^t.
\end{eqnarray*}

The recursion above gives us inductively the following
\begin{equation*}
\E{ D(y^*)- D(y^{k}) } \leq  \rho^k \left( D(y^*)- D(y^{0}) \right)  + \frac{\sum\left(d_i\sigma_i^2\right)}{4m}\sum_{t=1}^k \rho^{k-t}\psi^{t},
\end{equation*}
which concludes the proof of the theorem.
\subsection{Proof of Corollary \ref{C: noisy gossip special}}
\label{proofPrivacy9}
Note that we have
\begin{eqnarray*}
\psi^t
&=&\frac{1}{\sum_{i=1}^n d_i\sigma_i^2 }
\sum_{i=1}^n d_i \sigma^2_i \left(1-\frac{d_i}{m}\left(1-\left(1-\frac{\gamma}{d_i}\right)\right) \right)^{t}\\
&=&
\frac{1}{\sum_{i=1}^n d_i\sigma_i^2 }
\sum_{i=1}^n d_i \sigma^2_i\left(1-\frac{\gamma}{m}\right)^{t}
\quad = \quad \left(1-\frac{\gamma}{m}\right)^{t}.
\end{eqnarray*}

In view of Theorem~\ref{T: ng general convergence}, this gives us the following:
\begin{eqnarray*}
\E{ D(y^*)- D(y^{k}) } 
&\leq &
 \left( D(y^*)- D(y^{0}) \right) \rho^k + \frac{\sum\left(d_i\sigma_i^2\right)}{4m}\sum_{t=1}^k \rho^{k-t}\psi^{t}
\\
&\leq & 
\left( D(y^*)- D(y^{0}) \right) \rho^k + \frac{\sum\left(d_i\sigma_i^2\right)}{4m}\sum_{t=1}^k \rho^{k-t}\left(1-\frac{\gamma}{m}\right)^{t}
\\
&\leq &
\left( D(y^*)- D(y^{0}) \right) \rho^k + \frac{\sum\left(d_i\sigma_i^2\right)}{4m}k \max\left(\rho,1-\frac{\gamma}{m}\right)^k
\\
&\leq  &
\left( D(y^*)- D(y^{0})+\frac{\sum\left(d_i\sigma_i^2\right)}{4m}k \right) \max\left(\rho,1-\frac{\gamma}{m}\right)^k.
\end{eqnarray*}

\chapter{Conclusion and Future Work}
\label{ChapterConclusion}
\section{Conclusions}
In this thesis we studied the design and analysis of novel efficient randomized iterative methods for solving large scale linear systems, stochastic quadratic optimization problems, the best approximation problem and quadratic optimization problems. Using these methods we also proposed and analyzed efficient gossip protocols for solving the average consensus problem on large scale networks.

In Chapter~\ref{ChapterMomentum}, we studied the convergence of several stochastic optimization algorithms enriched with \textit{heavy ball momentum} for solving  stochastic optimization problems of special structure.  We proved global, non-asymptotic linear convergence rates of all of these methods as well as accelerated linear rate for the case of the norm of expected iterates. We also introduced a new momentum strategy called \textit{stochastic momentum} which is beneficial in the case of sparse data, and proved linear convergence in this setting. We corroborated our theoretical results with extensive experimental testing.

In Chapter~\ref{ChapterInexact}, we proposed and analyzed \textit{inexact} variants of several stochastic algorithms for solving quadratic optimization problems and linear systems. We provided linear convergence rate under several assumptions on the inexactness error. The proposed methods require more iterations than their exact variants to achieve the same accuracy. However, as we show through our numerical evaluations, the inexact algorithms require significantly less time to converge. 

In Chapter~\ref{ChapterGossip}, we presented a general framework for the analysis and design of \textit{randomized gossip} algorithms for solving the average consensus problem.  Using tools from numerical linear algebra and the area of randomized projection methods for solving linear systems, we proposed novel serial, block and accelerated gossip protocols for solving the average consensus and weighted average consensus problems. 

In Chapter~\ref{ChapterPrivacy}, we addressed the average consensus problem via novel asynchronous \textit{privacy preserving} randomized gossip algorithms. In particular, we proposed three different algorithmic tools for the protection of the initial private values of the nodes. The first two proposed algorithms ``Private Gossip via Binary Oracle" and ``Private Gossip via  $\epsilon$-Gap Oracle" are based on the same idea of weakening the oracle used in the gossip update rule. Instead of sharing their exact values, in these two protocols the chosen pair of nodes of each gossip step provide only categorical (or even binary) information to each other.  In the third protocol, ``Private Gossip via Controlled Noise Insertion", we systematically inject and withdraw noise throughout the iterations, so as to ensure convergence to the average consensus value, and at the same time protect the private information of the nodes. For all proposed protocols, we provide explicit convergence rates and evaluate practical convergence on common simulated network topologies.  

\section{Future Work}

Perhaps the most exciting direction for future work is to extend the analysis of the proposed randomized iterative methods to more general settings. In particular, the more natural extension of our results is the analysis of heavy ball momentum variants and inexact variants of the proposed methods (SGD, SN, SPP, SPM and SDSA) in the case of general convex or strongly convex functions. 

From numerical linear algebra viewpoint, we believe that the proposed randomized iterative methods of Chapters~\ref{ChapterMomentum} and~\ref{ChapterInexact} could have great potential to make a practical difference to iterative solvers for large-scale linear systems. In this aspect, a future effort needs to be devoted to the practical development and implementations of the algorithms. For example, one promising direction is to use new sophisticated sketching matrices $\bS$, such as the Walsh-Hadamard matrix \cite{pilanci2016iterative, lu2013faster} in the update rules of the proposed methods.

In this thesis we focused on algorithms with a fixed constant step-size. An interesting extension will be to study the effect of decreasing or adaptive choice for the relaxation parameter. This might provide novel insights, even in the case of quadratic functions and  (not necessarily consistent) linear systems. As we have mentioned in several parts of the thesis, the obtained results hold under the exactness condition, which as we explained, is very weak, allowing for virtually arbitrary distributions $\cD$ from which the random matrices are drawn. A different future direction will be the design of optimized distributions in order to improve further the convergence rates and the overall complexity of the proposed algorithms. 

In addition, we believe that the gossip protocols proposed in Chapters~\ref{ChapterGossip} and \ref{ChapterPrivacy} would be particularly useful in the development of efficient decentralized protocols.

Using the novel framework presented in this thesis, many popular projection methods can be interpreted as gossip algorithms when used to solve linear systems encoding the underlying network. This can lead to the development of novel distributed protocols for average consensus.  

Our work on gossip algorithms is amenable to further extensions. For instance, the proposed novel gossip protocols (block, accelerated, privacy-preserving) can be extended to the more general setting of multi-agent consensus optimization, where the goal is to minimize the average of convex or non-convex functions $(1/n) \sum_{i=1}^n f_i(x)$ in a decentralized way. Such protocols will be particularly useful in settings where the data describing a given optimization problem is so big that it becomes impossible to store it on a single machine. These situations often arise in modern machine learning and deep learning applications.

\bibliographystyle{plain}
\bibliography{ms}

\appendix
\chapter{Notation Glossary}
\section{Notation used in Chapters~\ref{ChapterIntroduction}, \ref{ChapterMomentum} and \ref{ChapterInexact}}
\label{ncasjkxasoa}

\begin{table}[!h]
\begin{center}
\begin{tabular}{|c|l|c|}
 \hline
 \multicolumn{2}{|c|}{{\bf The Basics}}\\
 \hline
$\mA, b$    & $m \times n$ matrix and $m\times 1$ vector defining the system $\mA x =b$\\
$\cL$ &  $\{x\;:\; \mA x = b\}$ (solution set of the linear system) \\
$\mB$    & $n \times n$ symmetric positive definite matrix \\
$\langle x, y \rangle_{\mB}$ & $x^\top \mB y$ ($\mB$-inner product) \\
$\|x\|_{\mB}$ & $\sqrt{\langle x, x \rangle_{\mB}}$ ($\mB$-norm)  \\
$\mS$    & a random real matrix with $m$ rows  \\
$\cD$    & distribution from which matrix $\mS$ is drawn ($\mS\sim \cD$) \\ 
$\mH$ & $\mS (\mS^\top \mA \mB^{-1} \mA^\top \mS)^{\dagger} \mS^\top$ \\
$\mZ$ & $\mA^\top \mH \mA$\\
$\range{\mM}$ & range space of matrix $\mM$  \\
$\kernel{\mM}$ & null space of matrix $\mM$  \\
$\Exp[\cdot]$ & expectation\\
 \hline
 \multicolumn{2}{|c|}{{\bf Projections}}\\
  \hline
$\Pi_{\cL,\mB}(x)$ & projection of $x$ onto $\cL$ in the $\mB$-norm\\
$\mB^{-1}\mZ$ & projection matrix, in the $\mB$-norm, onto $\range{\mB^{-1}\mA^\top \mS}$ \\
 \hline
 \multicolumn{2}{|c|}{{\bf Optimization} }\\
 \hline
$\cX$ &  set of minimizers of $f$  \\ 
  $x^*$ & a point in $\cL$\\
$f_{\mS}$, $\nabla f_{\mS}$, $\nabla^2 f_{\mS}$ & stochastic function, its gradient and Hessian \\
$\cL_{\mS}$ & $\{x\;:\; \mS^\top \mA x = \mS^\top b\}$ (set of minimizers of $f_{\mS}$)  \\
$f$ & $\Exp[f_{\mS}]$\\
$\nabla f$ & gradient of $f$ with respect to  the $\mB$-inner product \\
$\nabla^2 f$ & $\mB^{-1}\Exp[\mZ]$ (Hessian of $f$ in the $\mB$-inner product) \\
 \hline
 \multicolumn{2}{|c|}{{\bf Eigenvalues} }\\
 \hline
 $\mW$ & $\mB^{-1/2}\Exp[\mZ]\mB^{-1/2}$ (psd matrix with the same spectrum as $\nabla^2 f$)\\
$\lambda_1,\dots,\lambda_n$ & eigenvalues of $\mW$\\
$\lambda_{\max}, \lambda_{\min}^+$ & largest and smallest nonzero eigenvalues of $\mW$\\
 \hline
 \multicolumn{2}{|c|}{{\bf Algorithms} }\\
 \hline
$\omega$ & relaxation parameter / stepsize  \\
$\beta$ & heavy ball momentum parameter  \\
$\gamma$ & stochastic heavy ball momentum parameter  \\
$\epsilon^k$ & inexactness error \\
$q$ & inexactness parameter\\
$\rho$ & $1-\omega(2-\omega) \lambda_{\min}^+$\\
\hline
\end{tabular}
\end{center}
\caption{Frequently used notation appeared in Chapters~ \ref{ChapterIntroduction}, \ref{ChapterMomentum} and \ref{ChapterInexact}}
\label{tbl:notationInexact}
\end{table}

\newpage

\section{Notation used in Chapter~\ref{ChapterGossip}}
\label{NotationTable}

\begin{table}[!h]
\begin{center}
\begin{tabular}{|c|l|c|}
 \hline
 \multicolumn{2}{|c|}{{\bf The Basics}}\\
 \hline
$\mA, b$    & $m \times n$ matrix and $m\times 1$ vector defining the system $\mA x =b$\\
$\cL$ &  $\{x\;:\; \mA x = b\}$ (solution set of the linear system) \\
$\mB$    & $n \times n$ symmetric positive definite matrix \\
$\langle x, y \rangle_{\mB}$ & $x^\top \mB y$ ($\mB$-inner product) \\
$\|x\|_{\mB}$ & $\sqrt{\langle x, x \rangle_{\mB}}$ ($\mB$-norm)  \\
$\mM^{\dagger}$ & Moore-Penrose pseudoinverse of matrix $\mM$  \\
$\mS$    & a random real matrix with $m$ rows  \\
$\cD$    & distribution from which matrix $\mS$ is drawn ($\mS\sim \cD$) \\ 
$\mH$ & $\mS (\mS^\top \mA \mB^{-1} \mA^\top \mS)^{\dagger} \mS^\top$ \\
$\mZ$ & $\mA^\top \mH \mA$\\
$\range{\mM}$ & range space of matrix $\mM$  \\
${\rm Null}({\mM})$ & null space of matrix $\mM$  \\
$\Prob(\cdot)$ & probability of an event\\
$\Exp[\cdot]$ & expectation\\
 \hline
 \multicolumn{2}{|c|}{{\bf Projections}}\\
  \hline
$\Pi_{\cL,\mB}(x)$ & projection of $x$ onto $\cL$ in the $\mB$-norm\\
$\mB^{-1}\mZ$ & projection matrix, in the $\mB$-norm, onto $\range{\mB^{-1}\mA^\top \mS}$ \\
 \hline
 \multicolumn{2}{|c|}{{{\bf Graphs}} }\\
 \hline
$\cG = (\cV,\cE)$ & an undirected graph with vertices $\cV$ and edges $\cE$ \\
$n$ & $=|\cV|$ (number of vertices)\\
$m$ & $=|\cE|$ (number of edges) \\
$e = (i,j) \in \cE$ & edge of $\cG$ connecting nodes $i,j\in \cV$  \\
$d_i$ & degree of node $i$  \\
$c \in \R^n$ & $=(c_1,\dots,c_n)$; a vector of private values stored at the nodes of $\cG$  \\
$\bar{c}$ &$\bar{c}=\frac{\sum_i^n \bB_{ii} c_i}{\sum_i^n \bB_{ii}}$ (the weighted average of the private values)  \\
$\bQ \in \R^{m\times m}$ & Incidence matrix of $\cG$  \\
$\bL \in \R^{n\times n}$ & $=\bQ^\top \bQ$ (Laplacian matrix of $\cG$)  \\
$\bD \in \R^{n\times n}$ & $=\text{\textbf{Diag}}(d_1,d_2,\dots, d_n)$ (Degree matrix of $\cG$)  \\
$\bL^{rw} \in \R^{n\times n}$ & $=\bD^{-1} \bL$ (random walk normalized Laplacian matrix of $\cG$)  \\
$\bL^{sym} \in \R^{n\times n}$ & $=\bD^{-1/2} \bL \bD^{-1/2}$ (symmetric normalized Laplacian matrix of $\cG$)  \\
$\ac(\cG)$ & $=\lambda_{\min}^+(\bL)$ (algebraic connectivity of $\cG$)  \\
 \hline
 \multicolumn{2}{|c|}{{\bf Eigenvalues} }\\
 \hline
 $\mW$ & $\mB^{-1/2}\Exp[\mZ]\mB^{-1/2}$ (psd matrix)\\
$\lambda_1,\dots,\lambda_n$ & eigenvalues of $\mW$\\
$\lambda_{\max}, \lambda_{\min}^+$ & largest and smallest nonzero eigenvalues of $\mW$\\
 \hline
 \multicolumn{2}{|c|}{{\bf Algorithms} }\\
 \hline
$\omega$ & relaxation parameter / stepsize  \\
$\beta$ & heavy ball momentum parameter \\
$\rho$ & $1-\omega(2-\omega) \lambda_{\min}^+$\\
\hline
\end{tabular}
\end{center}
\caption{Frequently used notation appeared in Chapter~\ref{ChapterGossip}.}
\label{tbl:notation}
\end{table}

\newpage

\section{Notation used in Chapter~\ref{ChapterPrivacy}}

\begin{table}[!h]
\begin{center}
\begin{tabular}{|c|l|c|}
 \hline
 \multicolumn{2}{|c|}{{\bf Graphs} }\\
 \hline
$\cG = (\cV,\cE)$ & an undirected graph with vertices $\cV$ and edges $\cE$ \\
$n$ & $=|\cV|$ (number of vertices)\\
$m$ & $=|\cE|$ (number of edges) \\
$e = (i,j) \in \cE$ & edge of $\cG$ connecting nodes $i,j\in \cV$  \\
$d_i$ & degree of node $i$  \\
$c \in \R^n$ & $=(c_1,\dots,c_n)$; a vector of private values stored at the nodes of $\cG$  \\
$\bar{c}$ &$\bar{c}=\frac{1}{n}\sum_i c_i$ (the average of the private values)  \\
$\bA \in \R^{m\times n}$ &  \\
$\bL \in \R^{m\times m}$ & $=\bA \bA^\top$ (Laplacian of $\cG$)  \\
$\ac(\cG)$ & $=\lambda_{\min}^+(\bL)$ (algebraic connectivity of $\cG$)  \\
$\beta(\cG)$ & $=n/\ac(\cG)$  \\

\hline
\multicolumn{2}{|c|}{{\bf Randomness}}\\
\hline
$\Exp$ & expectation  \\
$\Prob$ & probability  \\
$\Var$ & variance  \\
$v_i^k$ & random variable from $N(0,\sigma_i^2)$   \\

\hline
\multicolumn{2}{|c|}{{\bf Optimization} }\\
\hline
$P: \R^n \to \R$ & primal objective function   \\
$D: \R^m \to \R$ & dual objective function  (a concave quadratic)  \\
$y \in \R^m$ & dual variable  \\
$y^* \in \R^m$ & optimal dual variable  \\
$x \in \R^n$ & primal variable  \\
$x^* \in \R^n$ & $=\bar{c}\ones$ (optimal primal variable)  \\
$\ones$ & a vector of all ones in $\R^n$  \\

\hline
\multicolumn{2}{|c|}{{\bf Summation} }\\
\hline

$\sum_{i}\sum_j$ & sum through all ordered pairs of $i$ and $j$ \\
$\sum_{(i,j)}$ & sum through all unordered pairs of $i$ and $j$ \\
$\sum_{(i,j) \in \cE}$ & sum through all edges of $\cG$ \\

\hline
\end{tabular}
\end{center}
\caption{Frequently used notation appeared in Chapter~\ref{ChapterPrivacy}.}
\label{tbl:notationPrivacy}
\end{table}

\end{document}